\documentclass{amsart}
\usepackage[top=1in, bottom=1in, left=1.375in, right=1.375in]{geometry}
 \usepackage{amsmath,amssymb}
 \usepackage{algpseudocode}
 \usepackage{algorithm}
 \usepackage{algorithmicx}
 \usepackage{caption}
 \usepackage{subcaption}
 \usepackage{amsthm}
 \usepackage{forest}
 \numberwithin{equation}{section}
 \usepackage{amsfonts}
 \usepackage{graphicx}
  \graphicspath{{converted_eps/}}

\usepackage{amsmath}
\usepackage{amssymb}
\usepackage{amsthm}
\usepackage{amsmath}
\usepackage{setspace}
\usepackage{xcolor}
\usepackage{fancyhdr}
\usepackage{amssymb}
\usepackage{amsthm}
\usepackage{listings}
    \usepackage{cite}
    \usepackage{tabularx}
    \usepackage{epsfig}
    \usepackage{float}
    \usepackage{ listings}
    \usepackage{appendix}
    \usepackage{hyperref}
 \theoremstyle{plain}
  \newtheorem{thm}{Theorem}
 \newtheorem{prop}{Proposition}[section]
 \newtheorem{lem}[prop]{Lemma}
 
 \theoremstyle{definition}

 \theoremstyle{remark}
 \newtheorem{remark}[prop]{Remark}
         
\let \tts = \textstyle      
  \newcommand{\bseq}{\begin{subequations}}
 \newcommand{\eseq}{\end{subequations}}
 \let \mf = \mathsf

\def \ss {\mathbf{s}}
\def \mBs  {\mathsf{Bs}}

\def \CCs { \mathsf{C}}

 \let\pa=\partial
 \let\al=\alpha
 \let\b=\beta
 \let\d=\delta
 \let\g=\gamma
 \let\e=\varepsilon
  
 \let \kp = \kappa
 \let\lam=\lambda

 \let\f=\frac
 \let\inf = \infty
 \let \les = \lesssim
  \let \gtr = \gtrsim
 \let\om=\omega
 
 \let \th = \theta
 \let \pr = \prime
 \let \vp = \varphi
 \let\G= \Gamma
\let\B = \Big
 \let\D=\Delta

 \let\Om=\Omega
 \let\td = \tilde
 \let\wt=\widetilde
 \let\wh=\widehat
 \let \olin = \overline

 \let\teq \triangleq
 
 \let\pa=\partial
 \let \bsh = \backslash
 \let \vs = \vspace

\newcommand{\BT}{\mathbb{T}}
\newcommand{\BZ}{\mathbb{Z}}

\def\proofsketch{2.3}
\def\secsharp{3}
\def\secapprvel{{4.3}}
\def\secEE{{5}}
\def\secvelPI{5.1}
\def\secASS{{7}}
\def\secWtwo{{5.7}}
\def\seccombvelerr{{5.8}}
\def\appwgPI{C.1}
\def\secineq{D}

\def\secsinguasym{5.1.5}
\def\secholconstant{5}
\def\secintker{{5.1}}
\def\seclinfsing{{6.1}}
\def\secudxsupp{6.2}
\def\secKNNerr{6.4.1}
\def\seccwerr{6.4.2}
\def\secuoneD{{7}}
\def\secholRtwo{{7.4}}

\def\applinfestsupp{8}

\def \seceulerblowup{6}

\def \seclinest{5.1-5.5}
\def \secnonlem{A}
\def \seclinfunc{5.6}
\def \secnonend{5.9}

\def \appsharp{B}
\def \secASSget{7}

\def \supptwosharp{5}
\def \supptwovel{6-7}
\def \supponenon{8}

 \def\cB{{\mathcal B}}

 \def\cF{{\mathcal F}}

 \def\cK{{\mathcal K}}
 \def\cL{{\mathcal L}}

 \def\cR{{\mathcal R}}

 \def\na{\nabla}

\def\lt{\left}
\def\rt{\right}
\def\one{\mathbf{1}}

 \def \weg{\wedge}

 \newcommand{\beq}{\begin{equation}}
 \newcommand{\eeq}{\end{equation}}
  \newcommand{\bal}{\begin{aligned} }
  \newcommand{\eal}{\end{aligned}}
 \newcommand{\ben}{\begin{eqnarray}}
 \newcommand{\een}{\end{eqnarray}}
 \newcommand{\beno}{\begin{eqnarray*}}
 \newcommand{\eeno}{\end{eqnarray*}}

 \newcommand{\Den}{\mathrm{Den}}

 \newcommand{\uu}{\mathbf{u}}

 \newcommand{\xx}{\mathbf{x}}

 \newcommand{\R}{\mathbb{R}}

 \newcommand{\vom}{\boldsymbol\omega}

\newcommand{\sgn}{\mathrm{sgn}}
 \newcommand{\supp}{\mathrm{supp}}

 \author{Jiajie Chen and Thomas Y. Hou}
\address{Jiajie Chen, Courant Institute, New York University, New York, NY 10003.   
\newline 
Current address: Department of Mathematics, University of Chicago, Chicago, IL 60637.
\newline
Email: jiajiechen@uchicago.edu
}

\address{
Thomas Y. Hou, Applied and Computational Mathematics, Caltech, Pasadena, CA 91125.
\newline 
Email: hou@cms.caltech.edu
}
  \date{}

\title[Stable  blowup of 2D Boussinesq equations]{Stable nearly self-similar blowup of the 2D Boussinesq and 3D Euler equations with smooth data II: Rigorous Numerics}
 
\begin{document}
 \begin{abstract}
This is Part II of our paper in which we prove finite time blowup of the 2D Boussinesq and 3D axisymmetric Euler equations with smooth initial data of finite energy and boundary. In Part I of our paper \cite{ChenHou2023a}, we establish an analytic framework to prove nonlinear stability of an approximate self-similar blowup profile using a combination of weighted $L^\infty$ and weighted $C^{1/2}$ energy estimates. We reduce proving nonlinear stability to verifying several inequalities for the constants in the energy estimate which  depend on the approximate steady state and the weights in the energy functional only. In Part II of our paper, we construct approximate space-time solutions with rigorous error control, which are used to obtain sharp stability estimates of the linearized operator in Part I.
We also obtain sharp estimates of the regular part of the velocity using numerical integration with computer assistance. These results enable us to verify that the constants in the energy estimate obtained in Part I \cite{ChenHou2023a} indeed satisfy the inequalities for nonlinear stability. The nonlinear stability further implies the finite time singularity of the axisymmetric 3D Euler equations with smooth initial data and boundary.

\end{abstract}
 \setcounter{tocdepth}{1}
  \hypersetup{bookmarksdepth=2}
 \maketitle
\tableofcontents

\vspace{-0.2in}
\section{Introduction}

The three dimensional incompressible Euler equations are one of the most fundamental nonlinear partial differential equations that govern the motion of the ideal inviscid fluid flow. It is closely related to the incompressible Navier-Stokes equations. Due to the presence of nonlinear vortex stretching, the global regularity of the 3D incompressible Euler equations with smooth initial data and finite energy has been one of the longstanding open questions in nonlinear partial differential equations.
Let $\uu$ be the divergence free velocity field and we define $\vom = \nabla \times \uu$ as the \emph{vorticity vector}. The 3D Euler equations governing the vorticity $\vom$ are given by
\begin{equation}
   \vom_{t} + \uu \cdot \nabla \vom = \vom \cdot \nabla \uu,
  \label{eqn_eu_w}
\end{equation}
where $\uu$ is related to $\vom$ via the \emph{Biot-Savart law}. The velocity gradient $\nabla \uu$ formally has the same scaling as vorticity $\vom$. Thus the vortex stretching term, $\vom \cdot \nabla \uu$, has a nonlocal quadratic nonlinearity in terms of vorticity.  Although many experts tend to believe that the 3D Euler equations would form a finite time singularity from smooth initial data, the nonlocal nature of the vortex stretching term could lead to dynamic depletion of nonlinearity, thus preventing a finite time blowup, see e.g. \cite{constantin1996geometric,deng2005geometric,hou2006dynamic}. The interested readers may consult the excellent surveys \cite{constantin2007euler,gibbon2008three,hou2009blow,kiselev2018,majda2002vorticity}  and the references therein. 

Our work is inspired by the computation of Luo-Hou \cite{luo2014potentially,luo2013potentially-2} in which they presented some convincing numerical evidence that the 3D axisymmetric Euler equations with smooth initial data and boundary develop a potential finite time singularity. In Part I of our paper \cite{ChenHou2023a}, we establish an analytic framework and obtain the essential stability estimates to prove finite time singularity of the 2D Boussinesq and 3D axisymmetric Euler equations with smooth initial data and boundary.
The main results of this paper are stated by the two informal theorems below. The more precise and stronger statement of Theorem \ref{thm1a} can be found in Theorem \ref{thm:main} in Section \ref{sec:lin}.

\begin{thm}\label{thm1a}
Let $\th$, $\uu$ and $\om$ be the density, velocity and vorticity in the 2D Boussinesq equations \eqref{eq:bous1}-\eqref{eq:biot}, respectively.
There is a family of smooth initial data $(\th_0, \om_0)$ with $\th_0(x,y)$ being even and $\om_0(x,y)$ being odd in $x$, such that the solution of the 2D Boussinesq equations develops a singularity in finite time $T<+\infty$. The velocity field $\uu_0$ has finite energy. 
The blowup solution $(\th(t), \om(t))$ is nearly self-similar in the sense that $(\th(t), \om(t))$ with suitable dynamic rescaling is  close to an approximate blowup profile $ (\bar \th, \bar \om)$ up to the blowup time. Moreover, the blowup is stable for initial data $ (\th_0, \om_0)$ close to $(\bar \th, \bar \om)$ in some weighted $L^{\inf}$ and $C^{1/2}$ norm.

 \end{thm}

\begin{thm}\label{thm1b}

Consider the 3D axisymmetric Euler equations in the cylinder $r,z \in [0, 1] \times \BT$. Let $u^\th$ and $\om^\th$ be the angular velocity and angular vorticity, respectively. The solution of the 3D Euler equations \eqref{eq:euler10}-\eqref{eq:euler20} develops a nearly self-similar blowup (in the sense described in Theorem \ref{thm1a}) in finite time for some smooth initial data $\om_0^{\th}$, $u_0^{\th}$ supported away from the symmetry axis $r=0$. The initial velocity field has finite energy, $u_0^{\th}$ and $\om_0^{\th}$ are odd and periodic in $z$.  The blowup is stable for initial data $ (u^\th_0, \om^\th_0)$ that are close to the approximate blowup profile $({\bar u}^\th, {\bar \om}^\th)$ after proper rescaling subject to some constraint on the initial support size.

\end{thm}

We first review some main ideas in our stability analysis of the linearized operator presented in Part I \cite{ChenHou2023a}. We use the 2D Boussinesq system as an example. Let $\bar{\omega}$, $\bar{\theta}$ be an approximate steady state of the dynamic rescaling formulation. We denote $W = (\omega, \theta_x, \theta_y)$ and decompose $W = \overline{W} +\widetilde{W}$ with $\overline{W} = (\bar{\omega},\bar{\theta}_x, \bar{\theta}_y)$. We further denote by $\cL$ the linearized operator around $\overline{W}$ that governs the perturbation $\widetilde{W}$ in the dynamic rescaling formulation (see Section \ref{sec:lin}):
\begin{equation}
\label{eq:model_nloc_0}
\widetilde{W}_t = \cL(\widetilde{W}).
\end{equation}
We decompose the linearized operator $\cL$ into a leading order operator $\cL_0$ plus a finite rank perturbation operator $\cK$, i.e 
$\cL = \cL_0 + \cK$.
The leading order operator $\cL_0$ is constructed in such way that we can obtain sharp stability estimates using weighted estimates and sharp functional inequalities.

In Part I \cite{ChenHou2023a}, we have performed the weighted energy estimates using a combination of weighted $L^{\infty}$ and $C^{1/2}$ norm.
In our analysis, we decompose $\widetilde{W} = \widetilde{W}_1 + \widetilde{W}_2$, where $\widetilde{W}_1$ is the main part of the perturbation, which is essentially governed by the leading order operator $\cL_0$ with a weak coupling to $\widetilde{W}_2$ through nonlinear interaction. The perturbation $\widetilde{W}_2$ captures the contribution from the finite rank operator.
The key is to show that the energy estimate of the main part $\widetilde{W}_1$ satisfies the inequalities stated in our stability Lemma \ref{lem:PDE_nonstab} (see Section \ref{sec:lin}). For this purpose, we need to obtain relatively sharp energy estimates for the leading order operator $\cL_0$ by subtracting a finite rank operator $\cK$. Without subtracting the finite rank operator, we would not be able to obtain linear and nonlinear stability of the approximate self-similar profile.

The constants in the weighted energy estimates obtained in Part I \cite{ChenHou2023a} depend on the approximate self-similar profile that we constructed numerically in Section {\secASS} of Part I \cite{ChenHou2023a} and the singular weights we use. In this paper and in the supplementary material  \cite{ChenHou2023bSupp}, we will provide sharp and rigorous upper bounds for these constants by estimating the higher order derivatives and then using interpolation estimates from numerical analysis. 
We also obtain sharp estimates of the regular part of the velocity, which is more regular than the vorticity, by bounding various integrals using numerical integration with computer assistance. These sharp estimates of the constants enable us to prove that the inequalities in our stability lemma hold for our approximate self-similar profile. Thus we can complete the stability analysis of the approximate self-similar profile and complete our blowup analysis for the 2D Boussinesq and 3D Euler equations. See Section \ref{proof-sketch} for more discussion of the main steps in our blowup  analysis.

We use the following toy model to illustrate the main ideas of our stability analysis by considering $\cK$ as a rank-one operator $\cK(\widetilde{W}) = a(x) P(\widetilde{W})$ for some 
operator $P$ satisfying (i) $P(\widetilde{W})$ is constant in space; 
(ii) $\|P(\widetilde{W})\| \leq c \|\widetilde{W}\| $.   
Given initial data $\widetilde{W}_0$, we decompose \eqref{eq:model_nloc_0} as follows 
\beq\label{eq:model2_decoup_0}
\bal
\partial_t \widetilde{W}_1(t) &= \cL_0 \widetilde{W}_1,  \quad \widetilde{W}_1(0) = \widetilde{W}_0 ,  \\
\partial_t \widetilde{W}_2(t) &= \cL \widetilde{W}_2 + a(x) P(\widetilde{W}_1(t)), \quad \widetilde{W}_2(0) = 0.
\eal
\eeq
It is easy to see that $\widetilde{W} = \widetilde{W}_1 + \widetilde{W}_2$ solves \eqref{eq:model_nloc_0} with initial data $\widetilde{W}_0$ since $\cL = \cL_0 + a(x) P$. 
By construction, the leading operator $\cL_0$ has the desired structure that enables us to obtain sharp stability estimates. The second part $\widetilde{W}_2$ is driven by the rank-one forcing term $a(x) P(\widetilde{W}_1(t))$. Using Duhamel's principle and the fact that $P(\widetilde{W}_1(t))$ is constant in space, we yield 
\beq\label{eq:model_nloc4_0}
\widetilde{W}_2(t) = \int_0^t P(\widetilde{W}_1(s)) e^{\cL (t-s) } a(x) ds.
\eeq
If $\wt W_1$ is linearly stable in some $L^{\inf}(\vp)$ space, by checking the decay of $e^{\cL (t)}a(x)$ in the energy space for large $t$, we can obtain the stability estimate of $\wt W_2$. 
Note that $e^{\cL (t) } a(x)$ is equivalent to solving the linear evolution equation $v_t = \cL(v)$ with initial data $v_0 = a(x)$. We can solve this initial value problem by constructing a space-time solution with rigorous error control.

We remark that our stability analysis is performed mainly for $\widetilde{W}_1$ since $\wt W_2$ is driven by $\wt W_1$. The approximation errors in constructing the space-time approximation to $\widetilde{W}_2$ can be controlled by the decay estimate of $\widetilde{W}_1$. Moreover, the region where we need to modify the linearized operator by a finite rank operator is mainly located in a small sector near the boundary where we have the smallest amount of damping. The total rank is less than 50. 
In our construction of approximate solution to $\widetilde{W}_2$, we need to solve the linear PDE \eqref{eq:model_nloc_0} in space-time with a number of initial data, which can be implemented in full parallel.

There has been a lot of effort in studying 3D Euler singularities.  The most exciting recent development is Elgindi's breakthrough result in which he proved finite time singularity of the axisymmetric Euler equation with no swirl for $C^\alpha$ initial vorticity \cite{elgindi2019finite} (see also \cite{elgindi2019stability}). In \cite{chen2019finite2}, we established finite time blowup of the 2D Boussinesq and the 3D axisymmetric Euler equations with $C^{1,\al}$ velocity, large swirl and boundary in a setting similar to the Hou-Luo scenario \cite{luo2013potentially-2,luo2014potentially}. See also \cite{chen2024remarks} for further developments. Earlier efforts include the Constantin-Lax-Majda (CLM) model \cite{CLM85}, the De Gregorio (DG) model \cite{DG90}, the generalized CML (gCLM) model \cite{OSW08} and the Hou-Li model \cite{hou2008dynamic}. See also \cite{CLM85,Elg17,chen2020singularity,chen2020slightly,chen2019finite,Elg19,Cor10,chen2021regularity} for the De Gregorio model and for the gCLM model with various parameters. Inspired by their work on the vortex sheet singularity \cite{caflisch1989a}, Caflisch and Siegel have studied complex singularity for 3D Euler equation, see \cite{caflisch1993,caflisch2009} and also \cite{Frisch2006} for the complex singularities for 2D Euler equation.
 
In \cite{choi2014on}, the authors proved the blowup of the Hou-Luo model proposed in \cite{luo2014potentially}. In \cite{chen2021HL}, Chen-Hou-Huang proved the asymptotically self-similar blowup of the Hou-Luo model by extending the method of analysis established for the finite time blowup of the De Gregorio model by the same authors in \cite{chen2019finite}. In \cite{choi2015finite,kiselev2018finite,hoang2018blowup,hoang2020singular,kiselev2021}, the authors proposed several simplified models to study the Hou-Luo blowup scenario \cite{luo2014potentially,luo2013potentially-2} and established finite time blowup of these models.  In \cite{elgindi2017finite,elgindi2018finite}, Elgindi and Jeong proved finite time blowup for the 2D Boussinesq and 3D axisymmetric Euler equations in a domain with a corner using $\mathring{C}^{0,\al}$ data.

The rest of the paper is organized as follows. In Section \ref{sec:lin}, we review the analytic framework that we established in Part I \cite{ChenHou2023a} and state 
the key lemmas which we use to prove the finite time blowup of the 2D Boussinesq and 3D Euler equations with smooth initial data. 
In Section \ref{sec:lin_evo}, we discuss the construction of the approximate space-time solution to the linearized operator $\cL$. This is crucial to obtain sharp estimates of the perturbed operator $\cL - \cK$ in the stability analysis. In Section \ref{sec:vel_comp}, we show how to estimate the $L^\infty$ and H\"older norms of the regular part of the velocity. Some technical estimates and derivations are deferred to the Appendix.

\section{Review of the analytic framework from Part I \cite{ChenHou2023a}}\label{sec:lin}

In this section, we will review some main ingredients in our analytic framework to establish stability analysis that we presented in Part I \cite{ChenHou2023a}. We will mainly focus on the 2D Boussinesq equations since the difference between the 3D Euler and 2D Boussinesq equations is asymptotically small.
As in our previous works \cite{chen2019finite,chen2019finite2,chen2021HL}, we will use the dynamic rescaling formulation for the 2D Boussinesq equations to study the linear stability for the linearized operator around the approximate steady state of the dynamic rescaling equations. 
Passing from linear stability to nonlinear stability is relatively easier by treating the nonlinear terms and the residual error as small perturbations to the linear damping terms.

Denote by $\om^{\th}$, $u^{\th}$ and $\phi^\th$  the angular vorticity, angular velocity, and angular stream function, respectively.
The 3D axisymmetric Euler equations are given below:
\beq\label{eq:euler10}
\pa_t (ru^{\th}) + u^r (r u^{\th})_r + u^z (r u^{\th})_z = 0, \quad 
\pa_t (\f{\om^{\th}}{r}) + u^r ( \f{\om^{\th}}{r} )_r + u^z ( \f{\om^{\th}}{r})_z = \f{1}{r^4} \pa_z( (r u^{\th})^2 ),
\eeq
where the radial velocity $u^r$ and the axial velocity $u^\th$ are given by the Biot-Savart law:
\beq\label{eq:euler20}
-(\pa_{rr} + \f{1}{r} \pa_{r} +\pa_{zz}) {\phi^\th} + \f{1}{r^2} {\phi^\th} = \om^{\th}, 
\quad  u^r = -\phi^\th_z, \quad u^z = \phi^\th_r + \f{1}{r} {\phi^\th} , 
\eeq
with the no-flow boundary condition ${\phi^\th}(1, z ) = 0$ on the solid boundary $r = 1$ 
and a periodic boundary condition in $z$. For 3D Euler blowup that occurs at the boundary $r=1$, we know that the scaling properties of 
the axisymmetric Euler equations are asymptotically the same as those  of the 2D Boussinesq equations \cite{majda2002vorticity}. Thus, we also study the 2D Boussinesq equations on the upper half space:
\begin{align}
\om_t +  \uu \cdot \na \om  &= \th_{x},  \label{eq:bous1}\\
\th_t + \uu \cdot  \na \th & =  0 , \label{eq:bous2} 
\end{align}
where the velocity field $\uu = (u , v)^T : \R_+^2 \times [0, T) \to \R^2_+$ is determined via the Biot-Savart law
\beq\label{eq:biot}
 - \D \phi = \om , \quad  u =  - \phi_y , \quad v  = \phi_x,
\eeq
where $\phi$ is the stream function with the no-flow boundary condition $\phi(x, 0 ) = 0$ at $y=0$. By making the change of variables $ \td{\th} \teq (r u^{\th})^2,  \td{\om} = \om^{\th} / r$, we can see that $\td{\th}$ and $\td{\om}$ satisfy the 2D Boussinesq equations up to the leading order for $r \geq r_0 >0$.

\subsection{Dynamic rescaling formulation}

Following \cite{chen2019finite,chen2019finite2,chen2021HL}, we consider the dynamic rescaling formulation of the 2D Boussinesq equations. Let $ \om(x, t), \th(x,t) , \uu(x, t)$ be the solutions of \eqref{eq:bous1}-\eqref{eq:biot}. Then it is easy to show that 
\beq\label{eq:rescal1}
\bal
  \td{\om}(x, \tau) &= C_{\om}(\tau) \om(   C_l(\tau) x,  t(\tau) ), \quad   \td{\th}(x , \tau) = C_{\th}(\tau)
  \th( C_l(\tau) x, t(\tau)),  \\
    \td{\uu}(x, \tau) &= C_{\om}(\tau)  C_l(\tau)^{-1} \uu(C_l(\tau) x, t(\tau)) , 
\eal
\eeq
are the solutions to the dynamic rescaling equations
 \beq\label{eq:bousdy0}
\bal
\td{\om}_{\tau}(x, \tau) + ( c_l(\tau) \xx + \td{\uu} ) \cdot \na \td{\om}  &=   c_{\om}(\tau) \td{\om} + \td{\th}_x , \qquad 
\td{\th}_{\tau}(x , \tau )+ ( c_l(\tau) \xx + \td{\uu} ) \cdot \na \td{\th}  = c_{\th} \td \th,
\eal
\eeq
where $\td \uu = (\td u, \td v)^T = \na^{\perp} (-\D)^{-1} \td{\om}$, $\xx = (x, y)^T$, 
\beq\label{eq:rescal2}
\bal
  C_{\om}(\tau) = \exp\lt( \int_0^{\tau} c_{\om} (s)  d \tau\rt), \ C_l(\tau) = \exp\lt( \int_0^{\tau} -c_l(s) ds    \rt) , \  C_{\th}  =  \exp\lt( \int_0^{\tau} c_{\th} (s)  d \tau\rt),
\eal
\eeq
$  t(\tau) = \int_0^{\tau} C_{\om}(\tau) d\tau $ and the rescaling parameters $c_l(\tau), c_{\th}(\tau), c_{\om}(\tau)$ satisfy \cite{chen2019finite2}
\beq\label{eq:rescal3}
c_{\th}(\tau) = c_l(\tau ) + 2 c_{\om}(\tau).
\eeq

To simplify our presentation, we still use $t$ to denote the rescaled time in \eqref{eq:bousdy0} and simplify $\td \om, \td \th$ as $\om, \th$
\beq\label{eq:bousdy1}
\bal
&\om_t + (c_l x + \uu) \cdot \na \om = \th_x + c_{\om} \om ,\quad  \th_t + (c_l x + \uu)\cdot \na \th =  c_{\th} \th .
\eal
\eeq
Following \cite{chen2021HL}, we impose the following normalization conditions on $c_{\om}, c_l$
\beq\label{eq:normal}
c_l = 2 \f{\th_{xx}(0) }{\om_x(0) }, \quad c_{\om} = \f{1}{2} c_l + u_x(0), \quad c_{\th} = c_l + 2 c_{\om}.
\eeq
For smooth data, these two normalization conditions play the role of enforcing 
\beq\label{eq:normal1}
\theta_{xx}(t,0)=\theta_{xx}(0,0), \quad \omega_x(t,0)=\omega_x(0,0)
\eeq
for all time. 

We remark that the dynamic rescaling formulation was introduced in \cite{mclaughlin1986focusing,  landman1988rate} to study the self-similar blowup of the nonlinear Schr\"odinger equations. This formulation is closely related to the modulation technique in the literature and has been developed by Merle, Raphael, Martel, Zaag, and others, see, e.g.  \cite{merle1997stability,kenig2006global,merle2005blow,
martel2014blow,merle2015stability,buckmaster2019formation,buckmaster2019formation2}. 
Moreover, it is related to the method of modulation equations developed by 
Soffer and Weinstein \cite{weinstein1985modulational,soffer1990multichannel,soffer2006soliton}. 
Recently, this method has been applied to study singularity formation in incompressible fluids \cite{chen2019finite2,elgindi2019finite} and related models 
\cite{chen2019finite,chen2020slightly,chen2021regularity,chen2020singularity}. The more precise statement of our Theorem \ref{thm1a} is stated as follows.
\begin{thm}\label{thm:main}
Let $(\bar{\th},\bar{\om}, \bar \uu,  \bar c_l, \bar c_{\om})$ be the approximate self-similar profile constructed in Section {\secASS} of Part I \cite{ChenHou2023a} and $E_* = 5 \cdot 10^{-6}$. For initial data $\th_0(x, y)$ even in $x$ and $\om_0(x, y)$ odd in $x$ of \eqref{eq:bousdy1} satisfying $ E ( \om_0 - \bar \om,  \th_{0,x} - \bar \th_x ,  \th_{0, y} - \bar \th_y ) <  E_*$, 
we have 
\beq\label{eq:thm_est}
 || \om - \bar \om ||_{L^{\inf}}, \  || \th_x - \bar \th_x ||_{L^{\inf}} ,
\  || \th_y - \bar \th_y ||_{\inf} < 200 E_* , \quad | u_x(t, 0) - \bar u_x(0)| , \  | \bar c_{\om} - c_{\om}| < 100 E_*
 \eeq
for all time. In particular, we can choose smooth initial data $\om_0, \th_0 \in C_c^{\inf}$ in this class with finite energy $||\uu_0||_{L^2} < +\infty$ such that the solution to the physical equations \eqref{eq:bous1}-\eqref{eq:biot} with these initial data  blows up in finite time $T$. 
\end{thm}

The energy $E$ is quite complicated, and we refer to Section {\proofsketch} in Part I \cite{ChenHou2023a} for its formula.

\vspace{0.1in}
\paragraph{\bf{Nearly self-similar blowup and the blowup time}}
Based on the main Theorem \ref{thm:main}, the vorticity in the physical space \eqref{eq:bous1},\eqref{eq:bous2} $\om_{phy}$ has the following form 
\[
  \om_{phy}(x, t(\tau)) = C_{\om}^{-1}(\tau) \om_{ss}( C_l(\tau)^{-1} x, \tau ),
  \quad || \om_{ss}(\tau) - \bar \om ||_{L^{\infty}} \ll 1 ,
\]
where $\om_{ss}$ is the self-similar variable ($\td \om$ in \eqref{eq:rescal1}). 
We can generalize the rescaling parameters $C_{\om}, C_l, C_{\th}$ \eqref{eq:rescal2} to $C_{\om}(\tau)
= C_{\om}(0) \exp( \int_0^{\tau} c_{\om}(s) ds ), 
 C_l(\tau) = C_l(0) \exp\lt( \int_0^{\tau} -c_l(s) ds  \rt) $, $ C_{\th}  = C_{\om}^2 C_l^{-1},
 t(\tau) = \int_0^{\tau} C_{\om}(\tau) d\tau $. Using the estimates \eqref{eq:thm_est} and $c_l(\tau) \equiv \bar c_l$ \eqref{eq:normal},\eqref{eq:normal1}, we obtain
 \[
 \bal
 & || \om_{phy}(\tau)||_{L^{\infty}} \approx C_{\om}^{-1}(\tau) || \bar \om ||_{L^{\infty}}, 
 \quad C_{\om}(\tau) \approx C_{\om}(0) e^{ \bar c_{\om} \tau},
 \quad C_l(\tau) = C_l(0) e^{ - \bar c_l \tau}, \\
& T = t(\infty)  \approx C_{\om}(0) |\bar c_{\om}|^{-1}
=  \f{ || \bar \om ||_{L^{\infty}} }{  || \om_{phy}(0)||_{L^{\infty}}   |\bar c_{\om}| } ,
\quad T - t(\tau) \approx C_{\om}(\tau) |\bar c_{\om}|^{-1}, \\
&  C_l(\tau)^{-1} \approx C (T- t(\tau))^{\bar c_l / \bar c_{\om}}, 
 \quad  \om_{phy}(\tau) \approx  (T-t(\tau))^{-1} |\bar c_{\om}|^{-1} 
 \bar \om( C x (T- t(\tau))^{\bar c_l / \bar c_{\om}}) ,
\eal
 \]
 with $\bar c_l / \bar c_{\om} \approx -2.92 < 0$, for some $C>0$ depending on $\bar c_l, \bar c_{\om}, C_l(0), C_{\om}(0)$. The notation $\approx$ means that the relation holds approximately. The exact relation can be inferred from \eqref{eq:thm_est}, \eqref{eq:rescal1}, \eqref{eq:rescal2}. The blowup time is approximately inversely proportional to $|| \om_{phy}(0)||_{L^{\infty}}$. Since we only prove that $\om_{ss}(\tau)$ is sufficiently close to the approximate profile $\bar \om$ and do not prove convergence of $\om_{ss}(\tau)$ as $\tau \to \infty$, Theorem \ref{thm:main} does not imply \textit{asymptotically} self-similar blowup.

\subsection{The main steps in the proof of Theorem \ref{thm:main}}
\label{proof-sketch}

We will follow the framework in \cite{chen2019finite,chen2019finite2,chen2021HL} to establish finite time blowup by proving the nonlinear stability of an approximate steady state to \eqref{eq:bousdy1}. 
We divide the proof of Theorem \ref{thm:main} into proving the following lemmas. The energy norm below is defined in Section 5 in Part I \cite{ChenHou2023a} for energy estimates, and the requirement of smallness is incorporated in the conditions \eqref{eq:PDE_nondiag}, e.g. the term $a_{ij, 3}$, in Lemma \ref{lem:main_stab}.

The upper bar notation is reserved for the approximate steady state, e.g. $\bar \om, \bar \th$. 
Given the approximate steady state $\bar \om, \bar \th, \bar c_l, \bar c_{\om}$, we denote by $\olin \cF_i$ and $\bar F_{\om}, \bar F_{\th}$ the residual error
\beq\label{eq:bous_err}
\bal
 \bar F_{\om}  & = - (\bar c_l x + \bar \uu ) \cdot \na \bar \om + \bar \th_x + \bar c_{\om} \bar\om   , \quad  \bar F_{\th} = - (\bar c_l x + \bar \uu ) \cdot \na \bar \th + \bar c_{\th} \bar \th , \\
\olin \cF_1 & \teq \bar F_{\om}, \quad \olin \cF_2 \teq \pa_x \bar F_{\th}, \quad \olin \cF_3 \teq \pa_y \bar F_{\th}, \quad \bar c_{\th} = \bar c_l + 2 \bar c_{\om}.
 \eal
 \eeq

We have the following nonlinear stability Lemma for  $L^{\inf}$-based energy estimate, which
is proved in Appendix A.1 of Part I \cite{ChenHou2023a}.

\begin{lem}\label{lem:PDE_nonstab}
Suppose that $f_i(x, z, t) : \R^2_{++} \times \R^2_{++} \times [0, T] \to \R, 1\leq i \leq n$, satisfies 
\beq\label{eq:PDE_nonstab_1}
\pa_t f_i + v_i(x, z) \cdot \na_{x, z} f_i = -a_{ii}(x, z, t)  f_i+ B_i(x, z, t) + N_i(x, z, t) + \bar \e_i,
\eeq
where $v_i(x, z, t)$ are some vector fields Lipschitz in $x, z$ with $v_i |_{x_1 = 0} = 0, v_i |_{z_1 = 0} = 0$. For some $\mu_i > 0$, we define the energy 
\[
E(t) = \max_{1 \leq i \leq n} (\mu_i || f_i||_{L^{\inf}}).
\]
Suppose that $B_i, N_i$ and $\bar \e_i$ satisfy the following estimate 
\beq\label{eq:PDE_nonstab_2}
\bal
\mu_i (|B_i(x, z, t) | + |N_i(x, z, t)| +| \bar e_i|) \leq  \sum_{j \neq i} ( |a_{ij}(x, z, t)| E(t) 
+ | a_{ij, 2}(x, z, t)| E^2(t) + | a_{ij, 3}(x, z, t)|  ).
\eal
\eeq

If there exists some $E_*, \e_0, M > 0$ such that 
\beq\label{eq:PDE_nondiag}
\bal
&  a_{ii}(x, z, t) E_*  - \sum_{j \neq i} ( |a_{ij}|  E_* + |a_{ij, 2}| E_*^2 
+ |a_{ij,3}(x, z, t) | ) >  \e_0 ,   \\
&  \sum_{j \neq i} ( |a_{ij}|  E_* + |a_{ij, 2}| E_*^2 
+ |a_{ij,3}(x, z, t) | ) < M ,
\eal
\eeq
for all $x, z$ and $t \in [0, T]$. Then for $E(0)< E_*$, we have $E(t) < E_*$ for $t \in [0, T]$. 
\end{lem}

\begin{lem}\label{lem:main_ASS}
There exists a nontrivial approximate steady state $(\bar \om, \bar \th, \bar c_l, \bar c_{\om})$ to \eqref{eq:bousdy1}, \eqref{eq:normal} with $\bar \om, \bar \th \in C^{4, 1}$ and residual errors $\bar \cF_i, i=1,2,3$ \eqref{eq:bous_err} sufficiently small in some energy norm. 
\end{lem}

The construction of an approximate self-similar profile with a small residual error stated in Lemma \ref{lem:main_ASS} is provided in Section {\secASS} of Part I \cite{ChenHou2023a} and
the properties of $(\bar \om, \bar \th, \bar c_l, \bar c_{\om})$ are described in Section 2.4 of Part I \cite{ChenHou2023a}. We will estimate the local part of the residual error in Appendix \ref{sec:resid}. We linearize \eqref{eq:bousdy1} around $(\bar \om, \bar \th, \bar c_l, \bar c_{\om})$ 
and perform energy estimate of the perturbation $W= (\om, \th_x, \th_y )$ in Section {\secEE} in Part I \cite{ChenHou2023a}. In our estimates, we need to control a number of nonlocal terms.

\begin{lem}\label{lem:main_vel}

Let $\om$ be odd in $x_1$. Denote $\d(f, x, z) = f(x) - f(z)$.  
There exists finite rank approximations $\hat \uu, \wh{\na \uu}$ for $\uu(\om), \na \uu(\om)$ with rank less than $50$ such that we have the following weighted $L^{\inf}$ and directional H\"older estimate for $f = u, v, \pa_l u, \pa_l v, x, z \in \R_2^{++}, i = 1, 2,\g_i >0$
\beq\label{eq:main_vel}
\bal
  | \rho_f( f - \hat f)(x) | &\leq C_{ f, \inf}(x, \vp,  \psi_1,   \g) \max(  || \om \vp||_{\inf} , s_f \max_{j=1,2} \g_j  [ \om \psi_1]_{C_{x_j}^{1/2}(\R_2^{+}) } ), \\
   \f{ | \d(\psi_f (f - \hat f), x, z ) | }{ |x-z|^{1/2}} 
  & \leq C_{ f, i}(x, z, \vp, \psi_1,  \g)\max(  || \om \vp||_{\inf} , 
  s_f \max_{j=1,2} \g_j  [ \om \psi_1]_{C_{x_j}^{1/2}(\R_2^{+}) } ), 
 \eal
\eeq
with $ x_{3-i} = z_{3-i}$, where $s_f = 0$ for $f = u, v$, $s_f =1$ for $f =\pa_l u, \pa_l v$, the functions $C(x), C(x, z)$ depend on $\g$, the weights, and the approximations, the singular weights $\vp = \vp_1, \vp_{g, 1}, \vp_{elli},  \psi_{\pa u} = \psi_1, \psi_u$  are defined in \eqref{wg:linf},  the weight $\rho_{10}$ for $\uu$ and the weight for $\rho_{ij} $ for $\na \uu$ with $i+j=2$ are given in \eqref{wg:linf}. In the estimate of $f = u, v$, we do not need the H\"older semi-norm and we set $s_f = 0$. Moreover, $C(x), C(x, z)$ are bounded in any compact domain of $\R_2^{++}$. We have an additional estimate for $\rho_4 (u - \hat u)$ similar to the above with $\rho_4$ \eqref{wg:linf} singular along $x_1 = 0$.

Furthermore, we have the following estimate using the localized norm. There exist $D_1, D_2, .. D_n \subset \R_2^{++}$ and $ D_S \in \R^+_2 $ depending on $x$ in the $L^{\inf}$ estimate and $x, z$ in the $C_{x_i}^{1/2}$ estimate, such that
\[
\bal
   | \rho_f( f - \hat f)(x) | & \leq 
\sum_{j}  C_{ f, \inf, j}(x, \vp, \psi_1, \g) || \om \vp||_{L^{\inf}(D_j) } +  C_{f, \inf, S}(x, \vp, \psi_1, \g) 
\max_{l =1,2}(    \g_l  [ \om \psi_1]_{C_{x_l}^{1/2}( D_S ) }), \\
    \f{ | \d(\psi_f (f - \hat f), x, z ) | }{ |x-z|^{1/2}} 
  & \leq \sum_{j}  C_{ f, i, j}(x,z, \vp, \psi_1, \g) || \om \vp||_{L^{\inf}(D_j) }  +  C_{f, i, S}(x,z,  \vp, \psi_1, \g)  \max_{l=1,2}(    \g_l  [ \om \psi_1]_{C_{x_j}^{1/2}( D_S ) }), 
 \eal
 \]
 for $x_{3-i} = z_{3-i}, \vp = \vp_{elli}$ and the same notation as above, where $C_{f,\inf,S}, C_{f,i, S} = 0$ for $f = u, v$. Similarly, we have an estimate for $\rho_4 (u - \hat u)$ using localized norm with $C_{f, \infty, S} = 0$ similar to the above. 

\end{lem}

Since the weights $\rho_{10} \sim |x|^{-3}, \psi_1 \sim |x|^{-2}, \psi_u $ are singular near $x=0$, without subtracting the approximation $\hat f$ from $f$, $\rho_f f$ is not bounded near $x=0$. We design the finite rank approximations $\hat \uu, \wh{\na \uu}$ in Section {\secapprvel} in Part I \cite{ChenHou2023a}.

Based on these finite rank approximations, we can decompose the perturbations. 

\begin{lem}\label{lem:main_EE}
There exists $m < 50$ approximate solutions $\hat F_i$ to the linearized equations $\pa_t W = \cL W$ of \eqref{eq:bousdy1} around $(\bar \om, \bar \th, \bar c_l, \bar c_{\om})$ in Lemma \ref{lem:main_ASS} from given initial data $\bar F_i(0)$
with residual error $\cR$ small in the energy norm. Further  we can decompose the perturbation $ W = W_1 + \widehat W_2$ with the following properties. (a) $\hat W_2$ is constructed based on $\widehat F_i$, see Section 4.2.4 of Part I \cite{ChenHou2023a}; (b) $W_1$ satisfies the equations with the leading order linearized operator $(\cL -  \cK) W_1$ up to the small residual error $\cR$ for some finite rank operator $\cK$, and $W_1$ depends on $\widehat W_2$ weakly at the linear level via $\cR$. 
The functionals $a_i(W_1), a_{nl, i}(W)$ in the construction of $\widehat W_2$ and $\cK$ (see Section 4.2.4 of Part I \cite{ChenHou2023a}) are related to the finite rank approximations in Lemma \ref{lem:main_vel}.

Moreover, there exists an energy $E_4(t)$ for $W_1, W$ (see Section 5.6.3. of Part I \cite{ChenHou2023a}) that controls the weighted $L^{\inf}$ and $C^{1/2}$ seminorm of $W_1$ such that under the bootstrap assumption $E_4(t) < E_{*0}$ with $E_{*0} > 0$, we can establish nonlinear energy estimates for $E_4(t)$ using the estimates in Lemma \ref{lem:main_vel}. 

\end{lem}

If the bounds in Lemma \ref{lem:main_vel} are tight, and the residual error in the constructions of $(\bar \om, \bar \th), \widehat F_i$ are small enough, we can use Lemma \ref{lem:PDE_nonstab} to obtain nonlinear stability.

\begin{lem}\label{lem:main_stab}

For $E_* = 5 \cdot 10^{-6}$, the coefficients in the nonlinear energy estimates of $E_4(t)$ satisfy the conditions \eqref{eq:PDE_nondiag}, and the statements in Theorem \ref{thm:main} hold true.

\end{lem}

The main purpose of Part II of our paper is the following. Firstly, we construct the approximate $\hat F_i(t)$ in Lemma \ref{lem:main_EE} numerically, and estimate its piecewise derivatives and the local residual error in Section \ref{sec:lin_evo}. 
Secondly, in Section \ref{sec:vel_comp}, we obtain sharp estimates of the constants in Lemma \ref{lem:main_vel}, which only depend on the weights. 
Thirdly, we estimate piecewise bounds of the approximate steady state in Appendix \ref{app:solu}, the singular weights in Appendix \ref{app:wg_tot}, some explicit functions related to the approximate solutions in Appendix \ref{app:explcit}. We remark that all of these estimates and constants depend on the given weights, some operators and functions, e.g. the approximate steady state and the specific initial conditions.
With these estimates and constants, we obtain the concrete values of the inequalities in \eqref{eq:PDE_nondiag} and Lemma \ref{lem:main_stab}, which are given in Appendix {\secineq} in Part I \cite{ChenHou2023a}. We further verify the inequalities for the stability conditions in Lemma \ref{lem:main_stab}.

Let us make a few comments on the above lemmas. Firstly, our energy estimate is based on weighted $L^{\inf}$ functional spaces, which is crucial for extracting the damping terms for the energy estimate. 
See Section 2.7 of Part I \cite{ChenHou2023a} for the motivations. 
Given $\om \in C^{1/2}$, we have $\uu \in C^{3/2}, \na \uu \in C^{1/2}$. To establish the nonlinear stability conditions \eqref{eq:PDE_nondiag} in Lemma \ref{lem:main_stab}, we need sharp constants in the estimates in Lemma \ref{lem:main_vel}. We use some techniques from optimal transport to obtain sharp $C^{1/2}$ estimate of $\na \uu$ in Section {\secsharp} of Part I \cite{ChenHou2023a}. This corresponds to the limiting case in the $C_{x_i}^{1/2}$ estimate in Lemma \ref{lem:main_vel} for a fixed $x$ with $|x-z| \to 0$ and captures the most singular part in the estimates in Lemma \ref{lem:main_vel}. The constants in the sharp $C^{1/2}$ estimate established in Part I \cite{ChenHou2023a} are given by several integrals. In Section {\secholconstant} of the supplementary material II \cite{ChenHou2023bSupp}, we estimate these integrals.

Other parts of the estimates in Lemma \ref{lem:main_vel} are more regular since we work with the regular part of the velocity integral with a desingularized kernel. Given $\om \in C^{1/2}$, we can reduce the estimates of these more regular terms to estimate some explicit $L^1$ integrals.
We can obtain sharp estimates of these more regular integrals using some numerical quadrature with computer assistance. See Section \ref{sec:vel_comp}.

By designing $\cK$ to approximate the nonlocal terms, we can obtain much better linear stability estimates for $\cL - \cK$.  
After we have shown that the stability conditions \eqref{eq:PDE_nondiag} are satisfied, we have 
nonlinear stability estimate $E_4(t)< E_*$ for all $t>0$ using Lemma \ref{lem:PDE_nonstab}, which implies the bounds in Theorem \ref{thm:main}. 
The remaining steps of obtaining finite time blowup from smooth initial data and finite energy follow \cite{chen2019finite} and a rescaling argument. We remark that the variable $\widehat W_2$ in Lemma \ref{lem:main_EE} (see full definition in Section 4.2.4 of Part I \cite{ChenHou2023a}) plays an auxiliary role, and we do not perform energy estimate on $\widehat W_2$ directly.

Note that all the nonlocal terms in the linearized equations are not small. Without the sharp $C^{1/2}$ estimate, with the choice of energy $E_4$, the stability conditions in \eqref{eq:PDE_nondiag} and Lemma \ref{lem:main_stab} fail in the weighted H\"older estimate. Without the finite rank approximations for the nonlocal terms in Lemma \ref{lem:main_vel}, \ref{lem:main_EE}, the stability conditions for weighted $L^{\inf}$ estimate also fail. 

\vs{0.1in}
\paragraph{\bf{Rigorous numerics}}
We need to track two types of errors for rigorous numerics. 
The first type is the discretization error, e.g. the error terms in the Trapezoidal rule and in the interpolating polynomials. The second type is the round-off error in the computation. 
We use numerical analysis to estimate {\it all} the discretization errors, and use {\it only} the basic interval arithmetic \cite{moore2009introduction,rump2010verification}, see e.g. \eqref{eq:func_intval}, \eqref{eq:func_intval2}, in the INTLAB package \cite{Ru99a} from MatLab to track the round-off error.

In our nonlinear estimates, we use a singular weight $\vp$ like $|x|^{-3}$ near $x=0$ to measure the residual error $\bar \cF_i$. 
 To obtain a small weighted residual error $|\vp \bar \cF_i|$ near $x=0$, we choose the mesh $y_i$ \eqref{eq:ASS_mesh} representing the approximate profile to be exact floating point 
numbers to reduce the round-off error near $x =0$.

The codes for the computations are implemented in MatLab and can be found in \cite{ChenHou2023code}. 
The estimates of the constants in Lemma \ref{lem:main_vel}, integrals in Section \ref{sec:vel_comp}, and the constructions and estimates of the approximate space-time solutions in Lemma \ref{lem:main_EE} and in Section \ref{sec:lin_evo} are performed in parallel using the Caltech High Performance Computing\footnote{See more details for Caltech HPC Resources \url{https://www.hpc.caltech.edu/docs/resources.html}
}. Other computer-assisted estimates and the main part of the verifications are done in Mac Pro (Rack,2019) with 2.5GHz 28‑core Intel Xeon W processor and 768GB (6x128GB) of DDR4 ECC memory.

\subsection{Dependency tree}

The following tree schematizes various intermediate steps and related sections that lead to the main stability result Theorem \ref{thm:main}, which implies the blowup result Theorem \ref{thm1a} for the 2D Boussinesq equations. The blowup for the Euler equations in Theorem \ref{thm1b} is proved by a perturbative argument in Section {\seceulerblowup} in Part I \cite{ChenHou2023a}. 

Below, \textit{Thm, Lem, App, Sec, P1, P2, Supp1, Supp2} are short for Theorem, Lemma, Appendix, Section, Paper I \cite{ChenHou2023a}, Paper II (the current paper), the Supplementary material for Paper I \cite{ChenHou2023aSupp} and Paper II \cite{ChenHou2023bSupp}, respectively. 
We present a few more detailed derivations in the Supplementary materials \cite{ChenHou2023aSupp}, \cite{ChenHou2023bSupp}, which expand and generalize discussions in the main papers and are less essential. Moreover, we present several explicit formulas we used in our computer-assisted 
estimates for the quantities derived in the main papers.

\begin{figure}
\centering
\begin{forest}
 for tree={l sep=3pt,
    s sep = 1pt,
    grow=east,
    parent anchor=east,
    child anchor=west,
    }
      [Thm \ref{thm:main}: Proved by  \\
      Nonlinear stability 
        lems: App \secnonlem{,} P1 (or Lem \ref{lem:PDE_nonstab}{,} P2){,}\\
        \& inequalities: \\
        App \secineq{,} P1 (summarized in Lem \ref{lem:main_stab}{,} P2),  text width= 3cm 
              [Approximate \\ profile: \\  Lem \ref{lem:main_ASS}{,} P2, text width=2cm                                                                          
                    [Estimate residual error{:} \\
                    Sec \ref{sec:err_idea}(ideas){,} \\
                    App \ref{sec:resid}{,} \ref{app:explcit} \& \ref{app:linf_est} P2, 
                    text width=4cm
                    ]
                    [Estimate profile: \\
                     App \ref{app:piece_pol}{,} \ref{sec:est_appr_far}{,} \ref{app:explcit}{,} P2, text width=4cm]    
                    [Construct profile{:} \\ Sec \secASSget {,} P1{,} \\
                    App \ref{app:solu_rep} \& App \ref{app:explcit}{,} P2, text width=4cm ]  
              ]
              [Nonlinear \\ estimates,  text width=2cm 
                  [Estimate nonlinear \\ terms:
                    Sec \seclinfunc-\secnonend{,} P1, text width=4cm
                        [Estimate similar\\ nonlinear
                         \& error terms: \\Sec {\supponenon} Supp1, text width=4cm
                        ]
                   ]
                  [Estimate finite rank \\
                    part $\hat W_2${,} Sec \secWtwo{,} P1, text width=4cm      
                       [Estimate the residual  \\
                       error for $\hat W_2$ \\
                        Sec \ref{sec:err_idea} (ideas){,} \\
                        App \ref{sec:resid}{,} \ref{app:explcit} \& \ref{app:linf_est} P2, text width=4cm]                            
                        [Construct finite rank \\ part $\hat W_2${:} Lem \ref{lem:main_EE}{,}  \\
                        Sec \ref{sec:lin_evo} \& App \ref{app:explcit}{,} P2, text width=4cm] 
                  ]
              ]     
              [Linear \\
              stability,   text width=1.8cm    
                  [Estimate nonlocal terms: \\
                   Lem 2.3{,} \\ Sec \ref{sec:vel_comp} \& App \ref{app:ker}{,} P2, text width=4cm
                      [A few more similar cases \\
                      and explicit formulas: Sec {\supptwovel} Supp2, text width=4cm]
                  ]
                  [Linear energy estimates{:}\\ Sec \seclinest{,} P1,  text width=4cm]
                  [Finite rank perturbation{:} \\ Sec 4{,} P1,  text width=4cm]
                  [Sharp H\"older estimates: \\ Sec {\secsharp} \& App {\appsharp}{,} P1, text width= 4cm
                      [Compute the sharp \\ constant:
                      Sec {\supptwosharp} Supp2, text width= 4cm]
                  ]
              ]                                 
            ]
      ]
\end{forest}

\vspace{-0.05in}
\captionsetup{width=0.9\textwidth} 
 \caption*{Level 2: main steps. Level 3: main estimates and methods. Level 4: compute and estimate explicit functions or constants, and additional similar estimates. 
}
\end{figure}

In Section {\applinfestsupp} in Supp2, we generalize the standard interpolation estimate in numerical analysis to derive higher order interpolation estimates, which are used to estimate the residual error effectively. See the discussion in Section \ref{sec:err_idea}.
In Appendix \ref{app:wg_tot}, we derive piecewise bounds for various weights, which are used 
in the weighted estimates of the nonlocal terms, the residual error, and in the linear, nonlinear estimates for stability.

In Appendix \ref{app:nota}, we collect the main notations used in this paper.

\section{Construct and estimate the approximate solution to the linearized equations}\label{sec:lin_evo}

As we described in Section 2 of Part I \cite{ChenHou2023a} (see also the Introduction), we need to construct the approximate solutions to $e^{\cL t} F_0$ for several initial data $\bar F_i, \bar F_{\chi, i}$. In this section, we discuss how to construct these space-time solutions numerically with some vanishing properties at the origin with rigorous error control.

The linearized equations associated with $\cL$ 
 read 
\beq\label{eq:lin_evo}
\bal
\pa_t \om &=  - (\bar c_l x +\bar u) \cdot \na \om + \eta + \bar c_{\om} \om 
 -  \uu \cdot \na \bar \om + c_{\om} \bar  \om  =  \cL_1 (\om, \eta, \xi) ,\\  
 \pa_t \eta & =  - (\bar c_l x +\bar u) \cdot \na \eta +
(2 \bar c_{\om} - \bar u_x) \eta  - \bar v_x \xi -  \uu_x \cdot \na \bar \th   
- \uu \cdot \na \bar \th_x  + 2 c_{\om} \bar \th_x =  \cL_2( \om, \eta, \xi )  , \\
  \pa_t \xi & = -  (\bar c_l x +\bar u) \cdot \na  \xi +
(2 \bar c_{\om} + \bar u_x) \xi - \bar u_y \eta  - \uu_y \cdot \na \bar \th   
- \uu \cdot \na \bar \th_y  + 2 c_{\om} \bar \th_y  = \cL_3( \om, \eta, \xi) , \\
\eal
\eeq
with normalization condition 
\beq\label{eq:normal_pertb}
c_{\om} =  u_x(0), \quad  c_l \equiv 0 .
\eeq
Although $\eta, \xi$ represent $\th_x, \th_y$ in the Boussinesq equations, we will consider initial data $(\om_0, \eta_0, \xi_0)$ with $ \pa_y \eta_0 \neq \pa_x \xi_0$. Thus, we do not have the relation $\pa_y \eta = \pa_x \xi$ and will treat $\eta, \xi$ as two independent variables. The solutions $\om(x), \eta(x)$ are odd in $x_1$, $\xi(x)$ is even in $x_1$ with $\xi(0, y) = 0$. We consider initial data $(\om_0, \eta_0, \xi_0 ) = O(|x|^2)$ near $x=0$. 
Using a direct calculation, we can show that these vanishing conditions are preserved in time
\beq\label{eq:lin_evo_order}
\om(t, x),  \ \eta(t, x) , \  \xi(t, x) = O(|x|^2) .
\eeq

We introduce the bilinear operator $B_{op,i}( (\uu, M), G)$ for $(\uu, M), G = (G_1, G_2, G_3) $
\beq\label{eq:Blin_gen}
\bal
\cB_{op, 1}  & = - \uu \cdot \na G_1 + M_{11}(0) G_1, 
\quad \cB_{op, 2}= - \uu \cdot \na G_2 +2 M_{11}(0) G_2 - M_{11} G_2 - M_{21} G_3, \\
\cB_{op, 3} & = - \uu \cdot \na G_3 + 2 M_{11}(0) G_3 - M_{12} G_2 - M_{22} G_3.
\eal
\eeq
If $M=  \na \uu, M_{11} = u_x, M_{12} = u_y, M_{21} = v_x, M_{22} = v_y$, then we drop $M$ to simplify the notation 
\beq\label{eq:Blin}
\bal
\cB_{op, 1}(\uu, G )  & = - \uu \cdot \na G_1 + u_x(0) G_1, 
\quad  \cB_{op, 2}
= - \uu \cdot \na G_2 +2 u_x(0) G_2 - u_x G_2 - v_x G_3, \\
\cB_{op, 3}
& = - \uu \cdot \na G_3 + 2 u_x(0) G_3 - u_y G_2 - v_y G_3.
\eal
\eeq

The main result in this section is the following. Given $n$ initial data $ \bar W_i = (\bar W_{i, 1}, \bar W_{i, 2}, \bar W_{i, 3})$ and $n$ functions $c_i(t) (i=1,2,..,n)$ which are Lipschitz and bounded in $t$, we construct approximate space-time solution $\hat W_i  =(\hat W_{i, 1}, \hat W_{i, 2}, \hat W_{i, 3})$, the combined solution $\hat G = (\hat G_1, \hat G_2, \hat G_3)$ and the approximate stream functions $\hat \phi_i^N, \hat \phi^N$ and the error $\hat \e$ associated with $ \hat W_{i,1}, \hat G_1$
\beq\label{eq:lin_evo_main1}
 \hat G = \sum_{ i\leq n} \int c_i(t- s) \hat W_i(s) ds , \quad \hat \phi^N = \sum_{ i\leq n} \int c_i(t- s) \hat \phi^N_i(s) ds ,  \quad \hat \e =  \sum_{ i\leq n} \int c_i(t- s) ( \hat W_{i,1} + \D \hat \phi^N_i )(s) ds ,
\eeq
with residual error 
\beq\label{eq:lin_evo_err_def}
\cR =   \sum_{i \leq n} \B( c_i(t) ( \hat W_i(0) - \bar W_i ) +  \int_0^t c_i(t-s) (\pa_s - \cL) \hat W_i(s) ds \B),
\eeq
vanishing $O(|x|^3)$ near $x=0$. Moreover, we can decompose $\cR = (\cR_1, \cR_2, \cR_3)$ as follows 
\beq\label{eq:lin_evo_main2}
\bal
& \cR_j(t) =  \cR_{loc, 0, j}(t) + \cR_{nloc, j}(t), \quad   \cR_{loc,0, j} = \sum_{i\leq n} \int_0^t c_i(t-s) \cR_{num,i,  j}(s) ds , \quad \cR_{num, i, j} = O(|x|^3) ,  \\
& \cR_{nloc, j} = P_j - D_j^2 P_j(0) \chi_{j, 2}, \quad P_j = -  \cB_{op, j}( \uu( \bar \e), \hat G  ) - \cB_{op, j}( \uu(\hat \e), (\bar \om, \bar \th_x, \bar \th_y )  ,  
\eal
\eeq
with $j=1,2,3$, where $\chi_{j 2}$ is defined in \eqref{eq:solu_cor}, and $\bar \e = \bar \om - (-\D) \bar \phi^N$ is the error of the approximate stream function for $(-\D)^{-1} \bar \om$, $\cR_{num,i,j}(t, x)$ depends on $\hat W_i, \hat \phi_i^N$ in $x$ locally. We have absorbed the initial error in $\cR_{num, \cdot}$. We derive the above decompositions and estimates of $\cR_{loc, 0, j}, \hat G, \hat \phi^N, \hat \e$, in Section \ref{sec:stop}-\ref{sec:vel_err}. See \eqref{eq:W2_est1}, \eqref{eq:lin_evo_err2}. The error in constructing the stream function $\hat \phi_i^N$ associated with $\hat W_{i, 1}$ leads to nonlocal error, e.g. $\uu(\hat \e)$, in constructing the velocity. We combine the estimate of the nonlocal error in $P_j$ and perturbation in Section {\seccombvelerr} in Part I \cite{ChenHou2023a}. Furthermore, we track the piecewise bounds of the following quantities  
\beq\label{eq:lin_evo_main3}
\bal
 & \int_0^{\inf} |\pa_x^k \pa_y^l F(t)| dt, \
F =\hat W_{i, j},  \quad  F = 
\hat \phi^N_{i}, 
\ F =  \hat \phi^N_{i} - \pa_{xy} \hat \phi^N_i(0) xy  , 
\   F=  \hat W_{i, 1 } + \D \hat \phi_i^N, 
  \\
& F =   \ell_j \hat W_{i, j} - x \pa_x  \hat W_{i, j} + y \pa_y \hat W_{i, j} - D_j^2 \hat W_{i, j}(0) f_{\chi, j} , \ D^2 = (\pa_{xy}, \pa_{xy}, \pa_{xx}), \ell = (1,1,3), 
 \eal
\eeq
for $j =1,2,3, i=1,2,.., n$, where $f_{\chi, j}$ is defined in \eqref{eq:cutoff_near0_all}. We track the $C^2$ bound of $\hat W_{i, j}$ and $C^4$ bounds for others following \eqref{eq:W2_est1}, and use these bounds to control $\wh W_2$ in Lemma \ref{lem:main_EE} and use them in the nonlinear energy estimates in Section {\secEE} in Part I \cite{ChenHou2023a}.

In practice, we choose the initial data $ \bar F_i$ given in Appendix C.2.1 in Part I \cite{ChenHou2023a}, and $c_i(t)$ some functionals of the perturbation $W_1, \hat W_2$ related to the finite rank perturbation.

In the rank-one construction below, we suppress the mode index, so the subscript $i=1,2,3$ in $\widehat W_i^{(1)}, \hat W_i^{(2)}, \hat W_i $ denotes a component. In Section \ref{sec:vel_err}, where we assemble the estimates for all modes, 
the subscript $i$ in $\widehat W_i^{(1)}, \hat W_i^{(2)}, \hat W_i $  instead denotes the solution of the $i$-th mode.

\vs{0.1in}
\paragraph{\bf{Numerical methods}}
We solve \eqref{eq:lin_evo} using the numerical method outlined in Section {\secASS} of Part I \cite{ChenHou2023a} to obtain the solution $ (\om_k, \eta_k, \xi_k)$ at discrete time $t_k$. Since $\xi$ is even with $\xi(0, y) =0$, we write $\xi = x \zeta$ for an odd function $\zeta$. We use the adaptive mesh discussed in Appendix \ref{app:solu_rep} to discretize the spatial domain. Then we represent $\om, \eta, \zeta$ using the piecewise 6-th order B-spline \eqref{eq:w_spline}. See Appendix \ref{app:solu_rep}. To solve the stream function $-\D \phi = \om$ numerically, we use the B-spline based finite element method and obtain the numerical approximation $\phi^N$ for $(-\D)^{-1} \om$. Then we can construct the velocity $\uu^N = \na^{\perp} \phi^N $.

The gradients of several initial conditions $\bar F_i$ are relatively large and the linearized equations \eqref{eq:lin_evo} involve $\na \hat W$. 
To obtain a better approximation of the solution, we represent $ \om, \eta,  \zeta$ using a 
mesh $Y \times  Y$ with $Y$ refining the mesh $y$ \eqref{eq:ASS_mesh} in Appendix \ref{app:solu_rep} by a factor of three:
\[
Y_{3i + j} = y_i + (y_{i+1} - y_i) j/3 , \quad 0 \leq j \leq 3.
\]
Since solving the Poisson equation is the main computational cost in each time step, we still represent $\phi^N$ using the coarse mesh $y \times y$ and solve it from source term with grid points value $ \om(y_i, y_j)$.

In the temporal variable, we use a third order Runge-Kutta method to update the PDE. To reduce the round-off error near $x=0$, where we require a very small error in solving the linear PDE, we use a multi-level representation. We refer more details to Section {\secASS} in Part I \cite{ChenHou2023a}. To keep the residual error smooth near $x=0$, we apply a weak numerical filter near $x=0$ every three steps. We do not add the semi-analytic part in constructing $(\om_k, \eta_k, \xi_k)$ for efficiency consideration and that the far-field behavior of the solutions is changing over time.

After we obtain the numerical solution $(\om_k, \eta_k, \xi_k, \phi_{k, 1}^N)$ at discrete time, 
we will perform two rank-one corrections and interpolate the solution in time using a cubic polynomial to obtain the approximate space time solution $\hat W$, and estimate residual error 
in the energy space {\textit{a-posteriori}}.

\subsection{A posteriori error estimates: decomposition of errors}\label{sec:deomp_err}

Since we cannot solve the Poisson equation exactly, we decompose  the stream function $\bar \phi, \phi$ as follows 
\beq\label{eq:elli_err0}
\bar \phi = (-\D)^{-1}\bar \om = \bar \phi^N + \bar \phi^e,  \quad  \phi =(-\D)^{-1} \om = \phi^N + \phi^e,
\eeq
where $\bar \phi^N, \phi^N$ constructed using finite element method are the numeric approximation of the stream function, and the short hands $N, e$ denote \textit{numeric, error}, respectively. We use similar notations below for other nonlocal terms since we cannot construct them exactly.  We will construct $ \bar \phi^N , \phi^N$ numerically and treat $\bar \phi^e, \phi^e$ as error. The reader should not confuse $\phi^N$ with the $N$-th power of $\phi$. We will never use power of $\phi$ throughout the paper. Similarly, we denote by $\uu^N, \uu^e$ the velocities corresponding to $ \phi^N, \phi^e$. For example, we have 
\beq\label{eq:u_Ne}
\uu^N =  \na^{\perp}  \phi^N, \quad \uu^e = \na^{\perp} \phi^e
= \na^{\perp} (-\D)^{-1}( \om - (-\D) \phi^N),  \quad c_{\om}^N = u^N_x(0), 
\quad c_{\om}^e = u^e_x(0).
\eeq

The above decomposition leads to the following decomposition of the operator $\cL$
\beq\label{eq:decomp_L}
\bal
\cL_1   &= \cL_1^N+ \cL_1^{ e} +  \cL_1^{\bar e},  \quad \cL_{2}  = \cL_2^N  + \cL_2^e  +  \cL_2^{ \bar e}, \quad \cL_3   = \cL_3^N + \cL_3^e + \cL_3^{ \bar e}, \\
  \cL_1^N &= \eta + \bar c_{\om}^N \om  - (\bar c_l x + \bar \uu^N ) \cdot \na \om
+ c_{\om}^N   \bar \om- \uu^N  \cdot \na \bar \om
, \\
\cL_1^e  & = c_{\om}^e \bar \om -  \uu^e \cdot \na \bar \om, \quad
\cL_1^{\bar e}  = \bar c_{\om}^e \om - \bar \uu^e \cdot \na \om, \\
 \cL_2^N  & =  - (\bar c_l x +\bar \uu^N ) \cdot \na \eta + (2 \bar c^N_{\om} - \bar u^N_x) \eta  - \bar v^N_x \xi
 - \uu_x^N  \cdot \na \bar \th   - \uu^N \cdot \na \bar \th_x  + 2 c_{\om}^N  \bar  \th_x ,\\
\cL_2^e & = - \uu_x^e  \cdot \na \bar \th  - \uu^e  \cdot \na \bar \th_x  + 2 c^e_{\om}  \bar \th_x  , \quad 
\cL_2^{\bar e}  =  - \bar \uu^e \cdot \na \eta + (2 \bar c^e_{\om} - \bar u^e_x) \eta 
- \bar v^e_x \xi ,
\\
 \cL_3^N  & =  - (\bar c_l x +\bar \uu^N ) \cdot \na \xi + (2 \bar c^N_{\om} - \bar v^N_y) \xi  - \bar u^N_y \eta
 - \uu_y^N   \cdot \na \bar \th   - \uu^N   \cdot \na \bar \th_y + 2 c_{\om}^N  \bar \th_y ,\\
\cL_3^e & = - \uu_y^e   \cdot \na \bar \th  - \uu^e   \cdot \na \bar \th_y  + 2 c^e_{\om}  \bar \th_y , \quad 
\cL_3^{\bar e }
= - \bar \uu^e \cdot \na \xi + (2 \bar c^e_{\om} - \bar v^e_y ) \xi - \bar u^e_y \eta,
\eal
\eeq
where $\cL_i^e, \cL_i^{\bar e}$ denote the errors from $\phi^e, \bar \phi^{e}$, respectively. These operators depend on $\om, \eta, \xi$, and we drop the dependence in \eqref{eq:decomp_L} to simplify the notations.

\subsection{First correction and the construction of $\phi^N$}\label{sec:lin_evo_1stcor}

According to the normalization condition and \eqref{eq:lin_evo_order}, the solution to \eqref{eq:lin_evo} satisfies $\om_x(0, t) = \eta_x(0, t) =0$. To obtain an approximate solution with this condition, we make the first correction 
\beq\label{eq:lin_evo_1stcor}
\om_k \to \om_k - \om_{k, x}(0,0) \chi_{11} ,  \quad \eta_k  \to \eta_k - \eta_{k, x}(0, 0) \chi_{21} ,
\eeq
where $\chi_{i1}$ are cutoff functions defined in \eqref{eq:solu_cor} with $\chi_{i1} = x + O(|x|^4)$
for $i=1,2$ near $0$. We do not modify $\xi_k$ since $\xi_k$ already vanishes quadratically near $(0,0)$. 
We remark that the first correction does not change the second order derivatives of the solution near $0$ and $c_{\om}$ since 
\[
\pa_{xy} \chi_{11} (0) = \pa_{xy} \chi_{21}(0) = 0, \quad c_{\om}(\chi_{11})
= -\pa_{xy}\phi_1(0) = 0.
\]
where $\phi_1$ is defined below
\beq\label{eq:solu_cor1}
\phi_1 = -\f{x y^2}{2} \kp_*(x) \kp_*(y),
\eeq
where $\kp_*(x)$ is the cutoff function chosen in \eqref{eq:cutoff_near0} in Appendix \ref{app:cutoff_near0} satisfying $\kp_*(x) = 1 + O(|x|^4)$ near $x=0$, and $\phi_1$ satisfies $-\D \phi_1 = x + O(|x|^4)$. For the numeric stream function $\phi^N_{k, 1}$ constructed at the beginning of this Section \ref{sec:lin_evo}, we correct it as follows 
\[
\phi_{k, 1}^N \to \phi_{k, 1}^N + \pa_{x} \Delta \phi_{k,1}^N( 0 ) \phi_1 \teq \phi_k^N. 
\]
Since $\pa_x \D \phi_1(0) = -1$, this allows us to obtain 
\beq\label{eq:psi_err_2nd}
\bal
& \pa_x (-\D) \phi_k^N (0)= - \pa_x  \D \phi_{k, 1}^N (0) + \pa_{x} \Delta \phi_{k,1}^N(0) =0, \\
& \D \phi_k^N = O(|x|^2), \quad  \om_k - (-\D) \phi_k^N = O( |x|^2).
\eal
\eeq

We further extend it to Lipschitz continuous solutions $\wh W^{(1)} \teq ( \hat \om^{(1)}(t), \hat \eta^{(1)}(t), \hat \xi^{(1)}(t))$ and $\hat \phi^{N, (1)}$ in time using a cubic polynomial interpolation in $t$. See section \ref{sec:cubic_time} for more details. Here, we use $\hat f^{(1)}$ to denote the solution with the first correction.

\subsection{The second correction}\label{sec:lin_evo_2ndcor}
The error 
\[
( \pa_t -\cL_i ) ( \hat \om^{(1)}(t) ,  \hat \eta^{(1)}(t), \hat \xi^{(1)}(t) )
\]
may not vanish to the order $O(|x|^3)$, which is a property 
that we require in the energy estimate. Then we add the second correction 
\[
( \hat \om^{(1)}(t) ,   \hat \eta^{(1)},  \hat \xi^{(1)}(t),  \hat \phi^{N, (1)})   \to 
( \hat \om^{(1)}(t) + a_1(t) \chi_{12},   \hat \eta^{(1)} + a_2(t) \chi_{22},  \hat \xi^{(1)}(t) + a_3(t) \chi_{32},  \hat \phi^{N, (1)} + a_1(t) \phi_2),
\]
so that the error satisfies
\beq\label{eq:lin_evo_cor2}
\e^{(2)}_i \teq ( \pa_t -\cL_i ) (\hat \om^{(1)}(t) + a_1(t) \chi_{12}, \hat \eta^{(1)}(t) + a_2(t)\chi_{22}, \hat \xi^{(1)}(t) + a_3(t) \chi_{32} ) = O(|x|^3)
\eeq
near $x = 0$.
We use the following functions for these two corrections 
\beq\label{eq:solu_cor}
\bal
\chi_{11} &= - \D \phi_{1}, \quad \phi_1 = -\f{x y^2}{2} \kp_*(x) \kp_*(y) , \quad
  \chi_{21} = x \kp_*(x)\kp_*(y),  \\
\chi_{12}  & = - \D \phi_{2}, \quad \phi_2 = -\f{x y^3}{ 6} \kp_*(x) \kp_*(y) , 
 \quad    \chi_{22} = xy \kp_*(x ) \kp_*(y), \quad \chi_{32} = \f{x^2}{2}  \kp_*(x) \kp_*(y),
 \eal
\eeq
where $\kp_*(x)$ is chosen in \eqref{eq:cutoff_near0}, $\chi_{\cdot, 1}$ is used for the first correction, and $\chi_{\cdot, 2}$ for the second correction. We do not have $\chi_{31}$ since we do not need the first correction for $\xi$ \eqref{eq:lin_evo_1stcor}. Since $\kp_*(x) $ satisfies $ \kp_*(x) = 1 + O(|x|^4)$ near $x =0$, the behaviors of the above functions near $x =0$ are given by 
\[
 \chi_{11} = x + l.o.t., \  \chi_{21} = x + l.o.t., \ \chi_{12} = xy + l.o.t., \ \chi_{22} = xy + l.o.t., \ \chi_{32} = x^2 / 2 + l.o.t.
\]
We choose $\chi_{1j} = -\D \phi_j$ for the correction of $\om$ so that its associated velocity $\na^{\perp} (-\D)^{-1} \chi_{1j}$ can be obtained explicitly. We do not need such form for the correction of $\eta, \xi$ since we do not compute the velocity of $\eta, \xi$.

For cutoff functions $\chi_1, \chi_2 , \chi_3$ with 
\beq\label{eq:solu_cor2}
c_{\om}(\chi_1) = -\pa_{xy}(-\D)^{-1} \chi_1 = 0,
\eeq
e.g. $\chi_i = \chi_{i2}$ chosen above, we have the following formulas of $\cL_i( a_1(t) \chi_1, a_2(t) \chi_2, a_3(t) \chi_3) $ \eqref{eq:lin_evo}
\[
\bal
 \cL_1( a_1  \chi_1, a_2 \chi_2 , a_3 \chi_3)
&= a_1(t)\B( - ( \bar c_l x + \bar \uu) \cdot \na \chi_1 + \bar c_{\om} \chi_1 - \uu(\chi_1) \cdot \na \bar \om  \B)+ a_2(t) \chi_{2} , \\
  \cL_2( a_1  \chi_1, a_2 \chi_2 , a_3 \chi_3)
&= a_2( t) \B( - (\bar c_l x + \bar \uu) \cdot \na  \chi_2 + (2 \bar c_{\om} - \bar u_x) \chi_2\B)  
- a_3(t) \bar v_x \chi_3 - a_1( t)  \B( \uu(\chi_1) \cdot \na \bar \th \B)_x , \\
  \cL_3( a_1  \chi_1, a_2 \chi_2 , a_3 \chi_3)
&= a_3( t) \B( - (\bar c_l x + \bar \uu) \cdot \na  \chi_3 + (2 \bar c_{\om} + \bar u_x ) \chi_3\B)  
- a_2(t) \bar u_y \chi_2 - a_1( t)  \B( \uu(\chi_1) \cdot \na \bar \th \B)_y , \\
\eal
\]
where $\uu (\chi_1)$ is the velocity associated with $\chi_1$. We want to apply the above formulas to the second corrections $\chi_{i2}, i=1,2,3$ in \eqref{eq:solu_cor}. We use the Hadamard product  
\beq\label{eq:hada}
 (A \circ B)_i = A_i B_i ,
\eeq
and \eqref{eq:decomp_L} to simplify the notation as follows
\beq\label{eq:decomp_cor}
\bal
& \cL_i( a \circ \chi) = Cor_{ij}(x; \chi)  a_j(t) , \quad 
Cor_{ij}(x; \chi) = Cor_{ij}^N(x; \chi) + Cor^{\bar e }_{ij}(x; \chi) , \\
&\cL_i^N(a \circ \chi ) \teq Cor_{ij}^N(x; \chi) a_j(t),
\quad   \cL_i^{\bar e}(a \circ \chi ) \teq Cor^{\bar e }_{ij}(x; \chi) a_j(t).
\eal
\eeq

Note that $\cL_i^e( a \circ \chi) = 0$ since we can obtain $\uu(\chi_1)$ explicitly for $\chi_1 = \chi_{11}, \chi_{12}$ \eqref{eq:solu_cor}.

Next, we derive the equations for $a_i(t),i=1,2,3$. Using \eqref{eq:lin_evo} and the condition 
\[
\pa_{xy} \e_1^{(2)}(0) = \pa_{xy} \e_2^{(2)}(0) = \pa_{xx} \e_3^{(2)}(0) = 0,
\]
from \eqref{eq:lin_evo_cor2}, we obtain the following ODEs for $a_i(t)$:
\beq\label{eq:ODE_A}
\bal
\dot a_1(t)  & = ( -2 \bar c_l + \bar c_{\om}) a_1(t) + a_2(t) - F_1(t) , \\
\dot a_2(t)  & = ( -2 \bar c_l + 2 \bar c_{\om} - \bar u_x(0)  ) a_2(t) - F_2(t) , \\
\dot a_3(t)  & = ( -2 \bar c_l + 2 \bar c_{\om} - \bar u_x(0)) a_3(t) - F_3(t) , \\
\eal
\eeq
where $F(t) = (F_1(t), F_2(t), F_3(t))^T$ is the error associated to the second order derivatives of $(\pa_t - \cL) \hat W^{(1)}$ near $0$. More precisely, we have 
\beq\label{eq:ODE_err}
\bal
F_1(t) & = \pa_{xy} ( \pa_t - \cL_1) \wh W^{(1)} (0) = \f{d}{dt} \hat\om^{(1)}_{xy}(t, 0) - ( -2  \bar c_l + \bar c_{\om}) \hat \om^{(1)}_{xy}(t, 0) -  \hat \eta^{(1)}_{xy}(t, 0) 
- c_{\om}(t) \bar \om_{xy}(0) , \\
F_2(t) & = \pa_{xy} ( \pa_t - \cL_2) \wh W^{(1)} (0)  =  \f{d}{dt} \hat \eta^{(1)}_{xy}(t, 0) - ( -2  \bar c_l + 2 \bar c_{\om} -\bar u_x(0)) \hat \eta^{(1)}_{xy}(t, 0) -  c_{\om}(t) \bar \th_{xxy}(0) ,\\
F_3(t) & =  \pa_x^2 ( \pa_t - \cL_3) \wh W^{(1)} (0)
= \f{d}{dt} \hat \xi^{(1)}_{xx}(t, 0) -  ( -2  \bar c_l + 2 \bar c_{\om} -\bar u_x(0)) \hat \xi^{(1)}_{xx}(t, 0) -  c_{\om}(t) \bar \th_{xxy}(0) .
\eal
\eeq
Denote 
\beq\label{eq:diff_op}
D^2 = (D^2_1, D^2_2, D^2_3) = (\pa_{xy}, \pa_{xy}, \pa_x^2)^T.
\eeq
Then we can simplify \eqref{eq:ODE_err} as 
\beq\label{eq:ODE_err_dec}
F_i = D_i^2 ( \pa_t - \cL_i) \hat W^{(1)}(0)
= D_i^2 ( \pa_t - \cL_i^N - \cL_i^e - \cL_i^{\bar e}) \hat W^{(1)}(0) .
\eeq

Denote by $M$ the coefficients in \eqref{eq:ODE_A}
\beq\label{eq:decomp_M}
M = \lt( 
\begin{array}{ccc}
 -2 \bar c_l + \bar c_{\om}  & 1  &  0  \\
 0 &   -2 \bar c_l + 2 \bar c_{\om} - \bar u_x(0)   & 0 \\
 0 &  0 &  -2 \bar c_l + 2 \bar c_{\om} - \bar u_x(0)  . \\
\end{array}
\rt) \teq M^N +  M^{\bar e},
\eeq
where the last identity is based on the decomposition $\bar c_{\om} = \bar c_{\om}^N + \bar c_{\om}^e, \bar u_x(0) = \bar u_x^N(0) + \bar u_x^e(0)$, and $M^{\bar e}$ only contains the contribution from $ \bar c_{\om}^e, \bar u^{\bar \e}_x(0)$. According to the normalization condition \eqref{eq:normal_pertb}, we have $\bar u_x(0)^e = \bar c_{\om}^e$. It follows 
\beq\label{eq:decomp_Me}
M^{\bar e} =  \bar c_{\om}^e   I_3.
\eeq

We simplify the ODE for $a = (a_1, a_2,a_3)^T$ as 
\beq\label{eq:lin_ODE_a}
\dot a_i(t) = M_{ij} a_j(t) - F_i(t), \quad \dot a(t) = M a - F
= M a - e_i D_i^2 (\pa_t - \cL_i ) \hat W^{(1)} (0) .
\eeq

Recall $\chi_{\cdot 2} = (\chi_{12}, \chi_{22}, \chi_{32})$ from \eqref{eq:solu_cor}. 
In the $i-th$ equation, the overall error for the approximate solution $\wh W^{(1)} + a(t) \circ \chi_{\cdot 2} $ is 
\beq\label{eq:lin_evo_err1}
\bal
& (\pa_t - \cL_i ) (\wh W^{(1)} + a(t) \circ \chi_{\cdot 2}) 
=  (\pa_t - \cL_i^N) (  a(t) \circ \chi_{\cdot 2})  + \B( (\pa_t - \cL_i^N) \wh W^{(1)} \\ 
 & - \cL_i^e ( \wh W^{(1)} + a(t) \circ \chi_{\cdot 2}) 
 - \cL_i^{\bar e} ( \wh W^{(1)} + a(t) \circ \chi_{\cdot 2}) \B) \teq  J + I. 
\eal
\eeq

Note that in the above notation, $\pa_t$ acts on $ a_i(t) \chi_{i, 2}$. For $J$, using the ODE for $a(t)$ \eqref{eq:lin_ODE_a}, \eqref{eq:decomp_cor}, \eqref{eq:ODE_err_dec}, and \eqref{eq:decomp_M}, 
we get
\[
\bal
J & = (M_{ij} a_j - F_i) \chi_{ i 2} - Cor^N_{ij}(x; \chi_{\cdot 2}) a_j  \\
 & = ( M^N_{ij}  \chi_{i2} -  Cor^N_{ij}(x; \chi_{\cdot 2})  ) a_j
 + M^{\bar e}_{ij} a_j \chi_{i2} 
 - D_i^2 (\pa_t - \cL_i^N  - \cL_i^e - \cL_i^{\bar e}) \wh W^{(1)}(0)  \chi_{i 2} 
 \teq J_1 + J_2 + J_3 ,
 \eal
\]
where we have summation \emph{only} over $j=1,2,3$. Since $\cL^e( a(t) \circ \chi_{\cdot 2}) = 0$, using  the above decomposition and combining $I, J_2, J_3$  , we yield 
\beq\label{eq:lin_evo_err3}
\bal
I + J_2 + J_3 
 &= \B( (\pa_t - \cL_i^N) \wh W^{(1)} - D^2_i( \pa_t -\cL_i^N) \wh W^{(1)}(0) \chi_{i2}   \B)
-  \B( \cL_i^e  \wh W^{(1)} - D_i^2 \cL_i^e \wh W^{(1)}(0) \chi_{i2} \B)  \\
& \quad  - \B( \cL_i^{\bar e}(   \wh W^{(1)} + a(t) \circ \chi_{\cdot 2}) 
- D_i^2 \cL_i^{\bar e} \wh W^{(1)}(0) \chi_{i2}  - M_{ij}^{\bar e} a_j \chi_{i2} \B)
 \teq I_{i, N} + I_{i, e} + I_{i, \bar e}. 
\eal
\eeq

Next, we check that $J_1, I_{i, N}, I_{i, e}, I_{i, \bar e}$ have a vanishing order $O(|x|^3)$. This is clear for $I_{i, N}, I_{i, e} $. Since we correct the second order derivatives and $\hat \om^{(1)}, \hat \eta^{(1)}, \hat \zeta^{(1)}$ are odd with $\hat \xi^{(1)} = x \hat \zeta^{(1)} $, we get $ \pa_x^i \pa_y^j I_{i, N}$, $\pa_x^i \pa_y^j I_{i, e} = 0 ,i+j \leq 2$ at the origin. 
For $J_1$, we note that it is a linear combination of $a_j$ with given coefficients $ M_{ij}^N \chi_{i 2} - Cor_{ij}^N$. Its cubic vanishing order follows from the definition. For example, 
for the summand in $J_1$ with $i=j=1$ and a minus sign, 
using definition \eqref{eq:decomp_L}, \eqref{eq:decomp_cor} for $Cor^N_{ij}$ 
with $c_{\om}(\chi_{12}) = 0$, and \eqref{eq:decomp_M} for $M_{ij}^N$, we have 
\[ 
\bal
S & =   a_1(t) \cdot ( Cor^{N }_{11}(x) - M^{ N }_{11} \chi_{12} ) 
=  a_1(t)\B( - ( \bar c_l x + \bar \uu^N) \cdot \na \chi_{12} + \bar c_{\om}^N \chi_{12} 
- \uu^N( \chi_{12}) \cdot \na \bar \om
- (- 2 \bar c_l + \bar c_{\om}^N) \chi_{12} \B) \\
& = a_1(t)\B(  \bar c_l ( 2 \chi_{12} - x \cdot \na \chi_{12} ) - \bar \uu^N \cdot \na \chi_{12} -\uu^N(\chi_{12}) \cdot \na \bar \om \B)
\eal
\]
Since $\chi_{12} = xy + O(|x|^4), \uu^N(\chi_{12})= \na^{\perp} \phi_2 = O(|x|^3)$ by \eqref{eq:solu_cor} 
near $0$ 
 (velocity $\uu^N(\chi_{12}) = \uu( \chi_{12})$ is derived analytically and has $0$ error), 
$\bar u^{ N} = \bar u_x^N (0) x + O(|x|^2), \bar v^N = - \bar u_x^N(0) y + O(|x|^2)$ near $0$, we have $S = O(|x|^3)$ near $0$.
The vanishing order of other terms in $J_1$ are obtained similarly. For $J_1$, we estimate the weighted norm for $ Cor^{N}_{ij}(x) - M^{N}_{ij} \chi_{i2} $  and then apply the triangle inequality to further bound $J_1$. Similarly, for a fixed $i$, we have the following vanishing 
\[
\cL_i^{\bar e}( a(t) \circ \chi_{\cdot 2}) - M_{ij}^{\bar e} a_j \chi_{i2}
 = Cor_{ij}^{\bar e} a_j(t) -  M_{ij}^{\bar e} a_j \chi_{i2} = O( |x|^3),
 \quad D_i^2\cL_i^{\bar e}( a(t) \circ \chi_{\cdot 2})(0) = M_{ij}^{\bar e} a_j.
\]

Thus, we can rewrite $I_{i, \bar e}$ as follows 
\beq\label{eq:lin_evo_err1_Ie}
I_{i, \bar e} = - \B( \cL_i^{\bar e}(   \wh W^{(1)} + a(t) \circ \chi_{\cdot 2}) 
- D_i^2 \cL_i^{\bar e} ( \wh W^{(1)} + a(t) \circ \chi_{\cdot 2}) (0) \chi_{i2}  \B),
\eeq
which clearly has a cubic vanishing order. Note that $\wh W^{(1)} + a(t) \circ \chi_{\cdot 2}$ is our final approximate solution for solving \eqref{eq:lin_evo}.

In summary, to estimate the error $(\pa_t - \cL)( \wh W^{(1)} + a \circ \chi_{\cdot 2} )$, we will 
estimate $J_1, I_{i, N}, I_{i,e}, I_{i, \bar e}$ separately. 
The term $I_{i, N}$ is the local error of solving \eqref{eq:lin_evo} numerically, $I_{i,e}, I_{i, \bar e}$ are due to the error of solving the Poisson equations for $\wh \om^{(1)}$ and $\bar \om$. Since we use a cubic polynomial interpolation to obtain the continuous function $\hat W^{(1)}(t)$, the errors
$I_{i, N}, I_{i, e}$ are piecewise cubic polynomials in time, and we track the coefficients of these polynomials to verify that they are small. We discuss the estimate of nonlocal error in Section \ref{sec:vel_err}.

\subsection{Cubic interpolation in time}\label{sec:cubic_time}

Given the numerical solution with the first correction $ \wh W_n^{(1)} = ( \hat \om_n^{(1)}, \hat \eta_n^{(1)}, \hat \xi_n^{(1)})$, we use a piecewise cubic interpolation to construct $ \wh W^{(1)}(t, x)$ over $(t, x) \in [0, T] \times \R_2^+$. We partition the whole time interval $[0, T]$ into small subintervals $[ 3mk, 3(m+1) k]$ with length $3k$. For $s \in [- 3k / 2, 3k /2]$ and $t_m = 3 mk$, we construct 
\[
\bal
\hat W^{(1)}( s  + t_m  + \f{3k}{2})
&= \f{1}{16} (- W_0 + 9 W_1 + 9 W_2 - W_3)
+ \f{1}{24} ( W_0 - 27 W_1 + 27 W_2 - W_3 )  \f{s}{k} \\
&\quad + \f{1}{4} (W_0 - W_1  - W_2 + W_3) ( \f{s}{k}  )^2
+ \f{1}{6}( - W_0 + 3 W_1 - 3 W_2 + W_3) ( \f{s}{k}  )^3 \\
& \teq  \sum\nolimits_{ 0\leq i \leq 3} C_i \cdot V \f{1}{i!} (\f{s}{k})^i , \quad V = (W_0, W_1, W_2, W_3), 
\eal
\]
where $k$ is the time step, $W_i = \hat W_{3 m+i}^{(1)}$ for $ t_m = 3 m k$, and $C_i \in \R^4$ is the coefficient determined by the interpolation formula.  A direct calculation yields 
\[
\bal
\pa_t \wh W^{(1)} - \cL \wh W^{(1)} &= \sum\nolimits_{ 1\leq i \leq 3 }  \f{ C_i \cdot V}{k} \f{1}{ (i-1)!} (\f{s}{k})^{i-1}
- \sum\nolimits_{0\leq i\leq 3} \cL( C_i \cdot V)  \f{1}{i!} (\f{s}{k})^i \\
& = \sum\nolimits_{0\leq i \leq 2} \B( \f{C_{i+1} \cdot V}{k} - \cL( C_i \cdot V) \B)  \f{1}{i!} (\f{s}{k})^i 
- \cL( C_3 \cdot V) \f{s^3}{6 k^3}.
\eal
\]
To estimate $\pa_t \wh W^{(1)} - \cL \wh W^{(1)}$, we will use the triangle inequality and estimate $ \f{C_{i+1} \cdot V}{k} - \cL( C_i \cdot V), \cL( C_3 \cdot V) $ rigorously using the methods in Section \ref{sec:err_idea}, \ref{sec:vel_err}.

Applying the triangle inequality and integrating the error over $s \in [ - \f{3k}{2}, \f{3k}{2}]$ yield 
\beq\label{eq:cubic_int}
\bal
& \int_{|s| \leq 3k/2} |\pa_t \hat W^{(1)} - \cL \hat W^{(1)}| ds 
\leq 
\sum_{ i \leq 2} \B| \f{C_{i+1} \cdot V}{k} - \cL( C_i \cdot V) \B| \int_{|s| \leq 3k/2} \f{1}{i!} |\f{s}{k}  |^i  \\
&  + |\cL( C_3 \cdot V) | \int_{ |s| \leq \f{3k}{2} } \f{1}{6} | \f{s}{k}|^3 
= k \B(  \sum_{ i \leq 2} \B| \f{C_{i+1} \cdot V}{k} - \cL( C_i \cdot V) \B| 
C_I(i) + |\cL( C_3 \cdot V) | C_I(3) \B),
\eal
\eeq
where 
\[
C_I = [3, \f{9}{4}, \f{9}{8},  \f{27}{64}  ].
\]

\subsubsection{Decomposing the time interval for parallel computing }

To verify that the posteriori error is small, we need to estimate the error rigorously 
at each time step, which takes a significant amount of time. Consider a partition of the time interval $0 = T_0 < T_1 < .. < T_n = T$, where $T$ is the final time of the computation.
To reduce the computational time, we first solve the equations on $[0, T]$ without any rigorous verification and save the solution $(\om_k, \eta_k, \xi_k, \phi_{k, 1}^N)$ at $t_k = T_i$. Since we do not need to perform verification at this step, the running time for each time step is short. Then we solve the equations on a smaller time interval $[T_i, T_{i+1}], i=0, 1,2..., n-1$ using $W(T_i)$ as the initial data and then perform the verification in each time interval in parallel. At the end of each time interval $[T_i, T_{i+1}]$, we use the pre-computed data $W(T_{i+1})$, which is the same as the initial data for next time interval $[T_{i+1}, T_{i+2}]$ for verification. This guarantees that we use the same discrete solution  $(\om_k, \eta_k, \xi_k, \phi_{k, 1}^N)$ for verification in $[T_i, T_{i+1}]$ and $[T_{i+1}, T_{i+2}]$.

\subsection{Compactly supported in time}\label{sec:stop}

To construct an approximate solution, we do not need to solve the linearized equations \eqref{eq:lin_evo} for all time. In fact, since the solution decays in certain norm as $t$ increases, we stop the computation at time $T$ if $ \hat W^{(1)} - D^2 \hat W^{(1)} \circ \chi_{\cdot, 2}$ is small in the energy norm. Then we extend $\hat W^{(1)}(t, \cdot) $ trivially for $t > T$
\[
  \wh W^{(1)}(t, \cdot) = 0, \quad t > T. 
\]

As a result, the error in $i$-th equation satisfies 
\[
  \cR_i =  ( (\pa_t - \cL) \wh W^{(1)}  )_{, i}
  = (  (\pa_t  - \cL ) \wh W^{(1)} )_{, i} \one_{t \leq T} - \d_{T}(t) \wh W^{(1)}_i(T). 
\]
Let  $F = (F_1, F_2, F_3), F_i  = D_i^2 (\pa_t - \cL_i) \wh W^{(1)} \B|_{x=0}$ for $ t \leq T$, where $D^2 = (D_{xy}, D_{xy}, D_x^2)$. Then similarly, we get 
\[
  F_{ext} \teq D^2 (\pa_t - \cL) \wh W^{(1)} |_{x=0} \cdot \one_{t  \leq T} - D^2 \wh W^{(1)}(T, 0) \d_T   
  = F(t) \one_{t \leq T} - F_{end}(T) \d_T, \  F_{end}(T) \teq D^2 \wh W^{(1)}(T, 0).
\]

We will test the above formulas with some Lipschitz function in time and the above formulas are well defined. Recall that the coefficients of the second correction $a$ satisfy \eqref{eq:lin_ODE_a}. Although $\hat W$ only has finite support in time, to achieve the vanishing order \eqref{eq:lin_evo_cor2} for all time, we need to solve the ODE {\textit exactly} for all time. If we stop solving the ODE at time $T$, we cannot achieve \eqref{eq:lin_evo_cor2} at time $T$. 
Moreover, we cannot solve the ODE using a numerical method, e.g. the Runge-Kutta method, since it leads to an error. Instead, we solve the ODE exactly by diagonalizing the system. We introduce the following notations 
\beq\label{eq:ODE_modi_F}
\bal
&\lam_1 = -2 \bar c_l + \bar c_{\om}, \quad \lam_2 = \lam_3 = -2 \bar c_l + 2 \bar c_{\om} - \bar u_x(0), \quad \lam_1 - \lam_2 = - \bar c_l / 2, \\
&\td a_1 = a_1 + \f{a_2}{ \lam_1 -\lam_2}, \quad \td F_1 = F_1 + \f{F_2}{\lam_1  - \lam_2}, \quad \td a_i = a_i, \quad \td F_i = F_i, \ i = 2, 3,
\eal
\eeq
and similar notations for $\td F_{ext}$, where we have used \eqref{eq:normal} to get $\lam_1 - \lam_2 = -\bar c_l /2$. The coefficients satisfy $ \lam_1 \approx -7, \lam_2 = \lam_3 \approx -5.5$. We diagonalize \eqref{eq:ODE_A} as follows 
\[
\f{d}{dt} \td a_i = \lam_i \td a_i - \td F_{ext,i}.
\]

Using Duhamel's formula and the definition of $\td F_{ext, i}$, we yield 
\beq\label{eq:A_ode}
\bal
 \td a_j(t) & = e^{ \lam_j t } \td a_j(0) - \int_0^t e^{ \lam_j(t-s)}  \td F_{ext, j}(s) ds  \\
 & =  e^{\lam_j t} \td a_j(0) - \int_0^{ t \weg T } e^{\lam_j(t-s)} \td F_j(s) ds
 + \td F_{end, j}(T) e^{\lam_j(t-T) } \one_{t \geq T}  
 \teq S_1 + S_2 + S_3.
 \eal
\eeq

For rank one perturbation, the full solution $\hat W$ with two corrections in \eqref{eq:lin_evo_main1},\eqref{eq:lin_evo_err_def}  is given by 
\beq\label{eq:solu_full}
\wh W = \wh W^{(1)} + a \circ \chi_{\cdot 2},
\quad \hat \phi^{ N} =  \hat \phi^{N, (1)} + a_1(t) \phi_2 ,
\eeq
where $\chi_{\cdot 2}, \phi_2$ are defined in \eqref{eq:solu_cor}.
With the above extension and the decomposition of error \eqref{eq:lin_evo_err1}-\eqref{eq:lin_evo_err3}, the residual error for rank-one perturbation 
\eqref{eq:lin_evo_err_def} with $n=1$ is given by 
\beq\label{eq:lin_evo_err2}
\bal
  & \cR  = c(t) ( \wh W_0^{(1)} + a_0 \circ \chi_{\cdot 2} - \bar W_0  ) + \int_0^t c(t-s) ( \pa_t - \cL )( \wh W^{(1)} + a \circ \chi_{\cdot 2} ) ds = \cR_{loc, 0, \cdot} + \cR_{nloc}  \\
  & \cR_{loc, 0, \cdot}  = c(t) ( \wh W_0^{(1)} + a_0 \circ \chi_{\cdot 2} - \bar W_0  ) 
    - ( \wh W^{(1)}(T )- D^2 \wh W^{(1)}(T) \circ \chi_{\cdot 2} ) c(t - T) \one_{t \geq T} \\
       &  +   \int_0^{t \weg T} c(t- s) \big( \sum_{ i\leq 3} 
e_i  I_{i, N}(s)  \big)  ds     + \int_0^t c(t - s)\big( \sum_{i\leq 3} e_i  J_{1, i}(s) \big)     ds 
 = \int_0^t c(t- s) \cR_{num}(s) ds, 
\\    
  & \cR_{nloc}  =    \int_0^{t \weg T} c(t- s) \sum_{  i\leq 3} 
e_i  I_{i, e}(s)   ds   + \int_0^t c(t - s) \sum_{i\leq 3} e_i  I_{i, \bar e}(s)  ds, \\
\eal
\eeq
where $I_{i, N}, I_{i, e}, I_{i, \bar e}$ are given in \eqref{eq:lin_evo_err3}, $J_{1,i}$ means $J_1$ \eqref{eq:lin_evo_err3} in  $i$-th equation, and $\cR_{num} \in \R^3$ is 
\beq\label{eq:lin_evo_num}
\bal
 \cR_{num}(s) \teq \d_0 \cdot ( \wh W^{(1)}_0 + a_0 \circ \chi_{\cdot 2} - \bar W_0  ) 
 - \d_T \cdot ( \wh W^{(1)}(T )- D^2 \wh W^{(1)}(T) \circ \chi_{\cdot 2} )
 + \tts{\sum}_{j\leq 3} e_j( \one_{s \leq T} I_{j, N}  + J_{1, j}  ).
    \eal
    \eeq
We only integrate the integrals for $e_i I_{i, N}, e_i I_{i, e}$ up to $\min(t, T)$ since these two integrands \eqref{eq:lin_evo_err3} do not involve $a_i(t)$ and have compact support $[0, T]$ in time. 
We obtain the local part $\cR_{loc, 0,\cdot}$ in \eqref{eq:lin_evo_main2} for $n=1$.  
The first term is the initial interpolation error for $\bar W_0$, and we choose $a_0 \in \R^3$ to achieve vanishing order $ \wh W^{(1)}_0 + a_0 \circ \chi_{\cdot, 2} -\bar W_0 = O(|x|^3)$. We use $\cR_{loc, 0, \cdot}, \cR_{nloc}$ to denote the error that depends on the solution locally and nonlocally. 
We use the bootstrap assumption to obtain uniform control of $c(t)$ in $t$. See Section {\secWtwo} in Part I \cite{ChenHou2023a}.
The error estimate of the local part $\cR_{loc, 0, \cdot}$ follows Section \ref{sec:err_idea}. 
Moreover, we extract the essentially local part from $\cR_{nloc}$ and can estimate it with $\cR_{loc, 0 ,j}$ together \eqref{eq:resid_loc}. 
We decompose the nonlocal part $\cR_{nloc}$ in Section \ref{sec:vel_err}. 
To control the terms involving $a_i$, e.g. $J_{1, i}$ above \eqref{eq:lin_evo_err3}, 
we can estimate the weighted norm of the functions $ Cor^N_{ij}(x) - M^N_{ij} \chi_{i2}$ and then
 only need to estimate the integral of $\td a_j$. 

 Denote $x\weg y \teq \min(x, y)$. Since the factor $\lam_j < 0$, using the formula of $\td a_j$ \eqref{eq:A_ode}, we obtain 
\[
\bal
 & \int_0^{\inf} |S_1| dt = \f{1}{|\lam_j|} |\td a_j(0)|, \quad 
 \int_0^{\inf} |S_3(t)| dt = \int_T^{\inf} | \td F_{end,j}(T)| e^{\lam_j(t- T)} dt 
 = \f{1}{|\lam_j|} |\td F_{end,j}(T)|, \\
 & \int_0^{\inf} |S_2(t)| dt 
 \leq \int_0^{\inf} ( \int_0^{t\weg T} e^{\lam_j(t-s)} |\td F_j(s) | ds ) dt
 =  \int_0^T |\td F_j(s)| ( \int_s^{\inf} e^{\lam_j(t-s)} dt ) ds 
 = \f{1}{| \lam_j|} \int_0^T | \td F_j(s)|,
\eal
\]
with $j=1,2,3$. It follows 
\[
\int_0^{\inf} |\td a_j(t)| d t \leq \f{1}{ |\lam_j|} \B( |\td a_j(0)| 
+ \int_0^T |\td F_j(s)| ds + |\td F_{end,j}(T)| \B).
\]

 In the estimate of integral of $\td F_j$, 
\eqref{eq:ODE_err}, \eqref{eq:ODE_modi_F}, we use $\bar c_{\om} = \bar c_{\om}^N + \bar c_{\om}^e,c_{\om} = c_{\om}^N + c_{\om}^e$ \eqref{eq:u_Ne} and track the terms involving  $\bar c_{\om}^N, c_{\om}^N$ in $I_{e1}$ and error separately, 
\[
\bal
F_{j,e1} &= \int_0^T | F_j^N(t) | dt , \
F_{j,e2} =\int_0^T |\bar c_{\om}^e D_j^2 \wh W^{(1)}(t, 0)| dt, \
F_{j,e3} =  \int_0^T |c_{\om}^e(t, 0) D_j^2 \bar W(0)| dt , \\
F_j^N  &= D_j^2 ( \pa_t - \cL_j^N) \wh W^{(1)}(0),  
\quad D^2 = (\pa_{xy}, \pa_{xy}, \pa_{xx}), \quad j =1,2, 3.
\eal
\] 
 From \eqref{eq:u_Ne}, we get $2 \bar c_{\om}^{e} - \bar u^e_x(0) = \bar u^e_x(0)$ and only $1$ unit of error $I_{e2}$ in $F_j(t), j=2,3$.  We track $\td F_j$ \eqref{eq:ODE_modi_F} similarly. 
Since $\wh W^{(1)},  F, \td F, F^N, \td F^N$ \eqref{eq:ODE_err} are cubic in time, we can estimate the above integrals  following \eqref{eq:cubic_int}. Note that $|\td F_{end,j}(T)|$ does not involve the nonlocal error. Using the linear relation between $a_j, \td a_j$, we can estimate $a_j$. 

Using the above estimates, we can represent the rank-one solution and estimate it as follows 
\beq\label{eq:W2_est1}
\bal
\hat G(t, x) & = \int_0^t c(t-s)  \hat W (s) ds , 
\quad 
\hat W = \hat W^{(1)} + a \circ \chi_{\cdot 2} 
 \\
  |\pa_x^i \pa_y^j \hat G_l(t, x)| & \leq \sup_{ t>0 } |c(t)|  \B( \int_0^T |\pa_x^i \pa_y^j \hat W^{(1)}_l(t) | dt 
 + |\pa_x^i \pa_y^j \chi_{ l 2}| \int_0^{\inf} |a_l(t)| dt \B) .
 \eal
\eeq

Similarly, we can bound other quantities for $\hat G$ 
and complete the estimates in \eqref{eq:lin_evo_main1}.

We generalize the above formula and estimate directly to the finite rank perturbation operator using linearity. For different initial data $\bar W_0$ related to the finite rank perturbation, we choose a different stopping time  $T(\wh W^{(1)}_0)$ to save computation cost. In practice, we construct the numerical solution up to time $T( \wh W^{(1)}_0)\leq T =12$. At that time, the solution $\hat W^{(1)}(T)$ is very small, which can be treated as a small perturbation. See figures in Section 4.3 in Part I \cite{ChenHou2023a}.

\begin{remark}\label{rem:W2_para}
Using linearity and the triangle inequality, we can assemble the estimates for $\cR$ \eqref{eq:lin_evo_err_def} from the estimates of each mode $\hat W_i$ in \eqref{eq:lin_evo_main1}, \eqref{eq:lin_evo_err_def}.
In practice, this means that we can implement the above estimate for each individual mode completely in parallel. 
\end{remark}

\paragraph{\bf{Finite support of the $c_{\om}$ term in time}}

In Section 5 of Part I \cite{ChenHou2023a}, we need to use $c_{\om}(f )$, where $c_{\om}(f) = u_x(f)(0) = -\pa_{xy}(-\D)^{-1} f(0)$. Since we choose the cutoff function $\chi_{12}$ for the second correction of $\hat \om$ with properties \eqref{eq:solu_cor}, \eqref{eq:solu_cor2}, we get
\[
c_{\om}( \hat W_1^{(1)} + a_1(t) \chi_{12}  )
= c_{\om}( \hat W_1^{(1)}),
\]
and it is supported in $[0, T]$.

\subsection{Ideas of estimating the norm of the error}\label{sec:err_idea}

In this section, we discuss how to estimate the error derived in the previous section, e.g.  
$I_{i, N}$ \eqref{eq:lin_evo_err3}, a-posteriori. The general idea is to first evaluate $f$ on some grid points and estimate the higher order derivatives of $f$ in a domain $D$. Then we can construct an approximation $\hat f$ of $f$ by interpolating the values of $f$ at different points. The approximation error $ f - \hat f $ can be bounded by $ C_k || f||_{C^k} h^k$, where $h$ measures the size of the domain. If the mesh $h$ is sufficiently small, the error term is small.
See a simple second order error estimate in \eqref{est_2nd}.

To develop an efficient method for rigorous estimates, we have the following considerations. Firstly, we should evaluate as a small number of points as possible so that the method is efficient. Secondly, most functions $f$ in the verification are complicated, e.g. $I_{i, N}$ \eqref{eq:lin_evo_err3}, and it is difficult to obtain the sharp bound of the higher derivatives. Instead, we first estimate the piecewise derivatives of some simple functions, e.g. piecewise polynomials $(\hat \om, \hat \eta)$ or semi-analytic solutions following Appendix \ref{app:solu}, \ref{app:explcit}. Then we use the triangle inequality and the Leibniz rule to estimate the products of these simple functions, and their linear combinations.
Yet, in general, this approach overestimates the derivatives significantly. To compensate the overestimates, we use higher order interpolations and estimates with error bounds $C h^k, k = 3, 4, 5$, which provide the small factor $h^k$. We develop three estimates based on different interpolations: the Newton interpolation, the Lagrangian interpolation, and the Hermite interpolation in Section {\applinfestsupp} in the supplementary material II \cite{ChenHou2023bSupp} (attached to this paper). The 1D interpolating polynomials are standard, and we generalize them to construct 2D interpolating polynomials.

We want to estimate the constant $C$ in the error bound $C h^k$ as sharp as possible to reduce the computational cost and improve the efficiency. In fact, when $k=4$, if we can obtain an interpolation method and reduce the constant $C$ to $ \f{C}{16}$, to achieve the same level of error, we can increase $h$ to $2h$. In this verification step, since the domain is 2D, it means that we can evaluate only $\f{1}{4}$ of the grid point values of $f$, which 
can reduce the computational cost by $75\%$.

Using the above method, we can obtain a sharp estimate of the derivatives of $f$. Using the method in Section {\applinfestsupp} in the supplementary material II \cite{ChenHou2023bSupp} and Taylor expansion, we can further estimate the weighted norm of $f$ with a singular weight near $0$. 
We discuss the estimate of the nonlocal error in Section \ref{sec:vel_err}. Using these $L^{\inf}$ estimates of $f  $ and its derivatives, we can further develop H\"older estimate for $f$. See Section \ref{app:hol_func1}. We remark that the numerical solutions are regular, e.g. the approximate steady state and the solutions to the linearized equations are $C^{4, 1}$. We use these methods to estimate piecewise $L^{\inf}(\vp_{evo,i})$ norm of the local residual error $\cR_{num, i}$ \eqref{eq:lin_evo_num} and the $C_{x_i}^{1/2}$ partial H\"older seminorm of $ \cR_{num,i} \psi_i $, where $\vp_{evo, i}, \psi_i$ are defined in \eqref{wg:lin_evo}.

We remark that the weights $\vp_{evo,i}$ and $\vp_i, i=2,3$ in the $L^{\inf}$ energy estimate 
(see Section {\secEE} in \cite{ChenHou2023a}) for $\eta, \xi$ are similar but with different coefficient $ p_{5, \cdot}, p_{6, \cdot}$. Since $\vp_i$ and $\vp_{evo,i}$ are equivalent, after we obtain the piecewise weighted $L^{\inf}(\vp_{evo,i})$ estimate of the error, we can obtain piecewise weighted $L^{\inf}(\vp_i)$ estimate by estimating the ratio $ \vp_i /\vp_{evo, i}$. Similarly, we can obtain weighted $L^{\inf}(\vp_{g, i})$ estimate of the error, where $\vp_{g, i}$ is another weight in the energy estimate in Section {\secEE} in \cite{ChenHou2023a}.  

\vs{0.1in}

\paragraph{\bf{Estimate the local part of the residual error}}
Using the above methods, we can estimate the local part of the residual error  $\bar \cF_i$ for the approximate steady state and discuss the estimate in Appendix \ref{sec:resid}.
We further extract the local part of $\cR_{nloc}$ \eqref{eq:lin_evo_err2}, which has the form \eqref{eq:lin_evo_main2} obtained in Section \ref{sec:vel_err}, and combine it with $\cR_{loc, 0, j}$ to get the essentially local residual error 
\beq\label{eq:resid_loc}
\bal
\cR_{loc, i} & = \cR_{loc, 0, i} + \cR_{dif, i}  + M_i  , \quad 
\cR_{dif, i} \teq  D_i^2 \cB_{op, i}( \uu(\bar \e ), \hat G)(0)\cdot
( \chi_{i2} - f_{\chi, i}  )  ,  \\
M_i & \teq  \cB_{op, i}( \uu(\hat \e), \bar W ) -  D_i^2 \cB_{op, i}( \uu(\hat \e), \bar W )(0) \chi_{i 2}
- \cB_{op,i}( \uu_A(\hat \e_1), (\na \uu)_A(\hat \e_1), \bar W).
\eal
\eeq
where $\chi_{i2}$ is defined in \eqref{eq:solu_cor}. By definition \eqref{eq:Blin_L} and following derivation of \eqref{eq:ODE_err_dec}, we get 
\[
  D_i^2 \cB_{op, i}( \uu(\bar \e ), \hat G)(0) = u_x(\bar \e)(0) V_i, \quad
  V =(  \hat G_{1,xy}(0), \hat G_{2, xy}(0), \hat G_{3,xx}(0) ). 
\]
To estimate each term, we follow Section \ref{sec:err_idea} and Appendix \ref{sec:resid}. We perform the decomposition \eqref{eq:uerr_dec1}  $\uu(\hat \e ) = \uu_A(\hat \e_1) + \hat{\uu}(\hat \e_1) + \uu(\hat \e_2)$ and similar decomposition for $\na \uu(\hat \e)$, with $(\bar \e, \chi_{\bar \e})$ in \eqref{eq:uerr_dec1}  replaced by $(\hat \e, \chi_{\hat \e})$, where $\chi_{\hat \e} $ is defined in \eqref{eq:cutoff_near0_all}. Using linearity of $B_{op, i}$, we get 
\[
\bal
& \cB_{op, i}( \uu(\hat \e), \bar W ) 
- \cB_{op,i}( \uu_A(\hat \e_1), (\na \uu)_A(\hat \e_1), \bar W)
= II_i(\hat \e_1) + II_i(\hat \e_2), \\
&  II_i(\hat \e_1) = \cB_{op, i}( \hat \uu(\hat \e_1), \wh{ \na \uu}(\hat \e_1), \bar W), \   II_i(\hat \e_2) =  \cB_{op, i}( \uu(\hat \e_2),  \na \uu (\hat \e_2), \bar W)  .
\eal
\]
We have $\uu_A= O(|x|^3), (\na \uu)_A = O(|x|^2)$ near $0$, which implies 
$\cB_{op,i}( \uu_A(\hat \e_1), (\na \uu)_A(\hat \e_1), \bar W) = O(|x|^3)$ \eqref{eq:Blin_gen} and 
\[
M_i = II_i(\hat \e_1) + II_i(\hat \e_2) -  D_i^2 \cB_{op, i}( \uu(\hat \e), \bar W )(0) \chi_{i 2} =  II_i(\hat \e_1) + II_i(\hat \e_2) - D_i^2 (  II_i(\hat \e_1) + II_i(\hat \e_2)  )(0)  \chi_{i 2} .
\]
The term $I_{i,N}$  in $\cR_{loc, 0, i}$ \eqref{eq:lin_evo_err3},\eqref{eq:lin_evo_err2} is similar to $II_i^N$, and $M_i$ has a similar form as  $II_{i}(\bar \e_1) + II_i( \bar \e_2)$ in Appendix \ref{sec:resid}. We have done the above decomposition for $\uu(\bar \e)$ in \eqref{eq:uerr_dec1}, \eqref{eq:bous_errM}, and refer therein for more details.
Then the estimate of $ \cR_{loc,  j}$ is similar to that in Appendix \ref{sec:resid}. See Section {\seccombvelerr} in \cite{ChenHou2023a} for more discussion of the above forms.

\vs{0.1in}
\paragraph{\bf{Error for the initial data and at stopping time}}
The error $\wh W^{(1)}(T )- D^2 \wh W^{(1)}(T) \circ \chi_{\cdot 2}$ at the stopping time has compact support and its estimate follows the methods in Section \ref{sec:err_idea}. 
To bound the initial interpolation error $ err_{in} \teq \wh W^{(1)}_0 + a_0 \circ \chi_{\cdot 2} - \bar W_0$ \eqref{eq:lin_evo_err2} in a large domain, we follow similar methods. The error involves $ \bar \om, \bar \th$ which are supported globally. 
To bound $err_{in}$ in the middle and far-field, since $\hat W^{(1)}_0 + a_0 \circ \chi_{\cdot, 2} = 0$, combining all the initial data from the finite rank perturbation (see Appendix C.2.1 of Part I \cite{ChenHou2023a}), we need to estimate 
\begin{eqnarray*}
 &&I_1 =  c_{\om}(\om_1) \bar \om - \hat \uu(\om_1) \cdot \na \bar \om , \quad  I_2 =  2 c_{\om}(\om_1)  \bar \th_x 
   - \hat \uu(\om_1) \cdot  \na \bar \th_x - \hat \uu_{x }(\om_1) \cdot \na \bar \th, \\ 
&&I_3 = 2 c_{\om}(\om_1)  \bar \th_y
   - \hat \uu(\om_1) \cdot  \na \bar \th_y - \hat \uu_{y  }(\om_1) \cdot \na \bar \th ,
\end{eqnarray*}
for large $|x|$. The approximation terms near $0$ defined in Section 4.2.1 of Part I \cite{ChenHou2023a} are supported near $0$ and decay to zero as $|x| \rightarrow \infty$. In the far-field, $\hat \uu(\om_1)$ is only a rank-one term. We estimate the above terms using \eqref{eq:err_mid_far}, \eqref{eq:err_mid_far2}  with $a = c_{\om}(\om_1)$ and the estimates in Section \ref{sec:resid}.

\subsection{Posteriori error estimates of the velocity}\label{sec:vel_err}
In this section, we show that the nonlocal error in \eqref{eq:lin_evo_err2} has the desired forms in \eqref{eq:lin_evo_main2}. Then we combine the estimate of such terms with the nonlinear energy estimate in Section {\seccombvelerr} in \cite{ChenHou2023a}. Using \eqref{eq:Blin} and the definition of $\cL^{\bar e}, \cL^e$ \eqref{eq:u_Ne}, \eqref{eq:decomp_L}, we have 
\beq\label{eq:Blin_L}
\cL_j^{\bar e}( G) =\cB_{op, j}( \uu(\bar \e), G), \quad 
\cL_j^{ e}( G) =\cB_{op, j}( \uu( G_1 - (-\D) \phi_G^N) , \bar W ),
\eeq
for a vector-valued function $G$, where $\phi_G^N$ is the numerical stream function associated with $G_1$.

Given $c_i(t)$ Lipschitz in $t$ and initial data $\bar W_i, i=1,2.., n$,  we construct $\hat W_i(t)$ following previous sections and $\hat G$ using \eqref{eq:lin_evo_main1}.
Using the derivations in \eqref{eq:lin_evo_err2}, \eqref{eq:lin_evo_err3}, \eqref{eq:lin_evo_err1_Ie} and the above relation, the  contribution from the error type $I_{i,j, \bar e}$ term to the error \eqref{eq:lin_evo_err_def} in the $j$-th equation for $i$-th mode  is the following
\[
\bal
 \cR_{j}^{\bar e } 
\teq \sum_{i\leq n} \int c_i(t- s) I_{i, j ,\bar e}(s) d s 
=  - ( \cR_{j0}^{\bar e } - D_j^2 \cR_{j0}^{\bar e}(0) \chi_{j 2} ),  
\quad \cR_{j0}^{\bar e } \teq  \sum_{i \leq n} \int_0^t c_i(s) \cB_{op, j}( \uu(\bar \e), \hat W_i(t-s)) d s .
\eal
\]
where $i$ denotes the $i$-th mode and $j$ denotes the $j$-equation.  Since $\cB_{op, j} $  is bilinear and $c_i(t)$ is spatial-independent and Lipschitz in $t$, we get 
\beq\label{eq:Blin_L2}
\cR_{j0}^{\bar e} = \cB_{op, j}\big( \uu(\bar \e),   \,
\sum_{i \leq n} \int_0^t  c_i(t-s)   \hat W_i(s) ds \big)
= \cB_{op, j} ( \uu(\bar \e), \wh  G(t)).
\eeq

Denote by $ \wh G^{(l)}, \hat W_i^{(l)}\in \R^3$ the approximate solution with extension in $t$ in Section \ref{sec:cubic_time}, and the first correction $l=1$ in Section \ref{sec:lin_evo_1stcor} or two corrections $l=2$ in Section \ref{sec:lin_evo_1stcor}, \ref{sec:lin_evo_2ndcor}. Let $\hat \phi_i^{N, (l)}$ be the stream function associated with  $\hat W_i^{(l)}$ constructed numerically with the first correction for $l=1$ and both corrections for $l=2$. 
In particular, the full solution is given by $ \hat G = \hat G^{(2)} , \hat W_i = \hat W_i^{(2)},  \hat \phi_i^N = \hat \phi_i^{N, (2)}$ \eqref{eq:solu_full}. 
We construct the stream function $\hat \phi^{N, (l)}$ associated with $\hat G^{(l)} \in \R^3$ and error $\hat \e$ as follows 
\[
\bal
\hat \phi^{N, (l)} \teq   \sum_{i \leq n } \int_0^t c_i(s)   \hat \phi^{N, (l)}_i ( t-s) ds, 
\quad 
\hat \e = \hat G_1^{(1)} + \D  \hat \phi^{N, (1) }
= \sum_{i \leq n } \int_0^t c_i( s) ( \hat W_{i,1}^{(1)}  + \D  \hat \phi_i^{N, (1)} ) ( t-s) ds .
 \eal
\]
Since we can obtain $\uu(a(t) \chi_{1 2})$ exactly for the second correction (see Section \ref{sec:lin_evo_2ndcor}), we have
\beq\label{eq:Blin_L3}
\hat \e(t) =  \hat G_1^{(1)} - (-\D) \hat \phi^{N, (1)} = \hat G_1^{(2)} - (-\D) \hat \phi^{N, (2)}
= \hat G_1 - (-\D) \hat \phi^{N}.
\eeq

In practice, we estimate $\hat \e$ using the first identity since it does not involve $a_i(t)$ and the integrand $ \hat W_{i, 1}^{(1)} + \D \hat \phi_i^{N, (1)}$ is piecewise cubic in time. 
 We decompose $\hat \e$ as follows
\beq\label{eq:uerr_hat1}
\hat \e_2 =  \hat  \e_{xy}(0) \D( \f{x y^3}{6}  \chi_{  \hat \e}), \quad 
\hat \e = (\hat \e -  \hat \e_2 ) +  \hat \e_2  \teq \hat \e_1 + \hat \e_2 ,
\eeq
where $ \chi_{\hat \e}$ is defined in \eqref{eq:cutoff_near0_all}.
Since $\hat \e$ only vanishes $O(|x|^2)$ near $0$, we perform the above decomposition so that $\hat \e_1 =O(|x|^3)$ near $0$. See Appendix \ref{sec:resid} and Section {\seccombvelerr} in \cite{ChenHou2023a} for motivations of \eqref{eq:uerr_hat1}.
We estimate $\hat \e_1, \hat \e_{xy}(0), \hat \phi^N$ following \eqref{eq:W2_est1}.
We establish \eqref{eq:lin_evo_main1}.

Similarly, using linearity, \eqref{eq:Blin_L}, and \eqref{eq:Blin_L3}, we can rewrite the residual error in \eqref{eq:lin_evo_err2} from $I_{j,  e}$ term in \eqref{eq:lin_evo_err3} related to $\cL_j^{ e}( \cdot) $
as follows 
\[ 
\bal
  \cR_{j}^{  e } 
 & \teq \sum\nolimits_{i\leq n} \int c_i(t -s) I_{i, j, e}(s) d s 
 = - (\cR_{j0}^{  e } - D_j^2 \cR_{j0}^{  e}(0) \chi_{j 2} ),  \\
 \cR_{j0}^{  e }  &= \sum\nolimits_{i\leq n} \int_0^t c_i(t-s) \cL_j^e( \hat W^{(1)}_i)(s) d  s 
=  \sum\nolimits_{i\leq n} \int_0^t c_i(t-s) \cB_{op, j}( \uu( \hat W_{i,1}^{(1)} + \D \hat \phi_i^{N, (1)} )(s), \bar W ) d  s \\
&  = \cB_{op, j}( \uu( \hat \e(t) ), \bar W ), 
\eal
\]
which along with \eqref{eq:Blin_L}-\eqref{eq:Blin_L2} for $\cR_j^{\bar e}$ establish the formula for $\cR_{nloc}$ in \eqref{eq:lin_evo_main2}.

\section{Estimate the norm of the regular part of the velocity}\label{sec:vel_comp}

In this section, we derive the constants in the upper bound in Lemma \ref{lem:main_vel}. We have constructed the finite rank approximation $\hat f$ for $f$ in Lemma \ref{lem:main_vel} in Section {\secapprvel} in Part I \cite{ChenHou2023a}. The estimate of the most singular part, e.g. $u_{x, a, b}( \om \psi)$, in the $C^{1/2}$ estimate in Lemma \ref{lem:main_vel} can be obtained using the sharp H\"older estimates in Section {\secsharp} of Part I \cite{ChenHou2023a}, where $u_{x, a, b}$ is defined via a  localized kernel. 
In this section, we estimate other terms in Lemma \ref{lem:main_vel}, e.g. $ I = \psi u_x(\om) - u_{x, a, b}( \om \psi) - \psi \hat u_x(\om) $, involving the velocity with desingularized kernels, which are more regular.

In Section \ref{sec:int_method}, we outline the strategies in the estimate and decompose the 
integrals from the nonlocal terms into several parts based on their regularities.
In Section \ref{sec:vel_linf}, we perform the $L^{\inf}$ estimates in Lemma \ref{lem:main_vel} and derive the constants. In Section \ref{sec:vel_hol_comp}, we perform the H\"older estimate of different parts. In Section \ref{sec:hol_comb}, we combine the H\"older estimate of different parts, which provide the constants in Lemma \ref{lem:main_vel}. In particular, we reduce the $L^{\inf}$ estimates and the $C^{1/2}$ estimates in Lemma \ref{lem:main_vel} to bounding some explicit $L^1$ integrals depending on the weights, which can be estimated by a numerical quadrature with rigorously error control. 
We estimate these integrals with computer assistance. See discussions in Section \ref{proof-sketch}.

We will apply the second estimates in Lemma \ref{lem:main_vel} for the nonlocal error, e.g. $\uu(\bar \e)$ and $\bar \e$ is the error of solving the Poisson equations. Since we can estimate piecewise bounds of $\bar \e$ following Section \ref{sec:err_idea}, instead of using global norm, we improve the estimate using the localized norms, which are much smaller than the global norm. See Section \ref{sec:vel_loc_est}.

The kernels associated with $\uu , \na \uu$ are given by
\beq\label{eq:kernel_du}
\bal
& K_1 \teq \f{y_1 y_2}{|y|^4}, \quad  K_2 \teq  \f{1}{2}\f{ y_1^2 - y_2^2}{ |y|^4}, \quad 
K_u  \teq  \f{  y_2}{2 |y|^2} , \quad  K_v \teq -\f{ y_1}{2 |y|^2}, \\
& K_{u_x} = - K_1, \quad K_{u_y} = K_{v_x} = K_2. 
\eal
\eeq
Here, we have dropped the constant $\f{1}{\pi}$, e.g. $ u_x(\om)= -\pa_{xy}(-\D)^{-1} = \f{1}{\pi} K_{u_x} \ast \om$. One needs to multiply $\f{1}{\pi}$ back to obtain the final estimate.

\vspace{0.1in}
\paragraph{\bf{Difficulties in the computations}}

In addition to the difficulties discussed in Section {\secvelPI} of Part I \cite{ChenHou2023a}, e.g. singularities caused by the weights and kernels, the singular integral introduces several technical difficulties in our estimates. To address these difficulties, we need to consider different scenarios and decompose the domain of the integrals carefully in our computer assisted estimates. 
Given $\om \vp  \in L^{\inf}$, the velocity $\uu$ and the commutator $\psi \cdot (\na \uu)(\om) - (\na \uu)(\om \psi)$  are only log-Lipschitz. The logarithm singularity introduces several difficulties. For example, if $ \uu$ is Lipschitz, a natural approach to estimate its H\"older norm in terms of $|| \om \vp||_{\inf}$ is to estimate the piecewise bound of $\uu$ and $\pa \uu$, which are local in $\uu$, and then use the method in Section \ref{app:hol_func1}. However, since $\uu$ is only log-Lipschitz, we need to perform a decomposition of $\uu$ into the regular part and the singular part carefully. For different parts, we will apply different estimates. See Section \ref{sec:int_loglip} for the ideas. For $\na \uu$, the estimates are more involved since it is more singular.

\subsection{Several strategies}\label{sec:int_method}

We outline several strategies to estimate the nonlocal terms.

\subsubsection{Integral with approximation}

In our computation of $\uu_A = \uu - \hat \uu, ( \na \uu)_A = \na \uu - \wh  {\na \uu}$, where the approximation terms $\hat \uu, \wh  {\na \uu}$ are defined in Section {\secapprvel} of Part I \cite{ChenHou2023a}, the rescaling argument still applies. Note that we do not have 
$\pa \uu_A = (\pa \uu)_A$ since we design approximations for $\uu, \na \uu$ separately. We consider one approximation term $  c(x) \int \one_{y \notin S}  K(x_a, y) \om(y)  dy$ for $ \int K(x, y) \om(y) d y$ to illustrate the ideas, where $S$ is the singular region associated with $x_a$. 
Suppose that $K$ is $-d$-homogeneous. We want to estimate 
\[
I = \rho(x) \int_{\R^2}  ( K(x, y) - c(x) K( x_a, y) \one_{y \notin S }  ) W(y) dy,
\]
where $W$ is the odd extension of $\om$ from $\R^2_+$ to $\R^2$ (see \eqref{eq:ext_w_odd}). Denote 
\beq\label{eq:flam}
f_{\lam}(x) \teq f(\lam x).
\eeq
We choose $\lam \asymp |x|$ and denote $x = \lam \hat x, y = \lam \hat y, x_a = \lam \hat x_a$. 
Since $K(\lam z) =\lam^{-d} K(z) $,  we have 
\beq\label{eq:scal_appr}
\bal
I  &= \rho_{\lam}(\hat x) \int_{\R^2} 
\B( K( \lam \hat x,  \lam \hat y ) - c(\lam \hat x) \one_{ \lam \hat y \notin S} K( \lam  \hat x_a,  \lam \hat y) \B) W( \lam \hat y) \lam^2 d \hat y \\
&= \lam^{2- d} \rho_{\lam}(\hat x) \int_{\R^2} 
\B( K(   \hat x,   \hat y ) - c(\lam \hat x) \one_{  \hat y \notin S / \lam} K(   \hat x_a ,   \hat y) \B) W_{\lam}(   \hat y)  d \hat y. 
\eal
\eeq
The singular region becomes $S / \lam$ and close to $x_a / \lam = \hat x_a$. For example, if $S = \{ y: \max_i |y_i - x_{a,i}| \leq a \}$, we have $S / \lam = \{  y: \max_i |y_i - \hat x_{a,i}| \leq a / \lam \}$. For the above integral,  we will symmetrize the kernel and then estimate it using the norms $ || W \vp ||_{\infty}$ and $[\om \psi]_{C_{x_i}^{1/2}}, i=1,2$ \eqref{int:scal_w}.

\vspace{0.1in}
\paragraph{\bf{The bulk and approximation }}
To take advantage of the scaling symmetry and overcome the singularity, in our computation for $x$ away from the origin and not too large, we choose several dyadic rescaling parameters $\lam = 2^i, i \in I$, e.g. $I= \{ -4, -3, .., 10 \}$. Then for any $x$ with $ \max(x_1, x_2) \in [2^i x_c, 2^{i+1 } x_c]$, we can choose  $\lam = 2^i$ so that the rescaled $\hat{x} = \f{x}{\lam}$ satisfies 
\beq\label{int:sing_loc}
\hat{x} \in
\begin{cases}
 [ x_c, 2 x_c] \times [0, 2 x_c] \teq \Om_1,  \quad \mathrm{ if } \  x_2  \leq x_1 , \\
  [0,   x_c] \times [x_c, 2 x_c] \teq \Om_2,  \quad \mathrm{ if } \ x_2  >  x_1 . \\
 \end{cases}
\eeq
We also choose $x_i$ ($(x_i, 0)$ is the singularity) and the size of the singular region $t_i$ for the approximation term defined in Section 4.3.2 of Part I \cite{ChenHou2023a} such that $x_i / \lam$ is on the grid point of the mesh and the boundary of the singular region $\{ y: |x_i - y_1| \vee | y_2|  \geq t_i / \lam\}$, which aligns with one of the edges of a mesh cell. For example, this can be done by choosing the following $y$ mesh in the near-field to discretize the $y$-integral, $x_i$, and $t_i$
\[
y_{1, i} = ih, \quad y_{2, i}= ih, \quad   x_i = 2^{n_i} h,  \quad t_i = 2^{m_i} h.
\]
Then when we discretize the rescaled integral in $y$, e.g. \eqref{eq:scal_appr}, the singular region is the union of several mesh cells.  For large $y$, it is away from the singularity $\hat x$. Then we can use an adaptive mesh in $y_1, y_2$ to discretize the integral. 

We remark that in \eqref{eq:scal_appr}, if $x_a \neq 0$ and $ x_a / \lam \asymp x_a / |x|$ is too large or too small, since $c(x)$ is supported near $x_a$,
$c( \lam \hat x) = c(x)$ will be $0$. This means that when we compute $\uu_A(x), ( \na \uu)_A$, if the coefficient of an approximation term with center $x_i$ and parameter $t_i$ is nonzero, e.g., $c(x) \neq 0$, then $\lam$ is comparable to $x_i$ when we rescale the integral by $\lam$.  Thus $ \hat x_i = x_i / \lam$ is on the grid. We also choose $t_i$ such that $t_i / \lam$ is a multiple of mesh size $h$ for $\lam$ comparable to $x_i$.

\begin{remark}
Using the scaling symmetry and rescaling the integral by dyadic scales, we can compute the integral for $x \in [0, D]^2 \bsh [0, d]^2$ with roughly $O( \log(D/ d))$ computational cost.
\end{remark}

\paragraph{\bf{The near-field and the far-field}}

Recall the notations from Section {\secapprvel} in Part I \cite{ChenHou2023a}
\beq\label{eq:u_appr_near0_coe}
\bal
& C_{u0} = x , \quad C_{v0} = -y, \quad C_{u_x 0} = 1, \quad C_{u_y 0} = C_{v_x 0} = 0,  \\
& C_{u_x}  = -(x^2 - y^2), \quad C_{v_x} = 2x y,  \quad C_{ u_y} =  2 xy,  \quad  C_u  = - ( \f{1}{3} x^3 - x  y^2 ),  \ C_v  = x^2 y - \f{1}{3} y^3, \\
 & K_{ux0} = - \f{4y_1 y_2}{|y|^4}, \quad   K_{00} = \f{24 y_1 y_2(y_1^2 - y_2^2)}{|y|^8}, \quad \cK_{00}(\om) \teq \f{1}{\pi} \int_{\R^2_{++}} K_{00}(y) \om(y) dy.
\eal
\eeq

If $x$ is sufficiently small, i.e. $\max ( x_1,x_2) < \min_{ i\in I} 2^i x_c$, we choose $\lam = \max(x_1, x_2) / x_c$ so that the rescaled $\hat{x} = \f{x}{\lam}$ is on the line $x_1 = x_c$ or $x_2 = x_c$. Assuming $\vp(x) \geq |x|^{-\b_1} |x|_1^{-\b_2}$, $\rho \sim |x|^{-\al}$ near $x=0$, and $K$ is $-d$-homogeneous, then we get
\beq\label{int:scal_asymp}
\bal
| \rho(x) \int_{\R_2^{++}} K(x, y) \om(y) dy | 
&\leq || \om_{\lam} \vp_{\lam}||_{L^{\infty}} \rho_{\lam}(\hat{x})
\int_{\R_2^{++}}  | K(\hat{x}, \hat{y}) | \vp_{\lam}(\hat{y})^{-1} \lam^{2-d} d\hat{y}  \\
& \leq || \om_{\lam} \vp_{\lam}||_{L^{\infty}}  \lam^{\b_1 + \b_2 + 2 - d} \rho_{\lam}(\hat{x})
\int_{\R_2^{++}}  | K(\hat{x}, \hat{y}) | |\hat{y}|^{\b_1} \hat{y}_1^{\b_2} d\hat{y}  .\\
\eal
\eeq

As $x \to 0$, $\lam \to 0$. The factor $\lam^{\b_1 + \b_2 + 2-d}$ absorbs the large factor $\lam^{-\al}$ in $\rho_{\lam}(\hat{x})$. In our estimate of $\uu_A, \na \uu_A$, we have $\b_1 + \b_2 =2.9$ for $\vp_1, \vp_{g, 1}$, 2.5 for $\vp_{elli}$ \eqref{wg:linf},  $(\al, d) = (2, 2)$ for $(\psi_1, \na \uu_A)$ \eqref{wg:hol}, $(\al,  d) = (3, 1)$ for $( \rho_{10}, \uu_A)$ \eqref{wg:linf}. We have $\b_1 + \b_2 + 2 - d  - 2  > 0$.

In general, the above integral may not be integrable due to the growing weight $|y|^{\b_1} y_1^{\b_2}$. For $\uu_A, \na \uu_A$ with small $x$, it takes the form (see Section 4.3 of Part I \cite{ChenHou2023a})
\beq\label{eq:int_near0}
 f(x) - C_{f0}(x) u_x(0) - C_f(x) \cK_{00} 
 = \int_{\R^2_{++}} ( K^{sym}_f  +  C_{f0}(x) \f{4}{\pi} \f{y_1 y_2}{|y|^4} - C_f(x) K_{00}(y) ) \om(y) dy,
\eeq
where $C_{f0}, C_f$ and $K_{00}$ are defined in \eqref{eq:u_appr_near0_coe}, 
and $f = u, v, u_x, v_x, u_y, v_y$. In particular, the associated kernel has a much faster decay rate $|y|^{-6}$, which will be shown in Appendix \ref{app:decay}.
Thus, the integral is integrable. 

Since $\lam = \max( x_1, x_2) / x_c$ is very small, $\rho_{\lam}(\hat{x})$ can be well approximated by the most singular power $ c \lam^{-\al}|x|^{-\al}$ for some $c>0$, 
which can be estimated effectively after factorizing out $\lam^{-\al}$.

Similarly, if $x$ is sufficiently large, i.e. $ \max(x_1, x_2) > \max_{i \in I} 2^{i+1} x_c$, we choose $\lam = \f{ \max(x_1, x_2)}{x_c }$ so that the rescaled $\hat{x} = x / \lam$ is on the line $ x_1 = x_c$ or $x_2 = x_c$. Since $\lam$ is sufficiently large, we can estimate the weight $\rho_{\lam}, \vp_{\lam}$ based on their asymptotic behavior.

\vs{0.1in}
\paragraph{\bf{Integral near $0$}}
We have an approximation  $ I = - C_{f0}(x) K_{ux0}(y) - C_2(x) K_{00}(y)$
\eqref{eq:u_appr_near0_coe} for $K_{f0}^{sym}(x, y)$ with some smooth coefficients $ C_2$ ($C_2$ may not be $C_f$). 
The term $C_{f0}(x) K_{ux0}(y)$ and $K_f$ are both $-d$ homogeneous, $d=1$ or $2$. 
Since $K_{ux0}, K_{00}$ are singular near $0$, after rescale the integral following \eqref{eq:scal_appr}, we decompose the symmetrized integral for $y$ near $0$ as follows 
\beq\label{eq:u_tof}
\bal
 & II  = \int_{\R_2^{++}} \B( K_f^{sym}( \hat x, \hat y  ) \lam^{2-d} - \lam^{2-d}  C_f(  \hat x)  K_{ux0}(\hat y) 
 - C_2(\lam \hat x ) K_{00}(\hat y) \lam^{-2} \B) \om(\lam \hat y)  d \hat y  \\
 & =\lam^{2-d} \B( \int_{\R_2^{++}} \B( K_f^{sym}( \hat x, \hat y  )  -   C_f(  \hat x) 
\one_{|\hat y|_{\inf} \geq k_{01} h} K_{ux0}(\hat y) 
 -  \lam^{-4 + d} C_2(\lam \hat x )  \one_{ |\hat y|_{\inf} \geq  k_{02} h } K_{00}(\hat y) 
\B) \om( \lam \hat y ) d \hat y  \\
 & \quad - 
  \int_{ \R_2^{++}}  \B(   C_f(  \hat x) 
\one_{|\hat y|_{\inf} \leq k_{01} h} K_{ux0}(\hat y) 
 -  \lam^{-4 + d} C_2(\lam \hat x )  \one_{ |\hat y|_{\inf} \leq  k_{02} h } K_{00}(\hat y) 
\B)  \om(\lam \hat y) d \hat y \B)
\eal
\eeq
for some small integers $k_{0i}$ with $k_{0i} h < |\hat x|_{\inf}/2$, e.g. $k_{01} = 4, k_{02} = 20$, where $|a|_{\inf} = \max(a_1, a_2)$ and $h$ is chosen in \eqref{int:para1}. We will estimate the first integral with regular integrand  near $\hat y = 0$ using the method in Section \ref{sec:T_rule}, and the last two integrals for $|\hat y|_{\inf} \leq k_{01} h, |\hat y|_{\inf} \leq k_{02} h$ analytically in Section \ref{sec:int_near}. We perform the above decomposition since  $  K_{00}(\hat y), K_{ux0}(\hat y) $ are too singular to estimate them numerically.

We apply the above decompositions to the integrals in both $L^{\inf}$ and $C^{1/2}$ estimates. We also apply the above decompositions to the approximation terms and estimate integral of $K_{ux0}$ separately near $y=0$.

\subsubsection{The scaling relations}\label{sec:scal_prop}

We discuss several scaling relations, which will be useful in later computation. For a $-d$-homogeneous kernel $K$, i.e., $K(\lam x) = \lam^{-d} K(x)$, we have 
\[
I(x) = \rho(x) \int K(x , y) \om( y) dy 
=  \rho_{\lam}(  \hat x ) \int K( \hat x, \hat y) \om_{\lam}( \hat y) \lam^{2-d}
d \hat y \teq  \lam^{2-d} I_{\lam}(\hat x),
\]
where $ x = \lam \hat x, y = \lam \hat y$. To compute the derivative of $I(x)$, using the chain rule, we have 
\[
\pa_{x_i} I( x) = \lam^{2-d} \f{ d \hat x_i}{ d x_i} \pa_{ \hat x_i} I_{\lam}( \hat x) 
= \lam^{1-d} \pa_{ \hat x_i} I_{\lam}( \hat x).
\]

For the $L^{\inf}$ part, clearly, we get $| I(x) | = | I_{\lam}(\hat x)|$. To compute the H\"older norm, we use the following relation $|x-z| = \lam | \hat x - \hat z|$ and 
\[
\bal
\f{ |I(x) - I(z)| }{ |x-z|^{1/2}}
&= \lam^{-1/2} \f{ |I_{\lam}(\hat x) - I_{\lam}( \hat z) | }{  |\hat x- \hat z|^{1/2} } . \\
\eal
\]

In particular, for $i=1,2$, we have 
\beq\label{int:scal_w}
|| \om_{\lam} \vp_{\lam} ||_{\inf} = || \om \vp||_{\inf}, \quad [ \om_{\lam} \psi_{\lam}]_{C_{x_i}^{1/2}} = \lam^{1/2} [ \om \psi]_{C_{x_i}^{1/2}} , 
\quad [f]_{C_{x_i}^{1/2}} := \sup_{y, z: y_i = z_i} \f{ | f(y)-f(z) |}{ |y-z|^{1/2}} .
\eeq
Using these scaling relations, we can perform the estimate in a rescaled domain with any $\lam>0$. 
\subsubsection{Mesh and the Trapezoidal rule}\label{sec:T_rule}

After rescaling the integral with suitable scaling factor $\lam$, we can restrict the rescaled singularity $\hat{x} \in [0, 2 x_c]^2 \bsh [0, x_c]^2$  (see \eqref{eq:scal_appr}, \eqref{int:sing_loc} ).

If a domain $Q$ is away from the singularity $\hat x$ of the kernel, applying \eqref{int:scal_w}, we get
\beq\label{int:L1_1}
\int_Q |K(\hat x, y)| |\om_{\lam}(y)| dy  
\leq || \om_{\lam} \vp_{\lam}||_{\inf} \int_Q | K(\hat x, y)| \vp_{\lam}^{-1}(y) dy 
= || \om \vp||_{\inf} \int_Q | K( \hat x, y)| \vp_{\lam}^{-1}(y) dy .
\eeq

Then, it suffices to estimate the integral of an explicit function $| K( \hat x, y)| \vp_{\lam}^{-1}(y)$. If in addition, the region $Q$ is small, e.g. $Q$ is the grid $[y_i, y_{i+1}] \times [y_j, y_{j+1}]$ introduced below, we further apply
\[
\int_Q |K(\hat x, y)| |\om_{\lam}(y)| dy  
\leq || \om \vp||_{\inf}  || \vp_{\lam}^{-1}||_{L^{\inf}(Q)}  \int_Q |K( \hat x, y)| dy.
\]
Since the domain $Q$ is small, the estimate is sharp.
We use the following method to estimate  $\int | K( \hat x_i, y)  | dy $ for a suitable kernel $K$ and $\hat x_i$ on the grid points.

We consider the estimate of the $L^1$ norm of some function $f$ in $\R_2^{++}$, e.g. $f = K(\hat x_i, y) $ mentioned above. To discretize the integral, we design uniform mesh in the domain $[0, b]^2$ covering $\Om_1$ and $\Om_2$ with mesh size $h $ and adaptive mesh in the larger domain $[0, D]^2$
\beq\label{eq:int_mesh_y}
0 = y_0 < y_1 < ..< y_n = D, \quad  y_i = ih, \  i \leq b / h.
\eeq
The finer mesh in the near field $[0,b]^2$ allows us to estimate the integral with higher accuracy. We choose sparser mesh in the far-field since $y$ is away from the singularity $\hat x$
and the kernel decays in $y$. We partition the integral as follows 
\beq\label{int:parti1}
\int_{\R_2^{++}} |f(y)| dy = \sum_{ 0\leq i, j \leq n-1 } \int_{ [y_i, y_{i+1}] \times [y_j, y_{j+1}]} |f(y)| dy + \int_{ y \notin D} |f(y)| dy. 
\eeq

We focus on how to estimate the first part for nonsingular $f$. In Section \ref{sec:int_beyond}, we estimate the integral beyond $[0, D]^2$ using the decay of the integral. We will discuss how to estimate the integral near the singularity of the kernel in a later subsection.

Denote $Q = [a, b] \times [d,c], h_1 = b-a, h_2 = d-c$. We use the Trapezoidal rule
\[
\int_{ [a, b] \times [c, d]} |f(y)| dy \leq T ( |f|, Q) + Err( f),
\]
where 
\[
T(f , Q) \teq  \f{(b-a)(d-c) }{4} ( f(a, c) + f(a,d) + f(b, c) + f(b,d)).
\]

The error estimate of the above Trapezoidal rule is not obvious due to the absolute sign. In fact, even if $f$ is smooth, $|f|$ is only Lipschitz near the zeros of $f$.  Since the set of zeros is hard to characterize and that $|f|$ can have low regularity, we do not pursue higher order quadrature rule. We have the following error estimate. 

\begin{lem}[Trapezoidal rule for the $L^1$ integral]\label{lem:T_rule}
 For $f \in C^2(Q)$, we have 
\[
\int_Q |f(y)| dy \leq T( |f| , Q) +  \f{|Q| }{12} ( h_1^2 || f_{xx}||_{L^{\infty}(Q)}  
+ h_2^2 | f_{yy}||_{L^{\infty}(Q)} ) .
\]

\end{lem}  

\begin{remark}
The above estimate shows that the Trapezoidal rule remains second order accurate from the above. 
In particular, this error estimate is comparable to the case without taking the absolute value. 
\end{remark}

\begin{proof}
Define the linear interpolation of $f$ in $Q$
\[
L(f) = \sum_{i=1}^4 \lam_i(x) f_i, \quad E(f) = f - L(f) ,
\]
where $\lam_i(x)$ is linear and satisfies $\sum \lam_i (x)= 1$ and $\lam_i(x) \geq 0$ for $x \in Q$. Using the triangle inequality, we obtain
\[
\int_Q |f| dy \leq \int_Q |E(f)| dy + \int_Q \lam_i(x) |f_i| dy 
= T( |f|, Q) +\int_Q |E(f)| dy.
\]
We have the standard error bound for linear interpolation $E(f)$
\beq\label{eq:lag_err_lin}
|E(f)| \leq \f{ || f_{xx}||_{L^{\infty}(Q) }} {2} |(x - a)(x -b )| +\f{ || f_{yy}||_{L^{\infty}(Q) }} {2}|(y-c )(y-d)|,
\eeq
which can be obtained by first applying interpolation in $x$ and then in $y$. It can also be established using the error estimate for the 2D Lagrangian interpolation with $k=2$ in Section {\applinfestsupp} in the supplementary material II \cite{ChenHou2023bSupp}.
Integrating the above estimate in $x,y$ and using $ \f{1}{2} \int_0^1 t(1-t) dt = \f{1}{12}$ conclude the proof.
\end{proof}

To estimate the integral $\int |K(x, y)|$ for all $\hat x \in \Om_1 , \Om_2$ \eqref{int:sing_loc}, we discretize $[0, 2a]^2$ using uniform mesh with mesh size $h_x = h / 2$. We use the above method to estimate $\int | K( \hat x_i, y) | dy $ for $x_i$ on the grid points. After we estimate the derivatives of the kernel, we use the following Lemma to estimate the integral for any $x$ in a domain.

\begin{lem}\label{lem:int_err_x}
Suppose that $K(x, y) \in C^2( P \times Q )$, $P = [a_1, b_1] \times [a_2, b_2], h_i = b_i - a_i, i=1,2$, and $Q = [a, b] \times [c, d]$.
Let $L(K)(x, y)= \sum_{i, j=1,2 } \lam_{ij}(x) K( (a_i, b_j), y)$ be the linear interpolation of $K(x, y)$ in $x$ using $K( (a_i, b_j), y  ), i, j =1,2$. Then for any $x \in P$, we have 
\[
\bal
\int_Q |K(x, y)| dy
\leq \sum_{i, j = 1, 2} \lam_{ij}(x) \int_{Q} | K( (a_i, b_j), y)| dy 
+ \B( \f{ h_1^2}{8}  || K_{xx} ||_{L^{\infty}(P \times  Q)} 
 + \f{ h_2^2}{8}   || K_{yy}||_{L^{\infty}(P \times Q)}  \B) |Q| .\\
\eal
\]

\end{lem}

The proof follows from \eqref{eq:lag_err_lin}, the triangle inequality and $\f{1}{2}|t(1-t)| \leq \f{1}{8}$ for $ t\in[0, 1]$. We will apply the above Lemma and sum $Q$ over all the near-field domains
$Q_{ij} = [y_i, y_{i+1}] \times [y_j, y_{j+1}]$ \eqref{eq:int_mesh_y}. Since $\sum_{ij} \lam_{ij}(x) =1$, we can simplify the first term as follows 
\[
\sum_{i, j = 1, 2} \lam_{ij}(x) \sum_{ k, l \leq n } \int_{Q_{k l}} | K( (a_i, b_j), y)| dy 
\leq \max_{1\leq i, j \leq 2} \sum_{ k, l \leq n } \int_{ Q_{k l} } | K( (a_i, b_j), y)| dy  .
\]

Therefore, it suffices to estimate the integral for $x$ on the grid points and the piecewise derivative bounds of the kernel.

We apply Lemmas \ref{lem:T_rule}, \ref{lem:int_err_x} to estimate the weighted integral related to the velocity. The integrands take the form \eqref{int:sym_integ1},\eqref{int:sym_integ2}, \eqref{int:decom_ux}. To estimate the error in the above integrals, we need to obtain piecewise $L^{\inf}$ estimate of the derivatives of the integrands in $P, Q$. We estimate the derivatives of the weights in Appendix \ref{app:wg} and the kernel in Appendix \ref{app:ker}.

\vspace{0.1in}

\paragraph{\bf{Parameters for the integrals}}
In our computation, we choose 
\beq\label{int:para1}
h_x = 13 \cdot 2^{-12}, \quad h = 13 \cdot 2^{-11}, \quad x_c = 13 \cdot 2^{-5},
\eeq
which can be represented exactly in a binary system, to reduce the round off error. The approximate values of the above parameters are $h_x \approx 0.0032, h \approx 0.0064, x_c \approx 0.4$.
For $x \in [0, 2 x_c]^2 \bsh [0, x_c ]^2$ \eqref{int:sing_loc}, we have 
\beq\label{int:sing_loc_x}
 \max(x_1, x_2) \geq x_c = 64 h = 128 h_x .
\eeq
In our decomposition of the integral, e.g. \eqref{int:decom_ux}, \eqref{int:decom_linf_ux},
\eqref{int:decom_linf_u}, we impose a constraint on the size of the singular region to satisfy $(k+1) h <  x_c$ such that the region does not cover the origin.

\subsubsection{Decomposition, commutators and the Lipschitz norm}\label{sec:comp_ux}

The most difficult part of the computation is to estimate the H\"older norm of $\na \uu$, and we discuss several strategies. 
In this computation, we cannot first estimate the local Lipschitz norm of $\na u$ and then obtain the local H\"older norm due to the difficulties discussed at the beginning of Section \ref{sec:vel_comp}. We need to decompose the integral related to $\na u$ into several parts according to the distance between $y$ and the singularity and use different estimates for different parts.

We focus on the integral related to $u_x$ without subtracting any approximation term and assume that $x \in [0, 2 x_c]^2  \bsh [0, x_c]^2$. The approximation term $\wh{ \na \uu}$ is nonsingular and can be estimated using the method in Section \ref{sec:T_rule}. Let $h$ be the mesh size in the discretization of the integral in $y$. Suppose that \beq\label{eq:int_loc_x}
x \in \R_2^{++}, \quad x_2 \leq x_1, \quad  x \in  B_{i_1, j_1}(h_x) 
\subset B_{ij}(h) ,  \quad j \leq i,
\eeq
where $h_x = h / 2$ and $B_{lm}(r)$ is defined as 
\beq\label{int:box_B}
B_{lm}(r) =   [l r, (l+1) r]  \times [ mr, (m+1) r] .
\eeq
Denote by $R(x, k)$ the rectangle covering $x$
\beq\label{eq:rect_Rk}
R(x, k) \teq [ (i-k) h, (i+1 + k) h] \times [ (j-k) h, (j+1 + k) h] 
\eeq
for any $k > 0$. If $k\in Z^+$, the boundary of $R(x, k)$ is along with the mesh grid and is at least $kh$ away from $x$. Denote by $R_s, R_{s,1}, R_{s,2}$ different symmetric rectangles with respect to $x$ 
\beq\label{eq:rect_Rsk}
\bal
R_s(x, k)  & \teq [x_1 - kh, x_1 + kh] \times [x_2 - kh, x_2 + kh ]  ,  \\
R_{s,1}(x, k) & \teq [x_1 - kh, x_1 + kh]  \times [ (j-k) h, (j+1 + k) h] , \\
R_{s, 2}(x, k) &\teq [ (i-k) h, (i+1 + k) h] \times  [x_2 - kh, x_2 + kh] .
\eal
\eeq
We have $R_s(x, k) \subset R_{s,i} (x,k)\subset R(x, k),i=1,2$. We introduce the upper, lower parts of $R(x,k)$
\beq\label{eq:rect_Rk+}
R^+(x, k) \teq R(x, k) \cap \{ y : y_2 \geq x_2\}, \quad R^-(x, k) \teq R(x, k) \cap \{ y : y_2 \leq x_2\}.
\eeq
We use similar notations for $R_s(x, k), R_{s,1}(x, k), R_{s, 2}(x, k)$. We further introduce the intersection of the rectangle and four half planes with reflection 
\beq\label{eq:rect_NE}
\bal
R(x,k, N)  &= R(x, k) \cap \{y : y_2 \geq 0 \}, \quad R(x, k, S) = \cR_2 ( R(x, k) \cap \{y : y_2 \leq 0 \} ),  \\
R(x,k, E)  &= R(x, k) \cap \{y : y_1 \geq 0 \}, \quad R(x, k, W) = \cR_1 ( R(x, k) \cap \{y : y_1 \leq 0 \} ),  
\eal
\eeq
where $N, E, S, W$ are short for \textit{north, east, south, west}, respectively and the reflection operators $\cR_1, \cR_2$ are given by 
\[
\cR_1( y_1, y_2) = (- y_1, y_2), \quad \cR_2(y_1, y_2) = (y_1, -y_2).
\]
It is clear that $R(x, k , S) \subset \R_2^{+}, R(x, k, W) \subset \{ y: y_1 \geq 0\}$. An illustration of these domains is given in Figure \ref{fig:sing_Rk}.
If $x, y \in \R_2^{++}$, we have the equivalence
\beq\label{eq:rect_equi1}
\bal
& (y_1, -y_2) \notin R(x, k)  \iff (y_1, -y_2) \notin R(x, k) \cap \{ y : y_2 \leq 0 \}
\iff  y \notin R(x, k, S),  \\
& (y_1, -y_2) \in R(x, k)  \iff y \in R(x, k, S) .
\eal
\eeq
 The above notations will be very useful 
 in our later decomposition of the symmetrized kernel.  
\begin{figure}[t]
\centering
\begin{subfigure}{.42\textwidth}
  \centering
  \includegraphics[width=0.64\linewidth]{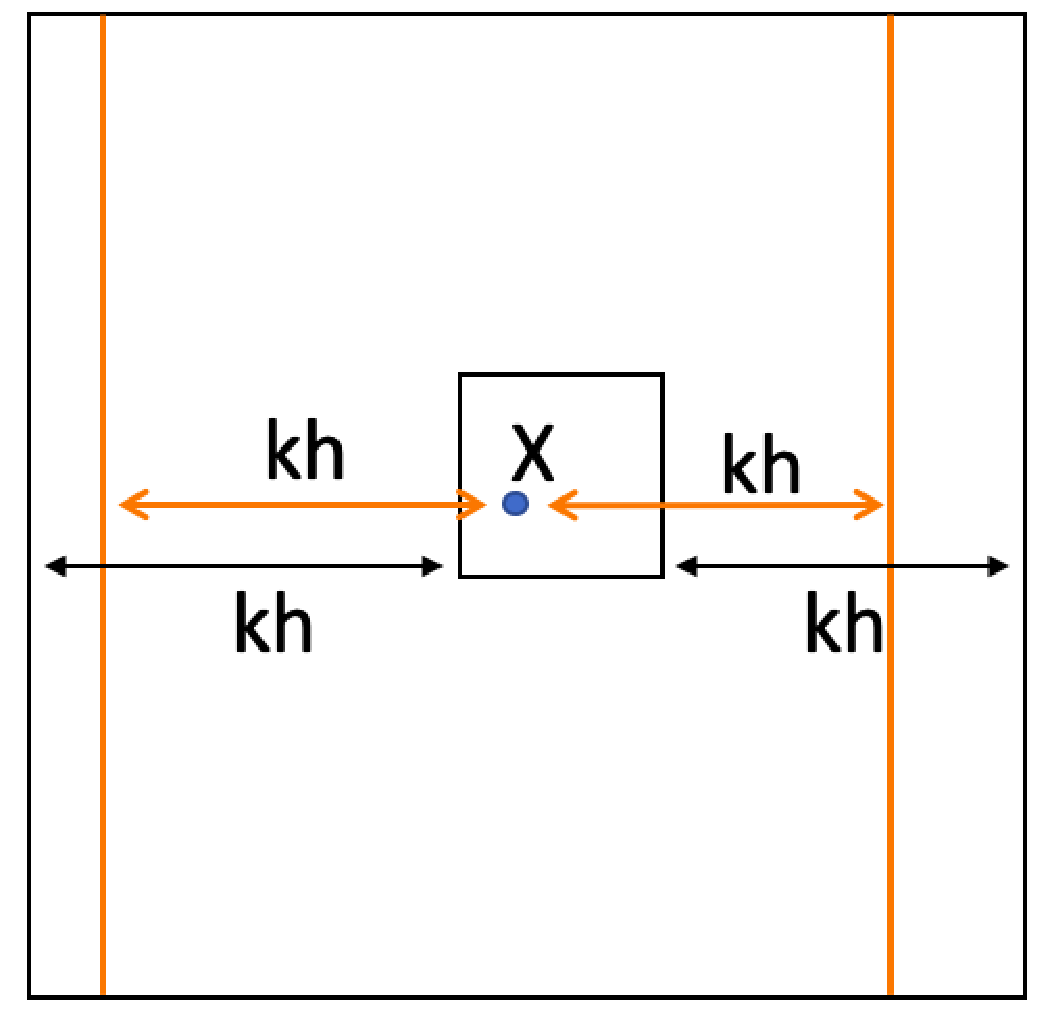}
    \end{subfigure}\begin{subfigure}{.58\textwidth}
  \centering
  \includegraphics[width=0.8\linewidth]{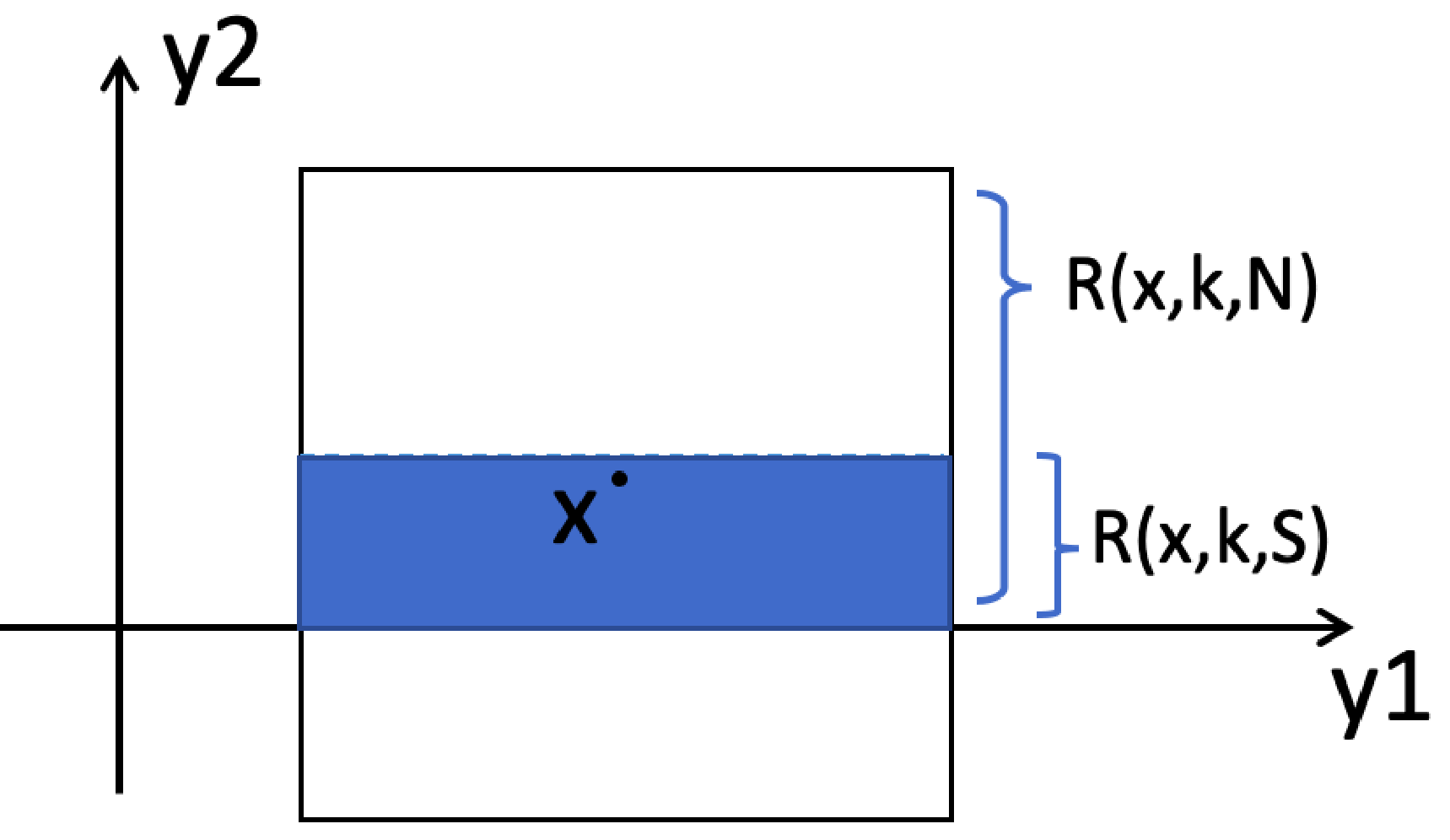}
    \end{subfigure}
\caption{Left: The large box is $R(x, k)$ and the red box is $R_{s,1}(x, k)$. The small box containing $x$ has size $h \times h$. Right: The upper box is $R(x, k, N)$, and the shaded box is $R(x, k, S)$, the reflection of the region below the $y$-axis. }
\label{fig:sing_Rk}
\end{figure}

Define the odd extension of $\om$ in $y$ from $\R_2^+$ to $\R_2$
\beq\label{eq:ext_w_odd}
W(y) = \om(y)  \ \mathrm{ for } \ y_2 \geq 0, \quad W(y) = - \om(y_1, -y_2) \ \mathrm{ for } \ y_2 < 0. 
\eeq
$W$ is odd in both $y_1$ and $y_2$ variables. Since we fix $x$ \eqref{eq:int_loc_x} below, 
for simplicity, we drop $x$ 
 in the $R$ notation. For $k > k_2, k, k_2 \in Z^+$, we decompose the weighted $u_x(x)$ integral as follows \footnote{Since we have no flow boundary condition for the velocity and stream function $- \D \phi = \om, \phi(0, y) = 0$, the Poisson integral formula for $\uu = \na^{\perp} \phi$ is equivalent to $\na^{\perp} (-\D_{2D})^{-1} W$ for $\D_{2D}$ defined in the whole space case.}
\beq\label{int:decom_ux}
\bal
&\psi(x) \int K_1(x - y) W(y) dy = \psi(x) \int_{R(k)^c} K_1(x-y) W(y) dy  \\
 & + \int_{ R_{s, 1}(k)}  K_1( x-y)  \psi(y) W(y) dy  + \int_{ R(k) \bsh R_{s, 1}(k)}  K_1(x- y) \psi(y) W(y) dy  \\
 & + \int_{ R(k) \bsh R(k_2)}  K_1(x-y) (\psi(x) -\psi(y) ) W(y) dy    + \int_{ R(k_2)}  K_1(x-y)(\psi(x) -\psi(y) ) W(y) dy \\
 &\teq I_1(x, k) + I_2(x,k) + I_3(x, k) + I_4(x,k,k_2) + I_5(x, k_2),
\eal
\eeq
where 
\[
K_1(s) \equiv \frac{s_1 s_2}{|s|^4}.
\]
We drop $-\f{1}{\pi}$ in the integrand $- \f{1}{\pi} K_1(s)$ for $u_x(x)$ \eqref{eq:kernel_du} at this moment to simplify the notation. We will estimate different parts in Section \ref{sec:vel_hol_comp}.

\vspace{0.1in}
\paragraph{\bf{Principle and Log-Lipschitz integrand}}
Our major motivation for the decomposition \eqref{int:decom_ux} and the integrands 
\eqref{int:sym_integ1} with symmetrization is to obtain integrands $J$ which is at least locally log-Lipschitz satisfying $J(x, y) \les |x-y|^{-1}$ for $y$ close to $x$, and the more singular one $J_S$. We will estimate the integral of $J$ using the Trapezoidal rule in Section \ref{sec:T_rule} and $J_S$ analytically.

\subsubsection{Symmetrization}\label{sec:int_sym}

After we obtain the decomposition, we use the odd symmetry of $W$ in $y_1, y_2$ to symmetrize the integral and reduce the integral over $\R_2$ to the first quadrant $\R_2^{++}$. 
This enables us to exploit the cancellation in the integral and obtain a sharper estimate. In our computation,  we symmetrize the integrals in $I_1(x,k)$ and $I_4(x, k, k_2)$, which are more regular. For a given kernel $K(x, y)$, we denote by $K^{sym}$ the symmetrization of $K$
\beq\label{int:ker_sym}
K^{sym}(x, y) \teq K(x, y) - K(x, -y_1, y_2) - K(x, y_1, -y_2) + K(x, -y).
\eeq

We show how to symmetrize $I_1(x,k)$ as an example. Recall the notations in \eqref{eq:rect_NE}, \eqref{eq:int_loc_x}. We assume $x_1 \geq x_2$. We choose $ k < i$ so that  $R(x, k) \subset \{ y: y_1 > 0\}$ and $R(x, k, W) = \emptyset$. By definition \eqref{eq:rect_Rk}, the domains $R(x, k), R(x, k, N), R^+(x,k)$ etc are the same for all $x \in B_{i_1, j_1}(h_x)$. Yet, $R(x, k)$ may cross the boundary $y_2 = 0$, i.e. $R(x, k, S) \neq \emptyset $. See the right figure in Figure \ref{fig:sing_Rk} for a possible configuration. Using the equivalence \eqref{eq:rect_equi1} and the property that $W$ is odd in $y_1$ and $y_2$, for general $x\in \R_2^{++}$ (without $x_1 \geq x_2$), we can symmetrize $I_1(x,k)$ as follows 
\beq\label{int:sym_I1}
\bal
I_1(x, k) = \psi(x) \int_{ \R_2^{++}}  & \B( K_1(x-y) \one_{ y \in R(k)^c} 
- K_1( x_1 - y_1, x_2 + y_2) \one_{ y  \notin R( k , S)}   \\
&   - K_1( x_1 + y_1, x_2 - y_2)  \one_{y \notin R(k, W)  } + K_1(x+y) \B) \om(y) dy. 
\eal
\eeq

For $I_4(x)$ \eqref{int:decom_ux}, we choose the weight $\psi(y)$ \eqref{wg:hol}, \eqref{wg:linf} even in $y_1, y_2$. 
Then the symmetrization of $I_4$ is 
\beq\label{int:sym_I4}
\bal
I_4(x, k, k_2) = & \int_{\R_2^{++}}  
\B(   K_1(x-y)   \one_{ y \in R(k) \bsh R(k_2)} 
- K_1( x_1 - y_1, x_2 + y_2) \one_{ y \in R( k , S) \bsh R(k_2, S)}   \\
&  \quad   - K_1( x_1 + y_1, x_2 - y_2)  \one_{y \in R(k, W) \bsh R(k_2, W)  } \B)
( \psi(x) - \psi(y) ) W(y) dy.
\eal
\eeq
In \eqref{int:sym_I4}, we do not have the term $K_1(x+y)$ since for $y \in \R_2^{++}$, $x+ y \geq  x_c > (k+1) h$ and $- y \notin R(k)$. See the discussion below \eqref{int:sing_loc_x}. Thus after symmetrizing the kernel in $I_4$, we do not have such a term.

Though the symmetrized kernel is complicated, since these regions $R(l), R(l, \al), \al = N, E, l = k , k_2$ \eqref{eq:rect_Rk}, \eqref{eq:rect_NE} can be decomposed into the union of the mesh girds $[y_i, y_{i+1}] \times [y_j, y_{j+1}]$, in each grid, the indicator functions are constants. See also Remark \ref{rem:int_indic}. 
In each grid $y \in [y_i, y_{i+1}] \times [y_j, y_{j+1}]$, we can write the integrand in $I_1 + I_4$ as 
\beq\label{int:sym_integ1}
\bal
 J & =  K^{NC}(x, y)  \cdot \psi(x) + K^C(x, y) \cdot (  \psi(x) - \psi(y) ),  \\
  \pa_{x_i} J & = (K^{NC} + K^C) \pa_{x_i} \psi(x) + 
  \pa_{x_i} K^{NC} \cdot \psi(x) + \pa_{x_i} K^C(x, y) \cdot (  \psi(x) - \psi(y) ),  \\
 \eal
\eeq
where $NC, C$ are short for \textit{non-commutator, commutator}, respectively. 

For $y$ close to $x$, $J$ is at least locally log-Lipschitz. See the Principle before Section \ref{sec:int_sym} for motivation.
For $y$ away from $x$, e.g. $|y_1| \vee |y_2| \geq 4 x_c$ in our computation, we have 
\beq\label{int:sym_integ2}
J =  K^{sym}(x, y)  \psi(x).
\eeq

In practice, we assemble the symmetrized integrand in $I_1 + I_4$ in $\R_2^{++}$ together. Using \eqref{int:sym_integ1}, we only need to assemble $K^{NC}, K^C$. 
We first initialize the integrand with $ (K^{NC},  K^C) = ( K^{sym} ,  0)$. 
To assemble the integrand in the singular regions, we perform two replacements. 
In the first replacement, we pretend that $R(k_2) = \emptyset$ and replace the integrand in $R(k) \cap \R_2^{++} $. Based on  $x \in B_{ij}(h)$ \eqref{eq:int_loc_x}, we determine the regions $R(x, k), R(x, k, S)$ \eqref{eq:rect_Rk}, \eqref{eq:rect_NE}. Since $x_1 \geq x_2$, we get $R(x, k, W) = \emptyset$. See Figure \ref{fig:sing_Rk}. We partition $R(k ) \cap \R_2^{++} $ as follows 
 \beq\label{eq:rect_Rk_dec1}
  R(k) \cap \R_2^{++} = R(k, N) = ( R(k, N ) \bsh R(k, S)) \cup R(k, S) \teq D_1 \cup D_2.
 \eeq
According to \eqref{int:sym_I1}, \eqref{int:sym_I4} ($R(k_2)=\emptyset$),  for $i = 1,2$, we first replace $(K^{NC}, K^C)$ in $D_i$ by 
\beq\label{eq:sym_int_rep1}
 ( K^{NC}, K^C) = (K^{sym} - K^C_i, K^C_i) , \  K^C_1 = K_1(x- y),
 \  K^C_2 = K_1(x-y) - K_1(x_1- y_1, x_2 + y_2), 
 \eeq
 respectively, where $K^C$ is from the integrand in \eqref{int:sym_I4}. We have $i$ singular terms in $D_i$ in \eqref{int:sym_I4}.

 In the second replacement, we replace the integrand in the smaller singular region $R(k_2) \cap \R_2^{++} \subset R(k) \bsh \R_2^{++}$. Outside this region, we have obtained the symmetrized integrand using \eqref{eq:sym_int_rep1}. Since we assume $x_1 \geq x_2$, we get $R(k, W) = \emptyset$ (see discussion below \eqref{int:ker_sym}) and $\one_{y \notin R(k, W)} \equiv 1, 
\one_{y \in R(k, W)}  = 0$. Similar to $R(k) \cap \R_2^{++}$ (see Figure \ref{fig:sing_Rk}), we can decompose 
 \[
R(k_2) \cap \R_2^{++} = ( R(k_2, N) \bsh R(k_2, S) ) \cap R( k_2, S) \teq D_3 \cup D_4.
 \]

In $D_4= R(k_2, S) \subset R(k_2), R(k, S)$, from \eqref{int:sym_I1}, \eqref{int:sym_I4}, we completely remove the $K_1(x - y), K_1(x_1 - y_1, x_2 + y_2)$ terms in the integrand and have
\[
(K^{NC}, K^C) = ( K_1(x+y) - K_1(x_1 + y_1, x_2 - y_2)   ,  \ 0 ). 
\]

In $D_3$, since $D_3 \subset R(k, N) = D_1 \cup D_2$ \eqref{eq:rect_Rk_dec1}, there are two cases. In $D_3 \cap D_1, D_1 = R(k, N) \bsh R(k, S)$, we have three non-singular terms from \eqref{int:sym_I1} and $0$ term from \eqref{int:sym_I4} and get
\[
(K^{NC}, K^C) = ( K_1(x+y) - K_1(x_1 + y_1, x_2 - y_2)- K_1(x_1 - y_1, x_2 + y_2), \ 0).
\]

 In $D_3 \cap D_2, D_2 = R(k, S)$, we have two terms from \eqref{int:sym_I1} and one term from \eqref{int:sym_I4}. We get 
 \[
(K^{NC}, K^C) = ( K_1(x+y) - K_1(x_1 + y_1, x_2 - y_2), - K_1(x_1 - y_1, x_2 + y_2)).
 \]

For $x_1 < x_2$, we assemble the integrand similarly. Using \eqref{int:sym_integ1}, we obtain the integrand $\pa_{x_i} J$ for the H\"older estimate. 

\vs{0.1in}
\paragraph{\bf{$C_y^{1/2}$ estimate of $u_y, v_x$ }}
In the $C_y^{1/2}$ estimate of $u_y, v_x$ with kernel $K_2$ \eqref{eq:kernel_du}, we symmetrize the integrand $K(x-y) (\psi(x) - \psi(y)$, see \eqref{int:decom_uy} in Section \ref{sec:cy_spec}. 
In this case, the symmetrized integrand $W(y) T$ is similar to \eqref{int:sym_I1} with $\psi(x)$ replaced by $\psi(x) - \psi(y)$ and 
\[
T = (\psi(x) - \psi(y) ) \B(K_2(x-y) \one_{ y \in R(k)^c} 
- K_2( x_1 - y_1, x_2 + y_2) \one_{ y  \notin R( k , S)}  
   - K_2( x_1 + y_1, x_2 - y_2)  \one_{y \notin R(k, W)  } + K_1(x+y) \B).
\]

Due to the weight $(\psi(x) - \psi(y) ) $, we always have $K^{NC} = 0$. We initialize the $T$ using \eqref{int:sym_integ1} with $K^C = K_2^{sym}$  \eqref{int:ker_sym}.
In the singular region $R(x,k) \cap \R_2^{++}$, we only need to perform one replacement.
Similar to \eqref{eq:sym_int_rep1}, we use  \eqref{eq:rect_Rk_dec1} and replace the integrand as follows 
\[
K^C = K_2^{sym} - K_2(x-y),   y \in R(k, N) \bsh R(k, S), 
\  K^C = K_2^{sym} - ( K_2(x-y) -  K_2(x_1-y_1, x_2+y_2) ) , y \in R(k, S).
\]
We remove the most singular integrand in $ R(k, N) \bsh R(k, S)$ and the most two singular integrands in $D_2 = R(k, S)$ to make $T$ locally log-Lipschitz. See the Principle before Section \ref{sec:int_sym}.

\vs{0.1in}
\paragraph{\bf{$L^{\inf}$ estimate}} For $L^{\inf}$ estimate, we do not multiply the integrand by the weight $\psi(x)$ or the commutator. We decompose the integral as \eqref{int:decom_linf_ux}, and symmetrize the nonsingular part in $I_1$ using \eqref{int:sym_I1} without the weight $\psi(x)$. Symmetrizing $I_4$ \eqref{int:decom_linf_ux} is similar. 
We initialize the symmetrized integrand as $K^{sym}$ \eqref{int:ker_sym}, and then 
replace it in $R(k) \cap \R_2^{++}$. 
Without loss of generality, we assume $x_1 \geq x_2$ and  have the decomposition \eqref{eq:rect_Rk_dec1}. Similar to \eqref{eq:sym_int_rep1}, we replace the integrand as follows 
\[
 K^{sym} - K_1(x-y), y \in R(k, N) \bsh R(k, S) , \quad K^{sym} - ( K_1(x-y) - K_1(x_1-y_1, x_2+y_2) ) , y \in R(k, S).
\]
That is, we remove one and two singular terms in $ R(k, N) \bsh R(k, S), R(k, S)$, resepctively, making the integrand at least locally log-Lipschitz. See the Principle before Section \ref{sec:int_sym}.

\subsubsection{Integral in domains depending on $x$}\label{sec:int_nsym1}

In the computation, we need to estimate several integrals in the domains $D(x)$ depending on $x$, e.g. $I_3$ in \eqref{int:decom_ux}. Our fundamental idea is to cover $D(x)$ by some piecewise constant domains, which will be essentially treated as fixed domains. 
By refining the location of $x$, we can obtain tight covering.

We use the $L^{\inf}$ estimate of $I_3$ to illustrate the ideas. A direct estimate yields 
\[
|I_{3}(x)| 
\leq || W \vp||_{\inf} \int_{ R(k)\bsh R_{s,1}(k)} | K_1(x- y)| \psi(y) \vp^{-1}(y)  dy.
\]
We cannot apply the method in Section \ref{sec:T_rule} to first estimate $I_3(x)$ for $x$ on the grid points and then estimate $\pa^2 I_3(x)$ for the error since the kernel is singular and the error part associated with $\pa^2 I_3(x)$ is more singular (see Lemma \ref{lem:int_err_x}). 

Denote $f = \psi \vp^{-1}$. We consider a change of variable $y = x + s$ to center our analysis around the singularity $x$. The domain for $s$ is 
\beq\label{eq:int_sing_dom1}
\{ y \in R(k) \bsh R_{s, 1}(k) \} = \{ s \in R(k ) - x \} \cap \{ |s_1| \geq k  h \} \teq D(x, k).
\eeq
It suffices to estimate 
\beq\label{eq:int_sing_nsym1}
J = \int_{s \in D(x, k)} |K_1(-s)| f(x+s) dy, \quad f \geq 0,
\eeq
for all $x \in B_{i_1, j_1}( h_x)$ \eqref{eq:int_loc_x}. We want to further simplify the above domain so that it does not depend on $x$.
Recall the location of $x$ \eqref{eq:int_loc_x}. To obtain a sharp estimate, we further partition the location of $x \in B_{i_1, j_1}(h_x)$ as follows
\beq\label{eq:int_sing_nsym1_AB}
A_{ a} = [i_1 h_x + a h_x / m, i_1 h_x + (a + 1) h_x / m], \  B_{ b} \teq 
[j_1 h_x + b h_x / m, j_1 h_x + (b + 1) h_x / m] , \  
\eeq
for some $m \in Z^+$ and  $0 \leq a, b \leq m-1$. Clearly, $A_a \times B_b$ is a partition of $B_{i_1 j_1}(h_x)$. Recall \eqref{eq:int_loc_x} and \eqref{eq:rect_Rk}. We have 
\[
R(x, k) = [ (i- k) h, (i+1 + k) h] \times [ (j-k) h , (j+1 + k) h]. 
\]
Now, for $x \in A_a \times B_b$, since $|s_1| \geq kh$, we have 
\beq\label{eq:int_sing_nsym1_X}
\bal
s_1  = y_1 - x_1 & \in [ (i-k) h- i_1 h_x - (a+1) h_x / m, - kh] \cup  [ kh, (i + 1 + k) h - i_1 h_x - a h_x / m ]   \\
& \teq X_{l,a} \cup X_{r, a},
\eal
\eeq
where the subscripts \textrm{l, r} are short for \textrm{left, right}, respectively.  Similarly, for $s_2$, we have 
\beq\label{eq:int_sing_nsym1_Y}
\bal
s_2  &= y_2 - x_2 \in [ (j-k) h - j_1 h_x - (b+1) h_x / m ,  (j + k + 1) h - j_1 h_x - b h_x / m ]  \\
& \teq [ (j-k) h - j_1 h_x - (b+1) h_x / m, - kh] \cup [-kh, kh] \cup [ kh, (j + 1+ k) h - j_1 h_x - b h_x / m ] \\
&\teq  Y_{d, b} \cup Y_{m, b} \cup Y_{ u, b}
\eal
\eeq
where the subscripts \textrm{d, m, u} are short for \textrm{down, middle, upper}, respectively. 
Note that the intervals $X, Y$ do not depend on $x$. We have 
\beq\label{eq:int_sing_nsym1_dom2}
D(x, k) \subset (X_{l,a} \cup X_{r, a} ) \times (Y_{d, b} \cup Y_{m, b} \cup Y_{ u, b}).
\eeq
Now, we can decompose $J$ \eqref{eq:int_sing_nsym1} as follows 
\[
J  \leq \sum_{\al =l, r, \b = d, m, u } J_{\al, \b} , \quad J_{\al, \b}\teq \int_{ X_{\al, a} \times Y_{\b, b}}  |K_1(-s)| f(s + x) dy, \ \al = l, r, \ \b = d, m, u .
\]
See the left figure in Figure \ref{fig:sing_nsym1} for different domains in the above decomposition. From the definitions of $X, Y$, the total width of the left and the right domains $X_{\al, a} \times (Y_{d, b} \cup Y_{m, b} \cup Y_{ u, b}), \al = l, u$ is 
\[
   |X_{l, a}| + |X_{r, a}| = h + h_x / m.
\]
For a fixed $x$, from the definition \eqref{eq:rect_Rk}, the width of $R(k) \bsh R_{s,1}(k) $ is $h$. We choose a large $m$ and further partition the location of $x$ so that we do not overestimate the region too much.

\begin{figure}[t]
\centering

  \includegraphics[width=0.32\linewidth]{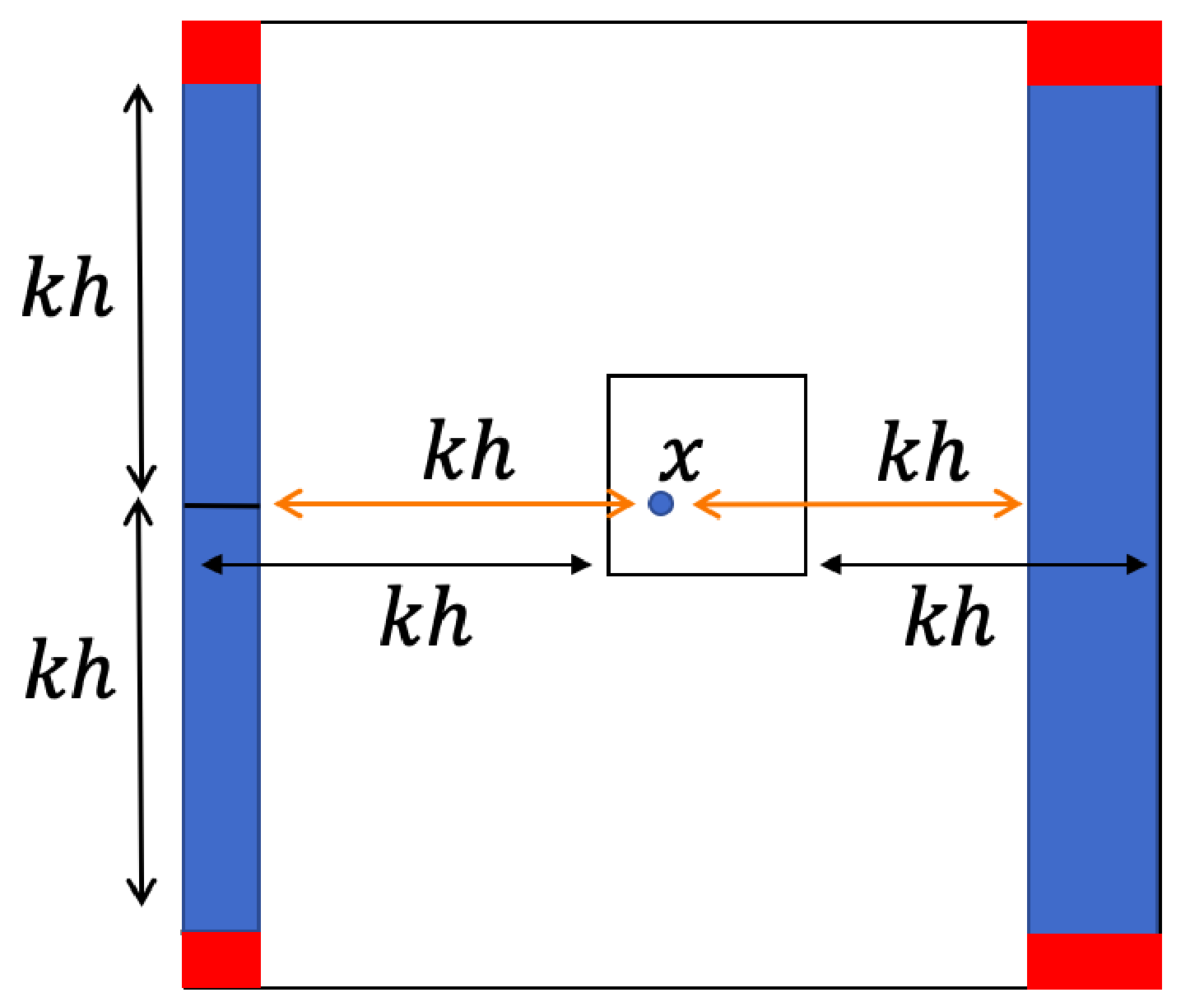}
  \includegraphics[width=0.32\linewidth]{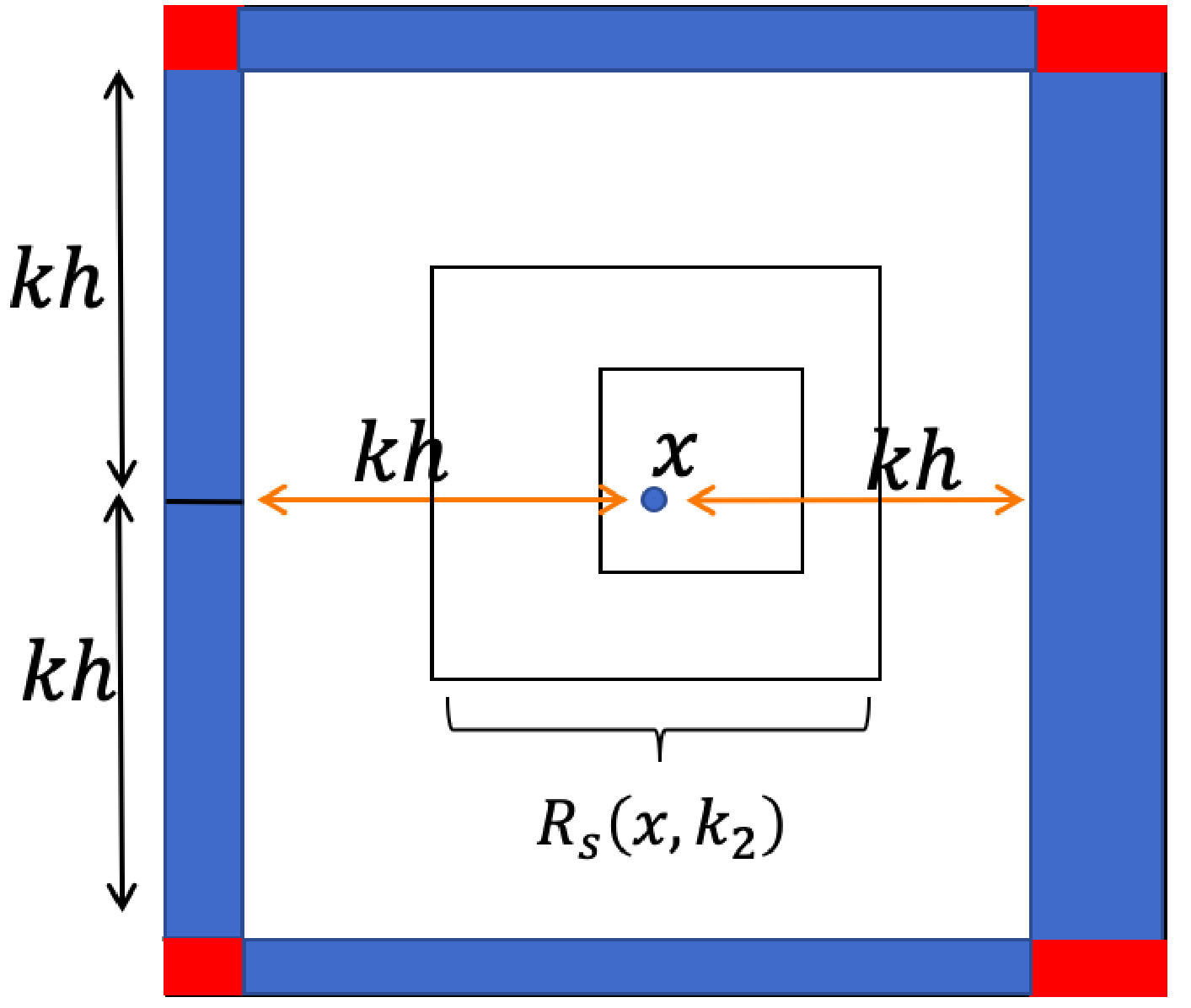}
      \includegraphics[width = 0.25\textwidth  ]{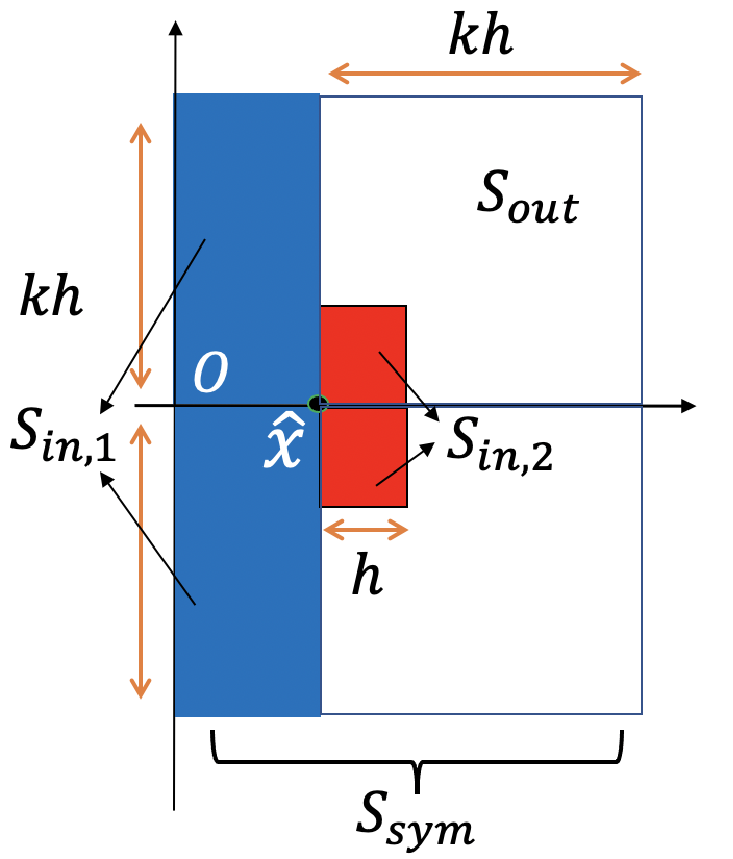}

\caption{
The largest box in the left and middle figure is $R(x, k)$. Left: The left and right blue regions are $X_{l, a} \times Y_{m, b}, X_{r, a} \times Y_{m, b}$.
The four red regions correspond to $X_{\al, a  } \times Y_{\b, b}, \al = l, u, \b = d, u$. 
Middle: Illustration of $R(x, k) \bsh R_s(x, k)$ and $R_s(x, k_2)$. $R(x, k) \bsh R_s(x, k)$ consists of the blue and the red regions. 
Right: different regions near the singularity for $u/x_1$. Blue, red, and white regions represent 
 $S_{in, 1}, S_{in, 2}, S_{out}$, respectively.
 }
\label{fig:sing_nsym1}
\end{figure}

For a small domain $Q = [a,b]\times [c, d]$ , we can estimate the integral as follows 
\beq\label{eq:int_sing_nsym1_Q}
\int_Q  |K_1(-s) | f(x + s) ds \leq \int_Q |K_1(-s)| ds 
|| f||_{ L^{\inf} (B_{i_1 j_1}(h_x) + Q )}.
\eeq
Since $Q$ is given, $K_1(s)$ is explicit and has scaling symmetries, we can estimate the integral of $|K_1(s)|$ easily. For example, if $Q = [ah, bh]^2$, we can use the scaling symmetries of $K_1(s)$ to obtain $\int_Q |K_1(-s)| = h^{\b} \int_{[a,b]^2} |K_1(-s)|$ for some $\b$. Moreover, for many kernels in our computations, e.g. $K(s) = \f{s_1 s_2}{|s|^4}$, we have explicit formulas for the integral. See Section {\secintker} in the supplementary material II \cite{ChenHou2023bSupp}.

We apply the above method to estimate the integral in $X_{\al, a} \times Y_{\b, b }, \al = l, r, \b = d, u$ (red region in Figure \ref{fig:sing_nsym1}). 
Since $Y_{m, b} = [-kh, kh]$, for the integral in $X_{\al, a} \times Y_{m, b}$ (blue region), we further decompose it 
\beq\label{eq:int_sing_nsym1_gd}
J_{\al, m} = \sum_{ -k \leq t \leq k-1} \int_{ X_{\al, a} \times [th, (t+1) h ]} |K_1(-s)| f(s+x) dy ,
\eeq
and then apply the above method to estimate it.

Next, we further simplify $|| f||_{ L^{\inf} (B_{i_1 j_1}(h_x) + Q )}$ in the above estimate. From \eqref{eq:int_loc_x}, we get 
\[
 ih \leq i_1 h_x < (i_1 + 1) h_x \leq (i+1) h, \quad  
 j h \leq j_1 h_x <(j_1 + 1) h_x \leq (j+1) h. 
\]
For $X_{l, a}$ \eqref{eq:int_sing_nsym1_X} with $0 \leq a \leq m-1$, we have the lower bound for the endpoint
\[
 (i-k) h- i_1 h_x - (a+1) h_x / m 
 \geq  (i-k) h- i_1 h_x - h_x \geq  (i-k) h -( (i+1) h - h_x ) - h_x = - kh -h .
\]
See the left figure in Figure \ref{fig:sing_nsym1}. The width of blue region is less than $h$.
Similarly, we can cover the intervals of $X, Y$ \eqref{eq:int_sing_nsym1_X}, \eqref{eq:int_sing_nsym1_Y} uniformly for $0\leq a , b \leq m-1$ and obtain 
\[
\bal
X_{l, a} &\subset [ (i-k )h - i_1 h_x - h_x, -kh ] \subset [ -(k+1) h, - kh] ,  \\
X_{r, a} &\subset [ kh, (i+1 + k) h - i_1 h_x ] \subset [kh, (k+1) h],  \\
Y_{d, b} &\subset [ - (k+1) h, - kh],  \quad  Y_{u, b} \subset [ kh, (k+1) h].
\eal
\]

Thus, we only need to estimate the $L^{\inf}$ norm of $f$ in 
\[
 Q_{i_1 j_1}(h_x) + [ \al h, (\al + 1) h] \times [ \b h, (\b + 1) h], \quad \al = - k-1, k, \quad \b = -(k+1), -k , .., k.
\]
These estimates are independent of the choice of $m, a, b$. Since the size of each domain is at most $2h \times 2h$, the above estimates based on \eqref{eq:int_sing_nsym1_Q} are sharp. We estimate the piecewise bound of the weights $\psi, \vp$ in Appendixes \ref{app:wg},\ref{app:wg_radial},\ref{app:wg_mix}.

Using the above decomposition and estimates, we obtain the estimate of $J$ \eqref{eq:int_sing_nsym1}  for $x \in  A_a \times B_b$ \eqref{eq:int_sing_nsym1_AB}. Similarly, we can estimate $J$ for any $0 \leq a, b \leq m-1$. Taking the maximum of these $m^2$ estimates, we obtain the estimate of $J$ and $I_3(x)$ for all $x \in B_{i_1 j_1}(h_x)$.

\subsubsection{First generalization: integral in a ring}\label{sec:int_nsym_bd0}

We generalize the above ideas to estimate the integrals in domain $D = R(x, k) \bsh R(x, k_2) =  R(k) \bsh R(k_2)$ 
\[
J = \int_{R(k) \bsh R(k_2)} |K(y-x)|  |f(y)| dy = \int_{s \in D(x, k)} |K(s)| |f( x + s) d y|, 
\quad D(x, k) \teq R(k) \bsh R(k_2) - x  
\]
with $2\leq k_2 = k - \f{i}{2} < k$ for some integer $i\geq 1$ and some kernel $K(z)$. 
Note that the inner region $R(k_2)$ is different from \eqref{eq:int_sing_dom1}. See $I_4$ in \eqref{int:decom_ux} for an example of this integral region. Suppose $x \in B_{ij}(h)$ \eqref{eq:int_loc_x}. We partition location of $x$ similar to \eqref{eq:int_sing_nsym1_AB} and introduce $p_l, q_l$
\beq\label{eq:singu_mbulk1}
\bal
& A_a = [ i h + ah / m, ih + (a+1)h / m] , B_b = [ j h + b h /m , j h + (b+1) h / m  ],  \ 0 \leq a ,b \leq m-1, \\ 
& p_1 = -k_2 - a / m, \  p_2  = k_2 + (m-a-1) / m, \ 
p_3 = -k_2 - b /m, \ p_4 = k_2 + (m-b-1) / m , \\
& q_1  = - k - (a+1) / m,\ q_2 = k + (m-a) / m, \ q_3 = - k -(b+1) / m, \ q_4 = k + (m-b) / m.
\eal
\eeq
For a fixed $x \in A_a \times B_b$, by comparing the boundaries of the following four rectangles, 
we get 
\[ 
\bal
& D_{in} \teq 
[p_1 h , p_2 h] \times [p_3 h , p_4 h]
\subset R(k_2)  - x \subset R(k) - x   \subset [ q_1 h, q_2h ] \times [q_3 h, q_4 h]  \teq D_{out} .
\eal
\]
To obtain the above inclusions, for example, for $s = y - x, y \in  R(k_2)$, we use
\[
\min_{y \in R(k_2)}  y_1 - x_1 =  ih - k_2 h - x_1 \leq
ih - k_2 h - ( ih + ah/m) = - k_2 h - ah / m = p_1 h,
\]
uniformly for $x \in A_a \times B_b \subset B_{ij}(h)$. 
For $R(k) - x \subset D_{out}$, we have $ q_1 h \leq \min_{y \in R(k)}  y_1 - x_1$. Other bounds for the inclusions are obtained similarly. We yield $D(x, k) \subset D_{ring}$, where
\begin{equation}
\label{D-ring}
 D_{ring} \teq  D_{out} \bsh D_{in}
 \end{equation}
 is fixed for $x \in A_a \times B_b$. 

 It suffices to estimate the integral $J$ in $D_{ring}$. We partition $s \in D_{ring}$ using mesh 
\beq\label{eq:singu_mbulk2}
Z_1 = \{ -k\leq i \leq k, i \in \BZ \} \cup \{p_1, p_2, q_1, q_2 \}
,\quad Z_2 =  \{ -k\leq i \leq k, i \in \BZ \} \cup \{p_3, p_4, q_3, q_4 \},
\eeq
and then order them in an increasing order $z_{l,1} < z_{l,2} <.. < z_{l,2k+5} \in Z_l, l = 1,2$. Note that we do not multiply $z_{l, c}$ by $h$ here. We estimate the integral $J_{cd}$ in each grid $Q_{c, d} = [z_{1, c} h, z_{1, c+1} h ] \times [z_{2, d} h , z_{2, d+1} h ]$ following \eqref{eq:int_sing_nsym1_Q} and using the norm $|| f||_{L^{\inf}(x + Q_{c, d})}$. We turn off the integral in region $Q_{c,d}$ if $Q_{c,d} \subset D_{in}$ since it is not in $D_{ring}$ \eqref{D-ring}.

\vspace{0.1in}
\paragraph{\bf{Uniform covering}}
For fixed $c, d$, we want to cover $ x + Q_{c,d }$ uniformly for $x \in A_a \times B_b$ and all $0\leq a, b\leq m-1$ (the sub-partition of $x$) to bound $|| f||_{L^{\inf}(x + Q_{c,d})}$. Since we add $4$ extra points $p_l, q_l$ in $Z_1$ and $Z_2$, and order them in an increasing order, the region $Q_{c, d}$ can change for fixed $c, d$ but with different $a, b$. We show that the $2k+4$ intervals $[z_{1, c}, z_{1, c+1}] , 1 \leq c \leq 2 k + 4 $ can be covered by $[\al_l, \b_l]$ 
uniformly for $a, b$
\beq\label{eq:singu_mbulk32}
\bal
&  [ \al_l, \b_l ] , \quad \al_l \in Z_1^l, \b_l \in Z_1^u, \
Z_1^l\teq  \{ -(k+1) \leq i \leq k , i \in \BZ \} \cup \{ - s_0 -2, s_0  \}, \\
& Z_1^u \teq \{ - k \leq i \leq k + 1 , i \in \BZ \} \cup \{ - s_0 , s_0  + 2  \},
 \quad s_0 = \lfloor k_2 \rfloor,
 \eal
\eeq
with $\al_l, \b_l$ increasing. From \eqref{eq:singu_mbulk1} and the definition of $s_0$, we get 
\beq\label{eq:singu_mbulk4}
p_1 \in [-s_0 - 2, - s_0], \  p_2 \in [s_0, s_0 + 2], \ q_1 \in [-k-1, -k], \  q_2 \in [k, k+1].
 \eeq

The uniform covering is based on the following observations. Suppose that $ u_i \leq v_i, i=1,2,.., n $ ($u_i, v_i$ may not be increasing). Let us denote by $\{U_i\}$ the re-ordering of $\{u_i\}$ in an increasing order and denote by $\{V_i\}$ the re-ordering of $\{v_i\}$ in an increasing order. Then we have $U_i \leq V_i$. In fact, for any $k \leq n$, from $u_i \leq v_i$, $V_k$ is larger than $ u_{j}$ with at least $k$ different indexes $j$. Since $U_k$ is the $k$-smallnest value in $\{ u_i \}_i$, we get $V_k \geq U_k$.

From \eqref{eq:singu_mbulk2}, \eqref{eq:singu_mbulk4}, since $q_2 = \max_c z_{1, c}, q_1 = \min_c z_{1, c}$, we get
\[
\bal
& \{ z_{1, c}, c \leq 2 k + 4 \} = \{ -k \leq i \leq k,  i \in \BZ \} \cup \{p_1, p_2, q_1\}, 
- k - 1 \leq q_1,  -s_0 - 2 \leq p_1,   s_0 \leq p_2 ,\\
& \{ z_{1, c+1}, c \leq 2 k + 4 \} = \{ -k \leq i \leq k,  i \in \BZ \} \cup \{p_1, p_2, q_2 \}, \
p_1 \leq - s_0, \ p_2 \leq s_0 + 2, q_2 \leq k+ 1.
\eal
 \]
We can bound each component in $Z_1^l$ \eqref{eq:singu_mbulk32} by a component in the above list. 
Using the above observations, after reordering two sequences in an increasing order, which gives $\{\al_c\}, \{ z_{1, c} \}_{ c \leq 2 k + 4} $, we get $\al_c \leq z_{1, c}, c \leq 2 k + 4$ 
\eqref{eq:singu_mbulk32}. Similarly, we obtain $ z_{1, c+1} \leq \b_c$, and yield $[z_{1, c}, z_{1, c+1}] \in [\al_c, \b_c],  c \leq 2k +4$. 

Similarly, we obtain $[z_{2, d}, z_{2, d+1}] \subset [\al_d, \b_d]$. Thus, we get $ [z_{1, c}, z_{1, c+1}] \times [z_{2, d}, z_{2, d+1}] \in [\al_c, \al_{c+1}] \times [\b_d, \b_{d+1}]$ uniformly for the sub-partition of $x\in A_a \times B_b$ with  $ 0 \leq a, b \leq m-1$, and can cover $x + Q_{cd} $ by $B_{i_1 j_1}(h_x) + [\al_c h, \al_{c+1} h] \times [\b_d h, \b_{d+1} h]$ \eqref{eq:int_loc_x} for $x \in B_{i_1 j_1}(h_x)  \subset B_{ij}(h)$.

\subsubsection{Second generalization: the boundary terms}\label{sec:int_nsym_bd}

We generalize the method to estimate some boundary terms. We estimate the $x_1-$derivative of $I_3(x)$ \eqref{int:decom_ux} to illustrate the ideas. 
In $\pa_1 I_3$, we have an extra boundary term $I_{32}$
 \[
\pa_1 I_3(x) =  \int_{ R(k) \bsh R_{s, 1}(k)} \pa_{x_1} K_1(x- y)  (W\psi)(y)  dy 
- \int_{ (j-k)h }^{ (j+1 + k) h} K_1(x-y) (W \psi)(y) \B|_{y_1 = x_1 -kh }^{x_1 + kh} d y_2
\teq I_{31} + I_{32},
 \]`'
 where we have used the domain for $R(x, k)$ \eqref{eq:rect_Rk}.

For $I_{31}$, we apply the method in Section \ref{sec:int_nsym1} to estimate it.  Denote $\G_k \teq   [ j-k)h, (j+1 + k) h]$. Using a change of variable $y = x + s$, we can rewrite $I_{32}$ as follows 
 \[
I_{32} = - \int_{ s_2 \in \G_k - x_2} \B( K_1(- kh, -s_2)  (W \psi)( x_1 +  kh, x_2 + s_2 )
 - K_1( kh, -s_2)  (W \psi)( x_1 -  kh, x_2 + s_2 ) \B) d s_2.
 \]

We partition the location of $x$ and assume $x \in A_a \times B_b \subset B_{i_1, j_1}(h_x)$ \eqref{eq:int_sing_nsym1_AB}. From \eqref{eq:int_sing_nsym1_Y}, we have 
\[
s_2 \in \G_k -x_2 \subset Y_{d, b} \cup Y_{m, b} \cup Y_{u, b}.
\]

Using the above decomposition and $ |W \psi(x)| \leq || W \vp||_{\inf}  f(x)$, $f = \psi \vp^{-1}$,  we obtain 
\[
|I_{32}| \leq || W \vp||_{\inf} \sum_{ \al = \pm, \b = d, m, u} M_{\al, \b},
\quad M_{\al, \b} \teq  \int_{ Y_{\b, b}} | K_1( -\al kh, -s_2)| \cdot  |f(x_1 + \al kh,
x_2 + s_2 ) |d s_2,
\]
for $\al = \pm, \b = u, m, d$. For $ \b = u, d$, the domain $Y_{\b, b}$ is small $|Y_{\b, b}| \leq h$. We apply the method in \eqref{eq:int_sing_nsym1_Q} to estimate $M_{\al, \b}$. The only difference is that we need consider a 1D integral here
\[
\int_Q | K_1(-\al kh, -s_2) | d s_2
\]
for some interval $Q$, rather than a 2D integral in \eqref{eq:int_sing_nsym1_Q}. For $M_{\al, m}$, we decompose the domain $Y_{m, b}$ into small intervals with length $h$ similar to \eqref{eq:int_sing_nsym1_gd} and then apply the method in \eqref{eq:int_sing_nsym1_Q}. 

We combine these estimates to bound $I_{32}$ for $x \in A_a \times B_b$. Then, we maximize the estimates over $0\leq a, b \leq m-1$ to bound $I_{32}$ for $x \in B_{i_1, j_1}(h_x)$.

\subsubsection{Third generalization}\label{sec:int_nsym2}

In some of the computations, we need to estimate 
\[
 J =  \int_{ R(k) \bsh R_s(k_2) }  | K(x - y)|  f(y) dy
\]
for some $k_2 < k$ with $ 2 k_2 , k\in Z^+$, where $R_s(k)$ is defined in \eqref{eq:rect_Rsk}.  Similarly, we use
\[
R_s(k_2) \subset R_s(k) \subset R(k), \quad R(k) \bsh R_s(k_2) = 
R(k) \bsh R_s(k) \cup R_s(k) \bsh R_s(k_2),
\]
and a change of variable  $y = x + s$ to obtain 
\[
J= ( \int_{ s \in R(k) - x , |s_1| \vee |s_2| \geq k h } 
+ \int_{  k_2 h \leq  |s_1| \vee |s_2| \leq k h}  ) K(-s) f( x + s ) dy \teq J_1 + J_2.
\]

Compared to $R(k) \bsh R_{s,1}(k)$, the domain $R(k) \bsh R_s(k)$ contains two more parts 
\[
X_{m, a} \teq [-kh, kh], \quad  X_{m, a} \times Y_{u, b}, \quad X_{m, a} \times Y_{d, b},
\] 
i.e., the upper and lower blue regions in the right figure in Figure \ref{fig:sing_nsym1}. 
The integral in these regions is estimated similar to that in $X_{\al, a} \times Y_{m, b}$ \eqref{eq:int_sing_nsym1_dom2}, and the estimate of $J_1$ is similar to $J$ in \eqref{eq:int_sing_nsym1}. 

For $J_2$, the domain is simpler. Since $2k_2 \in Z^+$, we partition the domain into $h_x \times h_x$ grids
\[
J_2 = \sum_{ (c, d) \in S_{k} \bsh S_{k_2}  }  \int_{ [ c h_x , (c+1) h_x  ] \times 
[d h_x, (d+1) h_x ] } |K(-s)| f( s + x) ds, \quad  S_l \teq \{ -k \leq c < k, -k \leq d < k  \}.
\]
For each integral, we estimate it using the method in \eqref{eq:int_sing_nsym1_Q}. The remaining steps are the same as those of $J$ in \eqref{eq:int_sing_nsym1} studied previously.

\begin{remark}
In the estimates in Section \ref{sec:int_nsym1}-\ref{sec:int_nsym2}, we use the important property that the weights are locally smooth to move them outside the integral. Moreover, we use the fact that the singular region depend on $x$ 
monotonously 
to cover it effectively.
Since the integral $\int_Q |K_1(s)| dy$ for different $Q,a, b$ in the above estimates does not depend on $x$, we first compute these integrals once and store them, and then use them in later estimate of different $x$.  
 \end{remark}

\subsubsection{Taylor expansion near the singularity}\label{sec:tay}

We need to estimate the integral 
\[
 J(x) \teq \int_{D } \pa_{x_i} \B(  K(x-y) (\psi(x) - \psi(y) ) W(y) \B) dy ,
\]
for $k_2 < k$ in some region $D$ close to the singularity $x$. For example,  $D =R(x, k_2) \bsh R(x, k_3)$, $R(x, k_3) \bsh R_{s 1}(x, k_3)$ in $\pa_{x_i} I_{5,0}, \pa_{x_i} I_{5, 1}$  \eqref{eq:int_holx_I5},
To obtain a sharp estimate, we perform Taylor expansion on $\psi(x)$. We focus on $\pa_{x_1}$. Denote $z = x - y,  x_m  = \f{x+ y}{2} $. A direct computation yields
\[
\bal
I =   \pa_{x_1} ( K(x-y) \psi(x) - \psi(y))
 = (\pa_1 K)(x- y)   (\psi(x) - \psi(y))
+ K(x-y) \pa_1 \psi(x). \\
\eal
\]
Using Taylor expansion of $\psi$ at $x_m$ and following \eqref{eq:integ_Ta1}, we get
\[
\bal
&\psi(x) - \psi(y) = (x-y) \cdot \na \psi(x_m) + \e_1, 
\quad \psi_x(x) = \psi_x(x_m) + \e_2 \\
& |\e_1 | \leq \sum_{i+j = 2} c_{ij} || \pa_x^i \pa_y^j \psi ||_{L^{\inf}(Q(y))}  |z_1|^i |z_2|^j, \quad | \e_2| \leq \f{1}{2}  (||\pa_{xx} \psi||_{L^{\inf}(Q(y))} |z_1|
+ ||\pa_{xx} \psi||_{L^{\inf}(Q(y))} | z_2|) ,
\eal
\]
where $c_{20} = \f{1}{4}, c_{11} = \f{1}{2}, c_{02} = \f{1}{4}$, and we have written $z_i = x_i-y_i$ and $Q(y)$ is one of the four quadrants $ D \cap \{ y: sgn ( y_i - x_i) = \pm 1 \} $ covering both $x, y$.  Combining the term with the same derivative of $\psi$, we need to estimate 
the following integrals 
\[
\bal
 & |\int_D  \psi_x(x_m) (\pa_1 K(z) z_1 + K(z) ) W( y) dy| , \quad 
 |\int_D \psi_y(x_m) \pa_1 K(z) z_2  W( y) dy| \\
 &  \int_D |\pa_x^i \pa_y^j \psi|_{L^{\inf}(Q(y))}  |\pa_1 K(z) z_1^i z_2^j W(y)| dy,  i+j =2, \
\int_D   | \pa_x^{i+1} \pa_y^j  \psi |_{L^{\inf}(Q(y))}  | K(z) z_1^i z_2^j W(y)| dy, i+j = 1 ,
\eal
\]
We partition the region of  $z = x- y \in x - D$, e.g. $D = R(k_2) \bsh R(k_3)$ \eqref{eq:int_holx_I5}
into small mesh, and estimate the piecewise bounds of weights and  each integral following Sections \ref{sec:int_nsym1}-\ref{sec:int_nsym2}. 

We estimate the integral of $ | \pa_1^i \pa_2^j K(z) z_1^k z_2^l|$ in Section {\secintker} in the supplementary material II \cite{ChenHou2023bSupp}.

\subsubsection{H\"older estimate of log-Lipschitz function}\label{sec:int_loglip}

In some computation, we need to perform $C^{1/2}$ estimate of some log-Lipschitz function. We consider an example to illustrate the ideas 
\[
F(x) = \int_{ \max_i |x_i - y_i| \leq b } K( x, y) f(y) dy,   \quad  |K(x, y) | \leq C_1 |x-y|^{-1},
\quad |\pa K(x, y)| \leq C_2 |x-y|^{-2} ,
\]
for some constant $C_1, C_2$. Given $f \in L^{\inf}$, $F$ is log-Lipschitz. To estimate $[f]_{C_x^{1/2}}$, we cannot first estimate the piecewise values  of $f$  and $\pa_x f$ and then combine them to obtain the $C_x^{1/2}$ estimate. Instead, given $x, z$,  for $a$ to be determined, we decompose $F$ into the smooth part and the singular part
\[
\bal
F_1(x) \teq \int_{ a \leq  \max_i |x_i - y_i| \leq b } K( x, y) f(y) dy,   \quad 
F_2(x) \teq \int_{  \max_i |x_i - y_i| \leq a } K( x, y) f(y) dy .  
\eal
\]

Using the assumptions of the kernel, we have 
\[
|\pa_{x_1} F_1(x)| \leq C_3 \log \f{b}{a} || f||_{\inf}, \quad |F_2(x)| \leq  C_4 |a| \cdot || f||_{\inf},
\]
where the constants $C_3, C_4$ depend on $b, C_1, C_2$. Applying the above estimates, we obtain 
\[
\f{|F(x) - F(z)|}{ |x_1 - z_1|^{1/2}}
\leq 
\f{|F_1(x) - F_1(z)| + |F_2(x) - F_2(z)|}{ |x_1 - z_1|^{1/2}}
\leq  \B( C_3 \log \f{b}{a} \cdot |x_1 - z_1|^{1/2} + 2 C_4 |a| |x_1 - z_1|^{-1/2} \B) || f||_{\inf}.  	
\]

We optimize the estimates by choosing $a = C_5 |x_1  -z_1|$ for some constant $C_5$ depending on $C_3, C_4$. Then we establish the estimate. The above simple estimates show that the choice of $a$ depends on $|x-z|$. Thus, in our later H\"older estimates, we perform decomposition guided by the above estimates and optimize the choice of size of the singular region $[-a, a]^2$. 
On the other hand, since for different $|x-z|$ we need to choose different $a$, it increases the technicality of the computer-assisted estimates.

\subsection{$L^{\inf}$ estimate}\label{sec:vel_linf}

Let $\wh {u_{x}}$ be the approximation term of $u_x$ (see Section 4.3 of Part I \cite{ChenHou2023a}). We focus on the estimate of the piecewise $L^{\inf}$ norm of $u_{x, A} = u_x - \wh { u_{x}}$, 
which is a representative case. For simplicity, we assume the rescaling factor $\lam = 1$. We assume that $x$ satisfies \eqref{eq:int_loc_x} without loss of generality. We want to estimate $u_{x, A}$ for \textit{all} $x \in B_{i_1 j_1}(h_x)$.

We can write $ u_{x, A} = u_x - \hat u_x$ as follows 
\[
u_{x, A} = \int ( K(x-y)  - \hat K( x, y)  ) W(y) dy, \quad K_A \teq K(x- y ) - \hat K(x, y),
\]
where $\hat K(x, y)$ is the kernel for the approximation term and $W$ is the odd extension of $\om$ (see \eqref{eq:ext_w_odd}). From Sections 4.3.2 and 4.3.3 of Part I \cite{ChenHou2023a}, we remove the singular part in $\hat K$, and then $\hat K$ is nonsingular. Given $x$ with \eqref{eq:int_loc_x}, similar to \eqref{int:decom_ux}, for $k \geq k_2$, we perform the following decomposition 
\beq\label{int:decom_linf_ux}
\bal
u_{x, A}  &= ( \int_{R(k)^c} + \int_{ R(k) \bsh R_s(k_2) } + 
\int_{ R_s(k_2)} )   K(x-y) W(y) dy - \int \hat K(x, y) W(y) dy \\
&\teq I_1 + I_2 + I_3 + I_4.
\eal
\eeq
where $R_s(k)$ is the symmetric singular region \eqref{eq:rect_Rsk}. See Section \ref{sec:int_linf_para} for the choice of $k$.

Since $I_1 + I_4$ is nonsingular, we use the ideas in Section \ref{sec:int_sym} to symmetrize the kernels in $I_1 + I_4$. Then we use the method in Section \ref{sec:T_rule} to estimate it.

\begin{remark}\label{rem:int_indic}
In our computation, the domain $[0, D]^2 \cap R(k)^c$ can be decomposed into the union of small grids $[y_i, y_{i+1}] \times [y_j, y_{j+1}]$ \eqref{eq:int_mesh_y} since the boundary of $R(x, k)$ aligns with the mesh \eqref{eq:rect_Rk}. In particular, in each grid, the indicator function is constant, and the integrand is smooth in $y$. 

\end{remark}

Next we consider $I_2$. The domain of the integral is close to the singularity. If we use the method in Section \ref{sec:T_rule} to estimate it, the error will be quite large since $\pa^2 K(x-y)$ is very singular. 
We want to estimate $I_2$ using $|| W \vp||_{\inf}$ and the singular part $I_3$ using $[ W \psi_1 ]_{C^{1/2}}$. Since $K(z) $ is singular of order $-2$, we expect an estimate 
\[
|I_2| + |I_3| \les \log \f{k}{k_2} \vp^{-1}(x)  || W \vp||_{L^{\inf}[R(k)]} + \psi^{-1}(x) k_2^{1/2} [ W \psi]_{C_x^{1/2}}.
\]
Note that the weights $\vp, \psi$ have a different order of singularity for small $x$ and a different rate of decay. Moreover,
we need to control the right hand side using the energy,
which assigns different weights to two norms (seminorms). Thus, to obtain a sharp estimate, we need to optimize the choice of $k_2$.

Firstly, we consider $k_2 = 2, 2 + \f{1}{2}, .., k$, we use the method in Section \ref{sec:int_nsym2} to estimate $I_2$. We also consider very small $k_2 < 2$. In this case, we further decompose $I_2$ as follows 
\[
I_2 = ( \int_{R(k) \bsh R_s(2)} + \int_{ R_s(2) \bsh R_s(k_2)} ) K(x- y) W(y) dy \teq I_{21} + I_{22}.
\]

For $I_{21}$, we apply the method in Section \ref{sec:int_nsym2}. For $I_{22}$, we use a change of variables $y = x+ s h$ 
\[
|I_{22}|  = \B| \int_{ k_2  \leq |s_1| \vee |s_2| \leq 2 } K(-sh) W( x + s h ) h^2 ds \B| .
\]
Since the region is very small, $x+ s h \in B_{i_1 j_1}(h_x) + [-2h, 2h]$, and $K_1( h s) = h^{-2} K_1(s)$, we get 
\[
|I_{22}| \leq || W \vp||_{\inf}  || \vp^{-1}||_{L^{\inf}(B_{i_1 j_1}(h_x) + [-2h, 2h]) } 
\int_{ k_2  \leq |s_1| \vee |s_2| \leq 2 } |K(s)| ds.
\]
The integral can be computed explicitly and has the order $\log \f{2}{k_2}$.  

It remains to estimate the most singular part $I_3$ for different $k_2$. Using a change of variables $y = x + s h$, the scaling symmetries, and the above derivations, we get 
\[
I_3 = \int_{ [-k_2, k_2 ]^2} K(-s) W(x+ s h) ds.
 \]
To use the H\"older norm of $ W \psi$, we decompose it as follows
\beq\label{eq:ux_sing_bd}
I_3 =  \int_{ [-k_2, k_2 ]^2}  K(-s)   (W\psi)(x+ s h) ( \f{ 1}{\psi(x + sh)} -  \f{1}{\psi(x)}  )
+  K(-s)  \f{ (W \psi)(x+sh) }{ \psi(x)}  ds \teq I_{31} + I_{32}.
\eeq

For $I_{32}$, using the H\"older seminorm, the odd symmetry of $K(s) = c \f{s_1 s_2}{|s|^4}$ in $s_1$, and $|(W\psi)(x + sh) - (W\psi)(x - sh)| \leq \sqrt{2s_1 h}$, we get 
\[
|I_{32} | \leq \f{ h^{1/2}}{\psi(x)} [W \psi]_{C_x^{1/2}} \int_{ [0, k_2] \times [-k_2, k_2]} | K(s) | \sqrt{ 2 s_1}  ds
=  \f{ 2 k_2^{1/2}h^{1/2}}{\psi(x)} [ W \psi]_{C_x^{1/2}} \int_{ [0, 1]^2} | K(s) | \sqrt{ 2 s_1}  ds,
\]
where we used the scaling symmetry of $K$ and a change of variables $s \to k_2s$ in the last equality.

\subsubsection{The commutator}\label{sec:int_linf_ux_com}
For $I_{31}$, we apply the simple Taylor expansion to $f = \psi^{-1}$
\beq\label{int:linf_ux_com1}
 |f(x+ s h) - f(x)| \leq | f_x(x) h s_1 + f_y(x) h s_2| + h^2( \f{m_{20} s_1^2 }{2} + m_{11} s_1 s_2 + \f{ m_{02} s_2^2}{2}  ) , 
  \eeq
where $m_{ij}$ is the bound for the second derivatives of $\psi^{-1}$ 
\[
m_{ij}( s) = \max_{  B_{i_1 j_1 }( h) + I( \sgn(s_1))  \times I( \sgn(s_2))   } || \pa_x^i  \pa_y^j ( \psi^{-1})||_{L^{\inf}  }, 
\quad I_+ = [0, k_2 h], \quad I_- = [-k_2 h,0].
 \]

Note that $m_{ij}$ is constant in each quadrant of $[-k_2  , k_2 ]$. We plug in the expansion \eqref{int:linf_ux_com1} to estimate $I_{31}$. We only discuss a typical term $m_{20} s_1^2 h^2 $
\[
I_{31,20} \teq h^2 \int_{ [-k_2 ,k_2 ]^2} |K(-s) (W\psi)(x+s h) | m_{20}(s) \f{s_1^2}{2} ds .
\]

If $k_2 \geq 2$, we can further partition $[-k_2, k_2]^2$ into $B_{2p, 2q}(1/2)=[p, p+1/2] \times [q, q + 1/2], -k_2 \leq p, q \leq k_2 -1/2$, where we use the notation \eqref{int:box_B}.  For each grid $B_{2p, 2q}(1/2)$, the sign of $s$ and $m_{20}(s)$ are fixed, and we have 
\[
\int_{B_{2p, 2q}( \f{1}{2} )} |K(-s)| (W\psi)(x+s h) m_{20}(s) \f{s_1^2}{2} ds 
\leq  m_{20} || W \vp||_{\inf} \int_{B_{2p, 2q}( \f{1}{2}) } \f{ |K(s) |s_1^2}{2}  (\f{ \psi}{\vp}) (x+sh) ds.
\]
The last integral can be estimated using the method in \eqref{eq:int_sing_nsym1_Q}. Combining the estimate of integral in different regions $B_{2p, 2q}(1/2)$, we obtain the estimate of $I_{31, 20 }$. Similarly, we can estimate the contributions of other terms in \eqref{int:linf_ux_com1} to $I_{31}$. 

For small $k_2 \leq 2$, we do not partition the domain. We denote $D(k_2) = B_{i_1, j_1}(h_x) + [-k_2 h, k_2h]^2$. For $s \in [-k_2, k_2]$, we use $x+ sh \subset D(k_2) \subset D(2)$ to get
\beq\label{int:linf_ux_com2}
 |f(x+ s h) - f(x)| \leq || f_x||_{L^{\inf}(D(k_2))} s_1 h
 + || f_y||_{L^{\inf}(D(k_2))} s_2 h. 
 \quad  | W \psi(x + s h)| \leq || W \vp||_{\inf} ||\f{\psi}{\vp}||_{L^{\inf}( D(2) )}.
\eeq
Plugging the above estimate into $I_{31}$, we get 
\[
I_{31} \leq \sum_{ (i,j) = (1,0), (0,1) }
h ||\pa_x^i \pa_y^j ( \psi^{-1})||_{L^{\inf}(D(k_2))} || W \vp||_{\inf}  || \f{\psi}{\vp} ||_{L^{\inf}( D(2))}
 \int_{ [-k_2, k_2]^2} |K(s) s_1^i s_2^j | ds .
\]
Using the scaling symmetry, we can reduce the last integral to $ k_2^{i+j} \int_{ [-1,1]^2}
|K(s) s_1^i s_2^j | ds $.

We apply the above estimates to a list of $k_2$, and bound different norms using $\max(  || \om \vp||_{\inf} $ , $\max_i \g_i  [ \om \psi_1]_{C_{x_i}^{1/2}(\R_2^{+}) } )$. 
Then by optimizing the $k_2$, we obtain the sharp estimate of $u_{x, A}$.

In \eqref{int:linf_ux_com1}, we do not bound $f(x + sh) - f(x)$ directly using the estimate \eqref{int:linf_ux_com2} since $s$ is large. Instead, we perform a higher order expansion.

\vs{0.1in}
\paragraph{\bf{Estimate of $u_y, v_x$}}
The estimates of $u_y, v_x$ follow similar strategies and estimates. The only difference is the 
estimate of the most singular term similar to $I_{32}$ \eqref{eq:ux_sing_bd} for $u_y, v_x$ due to different symmetry property of the kernel. We estimate it using a combination of norms $|| \om \vp||_{\inf}$, and semi-norms $[ \om \psi]_{C_{x_i}^{1/2}}$, and refer it to Section {\seclinfsing} in the supplementary material II \cite{ChenHou2023bSupp}.

\subsubsection{Estimate of $\uu_A$}
The estimate of $\uu_A$ is much simpler since it is more regular. Let $K$ and $\hat K$ be the kernel of $u, v$ and its approximation term, respectively. For $f = u$ or $v$, we perform a decomposition similar to \eqref{int:decom_linf_ux}
\beq\label{int:decom_linf_u}
f_{A} = ( \int_{R(k)^c} + \int_{ R(k) \bsh R_s(k) } + 
\int_{ R_s(k)} )   K(x-y) W(y) dy - \int \hat K(x, y) W(y) dy
\teq I_1 + I_2 + I_3 + I_4.
\eeq
The estimates of $I_1 + I_4$ follow the method for $u_{x, A}$. For $I_2$, we use the method in Section \ref{sec:int_nsym1}. For $I_3$, since $K$ has a singularity of order $|x|^{-1}$, which is locally integrable, we use a change of variable $y = x + sh$ to obtain
\[
I_3 =  h \int_{[-k, k]^2} K(-s) W(x + s h) ds .
\]
Then we partition  $[-k, k]^2$ into small grids, and use the method in \eqref{eq:int_sing_nsym1_Q} to estimate the integral in each grid. Here, we get a factor $h$ in the change of variables since $K(\lam s) = \lam^{-1} K(s)$.

\subsubsection{Choice of parameters}\label{sec:int_linf_para}

Recall the choice of several parameters $a, h, h_x$ from \eqref{int:para1}. We choose $ 3 \leq k \leq 10 $. We choose $k$ for the size of the singular region $kh$ \eqref{int:decom_linf_ux}, \eqref{int:decom_linf_u}  not so small such that the error $h^2 \pa^2 K$ in Lemma \ref{lem:T_rule}, which has the order $h^2 |x-y|^{-\al-2}$ near the singularity, is smaller than the main term $K$, which has the order $|x-y|^{-\al}, \al = 1, 2$. Since we will estimate $I_1 + I_4$, $I_2$, $I_3$ in the decomposition separately using the triangle inequality, we do not choose $k$ to be too large so that we can exploit the cancellation in $I_1 + I_4$.

\subsection{H\"older estimates}\label{sec:vel_hol_comp}

We want to estimate $\f{ | f(x) - f(z)| }{ |x - z|^{1/2}}$ for any $x, z \in \R_2^{++}$ with $x_1 = z_1$ or $x_2 = z_2$ and some function $f$, e.g. $f = u_{x, A}$. Without loss of generality, we assume $|z| > |x|$. Then in the $C_x^{1/2}$ estimate, we have $x_1 < z_1, x_2 = z_2$; in the $C_y^{1/2}$ estimate, we have $x_1 = z_1, x_2 < z_2 $. 
Applying the rescaling argument in Section \ref{sec:int_method}, we can restrict $\hat x = \f{x}{\lam}$ to $  \hat x \in [0, 2 x_c ]^2 \bsh [0, x_c]^2 $. For this reason, we assume $\lam = 1$ for simplicity.  We will only estimate the H\"older difference for comparable $x, z$:  $|x| \asymp |z|$. If $|z| \gg |x|$, we simply apply the $L^{\inf}$ estimate to $f(x), f(z)$ and use the triangle inequality.

We focus on the H\"older estimate  of $u_{x, A}$, which is a representative and the most important nonlocal term to estimate in our energy estimate. 

\subsubsection{$C_x^{1/2}$ estimate}

Recall $I_i$ from the decomposition \eqref{int:decom_ux} and $K_1(s) = \f{s_1 s_2}{|s|^4}$. We apply the same decomposition to $u_{x, A}(z)$. We assume that the approximation term $\hat u_{x}$ (see Section {\secapprvel} of Part I \cite{ChenHou2023a}) takes the following form 
\beq\label{int:decom_ux_appr}
\hat u_x(x) = \int \hat K_1(x, y) W(y) dy , \quad   I_6(x) \teq \psi(x) \hat u_x(x) = \psi(x) \int \hat K_1(x, y) W(y) dy ,
\eeq
with a nonsingular kernel $\hat K_1$. We first discuss how to estimate the regular part $I_1, I_3, I_4$ in \eqref{int:decom_ux} and $I_6$, which are Lipschitz. We will apply the sharp H\"older estimate in Lemmas 3.1-3.5 in Section {\secsharp} of Part I \cite{ChenHou2023a} to estimate the most singular part $I_2$. The most technical part is to estimate $I_5$, which is log-Lipschitz since the kernel $K_1(x-y)(\psi(x) - \psi(y))$ has a singularity  of order $-1$. We assemble the estimates of different parts to estimate $[u_{x, A} \psi ]_{C_x^{1/2}}$ in Section \ref{sec:hol_comb}.

\subsubsection{Estimates of the regular terms $I_1, I_3, I_4, I_6$}\label{sec:int_hol_reg}

Recall $I_1, I_3, I_4$ from \eqref{int:decom_ux} and $I_6$ from \eqref{int:decom_ux_appr}. Since the integrands in $I_1, I_3, I_4$ are supported at least $k_2 h$ away from the singularity $x$, if $W$ is in some suitable weighted $L^{\infty}$ space,  $I_1, I_3, I_4$ 
are piecewise smooth and their derivatives can be bounded by $|| W \vp||_{\inf(\R_2^{++})}  = || \om \vp||_{\inf}$.
Their derivatives jump when $R(x, k), R(x, k_2)$ change, or equivalently, $x$ moves from one grid to another. For $x \in B_{i_1, j_1}(h_x)$ \eqref{eq:int_loc_x}, these rectangle domains are the same, and these functions are smooth. The approximation term $I_6$ \eqref{int:decom_ux_appr} is locally smooth in $x$. To exploit the cancellation, we combine the estimates of $I_1, I_4, I_6$ together. We symmetrize the kernel in  $I_1(x) + I_4(x) - I_6(x)$ following Section \ref{sec:int_sym} and use the method in Section \ref{sec:T_rule} to estimate the derivatives of $I_1(x) + I_4(x) - I_6(x)$. See also \eqref{int:sym_integ1}, \eqref{int:sym_integ2} for the form of the symmetrized integrands in these integrals.

We estimate the piecewise Lipschitz norm of $I_3$ using the method in Sections \ref{sec:int_nsym1}, \ref{sec:int_nsym_bd}. We choose integer $k, k_2$ in the decomposition \eqref{int:decom_ux} . Then in each grid $[y_i, y_{i+1}] \times [y_j, y_{j+1}]$, the indicator functions in $I_1 + I_4 - I_6$, e.g. $\one_{R(k)^c}, \one_{ R(k) \bsh R(k_2)}$, are constant. See Remark \ref{rem:int_indic}. We will combine the estimates of different terms in Section \ref{sec:hol_comb}, e.g. $I_1 + I_4 -I_6, I_3$ and part of $I_{5\cdot}$ defined later in \eqref{eq:int_holx_I5}, and obtain some H\"older continuous functions when $x$ moves from one grid to another.   
We assemble the H\"older estimates in Section \ref{sec:hol_comb}.

\subsubsection{$C_x^{1/2}$ estimate of $I_2$}\label{sec:int_hol_I2}

We first estimate the second term $I_2$ in \eqref{int:decom_ux}. Recall $R(x, k), R_{s1}(x,k), R_s(x, k)$ from \eqref{eq:rect_Rk}, \eqref{eq:rect_Rsk} and the location of $x$ \eqref{eq:int_loc_x}. We have 
\[
x_2 - (j-k)h \leq (j+1 ) h - (j -k) h = (k+1) h, \quad (j+1 + kh) - x_2 \leq  (j+1 + kh) - j h = (k+1) h.
\]
Since $x_2 = z_2$, using Lemma 3.1 from Section {\secsharp} of Part I \cite{ChenHou2023a} with $(a, b_1, b_2) = (kh, x_2 - (j-k)h, (j+1 + k) h - x_2) $ and $|b_1|, |b_2| \leq (k+1) h$, we obtain 
\[
\f{1}{|x-z|^{1/2}} | I_2(x, k) - I_2(z, k)| 
\leq C_1( \f{(k+1) h}{ |x-z|}) [ W \psi ]_{C_x^{1/2}} 
=C_1( \f{(k+1) h}{ |x-z|}) [ \om \psi ]_{C_x^{1/2}}.
\]

We only apply the H\"older estimate to $|x-z| \leq \f{ kh}{2}$ (rescaled $x,z$) and the assumption $a \geq \f{1}{2}|x_1-z_1|$ in Lemma 3.1 in Part I \cite{ChenHou2023a} is satisfied. For $I_2(x, k)$ associated with other terms $u, v, u_y, v_x$, we can estimate it using similar ideas and Lemmas 3.1-3.5 in Part I \cite{ChenHou2023a}. The $C_y^{1/2}$ estimate of $I_2(x, k)$
is completely similar. See Section \ref{sec:hol_other} for more details.

\subsubsection{$C_x^{1/2}$ estimate of $I_5$}\label{sec:int_holx_I5}
For $I_5$ \eqref{int:decom_ux}, $K_1(x-y) (\psi(x)-\psi(y))$ is singular of order $-1$ near $y = x$. Given $W \in L^{\infty}(\vp)$, $I_5$ is log-Lipschitz. 
There are several approaches to estimate its H\"older norm, see e.g., Section \ref{sec:int_loglip}. 
We use part of the $C_x^{1/2}$ seminorm of $\om$ to get a better estimate. We choose $k_3 = k_2 - \f{i}{2} \geq 2, i=0,1,2,.., 2k_2-4$ and further decompose $I_5$ as follows 
\beq\label{eq:int_holx_I5}
\bal
I_5(x, k_2) 
&=\B( \int_{R(k_2) \bsh R(k_3)} 
+  \int_{R(k_3)  \bsh R_{s,1}(k_3)}  + \int_{  R_{s,1}(k_3)} \B)  K_1(x-y) ( \psi(x)- \psi(y)) W(y) dy  \\
& \teq I_{5,0}(x, k_2, k_3) +  I_{5,1}(x, k_3) + I_{5,2}(x, k_3).
\eal
\eeq
The domain in $I_{5, 0}$ depends on $x$. For $x$ in a grid cell, it does not change with $x$.
We estimate $\pa_{x_1} I_{5,0}$ using Taylor expansion in Section \ref{sec:tay} and following the method in Section \ref{sec:int_nsym_bd0}.
We estimate the $x$-derivative of $I_{5, 1}$ using the method in Sections \ref{sec:int_nsym1}, \ref{sec:int_nsym_bd}. We have
\beq\label{eq:int_holx_I51}
\bal
\pa_{x_1} I_{5, 1}
& = \int_{ R(k_3) \bsh R_{s, 1}(k_3)}
 \pa_{x_1} \B( K_1(x-y) ( \psi(x)- \psi(y)) \B)  W(y) dy  \\
& \quad - \int_{ (j-k_3 )h }^{ (j+1 + k_3) h} K_1(x-y)  (\psi(x) -\psi(y)) W(y) \B|_{y_1 = x_1 -k_3 h }^{x_1 + k_3 h} d y_2  .
\eal
\eeq
We estimate the first part following Section \ref{sec:tay}, and the second part following Section \ref{sec:int_nsym_bd}.

For $I_{5,2}$, we will estimate it using a method similar to that of $I_2$. See the left figure in Figure \ref{fig:hol_ux} for the domains of the integrals in $I_{5,2}(x), I_{5,2}(z)$. The integrand satisfies
\[
\bal
K_1(x-y)  ( \psi(x)- \psi(y)) W(y) 
&=  \psi(x) K_1(x- y) ( \psi^{-1}(y) - \psi^{-1}(x) ( W \psi)(y) \\
&\approx \psi(x) \pa_i (\psi^{-1}(x)) \cdot K_1(x- y) (y_i - x_i) ( W \psi)(y) .
\eal
 \] 
 Thus, $I_{5,2}(x)$ can be seen as a weighted version of $I_2$ \eqref{int:decom_ux} with a weight $\psi(x) \pa_i (\psi^{-1}(x))$, a more regular kernel $K_1(x-y)(y_i - x_i)$, and a smaller domain $R_{s, 1}(k_3)$. Since the kernel is more regular and the domain is smaller, our estimate for $I_{5,2}$ is much smaller than that of $I_2$. 

Now, we justify this approach. Using a change of variables $y = x + s, s \in R_{s,1}(k_3) - x$ 
and the above identity, we yield
\[
I_{5,2}(x, k_3) =    \psi(x) 
 \int_{ R_{s,1}(k_3) - x} K_1(-s)  ( \psi^{-1}(x+s) - \psi^{-1}(x)) (W\psi) (x+s) ds.
\]

Using Newton's formula $f(1) = f(0) + f^{\pr}(0) + \int_0^1 (1-t)f^{\pr \pr}(t) dt$ for $f(t) = \psi^{-1}(x + t s)$, we get 
\[
\bal
\psi^{-1}(x+s) - \psi^{-1}(x)
&=s \cdot \na \psi^{-1}(x) + \int_0^1 (1 - t) \B( s \cdot (\na^2\psi^{-1}) (x + ts) \cdot s \B) dt  \\
&= \sum_{i=1,2} s_i \pa_i (\psi^{-1}) (x)
+ \sum_{ 0\leq i \leq 2}  \binom{2}{i} s_1^i s_2^{2-i} \int_0^1 (1- t)  \pa_1^i \pa_2^{2-i} (\psi^{-1})( x + ts )d t.
\eal
\]

Denote 
\[
\bal
& Q_{ij}(x) = \psi(x) \int_0^1 (1- t) \pa_1^i \pa_2^j (\psi^{-1}) ( x + ts )d t , \ i+ j = 2,
\quad  D(x) = R_{s,1}(x, k_3) - x, \\
& Q_{ij}(x) = \psi(x) \cdot \pa_1^i \pa_2^j( \psi^{-1})(x) = -\f{ \pa_1^i \pa_2^j \psi(x)}{ \psi(x)}, \ i+j = 1,  \quad P_{ij}(x) = \int_{D(x)} K_1(-s) s_1^i s_2^j ( W \psi) (x + s) ds.\\
\eal
\]

Using the above expansion and notations, we get 
\[
I_{5,2}(x, k_3) = \sum_{i+j = 1} P_{ij} Q_{ij}
+ \sum_{ i + j = 2} \binom{2}{i} P_{ij} Q_{ij}.
\]

Next, we use the above decomposition to estimate $I_{5,2}(x, k_3) - I_{5,2}(z, k_3)$. The leading order terms are $P_{ij} Q_{ij}$ with $i+j =1$. By definition of $R_{s, 1}$ \eqref{eq:rect_Rsk}, we observe that if  $x_2 = z_2$, we have 
\[
D(x) = R_{s,1}(x, k_3) - x = R_{s,1}(z, k_3) - z = D(z) .
\]
Suppose that $x_1 < z_1$. We perform a decomposition 
\beq\label{int:holx_ux_J}
\bal
&|P_{ij}(x) Q_{ij}(x) - P_{ij}(z) Q_{ij}(z)|
\leq J_1 + J_2,  \\
& J_1 \teq  | Q_{ij}(z)  ( P_{ij}(x) - P_{ij}(z)  ) |, \quad
J_2 \teq |P_{ij}(x)  ( Q_{ij}(x) - Q_{ij}(z) ) | .
\eal
\eeq
Using $D(x) = D(z)$, we bound $J_1$ as follows 
\[
\bal
|J_1| & \leq |Q_{ij}(z)| \B| \int_{D(x)} K_1(-s) s_1^i s_2^j ( ( W \psi )(x+s) - (W\psi)(z+s)  ) ds \B|  \\
& \leq | Q_{ij}(z)| \cdot |x-z|^{1/2}  || \om \psi||_{C_x^{1/2}}
\int_{s \in D(x)} | K_1(s) s_1^i s_2^j | ds .   \\
\eal
\]

The term $Q_{ij}$ only depends on the weight and is smoother than $P_{ij}$. 
We can estimate $Q_{ij}(x) - Q_{ij}(z)$ by bounding $\pa_1 Q_{ij}$ since $Q_{ij}$ is locally smooth. For $P_{ij}$ in $J_2$, we use the method in \eqref{eq:int_sing_nsym1_Q} to bound it by $C || \om \vp||_{\inf}$ with some constant $C$. Then we obtain the estimate
\[
|J_2| \leq  C_2  |x-z| \cdot || \om \vp||_{L^{\infty}}
\]
for some constant $C_2$. Note that the second order term $P_{ij} Q_{ij}, i+j = 2$ is much smaller than the leading order terms. For $|x-z|$ not too small, we can estimate its contribution trivially 
\beq\label{eq:PQ_linf}
\f{1}{|x-z|^{1/2}} | P_{ij}(x) Q_{ij}(x) - P_{ij}(z) Q_{ij}(z) |
\leq \f{1}{|x-z|^{1/2}} ( | P_{ij}(x) Q_{ij}(x)| + |P_{ij}(z) Q_{ij}(z) | ).
\eeq
We optimize the above two estimates. 

In summary, to obtain the above estimates, we estimate piecewise bounds for $|Q_{ij}(x)|$, $P_{ij}(x), | \pa_k Q_{ij}(x)| $, and the integrals $\int_{D(x)} | K_1(s)  s_1^i s_2^j | ds, \  i+ j =1,2 .$

\begin{figure}[t]
\centering
\begin{subfigure}{0.47\textwidth}
  \centering
  \includegraphics[width=0.9\linewidth]{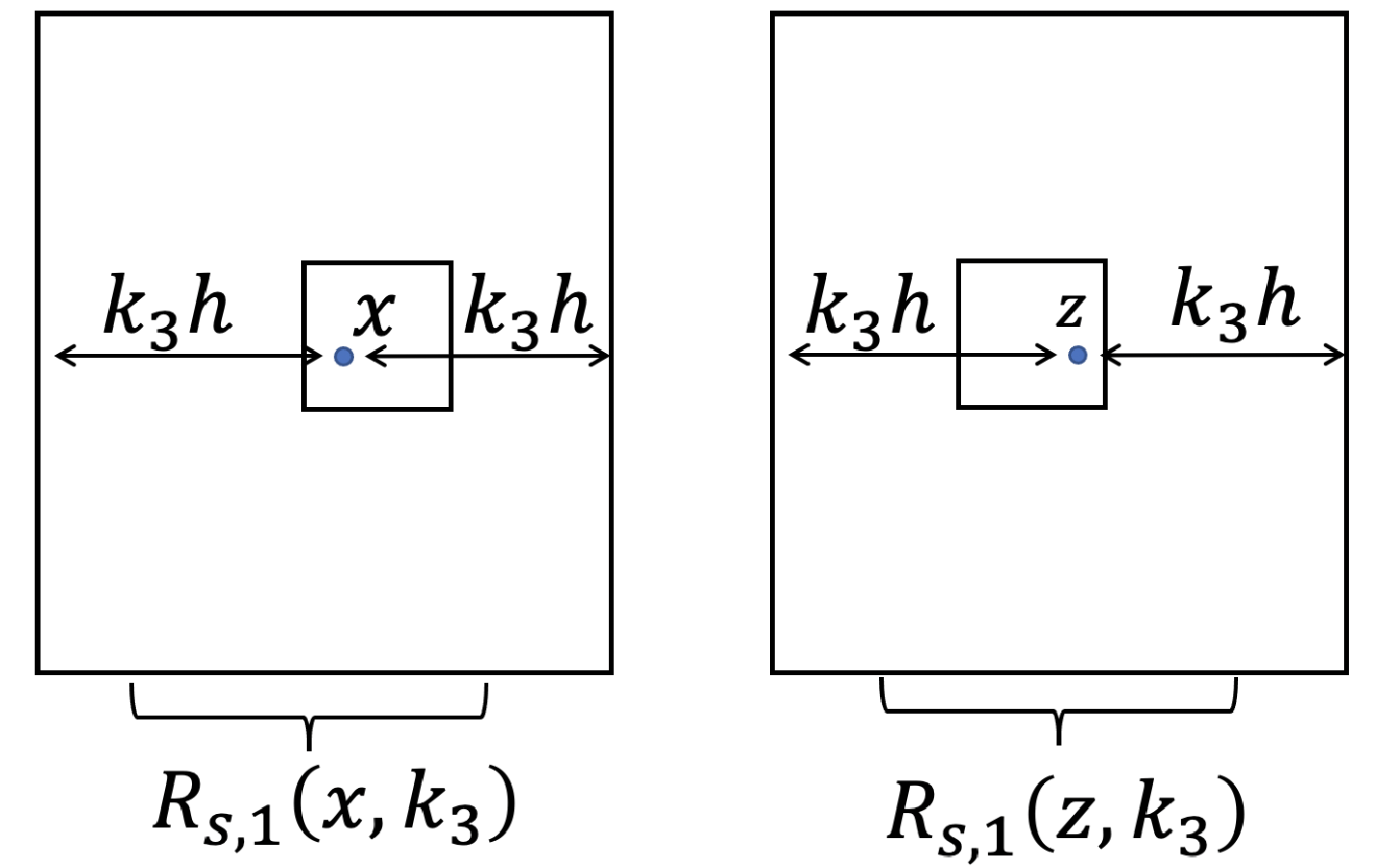}
\end{subfigure}\begin{subfigure}{0.53\textwidth}
  \centering
  \includegraphics[width=0.85\linewidth]{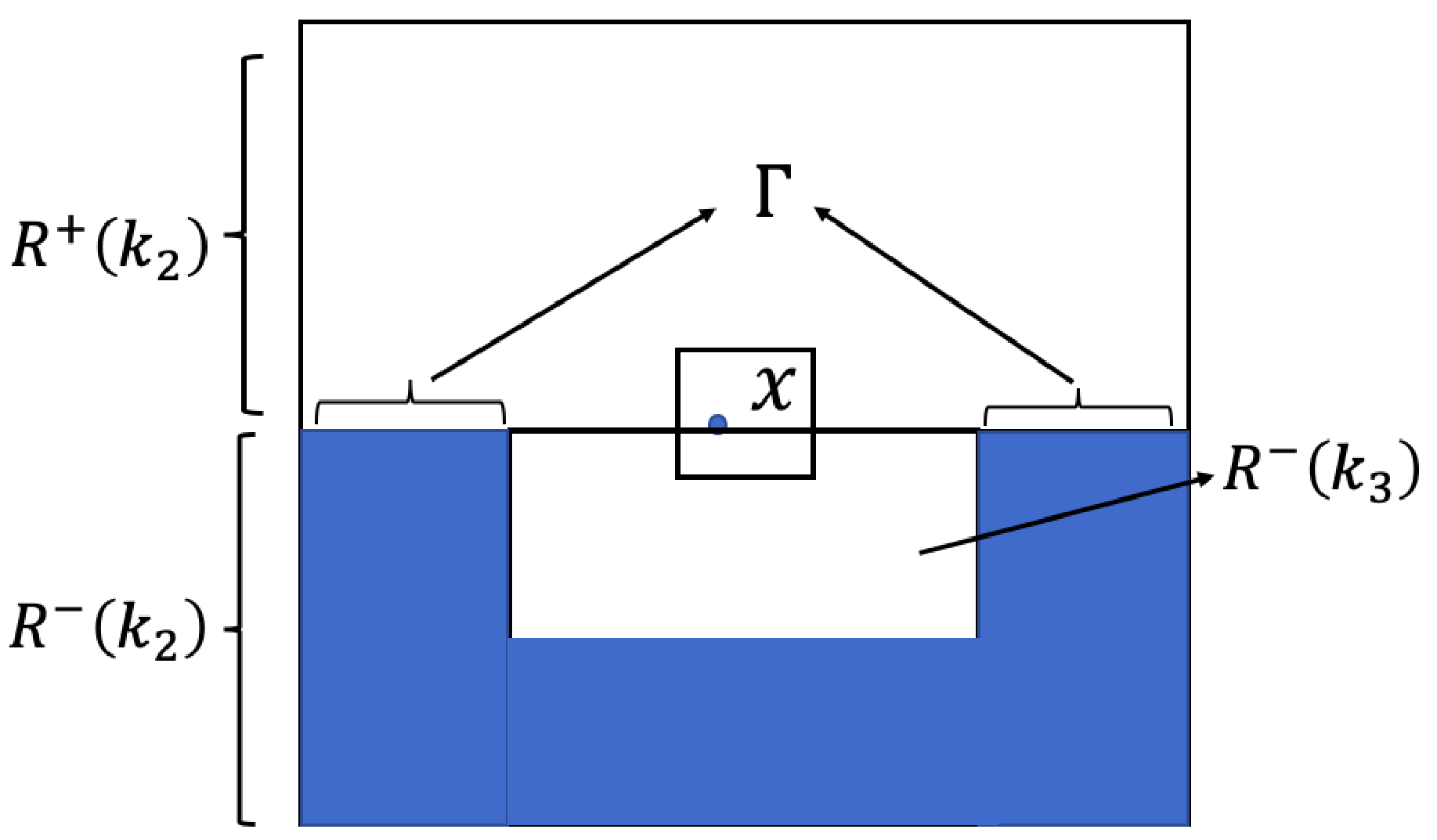}
\end{subfigure}
\caption{Left: $R_{s,1}(x, k_3)$ and $R_{s,1}(z, k_3)$ with $x_2 = z_2$. The small square is a mesh grid  containing $x$ or $z$. $x, z$ can have different locations relative to the grids.
Right: The large rectangle is $R(k_2)$, the upper part is $R^+(k_2)$, and the lower part is $R^-(k_2)$. The blue region is $R^-(k_2) \bsh R^-(k_3)$. $\Gamma$ is part of its boundary.} 
\label{fig:hol_ux}
\end{figure}

The above estimate of $I_{5}(x, k_2)$ can be generalized to the $C_x^{1/2}$ estimate of $u, v, v_x, u_y$. Yet, it does not apply to the $C_y^{1/2}$ estimate of $\uu, \na \uu$ since it requires the estimate of $ (W \psi)(x + s) - (W\psi) (z+s)$ for $s$ in some rectangle $R= D(x) = D(z)$. However, since $W$ is discontinuous across the boundary $y = 0$, $W\psi \notin C_y^{1/2}(R)$ if $x+s, z+s$ are not in the same half plane. If $x_1 < x_2$, then the rectangles $R(x, k_2), R(z,k_2)$ will not intersect the boundary and the previous estimate holds true. If $x_1 > x_2$, we consider two modifications for different kernels in the following subsections.

\subsubsection{Ideas of the $C_y^{1/2}$ estimates of $I_5$}\label{sec:int_holy_idea}

The main idea in the following $C_y^{1/2}$ estimates is to use a combination of the estimates for the log-Lipschitz function in Section \ref{sec:int_loglip} and the estimate in Section \ref{sec:int_holx_I5}. The latter provides better estimates, and we try to use this method \textit{as much as possible}. Following the ideas in Section \ref{sec:int_loglip}, we decompose $I_5(x)$ into the singular part and nonsingular part with different size $k_3$ of the singular region 
\[
I_5(x) = I_{5, S}(x, k_3) + I_{5, NS}(x, k_3). 
\]

Although we cannot apply the second method to the whole $I_5(x)$, 
we can apply it to the integrals in the upper part of the regions, e.g. $R^+(k_2), R^+(k_3)$ \eqref{eq:rect_Rk+}, since these integrals only involve $W \psi$ in $\R_2^+$ and we have $W \psi \in C^{1/2}$. Thus, we will further decompose some of the regions into the upper part and the lower part, and then apply the first method to the lower part, and the second method to the upper part.

\subsubsection{$C_y^{1/2}$ estimate of the velocity with a kernel of the first type}\label{sec:int_holy_K1}

The kernels 
\beq\label{int:ker_type1}
K = \f{y_1 y_2}{ |y|^4},\quad  \f{y_2}{|y|^2} 
\eeq
associated with $u_x = -\pa_{xy} (-\D)^{-1}\om, u =-\pa_y (-\D)^{-1} \om$ vanish when $y_2 = 0$. We call them the first type kernel. Let $K$ be a kernel of the first type. We use the following decomposition 
\beq\label{int:holy_sing0}
\bal
I_5(x, k_2) & = ( \int_{R^+( k_2)} + \int_{R^-(k_2)} ) K(x-y) (\psi(x) - \psi(y) ) W(y) dy 
\teq I_5^+(x, k_2) + I_5^-(x, k_2) 
\eal
\eeq
See the right figure in Figure \ref{fig:hol_ux} for $R^{\pm}(k_2)$. Since $ R^+(x, k_2), R^{+}(z, k_2) \subset \R_2^+$, we can decompose 
\[
I_5^+ =  I_{5, 1}^+ + I_{5,2}^+
\]
into the integral in the regions $I_{5, 1}^+: R^+(k_2)\bsh R_{s, 2}^+(k_2)$ and 
$I_{5, 2}^+:R_{s, 2}^+(k_2) $, and apply the same argument as that for $I_{5,1}(x, k_3), I_{5,2}(x, k_3)$ in Section \ref{sec:int_holx_I5} to obtain the desired estimates by restricting all the derivations in $R^+(x, k_2), R^+(z, k_2)$. Note that here, we do not further choose smaller window $R^+(x, k_3)$ to decompose $I_5^+(x, k_2)$, i.e. $k_3 = k_2$ and $I_{5,0} = 0$ in \eqref{eq:int_holx_I5}. For $I_{5, 1}^+$, similar to \eqref{eq:int_holx_I51}, we get a boundary term from $\pa_2 ( R^+(k_2) \bsh R_{s2}^+(k_2)) = [ (i-k_2) h, (i+1 + k_2 h) ] \times \{  x_2 + k_2 h \} $. See \eqref{eq:rect_Rsk}, \eqref{eq:rect_Rk} for $R^+(k), R^+_{s2}(k)$.

For the lower part $I_5^-(x, k_2)$, it is log-Lipschitz if $W \in L^{\infty}(\vp)$. We cannot bound its derivative using $|| W \vp||_{\inf}$. We face the difficulty discussed at the beginning of Section \ref{sec:vel_comp}.

Alternatively, we follow the ideas in Section \ref{sec:int_loglip}. We decompose it into the smooth part and rough part. We introduce $  0< k_3 < k_2$ and consider the following decomposition 
\beq\label{int:holy_sing1}
\bal
I_5^-(x, k_2)  &= ( \int_{ R^-(k_2) \bsh R^-(k_3)} 
+ \int_{ R^-(k_3) \bsh R_{s,2}^-(k_3) }
+ \int_{ R_{s,2}^-(k_3) } ) K(x-y) (\psi(x) - \psi(y) ) W(y) dy  \\
& \teq I^-_{5, 0}(x, k_2) + I^-_{5,1}(x, k_3) + I^-_{5,2}(x, k_3).
\eal
\eeq

See the right figure in Figure \ref{fig:hol_ux} for an illustration of different domains. Recall that $k_2 \in Z_+$. We choose $k_3 =  k_2 - \f{i}{2} \geq 2, i=0, 1,2.., 2 k_2 - 4$. Since the integrand in $I^-_{5,0}$ supports at least $k_3 h$ away from the singularity, $I^-_{5,0}(x, k_2)$ is piecewisely smooth. We can estimate $ \pa_{x_2} I^-_{5, 0}(x, k)$ following Sections \ref{sec:int_nsym_bd0}, \ref{sec:tay}. The domain $R^-(k_2) \bsh R^-(k_3)$ is not piecewise constant since the upper part of its  boundary, i.e. 
\[
  \G = \{ (y_1, x_2) :  y_1 \in [ (i- k_2) h , (i+1 + k_2)  h]   \bsh [ (i- k_3) h , (i+1 + k_3)  h] \}, 
\]
depends on $x_2$. See Figure \ref{fig:hol_ux} for an illustration of $\G$.  Taking $x_2$ derivative on $I_{5,1}^-$, we get 
\beq\label{eq:int_K1_I51n}
\bal
|\pa_{x_2} I^-_{5, 0}(x, k_2) |
&\leq \B| \int_{ R^-(k_2) \bsh R^-(k_3)}  \pa_{x_2} J(x, y)  W(y) dy \B| +  \B| \int_{y \in  \G}  J(x, y) W(y)  d y_1 \B|, \\
 J(x,y) & =  K(x-y) (\psi(x) - \psi(y) ) 
\eal
\eeq

Since $y \in \G \subset \{ y : y_2 = x_2\}$ and that $K(y_1, 0 ) \equiv 0$, the second term vanishes. The first term can be estimated using a change of variables $y = x + s$ and the method in Section \ref{sec:tay}, Section \ref{sec:int_nsym_bd0}, 
since its support is at least $k_3 h$ away from the singularity.

For $I^-_{5,1}$, it is also piecewise Lipschitz, we estimate the $x_2$ derivative similar to $I_{5,1}$ in \eqref{eq:int_holx_I51}
\beq\label{eq:int_K1_I51n_nsym}
\bal
 |\pa_{x_2} I^-_{5, 1} |
 & \leq 
  \B| \int_{ R^-(k_3) \bsh R_{s, 2}^-(k_3)} \pa_{x_2} J(x,y) W(y)  dy \B| 
 +  \B| \int_{ (i-k_3) h }^{ (i+1 + k_3)h }   J(x, y) W(y) \B|_{y_2 = x_2 - k_3 h} d y_1 \B| . \\
 \eal
\eeq
Different from $I_{5,1}$ in \eqref{eq:int_holx_I51}, the boundary term in the above estimate only involves the lower part $y_2 = x_2 - k_3 h$ since the domain in $I_{5,1}^-$ is 
$ R^-(k_3) \bsh R_{s, 2}^-(k_3)$.

For $I^-_{5,2}$, the kernel satisfies $K(x-y) (\psi(x) - \psi(y)) \sim |x-y|^{-1}$ for small $|x-y|$ and is locally integrable. 
We estimate its piecewise $L^{\inf}$ bound using the method in Section \ref{sec:int_linf_ux_com} for the commutator.

The above decomposition can be applied to estimate 
\[
\f{ | I_5^-(x, k_2) - I_5^-( z,k_2) | }{ |x-z|^{1/2}}  
\leq \min_{ k_3  = k_2 - \f{i}{2}}  \f{ | (I_{5,0}^- + I_{5,1}^-) (x, k_3) - (I_{5,0}^-
+ I_{5, 1}^-)(z,k_3) | }{ |x-z|^{1/2}}    +    \f{ | I_{5,2}^-(x, k_3) | + | I_{5,2}^-(z,k_3) | }{ |x-z|^{1/2}}
\]
for $|x-z|$ not too small, e.g. $|x-z| \geq d_s = \f{h}{10} $. When $|x-z|$ is sufficiently small,  the second term in the above estimate can be very large. 

According to the analysis in Section \ref{sec:int_loglip}, for $|x-z|$ very small, we need to choose $k_3 h \sim |x-z|$ to get the sharp estimate. Thus, we consider one more decomposition for $a \leq 1$
\beq\label{int:holy_sing2}
\bal
I_5^-(x, k_2) &= 
\int_{ R^-(k_2) \bsh R_s^-(a)}  K(x-y) (\psi(x) - \psi(y) ) W(y) dy  \\
&+ \int_{  R_s^-(a) }  K(x-y) (\psi(x) - \psi(y) ) W(y) dy  
\teq I^-_{5,3}(x, a) + I^-_{5,4}(x, a).
\eal
\eeq

The above decomposition is slightly different from \eqref{int:holy_sing1}. We choose $R_s^-(a)$ rather than $R^-(a)$, since we need to choose the singular region with size going to $0$ as  $|x-z| \to 0$.
Yet,  $R^-(a)$ \eqref{eq:rect_Rk} does not satisfy this requirement for $a \to 0$.
We can estimate the derivative of $I^-_{5,3}(x, a)$ following Sections \ref{sec:int_nsym1}-\ref{sec:int_nsym_bd}, and the $L^{\infty}$ norm of 
$I^-_{5,4}(x,a)$ following Section \ref{sec:int_linf_ux_com}. Again, in the computation of $\pa_{x_2} I^-_{5,3}(x, a)$, the boundary term vanishes due to $K(y_1, 0 ) \equiv 0$. In summary, we can obtain the following estimate 
\beq\label{int:log_lip_ux}
|\pa_{x_2} I^-_{5,3}(x,a)| \leq A(x) + B(x) \log(1/a), \quad  |I^-_{5,4}(x,a)| \leq C(x) a h ,
\eeq
for any $a \leq 1$, where $A(x), B(x)$ can be estimated following the method in Appendix \ref{app:loglip}, and the estimate of $C(x)$ follows the method in Section \ref{sec:int_linf_ux_com}.
Using the above estimates and the ideas in Section \ref{sec:int_loglip}, we can estimate $d_y( I_5^-(\cdot, k_2)  , x, z )$ for small $|x-z|$ by optimizing $a$, where $d_y$ is defined below 
\beq\label{eq:hol_dy}
d_y(f, x, z) = |  f(x) - f(z) | |x-z|^{-1/2}.
\eeq
We will assemble these estimates in Section \ref{sec:hol_comb}.

\subsubsection{$C_y^{1/2}$ estimate of the velocity with a kernel of the second type}\label{sec:int_holy_K2}

For the kernels $K_2 = \f{y_1^2 - y_2^2}{ |y|^4}$ and $\f{y_1}{|y|^2}$, they do not vanish on $y_2 = 0$ in general. We call them the second type kernel. 

If we use the strategies in the previous subsection, the boundary term in the computation of $
\pa_{x_2} I^-_{5, 0}(x,k_3), \pa_{x_2} I^-_{5,1}(x, k_3)$ or $\pa_{x_2} I^-_{5,3}(x, k_3)$ does not vanish on $\G$ and can be large. To avoid picking up a boundary term on $\G$ and apply the ideas in Section \ref{sec:int_holy_idea}, we consider another estimate on $I_5(x, k_2)$. For $k_3 = k_2 - \f{i}{2}, i = 0, 1, .., 2k_2-4$, we perform the following decomposition 
\beq\label{int:holy_K2_I5}
\bal
I_5(x, k_2) &= ( \int_{ R(k_2) \bsh R(k_3)} 
+ \int_{R^-(k_3)  \bsh R_{s,2}^-(k_3) } 
+ \int_{R^+(k_3)} + \int_{ R_{s, 2}^-(k_3)}
)
 K(x-y) (\psi(x) - \psi(y) W(y) dy  \\
 & \teq I_{5,0} + I_{5,1} + I_{5,2} + I_{5,3} .
\eal
 \eeq
Following the ideas in Section \ref{sec:int_loglip}, we estimate the derivative of the regular part and then the $L^{\inf}$ norm of the singular part. Indeed, we can estimate the $y$-derivative of $I_{5,0}$ following Sections \ref{sec:tay}, \ref{sec:int_nsym_bd0},
$I_{5,1}$ following the estimates of $I_{5, 1}, I_{5,1}^-$ in 
\eqref{eq:int_holx_I51}, \eqref{eq:int_K1_I51n_nsym}, 
and the $L^{\infty}$ norm of $I_{5,3}$ following Section \ref{sec:int_linf_ux_com}.
 The estimate of $I_{5,1}$ is similar to that of $I_4$ in Section \ref{sec:int_hol_reg}.
For $I_{5,2}$, since $R^+(k_3)$ is in $\R_2^+$, we decompose 
\[
 I_{5, 2} = I_{5, 2, 1} + I_{5,2 , 2}
\]
into the integral in the regions $I_{5, 2, 1}: R^+(k_3)\bsh R_{s, 2}^+(k_3)$ and 
$I_{5, 2, 2}:R_{s, 2}^+(k_3)$, and then estimate them following the method in the estimate of $I_{5, 1}, I_{5,2}$ in Section \ref{sec:int_holx_I5}.

After we estimate these quantities, we can estimate $d_y(I_5, x, z)$ \eqref{eq:hol_dy}
  for $|x-z|$ not too small by optimizing $k_3$.  To estimate $d_y( I_5,x,z)$ \eqref{eq:hol_dy} for sufficiently small $|x-z|$, following \eqref{int:holy_sing2}, we use the following decomposition 
\beq\label{int:holy_loglip1}
\bal
I_5(x, k_2) &= \int_{ R(k_2) \bsh R_s(a)} K(x-y) (\psi(x) - \psi(y) W(y) dy 
+  \int_{ R_s^+(a) } K(x-y) (  \psi(x) - \psi(y) ) W(y) dy \\
&+  \int_{R_s^-(a)}  K(x-y) (  \psi(x) - \psi(y) ) W(y) dy 
\teq I_{5,4} + I_{5,5} + I_{5,6}. 
\eal
 \eeq

Then we estimate the derivative of $I_{5,4}$ and the $L^{\infty}$ norm of $I_{5,6}$ as follows
\beq\label{int:holy_loglip2}
| \pa_{x_2} I_{5,4}| \leq A(x)  + B(x) \log(1/a) ,\quad |I_{5,6}| \leq C(x) a h,
\eeq
where the estimates of $A, B$ are given in Appendix \ref{app:loglip}, and the estimate of $C$ follows the method in Section \ref{sec:int_linf_ux_com}. The H\"older estimate of $I_{5,5}$ follows the method in the estimate of $I_{5,2}$ in Section \ref{sec:int_holx_I5}. With these estimates, we can further bound $d_y( I_5,x,z)$ 
\[
 d_x(f,x, z) \teq \f{ |f(x) - f(z)|}{ |x_1 - z_1|^{1/2}}, 
 \quad d_y(f,x, z) \teq \f{ |f(x) - f(z)|}{ |x_2 - z_2|^{1/2}}
\]
for sufficiently small $|x-z|$ by optimizing $a$. See Section \ref{sec:hol_comb}.

\begin{remark}

We do not apply the above computation with smaller window $[-ah, ah]^2$ in the $C_x^{1/2}$ estimate, since it leads to a worse estimate. See also the discussions in Section \ref{sec:int_holy_idea}.
\end{remark}

\subsubsection{H\"older estimate of $u, v, u_y, v_x$ }\label{sec:hol_other}
The ideas of the H\"older estimate for other terms are similar. For a kernel $K$ associated with $\uu, \na \uu$, we perform another decomposition similar to \eqref{int:decom_ux}
\beq\label{int:decom_ux2}
\bal
\psi(x) \int &K(x - y)  W(y) dy =  \int  
\B(   \psi(x) \one_{R(k)^c} + \one_{R_s(k)}  \psi(y) 
+\one_{ R(k) \bsh R_s(k)} \psi(y)  \\
&\quad + \one_{ R(k) \bsh R(k_2)} ( \psi(x) - \psi(y))
+ \one_{ R(k_2)} ( \psi(x) - \psi(y)) \B) K(x- y) W(y) dy \\
& \teq  I_1(x, k) + I_2(x,k) + I_3(x, k) + I_4(x,k,k_2) + I_5(x, k_2),
\eal
\eeq
Here, we use $R_s(x, k)$ \eqref{eq:rect_Rsk}, which is symmetric with respect to both $x_1$ and $x_2$, rather than $R_{s, 1}(x, k)$, since the singular region in the sharp H\"older estimate 
of $[u_y]_{C_{x_i}}^{1/2}, [v_x]_{C_{x_i}^{1/2} }, [u_x]_{C_y^{1/2}}$ in Lemma 3.3-3.5 in Part I \cite{ChenHou2023a} needs to be symmetric in both $x_1, x_2$. Denote by $I_{f6}(x, k_2)$ the approximation term for $f = u_x, u_y, v_x, u, v$. It takes the form similar to \eqref{int:decom_ux_appr}. 

We consider two cases of $\hat x \in [0, 2 x_c]^2 \bsh [0, x_c]^2 $ \eqref{int:sing_loc}. In the first case, we consider $\hat x \in [x_c, 2 x_c] \times [0, 2 x_c] \teq D_{X1}$, where we have $\hat x_1 \geq c \hat x_2$ for some constant $c>0$. In the second case, we consider  $\hat x \in [0, x_c] \times [  x_c, 2 x_c] \teq D_{X2}$, where we have $\hat x_1 \leq c \hat x_2$. We distinguish these two cases since in the second case, the singular region does not touch the boundary, we can apply the method in Section \ref{sec:int_holx_I5}.

\vs{0.1in}
\paragraph{\bf{$C_x^{1/2}$ estimate of $u_y, v_x$ }}
In the $C_x^{1/2}$ estimate of $u_y, v_x$, we follow Section \ref{sec:int_hol_reg} to estimate the regular part $I_1 + I_4 - I_6$ and $I_3$. We follow Section \ref{sec:int_hol_I2} and use Lemma 3.4 in Section {\secsharp} of Part I \cite{ChenHou2023a} to estimate $I_2$. For $I_5$, we follow Section \ref{sec:int_holx_I5}.

\vs{0.1in}
\paragraph{\bf{$C_y^{1/2}$ estimate of $u_x$ }}
We perform the decomposition \eqref{int:decom_ux2} rather than \eqref{int:decom_ux}. The estimates of $I_1 + I_4 - I_6$, $I_3$ follow Section \ref{sec:int_hol_reg}. For $I_2$, we use Lemma 3.3 in Section {\secsharp} of Part I \cite{ChenHou2023a}. 
We follow Section \ref{sec:int_holy_K1} to estimate $I_5$ if $\hat x \in D_{X1}$, and 
Section \ref{sec:int_holx_I5} if  $\hat x \in D_{X2}$.

We remark that we use the decomposition \eqref{int:decom_ux2} rather than \eqref{int:decom_ux} since in Lemma 3.3 in Section {\secsharp} of Part I \cite{ChenHou2023a}, we need to assume that the singular region around $x$ is symmetric in both $x_1$ and $x_2$. The same reasoning applies to $
C_{x_i}^{1/2}$ estimate of $u_y, v_x$.

\vs{0.1in}
\paragraph{\bf{$C_x^{1/2}$ and $C_y^{1/2}$ estimate of $u, v$}}
The H\"older estimates of $u, v$ are substantially easier since $u, v$ are more regular. We perform $C_x^{1/2}, C_y^{1/2}$ of $\rho \uu_A$ for another weight $\rho = \psi_u$ \eqref{wg:hol}. 
Below, we only use the weighted $L^{\inf}$ norm $|| \om \vp ||_{\inf}$. We decompose the integral as follows 
\beq\label{int:decom_u}
\bal
\rho(x) \int K(x - y)  W(y) dy &=  
 \int \B(    \one_{R(k)^c} \rho(x)  + \one_{ R(k)} \rho(x) \B)  K( x- y) W(y) dy \\
 &  \teq  I_1(x, k) + I_2(x,k) .
\eal
\eeq
We choose  $k$ smaller than that in \eqref{int:decom_ux} for $\na \uu$ since the kernel for $\uu$ is more regular. We follow Section \ref{sec:int_hol_reg} to estimate $I_1 - I_6$. For $I_2$, we follow the ideas in Sections \ref{sec:int_loglip}, \ref{sec:int_holy_K1}, \ref{sec:int_holy_K2} to estimate the log-Lipschitz function. We choose a list of $k_2$ and associated region $S(k_2)$ and decompose $I_2$ as follows 
\[
I_2(x, k) \teq  \B( \int_{ R(k) \bsh R(k_2) } + \int_{R(k_2) \bsh S(k_2)}
+  \int_{  S(k_2) } \B)  \rho(x) K(x- y) W(y) d y \teq I_{20}(x,k_2)+  I_{21}(x, k_2) + I_{22}(x, k_2).
\]
For large $k_2 = k, k-1/2, .., 2 $, we choose $S(k_2) = R_{s,i}(k_2)$ in the $C_{x_i}^{1/2}$ estimate, $i=1,2$. For $k_2 < 2$, we choose $S(k_2) = R_s(k_2)$.  For $I_{20}(x, k_2), I_{21}(x, k_2)$, we estimate its derivatives following the estimate of $I_{50}, I_{5, 1}$ 
\eqref{eq:int_holx_I5}, \eqref{eq:int_holx_I51}, respectively, or Section \ref{sec:int_nsym_bd0} when $k_2 \geq 2$, and the estimate of $I_{54}$ when $k_2 < 2$ in Section \ref{sec:int_holy_K2}. For $I_{22}(x, k_2)$, we estimate its $L^{\inf}$ norm following the estimate of $I_{53}$ when $k_2 \geq 2$, and the estimate of $I_{56}$ when $k_2 < 2$ in Section \ref{sec:int_holy_K2}. The estimate is simpler since the above kernel is much simpler than  $K(x -y) (\psi(x) - \psi(y) )$ in Section \ref{sec:int_holy_K2}.

\subsubsection{Special case: $C_y^{1/2}$ estimate of $u_y, v_x$ }\label{sec:cy_spec}

In this case, we apply Lemma 3.5 from Section {\secsharp} of Part I \cite{ChenHou2023a} to estimate the most singular part. Since in Lemma 3.5 from Section {\secsharp} of Part I, we do not localize the integral, we perform the following decomposition 
\beq\label{int:decom_uy}
\bal
\psi(x) \int K(x - y)  W(y) dy &=  
\int \B( \psi(y) + \one_{R(k_2)^c}( \psi(x) - \psi(y)) + \one_{ R(k_2)} ( \psi(x) - \psi(y)) \B) 
K( x- y) W(y) dy  \\
& \teq  I_1(x, k) + I_2(x,k) + I_3(x, k) .
\eal
\eeq

For $I_1$, we apply Lemma 3.5 from Part I \cite{ChenHou2023a}. We follow Section \ref{sec:int_holy_K2} to estimate $I_3$ if $\hat x \in D_{X1}$, and Section \ref{sec:int_holx_I5} if  $\hat x \in D_{X2}$. We follow Section \ref{sec:int_hol_reg} to estimate $I_2 - I_6$, where $I_6$ is the approximation terms for $u_y, v_x$ similar to \eqref{int:decom_ux_appr}. 
The symmetrized integrand is discussed in the paragraph ``$C^{1/2}$ estimate of $u_y, v_x$" in Section \ref{sec:int_sym}. There are additional difficulties since the weight  $\psi(y)$ and the symmetrized integrand $I(x, y) = K^C(x, y) (\psi(x) - \psi(y))$ for some kernel $K^C$ (see similar derivations in \eqref{int:sym_integ1},\eqref{int:sym_integ2}) are singular near $0$. 
Note that we do not have the $K^{NC}$ term. See the paragraph \textit{$C^{1/2}$ estimate of $u_y, v_x$} before Section \ref{sec:int_nsym1}.

The integral of $I(x, y)$ near $0$ or in the far-field require some additional estimates, which we discuss below. Since $y$ is away from the singularity $x$ in these cases, the symmetrized integral is given by $I = K^{sym}(x, y) (\psi(x) - \psi(y))$. See \eqref{int:sym_integ2} and 
Section \ref{sec:int_sym} for related discussions.

\vs{0.1in}
\paragraph{\bf{Estimate the integral near $0$}}

To estimate the $D_1 = \pa_{x_2}$ derivative, we use 
\[
|D_1 I| =| D_1 K^{sym} ( \psi(x) - \psi(y)) + K^{sym} \cdot D_1 \psi(x)|
\leq  | D_1 K^{sym}  \cdot \psi(x) + K^{sym} \cdot D_1 \psi(x) | +  |  D_1 K^{sym} \cdot \psi(y) |.
\]
For $y$ close to $0$, since $\psi$ is singular, $\psi(y)$ is much larger than $\psi(x)$, and $K^{sym}(x, y)$ is not singular. The main term in $D_1 I$ is given by $D_1 K^{sym} \psi(y)$. 
It follows 
\[
\int_Q |D_1 I \cdot W(y)|  dy 
\leq || W \vp ||_{\inf} \B( || \vp^{-1} ||_{L^{\inf}(Q)} \int_Q | D_1 K^{sym} \psi(x) + K^{sym} \cdot D_1 \psi(x)| dy 
+  ||  \f{\psi}{\vp} ||_{L^{\inf}(Q)}  \int_Q | D_1 K^{sym} | dy \B),
\]
where $Q$ is some grid near the origin. The integrands in both integrals do not involve the singular weight, and we can estimate them for each grid point $x$ using the previous methods.

To estimate the $X-$ discretization error, we need to estimate the integral of $\pa_{xi}^2 \pa_{x_2} J$. Since $\psi(y)$ is independent of $x$, we get 
\[
I  = K^{sym}(x, y) (\f{\psi(x)}{\psi(y)} - 1)   \psi(y), 
\int_Q |\pa_{x_i}^2 \pa_{x_2} I\cdot  W(y) | dy 
\leq || W \vp||_{\inf}|| \f{\psi}{\vp} ||_{L^{\inf}(Q)} \int_Q \B| \pa_{x_i}^2 \pa_{x_2} K^{sym}(x, y) (\f{\psi(x)}{\psi(y)} - 1)  \B| dy.
\]
The last integrand is not singular in $y$ near $y=0$, and we estimate it using the previous method, e.g. Section \ref{sec:T_rule}.

For $u_y, v_x$, we have a rank-one approximation $K_{app}(x, y)$ from $C_{u_y} \chi_0 K_{00}$ \eqref{eq:u_appr_near0_coe} 
(see Section 4.3.2 from Part I \cite{ChenHou2023a}). The full integrand with approximation term and weight is given by 
\[
\bal
 I_{app} & = K^{sym}(x, y) (\psi(x) - \psi(y)) - K_{app}(x, y) \psi(x) \\
 & = ( K^{sym}(x, y) - K_{app}(x, y) ) \psi(x) - K^{sym}(x, y) \psi(y)
 = I_{app,1} + I_{app, 2}.
 \eal
\]
For $y$ away from the singularity $x$ and $0$, $I_{app, 1}$ has the same form as the previous case, e.g. the $C_x^{1/2}$ estimate. 
We improve the error estimate $\pa_i^2 \pa_{x_2} I_{app}$ using the cancellation between the full symmetrized kernel $K(x, y)$ and $K_{app}$ from Lemma \ref{lem:ker_sym} and the estimate in \eqref{eq:ker_decay3} in Appendix \ref{app:decay}
and the property that $\psi(y)$ is much smaller than $\psi(x)$ for $|y|$ much larger than $|x|$.

\vs{0.1in}
\paragraph{\bf{Estimate in the far-field}}

For the tail part in this case, we have an improvement for small $|x|$ where $\chi_0(x) = 1$ due to the approximation term near $0$ 
\[
\hat f =  C_{f0}(x, y) u_x(0) + C_f(x, y) \cK_{00} 
=  C_f(x, y) \cK_{00} ,
\]
where $f = u_y, v_x$ and $\cK_{00}$ is defined in \eqref{eq:u_appr_near0_coe}, and we have used $C_{f0}(x, y) = 0$. 
Its associated integrand is given by 
\[
K_{app} \teq  \pi^{-1} C_f(x, y) K_{00}(y),
\]
where $K_{00}$ is defined in \eqref{eq:u_appr_near0_coe}. To estimate it, we use the following decomposition 
\[
D_1( I - \psi(x) K_{app})
= D_1( (K^{sym} - K_{app} ) \cdot \psi(x) ) 
- D_1 K^{sym} \cdot \psi(y)  \teq P_1 + P_2.
\]
We estimate $P_1$ using the method in Section \ref{sec:int_beyond}. Due to the approximation, $K^{sym} - K_{app}$ has a much faster decay for large $y$ beyond $[0, D]^2$. See \eqref{eq:ker_decay3} and Appendix \ref{app:decay}.
For $P_2$, we have 
\[
\int_{ \Om^c } |P_2| | W(y)| dy 
\leq || W \vp ||_{\inf}  \int_{\Om^c} | D_1 K^{sym} | \f{\psi}{\vp}(y)  dy,
\]
where $\Om = [0, D]^2$ with large $D$. The last integral is computed using the method in Section \ref{sec:int_beyond}.

\subsection{Estimate the integrals near $0$ and in the far field}\label{sec:int_beyond}

We use a combination of uniform mesh and adaptive mesh to compute the integral in a finite domain $[0, D]^2$, e.g. $D = 1000$. See Section \ref{sec:T_rule}. Since the kernel decays and the singularity is in the near-field, the integral beyond this domain is small, and we estimate it directly. In addition, for $y$ near $0$, we estimate the integrals 
(the last two integrals in \eqref{eq:u_tof}) from the approximations $u_x(0), K_{00}$ \eqref{eq:int_near0}, which is singular of order $|y|^{-2}$ or $|y|^{-4}$. For simplicity, we consider $\lam = 1$. The estimates can be generalized to other scaling parameter $\lam$. 
To estimate $\int_D k(y)\om(y) dy$ for $D$ near $0$ or $D$ in the far-field, following \eqref{int:L1_1}, we only need to estimate $\int_{D} |k(y)| \vp^{-1}(y) dy $. Since $|y|$ is either very small or very large, we can use the asymptotics of $\vp$ in these estimates.

\subsubsection{Near-field estimate}\label{sec:int_near}

Firstly, we estimate $\int_{ [0, R_1]^2} |k(y)| \vp^{-1}(y)  dy $ for  $k(y) = \f{y_1 y_2}{ |y|^4},  \f{y_1 y_2 (y_1^2 - y_2^2)}{|y|^8}$ related to $u_x(0) , K_{00}$ \eqref{eq:int_near0}. 
 We partition $[0, R_1] $ into 
\[
 0 = z_0  < z_1 < ... < z_n = R_1
\]
with $z_1$ much smaller than $R_1$. Denote $Q_{ij} =[ z_{i-1}, z_i] \times [z_{j-1} , z_j]  $. Clearly, we have 
\[
\int_{ [0, R_1]^2} | k(y) | \vp^{-1}(y) dy 
\leq \sum_{ 1 \leq i, j \leq n } I_{ij},
 \quad I_{ij} \teq
\int_{Q_{ij}}  | k(y) | \vp^{-1}(y) dy.
 \]
For $I_{ij}, (i,j) \neq (1, 1)$, we apply a trivial bound 
\beq\label{eq:integ_end_trivial}
I_{ij} \leq || \vp^{-1}||_{L^{\inf}(Q_{ij})} \int_{Q_{ij}} |k(y)| dy
\leq  |Q_{ij}| \cdot  ||k||_{L^{\inf}(Q_{ij})}
|| \vp^{-1}||_{L^{\inf}(Q_{ij})}.
\eeq

For $k(y) = \f{y_1 y_2}{ |y|^4},  \f{y_1 y_2 (y_1^2 - y_2^2)}{|y|^8}$, the estimate of $||k||_{L^{\inf}(Q_{ij})}$ is established in Appendix \ref{app:ker}. It remains to estimate the first term $I_{11}$. Denote $r = y_1$. Suppose that 
\[
\vp( x) \geq q  |x|^{a } ( \cos \b)^{ b }, \quad b \leq 0.
\]
See \eqref{wg:linf}. If $k(y) = \f{y_1 y_2}{ |y|^4}$ and $a<0$, we yield 
\[
\bal
 I_{11} & \leq q^{-1} \int_0^{ \sqrt 2 r  } \int_0^{\pi/2} \f{ \sin \b \cos \b}{r^2} r^{-a} 
 ( \cos \b)^{-b}  r dr d \b
 = q^{-1} \int_0^{\sqrt 2 r} r^{-a- 1} dr \int_0^{\pi/2} \sin \b (\cos \b)^{-b +1} d \b \\
 & = q^{-1}  \f{ (\sqrt 2 r)^{-a}}{-a}  \int_0^1 t^{-b + 1} d t 
 = q^{-1} \f{ (\sqrt 2 r)^{-a}}{-a} \f{1}{2 - b}.
 \eal
\]

If $k(y) = \f{y_1 y_2 (y_1^2 - y_2^2)}{|y|^8}$, we yield $|k(y)| \leq \f{1}{4} \f{\sin 4 \b}{ r^4}$. Since $b \leq 0$, if $a<-2$, we get $ \vp \geq q r^{a}$ and 
\[
\bal
I_{11} & \leq q^{-1} \int_0^{\sqrt 2 r } \int_0^{\pi/2} \f{1}{4} \f{ |\sin 4 \b|}{ s^4} 
s^{-a} s ds d \b
= \f{1}{4q } \int_0^{\sqrt 2 r} s^{-a-3 } d s
\f{1}{4}\int_0^{2\pi} |\sin \b| d \b \\
&= \f{1}{4q} \f{ (\sqrt 2 r)^{-a-2}}{-2-a } \int_0^{\pi/2} \sin \b d \b
= \f{1}{4q} \f{ (\sqrt 2 r)^{-a-2}}{-2-a }.
\eal
\]

\subsubsection{Far-field estimate}\label{sec:int_far}
Denote $a \vee b = \max(a, b)$.
To estimate  the far field integral $I \teq \int_{ y_1 \vee y_2 \geq R_0 } | k(y) | \vp^{-1}(y) dy $, we first pick sufficient large $R$, and then partition the domain 
\[
0 = z_0 < z_1 < .. < z_m = R_0 < z_{m+1} < ... < z_n = R_1 < + \inf.
\]

Denote $Q_{ij} = [z_{i-1}, z_i] \times [ z_{j-1}, z_j]$. Clearly, we have 
\[
I = \sum_{ m+1 \leq \max(i, j) \leq n } I_{ ij } 
+ J , \quad I_{ij} \teq \int_{Q_{ij}} |k(y)| \vp^{-1}(y) dy, \quad J = \int_{ y_1 \vee y_2 \geq R_1 } | k(y) | \vp^{-1}(y) dy.
\]

For $I_{ij}$, we apply the trivial estimate \eqref{eq:integ_end_trivial}. Suppose that 
\[
\vp \geq q r^a (\cos \b)^b , \quad |k(y)| \leq |y|^{-p} , \quad b \in [-1, 0], \quad p+ a > 2.
\]
We get 
\[
J \leq \f{1}{q} \int_{R_1}^{\inf} \int_0^{\pi/2} r^{-p- a} (\cos \b)^{-b}   r dr d \b
= \f{1}{q} \f{ R_1^{-p-a+2}}{ |p+a - 2|} \int_0^{\pi/2} (\cos \b)^{-b} d \b.
\]

Using H\"older's inequality and $b \in [-1, 0]$, we get 
\[
\int_0^{\pi/2} (\cos \b)^{-b} d \b  
\leq ( \int_0^{\pi/2} \cos \b d \b)^{-b} (\int_0^{\pi/2} 1 )^{1 + b}
= (\pi/2 )^{1 + b}.
\]

It follows 
\[
J \leq \f{1}{q} \f{ R_1^{-p-a+2}}{ |p+a - 2|}(\pi/2 )^{1 + b}.
\]

\vs{0.1in}
\paragraph{\bf{Application }}

We apply the above calculations to estimate the integral and its derivatives beyond the mesh $[0,D]^2$ \eqref{int:parti1}. Since the domain is far away from the singularity, the integrand is the symmetrized kernel, e.g., \eqref{int:sym_integ2}. 
From Appendix \ref{app:decay} and Lemma \ref{lem:ker_sym} in Appendix \ref{app:ker}, for $\uu_A, \na \uu_A, \pa_i ( \rho \uu_A), \pa_i (  \psi \na \uu_A)$, the integrand in the far-field ($y$ is large) satisfies 
\[
|K(x, y) | \leq C(x) \Den^{-k},
\] 
with some $k \geq 2$ and coefficients $C(x)$, where $\Den$ is defined in \eqref{int:Den}. 

In our computation, we rescale $x$ to $\hat x$ and restrict it to the near-field $[0, b]^2$ with $b < 2$. 
Note that $y \notin [0, D]^2$ and $|y|\geq D \gg b$. From \eqref{int:Den}, we get 
\[
\Den \geq \min_{ |z_1| \leq x_1, |z_2| \leq x_2} |y-z|^2 
\geq \min_{ |z_1| \leq x_1, |z_2| \leq x_2} (|y| - |z|)^2
\geq (|y| - |x|)^2 = |y|^2 (1 - \f{|x|}{|y|} )^2 .
\]
Since $\f{|x|}{|y|} \leq \sqrt{2 } b / D$, we yield 
\[
\Den \geq (1- C_s)^2 |y|^2, \quad C_s = \sqrt{2} b / D.
\]

It follows 
\[
\int_{ y \notin [0,D]^2} |K(x, y)| \vp^{-1}(y)dy 
\leq (1-C_s)^{-2k} C(x) \int_{ y \notin [0,D]^2} |y|^{-2k}\vp^{-1}(y)dy .
\]
Using the method in Section \ref{sec:int_far}, we can estimate the above integral.

\subsection{Estimate for very small or large $x$}\label{sec:hol_1D}

The rescaling argument and the methods in the previous subsections apply to the estimate of $\uu_A(x), \na \uu_A(x)$ for $x \in [0, x_M]^2 \bsh [0, x_m]^2, 0 < x_m < x_M$. For very small or large $x$, we cannot use a finite number of dyadic scales $\lam = 2^i$ to rescale $x$ such that $x / \lam \in [0, 2x_c]^2 \bsh [0, x_x]^2$.  Instead,  we choose $\lam = \f{ \max(x_1, x_2) }{x_c}$. We want to estimate the rescaled integral  with a $-d$-homogeneous kernel $K$
\[
 p(x) \int K(x - y) W(y) dy 
 = p_{\lam}(x) \int K(\hat x - \hat y) \lam^{2- d} W_{\lam}( \hat y) dy,
\]
uniformly for all small $\lam \ll 1$ or large $\lam \gg 1$, where $p$ is some weight and $p_{\lam}$ is defined in \eqref{eq:flam}. The rescaled singularity $\hat x = x / \lam$ satisfies $\max_i \hat x_i = x_c$.
We simplify $\hat x, \hat y$ as $x, y$.

We can use the asymptotic of the weights to estimate the integral, see e.g. \eqref{int:scal_asymp}. The new difficulty is that the estimate involves the rescaled weight $p_{\lam}(y)$. Since $\lam$ is not fixed and depends on $x$ that tends to $0$ or $\infty$, we cannot evaluate $p_{\lam}(y)$ and the integrand directly. In the following derivation, $\lam$ is comparable to $|x|$, which is either very small or very large. 

For $y$ away from the singular region, the integrand of the regular part is given by $J = K(x, y) \cdot p_{\lam}(x) $ \eqref{int:sym_integ2}. We choose a radial weight $p$ defined in Appendix \ref{app:wg} $p(x) = \sum_{1\leq i \leq n} q_i |x|^{a_i}$. See $\psi_1, \psi_u, \psi_{du}$ \eqref{wg:hol}. 
We introduce the asymptotics of these weights 
\[
R_{\lim} \teq \lim_{x \to A} \f{ D_1 p_{\lam}(x)}{ p_{\lam}(x)},
\quad p_{lim} = q_i |x|^{a_i}, 
\]
with $(A, i) = (0, 1)$ or $(A, i) = (\infty, n)$, where $(q_n, a_n)$ denotes the last power in the weight. We use the following decomposition to compute $D_1 J$ with $D_1 = \pa_{x_i}$
\[
\bal
|D_1 J| &= | D_1(  K(x, y) \cdot p_{\lam}(x)  )|
= |D_1 K(x, y) \cdot p_{\lam}(x)
 + K(x, y) \cdot D_1 p_{\lam}(x) | \\
 &= \B| p_{\lam}(x) \B\{  D_1 K(x, y) + R_{lim} K(x, y) +( \f{ D_1 p_{\lam}(x) }{ p_{\lam}(x)} - R_{lim}  ) K(x, y)   \B\} \B| . \\
 \eal
\]
Since we consider very small $\lam$ or very large $\lam$, the error term 
$ \f{ D_1 p_{\lam}(x) }{ p_{\lam}(x)} - R_{lim}  $ is small. Hence, we use a triangle inequality to bound $D_1 J$
\[
| D_1 J| 
\leq 
 p_{\lam}(x) \B|  D_1 K(x, y) + R_{lim} K(x, y)  \B|
 + p_{\lam}(x) \B| ( \f{ D_1 p_{\lam}(x) }{ p_{\lam}(x)} - R_{lim}  ) K(x, y)    \B|.
\]
The advantage of the above decomposition is that the main term $D_1 K(x, y) + R_{lim} K(x, y)$ does not depend on $\lam$ so that we can estimate it using previous methods.

Since the estimate of derivative of $u, v$ does not involve the commutator, see, e.g. \eqref{int:decom_u}, we can apply the above method to compute the integral of $D_1 u$ for small $x$ or large $x$.

For $y$ near the singular region, from \eqref{int:sym_integ1}, the symmetrized integrand is given by 
\[
J = K^{C} ( p_{\lam}(x) - p_{\lam}(y) ) + K^{NC} p_{\lam}(x),
\]
where we use $p$ for the weight. Firstly, we have
\[
\bal
|D_1 J | & = | D_1 K^{C} ( p_{\lam}(x) - p_{\lam}(y)  ) + D_1 K^{NC} p_{\lam}(x)
+ (K^C + K^{NC})  D_1 p_{\lam}(x)  | \\
\eal
\]
Denote  $K = K^C + K^{NC}$. We use the following method to bound $D_1 J$
\[
\bal
|D_1 J |
&\leq p_{\lam}(x) \B| D_1 K^C  \cdot ( 1 - \f{ p_{\lam}(y)}{ p_{\lam}(x)} )  +
D_1 K^{NC} +K \cdot \f{ D_1 p_{\lam}}{ p_{\lam}} \B| \\
& \leq p_{\lam}(x) 
\B\{ 
\B| D_1 K^C  \cdot ( 1 - \f{ p_{lim}(y)}{ p_{lim}(x)} )  +
D_1 K^{NC} +K \cdot \f{ D_1 p_{lim}}{ p_{lim}} \B|  \\
& \qquad + \B| D_1 K^C (\f{ p_{\lam}(y)}{ p_{\lam}(x)} - \f{ p_{lim}(y)}{ p_{lim}(x)}  )  \B|
+ K \B| \f{ D_1 p_{lim}}{ p_{lim}} -  \f{ D_1 p_{\lam}}{ p_{\lam}}  \B|
\B\} .
\eal
\]
The second and the third term on the right hand side can be seen as an error term. The main term $\B| D_1 K^C  \cdot ( 1 - \f{ p_{lim}(y)}{ p_{lim}(x)} )  +
D_1 K^{NC} +K \cdot \f{ D_1 p_{lim}}{ p_{lim}} \B|$ does not depend on $\lam$, and the singularity $x$ is in the near-field and away from $0$. We can apply all the delicate decompositions developed in previous sections to estimate $D_1 J$. 

In the  H\"older estimates, we need various bounds for the weights $p_{\lam}$. 
Using the asymptotics of $p(x)$, we can estimate the derivatives of $p_{\lam}$ for very small $\lam$ or very large $\lam$ uniformly. See Appendix \ref{app:wg}, \ref{app:wg_radial}. Once we obtain the estimates of $\psi_{\lam}$, and the weight $\vp_{\lam}$ in the $L^{\inf}$ norm $|| \om_{\lam} \vp_{\lam}||_{\inf}$, we can use the methods in the previous subsections and the scaling relations in Section \ref{sec:scal_prop} to perform the H\"older estimates. 

The $L^{\inf}$ estimate follows similar ideas and is much easier. We refer more details to Section {\secuoneD} in the supplementary material II \cite{ChenHou2023bSupp}.

We remark that since we have much larger damping coefficients in the energy estimates (see Section {\secEE} in Part I  \cite{ChenHou2023a}) near $x=0$ and in the far-field, the estimates of the nonlocal terms in these regions, though technical, only have minor effects on the nonlinear stability estimates. 

\subsection{Assemble the H\"older estimates}\label{sec:hol_comb}

In Section \ref{sec:vel_hol_comp}, we decompose the velocity in several parts and estimate them separately using the norms $|| \om \vp||_{\inf}, [\om \psi]_{C_{x_i}^{1/2}}$. In this section, we assemble these estimates and estimate 
\[
\d(f, x, z) \teq \f{ |f(x) - f(z)| }{ |x-z|^{1/2}},
\]
for $f = \psi_u \uu_A , \psi \na \uu_A$ with weights in \eqref{wg:hol}. To obtain better estimates, we combine some of the estimates.

In the proof of the first inequality in Lemma \ref{lem:main_vel}, we combine and bound different norms using $\max(  || \om \vp||_{\inf} ,  \max_{j=1,2} \g_j  [ \om \psi_1]_{C_{x_j}^{1/2}(\R_2^{+}) } )$. We apply the second inequality to the error $\e = \om - (-\D) \phi^N $ \eqref{eq:elli_err0} and can evaluate the localized norm using piecewise bounds of the error. See Section \ref{sec:vel_loc_est}.

To illustrate the ideas, we focus on the $C_x^{1/2}$ estimate, $x \in [x_c, 2 x_c] \times [0, 2 x_c]$, i.e. $x_1$ is large relative to $x_2$, $z_1 \geq x_1$, and $x_2 = z_2$. For general pairs $(x, z)$, we can rescale $(x, z)$ to $(\lam x, \lam z)$ such that $ \lam x \in [0, 2x_c]^2 \bsh [0, x_c]^2$. Using the scaling relations in \eqref{sec:scal_prop}, we can estimate the rescaled version of $\d(f, x, z)$. See also the discussion at the beginning of Section \ref{sec:vel_hol_comp}. 

We assume that $z_1 \in [x_c, 2(1 + \nu) x_c]$ with $\nu < 1$. For $z_1 \geq 2 (1 + \nu) x_c$, we have $z_1 > (1 + \nu) x_1$.  
Since $z_1, x_1$ are large relative to $z_2, x_2$, respectively, we have 
\[  |x-z|= |z_1 - x_1| \asymp |z_1| \gtr |x|, |z| .\]
Then, we can use the $L^{\inf}$ estimate and triangle inequality to estimate $\d(f, x, z)$. Note that we can estimate the piecewise $L^{\inf}$ norm of $ |x|^{-1/2} \rho(x) \uu_A(x)$ and $|x|^{-1/2} \psi \na \uu_A$ following Section \ref{sec:vel_linf}, where $\rho, \psi$ are the weights in the H\"older estimate of $\rho \uu_A, \psi \na \uu_A$. See Section {\secholRtwo} in the supplementary material II \cite{ChenHou2023bSupp} for more details.

We focus on $f = \psi u_{x, A}$. We partition the domain $D_{\nu} = [x_c , 2(1+ \nu) x_c] \times [0, 2 x_c]$ into $h_x \times h_x$ grids $D_{ij} , 1 \leq i \leq 2 (1 + \nu ) x_c / h_x, 1 \leq j \leq 2 x_c / h_x$. We apply the decomposition \eqref{int:decom_u} with the same parameters $k, k_2$ to $x$ in different grids $D_{ij}$. For $x \in D_{ij}$, using the method in Section \ref{sec:vel_hol_comp}, we obtain the estimate 
\beq\label{eq:hol_comb1}
\bal
 &f(x) = I_1(x) + I_2(x) + I_3(x) + I_4(x) + I_5(x) - I_6(x), \quad 
 I_{5} = I_{5, 0} + I_{5, 1} + I_{5, 2} , \\
 &|\pa_x (I_1 + I_4 + I_{5,0}- I_6)| \leq a_{ij, 1}  || \om \vp||_{\inf} ,  \ 
 |\pa_x I_3| \leq a_{ij, 2} || \om \vp||_{\inf}  , \ 
 |\pa_x I_{5, 1} | \leq a_{ij, 3} || \om \vp||_{\inf}  , 
 \eal
\eeq
for some constants $a_{ij, l}, b_{ij} \geq 0$, where $I_{5, 1}, I_{5,2}$ are defined and estimated in Section \ref{sec:int_holx_I5}. 

For $x, z \in D_{\nu}$ with $x_2 = z_2, z_1 \leq z_1$, we have $x \in D_{i_1, j}, z \in D_{i_2, j}$ for some $i_1 \leq i_2$. We apply the method in Section \ref{sec:int_hol_I2} to estimate $ \d( I_2, x, z) $ and the method in Section \ref{sec:int_holx_I5} to estimate $J_1$ related to $\d(I_{52}, x, z)$ \eqref{int:holx_ux_J}. These estimates contribute to the bound  $ C_{hol} [ \om \psi]_{C_x^{1/2}} $ for some $C_{hol}>0$, which can be computed.

\vspace{0.1in}
\paragraph{\bf{Regularity of the combination}}

While $I_1 + I_4 + I_{5,0} - I_6, I_3, I_{5,1}$ are only piecewise smooth and can be discontinuous when $x$ crosses the grids $D_{ij}$, the sum $I_{lip} =I_1 + I_4 + I_{5,0} - I_6 + I_3 +I_{5,1}$ is continuous and Lipschitz in $x_1$ for fixed $x_2$. In fact, by definition 
\eqref{int:decom_ux}, \eqref{eq:int_holx_I5}, we get
\[
\bal
I_{lip} & = \int_{R(k) \bsh R_{s, 1}(k_3)}  
K_1(x - y) (\psi(x) - \psi(y)) W(y) d y \\
& \quad + \psi(x) \int_{ R(k)^c} K_1( x-y) W(y) d y
+ \int_{R(k)\bsh R_{s, 1}(k)} K_1(x - y) \psi(y) W(y) d y - I_6  \\
& = \psi(x) \int_{ R^c_{s, 1}(k_3) } K_1(x - y)  W(y) d y
 - \int_{R_{s,1}(k) \bsh R_{s, 1}(k_3)} K_1(x - y)  W(y) \psi(y) d y - I_6.
\eal
\]
For fixed $x_2$, since $I_6$ \eqref{int:decom_ux_appr} for the approximation term is smooth in $x$ and the domain $R_{s, 1}(l)$ \eqref{eq:rect_Rsk} depends on $x_1$ continuously, we obtain that $I_{lip}$ is continuous in $x_1$ when $x$ crosses the grids $D_{ij}$. Since $I_{lip}(x)$ is smooth for $x \in D_{ij}$, we get that $I_{lip}$ is continuous and Lipschitz in $x_1$ with piecewise Lipschitz norm bounded by $a_{ij,1}+ a_{ij, 2} + a_{ij, 3}$.

Similarly, for the case in Section \ref{sec:int_holx_I5}, we have $I_{lip} = I_1 + I_3 + I_4 + I_{5,0} + I_{5, 1} - I_6$ \eqref{int:decom_ux}, \eqref{eq:int_holx_I5} is Lipschitz in $x_i$ in the $C_{i}^{1/2}$ estimate for fixed $x_{3-i},i=1,2$. 

For the case in Section \ref{sec:int_holy_K1},  
$I_{lip} = I_1 + I_3 + I_4 + I_{5,0}^- + I_{5, 1}^- +  I_{5, 1}^+ - I_6$  \eqref{int:decom_ux}, 
\eqref{int:holy_sing0}, \eqref{int:holy_sing1}
and $I_{lip} = I_1 + I_3 + I_4 +   I_{5, 3}^-  + I_{5, 1}^+   - I_6$ 
 \eqref{int:decom_ux}, \eqref{int:holy_sing0}, \eqref{int:holy_sing2} are Lipschitz, where $ I_{5, 1}^+$ associated with $I_5^+$  \eqref{int:holy_sing0} is defined similar to 
$ I_{5, 1}$ in \eqref{eq:int_holx_I5}.

For the case in Section \ref{sec:int_holy_K2}, 
$I_{lip} = I_1 + I_3 + I_4 + I_{5,0} + I_{5, 1} 
+  I_{5,2, 1} - I_6$ 
\eqref{int:decom_ux},
\eqref{int:holy_K2_I5}, 
and $I_{lip} =  I_1 + I_3 + I_4 + I_{5,4}- I_6$ are Lipschitz,  where $ I_{5, 2, 1}$ associated with $I_{5, 2}$  \eqref{int:holy_K2_I5} is defined similar to $ I_{5, 1}$ in \eqref{eq:int_holx_I5}.

In summary, the sum of the terms in $f(x)$ \eqref{eq:hol_comb1} 
with piecewise derivative estimates is Lipschitz. Using the triangle inequality, we obtain the piecewise Lipschitz bound for $I_{lip}$. 
The remaining parts in $f(x)$ \eqref{eq:hol_comb1} are continuous and are estimated by the piecewise $L^{\infty}$ bounds, e.g. $I_{5, 2}^-$ \eqref{int:holy_sing1}, $I_{5, 4}^-(a)$ \eqref{int:holy_sing2}, $I_{5, 3}$ \eqref{int:holy_K2_I5}, and the improved H\"older estimates, e.g. $I_{5, 2}$ \eqref{eq:int_holx_I5}

By averaging the piecewise derivative bounds and using the estimates in Appendix \ref{app:piece_deri}, for $x \in D_{i_1, j}, z \in D_{i_2, j}$, we can obtain 
\[
| I_{lip}(x) - I_{lip}(z)| \leq C_{lip} |x_1 - z_1| \cdot || \om \vp||_{\inf}
\]
for constant $C_{lip}$  depending only on $\{ a_{kl, j} \}_{k, l \geq 1, j\leq 3 } $ and the mesh $h_x$ explicitly. Hence, for the remaining terms in $f$ not estimated using the seminorm $[ \om \psi]_{C_x^{1/2}} $, e.g. $I_1 + I_4  - I_6 +  I_3 + I_{5,0} + I_{5, 1}$ and $J_2$ related to $I_{5,2}$ \eqref{int:holx_ux_J}, each term is continuous and they satisfy \footnote{In the previous version of this paper \cite{ChenHou2023b}, some term $f_l(x)$ is not continuous when $x$ crosses the grids. We have corrected this minor issue by reorganizing different terms so that each $f_l(x)$ is continuous. See the above paragraph \textit{Regularity of the combination}. Related computer-assisted estimates have been updated and the full nonlinear stability estimates remain valid.
}
\[
  f_R(x) = \sum_{ 1 \leq l \leq N} f_l(x), \quad |f_l(x) - f_l(z)| \leq \min( p_l |x_1 - z_1|, q_l)  \cdot  || \om \vp||_{\inf}
\]
for some $N$, where we can choose $q_l = \infty$ if we do not have $L^{\inf}$ estimate for $f_l(x)$. Similar consideration applies to $p_l$. In our problem, there are only a few terms and $N  < 10$. In the $C_{x_i}^{1/2}$ H\"older estimate of $P_{ij}Q_{ij}(x)$ (continuous in $x_i$) in $I_{52}$ \eqref{int:holx_ux_J}, we optimize two estimates (see the estimates between \eqref{int:holx_ux_J} and \eqref{eq:PQ_linf}), which is a nontrivial example of the above summand.

Now, for $x \in D_{i_1,j} , z \in D_{i_2, j}$, we have 
\beq\label{eq:hol_comb2}
\bal
 \f{ |f_R(x) - f_R(z)}{ |z_1 - x_1|^{1/2}}
&\leq \sum_{1 \leq l \leq N} \min( p_l \d^{1/2}, q_l \d^{-1/2} ) || \om \vp||_{\inf}, \\
 \d & = z_1 - x_1 \in [ \max( i_2 - i_1 -1, 0) h_x,  (i_2 - i_1 + 1) h_x ], 
 \eal
\eeq

The upper bound can be obtained explicitly by partitioning the range of $z_1 - x_1$ into finite many subintervals $M_l$ according to the threshold $\d_l = q_l / p_l$. In each $M_l$, the bound reduces to
\[
 P \d^{1/2} + Q \d^{-1/2}
\]
for some constants $P, Q$. It is convex in $\d^{1/2}$ and can be optimized easily and explicitly in any interval $[\d_l, \d_u], \d_l > 0$ .

\begin{remark}
We combine the estimates of different parts in \eqref{eq:hol_comb1} using \eqref{eq:hol_comb2} to obtain a sharp estimate. If one estimate different parts separately, the distance $\d = z_1 -x_1$ for the optimizer may not be achieved for the same value, which leads to an overestimate. We remark that for small distance $|z_1-x_1|$, such an overestimate can be significant since the 
ratio between the endpoints $|i_2 - i_1 + 1| / \max( i_2 - i_1 -1, 0) $ varies a lot. 
\end{remark}

In some estimates, e.g. the $C_y^{1/2}$ estimate of $u_x$ in Section \ref{sec:int_holy_K1}, we need to  decompose $I_5$ using different size of small singular region $k_3$. In such a case, we have a list of estimates associated to different $k_3$ for the part $f_R$ not estimated by $[\om \psi]_{C_x^{1/2}}$ or $[ \om \psi]_{C_y^{1/2}}$:
\[
 \f{ |f_R(x) - f_R(z) | }{ |z_1 - x_1|^{1/2}}
\leq \sum_{1 \leq l \leq N} \min( p_{l, k_3} \d^{1/2}, q_{l, k_3} \d^{-1/2} ) || \om \vp ||_{\inf}.
\]
For $|x_1-z_1|$ bounded away from $0$, e.g. $|x_1 - z_1| \geq \f{1}{10} h_x$, we can still partition the range of $|x_1 - z_1|$ and optimizing the above estimates first over $\d$ and then $k_3$.

\subsubsection{H\"older estimate for small distance}\label{sec:hol_small}

In some H\"older estimates, e.g. the $C_y^{1/2}$ estimate in Sections \ref{sec:int_holy_K1}, \ref{sec:int_holy_K2}, when $|x-z|$ is very small, e.g. $|x - z| \leq  c h_x $ with $c < 1$, we need to choose a singular region with size $a$ to be arbitrary small. See also Section \ref{sec:int_loglip} for the estimates of a log-Lipschitz function. In these estimates, we can decompose 
$f_R(x)$ that is not estimated using the H\"older norm of $\om \psi$ as follows 
\[
f_R(x) = f_1(x, a, b) + f_2(x, a),
\]
for  $a < b$ and $b$ is fixed.
We can estimate the derivative of $f_1$, and the $L^{\inf}$ norm for $f_2$ \[
| \pa_x f_1(x, a, b) | \leq ( A_i + B_i \log \f{ b}{a} ) || \om \vp||_{\inf}, \quad |f_2| \leq \f{C_i a}{2} || \om \vp||_{\inf}
\]
in each grid $D_{ij}$ for any $a \leq b$, see e.g., \eqref{int:log_lip_ux} and \eqref{eq:hol_comb1}. We drop $j$ since we consider $x, z$ with $x_2 = z_2$. For $t= |x-z| \leq h_x$, we get 
\beq\label{eq:hol_comb3}
\f{ | f(x) - f(z)|}{|x-z|^{1/2}}
\leq  (A + B \log \f{b}{a}) \sqrt{t} + \f{ Ca}{ \sqrt{t}} \teq F(a, t)
\eeq
where $A =\max(A_i, A_{i+1}), B = \max( B_i, B_{i+1}), C = \max( C_i, C_{i+1})$. For each $t \leq c h_x $,  we can optimize the above estimate  over $a \leq b$ explicitly. Then we maximize the estimate over $t \leq c h_x$ to obtain uniform estimate for small $|x-z| \leq c h_x$. We refer the derivations to Appendix \ref{app:hol_opt}.

\subsection{Improved estimate for the nonlocal error}\label{sec:vel_loc_est}

In Section \ref{sec:vel_err}, we discuss the estimates of the nonlocal error $\uu(\bar \e)$ based on the functional inequalities established in this section. 
Since the weight is singular $\vp \sim |x|^{-2} |x_1|^{-1/2}, \vp = \vp_{elli}$ \eqref{wg:linf}  near the origin, $ \bar \e_1 \vp $ is much larger near $x=0$. Due to the anisotropic mesh for large $x$ and small $y$, or small $x$ and large $y$, and the round off error, $\bar \e_1$ is not very small in these far-field regions. 
On the other hand, these regions are small since either $|(x, y)|$ is very small or the ratio $x/y, y/x$ is very small, and the error is very small in the bulk, e.g. $x = O(1)$. See Figure \ref{fig:elli_err} for the rigorous weighted bound of the error in the adaptive mesh. The weighted error of $\bar \e_1$ is larger near $0$, while the error for $\hat \e_1$ is larger in the far-field. If we simply use the global norm $||\om \vp||_{\inf}, \om = \bar \e, \hat \e$, and then apply the previous estimates to bound $\uu(\bar \e)$, we overestimate the nonlocal error significantly. For $x = O(1)$, where we have the smallest damping for the energy estimate, due to the decay of kernel and the smallness of these regions, the integral $\int K(x, y) \bar \e(y) dy$ near $y=0$ or in the far-field is very small.

\begin{figure}[h]
   \centering
      \includegraphics[width =0.9 \textwidth  ]{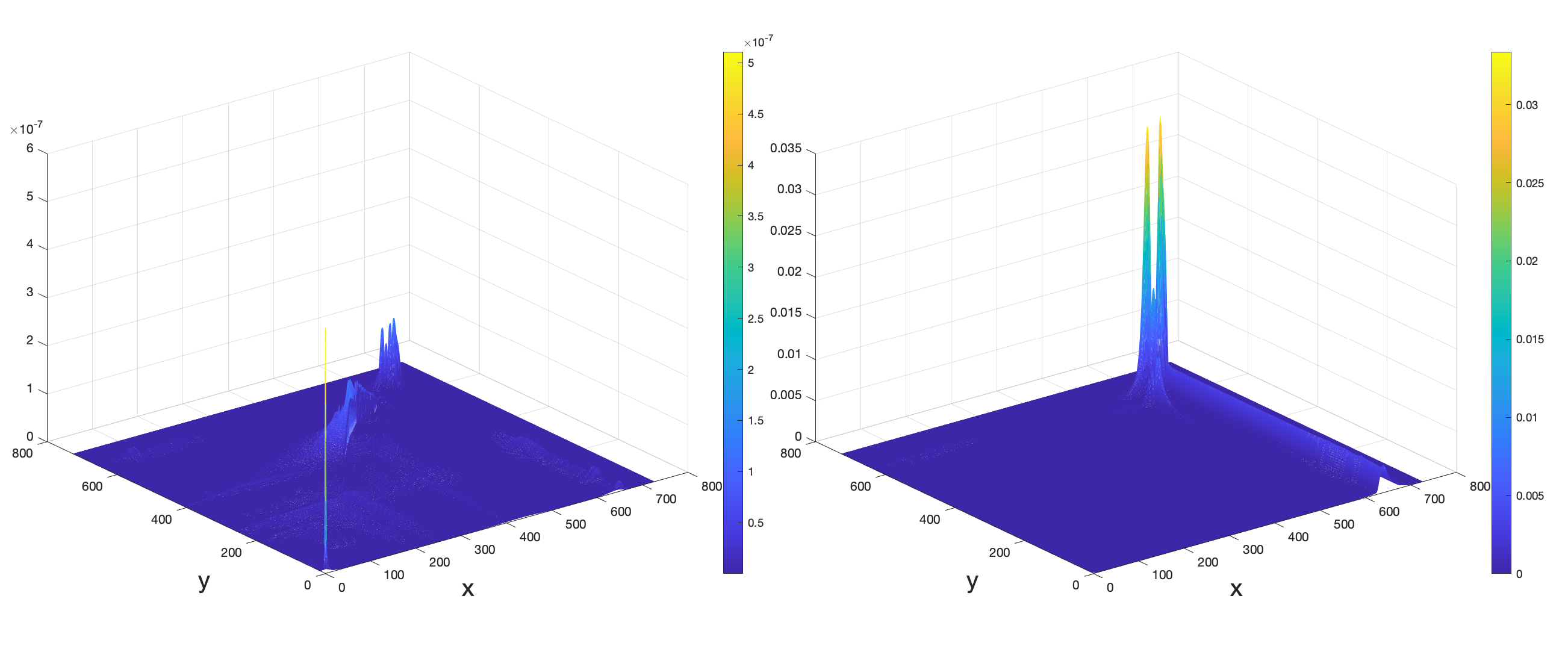}
      \caption{Piecewise $L^{\inf}(\vp_{elli})$ bound of the error $\bar \e_1, \hat \e_1$ in solving the Poisson equations. Left: error for the approximate steady steate. Right: error for the approximate space-time solution $\hat W_2$}
            \label{fig:elli_err}
 \end{figure}

Note that we can obtain the piecewise derivative bounds for the error $\bar \e_1, \hat \e_1$ and we partition the domain of the integral into different regions \eqref{int:decom_linf_ux}. Instead of using the global norm to bound the integral, we use the localized 
norms $|| W \vp_{elli}||_{l^{\inf}(D)}, [ W \psi_1]_{C_{x_i}^{1/2}(D)}$ \eqref{wg:linf}, \eqref{wg:hol} to exploit the smallness of the error in most part of the domains and 
improve the error estimate.

Recall the regions of rescaled $\hat x$ \eqref{int:sing_loc} and the mesh $y_i$ partitioning the domain \eqref{eq:int_mesh_y}. We fix a scale $\lam$ and assume $\hat \in [x_c, 2x_c] \times [0, 2 x_c]$. By definition, the singular region $R(\hat x, k)$ \eqref{eq:rect_Rk} satisfies 
\[
- R(\hat x, k) \cap \R_2^+,  \ R(\hat x, k) \cap \R_2^+ \subset [x_c - k h , 2 x_c + kh] \times [0, 2 x_c +  kh ] \teq S_{kh}.
\]
Thus, in the estimates of $I_2, I_3, I_4$ in \eqref{int:decom_linf_ux}, instead of using the global norm $|| W \vp ||_{L^{\inf}}$, we use $|| \om_{\lam} \vp_{\lam} ||_{L^{\inf}(S_{kh})}
= || \om \vp||_{L^{\inf}( \lam S_{kh})}$. For the error $\om = \bar \e, \hat \e$, 
we can bound $|| \om  \vp||_{L^{\inf}( \lam S_{kh})}$ by using the piecewise estimates of $\bar \e, \hat \e$ and covering the region $ \lam S_{kh}$. Similarly, we use the localized bound $[ \om_{\lam} \psi_{\lam}]_{C_{x_i}^{1/2}( S_{kh}) } = \lam^{1/2} 
[ \om  \psi ]_{C_{x_i}^{1/2}( \lam S_{kh}) } $ for the H\"older seminorm in the estimate of $I_2, I_3, I_4$, and similar localized norms for $I_5$.

For the regular part $I_1$, we partition $[0, D]^2, \R_2^{++}$ into disjoint domains: near-field $D_{n,i}$ the bulk $D_B$ and the far-field $D_{f,i}$, e.g. 
\[
D_{n, 1} =[8h, 16 h],  \  D_B =[0, 2]^2 \bsh D_{n, 1}, \  D_{f, 1} = [0, D]^2 \bsh [0, 2]^2,  \  D_{f, 2} = \R_2^{++} \bsh [0, D]^2,
\]
where $h$ is the mesh size in \eqref{eq:int_mesh_y}. Then we use the norm $|| \om_{\lam} \vp_{\lam} ||_{L^{\inf}(D) } = || \om \vp||_{L^{\inf}(\lam D)} $ for the estimate of the integral in region $D$.

In  \eqref{eq:u_tof}, we estimate the integral of $K_{00}(y)$ \eqref{eq:u_appr_near0_coe} for
$|\hat y|_{\inf} \leq k_{02}h$ and $|\hat y|_{\inf} \geq k_{02}h$ separately. Since the kernel is very singular near $0$, the $L^1$ estimate of the integral in $|\hat y|_{\inf} \leq k_{02}h$ in Section \ref{sec:int_near} is not very small. Since we can evaluate $ \om = \bar \e, \hat \e$, we change the rescaling from $\hat{y}$ back to $y$ by using $y = \lam \hat y$ in \eqref{eq:u_tof} 
\[
J = 
 \int_{|\hat y|_{\inf} \leq k_{02} h} K_{00}( \hat y) \om(\lam \hat y) d \hat y 
= 
\lam^2 \int_{ |y|_{\inf} \leq \lam k_{02} h} K_{00}( y) \om( y) dy,
\]
where we get $\lam^2$ since $K_{00}$ is $-4$ homogeneous. For a list of dyadic scales $\lam = 2^k$, we estimate the integral using Simpson's rule with very small mesh. 
This allows us to exploit the cancellation in the integral. For $|y|$ very close to $0$, we use Taylor expansion. See Section 
{\secKNNerr} in supplementary material II \cite{ChenHou2023bSupp} (attached to this paper) for more details.

In the estimate of the integral for very small $x$ or large $x$ in Section \ref{sec:hol_1D} (see more details in Section {\secuoneD} in the Supplementary Material II \cite{ChenHou2023bSupp}), we estimate the rescaled integral for $\lam \leq \lam_1$ and $\lam \geq \lam_n$ with small $\lam_1$ and large $\lam_n$ uniformly. In the case of $\lam \leq \lam_1$, we bound $|| \om_{\lam} \vp_{\lam} ||_{L^{\inf}([a, b]\times [c, d])} 
\leq || \om \vp ||_{L^{\inf}( \lam_1 [0, b]\times [0, d])} $. Other norms in different cases are estimated similarly.

We do not track the bound $|| \om_{\lam} \vp_{\lam}||_{L^{\inf}(Q_{ij})}$ in each small grid $Q_{ij}$ for computational efficiency. 

\appendix

\section{Weights and parameters}\label{app:wg_tot}

\subsection{Estimate of the weights}\label{app:wg}

Recall the 
weights for the H\"older estimate of $\om, \eta, \xi$ and $\uu$
\beq\label{wg:hol}
\bal
& \psi_1   = |x|^{-2} + 0.5 |x|^{-1} + 0.2 |x|^{-1/6}, \quad  \psi_{du} = \psi_1, \quad \psi_{u} = |x|^{5/2} + 0.2 |x|^{-7/6},  \\
&\psi_2  = p_{2, 1} |x|^{-5/2} + p_{2,2} |x|^{-1} + p_{2,3} |x|^{-1/2} + p_{2,4} |x|^{1/6}, \\
& \psi_3 = \psi_2,  \quad \vec{p}_{2, \cdot} = (0.46, 0.245, 0.3, 0.112), \\
\eal
\eeq
and the following weights for $\om$, $\rho_i$ for $\uu$ and the error 
\beq\label{wg:linf}
\bal
\vp_1 & = x^{-1/2} ( |x|^{-2.4} + 0.6 |x|^{-1/2} ) + 0.3 |x|^{-1/6}, \quad \vp_{g1} = \vp_1 +   |x|^{1/16}, \\
\vp_{elli} & = |x_1|^{-1/2} ( |x|^{-2} + 0.6 |x|^{-1/2}) + 0.3 |x|^{-1/6} ,  \quad \rho_{10} = |x|^{-3} + |x|^{-7/6} , \quad 
\rho_{20} = \psi_1.  \\
& \rho_3 = |x|^{-1} + |x|^{-1/6}, \quad \rho_4 =  x^{-1/2} ( |x|^{-2.5} + 0.6 |x|^{-1/2} ) + 0.3 |x|^{-1/6} .
\eal
\eeq

To estimate the weighted $L^{\inf}$ norm of the residual error in Section \ref{sec:lin_evo}, we use $\psi_i, \vp_{evo, i}$ 
\beq\label{wg:lin_evo}
\bal
& \vp_{evo, 1} =\vp_1, \quad \vp_{evo, 2}  = x^{-1/2} ( \td p_{5, 1} |x|^{-5/2} + \td p_{5, 2} |x|^{-3/2} + \td p_{5,3} |x|^{-1/6} ) + \td p_{5,4} |x|^{-1/4} +  \td  p_{5, 5} |x|^{1/7}, \\
& \vp_{evo, 3}  = x^{-1/2} ( \td p_{6, 1} |x|^{-5/2} + \td p_{6, 2} |x|^{-3/2} + \td p_{6,3} |x|^{-1/6} ) + \td p_{6,4} |x|^{-1/4} + \td p_{6, 5} |x|^{1/7}, \\
& \td p_{5, \cdot} = (0.42,  \ 0.135, \  0.216, \  0.182, \ 0.0349) \cdot \mu_0, \quad \mu_0 = 0.917, \\ 
& \td p_{6, \cdot}  = ( 2.5 \cdot \td p_{5,1} , 2.9 \cdot \td p_{5, 2},  \ 3.115 \cdot \td p_{5,3} ,  \ 1.82 \cdot \td p_{5,4},  \ 2.72 \cdot \td p_{5,5}  ), \\
\eal
\eeq
where $\vp_1$ is defined in \eqref{wg:linf}.

In our energy estimates and the estimates of the nonlocal terms, we need various estimates of the weights and their derivatives. From Appendix {\appwgPI} of Part I \cite{ChenHou2023a} and \eqref{wg:linf}, \eqref{wg:hol}, we have two types of weights. The first one is the radial weight 
\[
\rho(x, y) = \sum_i  p_i r^{a_i}, \quad  r = (x^2 + y^2)^{1/2},
\]
where $a_i$ is increasing and $p_i \geq 0$. We use these weights for the H\"older estimates. See e.g. \eqref{wg:hol}.

The second type of weights is the following 
\[
\rho(x, y) = \rho_1(r) |x|^{-\al} + \rho_2(r),
\]
where $\rho_1, \rho_2$ are the radial weights. 

We use $f_l, f_u$ to denote the lower and upper bound of $f$. We have the following simple inequalities 
\beq\label{eq:func_intval}
\bal
& (f-g)_{l} = f_l - g_u, \quad (f-g)_u = f_u - g_l , \quad (f+g)_{\g} = f_{\g} + g_{\g},  \\
& (f g)_l = \min( f_l g_l , f_u g_l, f_l g_u, f_u g_u),
\quad (f g)_u = \max( f_l g_l , f_u g_l, f_l g_u, f_u g_u).
\eal
\eeq
where $\g = l, u$. If $ g \geq 0$, we can simplify the formula for the product 
\beq\label{eq:func_intval2}
 (f g)_{l} =  \min( f_l g_l, f_l g_u), \quad ( fg)_u = \max( f_u g_l, f_u g_u).
\eeq

Given the piecewise bounds of $ \pa^j f, \pa^j g, j\leq k$, we can estimate $ \pa^k (fg)$ using the Leibniz rule
\beq\label{eq:lei}
 |\pa_x^i \pa_y^j (f g) | \leq  \sum_{k \leq i, l\leq j} \binom{i}{k} \binom{j}{l}
 |\pa_x^k \pa_y^l f| \cdot 
  |\pa_x^{i-k} \pa_y^{j-l} g | .
\eeq

\subsection{Radial weights}\label{app:wg_radial}

The advantage of radial weights $\rho$ is that we can estimate them easily. Since 
$\rho(x, y)$ is even in $x, y$, we restrict the estimate of piecewise bounds to the case of $x \geq 0, y \geq 0$. The bound in general domain
$D = [a, b] \times [c, d]$ can be obtained by decomposing $D$ into four quadrants and then using the symmetry and combining the bounds from different quadrants.

\subsubsection{Bounds for the derivatives}\label{sec:wg_radial_deri}
We can easily derive the derivatives and their upper and lower bound as follows. Firstly, we have 
\beq\label{eq:WG_radial_1}
 (\pa_x^i \pa_y^j \rho(x, y) )_{\g} = \sum_{1\leq k \leq n} p_k ( \pa_x^i \pa_y^j r^{a_k} )_{\g},
\eeq
where $\g = l, u$. Using induction, for any $\al, i, j$, we can obtain 
\[
\pa_x^i \pa_y^j r^{\al} = \sum_{  k \leq i + j, l \leq \min(j, 1)} C_{i,j,k,l}(\al) x^k y^l r^{\al -i - j- k -l}
= \sum_{  k \leq i + j, l \leq \min(j, 1)} ( C^+_{i,j,k,l}(\al) 
- C^-_{i,j,k,l}(\al)) x^k y^l r^{\al -i - j- k -l} ,
\]
with $C^{\pm}_{i,j,k,l}(\al) \teq \max( 0, C_{i,j,k,l}(\al))$. The bounds for $C^{\pm}_{i,j,k,l}(\al) x^k y^l r^{\al -i - j- k -l}$ are simple:
\beq\label{eq:WG_radial_3}
( C^{\pm}_{i,j,k,l}(\al) x^k y^l r^{\al -i - j- k -l} )_{\g}
= C^{\pm}_{i,j,k,l}(\al) x_{\g}^k y_{\g}^l r_{\g}^{\al -i - j- k -l}.
\eeq

In particular, we use the derivatives bound for $i+j \leq 4$ and we have
\[
\bal
& \pa_x r^a = a x r^{a-2},  \quad 
\pa_x^2 r^a = a r^{a-2} + a(a-2) x^2 r^{a-4}, \quad \pa_{xy} r^a = a(a-2) x y r^{a-4} , \\
& \pa_x^3 r^a =a(a-2) (a-4) x^3 r^{a-6} + 3 a(a-2) x r^{a-4} , 
\ 
\pa_x^2 \pa_y r^a = a (a-2) y r^{a-4} + a(a-2)(a-4) x^2 y r^{a-6} , \\
&\pa_x^4 r^a 
 = 3 a (a-2) r^{a-4} + 6 a (a-2)(a-4) x^2 r^{a-6}
 + a(a-2) (a-4)(a-6) x^4 r^{a-8} , \\
 & \pa_x^3 \pa_y r^a 
 = a(a-2)(a-4) xy r^{a-6} + 2 a (a-2)(a-4) xy r^{a-6}
 + a(a-2)(a-4)(a-6) x^3 y r^{a-8} , \\
 & \pa_x^2 \pa_y^2 r^a = a(a-2)(a-3) r^{a-4} 
 + a(a-2)(a-4)(a-6) x^2 r^{a-6}
 - x^4 a(a-2)(a-4)(a-6) r^{a-8}.
\eal
\]

Using \eqref{eq:func_intval}, the above identities, and linearity, we can obtain the upper and lower bounds for $\pa_x^i \pa_y^j \rho$. 
Since $\rho(x, y)$ is symmetric in $x, y$, we have $\pa_1^i \pa_2^j \rho(x, y)
= (\pa_1^j \pa_2^i \rho) (y,x)$ and can obtain piecewise bounds of $\pa_1^i \pa_2^j \rho$ from that of $\pa_1^j \pa_2^i \rho$.

For the estimate in Section \ref{sec:hol_1D}, we need to use the  estimates of $ \pa_x^i \pa_y^j \rho(\lam x)$ for very small $\lam \leq \lam_*$ or very large $\lam \geq \lam_*$ uniformly. Obviously, the bounds are mainly determined by the leading order power of $p(\lam x)$, i.e. $p_1 | \lam r|^{a_1}$ for small $\lam$ and $p_n |\lam r|^{a_n}$ for large $\lam$. We would like to estimate $(\pa_x^i \pa_y^j \rho(\lam x) )_{\g } \lam^{ - \b}$ for $\lam \leq \lam_*, \b = a_1$ and $\lam \geq \lam_*, \b = a_n$, $\g = l, u$.  Using the above derivations \eqref{eq:WG_radial_1}, 
we have 
\[
\lam^{-\b}(\pa_x^i \pa_y^j \rho(x, y) )_{\g} = \sum_{1\leq k \leq n} p_k ( \pa_x^i \pa_y^j \lam^{a_k -\b} r^{a_k} )_{\g}, \quad \g = l, u,
\]
and we only need to derive the upper and the lower bounds for $C^{\pm}_{i,j,k,l}(a_m) x^k y^l r^{\al -i - j- k -l} \lam^{a_m - \b}$ uniformly for $\lam \leq \lam_*, \b = a_1$ or $\lam \geq \lam_*, \b = a_n$. Since $a_i$ is increasing, in the first case, we have
\[
\lam^{a_1 - a_1} = 1, \quad  a_m - a_1 > 0,   \quad   (\lam^{ a_m - a_1})_l = 0,  \
(\lam^{ a_m - a_1})_u =\lam_*^{ a_m - a_1} , m > 1.
\] 
In the second case, we get 
\[
\lam^{a_n - a_n} = 1, \quad  a_m - a_n < 0,   \quad   (\lam^{ a_m - a_n})_l = 0,  \
(\lam^{ a_m - a_n})_u =\lam_*^{ a_m - a_n} , m > 1.
\]
In both cases, if $a_m = \b$, we get a trivial bound $1$ for $\lam^{a_m - \b}$; if $a_m \neq \b$, we get $ 0 \leq \lam^{a_m -\b} \leq \lam_*^{a_m -\b}$. Using these bounds for $\lam^{a_m -\b}$, \eqref{eq:WG_radial_3}, \eqref{eq:func_intval}, \eqref{eq:func_intval2}, we obtain the bounds for $ \lam^{-\b}\pa_x^i \pa_y^j \psi(\lam x)$ uniformly for small $\lam, \b = a_1$ and large $\lam , \b =a_n$.

We also need to bound $ M = \lam^{-\b} \rho_{\lam}(x) \B| \f{ \rho_{\lam}(y) }{ \rho_{\lam}(x)} - \f{ \rho_{lim}(y) }{ \rho_{lim}(x)} \B| $ used in Section \ref{sec:hol_1D}, uniformly for $\lam \leq \lam_*, \b = a_1, \rho_{lim}(y) = p_1 |y|^{a_1}$ or $\lam \geq \lam_*,\b = a_n, \rho_{lim}(y) = p_n |y|^{a_n}$. Using the formula of $\rho$ and a direct computation yield
\[
\f{\rho_{lim}(y) }{ \rho_{lim}(x)} = \f{|y|^{\b}}{|x|^{\b}}, \quad M \leq \sum_{i \leq n} p_i \lam^{a_i - \b } \B| |y|^{a_i} - |x|^{a_i} \f{ |y|^{\b}}{|x|^{\b}} \B|
\leq 
\sum_{i \leq n} p_i \lam_*^{a_i - \b } |y|^{\b} \B| |y|^{a_i-\b} - |x|^{a_i-\b} \B|.
\]
We remark that the leading power $ \lam_*^{a_i - \b}$ for $a_i = \b$ is cancelled due to $|y|^0 = |x|^0 = 1$ in the above estimate and we gain the small factor $\lam_*^{a_i - \b}$ for $a_i \neq \b$.

\subsubsection{Leading order behavior of $\pa \rho / \rho$}\label{sec:wg_radial_regu}
In our verification, we need to bound $ \pa \rho( \lam x)  / \rho(\lam x)$ as $\lam \to 0$ or $\lam \to \infty$ uniformly. A direct calculation yields 
\[
\f{ \pa_{x_i} \rho}{\rho}
= \f{x_i}{|x|^2} \f{ \sum_i p_i a_i r^{a_i} }{ \sum_i p_i r^{a_i}}
\teq \f{x_i}{|x|^2} S(x) ,  \quad S(x) \teq \f{ \sum_i p_i a_i r^{a_i} }{ \sum_i p_i r^{a_i}}.
\] 

For $x$ close to $0$, we introduce $b = a - a_1$. Clearly, we get $b_i \geq 0$ and 
\[
S(x) =  a_1 + \f{ \sum_i p_i b_i r^{a_i} }{ \sum_i p_i r^{a_i}} 
= a_1 + \f{ \sum_i p_i b_i r^{b_i} }{ \sum_i p_i r^{b_i}}  \teq a_1 + \f{A(r)}{B(r)}.
\]
Using $b_i\geq 0$ and the Cauchy-Schwarz inequalities, we yield 
\[
A^{\pr} B - A B^{\pr} = r^{-1} \B(  (\sum p_i b_i^2 r^{b_i} ) (\sum p_i r^{b_i}) 
- (\sum p_i b_i r^{b_i})^2 \B) 
= r^{-1} \f{1}{2} \sum_{ij} p_i p_j (b_i- b_j)^2 r^{b_i + b_j }
\geq 0,
\]
and thus $A / B$ is increasing. For $\lam \leq \lam_*, r \in [r_l, r_u]$, we get the uniform bound for $S(\lam x)$
\[
 a_1 \leq S( \lam x ) \leq a_1 + \f{ A(  \lam_* r_u )}{ B(\lam_* r_u)}.
\]

For $\lam = 1$, we simply obtain 
\[
     a_1 +  \f{ A( r_l )}{ B( r_l)} \leq S(x) \leq a_1 + \f{ A( r_u )}{ B( r_u)}.
\] 

Similarly, for $ \lam \geq  \lam_*$, $r \in [r_l, r_u]$, we get 
\[
a_n + \f{A( \lam_* r_l)}{ B(\lam_* r_l)} \leq S(\lam x) \leq a_n, 
\quad  \f{A(r)}{B(r)} = \f{ \sum_i p_i b_i r^{b_i}}{ \sum_i p_i r^{b_i}},
\]
where $b = a - a_n \leq 0$. Here, we have used that $A(r) / B(r)$ is increasing. Thought $b_i$ is negative, we still have $(A/B)^{\prime} = \f{A^{\pr} B - A B^{\pr}}{B^2} > 0$. From the above estimates, we yield 
\[
\bal
\lim_{\lam \to 0} \lam \f{\pa_{x_i} \rho }{ \ \rho }(\lam x)  & = \f{x_i}{|x|^2} a_1 \teq R_0(x), \quad 
|\f{\pa_{x_i} \rho }{\rho}(\lam x) - R_0(\lam x) |  \leq \lam^{-1} \f{x_i}{|x|^2} \f{ |A( \lam_* x) | }{|B(\lam_* x) | }, \  \lam \leq \lam_* , \\
\lim_{\lam \to \infty} \lam  \f{\pa_{x_i} \rho }{\ \rho }(\lam x) & = \f{x_i}{|x|^2} a_n \teq R_{\inf}(x), \quad 
|\f{\pa_{x_i} \rho }{\rho}(\lam x) - R_{\inf}(\lam x) | \leq \lam^{-1} \f{x_i}{|x|^2} \f{ |A( \lam_* x) |}{ | B( \lam_* x) |}, \ \lam \geq \lam_*. 
\eal
\]

\subsubsection{Bounds for the derivatives of $1/ \rho$}\label{sec:wg_radial_inv}
The bounds for $ d_x^i d_y^j \rho^{-1}$ is more complicated since $\rho^{-1}$ is not linear in the summand $p_i r^{a_i}$. We need such estimates in the estimate of the velocity. Firstly, using the bounds in Section \ref{sec:wg_radial_deri} and \eqref{eq:func_intval2}, we can obtain the upper and the lower bounds for $R_{ij}$
\[
R_{ij} = \f{ \pa_x^i \pa_y^j \rho}{\rho}.
\]
For $i + j = 1$ and $k=2,3$, we use the estimate in Section \ref{sec:wg_radial_deri} to obtain the bounds for
\[
R_{10} = \f{x}{|x|^2} S(x), \quad R_{0, 1} = \f{y}{|x|^2} S(x), \quad (R_{ij})^k.
\]

In our estimate, we need $\pa_x^i \pa_y^j \rho^{-1}$ for $i+j \leq 3$. A direct calculation yields 
\[
\bal
& \pa_x \rho^{-1} = - \f{\rho_x}{\rho^2} = - \f{R_{10} }{\rho}, \quad
\pa_{x x} \rho^{-1}= - \f{\rho_{xx}}{\rho^2}  + 2 \f{\rho_x^2}{\rho^3}
= \rho^{-1} (  - R_{20} +2  R_{10}^2 ),  \\
& \pa_{xy} \rho^{-1} = - \f{ \rho_{xy}}{\rho} + \f{2 \rho_x \rho_y}{\rho^3}
 = \rho^{-1}( - R_{11}  + 2 R_{10} R_{01} )  ,\\
 & \pa_{xxx} \rho^{-1} = -\f{\rho_{xxx}}{\rho^2} + \f{6 \rho_{xx} \rho_x}{\rho^3}
  - \f{6 \rho_x^3}{\rho^4} 
  = \rho^{-1}( -R_{30} + 6 R_{20} R_{10} - 6 R_{10}^3 ) , \\
  & \pa_{xxy} \rho^{-1} = - \f{\rho_{xxy}}{\rho^2} + \f{2 \rho_{xx} \rho_y}{\rho^3}
   + \f{4 \rho_x \rho_{xy}}{ \rho^3}  -6 \f{\rho_x^2 \rho_y }{\rho^4}
   = \rho^{-1}( - R_{21} + 2 R_{20} R_{01} + 4 R_{10} R_{11} - 6 R_{10}^2 R_{01} ).
\eal
\]

Next, we estimate  $ \pa_x^i \pa_y^j ( \pa_{x_l} \rho / \rho)$ for $i \leq 2, j =0 $ or $i=0, j\leq 2$. Denote $f =\pa_{x_l} \rho$. Using a direct computation, for $D_2 = \pa_x^{i_2} \pa_y^{j_2}$ with $ i_2 + j_2 = 1$, we yield 
\[
D_2 \f{ f }{\rho} = \f{  D_2 f }{ \rho} - \f{ f D_2 \rho}{\rho^2}
= \rho^{-1}( D_2 f - f R_{i_2, j_2} ).
\]

For $(i_2, j_2) = (2,0) , (0, 2)$, denote $i_3 = i_2 / 2, \;j_3 = j_2 / 2$, $D_3 = \pa_x^{i_3} \pa_y^{j_3}$. We yield 
\[
\bal
D_3^2 \f{ f}{\rho}
& =  \f{D_3^2 f }{\rho} 
- \f{2 D_3 f \cdot D_3 \rho}{\rho^2} + f D_3^2( \f{1}{\rho})  = \f{D_3^2 f }{\rho} 
- \f{2 D_3 f \cdot D_3 \rho}{\rho^2} + f ( - \f{D_3^2 \rho}{ \rho^2} + \f{2 ( D_3 \rho)^2}{\rho^3} ) \\
& = \rho^{-1}( D_3^2 f - 2 D_3 f R_{i_3, j_3} - f R_{i_2, j_2}
+ 2 f R_{i_3, j_3}^2 ),
\eal
\]
where we have used  $D_3^2 \f{1}{\rho} = D_3 ( -\f{D_3 \rho}{\rho^2}) = - \f{D_3^2 \rho}{ \rho^2} + \f{2 ( D_3 \rho)^2}{\rho^3} $.

Since we have estimated $\pa_x^i \pa_y^j \rho$ and $R_{ij}$, we can bound these derivatives of $D_1 \rho / \rho$ using \eqref{eq:func_intval}.

We also need to obtain the uniform estimates of $ \lam^{\b} \pa_x^i \pa_y^j ( \rho^{-1}( \lam x))$
for $\lam \leq \lam_*, \b = a_1$ and $\lam \geq \lam_*, \b = a_n$. 
Denote $\rho_{\lam}(x) = \rho(\lam x)$. For example, for $D_1 = \pa_{x_i}$, we have 
\[
 \lam^{\b} D_1 ( \rho_{\lam}^{-1}( x)) = -\lam^{1 + \b}  \f{ (D_1 \rho)( \lam x)}{ \rho_{\lam}^2(x) }
 = - \lam^{1+\b} \rho_{\lam}^{-1}(x) \lam^{-1} \f{x_i}{|x|^2} S(\lam x)
 = - \lam^{\b} \rho_{\lam}^{-1}(x) \f{x_i}{|x|^2} S(\lam x),
\]
which can be estimated using the estimates in Sections \ref{sec:wg_radial_deri}, \ref{sec:wg_radial_regu}. The power $\lam^{\b}$ and the leading power $\lam^{-\b}$ in $\rho_{\lam}^{-1}(x)$ cancel each other. The estimates of $ \lam^{\b} \pa_x^i \pa_y^j ( \rho^{-1}(\lam x) )$ with $i+j \geq 2$ and $\pa_x^i \pa_y^j \f{\pa_{x_l} (\rho_{\lam})}{ \rho_{\lam} }$ are similar, and follow from the above derivations for $\pa_x^i \pa_y^j \rho^{-1}, \pa_x^i \pa_y^j( \pa \rho / \rho)$, the piecewise estimates for $ \pa_x^i \pa_y^j \rho(\lam x)$ in Section \ref{sec:wg_radial_deri} and $\f{\pa \rho}{\rho}(\lam x)$ in Section \ref{sec:wg_radial_regu}, which are uniform in small $\lam \leq \lam_* $ or large $\lam \geq \lam_*$. We remark that in all of these estimates for $\rho_{\lam}(x)$, taking derivatives in $x$ does not change the asymptotic power in $\lam$.

\subsubsection{Improved estimates for $\rho^{-1}$ near $x=0$}
For the special case $a_1 = -2$, we can write 
\[
\rho(x) = r^{-2} \sum_i p_i r^{a_i + 2} = r^{-2} \td \rho(x), \quad \rho^{-1} = (x^2 + y^2) \td \rho(x)^{-1}
\]
To obtain a better estimate of $\rho^{-1}$, we use the fact that $x^2 +y^2$ is a polynomial. Firstly, we can obtain the bounds for $\pa_x^i \pa_y^j \td \rho^{-1}$. The bound for $S_0 = x^2 + y^2$ is trivial, e.g., 
\[
(\pa_x S_0)_{\g} = 2 x_{\g},  (\pa_y S_0)_{\g} = 2 y_{\g}, \ \g = u , l, \ \quad \pa_{x y} S_0 = 0,  \quad \pa_{xx} S_0 = \pa_{yy} S_0 = 2.
\]
Then using \eqref{eq:func_intval}-\eqref{eq:func_intval2}, we can bound $\rho^{-1}$.

\subsection{The mixed weight}\label{app:wg_mix}

For the second type of weights $W = \rho_1(r) |x|^{-1/2} + \rho_2(r)$, we can compute its derivatives and its upper and lower bounds using linearity and the Leibniz rule \eqref{eq:lei}. We consider $x, y \geq 0$. For example, we have 
\[
W_l =  \rho_{1, l} x_u^{-1/2} + \rho_{2, l}, \quad 
(W^{-1})_u = ( W_l)^{-1}, \quad
W_x = \pa_x \rho_1 x^{-1/2} - \f{1}{2} \rho_1 x^{-3/2} + \pa_x \rho_2.
\]

To obtain the upper bound for $\pa_x^i \pa_y^j W$, we use the Leibniz rule \eqref{eq:lei}:
\[
|\pa_x^i \pa_y^j W| \leq \sum_{ k \leq i} \binom{i}{k} |\pa_x^{i-k} \pa_y^j \rho_1| \f{ (2k-1)!!}{2^k}
x^{-1/2 - k} + | \pa_x^i \pa_y^j \rho_2|.
\]

We need to bound $\rho(r) / W(x, y)$ in the estimate of the integrals. Suppose that the leading and the last powers of $\rho$ is $a_1, a_n$. The leading and the last terms of $W$ are given by $p_i r^{b_i} \cos(\b)^{-\al_i}, \al_i \geq 0$.
\[
W \geq p_1 r^{b_1}, \quad  W \geq p_n r^{b_n}. 
\]

We estimate 
\[
\f{\rho}{W} \leq C_1 r^{a_1 - b_1}, \quad \f{\rho}{W} \leq C_2 r^{a_n - b_n},
\]
for all $x, y \in \R_2^+$. We apply the above estimates for $x$ near $0$ or $x$ sufficiently large.

Using $W(\lam x) \geq \rho_1(\lam x) \lam^{-1/2} |x_1|^{-1/2}$, 
 $W(\lam x) \geq \rho_2(\lam x)$,  the uniform estimates of $\rho_i(\lam x)$ in $\lam$ in Section \ref{sec:wg_radial_deri}, we can obtain the lower bound of $W(\lam x)$  and the upper bound of $ \f{\rho(\lam x)}{ W(\lam x)}$ uniformly in $\lam$.

\section{Estimate the derivatives of the velocity kernel and integrands }\label{app:ker}

In this appendix, we estimate the derivatives of the kernel $ - \f{1}{2\pi} \log |x|$ associated to the velocity $\uu = \na^{\perp}(-\D)^{-1} \om$ and its symmetrization \eqref{int:ker_sym}. These estimates are used to estimate the error terms in Lemmas \ref{lem:T_rule}, \ref{lem:int_err_x}. We will perform an additional estimate for $u$ with weight $\vp(x)$ singular along $x_1 = 0$ in Section \ref{sec:u_dx1}. Some additional derivations related to the estimate of the velocity are given in Appendix \ref{app:u_add}.

\subsection{Estimate the symmetrized kernel}\label{app:ker_sym}

In this section, we estimate the symmetrized kernel. We develop several symmetrized estimates for harmonic functions. Before we introduce the estimates, we have a simple 1D estimate, which is useful for later estimates.

\begin{lem}\label{lem:ta_1D}
We have 
\[
|f(x)  + f(-x) - 2 f(0) | \leq x^2 || f_{xx}||_{ L^{\inf}[-x, x] } ,
\quad |f(x)  + f(-x) - 2 f(0) - x^2 f_{xx}(0) | \leq \f{x^4}{12} || \pa_x^4 f||_{ L^{\inf}[-x, x] } .
\]
\end{lem}

\begin{proof}
Denote $G(x) = f(x) + f(-x)$. Clearly, $G$ is even and 
\beq\label{eq:ta_1D_pf1}
G(0) = 2 f(0) , \quad G^{\pr}(0) =0, \quad  \pa_x^2 G(0) = 2 f_{xx}(0), \quad  \pa_x^3 G(0) = 0.
\eeq
Using the Taylor expansion, we obtain 
\[
G(x) = G(0) + G^{\pr}(0) x + \f{ \pa_x^2 G(0) x^2 }{2}  + \f{\pa_x^3 G(0) x^3 }{6} + 
 \f{\pa_x^4 G(\xi ) x^4}{24},
\]
for some $\xi \in [0, x]$. Using \eqref{eq:ta_1D_pf1}, we get 
\[
|G(x) - G(0) - G^{\pr \pr}(x) \f{x^2}{2} |  \leq || \pa_x^4 G ||_{L^{\inf}[0, x]} \f{x^4 }{24}\leq || \pa_x^2 f ||_{L^{\inf}[-x, x]}  \f{x^4}{12} . 
\]
Plugging the identity \eqref{eq:ta_1D_pf1} into the above estimate proves the second estimate in Lemma \ref{lem:ta_1D}. The first estimate is simpler. 
\end{proof}

The following lemma is useful for estimating the symmetrized  kernel \eqref{int:ker_sym} and its derivatives.

\begin{lem}\label{lem:ker_sym}
Suppose that $ Q_x = [-x_1, x_1 ] \times [-x_2, x_2]$ and $f \in C^4(Q_x)$ is harmonic. Denote 
\beq\label{eq:ker_sym_G}
\bal
G_1(1, x)  &\teq f( x_1, x_2) + f(- x_1,  x_2) + f(  x_1, -  x_2) + f(-  x_1, - x_2) - 4 f(0, 0)  ,\\
G_2(1, x) &\teq f( x_1, x_2) - f(- x_1,  x_2) - f(  x_1, -  x_2) + f(-  x_1, -x_2)  ,\\
\hat G_1(x)  & \teq  2 x_1^2  f_{xx}(0, 0) + 2 x_2^2 f_{yy}(0, 0) , \quad \hat{G}_2(x) \teq 4 x_1 x_2 f_{xy}(0, 0) .
\eal
\eeq

We have
\begin{align}
& |G_1(1, x)|  \leq 2 |x|^2 || f_{xx}||_{L^{\infty}(Q_x)} , \quad 
| \pa_{x_i} G_1(1, x)| \leq 4 |x_i| \cdot || f_{xx}||_{L^{\infty}(Q_x)}, \label{eq:ta_G1}  \\
&| G_1(1, x) - \hat G_1( x)| \leq   \f{(x_1^4 + 6 x_1^2 x_2^2 + x_2^4 ) }{6} || \pa^4 f ||_{L^{\infty}(Q_x)}
\leq \f{ |x|^4 }{3} || \pa^4 f ||_{L^{\infty}(Q_x)} ,
\label{eq:ta_G1_x_appr}  \\
&| G_1(1, x_1, 0) -\hat G_1( x_1 , 0)  | \leq \f{1 }{6} x_1^4 || \pa^4 f ||_{L^{\infty}(Q_x)},  
\label{eq:ta_G1_x1_appr}  
\\
& | \pa_{x_i} ( G_1(1, x) - \hat{G}_1(x) ) |  \leq 
 \f{2}{3}(  3 x_{3-i}^2 x_{i} +  x_{i}^3   )  || \pa^4 f ||_{L^{\infty}(Q_x)}  
\leq \f{ 2 \sqrt{2}}{3} |x|^3 || \pa^4 f ||_{L^{\infty}(Q_x)},
\label{eq:ta_G1x_appr}  
\end{align}
where $|| \pa^4 f||_{L^{\infty}} = \max_{0\leq i \leq 4}  ||\pa_x^i \pa_y^j f ||_{L^{\infty}(Q_x)}$. For $G_2$, we have the following estimate 
\begin{align}
& |G_2(1, x)|  \leq 4 x_1 x_2   || f_{xy}||_{L^{\infty}(Q_x)} , \quad 
| \pa_{x_i} G_2(1, x)| \leq 4 |x_{3-i}| \cdot || f_{xy}||_{L^{\infty}(Q_x)}, 
\label{eq:ta_G2} \\
&  | G_2(1, x) - \hat {G}_2(x)| \leq \f{2 x_1 x_2 |x|^2}{3} || \pa^4 f ||_{L^{\infty}(Q_x)} ,
\label{eq:ta_G2_x_appr}
 \\
&  | \pa_{x_i} ( G_2(1, x) - \hat{G}_2(x) )| \leq 
 \f{2}{3}( 3 x_i^2 x_{3-i} +  x_{3-i}^3   )  || \pa^4 f ||_{L^{\infty}(Q_x)}  
\leq \f{2 \sqrt{2}}{3} |x|^3 || \pa^4 f ||_{L^{\infty}(Q_x)} .
\label{eq:ta_G2x_appr}
\end{align}

\end{lem}

Note that $G_1( \cdot, x)$ is even in $x_i$, and $G_2(\cdot, x)$ is odd in $x_i$. The polynomials of $x_i$ in the upper bounds (without absolute value) have the same symmetries. Similar properties hold for $\pa G_1, \pa G_2$. Moreover the above bound satisfies the differential relation. These properties are useful for tracking different bounds for $G_1, G_2$.

\begin{proof}
Recall $Q_x = [-x_1, x_1] \times [-x_2, x_2]$. Denote 
\[
A_{ij}(x) = || \pa_x^i \pa_y^j f ||_{L^{\inf}( Q_x) }. 
\]
Using Lemma \ref{lem:ta_1D}, for any $t \in [0, 1]$, we obtain
\[
| f( t x_1, x_2 ) + f(t x_1, - x_2) - 2 f(t x_1, 0) | \leq A_{02} x_2^2, \quad 
|f(x_1, 0) + f(-x_1, 0 ) - 2 f(0, 0)| \leq A_{20} x_1^2. 
\]
Since $f$ is harmonic function, we have $ \pa_x^{i+2} \pa_y^j f = - \pa_x^i \pa_y^{j+2} f$ and obtain $A_{i+2, j} = A_{i, j+2}$. Taking $ t = \pm 1$ in the above estimate and using the triangle inequality, we prove 
\[
|G(1, x) | \leq 2 A_{20} x_1^2 + 2 A_{02} x_2^2 = 2 A_{20}(x_1^2 + x_2^2) 
= 2 A_{20} |x|^2,
\]
which is the first estimate in \eqref{eq:ta_G1}. 

The second estimate in \eqref{eq:ta_G1} is simple. We consider $i=1$ without loss of generality. We get
\[
|\pa_{x_1} G_1(1, x) |
= | (  \pa_1 f) (x_1, x_2) - (\pa_1 f)(-x_1, x_2) +  (\pa_1 f)(x_1, -x_2)  -  (\pa_1 f)(-x_1, -x_2)  )|
\leq 4 x_1  A_{20}(x).
\]

For \eqref{eq:ta_G1_x_appr}, using Lemma \ref{lem:ta_1D}, we yield 
\beq
\bal
 |f( t x_1, x_2) + f(t x_1,  - x_2) - 2 f(tx_1, 0) - x_2^2 (\pa_2^2 f)(tx_1, 0)| 
 & \leq A_{04}(x) \f{x_2^4}{12}, \\
|\pa_2^2 f( x_1, 0) + \pa_2^2 f( - x_1, 0) - 2 \pa_2^2 f(0, 0) | & \leq x_1^2 A_{2, 2}(x), \\
|f(x_1, 0) + f( -x_1, 0) -2 f(0) - x_1^2 \pa_1^2 f(0)|  & \leq A_{40} \f{x_1^4}{12},
\eal
\eeq
for $t = \pm 1$. Combining the above  estimates and using the triangle inequality and $A_{40} = A_{22} = A_{04}$, we prove the first estimate in \eqref{eq:ta_G1_x_appr}. The second estimate follows from $2|x|^4 - x_1^4 - 6x_1^2 x_2^2 - x_2^4 = (x_1^2 - x_2^2)^2 \geq 0$.

Estimate \eqref{eq:ta_G1_x1_appr} follows from \eqref{eq:ta_G1_x_appr} by taking $x_2 = 0$.

For \eqref{eq:ta_G1x_appr}, we consider the estimate of $\pa_{x_1}$. The other case is similar. Using 
\[
\pa_1 f (x_1, s) - (\pa_1 f)(-x_1, s)  = \int_0^{x_1}  (\pa_1^2 f ) ( t, s) + (\pa_1^2 f )(-t, s) dt ,\]
we obtain
\[
\bal
\pa_1( G(1, x) &-  \hat G_1(x))
 =   (\pa_1 f)( x_1, x_2) - (\pa_1 f)( -x_1, x_2) 
+ (\pa_1 f)(x_1, -x_2) - (\pa_1 f) (-  x_1, -x_2 ) - 4 x_1 \pa_1^2 f(0) \\
& = \int_{0}^{x_1} \B( ( \pa_1^2 f )( z, x_2) + (\pa_1^2 f )(-z, x_2)
+ (\pa_1^2 f)(z, -x_2 ) + (\pa_1^2 f)(-z, x_2) - 4 \pa_1^2 f(0) \B) dz.
\eal
\]

Applying \eqref{eq:ta_G1}, we yield 
\[
|\pa_1( G(1, x) -  \hat G_1(x))| 
\leq  \int_0^{x_1} 2 (z^2 + x_2^2) dz    A_{4,0}(x)
= (\f{2}{3} x_1^3 + 2 x_1 x_2^2 ) A_{4,0}(x),
\]
and complete the proof of the first estimate in \eqref{eq:ta_G1x_appr}. For the second estimate, we use the AM-GM inequality to yield 
\beq\label{eq:ker_sym_pf2}
(3 x_2^2 x_1 + x_1^3)^2
 = (3 x_2^2 + x_1^2)^2 x_1^2 
 = \f{1}{4}(3 x_2^2 + x_1^2)^2 4 x_1^2 
 \leq \f{1}{4} (  \f{  2(3 x_2^2 + x_1^2) + 4 x_1^2   }{3} )^3 = 2 |x|^6.
\eeq
Taking a square root completes the estimate.

To estimate $G_2$ in \eqref{eq:ker_sym_G}, we rewrite it as follows 
\beq\label{eq:lem_ker_sym_G21}
\bal
 G_2(1, x) - c \hat G_2 (x)   = & \int_{ -x_1}^{x_1} \int_{-x_2}^{x_2} \pa_{12} f (z_1, z_2) - c  \pa_{12} f( 0) d z \\
 =& \int_0^{x_1} \int_0^{x_2}
  (\pa_{12}f) (z_1, z_2) + 
(\pa_{12} f)(-z_1, z_2) + ( \pa_{12} f) (z_1, -z_2)  \\ 
& + (\pa_{12} f)(-z_1, z_2 ) - 
4 c (\pa_{12} f)(0) d z,
\eal
\eeq
for $c = 0, 1$. The integrand has the same form as $G_1$ in \eqref{eq:ker_sym_G}. For $c =0$, using the above decomposition, we prove 
\[
|G_2(1, x)| \leq 4 x_1 x_2 A_{11}.
\]
When $c = 1$, using \eqref{eq:ta_G1x_appr}, we yield 
\[
|G_2(1, x) - \hat G_2 (x) | 
\leq A_{40} 2 \int_0^{x_1} \int_0^{x_2} |y|^2 dy  
= A_{40} \f{ 2}{3}  (x_1^3 x_2 + x_1 x_2^3 ) 
= A_{40} \f{ 2}{3}  x_1 x_2 |x|^2.
\]

To estimate the derivatives, we focus on $\pa_{x_1}$. Using the above representation, we obtain 
\begin{eqnarray*}
\bal
\pa_{x_1} ( G_2(1, x) - c \hat G_2 (x) ) 
= &\int_0^{x_2}  ((\pa_{12}f) (x_1, y_2) +  (\pa_{12} f)(-x_1, y_2))  \\
 & +  ( ( \pa_{12} f) (x_1, -y_2) + (\pa_{12} f)(-x_1,  - y_2 ) - 
4 c (\pa_{12} f)(0)) dy.
\eal
\end{eqnarray*}

We apply the same estimates to the integrands with $c =0, 1$ and yield 
\[
|\pa_{x_1} G_2(1, x)| \leq 4 x_2 A_{11}, \quad 
|\pa_{x_1} ( G_2(1, x) - \hat G_2 (x) ) | \leq A_{31} 2  \int_0^{x_2} (x_1^2 + y_2^2) dy_2
 = A_{31}  (2 x_1^2 x_2 + \f{2}{3} x_2^3  ).
\]
The second inequality in \eqref{eq:ta_G2x_appr} follows from \eqref{eq:ker_sym_pf2}. The above estimates imply \eqref{eq:ta_G2}-\eqref{eq:ta_G2x_appr}.
\end{proof}

Recall the kernels associated with $\na \uu, \uu$ in \eqref{eq:kernel_du}. These kernels are the derivatives of the Green function $-\f{1}{2\pi} \log |x|$ and are harmonic away from $0$. 
We have the following estimates for their derivatives.

\begin{lem}\label{lem:K_mix}
Denote $r = (x^2 + y^2)^{ \f{1}{2}}$ and $f(x, y) = \log r $. For any $i, j \geq 0$ with $i+ j \geq 1$, we have 
\[
 | \pa_x^i \pa_y^j f(x, y)| \leq (i+j - 1)! \cdot r^{- i - j}.
\]

As a result, for $K_1(y) = - \f{1}{2} \pa_{12} f(y), K_2(y)= - \f{1}{2} \pa_1^2 f(y)$, we have 
\[
\bal
 &|  K_i| \leq  \f{1 }{2 |y|^{2}},
  \quad | \pa^j_{y_1} \pa^{2-j}_{y_2}  K_i| \leq \f{3}{|y|^4}, \quad 
 | \pa^{j}_{y_1} \pa^{4-j}_{y_2}  K_i| \leq \f{60}{ |y|^6}, 
\quad | \pa^j_{y_1} \pa^{6-j}_{y_2} K_i| \leq \f{2520}{ |y|^8} , \quad i = 1, 2.
  \eal
\]
\end{lem}

\begin{proof}

Consider the polar coordinate $\b = \arctan( y/x),  r= (x^2 +y^2)^{1/2} $. We use induction on $n = i+j$ to prove 
\beq\label{eq:K_mix_pf1}
\pa_x^i \pa_y^j f = (n-1)!  \cos(n \b - \b_{ij} ) r^{-n},
\eeq
for some constant $\b_{ij}$. We have the formula 
\beq\label{eq:Dxy}
\pa_x g = (\cos \b \pa_r - \f{ \sin \b}{r} \pa_{\b} ) g, \quad  \pa_y  g= (\sin \b \pa_r + \f{ \cos \b}{r} \pa_{\b} ) g.
\eeq

Firstly, for $n = 1$, a direct calculation yields 
\[
\pa_x f  = \f{x}{r^2} = \f{\cos \b}{r}, \quad \pa_y f = \f{y}{r^2} = \f{\sin \b}{r} = \f{ \cos( \b- \pi/2)}{r}.
\]
Suppose that \eqref{eq:K_mix_pf1} holds for any $i,j$ with $i+j = n$ and $n \geq 1$. Now, since 
\[
\bal
\pa_x \pa_x^i \pa_y^j f  &=(n-1)! \pa_x ( \cos( n \b - \b_{ij}) r^{-n} )  \\
&= (n-1)! (-n \cos \b \cos( n \b - \b_{ij}) r^{-n-1} +  n \sin \b \sin( n\b - \b_{ij})  r^{-n-1} ) \\
&=  n! (  - \cos( n\b -\b_{ij} + \b)   r^{-n-1}) 
= n! \cos( (n+1) \b - \b_{ij} - \pi)  r^{-n-1},
\eal
\]
using a similar computation and $\sin( x) = \cos(x - \pi/2)$, we can obtain that $\pa_y \pa_x^i \pa_y^j f$ has the form \eqref{eq:K_mix_pf1}. Using induction, we prove \eqref{eq:K_mix_pf1}. The desired estimate  follows from \eqref{eq:K_mix_pf1}.
\end{proof}

Using the above two Lemmas, we can estimate the error in the discretization of the kernels $K(x, y)$ in both $x$ and $y$ directions.

\subsubsection{Estimate the kernels in the far field}\label{app:decay}
We apply Lemma \ref{lem:ker_sym} to estimate the decay of $F_1, F_2$
\beq\label{eq:ker_decay3}
\bal
&  F_0 \teq G(y-x) - G(y_1 - x_1, y_2 + x_2 )
- G(y_1 + x_1, y_2 - x_2) + G(y+x), \ G(y) = -\log |y|/ 2 , \\
& F_1 \teq F_0 - 4 x_1 x_2 \pa_{12} G(y), \ F_2 \teq F_1 - \f{ 2 (x_1^2 - x_2^2) x_1 x_2 }{3} \pa_1^3 \pa_2 G(y), 
\  I_{ijkl}( P ) \teq \pa_{x_1}^i \pa_{x_2}^j \pa_{y_1}^k \pa_{y_2}^l P(x, y).
\eal
\eeq

Note that for stream function $\phi = (-\D)^{-1} \om(y) = C \cdot G \ast W $, where $W$ is the odd extension of $\om$ from $\R_2^+$ to $\R_2^{++}$, since $G(z)$ is even in $z_i$, after symmetrization, we have 
\[
\td \phi(x) = \phi(x) - x_1 x_2 \phi_{12}(0)= C \int_{\R^2} G(y-x) W(y) dy
= C \int F_1(x, y) W(y) dy, 
\]
where $\phi_{12}(0)$ is related to $C_{f0} K_{ux0}$ in \eqref{eq:u_appr_near0_coe}. In the estimate of $\uu, \na \uu$ related to $\pa_{x_1}^i \pa_{x_2}^j \td \phi$, e.g. $(1, 1)$ for $u_x= -\pa_{x_1 x_2} \phi$, for $y \in Q$ away from the singularity, we get the symmetrized integrand
\[
\pa_{x_1}^i \pa_{x_2}^j \int_Q  F_1(x, y) W(y) dy 
= \int_Q \pa_{x_1}^i \pa_{x_2}^j F_1(x, y) W(y) dy.
\]

In the error estimate of the Trapezoidal rule Lemma \ref{lem:T_rule}, we estimate $ 
\pa_{x_1}^i \pa_{x_2}^j \pa_{y_i}^2 F_1(x, y)$, which is $I_{ij 20}(F_1)$ or $I_{ij02}(F_1)$ in \eqref{eq:ker_decay3}.
We apply the estimate of $F_2$ to $K_f - C_{f0} K_{ux0} - C_f K_{00}$ \eqref{eq:u_appr_near0_coe}. 
Below, we show that $I_{ijkl}(F_i), i=1,2$ has faster decay in $|y|$ than $ \pa_{x_1}^i \pa_{x_2}^j \pa_{y_1}^k \pa_{y_2}^l G(y + x)$.

By definition, we get $i_1, j_1 \leq 1$. Next, we fix $y$ and introduce 
\beq\label{eq:ker_sym_g}
g_{pq}(z) \teq  \pa_{y_1}^{p} \pa_{y_2}^{q} G( y + z) ,
\quad M_{G, k} \teq \max_{ a + b = k}  || (\pa_{y_1}^a \pa_{y_2}^b G)(y + \cdot ) ||_{L^{\inf}(Q_x)}, \ Q_x = [-x_1, x_1] \times [-x_2, x_2].
\eeq

Since $G$ is harmonic, we have
\beq\label{eq:ker_sym_sgn}
\bal
&\pa_{x_i}^k G( y_1 + s_1 x_1, y_2 + s_2 x_2) = s_i^k \pa^k_{y_i}  G( x_1 + s_1 y_1, x_2 + s_2 y_2) , \ s_l \in \{ \pm 1 \} , \\
& \pa_1^2 G(y) = - \pa_2^2 G(y), \quad \pa_{x_1 x_2} g_{ p q}(x) |_{x=0}
 = \pa_{y_1}^{p+1} \pa_{y_2}^{q+1} G(y), \quad
  \pa_{22} g_{rs}(0) = -\pa_{11} g_{rs}(0).
\eal
\eeq

\vs{0.1in}
\paragraph{\bf{Second approximation $F_2$}}
Note that taking $\pa_{y_i}$ in $F_i$ does not change the sign of coefficient of $G$ term in \eqref{eq:ker_decay3}. Applying \eqref{eq:lem_ker_sym_G21} with $c=1$ and $f(z) = g_{r s}(z)$ in $G_2$, we yield 
\[
\bal
& I_{pqrs}(F_2)
 =   \pa_{x_1}^p \pa_{x_2}^q \int_0^{x_1} \int_0^{x_2} g_{ {r+1}, {s+1}, all}(z) dz ,   \\
& g_{\al \b, all}(z) =  g_{\al \b}(z) + g_{\al \b}(-z)  + g_{\al \b}(z_1, -z_2) + g_{\al \b}(-z_1, z_2)
 - 4 g_{\al \b}(0) -  2 (z_1^2 - z_2^2) \pa_{11} g_{\al \b}(0) .
\eal
\]

If $\max(i, j) \leq 1$, using the above notation to $I_{ijkl}(F_2)$ and the estimate of $G_1 - \hat G_1$ in Lemma \ref{lem:ker_sym} with $f = g_{kl}$, and then integrating the bounds in $z_2$, we get 
\[
|I_{1 0 kl}(F_2) | =| \int_0^{x_2} g_{k+1,l+1,all}(x_1, z_2) d z_2| 
\leq M_{G, d_2} \int_0^{x_2} \f{x_1^4 + 6 x_1^2 z_2^2 + z_2^4}{6} d z_2
=   ( \f{x_1^4 x_2}{6} +\f{x_1^2 x_2^3}{3} + \f{x_2^5}{30} ) M_{G, d_2 }, 
\]
where $ d_2 = k+ l + 6$. Similarly, we get 
\[
|I_{01 kl}(F_2)| \leq ( \f{x_1^5 }{30} +\f{x_1^3 x_2^2}{3} + \f{x_1^4 x_2 }{6} ) A_{G, d_2 }, 
\quad I_{11kl}(F_2) \leq \f{x_1^4 + 6 x_1^2 x_2^2 + x_2^4}{6} A_{G, d_2}.
\]

If $\max(i,j) \geq 2, i+ j \leq 3$, without loss of generality, we consider $i \geq 2$. 
We choose $(i_1, j_1, k_1, l_1)  = ( i-2,  j, k+2, l) $. 
From \eqref{eq:ker_sym_sgn}, we get
\begin{eqnarray*}
 &&\pa_{x_1}^2 ( x_1 x_2 \pa_{12} G(y) ) =0,\\
  &&\pa_{x_1}^2 \pa_{y_1}^{k } \pa_{y_2}^{ l} ( \f{2 (x_1^2 - x_2^2) x_1 x_2}{3}  \pa_1^3 \pa_2 G(y) )
= 4 x_1 x_2   \pa_{y_1}^{k_1+1 } \pa_{y_2}^{l_1 + 1} G)(y)
= 4 x_1 x_2 \pa_{12} g_{k_1 l_1}(0).
\end{eqnarray*}

Using \eqref{eq:ker_sym_sgn} again, 
we rewrite $\pa_{x_1}^i \pa_{y_1}^k G(x+y) = \pa_{x_1}^{i_1} \pa_{y_1}^{k_1} G(x+y)$ and get
\beq\label{eq:ker_sym_G23}
I_{ijkl}(F_2)
= \pa_{x_1}^{i_1} \pa_{x_2}^{j_1} 
( g_{k_1 l_1}(x) 
- g_{k_1 l_1}(x_1, -x_2) - g_{k_1 l_1}(-x_1, x_2) 
+ g_{ k_1 l_1}(-x) -4 x_1 x_2 \pa_{12} g_{k_1 l_1}(0)
  ) . 
\eeq
The same derivation applies to the case of $ j \geq 2$, where we choose $(i_1, j_1,k_1, l_1) = (i, j-2, k, l+2)$. Since $i_1, j_1 \leq 1$, using the estimate of $G_2 - \hat G_2$ in Lemma \ref{lem:ker_sym} with $f = g_{k_1 l_1}$, we get 
\[
\bal
 |I_{20 kl}(F_2) | , |I_{02 kl}(F_2)| & \leq \f{2 x_1 x_2 |x|^2}{3} M_{G, d_2}, \    (i_1, j_1) = (0, 0),  \\
|I_{30 kl}(F_2) |, |I_{12 kl}(F_2)| & \leq 
\f{2}{3} ( 3 x_1^2 x_2 + x_2^3 ) M_{G, d_2} , (i_1, j_1) = (1, 0), \\
 |I_{21 kl}(F_2) |, |I_{03 kl}(F_2)| & \leq 
\f{2}{3} (  x_1^3  + 3 x_1 x_2^2 ) M_{G, d_2},  (i_1, j_1) = (0, 1), \ d_2 = k_1 + l_1 + 4 = k+ l + 6.
\eal
\]

Note that the form \eqref{eq:ker_sym_G23} can be seen as the $\pa_{x_1}^{i_1} \pa_{x_2}^{i_2} F_1$. If $4 \leq i+j \leq 5$, we still first perform \eqref{eq:ker_sym_G23} by choosing $(i_1, j_1, k_1, l_1) =(i-2, j, k+2, l)$ or $(i, j-2, k, l+2)$ and get 
\[
 I_{ijkl}(F_2) = I_{i_1 j_1 k_1 l_1}(\td F_1), 
\]
where $\td F_1$ is similar to $F_1$ in \eqref{eq:ker_decay3} with $G$ replaced by $
g_{i-i_1, j-j_1} =  \pa_{y_1}^{i-i_1} \pa_{y_2}^{j-j_1} G(y)$. Then we apply the estimate for the first approximation below with $i_1 + j_1 \leq 3$.

\vs{0.1in}
\paragraph{\bf{First approximation}}

The estimate of $I_{ijkl}(F_1)$ is similar. Denote
\[
\bal
i_2 = i - 2 \lfloor \f{i}{2} \rfloor , \ 
j_2 = j -  2 \lfloor \f{j}{2} \rfloor, \
k_2 = k + 2 \lfloor \f{i}{2} \rfloor,
\quad l_2 = l +  2 \lfloor \f{j}{2} \rfloor. 
\eal
\]
If $\max(i, j)  \leq  1$, we get $(i,j,k,l) = (i_2,j_2, k_2, l_2)$. 
Applying the estimate $G_2 - \hat G_2$ in Lemma \eqref{lem:ker_sym} with $f= g_{k_2 l_2}$, we get 
\[
\bal
I_{10kl}(F_1)  & \leq  \f{2}{3} x_2( x_2^2 + 3 x_1^2) || \pa^d G(y + \cdot )||_{L^{\inf}(Q_x)}
 = \f{2}{3} x_2( x_2^2 + 3 x_1^2)  M_{G, d },  \\ 
I_{01kl}(F_1)  & \leq  \f{2}{3} x_1( x_1^2 + 3 x_2^2) M_{G, d}, \quad 
I_{00kl}(F_1)  \leq \f{2 x_1 x_2|x|^2 }{3} M_{G, d},  \quad  d =  k_2 + l_2 + 4 = k+l+4.
\eal
\]

If $(i, j) = (1,1)$, we apply the estimate of $G_1 $ in Lemma \eqref{lem:ker_sym} with $f =\pa_{x_1 x_2} g_{kl}(x)$ ( $k, l$ are number of derivatives on $G(y + z)$) to get 
\[
|I_{11 kl} (F_1)| \leq 2 |x|^2 M_{G, d}, \quad d = k_2 + l_2 + 4 = k + l + 4.
\]

If $\max(i, j) \geq 2, i+ j \leq 3$,  $x_1 x_2 \pa_{12} G(0)$ vanishes in $I_{ijkl}$. We apply derivation similar to \eqref{eq:ker_sym_G23} without $4 x_1 x_2 \pa_{12} g_{k_2 l_2}(0)$ and the estimate of $G_2$ in Lemma \ref{lem:ker_sym} with $f = g_{k_2 l_2}$ to get 
\[
| I_{ijkl}(F_1)  
\leq 4 x_1^{1- i_2} x_2^{1- j_2} || \pa^2 g_{k_2 l_2}||_{L^{\inf}(Q_x)}
\leq  4 x_1^{1- i_2} x_2^{1- j_2} M_{G, d}, \quad d = k_2 + l_2 + 2 = k+ l + 4.
\]

To bound $M_{G, k}$, we apply Lemma \ref{lem:K_mix} to get 
\beq\label{eq:Den}
  M_{G, k} = \max_{ a + b = k}  || (\pa_{y_1}^a \pa_{y_2}^b G)(y + \cdot ) ||_{L^{\inf}(Q_x)}
  \leq \f{ (k-1)!}{2 \cdot \mathrm{Den}(x, y)^{k/2}  } , 
  \quad   \mathrm{Den}(x, y) = \min_{ z \in Q_x} |y-z|^2.
\eeq
It is not difficult to obtain that for $x, y \in \R_2^{++}$, we have
\beq\label{int:Den}
\mathrm{Den}(x, y) = \sum_{i=1,2} \min_{ |z_i| \leq x_i} |y_i- z_i|^2 = \sum_{i=1,2} (\max( y_i - x_i, 0))^2.
\eeq

Using the above estimates, for $|y| \gg |x|$, we get $ \mathrm{Den} \sim |y|^2$ and the decay estimate for $I_{ijkl}(F_1)$ \eqref{eq:ker_decay3} with a rate $|y|^{ -k - l -4 }$ and $I_{ijkl}(F_2)$ with a rate $|y|^{-k-l-6}$.

\subsection{Piecewise $L^{\inf}$ estimate of derivatives of the Green function }\label{app:green}
 
In this section, we develop sharp $L^{\inf}$ estimates of the derivatives of the Green function $G(x) =-\f{1}{2\pi} \log |x|$ and their linear combinations in a small domain $[a, b] \times [c, d]$. They will be used in Lemmas \ref{lem:T_rule}, \ref{lem:int_err_x} to estimate the error, especially near the singularity of the kernel. We remark that the linear combinations of $\pa_1^i \pa_2^j G$ can be quite complicated. If we simply use the triangle inequality to estimate it, 
we can overestimate some terms with cancellation significantly, especially near the singularity of $G$. These sharp estimates are useful for reducing the estimate of the error term in Lemmas \ref{lem:T_rule}, \ref{lem:int_err_x} without choosing very small mesh, which can lead to large computational cost.

\subsubsection{Coefficients of the derivatives of the Green function}\label{sec:psi_coe}

To simplify the notation, we drop $\f{1}{\pi}$ from $G$ and denote $ f_p = - \f{1}{2} \log |x|$. 
Firstly, we derive the formulas of $\pa_1^i \pa_2^j f_p $. Due to homogeneity, for $k + l \geq 1$, we assume
\beq\label{eq:psi_coe_anz}
\pa_{x_1}^{k} \pa_{x_2}^l f_p = \f{ \sum_{ i+ j  = k+l}  c_{ij} x_1^i x_2^j  }{ |x|^{ 2 (k+l) }}.
\eeq

Next, we derive the recursive formula for $c_{ij}$. Using induction, we can obtain 
\[
\bal
\pa_{x_1}^{k+1} \pa_{x_2}^l f_p & = \f{ \sum_{i+j = k+ l} c_{ij} i x_1^{i-1} x_2^j }{ |x|^{2 (k+l)} } 
- \f{2 (k+l) x_1}{ |x|^{2 (k+l + 1)} } \sum_{i + j = k+ l} c_{ij} x_1^i x_2^j \\
& = \f{1}{ |x|^{ 2(k+l + 1) }}
(  \sum_{i+j = k+ l}  c_{ij} i x_1^{i+1} x_2^j + c_{ij} i x_1^{i-1} x_2^{ j+2} 
- 2 ( k+ l) c_{ij} x_1^{i+1} x_2^j ) \\
& =\f{1}{ |x|^{ 2(k+l + 1) }}
(  \sum_{i+j = k+ l} ( c_{ij} i + c_{i+2, j-2} (i+2) - 2 (k+l) c_{ij}) x_1^{ i+1} x_2^j  ). 
\eal
\]

Therefore, we obtain the recursive formula 
\[
c_{i+1, j} = i c_{ij}  + (i+2)  c_{i+2, j-2} - 2(k+l) c_{ij},
\]
for all $i + j = k+ l$, or equivalently,
\[
c_{i, j} = (i-1) c_{i-1, j}  - 2(k+l) c_{i-1, j} + (i+ 1)  c_{i + 1, j-2} ,
\]
for all $i + j = k+ l +1$. Similarly, for $\pa_{x_2}$, we yield 
\[
c_{i, j} = (j-1) c_{i, j-1} - 2 (k+l) c_{i, j - 1} + ( j+1) c_{ i-2, j+1},
\]
for all $ i + j = k + l + 1$.

\subsubsection{Estimates of rational functions}

We use the above formulas to develop sharp estimates of the derivatives of $f_p$ and their linear combinations in a small grid cell $[y_{1l}, y_{1u}] \times [y_{2l}, y_{2u}]$. For $k < k_2$ and $S \subset \{ (i,j) : i+ j = k \}$, we estimate 
\beq\label{eq:psi_coe_est1}
I_S \teq \f{ \sum_{ (i,j) \in S} c_{ij} y_1^i y_2^j  }{ |y|^{k_2} } . 
\eeq
We assume that $I_S(x)$ is either odd in $x_i$ or even in $x_i$ for $i=1,2$. Clearly, this properties hold for $\pa_{x_1}^k \pa_{x_2}^l f_p $ \eqref{eq:psi_coe_anz}. Denote $i_1 = \min_{i \in S} i, j_1 = \min_{ j \in S} j$. We yield 
\[
I_S = \f{y_1^{i_1} y_2^{j_1}}{ |y|^{ { i_1 + j_1}} }  
\f{ \sum_{ (i,j) \in S} c_{ij} y_1^{ i- i_1} y_2^{ j - j_1}  }{ |y|^{k_2 - i_1 - j_1} }  .
\]

We further introduce 
\[
P  \teq \sum_{ (i,j) \in S} c^+_{ij} y_1^{ i- i_1} y_2^{ j - j_1} , \quad 
Q \teq \sum_{ (i,j) \in S} c^-_{ij} y_1^{ i- i_1} y_2^{ j - j_1}.
\]

We claim that $i-i_1, j - j_1$ are even for all $(i, j) \in S$. Since $I_S$ is either odd or even in $x_i, i=1,2$, the numerator $\sum c_{ij} x_1^i x_2^j$ in \eqref{eq:psi_coe_est1} have the same symmetries in $x_1, x_2$. 
In particular, each monomial $c_{ij} x_1^i x_2^j$ in \eqref{eq:psi_coe_est1} also enjoys the same symmetries in $x_1, x_2$ as $I_S$. 
If $i- i_1$ is odd for some $i$, then $c_{ij} x_1^{i-i_1} x_2^{j-j_1}$ must be odd in $x_1$. It implies $i-i_1 \geq 1$ for any $(i,j) \in S$ and contradicts the minimality of $i_1$. The same argument applies to $j_1$.

As a result, $P$ and $Q$ are monotone increasing in $|y_1|, |y_2| \geq 0$. For $|y_i|_l \leq |y_i| \leq |y_i|_u, i=1,2$, we can derive the upper and lower bounds for $P, Q$ and yield
\[
\bal
|I| & \leq  \f{ \max( P_u - Q_l, Q_u - P_l)}{ | y |_l^{k_2 - i_1 - j_1}}
\max_{y \in \Om} 
\f{ |y_1|^{i_1} |y_2|^{j_1}}{ |y|^{ { i_1 + j_1}} }   \\
& \leq \f{ \max( P_u - Q_l, Q_u - P_l)}{ | y |_l^{k_2 - i_1 - j_1}}
( \f{ |y_1|_u }{ ( |y_1|_u^2 + |y_2|_l^2 )^{1/2}} )^{i_1}  
( \f{ |y_{2}|_u }{ ( |y_{1}|_l^2 + |y_{2}|_u^2 )^{1/2} } )^{j_1} ,
\eal
\]
where $|y|_l$ is the lower bound of $|y|$ and we have used the fact that $z_i / |z|$ is increasing in $z_i$ for $z_i \geq 0$ to obtain its upper bound. Now, for $y_i \in [y_{il}, y_{iu}]$, we estimate $|y_i|_l, |y_i|_u$ as follows 
\beq\label{eq:psi_coe_est2}
\bal
& |y_i| \geq \max(0, |y_{il} + y_{iu} | /2 - (y_{iu} - y_{il}) /2  )  \teq |y_i|_l,  \\
 &  |y_i| \leq \max( |y_{il}|, |y_{iu}|)  \teq |y_i|_u, 
\  |y|_l \teq ( |y_1|_l^2 + |y_2|_l^2)^{1/2}. 
\eal
\eeq
Note that for $y_i \in [y_{il}, y_{iu}]$, $y_i$ can change sign.

\subsection{Improved estimate of the higher order derivatives of the integrands}\label{app:integ_est}

In the H\"older estimate, we need to estimate the derivatives of the integrands \eqref{int:sym_integ1}, \eqref{int:sym_integ2}, \eqref{int:decom_ux}, which take the form
\[
 K^C(x, y) (p(x) - p(y)) + K^{NC} p(x),
\]
for some weight $p$ and kernels $K^C, K^{NC}$. Using the estimates of the kernels in Appendix \ref{app:ker_sym}, \ref{app:green} and the weights in Section \ref{app:wg}, the Leibniz rule \eqref{eq:lei}, and the triangle inequality, we can estimate the derivative of the integrands.  However, such an estimate can lead to significant overestimates near the singularity of the integrand. We use the estimates in Appendix \ref{app:green} to handle the cancellations among different terms and obtain improved estimates for the integrand and its derivatives near the singularity: 
\beq\label{int:integ_T00}
 T_{00}(x, y) \teq K(y-x) (p(x) - p(y)), \quad   \pa_{x_i} T_{00}(x, y).
 \eeq
We choose weight $p(x)$ that is even in $x$ and $y$. The basic idea is to perform a Taylor expansion on $p(x) - p(y)$ and obtain the factor $|x-y|$, which cancels one order of singularity from $K(x, y)$. We use the formulas in Appendix \ref{app:green} to collect the terms with the same singularity and exploit the cancellation.

\subsubsection{Y-discretization}\label{sec:errC1_Y}

In the Y-discretization of the integral, we need to estimate the $y-$derivatives of the integrand \eqref{int:integ_T00}. For $a , b = 1, 2 $, denote 
\beq\label{eq:integ_D12}
D_1 = \pa_{a}, \quad D_2 = \pa_{b}, \quad x_m = \f{x+y}{2}.
\eeq
Next, we compute $\pa_{y_b}^j \pa_{x_a}^i T_{00}$. The reader should be careful about the sign. Note that 
\[
\pa_{x_{a}} ( K( y- x)  )= -( \pa_{a} K )( y- x)
= - D_1 K(y-x).
\]

Using the Leibniz rule, we get 
\[
\bal
\pa_{y_b}^2 \pa_{x_a}  T_{00} &= \pa_{y_b}^2( - D_1 K ( p(x )  - p(y)) + K \cdot D_1 p (x))
=   \pa_{y_b}^2 ( D_1 K  \cdot ( p(y )  - p(x)) + K \cdot  D_1 p (x))  \\
&= D_2^2 D_1 K \cdot ( p(y) - p(x)) + 2 D_2 D_1 K \cdot D_2p(y) 
+ D_1 K \cdot D_2^2 p(y) + D_2^2 K \cdot D_1 p(x) .\\
\eal
\]

We use Taylor expansion at $x = x_m$ and write
\beq\label{eq:integ_Ta1}
\bal
p(y)- p(x) =  (y-x) \cdot \na p( x_m) + p_{m,2,err}, 
\; \pa_i p(z) = \pa_i p( x_m) + ( \pa_i p(z) - \pa_i p(x_m) ), \ z = x, y , \\
| f(z) - f(x_m)- (z - x_m) \cdot \na f(x_m) | \leq  \f{1}{2} \f{d_1^2}{4}  || f_{xx} ||_{L^{\inf}(Q)} + \f{d_1 d_2}{4} || f_{xy}||_{L^{\inf}(Q)}
+ \f{1}{2} \f{d_1^2}{4}  || f_{xx} ||_{L^{\inf}(Q)}  \teq I_f.
\eal
\eeq
 for $d = y- x,  z = x , y$ and any $f$, where $Q$ is the rectangle covering $x, y$. Then $p_{m,2,err}$ is bounded by $2 I_p = O(|x-y|^2)$. Combining the terms involving $\na p$, we get
\beq\label{eq:errC1_Y1_T00}
\bal
\pa_{y_b}^2  \pa_{x_a}  T_{00} &  =  \sum_{i=1,2} \B( D_2^2 D_1 K \cdot (y_i - x_i) + \one_{D_2 = \pa_{i}} 2 D_2 D_1 K 
+ \one_{D_1 = \pa_{i}}  D_2^2 K \B) \cdot \pa_{x_i} p(x_m)  + D_2^2 D_1 K \cdot p_{m,2,err} \\
&\quad + 2 D_2 D_1 K \cdot ( D_2 p(y) - D_2p(x_m))
+ D_2^2 K \cdot ( D_1p(x) - D_1p(x_m))  
+ D_1 K \cdot D_2^2 p(y)  \\
& \teq  \B( \sum_{i=1,2} I_i  \cdot \pa_{x_i} p(x_m) \B)+ II_1 + II_2 + II_3 + II_4,  \\
 I_i & \teq D_2^2 D_1 K \cdot (y_i - x_i) + \one_{D_2 = \pa_{i}} 2 D_2 D_1 K 
+ \one_{D_1 = \pa_{i}}  D_2^2 K,
\eal
\eeq
where $\pa_1^i \pa_2^j K$ is evaluated at $ y - x$, and $II_i$ denotes the last four terms in the second equation. The first term is the most singular term. We combine the most singular terms to exploit the cancellation and improve the estimates. We estimate the kernels 
\beq\label{eq:K_mix}
\bal
&K_{mix}( D_1, D_2, i, s)( z_1, z_2) \teq  D_2^2 D_1 K(z)  z_i + \one_{D_2 = \pa_{i}} 2 D_2 D_1 K (z)
+ s \one_{D_1 = \pa_{i}}  D_2^2 K (z) ,
\eal
\eeq
with $s= \pm 1$ and $D_1, D_2 \in \{ \pa_1, \pa_2 \}$.
 Then we can bound $\pa_{y_b}^2 \pa_{x_a}  T_{00}$ using the triangle inequality. When $D_1 = D_2$, we have an improved estimate for $II_2, II_3$ 
 \beq\label{eq:errC1_Y1_II23}
 II_2 + II_3 
 = D_2^2 K ( D_2p(y) - D_2p( x_m) + (D_2 p(y) + D_2 p(x) - 2 D_2 p(x_m)) ). 
 \eeq
 We estimate $ D_2 p(y) + D_2 p(x) - 2 D_2 p(x_m) $ using \eqref{eq:integ_Ta1} with $f = D_2 p$ and $z = x, y$.

\subsubsection{The second singular term}\label{sec:errC1_Y_Ks2}

For $x = (x_1, x_2)$ close to the $y-$axis or the $x$-axis, since we have symmetrized the integral (see \eqref{int:sym_integ1} and Section \ref{sec:int_sym}), we have another singular term in the integrand
\[
T_{01} \teq K( y_1 - x_1, y_2 + x_2) ( p(x) - p(y) )  , \ \mathrm{ or }  \ 
T_{10} \teq K(y_1 + x_1, y_2 - x_2) ( p(x) - p(y) ).
\]

We have the first term if $x_2 < x_1$ and $x_2$ close to $0$, and the second term if $x_1 < x_2$ and $x_1$ close to $0$. We label the former case with $side = 1$ and the latter $ side = 2$. 
See the right figure in Figure \ref{fig:sing_Rk} for an illustration of the first case. The $T_{01}$ term is supported in the blue region $R(x, k, S)$.
Denote 
\beq\label{eq:errC1_Y_side}
(s_1, s_2) = (1, -1) \mathrm{ \ if \  } side = 1,  \quad (s_1, s_2) = (-1, 1) \mathrm{\ if\ } side = 2.
\eeq

\paragraph{ \bf{ Case I} }
If $(D_1, side) = (\pa_1, 1)$ or $( \pa_2, 2)$, we obtain 
\[
\pa_{x_a} K( y_1 - s_1 x_1, y_2 - s_2 x_2 ) = -\pa_{y_a}  K( y_1 - s_1 x_1, y_2 - s_2 x_2 ),
\]
for $( a ,s_1, s_2) = (1, 1, -1)$ or $(2, -1, 1)$. 
The computations for $ \pa_{y_b}^2 \pa_{x_1} T_{01}, \pa_{y_b}^2 \pa_{x_2} T_{10}$ are the same as \eqref{eq:errC1_Y1_T00} with $K$ and its derivatives evaluating at $z = (y_1 - s_1 x_1, y_2 - s_2 x_2)$.

We estimate $II_i$ in \eqref{eq:errC1_Y1_T00}  directly using 
the triangle inequality and the bounds for $\pa_1^i \pa_2^j K $ in Section \ref{app:ker_sym}, \ref{app:green}  and $p$ in Section \ref{app:wg}. For $I_i$ in \eqref{eq:errC1_Y1_T00} in the most singular term,
if $i = side$, from definition \eqref{eq:errC1_Y_side}, we get 
\[
s_i = 1, \quad s_{3-i} = -1, \quad   z_i = y_i - s_i x_i = y_i - x_i ,\quad z_{3-i} = y_{3-i} + x_{3-i}. 
\]
Therefore, it follows
\[
I_i = D_2^2 D_1 K(z) \cdot (y_i - x_i) + \one_{D_2 = \pa_{i}} 2 D_2 D_1 K (z)
+ \one_{D_1 = \pa_{i}}  D_2^2 K (z)
= K_{mix}( D_1, D_2, i, 1)(z),
\]
where $K_{mix}$ is defined in \eqref{eq:K_mix}.  If $ i \neq side$, we have $z_i = y_i + x_i \geq |y_i - x_i|, z_{3-i}= y_{3-i}- x_{3-i}$. We simply bound the summand using the triangle inequality 
\[
|I_i| \leq  | D_2^2 D_1 K(z) | \cdot | y_i - x_i| + \one_{D_2 = \pa_{i}} 2 | D_2 D_1 K (z) |
+ \one_{D_1 = \pa_{i}}  | D_2^2 K (z) |.
\]

\paragraph{\bf{Case II} }
If $(D_1, side) = (\pa_1, 2)$ or $( \pa_2, 1)$, we obtain 
\[
\pa_{x_a} K( y_1 - s_1 x_1, y_2 - s_2 x_2 )
=  (\pa_{y_a}  K) ( y_1 - s_1 x_1, y_2 - s_2 x_2 ),
\]
for $(a, s_1, s_2) = (1, -1, 1)$ or $(2, 1, -1)$. Recall the definitions of $D_1, D_2$ \eqref{eq:integ_D12}. Using the above identity, we yield 
\[
\pa_{y_b}^2 \pa_{x_a} T =  \pa_{y_b}^2  ( D_1 K \cdot (p(x) - p(y)) + K \cdot D_1 p ) 
= - ( \pa_{y_b}^2  ( D_1 K \cdot (p(y) - p(x)) - K \cdot D_1 p ) ),
\]
for $ T = T_{01}$ or $T_{1 0}$. Using an expansion similar to that in \eqref{eq:errC1_Y1_T00}, \eqref{eq:integ_Ta1}, we get 
\beq\label{eq:errC1_Y_refl}
\bal
- \pa_{y_b}^2 \pa_{x_a}  T &  =  \sum_{i=1,2} ( D_2^2 D_1 K \cdot (y_i - x_i) + \one_{D_2 = \pa_{i}} 2 D_2 D_1 K 
- \one_{D_1 = \pa_{i}}  D_2^2 K ) \cdot \pa_{x_i} p(x_m)  + D_2^2 D_1 K \cdot p_{m,2,err} \\
&\quad + 2 D_2 D_1 K \cdot ( D_2 p(y) - D_2p(x_m))
- D_2^2 K \cdot ( D_1p(x) - D_1p(x_m))  
+ D_1 K \cdot D_2^2 p(y)  \\
 & \teq  \B( \sum_{i=1,2} I_i  \cdot \pa_{x_i} p(x_m) \B)+ II_1 + II_2 + II_3 + II_4 , \\
 I_i & \teq D_2^2 D_1 K \cdot (y_i - x_i) + \one_{D_2 = \pa_{i}} 2 D_2 D_1 K 
- \one_{D_1 = \pa_{i}}  D_2^2 K , 
\eal
\eeq
where $\pa_1^i \pa_2^j K$ is evaluated at $z = (y_1 - s_1 x_1, y_2 - s_2 x_2)$. We bound $II_i$ using triangle inequality, the estimate \eqref{eq:errC1_Y1_II23}, and the bounds for  $K$, its derivatives, and $p$ in Sections \ref{app:ker_sym}, \ref{app:green}, and \ref{app:wg}. 

For $I_i$, if $i = side$, from \eqref{eq:errC1_Y_side}, we get $s_i = 1$ and $z_i = y_i - s_i x_i = y_i - x_i$. Hence, we get 
\[
I_i = D_2^2 D_1 K \cdot (y_i - x_i) + \one_{D_2 = \pa_{i}} 2 D_2 D_1 K 
- \one_{D_1 = \pa_{i}}  D_2^2 K = K_{mix}(D_1, D_2,i, -1)(z),
\]
where $K_{mix}$ is defined in \eqref{eq:K_mix}.

If $i \neq side$ and $ D_1 = D_2 = \pa_i$, we have $z_i = y_i - s_i x_i = y_i + x_i$ and get a cancellation between $D_2 D_1 K$ and $D_2^2 K$ and yield 
\[
|I_i| = | D_2^2 D_1 K \cdot (y_i - x_i) +  D_2 D_1 K  |
\leq | D_2^2 D_1 K | \cdot | y_i - x_i |  +   | D_2 D_1 K |.
\] 
Otherwise, we simply bound each term in $I_i$ using the triangle inequality.

\subsubsection{X-discretization}

For $K(s) =  \f{ s_1 s_2 }{|s|^4},\f{1}{2} \f{ s_1^2 - s_2^2 }{ |s|^4}$, we have $K(s) = K(-s)$. Denote \[
T =K(y-x) (p(x) - p(y)) = K(x- y) (p(x) - p(y)).
\]

In this section, we compute $\pa_{x_b}^i \pa_{x_a}^j T $. Using the Taylor expansion at $x$
\[
p(x) - p(y) = (x- y) \cdot \na p(x) + p_{x, 2,err},
\]
and calculations similar to those in Section \ref{sec:errC1_Y}, we get 
\beq\label{eq:errC1_X_T00}
\bal
\pa_{x_b}^2 \pa_{x_a} T &=  \pa_{x_b}^2 ( D_1 K \cdot (p(x) - p(y)) + K D_1 p(x) ) =  D_2^2 D_1 K \cdot (p(x) - p(y)) + 2 D_1 D_2 K \cdot D_2 p(x) \\
 & \quad + D_1 K \cdot D_2^2 p(x) + D_2^2 K \cdot D_1p(x) + 2 D_2 K \cdot D_1 D_2 p(x) + K \cdot D_1 D_2^2 p(x) \\
 & = \sum_{i=1,2} ( D_2^2 D_1 K \cdot (x_i - y_i) + \one_{D_2 = \pa_{i} } 2 D_1 D_2 K
+ \one_{ D_1 = \pa_{i}}   D_2^2 K ) \pa_i p(x) +  D_2^2 D_1 K  \cdot p_{x, 2, err}  \\
& \quad + D_1 K \cdot D_2^2 p(x) + 2 D_2 K \cdot D_1 D_2 p(x) + K \cdot D_1 D_2^2 p(x) 
\teq \B(  \sum_{i=1,2} I_i \cdot \pa_i p(x) \B) + II ,\\
I_i & \teq  D_2^2 D_1 K \cdot (x_i - y_i) + \one_{D_2 = \pa_{i} } 2 D_1 D_2 K
+ \one_{ D_1 = \pa_{i}}   D_2^2 K ,
\eal
\eeq
where $II$ consists of the last four terms in the third equation, $K$ and its derivatives are evaluated at $x - y$. Since $D_1, D_2 = \pa_{x_i}$, we get
\[
 I_i = D_2^2 D_1 K \cdot (x_i - y_i) + \one_{D_2 = \pa_{i} } 2 D_1 D_2 K
+ \one_{ D_1 = \pa_{i}}   D_2^2 K  = K_{mix}( D_1, D_2, i, 1)(x- y),
\]
where $K_{mix}$ is defined in \eqref{eq:K_mix}.
We use the bound for $K_{mix}$, $\pa_1^i \pa_2^j K$ and $p$ to estimate $D_2^2 D_1 T$.

\subsubsection{The second singular term}
Similar to Section \ref{sec:errC1_Y_Ks2}, we have the second singular term for $x$ close to the $x$-axis or $y$-axis
\[
T_{01} \teq K(x_1 - y_1, x_2 + y_2) (p(x) - p(y)), \quad 
T_{10} \teq K(x_1 + y_1, x_2 - y_2) (p(x) - p(y)).
\]
We have the former if $x_2 < x_1$ and $x_2$ close to $0$, 
and the latter if $x_1 < x_2$ and $x_1$ close to $0$. Using the definition of $side, s_1, s_2$ from Section \ref{sec:errC1_Y_Ks2} and  \eqref{eq:errC1_Y_side}, we get 
\[
\pa_{x_a}  K( x_1 - y_1 s_1, x_2 - y_2 s_2) = (D_1 K)( x_1 - y_1 s_1, x_2 - y_2 s_2).
\]
Then the computations of $D_2^2 D_1 T$ are the same as those in \eqref{eq:errC1_X_T00} with $\pa_1^i \pa_2^j K$ evaluated at $z = (x_1 - s_1 y_1, x_2 - s_2 y_2) $. We bound $II$ in \eqref{eq:errC1_X_T00} directly using the triangle inequality and the bounds for $\pa_1^i \pa_2^j K$ and $p$. For $I_i$ in \eqref{eq:errC1_X_T00}, if $i = side$, 
from \eqref{eq:errC1_Y_side}, we get $s_i$ and $z_i = x_i - s_i y_i = x_i - y_i$. It follows 
\[
I_i = D_2^2 D_1 K \cdot z_i + \one_{D_2 = \pa_{i} } 2 D_1 D_2 K
+ \one_{ D_1 = \pa_{i}}   D_2^2 K 
= K_{mix}(D_1, D_2, i, 1)( z).
\]
If $i \neq side$, we have $z_i = x_i + y_i > |x_i - y_i|$. We bound each term in $I_i$ separately by following the previous argument.

\subsection{Estimate of $u(x)$ for small $x_1$}\label{sec:u_dx1}

In the energy estimate, we need to estimate $ (u(x) - \hat u(x)) \vp(x)$ with weight $\vp$ singular along the line $x_1 = 0$, e.g.$\vp_1$ \eqref{wg:linf}, where $\hat u(x)$ is a finite rank approximation of $u(x)$. We use the property that $u$ vanishes on $x_1 = 0$ to establish such an estimate. 

By definition and symmetrizing the kernel using the odd symmetry of $\om$, we have
\[
u(x, y) = \f{1}{2\pi} \int_{ y_1 \geq 0}  \B( \f{ x_2 - y_2}{ |x - y|^2} - \f{ x_2 - y_2}{ (x_1 + y_1)^2 + (x_2-y_2)^2 } \B) \om(y) dy  = \f{1}{\pi} \int_{ y_1 \geq 0}   K( x, y) W(y) dy ,
\]
where 
\beq\label{eq:u_dx_ker}
\bal
K  &= \f{1}{2}   ( \f{ x_2 - y_2}{ |x - y|^2} - \f{ x_2 - y_2}{ (x_1 + y_1)^2 + (x_2-y_2)^2 } )
= x_1 \cdot \f{ 2 (x_2 - y_2) y_1 }{ |x - y|^2 ( (x_1 + y_1)^2 + (x_2-y_2)^2 ) } \\
&  \teq x_1 K_{du}(x, y)= x_1 \td K_{du}( x_1, y_1, x_2 - y_2),
\quad \td K_{du}(x, y, z) = \f{2 y z }{ ((x-y)^2 + z^2)( (x+y)^2 + z^2)}.
\eal
\eeq
We define $K_{app}$ as the symmetrized kernel in $\R_2^{++}$ for $\hat u$ similar to that in Section \ref{sec:vel_linf}. Since $W$ is odd in $y_2$, we can symmetrize the integral in $y_2$ and obtain the full symmetrized integrand 
\[
x_1 K_{du}(x, y) - x_1 K_{du}(x_1, x_2, y_1, -y_2) 
= x_1 (\td K_{du}(x_1, y_1, x_2  -y_2) - \td K_{du}(x_1, y_1, x_2 + y_2)).
\]

Since $K$ is $-1$ homogeneous, using a rescaling argument, for $ x = \lam \hat x, y = \lam \hat y$, , we have 
\beq\label{eq:u_dx_reg0}
\bal
& u =\f{\lam}{\pi} \int_{ \hat y_1 \geq 0 } 
 \B( \one_{S^c}(\hat y) K(\hat x, \hat y )  - K_{app, \lam}( \hat x, \hat y) \B) \om_{\lam}(\hat y) 
+ \one_{S}(\hat y) K(\hat x, \hat y )  \om_{\lam}(\hat y) d \hat y 
\teq I + II, \\
 \eal
\eeq
for some rescaled kernel $K_{app, \lam}(\hat x, \hat y)$, 
where $S = R(\hat x, k)$ is the singular region \eqref{eq:rect_Rk} adapted to $\hat x$. For $I$, we further rewrite it and estimate it as follows 
\[
\bal
I  &= \f{\lam}{\pi} \hat x_1 \int_{\hat y_1 \geq 0, \hat y \notin S} 
\B( \one_{S^c}(\hat y) K_{du}(\hat x, \hat y )  - \f{1}{\hat x_1} K_{app, \lam}( \hat x, \hat y) \B)  \om_{\lam}( \hat y) d \hat y . \\
\eal
\]
Since the integrand is not singular, we further symmetrize the integrand in $y_2$ and then use the method in Section \ref{sec:T_rule} to discretize and estimate the integral to obtain its tight bound.

\vs{0.1in}
\paragraph{\bf{Derivative bounds}}
To estimate the error in the Trapezoidal rule in Lemma \ref{lem:T_rule}, we need to bound 
$\pa_{x_i}^2 K_{du}(x, y), \pa_{y_i}^2 K_{du}(x, y)$. 
Since $ \f{1}{x} C_{u0}(x,y), \f{1}{x} C_u(x, y)$ \eqref{eq:u_appr_near0_coe} are smooth, from the construction in Section {\secapprvel}, the kernel $ \f{1}{  x_1} K_{app}( x,  y )$ and its rescaled version are regular in $ \hat x$. We estimate its derivatives following Section \ref{sec:int_method}. Since $K_{du}(x, y) = \f{1}{x_1} K(x, y)$
\eqref{eq:u_dx_ker},  $K(x, y)$ is harmonic in $y$, and $|\pa_{x_2}^2 K(x, y)| = |\pa_{y_2}^2 K(x, y) |$, we get 
\[
 \pa_{y_1}^2 K_{du}(x, y) = -  \pa_{y_2}^2 K_{du}(x, y), \quad |\pa_{y_2}^2 K_{du}(x, y)| = 
 | \pa_{x_2}^2 K_{du}(x, y) |.
\]
Thus, we only need to bound $|\pa_{x_1}^2 K_{du}|$ and $ |\pa_{y_1}^2 K_{du}|$, or $\pa_x^2 \td K_{du}$ and $ \pa_y^2 \td K_{du}$ using the relation \eqref{eq:u_dx_ker}. We derive the formulas of $ \pa_x^2 \td K_{du}$ and $ \pa_y^2 \td K_{du}$ and then estimate them using methods similar to that in Appendix \ref{app:green}. We have an improved estimate for 
$\pa_y \td K_{du}$ in $ \{x \} \times [y_l, y_u] \times [z_l, z_u]$ near the singularity. A direct computation yields 
\[
\bal
  &\pa_y^2 \td K_{du}(x, y, z) =24 y z \f{ ( z^4 - (x^2 -y^2)^2) ( x^2 + y^2 + z^2) }{ T_-^3 T_+^3} + 64 \f{ x^2 y^3 z^3}{ T_-^3 T_+^3}  \\
& = \f{y z}{T_-^2 T_+^2} 
\B( 12 ( \f{1}{T_-} + \f{1}{T_+} ) (z^4 - (x-y)^2 (x+y)^2) + 64 x^2 \f{y^2 z^2}{ T_- T_+}  \B), 
  \quad T_{\pm} = (x \pm y)^2 + z^2 .
  \eal
\]
where we have used $ \f{1}{T_-} + \f{1}{T_+} = 2 \f{ x^2 + y^2 + z^2}{ T_- T_+}$. We apply the estimate of $K_{du}$ to $x, y \geq 0$. Since $| \pa_y^2 \td K_{du} |$ is even in $z$, without loss of generality, we consider $z \geq 0$. Then for $P_2$, we have $z / T_-^{1/2} , y / T_+^{1/2}$ are increasing in $z, y$, respectively. To bound other terms, we simply use the monotonicity of 
the polynomials, \eqref{eq:psi_coe_est1}, interval operation \eqref{eq:func_intval}, \eqref{eq:func_intval2}, and follow Section \ref{sec:psi_coe}. For example, we use \eqref{eq:psi_coe_est2} to bound $(x-y)^2, (x+y)^2$ and 
\[
\bal
  0 \leq  \f{y}{ T_+^{1/2}} \leq \f{ y_u}{  ( ( x+ y_u)^2 + z_l^2 )^{1/2}}, \quad
 0 \leq \f{z}{ T_-^{1/2}} \leq \f{z_u}{ ( |x-y|_l^2 + z_u^2 )^{1/2} }. 
 \eal
\]

\vs{0.1in}
\paragraph{\bf{$\hat x_1$ not small}}
For $II$ in \eqref{eq:u_dx_reg0}, if $\hat x_1 \geq x_l = 2h> 0$ away from $0$, we have $|K_{du}(\hat x, \hat y ) |\les \f{1}{ x_l} \f{1}{ | \hat x-y|}$,
which is integrable near the singularity $ \hat x$. 
We estimate $II$ using
\[
|II| \leq \f{\lam}{\pi} \hat x_1 \int_{ \hat y_1 \geq 0, \hat y \in S} | K_{du}(\hat x, \hat y)| \vp_{\lam}^{-1}(\hat y) d \hat y || \om \vp||_{\inf}, \quad S = R(\hat x, k). 
\]
We follow Section \ref{sec:int_nsym1} by introducing $\hat y = \hat x + s, s \in S - \hat x$, decomposing $S - \hat x$ into the symmetric part $D_{sym}$ and non-symmetric part $D_{ns}$ and estimating the piecewise integral of $K_{du}(\hat x, \hat y)$ 
\[
\bal
& D_{sym} = R_s(\hat x, k) - \hat x, \ D_{ns} = ( R(\hat x, k) \bsh R_s(\hat x, k) )- \hat x,  \\
& | K_{du}( \hat x, \hat y) | \one_{\hat y_1 \geq 0}=  |F| \one_{\hat x_1 + s_1 \geq 0} , \ 
 F =  \f{  (\hat x_1 + s_1) s_2  }{ |s|^2 ( (s_1 + 2 \hat x_1)^2 + s_2^2 )},
\eal
\]
and piecewise bounds of $\vp_{\lam}^{-1}(y)$, where we have used \eqref{eq:u_dx_ker} to obtain the above formula. We observe that $|F|$ is even in $s_2$ and $F \geq 0$
for $s \in Q = [a, b] \times [c, d]$ with $ c, d \geq 0$. 
We estimate the piecewise integrals of $F$ in $Q$ in Section {\secudxsupp} in the supplementary material II \cite{ChenHou2023bSupp}. Denote $X_1^+ \teq \{ y : y_1 \geq 0 \}$. If $ \hat x_1 \geq kh$, we get $S \cap X_1^+ = R(\hat x, k) $ and the regions $D_{sym}, D_{ns}$ are the same as those in Section \ref{sec:int_nsym1}. If $ \hat x_1 \in [ih, (i+1) h), i < k$, the region $S$ touches $ \{ y : y_1 = 0 \}$ and we get
\[
S \cap X_1^+= [0, (i+k+1) h] \times [  (j-k) h, (j+1 + k)h], \ \mathrm{ \ for \ } x_2 \in [ j h, (j+1) h]
\]
In this case, the symmetric and non-symmetric region becomes smaller. We do not have the left edge in the middle figure in Figure \ref{fig:sing_nsym1}, part of the upper and the lower edge due to the restriction $\hat y_1 = s_1 + \hat x_1 \geq 0$. The estimate of the integrals for  $s \in S \cap X_1^+ - \hat x_1$ follows similar argument.

\vs{0.1in}
\paragraph{\bf{Small} $\hat x_1$}
The difficulty is to estimate $II$ for small $\hat x_1 \leq 2h$. It is not difficult to obtain that 
\beq\label{eq:u_dx_log}
| II| \les  \f{\lam}{\pi} || \om_{\lam} ||_{L^{\inf}(S)}  \hat x_1 |\log( \hat x_1 )|.
\eeq

Thus we cannot bound $II$ by $ C \hat x_1$ for some constant $C$ uniformly for small $\hat x_1$. Denote by 
\beq\label{eq:u_dx_change}
\bal
S_{sym} & = [0, \hat x_1 + k h] \times [ \hat x_2 - kh, \hat x_2 + k h],  \quad
S_{in, 1} = [0, \hat x_1] \times [ \hat x_2 - kh, \hat x_2 + k h],  \\
S_{in, 2} &= [ \hat x_1, \hat x_1 + h] \times [ \hat x_2 - h, \hat x_2 + h],  \quad 
S_{in } = S_{in, 1} \cup S_{in, 2} \\
S_{out} & = [\hat x_1, \hat x_1 + h k] \times [\hat x_2 - kh , \hat x_2 + kh] \bsh S_{in, 2},
\quad \hat y = \hat x + \hat x_1 s .
\eal
\eeq
See the right figure in Figure \ref{fig:sing_nsym1} for an illustration of different regions.
By definition, we have $S_{sym} = S_{out} \cup S_{in, 1} \cup S_{in, 2}$. Here $S_{in}$ captures the most singular region. Then $\hat y \in S_{in}$ is equivalent to 
\beq\label{eq:u_dx_change2}
\bal
& s \in \hat x_1^{-1} ( S_{in} - \hat x) 
= x_1^{-1} ( [ - \hat x_1, 0] \times [- kh, kh] \cup [0, h] \times [-h, h] ) 
\teq R_1(B_1) \cup R_2(B_2), \\
& R_1(B_1) = [- 1, 0] \times [ -\f{1}{B_1}, \f{1}{B_1}],
\  R_2(B_2) = [0, \f{1}{B_2}] \times [-\f{1}{B_2}, \f{1}{B_2} ], \  
B_1 = \f{ \hat x_1}{kh} , \  B_2 = \f{\hat x_1}{h} .
\eal
\eeq

We further decompose $II$ as follows 
\[
II = \f{\lam}{\pi} \hat x_1 \int_{y_1 \geq 0}
(\one_{ S \bsh S_{sym} } (\hat y) 
+ \one_{S_{out}}( \hat y) + \one_{S_{in,1}}(\hat y)
+ \one_{S_{in, 2}}( \hat y) ) K_{du}(\hat x, \hat y) \om_{\lam}(\hat y ) d \hat y
= \f{\lam \hat x_1}{\pi} ( II_1 + II_2 + II_{in, 1 } + II_{in, 2}).
\]

The integrals $II_1, II_2$ capture the non-symmetric part and the symmetric part away from the singularity. We apply $L^{\infty}$ estimate and the method in Sections \ref{sec:int_nsym1}, \ref{sec:int_nsym2}. For $II_{in, i}$, using a change of variables \eqref{eq:u_dx_change}, \eqref{eq:u_dx_change2}, we derive
\[
II_{in, i} =  \int_{ s \in R_i(B_i)}
K_{du}( \hat x, \hat x + \hat x_1 s) \hat x_1^2 
\om_{\lam}( \hat x + \hat x_1 s ) d s.
\] 

Note that $\hat y - \hat x = \hat x_1 s, \  \hat y_1 + \hat x_1 = \hat x_1 (2 + s_1), \ \hat y_2 - \hat x_2 = \hat x_1 s_2$.
By definition \eqref{eq:u_dx_ker}, we get 
\[
\bal
 K_{du}( \hat x, \hat x + \hat x_1 s) \hat x_1^2 
& = - \f{ 2 \hat x_1 s_2 \cdot ( \hat x_1 + \hat x_1 s_1) }{ \hat x_1^2 |s|^2 \cdot
\hat x_1^2( (s_1 + 2 )^2 + s_2^2 ) }   \hat x_1^2 
= - \f{ 2 (s_1 + 1) s_2}{ |s|^2 ( (s_1 + 2)^2 + s_2^2) } \teq - K_s(s), \\
 II_{in, i}  & = - \int_{R_i(B_i)} K_s(s) \om_{\lam}( \hat x + \hat x_1 s) ds .
\eal
\]
Since $K_s(s)$ is symmetric in $s_2$, we derive 
\[
\bal
|II_{in, 1}| \leq & || \om \vp||_{\inf}  (\max_{ z \in [-\hat x_1, 0] \times [0, k h] } \vp_{\lam}^{-1}(\hat x + z) 
+ \max_{ z \in [-\hat x_1, 0] \times [ - kh, 0] } \vp_{\lam}^{-1} )
J_1(B_1)   , \\
|II_{in, 2} | \leq  & || \om \vp||_{\inf}    (\max_{ z \in [0,  h ] \times [0, h ] }  \vp_{\lam}^{-1} 
+ \max_{ z \in [0, h ] \times [ -h, 0] } \vp_{\lam}^{-1} ) J_2(B_2)  , 
\eal
\]
where $B_i$ is given in \eqref{eq:u_dx_change2} and 
\[
J_1(B_1) = \B| \int_{[-1, 0] \times [0, 1/ B_1]} K_s(s) ds \B|
= \int_{ [-1, 0] \times [0, 1/B_1]} K_s( s) ds,
\quad 
J_2(B_2) = \int_{ [0,1 /B_2]^2} K_s(s) ds.
\]
The formula of $J_i$ can be obtained using the analytic integral formula for $K_s$, and obviously $J_i$ is decreasing in $B$. Note that $J_1(B)$ is bounded, but $J_2(B) \les 1 + \log(B) \les 1 + |\log \hat x_1|$, which relates to the estimate \eqref{eq:u_dx_log}. We refer the formulas of $J_i$ to Section {\secudxsupp} in the supplementary material II \cite{ChenHou2023bSupp}.

\subsection{Additional derivations}\label{app:u_add}

\subsubsection{Estimate of the log-Lipschitz integral}\label{app:loglip}

In this section, we derive the coefficient in the estimate of $\pa_{x_2} I_{5,4}(x)$ \eqref{int:holy_loglip1}, \eqref{int:holy_loglip2}. For $I_{5,4}$, we further decompose it as follows 
\[
I_{5,4} = \B( \int_{ R(k_2) \bsh R_s(k_2)} +
\int_{ R_s(k_2) \bsh R_s(b)}
+   \int_{ R_s(b) \bsh R_s(a)}  \B) K(x-y) (\psi(x) - \psi(y) W(y) dy \teq I_{5,4, 1} + I_{5,4, 2} + I_{5,4,3}.
\]
In practice, we choose $b = 2$. The first two terms are nonsingular and their derivatives can be estimated using the method in Sections \ref{sec:int_nsym1}-\ref{sec:int_nsym2}. In the estimate of $\pa_{x_i} I_{5,4}$, we only need to estimate the boundary term on $\pa R_s(a)$ since the boundary terms on $\pa R_s(k_2), \pa R_s(b)$ are canceled in $\pa_{x_i} I_{5,j}, j=1,2,3.$ For $I_{5,4,3}$, using the second order Taylor expansion to $\psi(x) - \psi(y)$ centered at $x$, we have
\[
\bal
& \pa_{x_2} ( K(x-y) (\psi(x) - \psi(y) ) )
 = (\pa_2 K)(x- y)  ( \psi(x) - \psi(y))
+ K(x-y) \pa_2 \psi(x) \\
= & (\pa_2 K(x- y) (x_2 - y_2) + K(x-y)) \pa_2 \psi(x)
+ \pa_2 K(x- y) (x_1 - y_1) \pa_1 \psi(x) 
+ \cR_K,
\eal
\]
where the remainder $\cR_K$ coming from the higher order term in the Taylor expansion satisfies 
\[
|\cR_K| \leq \sum_{i+j=2} || \pa_x^i \pa_y^j \psi||_{L^{\inf}(Q)} |x_1 - y_1|^i |x_2 - y_2|^j c_{ij},
\]
where $Q = B_{i_1 j_1}(h_x ) + [- b h, b h]^2$ and $c_{20} = c_{02} = \f{1}{2}, c_{11} = 1$. It follows 
\[
|\pa_{x_2} I_{5,4,3} |\leq || \om \vp||_{\inf} \sum_{ 0 \leq i \leq 1, 0 \leq j \leq i+1 } Scoe_{ij}(x) \cdot f_{ij}(a, b) ,
\]
where the coefficients $Scoe_{ij}(x)$ depend on the weight $\psi, \vp$, and $f_{ij}(a, b)$ 
bounds the integral
\beq\label{eq:fKy}
 \int_{ [-b, b ]^2 \bsh [-a, a]^2} | \pa_2 K (y)  \cdot y_1^i y_2^j +
\one_{(i,j) = (0, 1)} K(y) | dy  \leq f_{ij}(a, b).
\eeq
For example, $Scoe_{01}$ comes from the following estimate for $I_{5,4, 3}$
\[
\bal
& \int_{R_s(b) \bsh R_s(a)} |(\pa_2 K(x- y) (x_2 - y_2) + K(x-y)) \pa_2 \psi(x) |  \om(y) d y \\
\leq & || \om \vp||_{\inf} || \vp^{-1}||_{L^{\inf}(Q)}
\cdot | \pa_2 \psi(x)| 
\int_{ [-b, b]^2 \bsh [-a, a]^2} | \pa_2 K (s) s_2 + K(s) | d s .
\eal
\]

The function $ f_{ij}(a, b)$ satisfies the following estimates 
for some constants $ B_{1j}>0$
\[
f_{1j}(a, b) \leq  B_{1j} \log(b/a), \quad j = 1,2 .
\]
We refer the derivations to Section {\secsinguasym} in the supplementary material II \cite{ChenHou2023bSupp}.

\subsubsection{Optimization in the H\"older estimate}\label{app:hol_opt}
Consider 
\[
\max_{t \leq t_u} \min_{a \leq b} F(a, t) ,\quad F(a, t)=  (A + B \log \f{b}{a}) \sqrt{t} + \f{ Ca}{ \sqrt{t}} ,
\]
in the upper bound in \eqref{eq:hol_comb3}. For each $t \leq t_u$, 
we first optimize $F(a, t)$ over $a \leq b$. We assume that $A, B, C, b, c, h , h_x$ are given. Denote 
\[
t_u = c h_x, \quad t_1 =  \f{ C b}{B}. 
\] 

For a fixed $t$, since $\pa_a^2 F > 0, \pa_a F(0, t) <0$ and $ \pa_a F(a, t) = 0$ if $ a = \f{ B t}{C}$, we choose $a = \min( b, \f{B t}{C})$. For $ t \leq \f{ C b}{B} = t_1$, we get 
\[
\min_{a\leq b} F(a, t) \leq F(  \f{Bt}{C}, t)
= (A + B \log \f{ b C}{B} + B) \sqrt{t} - B \sqrt{t}\log t .
\]
The right hand side can be further estimated  by studying the concave function on $s=t^{1/2} \leq s_u$ 
\[
f(p, q,s ) = ( p - q \log s) s  \leq f(p, q, \min( s_u, s_* )), 
\quad  s_* = \exp( \f{p-q}{q} )
\]
with $p = A + B \log ( \f{b C}{B}) + B, q = 2 B, s_u = \min( t_u^{1/2}, t_1^{1/2})$. 
We get the above inequality since  $f(p, q, s)$ is increasing for $s\leq s_*$ and is decreasing for $s \geq s_*$.

If $ \f{ C b}{B} \leq t \leq t_u$, we choose $a = b$ and get 
\[
\min_{a \leq b} F(a, t) 
\leq F(b, t) = A \sqrt{t} + \f{ C b}{\sqrt t},
\]
which is convex in $t^{1/2}$. Thus its maximum is achieved at the endpoints.

\section{Representations and estimates of the solutions}\label{app:solu}

In Section {\secASS} of Part I \cite{ChenHou2023a}, we represent the approximate steady state as follows 
\beq\label{eq:ASS_solu}
\bal
 \bar \om & = \bar \om_1 + \bar \om_2, \qquad \bar \th = \bar \th_1 + \bar \th_2 , \qquad \bar \om_1 = \chi(r) r^{  \bar \al_1} g_1(\b), \qquad 
\bar \th_1 = \chi(r)  r^{ 1 + 2 \bar \al_1}  g_2(\b),  \\
 \bar \phi^N & = \bar \phi_1^N + \bar \phi_2^N + \bar \phi_3^N
 + \bar \phi_{cor}^N, 
\quad
  \bar \phi^N_3 = \bar a \chi_{\phi, 2D} , \quad \chi_{\phi, 2D} = - x y \chi_{\phi}(x) \chi_{\phi}(y) , \\
   \bar \phi_{cor}^N = - c \cdot \f{x y^2}{2} & \kp_*(x) \kp_*(y) = c \phi_1 ,   \quad 
  c = \pa_x( \bar \om + \D (\bar \phi_1^N + \bar \phi_2^N + \bar \phi_3^N) ),  \quad 
  | \bar \al_1  -   \f{ \bar c_{\om}}{\bar c_l} |  \ll 1, \quad \bar  \al_1  \approx  - \f{1}{3},
\eal
\eeq
where $\bar \om_2, \bar \th_2, \bar \phi_2^N$ have compact supports and are represented as piecewise polynomials, $ \bar a \in \R$ is some coefficient, $\kp_*$ is given in \eqref{eq:cutoff_near0}, $\phi_1$ is the same as \eqref{eq:solu_cor1}, $ \chi_{\phi}$ is given in \eqref{eq:cutoff_psi_near0}. We choose a small correction $\bar \phi_{cor}$ similar to that in Section \ref{sec:lin_evo_1stcor} so that $\bar \om + \D \bar \phi^N = O(|x|^2)$ near $0$.  We use upper script $N$ to distinguish the numerical approximation $\bar \phi^N$ for the exact stream function $\bar \phi = (-\D)^{-1} \bar \om$. 
The exponent $\bar \al_1$ and angular profiles $g_i(\b)$ are obtained by fitting the far-field asymptotics of an approximate steady state with $\bar \om_1= 0, \bar \th_1 = 0$. Then we construct $(\bar \om_1, \bar \th_1)$ using the above formulas. Afterward, we refine the construction of the near-field part $(\bar \om_2, \bar \th_2)$ and exponents $(\bar c_{\om}, \bar c_l)$ by fixing $(\bar \om_1 , \bar \th_1, \bar \al_1)$. 
See more details on how to find the semi-analytic part in Section {\secASS} of Part I  \cite{ChenHou2023a}. 
We will discuss how to estimate the semi-analytic part in Section \ref{sec:est_appr_far}. In the following sections, we discuss more details about the representations and establish rigorous estimate of the derivatives of $\bar \om, \bar \th$.

Note that we do not need an approximation term $\bar \phi_3$ for the stream function in solving the linearized equation in Section \ref{sec:lin_evo} since we can allow a larger residual error in Section \ref{sec:lin_evo}.

\subsection{Representations}\label{app:solu_rep}

In a large domain $[0, L]^2$, we use piecewise polynomials to represent the solution. 
We choose a large $L$ of order $10^{15}$ and design the adaptive mesh to partition $[0, L]$  
\[
y_{-5} < .. < y_0 =0 < y_1 < .. < y_{N-1} = L, \quad N = 748,  \quad y_{-i} = - y_i, \quad  1\leq i \leq 5.
\]

\vs{0.05in}
\paragraph{\bf{Adaptive mesh}}
We design three parts of the mesh $y_i, i \in I_j \teq [a_j, b_j], a_0 = 0$ as follows 
\beq\label{eq:ASS_mesh}
\bal
& y_i = \f{i}{256}, i = -5, -4,..,1, .., b_1, \quad  
 y_{a_2 + i} = y_{a_2} + F( i h_3) ,  
i = 1,.., b_2 - a_2, \\
&  y_{a_3 + i} = y_{a_3} \exp( i r_1 ), i = 1, .., b_3 - a_3, \quad  r_0 = 1.025, \ r_1 = 1.15  \\
& F(z) =  \f{h_2}{ h_3} z \exp( r z^2),  \ r = \log( \f{r_0}{ 1 + h_3} ) \f{1}{ (1+h_3)^2 - 1} , \  h_2 = \f{1}{128}, h_3 = \f{1}{b_2 - a_2} . 
\eal
\eeq

Since we need to estimate the weighted $L^{\inf}$ norm of the residual error with a singular weight of order $|x|^{-\b}, \b \approx 3$ near $x=0$, we use uniformly dense mesh near $0$ so that we have a very small residual error. We choose the parameters $\f{1}{256}, h_2 = \f{1}{128}$ since they can be represented exactly as floating point numbers. Thus, we can reduce the round-off error in the computation.
In the far-field, we use a mesh that grows exponentially fast in space. Note that the error estimate $f - I(f)$ for the $k-$th order interpolation of $f$ on $[y_i, y_{i+1}]$ reads
\[
 | f - I(f) | \leq  C (y_{i+1} - y_i)^k |\pa_x^k f| .
\]
For large $x$, we expect that $ \pa_x^k f $ has a decay rate $|y|^{-k - \al}$ if $|f| \les |y|^{-\al}$ for $\al >0$. Thus, to get a uniformly small error in the far-field, we just require $\f{y_{i+1} - y_i}{y_i} \leq \e$ with $\e < 1$. This allows us to choose an exponentially growing mesh in the far-field and cover a very large domain without using too many points. We use the second part of the mesh to glue the first part of the mesh, which grows linearly, and the third part of the mesh. The functions $F(z)$ behaves linearly for $z$ close to $0$, and it grows exponentially fast with rate $r_1$ for $z$ close to $1$: 
\[
F(1+h_3) / F(1) =( 1 + h_3) \exp( r ( ( 1+h_3)^2 - 1)) = (1+h_3) \exp( \log(r_0 / (1+h_3))) = r_0.
\]
Parameters $h_2, h_3$ control the mesh size $y_{a_2+1} - y_{a_2} = F(  h_3 )= h_2 \exp( r h_3^2) \approx h_2$. 
One can design another $F(z)$ by gluing the first and the third part of the mesh. The above 
explicit and simple form of $F(z)$ serves our purpose. We further glue $y_{i}, i \in [b_j, a_{j+1}], j= 1,2$ using the Lagrangian interpolation for $j=1$. For $j=2$, we interpolate the growth rate using $\exp( \log(r_0) l(i) + (1 - l(i) ) \log(r_1) )$ with $l(i)$ linear in $i \in [b_2 , a_{3}]$. Note that we do not use the specific property of the profile to design the adaptive mesh \eqref{eq:ASS_mesh}.

\vspace{0.1in}
\paragraph{\bf{B-Spline and the tensor structure}}

In our numerical computation, we compute the derivatives of the solution using the B-spline basis, see e.g., \eqref{eq:w_spline}, and do not use the Jacobian related to the adaptive mesh. In particular, we do not use derivatives of the map $f(i) = y_i$, and have more flexibility to design the mesh.

First, we introduce the standard $k-$th order B-spline $\mBs(x ; \ss, k)$ with parameters $\ss \in \R^{k+1}$:
\bseq\label{eq:Bspline}
{\small
\beq\label{eq:Bspline:a}
 \mf{Bs}(x ; \ss ,k ) \teq  \sum \nolimits_{ 0 \leq j \leq k} k \f{  ( s_j - x)^{k-1}}{  d_j(\ss) } \one_{s_j -x \geq 0}, 
 \quad 
 d_j(\ss) \teq \prod \nolimits_{0 \leq \ell \leq k, \, \ell \neq j}  (s_j - s_\ell) .
\eeq
}
 We can rewrite the B-spline function equivalently as 
\footnote{
Given $\ss$, this identity follows from $\sum_{0 \leq j \leq k}  \f{ ( s_j - x)^{k-1}}{  d_j(\ss) }  \equiv 0$ for any $x$. Using Lagrangian interpolation identity 
 $p(x) \equiv \sum_{0 \leq j \leq k} p(s_j) \f{ \prod_{ \ell \neq j } (x- s_\ell) }{ \prod_{ \ell \neq j} (s_j -s_\ell ) } $ 
for any polynomial $p$ with degree $\leq k$, 
choosing $p(x) = x^m, m \leq k-1$, and comparing the coefficients of $x^k$ on both sides of the identity, we yield  $0 = \sum_{0 \leq j \leq k} p(s_j) ( \prod_{ \ell \neq j}  (s_j -s_\ell )  )^{-1}
= \sum_{0 \leq j \leq k } \ s_j^m d_j(\ss)^{-1}$.
}
{\small
\beq\label{eq:Bspline:b}
  \mf{Bs}(x ; \ss, k  )  =  -  \sum \nolimits_{ 0 \leq j \leq k} \f{ k ( s_j - x)^{k-1}}{  d_j(\ss) } \one_{s_j -x < 0} . 
\eeq
}
Suppose that $s_i$ in $\ss \in \R^{k+1}$ is increasing in $i$. The above equivalent formula implies 
\beq\label{eq:Bspline_prop1}
\bal
\mBs(x; -\ss, k) &=  \mBs(-x; \ss, k), \quad \forall \, x  \in \R,  
\quad \supp(\mBs(\cdot ; \ss, k  ))  \subset [s_0, s_k],  \\
\eal 
\eeq

\eseq

 Next, we define the $i$-th \emph{scaled} 6-th order B-spline associated with the mesh $\{ y_i \}$ using \eqref{eq:Bspline} with $k=6$:
\beq\label{eq:spline}
B_{i}(x) = \CCs_i   \cdot \mBs( x ;  \{ y_{i+j - k /2 } \}_{ 0 \leq j \leq k} , k ) , \quad k = 6.
\eeq
The constant $\CCs_i$ will be chosen in \eqref{eq:spline_normC1}, \eqref{eq:spline_normC2} so that the stiffness matrix associated to these B-spline basis has a better condition number. 
By \eqref{eq:Bspline_prop1},  $B_i$ is supported in $[y_{i-3}, y_{i+3}]$ and is centered around $y_i$. Since $\om$ is odd in $x$, to impose this symmetry in the representation, we modify the first few basis 
in $x$-direction and define the \emph{modified scaled} B-spline $B_{1, i}$:
\beq\label{eq:spline_odd}
B_{1, i}(x) \teq  B_{i}(x) - B_{i}(-x),  \quad i \leq k/2-1,
\quad B_{1, i}(x) \teq B_i(x), \quad i \geq k / 2, \  \forall  x \geq 0.
\eeq
Since $B_{1, 0}(x) \equiv 0$ 
and $B_{1, i}(x) = B_i(x) =  B_{i}(x) - B_{i}(-x) $ for $x \geq 0$ and $i \geq k/2$ 
by definition \eqref{eq:spline} and the support property \eqref{eq:Bspline_prop1}, $B_{1, i}(x)$ extends naturally to an odd function in $\R$.

Let $n_1 = 720 < N$. We solve the dynamic rescaling equation \eqref{eq:bousdy1}-\eqref{eq:normal} on first $n_1 \times n_1, (y_i, y_j), i, j \leq n_1-1$ grids. 
Using the odd basis $B_{1,i}(x)$ \eqref{eq:spline_odd} in $x$-direction
and $B_j(y)$ \eqref{eq:spline} in $y$-direction, 
we represent the solution as
\beq\label{eq:w_spline0}
\bar \om_2(x ,y) = \sum_{ 0 \leq i \leq n_1 + 1 , \, -2  \leq j \leq n_1 +1 } a_{ij} B_{1, i}(x) 
 B_{  j}(y) ,  \quad n_1 = 720.
\eeq
where $a_{ij} \in \R$ is the coefficient. Here, in $y$-direction, we need two extra 
basis $B_{-2}(y), B_{-1}(y)$ for smooth approximation near the boundary $y=0$, 
and discuss it in \hyperref[para:spline_extra]{Extrapolation}.

We also use the B-spline basis to represent the stream function \eqref{eq:ASS_stream} and solve the Poisson equation using the B-spline based finite element method. We use the B-spline basis since it is easy to design a high-order numerical scheme to solve the Poisson equation. Each basis function in \eqref{eq:w_spline0}, \eqref{eq:th_spline}, \eqref{eq:ASS_stream} has the form $f(x) g(y)$, which allows us to evaluate and estimate the 2D function very effectively using the method in Appendix \ref{app:piece_pol_2D}.

\begin{remark}
While the method described below to obtain the coefficients $a_{i, j}$ is technical, since we perform a-posteriori estimates of the profiles and residual error using the given $a_{i, j}$, the method of deriving $a_{i, j}$ is \emph{not involved} in the a-posteriori estimates and the verification process.
\end{remark}

\vs{0.1in}
\paragraph{\bf{Extrapolation}}\label{para:spline_extra}

Near the boundary $y=0$, we need 2 extra basis functions $ a_{i, - j} B_{- j}(y), j = 1, 2$ 
that do not vanish for $y \geq 0$. Without these 2 functions, the representation \eqref{eq:w_spline0} does not approximate $\bar \om$ with a $6-th$ order error.
For $j=1,2$ and fixed $i$, we use a 7-th order extrapolation \cite{luo2013potentially-2,luo2014potentially} to determine $a_{i, -j}$ from $a_{i, l}$ with $ 0 \leq l \leq 6$:
\[
\bal
a_{i, -j} & = \sum_{ 0 \leq l  \leq 6} C_{3-j, l + 1 } a_{i,  l } ,  \quad 
 C_{1, \cdot}  = (28, -112, 210, -224,140, -48, 7), 
\  C_{2, \cdot} = (7, -21, 35, -35, 21, -7, 1) .  \ 
\eal 
\]
We choose $C_{j, l}$  such that the 7-th difference of $a_{i, j}, -2 \leq j \leq 6$ is $0$.  For fixed $i$, since each $a_{i,-j}$ depends linearly on
$\{a_{i,l}\}_{0\leq l\leq 6}$, we may absorb the two extrapolated modes
$a_{i,-1}B_{-1}(y), a_{i,-2}B_{-2}(y)$ into the basis functions
$a_{i, l} B_l(y)$, $0\leq l\leq 6$. Thus, \eqref{eq:w_spline0} can be rewritten as
\bseq\label{eq:w_spline}
\beq
\bar\om_2(x,y)
=
\sum\nolimits_{0\leq i,j\leq n_1+1}
a_{ij}B_{1,i}(x)B_{2,j}(y),
\eeq
with the same basis $B_{1,i}$ in $x$-direction, while $a_{i, j} B_{2, j}(y)$ combines $a_{i,-1}B_{-1}(y), a_{i,-2}B_{-2}(y), a_{i, j} B_j(y)$:
\beq
\bal
 B_{2, j}(y) & \teq B_j(y) + C_{2, j + 1} B_{-1}(y) + C_{1, j+1} B_{-2}(y), & \  0 \leq & j \leq 6, \\ 
   B_{2, j}(y) & \teq B_j(y), & & j \geq 7. 
 \eal 
\eeq

\eseq


The modified basis functions $B_{1, i}, B_{2, j}$ are still piecewise polynomials in $[y_l, y_{l+1}]$. 
With the above modifications \eqref{eq:spline}, \eqref{eq:spline_odd}, \eqref{eq:w_spline} 
of the standard basis $\mBs$ \eqref{eq:Bspline}, we obtain the final basis $B_{1,i}(x), B_{2,j}(y)$,
where the subscripts $1$ and $2$ indicate $x$- and $y$-directions, respectively. 

\vs{0.1in}
\paragraph{\bf{Far-field extension}}

In \eqref{eq:w_spline0},\eqref{eq:w_spline}, we use B-spline $B_{1,i}(x), B_{2,j}(y)$ up to $i, j \leq n_1+1$ rather than $n_1-1$ since the support of $B_{1, i}, B_{2,j}$ intersects $[0, y_{n_1 -1}]^2$ for $i,j \leq n_1 + 1$. To determine the extra coefficients, we first extend the grid point values of $\bar \om_2(x, y)$ from $(y_i, y_j)$  with $i, j \leq n_1-1$ to $(y_i, y_j)$ with $ i, j \leq n_1 + l_0-1$ by 
\[
\bar \om_2(y_{n_1+l}, y_j ) = P(y_{n_1+ l}; y_j), \quad l =0,1,..,l_0 -3, 
\]
where $P(\cdot ; y_j)$ is the Lagrangian interpolation polynomials on three points: 
\[
 P(p; y_j) = q,  \quad \mbox{for} \  (p, q) =  ( y_{n_1-1},  \bar \om_2(y_{n_1-1}, y_j)), \quad  (y_{n_1 + l_0-3}, 0), \quad (y_{n_1+ l_0-2}, 0) .
\]
We impose $\bar \om_2(y_{ n_1+l}, y_j) = 0, l= l_0-3, l_0-2, l_0-1$. Similarly, we extend $\bar \om_2( y_i, y_{n_1 +l})$. 
Note that $\bar \om_2$ is odd and $B_{1,0} = 0$. We solve the coefficients $a_{ij}, 1\leq  i\leq M, 0 \leq j \leq M $ from 
\[
\bar \om_2(y_p, y_q) = \sum_{ 1 \leq i \leq M, \, 0\leq  j\leq M  } a_{ij} B_{1,i}( y_p ) B_{2,j}(y_q), \quad  1 \leq p\leq M,  \ 0 \leq q \leq M, \ 
 M = n_1 + l_0-1 \; .
\]
The value $a_{0 j}$ is not used since $B_{1,0} \equiv 0$. To simplify the notation, we keep it.
We only keep $a_{ij}, i , j \leq n_1+1$ and obtain \eqref{eq:w_spline}. In practice, we choose $l_0 = 8$ and the above construction provides a solution with tail decaying smoothly to $0$ for $|y|_{\inf} \geq y_{n_1 + l_0-1}$.

To solve the dynamic rescaling equations numerically \eqref{eq:bousdy1}-\eqref{eq:normal1}
 (see Section {\secASS} Part I), we update the grid point value of $\om_{n+1}$ at time $t_{n+1}$, and then use the above method to obtain $a_{ij}$.

For the density $\bar \th_2$, the representation is similar 
\beq\label{eq:th_spline}
\bar \th_2(x, y) = x \sum\nolimits_{0\leq i, j\leq n_1 +1} \, a_{ij} B_{1,i}(x) B_{2, j}(y) .
\eeq
Here, we multiply $x$ since $\bar \th$ is even and vanishes $O(x^2)$ near $x = 0$.

For the stream function $\bar \phi^N_{2}$ \eqref{eq:ASS_solu}, we choose $n_2 > n_1$ and represent it as follows 
\bseq\label{eq:ASS_stream}
\beq
\bar \phi^N_{2}(x, y) = \sum\nolimits_{ 0\leq i, j \leq n_2- 1} a_{ij} \td B_{1, i}(x) \td B_{ 2, j}(y) \rho_p(y).
\eeq

Instead of using the above extension to determine the extra coefficients, we perform an additional extrapolation for the basis in the far-field similar to \eqref{eq:w_spline}, using the same linear relation in the $x$-direction ($l=1$) and the $y$-direction ($l=2$):
\beq\label{eq:ASS_stream:b}
\bal
\td B_{\ell, j}(z) & = B_{\ell, j}(z), & \quad & j \leq n_2 - 8,  \\
\td B_{\ell, j}(z) &= B_{\ell,j}(z) + C_{ 2, n_2 - j} B_{\ell, n_2}(z)  + C_{1, n_2 - j} B_{\ell, n_2 + 1}(z) , 
 & & j \geq n_2 - 7.
\eal 
\eeq
\eseq

We multiply $\rho_p(y)$ given below to impose the Dirichlet boundary condition
\beq\label{eq:psi_wg}
\rho_p(y) = \arctan(1 + y) - \arctan(1).
\eeq
We can obtain the exact formulas of $\pa_y^i \rho_p(y)$ using a symbolic computation. We use induction to obtain rigorous estimate of $\pa_y^i \rho_p(y)$. See Section \ref{sec:stream_wg}.

We choose $\CCs_i$ in \eqref{eq:spline} to be comparable to 
the local mesh size for $B_i$:
\beq\label{eq:spline_normC1}
\CCs_i = y_1 - y_0, \  i \leq 9, \quad \CCs_i = ( y_{i+1} - y_{i-1} ) / 2,  \   9 <i \leq n - 1 ,
\eeq
so that the summand in \eqref{eq:spline} has order $1$ for $x$ in the support $[y_{i-3}, y_{i+3}]$. 
For the extrapolation of  $a_n B_{n}, a_{n+1} B_{n+1}$ from $a_i B_i, i\leq n-1$, e.g. $n=n_2$ for \eqref{eq:ASS_stream}, 
we modify the last few constants:
\beq\label{eq:spline_normC2}
\CCs_i = (y_n - y_{n-1}) / 100, \  n - 9 \leq i \leq n + 1.
\eeq
 We choose $\CCs_i$ to be constant for $i$ close to $0$ or $i$ close to $n$ since we need to perform extrapolation, and the choice of the constant does not affect the extrapolation formula for $a_{ij}$.

The far-field extrapolation \eqref{eq:ASS_stream:b} is applied only to $\bar\phi_2^N$, not to $\bar\om_2,\bar\th_2$ in \eqref{eq:w_spline}, \eqref{eq:th_spline}.

\vs{0.1in}
\paragraph{\bf{Far-field angular profile}}

To represent  far-field angular profiles of $\bar \om_1, \bar \th_1, \bar \phi_1^N$ \eqref{eq:ASS_solu}, we design mesh $0 = \b_0 < .. < \b_m  = \pi/2$, 
$\b_{-i}=-\b_i, 1 \leq i \leq 6$, and use B-spline to represent $\bar \om, \bar \zeta = \f{\bar \th}{x_1}$ :
\[
  g( \f{ \pi}{2} - \b) = \sum\nolimits_{0\leq i \leq n_{\b} } b_i \td B_{1, i}^{(8)}(\b),  \quad g_{\phi}( \f{\pi}{2} -\b) =   ( ( \f{\pi}{2} )^2- \b^2) \sum\nolimits_{0\leq i \leq n_{\b} } b_{\phi,i} \td B_{1,i}^{(8)}(\b) , \quad  \b  \in [0, \f{\pi}{2} ] .
\]
where $n_\beta$ is the largest basis index. 
To define ``tilde" B-spline $\td B_{1, i}^{(8)}(\b)$, 
 we first define $B_{1,i}^{(8)}$ using the $8$-th order B-spline \eqref{eq:Bspline}  with odd modification \eqref{eq:spline_odd} and \emph{without} the scaled factor $\CCs_i$
\[
B_i^{(8)}(\b) \teq \mBs( \b ; \{ \b_{i+j-4} \}_{0\leq j \leq 8}, 8),
\quad 
B_{1,i}^{(8)}(\b) \teq B_i^{(8)}(\b)- B_i^{(8)}(-\b) ,  i \leq 3,
\quad B_{1,i}^{(8)}(\b)  = B_i^{(8)}(\b), \, i \geq 4.
\]

We further modify the B-splines $B^{(8)}_{1,i}(\b)$ supported near $\b=\pi/2$, with $i$ close to $n_{\b}$, using a  9-th order extrapolation similar to \eqref{eq:w_spline}, and denote the resulting modified B-splines by $\td B^{(8)}_{1,i}(\b)$. We choose an equispaced mesh $\b_i$  near $\b= \pi/2$ and determine extrapolation coefficients  similar to \eqref{eq:w_spline}.  Since this extrapolation only applies to basis with large $i$, for basis supported near $\beta = 0$, with small $i$, 
 $\td B^{(8)}_{1, i}(\b) = B^{(8)}_{1, i}(\b)$  are odd about $\b=0$.
 Since $\bar \om, \bar \zeta$ are odd about $x$, 
 their angular profiles are odd about the angle $\b =\pi/2$ for $(x,y) \in \R^2_+$.
Writing $g$ in terms of \(\pi/2-\b\) and using the odd symmetry of $\td B_{1,i}^{(8)}(\b)$ in $\b=0$,
we enforce that $g(\b)$ is odd about $\b = \pi/2$.

  The stream function $\bar \phi^N$ satisfies the boundary condition $\bar \phi^N(x, 0) = 0$. For its angular profile, we further require $g_{\phi}(0) = 0$, and use the factor $(\pi/ 2)^2- \b^2$ to impose this condition. 

  Note that evaluating the derivative $\pa_{\b}^k g$ at $\pi/2-\b$ introduces the sign factor $(-1)^{k}$
 \[
 (\pa_{\b}^k g)( \pi/2 - \b) = (-1)^k \pa_{\b}^k ( g( \pi/2-\b) )  =(-1)^k  \tts{\sum}_{0\leq i\leq n_{\b}}  b_i \pa_{\b}^k \td B_{1, i}^{(8)}(\b) .
 \]

We discuss how to obtain these angular profiles using the curve fitting in Section {\secASS} in \cite{ChenHou2023a}.

\subsection{Estimate of the derivatives of piecewise polynomials}\label{app:piece_pol}

Our approximate steady state in a very large domain is represented using piecewise polynomials 
and explicit functions.  We discuss how to estimate its derivatives. 
Let $f$ be an explicit function or a polynomial.  
To obtain a piecewise sharp bound of $f$ on $I = [x_l, x_u]$, we use the following standard error estimate 
\beq\label{est_2nd}
  \max_{ x \in I }|f(x)|  \leq \max( |f(x_l) | , |f(x_u)| ) + \tfrac{ h^2}{8} || f_{xx}||_{L^{\inf}(I)}, \quad h = x_u - x_l,
\eeq
If we can obtain a rough bound for $f_{xx}$, as long as the interval $I$ is small, i.e., $h$ is small, the error part is small. Similarly, if we can obtain a rough bound for $\pa_x^{k+2} f$, using induction and the above estimate recursively,
\[
 \max_{ x \in I }| \pa_x^i f(x)|  \leq \max( |\pa_x^i f(x_l) | , |\pa_x^i f(x_u)| ) + \tfrac{ h^2}{8} || \pa_x^{i+2}f||_{L^{\inf}(I)},
 \]
for $i=k, k-1, ..., 0$, we can obtain the sharp bound for $\pa_x^i f$ on $I$. We call the above method the second order method since the error term is second order in $h$.

\subsubsection{Estimate a piecewise polynomial in 1D}\label{app:piece_pol_1D}

Suppose that $p(x)$ is a piecewise polynomials on $x_0 < x_1< ..< x_n$ with degree $d$, e.g. Hermite spline. Denote $I_i = [x_i, x_{i+1}]$. Then $p(x)$ is a polynomial in each $I_i$ with degree $\leq d$. Our goal is to estimate $\pa_x^k p(x)$ in $I_i$ for all $k$ by only finite many evaluations of $p(x)$ and its derivatives. Firstly, we have 
\[
\pa_x^k p(x) = 0, \quad k > d , \quad \pa_x^d p(x) = c_p, 
\]
for some constant $c_p$ in $I_i$. 
Using induction from $ k =d-1, d-2, .., 0$, we have 
\[
 \max_{ x \in I_i }| \pa_x^k p (x)|  \leq \max( |  \pa_x^k p(x_i) | ,| \pa_x^k p (x_{i+1})| ) + \f{ h_i^2}{8} || \pa_x^{k+2} p ||_{L^{\inf}(I_i)}, \quad h_i = x_{i+1} - x_i.
\]

Since we know $\pa_x^{d+1} p(x) = 0$ on $I_i$, using the above method, we can obtain the sharp piecewise bounds for all derivatives of $p(x)$ on $I_i$. Using the above approach, we can estimate the derivatives of the angular profile defined in Section 7.1 of Part I \cite{ChenHou2023a} rigorously.

\subsubsection{Estimate a piecewise polynomial in 2D}\label{app:piece_pol_2D}

Now, we generalize the above ideas to 2D so that we can estimate the approximate steady state \eqref{eq:w_spline}. We assume that $p(x, y)$ is a piecewise polynomials in the mesh $Q_{ij} = [x_i, x_{i+1}] \times [y_j , y_{j+1}]$ with degree $d$. That is, in $Q_{ij}$, $p(x, y)$ can be written as a linear combination of 
\[
x^k y^l, \quad  \max(k, l) \leq d ,
\]
e.g. \eqref{eq:w_spline}. For \eqref{eq:w_spline}, we have $d=5$. Similar to the 1D case, we have 
\[
 \pa_x^k \pa_y^l p(x, y) = 0, \quad \max(k, l) > d .
\]
Moreover, we know $\pa_x^{d-1} \pa_y^{d-1}$ is linear in $x, y$. 

We use the following direct generalization of \eqref{est_2nd} to 2d
\beq\label{est_2nd_2D}
\bal
 \max_{ (x, y) \in Q } |f(x, y| 
 &\leq \max_{\al, \b = l, u} | f( x_{\al}, y_{\b} ) |  
 +  \f{ ||f_{xx}||_{L^{\inf}(Q)} (x_u-x_l)^2 }{8}
 + \f{ ||f_{yy}||_{L^{\inf}(Q)} (y_u- y_l)^2 }{8}, \\
   Q  &= [x_l, x_{u}] \times [y_l , y_u ] .
 \eal
\eeq

Denote 
\[
 A_{kl} \teq \max_{Q_{ij}} ||\pa_x^k \pa_y^lp||_{L^{\inf}(Q_{ij})},
 B_{kl} \teq \max_{\al, \b = l, u} | \pa_x^k \pa_y^l p( x_{\al}, y_{\b} ) | ,
\quad h_1 = x_{i+1} - x_i, \quad h_2 = y_{j+1} - y_j. 
\]
Since $p$ is given, we can evaluate $B_{kl}$. Clearly, we have $A_{kl} = 0$ for $\max(k, l) > d$. For $k = d-1, d$, using \eqref{est_2nd_2D} and induction in the order $l= d, d-1, d-2, .., 0$, we can obtain 
\[
A_{kl} \leq B_{kl} + \f{1}{8} ( h_1^2 A_{k+2, l } + h_2^2 A_{k, l+2} ).
\]
This allows us to bound $A_{kl}$ for $k = d, d-1$ and all $l$. Similarly, we can bound $A_{kl}$ for $l=d, d-1$ and all $k$. 

For the remaining cases, we can use induction on $n = \max(k, l) = d-2, d-1, .., 0$ to estimate
\[
A_{kl} \leq B_{kl} + \f{1}{8} ( h_1^2 A_{k+2, l } + h_2^2 A_{k, l+2} ).
\]
This allows us to estimate all derivatives of $p(x, y)$ in $Q_{ij}$.

\subsubsection{Estimate a piecewise polynomial in 2D with weights}

We consider how to estimate the derivatives of $ f = \rho(y) p(x, y)$, where $\rho $ is a given weight in $y$ and $p(x, y)$ is the piecewise polynomials in 2D. For example, our construction of the stream function \eqref{eq:ASS_stream} has such a form. Firstly, we can estimate the derivatives of $p(x, y)$ using the method in Appendix \ref{app:piece_pol_2D}. For the weight $\rho$, we estimate its derivatives in Section \ref{sec:stream_wg}. Then, using the Leibniz rule 
\eqref{eq:lei} and the triangle inequality, we can estimate the derivatives $f$
\[
|\pa_x^i \pa_y^j f |
\leq \sum_{l \leq j} \binom{j}{l} | \pa_x^i \pa_y^l p(x, y)|  |\pa_y^{j-l} \rho(y)|
\]
for high enough derivatives.

Now,  we plug the above bounds for $\pa_x^{i+2} \pa_j^y, \pa_x^i \pa_y^{j+2} f $ in \eqref{est_2nd_2D} and evaluate $ \pa_x^i \pa_y^j f $ on the grid points to obtain the sharp estimate of $\pa_x^i \pa_y^j f$.

\subsection{Estimate of the far-field approximation}\label{sec:est_appr_far}

We estimate the derivatives of 
\[
g(x, y) = g(r, \b) = A(r) B(\b), \quad r = (x^2 + y^2)^{1/2}, \quad \b = \arctan(y/x),
\]
where $(r, \b)$ is the polar coordinate. The semi-analytic parts of $\bar \om, \bar \th$ have the above forms.

\subsubsection{Formulas of the derivatives of $g$}
Firstly, we  use induction to establish 
\beq\label{eq:ansatz_appr}
F_{i,j} \teq \pa_x^i \pa_y^j g(r, \b) = \sum_{k + l \leq i + j} 
C_{i,j ,k,l}(\b) r^{-i-j + k} \pa_r^k A \pa_{\b}^l  B, 
\eeq
with $ C_{i,j,k,l} = 0$, for $k <0$, $l<0$ , or $k+l > i+j$. Let us motivate the above ansatz. Recall from \eqref{eq:Dxy} that 
\[
\pa_x = \cos \b \pa_r - \f{ \sin \b}{r} \pa_{\b}, \quad  \pa_y = \sin \b \pa_r + \f{ \cos \b}{r} \pa_{\b}.
\]

For each derivative $\pa_x$ or $\pa_y$, we get the factor $ \f{1}{r}$ or a derivative $\pa_r$, which leads to the form $r^{-i-j + k} \pa_r^k A$. Moreover, we get a derivative $\pa_{\b}$ and some functions depending on $\b$, which leads to the form $C_{i,j,k,l}(\b) \pa_{\b}^l B$.

For $D = \pa_x$ or $\pa_y$, a direct calculation yields 
\beq\label{eq:ansatz_appr2}
D F_{i, j} = \sum_{k+l \leq i+j} D( C_{i,j,k,l} r^{-i-j+k} )\cdot \pa_r^k A \pa_{\b}^l B
+ C_{i,j,k,l} r^{-i-j+k} ( D \pa_r^k A  \cdot \pa_{\b}^l B+  \pa_r^k A \cdot D \pa_{\b}^l B ).
\eeq
Using the formula of $\pa_x, \pa_y$, we get
\[
\bal
&\pa_x  ( C_{i,j,k,l}(\b) r^{-i-j + k} )=   - \sin \b\pa_{\b} C_{i,j,k,l} r^{-i-j - 1 + k}
+ (k - i - j ) \cos \b C_{i,j,k,l} r^{-i-j - 1 +k},  \\
& \pa_x \pa_r^k A = \cos \b \pa_{r}^{k+1} A, \quad \pa_x \pa_{\b}^l B
= - \f{\sin \b}{r} \pa_{\b}^{l+1} B , \\
\eal
\]
Using $\pa_x F_{i, j} = F_{i+1, j}$ and comparing the above formulas and the ansatz \eqref{eq:ansatz_appr}, we yield 
\beq\label{eq:ansatz_appr_ind1}
\bal
C_{i+1, j, k, l} &= (k-i-j)\cos \b C_{i,j,k,l} - \sin \b \pa_{\b} C_{i,j,k,l}
+ \cos \b C_{i,j, k-1, l}  - \sin \b C_{i,j, k, l-1} ,\\
\eal
\eeq
for $k \leq i+ j$. Similarly, for $D = \pa_y$, plugging the following identities
\[
\bal
&\pa_y  ( C_{i,j,k,l}(\b) r^{-i-j + k} )=   \cos \b \pa_{\b} C_{i,j,k,l} r^{-i-j - 1 + k}
+ (k - i - j ) \sin(\b) C_{i,j,k,l} r^{-i-j - 1 +k},  \\
& \pa_y \pa_r^k A = \sin \b \pa_{r}^{k+1} A, \quad \pa_y \pa_{\b}^l B
=  \f{\cos \b}{r} \pa_{\b}^{l+1} B  \\
\eal
\]
 into \eqref{eq:ansatz_appr2} and then comparing \eqref{eq:ansatz_appr} and \eqref{eq:ansatz_appr2}, we yield 
\beq\label{eq:ansatz_appr_ind2}
C_{i, j+ 1, k, l} = (k-i-j)\sin \b C_{i,j,k,l} + \cos \b \pa_{\b} C_{i,j,k,l}
+ \sin \b C_{i,j, k-1, l}  + \cos \b C_{i,j, k, l-1} . \\
\eeq

The based case is given by
\[
F_{0,0} = A(r) g(\b), \quad C_{0,0,0,0} = 1.
\]
Using induction and the above recursive formulas, we can derive $C_{i,j,k,l}(\b)$ in \eqref{eq:ansatz_appr}. 

\subsubsection{Estimates of $F_{i, j}$}

To estimate $F_{i, j}$, using \eqref{eq:ansatz_appr} and triangle inequality, we only need to estimate $\pa_r^k A, \pa_{\b}^l B(\b)$, and $C_{i,j,k,l}(\b)$. In our case, $B(\b)$ is piecewise polynomials, whose estimates follow the method in Appendix \eqref{app:piece_pol_1D}. Function $A(r)$ is some explicit function, which will be constructed and estimated in Section \ref{sec:est_radial}. To estimate $C_{i,j,k,l}(\b)$ on $\b \in [\b_1, \b_2]$, we use the second order estimate in \eqref{est_2nd} and the induction ideas in Section \ref{app:piece_pol_1D}. We can evaluate $ C_{i,j,k,l}$ using its exact formula. It remains to bound $\pa_{\b}^2 C_{i,j,k,l}.$

An important observation from \eqref{eq:ansatz_appr_ind1}, \eqref{eq:ansatz_appr2} is that $C_{i,j,k,l}$ is a polynomial on $\sin \b$ and $\cos \b$ with degree less than $i+j$, which can be proved easily using induction. In particular, we can write $ C_{i,j,k,l}$ as follows 
\[
C_{i,j,k,l} =  \sum_{ 0 \leq k \leq n} a_{k} \sin( k\b) + b_{k} \cos(k \b), \quad 
f \teq \pa_{\b}^2 C_{i,j,k,l} = \sum_{ 1\leq k \leq n} c_k \sin( k\b) + d_k \cos(k \b), \quad n = i+j 
\]
for some $a_k, b_k ,c_k, d_k \in \R$. It is easy to see that $C_{i,j,k,l}$ is either odd or even in $\b$ depending on $j-l$, which implies $c_k \equiv 0$ or $d_k \equiv 0$. Using Cauchy-Schwarz's inequality, we get 
\[
||f||_{\inf} \leq \sum_{1 \leq k \leq n } (|c_k| + |d_k|)
\leq \B(  n \sum_{ k \leq n} ( c_k^2 + d_k^2)  \B)^{1/2} 
= \B( \f{n}{\pi} \int_0^{2\pi} f^2  \B)^{1/2},
\]
where we have used orthogonality of $\sin kx , \cos kx$ and $|| f||_{L^2}^2 = \pi \sum_{k\leq n} ( c_k^2 + d_k^2 )$
in the last equality. It is easy to see that $ f^2$ is again a polynomial in $\sin \b, \cos \b$ with degree $\leq 2 n$. We fix $M > 2n$. For any $ 0 \leq k < M$, it is easy to obtain 
\[
\f{1}{2\pi} \int_0^{2\pi} e^{i k x} = \f{1}{M} \sum_{j=1}^M \exp( i \f{2 kj}{M}\pi )
= \d_{k0}.
\]

Using the above identity, we establish 
\[
|| g ||_{L^2}^2 = \f{2\pi}{M} \sum_{j=1}^M | g ( \f{2 j\pi}{M}) |^2, 
\]
for any polynomial $g$ in $\sin \b, \cos \b$ with degree $< M / 2$. Hence, we prove 
\[
|| f||_{\inf} \leq \B( \f{2n}{M} \sum_{k=1}^M f^2( \f{2j\pi}{M})  \B)^{1/2}.
\]

The advantage of the above estimate is that to obtain the sharp bound of $ C_{i,j,k,l}$, we only need to evaluate $C_{i,j,k,l} ,f = \pa_{\b}^2C_{i,j,k,l} $ on finite many points.

\subsubsection{From polar coordinates to the Cartesian coordinate}

We want to obtain the piecewise estimate of $F_{p, q} =  \pa_x^p \pa_y^q ( A(r) g(\b))$ on $Q_{ij} = [x_i, x_{i+1}] \times [y_j, y_{j+1}], 1 \leq i, j \leq n$.  Firstly, we partition the $(r, \b)$ coordinate into $r_1 < r_2 <..< r_{n_1},  0 = \b_0 < b_1 < ... < \b_{n_2} = \f{\pi}{2}$. Then we apply the methods in Section \ref{sec:est_appr_far} to bound $F_{p, q}$ on $S_{ij} \teq[r_i, r_{i+1}] \times [\b_j, \b_{j+1}]$. 
We cover $Q_{ij}$ by $S_{k,l}$ and transfer the bound from $(r, \b)$ coordinate to $(x, y)$ coordinate 
\[
\max_{x \in Q_{ij}} | F_{p ,q}(x) | \leq \max_{ S_{k, l} \cap Q_{ij} \neq \emptyset } || F_{p, q}(r, \b)||_{L^{\inf}(S_{k,l})} 
\]

For $(r, \b) \in Q_{i, j}$, we get \[
 r \in [ (x_i^2 + y_j^2)^{1/2},  \ (x_{i+1}^2 + y_{j+1}^2)^{1/2} ] , \quad  \b \in [ \arctan \f{y_j}{x_{i+1}}, \  \arctan \f{y_{j+1}}{x_i} ] .
\]

Therefore, we yield the necessary conditions for $Q_{i, j} \cap S_{k, l} \neq \emptyset$:
\[
x_{i+1}^2 + y_{j+1}^2 \geq r_k^2, \quad x_i^2 + y_i^2 \leq r_u^2, \quad 
 \arctan \f{y_{j+1}}{x_{i}}  \geq \b_{l}, \quad 
 \arctan \f{y_j}{x_{i+1}} \leq \b_{l+1} .
\]
Given $Q_{i, j}$, we maximize $ || F_{p, q}||_{L^{\inf}(S_{k,l})}$ over $(k,l)$ satisfying the above bounds to control $|| F_{p, q}||_{L^{\inf} (Q_{i,j})}$.

\subsection{Estimates of the residual error}\label{sec:resid}

Let $\chi_{ \bar \e} = 1 +O(|x|^4)$ be the cutoff function in \eqref{eq:cutoff_near0_all}.
Firstly, we decompose the error of solving the Poisson equations $\bar \e = \bar \om - (-\D  ) \bar \phi^N$ as follows 
\beq\label{eq:uerr_dec1}
\bal
  &  \bar \e  =  \bar \e_1 +  \bar \e_2, \quad \bar \e_2 =  \bar \e_{xy}(0) \D( \f{x y^3}{6}  \chi_{ \bar \e}) , \quad \uu(  \bar  \e_2) = \na^{\perp}(-\D)^{-1} \bar \e_2 = - \f{1}{6}  \bar \e_{xy}(0)\na^{\perp}  ( x y^3 \chi_{ \bar \e}) ,  \\
  & \uu(  \bar \e)   = \uu( \bar \e_1) + \uu( \bar \e_2)  = \uu_A(  \bar \e_1) + ( \hat \uu( \bar \e_1 ) + \uu( \bar \e_2) ) \teq \uu_A(  \bar \e_1) + \uu_{loc}(\bar \e) , 
 \eal
\eeq
where $\hat \uu$ is the approximation term for $\uu$ defined in Section 4.3 in Part I \cite{ChenHou2023a}. We perform the above correction near $0$ so that $\bar \e_1 = O(|x|^3)$ near $0$. 
We perform a similar decomposition for $(\na \uu)_A$. Note that we do not have $\pa_{x_i} \uu_A = (\pa_{x_i} \uu)_A$. Using the above decomposition and the notation \eqref{eq:Blin_gen}, we can rewrite the residual error $\bar \cF_i$ \eqref{eq:bous_err} with rank-one correction as follows 
\[
\bar \cF_i - D_i^2 \bar \cF_i(0) f_{\chi, i} = \bar \cF_{loc, i} + 
\cB_{op, i}( (\uu_A(\bar \e_1), (\na \uu)_A(\bar \e_1) ), \bar W ), 
\]
where $D^2 = (\pa_{xy}, \pa_{xy}, \pa_x^2 )$ is defined in \eqref{eq:diff_op} and $\bar \cF_{loc, i}$ is defined below in \eqref{eq:bous_errM}. Since $\uu_A(\bar \e_1) = O(|x|^3), (\na \uu)_A(\bar \e_1) = O(|x|^2)$ (see Section 4.3 in Part I \cite{ChenHou2023a} for 
these properties of $\uu_A = \uu - \hat \uu$), from \eqref{eq:Blin_gen} and \eqref{eq:uerr_dec1}, we get 
\[
\cB_{op, i}( (\uu_A(\bar \e_1), (\na \uu)_A(\bar \e_1) ), \bar W )= O(|x|^3),
\quad  u_x(\bar \e_2)(0)=0, \quad u_{x, A}( \bar \e_1)(0) =0. 
\]
Using these properties of $\cB_{op, i}$, we define $\bar \cF_{loc ,i}$ as follows
\beq\label{eq:bous_errM}
\bal
  \bar \cF_{loc, i} &= {II}_i - D_i^2 II_i(0) f_{\chi,i} , 
 \quad
   II_i   = \bar \cF_i - \cB_{op, i}( (\uu_A(\bar \e_1), (\na \uu)_A(\bar \e_1) ), \bar W ), \\
  \uu(\bar \om) & = \bar \uu = \bar \uu^N + \uu_{loc}(\bar \e) + \uu_A(\bar \e_1), 
  \ \bar c_{\om} = \bar c_{\om}^N +  u_x(\bar \e_1)(0), \ \bar c_{\om}^N \teq \f{\bar c_l}{2} +  \bar u^N_x(0), \ c_{\om}(\bar \e_1) \teq u_x(\bar \e_1)(0) , \\
II_1 & = - (\bar c_l x + \bar \uu^N + \uu_{loc}(\bar \e) ) \cdot \na \bar \om + \bar \th_x 
+ (\bar c^N_{\om} + \bar c_{\om}(\bar \e_1) )\bar \om, \\
II_2 & = - (\bar c_l x +\bar \uu^N + \uu_{loc}(\bar \e)  ) \cdot \na \bar \th_x + 2 (\bar c^N_{\om} + \bar c_{\om}(\bar \e_1)) \bar \th_x - (  \bar u_x^N + u_{x, loc}(\bar \e)  )  \bar \th_x - ( \bar v_x^N + v_{x, loc} (\bar \e) ) \bar \th_y, \\
II_2 & = - (\bar c_l x +\bar \uu^N + \uu_{loc}(\bar \e)  ) \cdot \na \bar \th_y + 2 ( \bar c^N_{\om} +  c_{\om}(\bar \e_1)) \bar \th_y - (  \bar u_y^N + u_{y, loc}(\bar \e)  )  \bar \th_x - ( \bar v_y^N + v_{y, loc} (\bar \e) ) \bar \th_y , \\
 \eal
\eeq
where $f_{\chi, i}$ is defined in \eqref{eq:cutoff_near0_all}, and we have used 
$\bar c_{\th} = \bar c_l + 2 \bar c_{\om}$ \eqref{eq:bous_err}, \eqref{eq:normal} for $\bar c_{\om}$. 
The above decomposition is essentially the same as \eqref{eq:decomp_L}. We apply the functional inequalities in Section \ref{sec:vel_comp} to estimate the nonlocal terms $\uu_A(\bar \e_1), (\na \uu)_A(\bar \e_1)$, and combine the estimate of $\cB_{op, i}( (\uu_A, (\na \uu)_A ), \bar W)$ with the energy estimate. See Section {\seccombvelerr} in Part I \cite{ChenHou2023a} for more details about the decompositions and estimates. 

The terms $II_i$ depend on the profile $\bar \om ,\bar \th, \bar \e$ locally. Using the decomposition \eqref{eq:uerr_dec1}, we can further decompose the above $II_i$ as follow 
\[
II_i = II_i^N + II_i(\bar \e_1) + II_i(\bar \e_2), \ 
II_i(\bar \e_1) = \cB_{op, i}( \hat \uu(\bar \e_1), \wh{ \na \uu}(\bar \e_1), \bar W), \ 
II_i(\bar \e_2) =  \cB_{op, i}( \uu(\bar \e_2),  \na \uu (\bar \e_2), \bar W), 
\]
where $II_i^N$ contain the terms in $II_i$ except the $u_{loc}, u(\bar \e_1)$ terms.

For $\hat \uu(\bar \e_1)$, it is a finite rank operator on $\bar \e_1$, and we can write it as 
\[
\hat \uu(\bar \e_1) = \sum_{i=1}^n a_i(\bar \e_1) \bar g_i(x)
\teq C_{\uu 0}(x) u_x( \bar \e_1)(0) 
+ \td {\hat \uu}(\bar \e_1) , \quad a_i(\bar \e_1 ) =  \int_{\R_2^{++}} \bar \e_1(y) q_i(y) dy , 
\]
for some functions $\bar g_i(x)$ and $q_i(y)$, where $C_{\uu0}(x)$ is given in \eqref{eq:u_appr_near0_coe}, and $ \td{ \hat \uu}( \bar \e_1)$ denotes other modes with $O(|x|^3)$ vanishing order near $0$. See Section {\secapprvel} in \cite{ChenHou2023a} for definition. 
We can obtain more regular estimates, e.g. $C^3$ estimates, of $\hat \uu(\e_1)$ since $\bar g_1(x)$ is smooth. Similarly, we decompose $ \wh {\na \uu}(\bar \e_1)$. We obtain piecewise estimates of $\pa_x^i \pa_y^j \bar \e_1, i + j \leq 1$ following the methods in Section \ref{sec:err_idea} and Section
{\applinfestsupp} in the supplementary material II \cite{ChenHou2023bSupp} (attached to 
arXiv version of \cite{ChenHou2023b}) and then the above integrals on $\bar \e_1$. 
The main term in $\hat \uu(\bar \e_1)$ is $C_{u0} u_x(0)$ with
\beq\label{eq:ux0_ep2}
\bal
& u_x( \bar \e_1)(0) = u_x(\bar \e)(0) = -\f{4}{\pi} \int_{\R_2^{++}} \bar \e(y) \f{y_1 y_2}{|y|^4} dy,  \\
 & u_x(\bar \e_2)(0) = -\pa_{xy}(-\D)^{-1} \bar \e_2(0) =  \tfrac16 \e_{xy}(0) \cdot \pa_{xy} (x y^3 \chi_{\bar \e}) |_{(0,0)} =0.
\eal
\eeq

Since the kernel $\f{y_1 y_2}{|y|^4}$ has a slow decay for large $|y|$ (not in $L^1$), we need to estimate $u_x(\bar \e)(0)$ carefully, using Simpson's rule. See Section {\seccwerr} in supplementary material II for Part II \cite{ChenHou2023bSupp}.

Using the above decomposition, we further decompose $ \hat \uu(\bar \e_1)$ 
\[
II_i( \bar \e_1)
= u_x(\bar \e)(0) \cB_{op, i}(  C_{\uu0}(x) , C_{\na \uu 0}(x), \bar W )
+ \cB_{op,i}( \wt { \wh {\uu}}, \wt { \wh { \na \uu} }, \bar W )
\teq II_{ i, M}(\bar \e_1) + II_{i, R}(\bar \e_1).
\]
Since $D_i^2$ is linear, we estimate each term $g_i - D_i^2 g_i(0) f_{\chi, i}$ for $g_i = II_{i, M}(\bar \e_1), II_{i, R}(\bar \e_1), II_i^N, II_i(\bar \e_2)$ to bound $\cF_{loc, i}$. 
To estimate $II_{i, R}$, since $ \td {\hat \uu}( \bar \e_1 ) = O(|x|^3)$ near $0$, 
(see Section {\secapprvel} in \cite{ChenHou2023a}), we get $D_i^2 II_{i, R}(\bar \e_1) = O(|x|^3)$ and estimate
\[
\quad  \td {\hat  \uu}(\bar \e_1) \rho_{10}, 
\ \pa_i \td {\hat  \uu}(\bar \e_1) \rho_{20}, \
  \wt{ \wh {\na \uu}} (\bar \e_1) \rho_{20},  \
\pa_i \wt{ \wh {\na \uu}} (\bar \e_1) \rho_3, \ \rho_4 \td {\hat \uu}( \bar \e_1 )
\]
for $\rho_{i0} $ \eqref{wg:linf} with $\rho_{i0} \sim |x|^{-4 + i}, i\leq 3$ near 0 
using the $C^3$ bounds of $\td {\hat \uu}, \wt {\wh {\na \uu}}$. Note that $\pa_i \td {\hat  \uu}  \neq \wt {\widehat { \pa_i \uu} } $. The former is the derivative of $  \td {\hat  \uu} $, and the later is the approximation term for $ \pa_i \uu$. With the above weighted estimate, we can bound a typical terms, e.g. $ \wt{ \wh {u_x} } \bar \th_x \vp_2$ in $II_{i, R}(\bar \e_1) \vp_2$ as follows 
\[
\wt {\wh {u_x} } \bar \th_x \vp_2
 = \wt {\wh {u_x}} \rho_{20} \cdot (\bar \th_x  \f{\vp_2}{\rho_{20}} ),  \ 
 \pa_x( \wt {\wh {u_x} }\bar \th_x ) \rho
 = (\pa_x   \wt {\wh{ u_x}} \bar \th_x + \wt { \wh{u_x}} \pa_x \bar \th_x) \rho= \pa_x   \wt {\wh{ u_x} } \rho_3  \cdot \f{ \bar \th_x \rho}{\rho_3}+ \wt {\wh{u_x} } \rho_{20} \cdot  \f{ \pa_x \bar \th_x \rho }{ \rho_{20}} ,
\] 
where $\vp_2$ is given in \eqref{wg:linf}. Each term $A, B$ in the above products  $A \cdot B$
is regular and we estimate each term and then the product to bound weighted $L^{\inf}$ and $C^1$ norm of $II_{i, R}(\bar \e_1)$. 

The remaining part in $II_i^N, II_{i, M}(\bar \e_1), II_i(\bar \e_2)$ depends on $(\bar \phi^N, \bar \om, \bar \th) $ locally and are given functions.
To estimate the weighted $L^{\inf}$ and $C^{1/2}$ norms of 
$g_i - D_i^2 g_i(0) f_{\chi, i} = O(|x|^3)$ with $g =  II_{i,M}(\bar \e_1), II_{i}(\bar \e)$, we follow the methods in Sections \ref{sec:err_idea}, \ref{sec:vel_err} with $\pa_t \bar \om = \pa_t \bar \th = 0$.

\vs{0.1in}
\paragraph{\bf{Estimate in the far-field}}
Since $\bar \om , \bar \th$ are supported globally, we need to estimate the error in the far-field. Recall the formulas of $\bar \om, \bar \om_1, \bar \th, \bar \th_1$ from \eqref{eq:ASS_solu}. We consider $|x|_{\inf} \geq R_1 \geq 10^{12} > 10 a_2$ beyond the support of $\bar \om_2, \bar \th_2, \bar \phi_2^N, \bar \phi_3^N, \bar \phi^N_{cor}$ \eqref{eq:ASS_solu} so that $\chi(r) = 1$ \eqref{eq:cut_radial} and 
\[
\bar \om = \bar \om_1 = \bar g_1(\b) r^{ \bar \al_1} , \quad  \bar \th = \bar \th_1 = r^{1 + 2 \bar \al_1 } \bar g_2(\b) , \quad  \bar \phi^N = \bar \phi_1^N  = r^{ 2 + \bar \al_1 } \bar f(\b).
 \]
 We estimate the angular derivatives of $f(\b), g_i(\b)$ using the methods in Section \ref{app:piece_pol_1D}.
Using the above representation, $ x \cdot \na r^{\b} = r \pa_r r^{\b} = \b r^{\b}$, $x \cdot \na (\pa \bar \th_1) = 2 \bar \al_1 (\pa \bar \th_1),
x \cdot \na \bar \om_1 =  \bar \al_1 \bar \om_1$, $\bar c_{\om} = \bar c_{\om}^N + \bar c_{\om}^{\bar \e}$ \eqref{eq:u_Ne}, and separating $\uu^N$ and $\uu_{loc}$ in \eqref{eq:bous_errM}, 
for $|x|_{\infty} \geq 10^{12}$, we obtain 
 \[
 \bal
 \bar \cF_{loc, 1}
&
 = \B( (\bar c^N_{\om} - \bar c_l \bar \al_1) \bar \om_1  - \bar \uu^N \cdot \na \bar \om_1 + \bar \th_{1, x} \B)+ \bar c_{\om}^{\e} \bar \om_1 -  \uu_{loc} \cdot \na \bar \om_1 \teq I_{11} + I_{12} , \\
 \bar \cF_{loc, 2} 
&  = \B(   (2 \bar c_{\om}^N - 2 \bar c_l  \bar \al_1 ) \bar \th_{1, x}
 -  \pa_x (\bar \uu^N \cdot \na \bar \th_1 ) \B) 
+ 2 \bar c_{\om}^{\e} \bar \th_{1,x} - \uu_{loc} \cdot  \na \bar \th_{1, x}
-  u_{x, loc}    \bar \th_x -  v_{x, loc}  \bar \th_y \teq I_{21} + I_{22},  \\
 \bar \cF_{loc, 3} 
&  = \B(   (2 \bar c_{\om}^N - 2 \bar c_l  \bar \al_1 ) \bar \th_{1, y}
 -  \pa_x (\bar \uu^N \cdot \na \bar \th_1 ) \B) 
+ 2 \bar c_{\om}^{\e} \bar \th_{1, y} - \uu_{loc} \cdot  \na \bar \th_{1, y}
-  u_{y, loc}   \bar \th_x -  v_{y, loc}  \bar \th_y \teq I_{31} + I_{32}, 
 \eal
 \]
 where we have simplified $\uu_{loc}(\bar \e)$ as $\uu_{loc}$ and used $f_{\chi, i} = 0$ \eqref{eq:cutoff_near0_all}, $\bar \cF_{loc,i} = II_i $ \eqref{eq:bous_errM} since $f_{\chi, j}$ is supported near $0$. 
 The terms $I_{11}, I_{21}, I_{31}$ are local with the form $r^{\g} q(\b)$ for some angular function $q$ and decay rate $\g$. We estimate their piecewise 
 $L^{\infty}$ and derivative bounds using 
 \eqref{eq:Dxy}. From our choice of $\bar \al_1$ \eqref{eq:ASS_solu}, $ \bar c_{\om}^N - \bar c_l \bar \al_1$ is very small. Thus the first term in $I_{11}, I_{21}, I_{31}$ is small. The second term in $I_{11}, I_{21}, I_{31}$ has faster decay rates $r^{2 \bar \al_1}, r^{ 3\bar \al_1}$ and is also very small.

\vs{0.1in}
\paragraph{\bf{Estimate of the velocity approximation}}
From \eqref{eq:uerr_dec1}, since $\bar \e_2$ is supported near $0$, we get $\uu_{loc} = \hat \uu(\e_1)$. For $I_{j2}$ in the above decomposition in the far-field, it remains to estimate 
\beq\label{eq:err_mid_far}
 \bar c_{\om}^{e} \bar \om - \hat \uu(\bar \e_1) \cdot \na \bar \om
 ,\    2 \bar c_{\om}^{e} \bar \th_x - \hat \uu_x(\bar \e_1)\cdot  \na \bar \th -  \hat \uu(\bar \e_1) \cdot \na \bar \th_x , \ 
 2 \bar c_{\om}^{e} \bar \th_y - \hat \uu_y(\bar \e_1)  \cdot \na \bar \th - \hat \uu(\bar \e_1) \cdot \na \bar \th_x.
\eeq
Note that $c_{\om}(\bar \e_1) = c_{\om}(\bar \e)$ \eqref{eq:ux0_ep2} and $c_{\om}(\bar \e) = \bar c_{\om}^e$ in our notation.  For any $a\in \R$, we estimate 
\[
A(f, g) = a  g - \hat \uu(f) \cdot \na g , \quad  
B_i(f, g) =  2 a  \pa_i g - \hat \uu(f) \cdot \na \pa_i g 
- \wh {\pa_i \uu}(f) \cdot \na g , \ i = 1,2.
\]
for $|x|_{\inf} \geq R_1$. From Sections 4.3.2--4.3.3 in Part I \cite{ChenHou2023a}, for $|x|_{\inf} \geq R_1$, $\hat \uu, \wh{ \na \uu }$ reduce to 
\[
\bal
& \hat u(f) = x_1 I_{far}(f), \quad  \hat v(f) = - x_2 I_{far}(f) , 
\quad  \wh {\pa_1 u(f)} = I_{far}(f), \quad \wh {\pa_2 v(f)} = - I_{far}(f), \\
&  \wh {\pa_2 u(f)} = \wh{ \pa_1 v(f)} = 0,  \quad 
 I_{far}(f) \teq  - \f{4}{\pi}\int_{ \max(y_1, y_2) \geq R_n} \f{y_1 y_2}{|y|^4} \om(y) dy, 
\eal
\]
where $R_n = 1024 \cdot 64 h_x$ is the largest threshold. Denote $b = I_{far}(f)$. A direct calculation yields 
\beq\label{eq:err_mid_far2}
\bal
&A(f, g)  = (a  - b) g + b (g - x_1 \pa_1 g + x_2 \pa_2 g) ,  \\
& B_1(f, g)  = 2 a \pa_1 g - b \pa_1 g - b x_1 \pa_{11} g + b x_2 \pa_{12} g 
= (2 a - 2b) \pa_1 g +  b ( \pa_1 g -  x_1 \pa_{11} g +  x_2 \pa_{12} g ) , \\
& B_2(f, g) = 2 a \pa_2 g + b \pa_2 g - b x_1 \pa_{12} g + b x_2 \pa_{22} g 
= (2 a - 2b) \pa_2 g +  b ( 3 \pa_1 g -  x_1 \pa_{11} g +  x_2 \pa_{12} g ) .\\
\eal
\eeq
Therefore, we only need to bound the functions following Section \ref{app:piece_pol}, e.g. $g - x_1 \pa_1  g + x_2 \pa_2 g $ and $g$, and the functional $b(f)$ and $a$. We apply these estimates for \eqref{eq:err_mid_far} with $a = \bar c_{\om}^{e}, f = \bar \e_1, g = \bar \om, \bar \th$.

 \section{Estimate of explicit functions}\label{app:explcit}

In this section, we estimate the derivatives of several explicit or semi-explicit functions using induction, including several cutoff functions used in the estimates and the weight in the stream function \eqref{eq:ASS_stream}.

\subsection{Estimate of the radial functions}\label{sec:est_radial}

\subsubsection{Estimate of the cutoff function}

We estimate the derivatives of the cutoff function 
\beq\label{eq:cutoff_exp}
\chi_e(x) = \B( 1 + \exp( \f{1}{x} + \f{1}{x-1} ) \B)^{-1}, 
\eeq
where $e$ is short for \textit{exponential}. In our verification, it involves high order derivatives of $\chi_e$. Although $\chi_e$ is explicit, its formula is complicated and is difficult to estimate. Instead, we use the structure of $\pa_x^i \chi_e$ and induction to estimate $\pa_x^i \chi_e$. Denote 
\[
p(x) = \f{1}{x} + \f{1}{x-1},  \quad f = \f{1}{1 + x},  \quad \chi_e = f(e^p).
\]

Firstly, we use induction to derive 
\[
d_x^k \chi_e = \sum_{i=1}^k (\pa^i f)(e^p) e^{ip} Q_{k, i}(x),
 \]
 where $Q_{k,i} = 0$ for $i >k, i <0$. A direct calculation yields 
 \[
 \bal
\pa \sum_{i=1}^k \pa^i f e^{ip} Q_{k, i}(x)
&= \sum_{i=1}^k (\pa^{i+1} f)(e^p) \cdot p^{\pr} e^{p} e^{ip} Q_{k, i}
+ (\pa^i f) \pa_x( e^{ip} Q_{k,i}) \\
& = \sum_{i=1}^k (\pa^{i+1} f)(e^p) \cdot e^{(i+1)p} p^{\pr} Q_{k, i}
+ (\pa^i f) e^{ip} ( ip^{\pr} Q_{k,i} +  Q^{\pr}_{k,i}). \\
\eal
 \]
 Comparing the above two equations, we derive
 \[
Q_{k+1, i } = p^{\pr} Q_{k, i-1} + i p^{\pr} Q_{k, i} + Q^{\pr}_{k, i} .
 \]
 The first few terms in $Q_{k,i}$ are given by 
 \[
Q_{0, 0} = 1, \quad Q_{1, 1} = p^{\pr}, \quad Q_{1, 0} = 0.
 \]

It is not difficult to see that $Q_{k,i}$ is a polynomial of $ \pa_x^j p, j \leq k$ with non-negative coefficients. We derive the expression of $Q_{k, i}$ in terms of $ \pa_x^j p, j \leq k$ symbolically. Thus, using triangle inequality, we only need to bound $\pa_x^j p$. We have 
\[
|\pa_x^n p(x)| = n! | x^{-n-1} + (x-1)^{-n-1} |
\leq n! (|z|^{-n-1} + 2^{n+1} ) , \quad z = \min(|x|, |1-x|).
\]
If $n$ is even, $x^{-n-1}$ and $(x-1)^{-n-1}$ have different sign, and we get better estimate
\[
|\pa_x^n p(x) | \leq  n! \max( |x|^{-n-1}, |x-1|^{-n-1}) = n! \cdot z^{-n-1}.
\]

Substituting the above bounds into the formula of $Q_{k, i}$, we can obtain the upper bound $Q^u_{k, i}(x)$ for $Q_{k,i}(x)$, which is a polynomial of $z^{-1}$ with positive coefficient. Since each term in $Q_{k,i}$ is given by $c_{i_1, i_2, .., i_m} \prod_{j=1}^m \pa_x^{i_j} p $ with $\sum i_j = k$, the above estimate implies 
\[
|c_{i_1, i_2, .., i_m} \prod_{j=1}^m \pa_x^{i_j} p |
\leq c_{i_1, i_2, .., i_m} \prod_{j=1}^m  i_j ! (|z|^{-i_j - 1} + 2^{i_j + 1}).
\]
Since $m \leq k$, the highest order of $z^{-1}$ in the upper bound is bounded by $2k$. Thus, we obtain that $Q^u_{k,i}$ is a polynomial in $z^{-1}$ with $\deg Q_{k,i}^u \leq 2k$. Next, we bound 
\[
|e^{i p} Q_{k,i}| \leq e^{i p} Q_{k,i}^u.
\]
For $k \leq 20, x \geq 1 - \f{1}{2k} \geq \f{1}{2}, z^{-1} = |x-1|^{-1} \geq 2 k$, a direct calculation implies that $e^{i p(x)} Q_{k,i}^u(x)$ is decreasing. In fact, for $l \leq 2 k $, we have $z = |x-1| = 1-x$ and 
\[
\bal
 &\pa_x ( \exp( i p(x) ) (1-x)^{-l} )
  = \exp( ip(x))( i p^{\pr} (1-x)^{-l} + l (1-x)^{-l-1}) \\
  = & \exp( ip(x)) \B( - \f{i}{x^2}  - \f{i}{(x-1)^2}  + l (1-x)^{-1} \B) (1-x)^{-l} \leq 0.
 \eal
\]
In the last inequality, we have used $ - \f{i}{1-x} + l \leq - 2k i + 2 k \leq 0 $. 

Note that $|(\pa_x^i f)(e^p)| = i!| (1+ e^p)^{-i-1}|\leq i!$.  Thus, for $x \in [x_l, x_u]$ with $x_l$ close to $1$, we get 
\[
|\pa_x^k \chi_e(x)| 
\leq \sum_{i=1}^k | (\pa^i f)(e^p) | e^{ i p(x)}   Q_{k,i}^u(x)
\leq \sum_{i=1}^k i! \f{e^{ip(x)} }{ ( 1+ e^p)^{i+1}}   Q^u_{k,i}(x)
\leq \sum_{i=1}^k  i!  e^{ i p(x_l)}   Q_{k,i}^u(x_l) .
\]

For $x$ away from $1$, we use monotonicities of $p, Q^u$ and the above estimate to estimate  piecewise bounds of $\pa_x^k \chi_e(x)$. Using the above derivatives bound, the symbolic formula of $\pa_x^k \chi_e$, and the refined second order estimate in Section \ref{app:piece_pol_1D}, we can obtain sharp bounds for $\pa_x^k \chi_e$. Remark that we only apply the above estimate to $k \leq 15$.

\subsubsection{Estimate of polynomial decay functions}\label{sec:cutoff_rati}

For cutoff function $\chi_e( \f{|x|-a}{b})$ based on the exponential cutoff function \eqref{eq:cutoff_exp}, it has rapid change from $|x| \leq a$ to $|x| \geq a+b$, which is not very smooth in the computational domain if there are not enough mesh for $x$ with $a \leq |x| \leq b$. We apply these cutoff functions to the far-field, e.g. $|x| \geq 10$, where the mesh is relatively sparse. Thus, we need another function similar to a cutoff function that has a slower change than the exponential cutoff function. We consider 
\beq\label{eq:cutoff_rati}
\chi(x) = \f{ x^7}{  (1 + x^2)^{7/2}}, \quad x \in \R_+.
\eeq
and will use its rescaled version, e.g., $\chi( \f{x-a}{b})$, in our verification. 

Firstly, we use induction to derive 
\[
\pa_x^{k} \chi = \f{p_k(x)}{ (1+x^2)^{ 7/ 2+ k}}, \quad p_0 = x^7.
\]
where $p_k(x)$ is a polynomial. A direct calculation yields 
\[
\pa_x^{k+1} \chi = \f{ p_k^{\pr}(x) (1 + x^2) - ( \f{7}{2} + k) \cdot 2 x p_k(x)}{ (1 + x^2)^{7/2 + k + 1}}. 
\]

Comparing the above two formulas, we yield 
\[
p_{k+1} = p_k^{\pr} (1 + x^2) - (7 + 2k) x p_k(x). 
\]

The first few terms are given by $p_0 = x^7,  p_1 = 7 x^6$. Using the recursive formula and $\deg p_1 = 6$, we yield 
\beq\label{eq:cutoff_rati_deg}
\deg p_{k+1} \leq \deg p_k + 1, \quad \deg p_k \leq k+5,  \quad k \geq 1.
\eeq

Since $p_k$ is a polynomial, the above recursive formula shows that $p_{k+1}$ is also a polynomial.

To estimate $\pa_x^k \chi$, we decompose $p_k $ into the positive and the negative parts. Suppose 
that $p_k = \sum_i a_i x^i$. We have 
\[
p_k = p_k^+ - p_k^-, \quad p_k^+ = \sum a_i^+ x^i , \quad p_k^- = \sum a_i^- x^i.
\]

For $x \geq 0$, $p_k^+,p_k^-$ are increasing. Thus, for $ x \in [x_l, x_u]$, we get 
\[
 |\pa_x^k \chi| \leq  \f{ \max( p_k^+(x_u) - p_k^-(x_l), p_k^-(x_u) - p_k^+(x_l) )  }{ (1 + x_l^2)^{7/2 + k} }.
\]

Next, we estimate $\pa_x^k \chi$ for large $x $. For $x \geq 2, k \geq 1$  and any polynomial $q(x)$ with non-negative coefficients and $\deg q \leq k + 5$, we yield 
\[
x q^{\pr} \leq (k+5) q, \quad  \f{ q^{\pr}(1 + x^2)}{ (7 + 2k) x q}
\leq \f{(1 + x^2) (k+5)}{ (7 + 2k) x^2 }
\leq \f{ 5 (k+5)}{4 (7+ 2k)} < 1.
\]

The first inequality follows by comparing the coefficients of $x q^{\pr}$ and $(k+5) q$, which are nonnegative. It follows 
\[
\pa_x \f{q}{(1+x^2)^{7/2 + k}} = \f{ q^{\pr}(1 + x^2) - (7/2 + k) 2x q}{ (1+x^2)^{7/2 + k+1}} \leq 0, \quad k \geq 1, \ x \geq 2.
\]
Thus $\f{q}{(1+x^2)^{7/2 + k}}$ is decreasing. For $k \geq 1$ and $x \geq x_l \geq 2$, using \eqref{eq:cutoff_rati_deg} and the monotonicity, we yield 
\[
|\pa_x^k(x)| \leq \f{ p_k^+(x) + p_k^-(x)}{ (1+x^2)^{7/2 + k}}
\leq \f{ p_k^+(x_l) + p_k^-(x_l)}{ (1+x_l^2)^{7/2 + k}}
\]

For $k= 0$, the estimate is trivial: $\chi(x) \leq 1$. Using these higher order derivative bounds, we can use the discrete values of $\pa_x^k \chi$ and the bound for $\pa_x^{k+2} \chi$ to obtain sharp bounds of $\pa_x^k \chi$. 

Note that $\chi_1(x-a) = \f{ (x-a)_+^7}{ (1 + (x-a)^2)^{7/2} }$ is only $C^{6, 1}$. Suppose that $a \in [x_l, x_u]$. Since $\chi_1$ is smooth on $x \leq a$ and on $ x \geq a$, we can still use first order estimate to estimate $\pa_x^k \chi_1$ as follows 
\[
|\pa_x^k \chi_1(x)| \leq \max_{\al \in \{l, u \} }  |\pa_x^k \chi_1(x_{\al})| +
\max( || \pa_x^{k+1} \chi_1||_{L^{\inf]}[x_l, a]}
|| \pa_x^{k+1} \chi_1||_{L^{\inf]}[ a, x_u]}
)  |x_u - x_l| .
\]

\subsubsection{Radial cutoff function}
Now, we construct the radial cutoff functions for the far-field approximation terms of $\om$ and $\phi$ as follows 
\beq\label{eq:cut_radial}
\bal
 \chi(r) & = \chi_1 (1 - \chi_2) + \chi_2, \quad \chi_1(r) = \chi_{rati}( \f{r-a_1}{l_1^{1/2}} ), 
\quad \chi_2(r) = \chi_{exp}(  \f{r -a_2}{9a_2} ) , \\
  a_1 & = 10, \quad l_1= 50000, \quad a_2 = 10^5,
\eal
\eeq
where $\chi_{exp}$ and $\chi_{rati}$ are defined in \eqref{eq:cutoff_exp} and \eqref{eq:cutoff_rati}, respectively. Using the estimates of $\chi_{rati}, \chi_{exp}$ established in the last two sections, the Leibniz rule \eqref{eq:lei}, and \eqref{est_2nd}, 
we can evaluate $\chi$ on the grid points and estimate its derivative bounds.

\subsection{Cutoff function near the origin}\label{app:cutoff_near0} 

For the cutoff function $\kp(x)$ used in Section \ref{sec:lin_evo}, we choose it as follows
\beq\label{eq:cutoff_near0}
\kp(x;  a , b ) = \kp_1(\f{x}{ a } ) (1 - \chi_e( \f{x}{ b} )  ) , \quad  \kp_1(x) = \f{1}{1 + x^4},  \quad
 \kp_*(x)= \kp(x ; \f{1}{3}, \f{3}{2} ) ,
\eeq
where $\chi_e$ is the cutoff function chosen in \eqref{eq:cutoff_exp}. We mostly use the cutoff $\kp_*$. Since $\chi_e(y) = 1$ for $y \geq 1$ and $\chi_e(y) = 0$ for $ y \leq 0 $. The above cutoff function is supported in $x \leq a_2$. Using Taylor expansion, we have the following properties for $\kp$
\[
\kp_1(x / a_1) = 1 + O(x^4), \quad  \kp(x) = 1 + O(x^4).
\]

For the cutoff functions $\chi_{NF}$ in Section 4.2.1 in Part I \cite{ChenHou2023a}, $\chi_{\bar \e}$ in \eqref{eq:uerr_dec1}, and $\chi_{\hat \e}$ in \eqref{eq:uerr_hat1}, we choose 
\beq\label{eq:cutoff_near0_all}
\bal
& \chi_{\bar \e}(x, y) = \kp(x; \nu_{\bar \e, 1}, \nu_{\bar \e, 2}) 
\kp(y; \nu_{\bar \e, 1}, \nu_{\bar \e, 2}) ,  \quad \nu_{\bar \e, 1 } = 1/192, \quad \nu_{\bar \e,  2} = 3/2 , \\ 
&  
 \chi_{ \hat \e}(x, y) = \kp_*(x) \kp_*(y) ,\quad  \chi_{NF}(x, y) = \kp(x; 2,10)  \kp(y ; 2, 10 ),  \\
 & f_{\chi, 1}  = \D( \f{x y^3}{6} \chi_{NF}(x, y)  ) ,
\quad f_{\chi, 2} = x y \chi_{NF}(x, y), \quad f_{\chi, 3} = \f{x^2}{2} \chi_{NF}(x, y), \\
\eal
\eeq

For the cutoff function in the stream function \eqref{eq:ASS_solu},  we choose 
\beq\label{eq:cutoff_psi_near0}
 \chi_{\phi } = \kp_2( \f{x} { \nu_{4,1}} ) (1 - \chi_e( \f{x}{ \nu_{4, 2} } )  ), \quad 
\kp_2(x) = \f{1}{1 + x^2}, \quad \nu_{4, 1} = 2,  \quad \nu_{4, 2} = 128.
\eeq

For $\kp_1(x), \kp_2(x)$, we use induction to obtain 
\[
\pa_x^k \kp_1(x) = \f{ P_k^+(x) - P_k^-(x)}{ (1 + x^4)^{k+1}}, 
\quad \pa_x^k \kp_2(x) = \f{ R_k^+(x) - R_k^-(x)}{ (1 + x^2)^{k+1}}, 
\]
for some polynomials $P_k^{\pm}, R_k^{\pm}$ with non-negative coefficients, and the same method as that in Section \ref{sec:cutoff_rati} to estimate the derivatives of $\pa_x^i \kp_1(x)$. The estimate of $\kp_1$ is simpler since $\kp_1$ has a simpler form. Using the Leibniz rule 
\eqref{eq:lei} and the triangle inequality, we can obtain estimate $\pa_x^l \kp_1(x)$ in $[a, b]$. Then we use these derivative estimates for $\pa_x^{l+2}   \kp_1(x)$, evaluate  $\kp(x; a_1, a_2)$ on the grid points, and then use \eqref{est_2nd} to obtain a sharp estimate of $\pa_x^l \kp_1(x)$ on $[a, b]$. The same method applies to estimate $\kp_2, \chi_{\phi}$.

For large $x$, e.g. $x \geq 100$, the above estimates can lead to a very large round off error. Instead, for $a \geq 2, a \in \BZ_+$, we use the Taylor expansion 
\[
F_a =   \f{1}{1 + x^a} = \sum_{k \geq 0} (-1)^k x^{- a(k+1) },
  \quad \pa_x^i  F_a = \sum_{k \geq 0} (-1)^{k+i} C_{i, k} x^{- a(k+1)  - i}, 
  \quad C_{i, k} = \prod_{0\leq j \leq i-1} ( a(k+1) + j) . 
\]
We want to bound $| \pa_x^i F_a| \leq C_{i, 0}(1 + C_{\e}) x^{- a -i} $ for $x \geq x_l = 100, i \leq 20$. For $k\leq 20$, we bound 
\[
C_{i, k} x^{ - a (k+1) -i} \leq C_{i, k} x_l^{ -(a-1) k} x^{ - a - i - k}
\leq C_{i, 0} \e_1 ^{ - a - i - k}, \quad \e_1 \teq \max_{ i \leq 20, k \leq 20} x_l^{-(a-1) k } C_{i, k} C_{i, 0}^{-1} .
\]
For the tail part $k > 20$, we consider $ G(k) = k \log  x - i \log(1 + k)$. Since $x > 21, i \leq 20$, we get
\[
\pa_k G = \log x - \f{i}{ 1 + k} \geq  \log x - 1 > \log 4 - 1 >0, \quad G(k) \geq G(21)
 = 21 \log x - i \log 21 > 0.
\]
It follows $x^k > (1+k)^i$. Using  $ \f{ a(k+1) + j}{ a + j} \leq 1 + k, C_{i, k} \leq C_{i, 0} (1 + k)^i$, and $a \geq 2$, we further get 
\[
C_{i, k} x^{- a(k+1) -i} 
\leq   x^{- k - a - i} C_{i, k} x^{- k}
\leq  x^{- k - a - i} C_{i, 0} (1 + k)^i x^{-k} \leq C_{i, 0} x^{- k - a - i}, \ k > 20.
\]
Combining the above estimates and $x \geq x_l> 10$, we obtain 
\[
\bal
 |\pa_x^i F_a | \leq C_{i, 0} x^{-a-i} C_a, 
 \quad  C_a \leq 1 + \e_1 \sum_{k=1}^{20} x^{-k} + \sum_{k \geq 21} x^{- k} 
 \leq 1 +  \f{ \e_1 x^{-1} }{1 - x^{-1}} 
 + \f{ x^{-21} }{1 - x^{-1}}
\leq 1 +  \f{\e_1}{x_l - 1} + x_l^{-20}.
\eal
\]

\subsection{Estimate of $\rho_p(y)$}\label{sec:stream_wg}

We estimate the weight $\rho_p(y)$ \eqref{eq:psi_wg} in the representation of the stream function. Using symbolic computation, e.g., Matlab or Mathematica, we yield 
\[
\bal
&\pa_x^9 \rho_p(y) = \f{  f_2(y) - f_1(y)}{ ( g(y))^9 }, \quad g(y) = 2 + 2 y + y^2, \\
& f_1 = 288 y^2 + 672 y^3 + 504 y^4, \quad f_2 = 16 + 168 y^6 + 72 y^7+ 9 y^8 .
\eal
\]
Since $f_1, f_2, g \geq 0$ are increasing in $y \geq 0$, for $y \in [y_l, y_u]$, we yield 
\[
|\pa_x^9 \rho_p(y)| \leq  \f{ \max (f_2(y_u) - f_1(y_l), f_1(y_u) - f_2(y_l) }{ ( g(y_l))^9 }.
\]

We have a trivial estimate similar to \eqref{est_2nd}
\beq\label{est_1st}
 \max_{ x \in I }|f(x)|  \leq \max( |f(x_l) | , |f(x_u)| ) +  \f{h}{2} || f_x||_{L^{\inf} (I) },
\eeq
which is useful if we do not have bound for $f_{xx}$.

Based on the above estimates, using the estimates \eqref{est_2nd}, \eqref{est_1st}, ideas in Section \ref{app:piece_pol_1D}, and evaluating $\rho_p$ on some grid points, we can obtain piecewise sharp bounds for $\pa_x^i \rho_p$ for $i \leq 8$.

\section{Piecewise $C^{1/2}$ and Lipschitz estimates}\label{app:linf_est}

In this section, we estimate the piecewise $C^{1/2}$ bound and Lipschitz bound for a function.

\subsection{H\"older estimate of the functions}\label{app:hol_func1}

In the following two sections, we estimate the H\"older seminorms $ [f ]_{C_x^{1/2}}$ or $ [f ]_{C_y^{1/2}}$ of some function $f$, e.g. $f = (\pa_t - \cL) \wh W$ in \eqref{eq:lin_evo_err1}, based on the previous $L^{\inf}$ estimates. We will develop two approaches. 

Below, we will assume $x, y \in \R_2^{++}$ since our function $f(x)$ defined on $x \in \R_2^+ (x_2 \geq 0)$ is either even or odd in $x_1$ and we can reduce essentially all estimates to the case of $\R^2_{++}$ using symmetry. Suppose that we have bounds for $ \pa_x f , \pa_y f  $ and $f $. 
Firstly, we consider the $C_x^{1/2}$ estimate. For $x_1 < y_1$ and $x_2 = y_2$,  we have \[
I = \f{ | f(x) - f(y) |}{ |x-y|^{1/2}}
\leq |x-y|^{1/2} \f{1}{|x-y|} \int_{x_1}^{y_1} | f_x(z_1, x_2) | d z_1.
\]
We further bound the average of $f_x$ piecewisely using the method in Appendix \ref{app:piece_deri} to obtain the first estimate. We have a second estimate 
\[
\bal
|I| & = \B| \int_{x_1}^{y_1} f_x( z_1,  x_2 ) dz \B| \cdot \f{1}{ |x-y|^{1/2}}
\leq || f_x  x^{1/2} ||_{\inf} \int_{x_1}^{y_1} z_1^{-1/2} d z_1  \cdot \f{1}{ |x-y|^{1/2}} \\
& \leq || f_x  x^{1/2} ||_{\inf} 2 \f{ y_1^{1/2} - x_1^{1/2}} { |x-y|^{1/2}}
= || f_x  x^{1/2} ||_{\inf} \f{ 2 \sqrt{ y_1 - x_1}}{  \sqrt{ x_1} + \sqrt{y_1}}.
\eal
\]

We also have a trivial $L^{\inf}$ estimate 
\[
\bal
|I| & \leq  || f  x_1^{-1/2}||_{\inf}  \f{ x_1^{1/2}  + y_1^{1/2} }{ |x - y |^{1/2}} , \quad
|I| \leq  || f||_{\inf} \f{2}{ |x - y|^{1/2}} .  \\
\eal
\]

Similar $L^{\inf}$ and Lipschitz  estimates apply to $ || f||_{C_y^{1/2}}$. 

Near the origin,  optimizing the above estimates, for $x_2 = y_2$, we obtain \[
 \B|\f{ f(x) - f(y)}{ |x- y|^{1/2}}\B|
 \leq \min( || f_x x^{1/2}||_{\inf} 2t, \ || f x_1^{-1/2}||_{\inf} t^{-1}   ), \quad t = \f{ \sqrt{y_1 - x_1}}{ \sqrt{x_1} + \sqrt{y_1}}.
\]

In the $Y-$direction, $x_1 = y_1, x_2 \leq y_2$, we use 
\[
\bal
 I_Y &= \B|\f{ f(x) - f(y)}{ |x- y|^{1/2}}\B|
 \leq \f{1}{ |x_2 - y_2|^{1/2}} \int_{x_2}^{y_2}  |f_y(x_1, z_2)| |z|^{1/2} \cdot |z|^{-1/2} d z_2
 \leq || f_y |x|^{1/2}||_{\inf}  \f{|x_2-y_2|^{1/2}}{ |x|^{1/2}} \teq A t, \\
  I_Y &\leq (|f(x) x_1| + |f(y) x_1 |^{1/2} ) (\f{x_1}{|x|})^{1/2} 
\cdot  \f{ |x|^{1/2}} { |x_2 - y_2|^{1/2}} \teq B t^{-1},
\quad t \teq \f{  |x_2 - y_2|^{1/2}}{|x|^{1/2}} , \ I_Y \leq \min(A t, B t^{-1} )
 \eal
\]
Since $x_1 \leq |x|$, $A, B$ are not singular near $x=0$. We derive the piecewise bounds for $A, B$ and then optimize two estimates to estimate $I_Y$. 

From the above estimates, to obtain sharp H\"older estimate of $f$, we estimate the piecewise bounds of $f, f x_1^{-1/2}, f |x|^{-1/2}, f_x, f_y$, $ f_x |x_1|^{1/2}, f_y |x|^{1/2}$, which are local quantities. These estimates can be established using the piecewise bounds of $\pa_x^i \pa_y^j f$ and the methods in Section {\applinfestsupp} in the supplementary material II \cite{ChenHou2023bSupp}.

\subsubsection{The second approach of H\"older estimate}\label{app:hol_func2}

We develop an additional approach to estimate  $I(f) = \f{ |f(x) - f(z) |}{ |x-z|^{1/2}}$ that is sharper if $|x-z|$ is not small and $f$ is smooth. We 
need the grid point values and derivative bounds of $f$. 

We estimate $I(f) = \f{ |f(x) - f(z) |}{ |x-z|^{1/2}}$ for $x \in [x_l, x_u], z \in [z_l, z_u]$. Denote by $\hat f$ the linear approximation of $f$ with $\hat f(x_i) =f(x_i)$ on the grid point $x_i$. We have the following Lemma.

\begin{lem}\label{lem:lin_interp_hol}
Suppose that $ f$ is linear on $[x_l, x_u], [z_l, z_u]$ and $x_l \leq x_u \leq z_l \leq z_u$. Then we have
\[
\max_{ x \in [x_l, x_u], z \in [z_l, z_u]}  \f{ | f(x) -  f(z)| }{ |x-z|^{1/2}}
= \max_{\al, \b \in \{l, u\} } \f{ |  f(x_{\al}) -  f(z_{\b})| }{ | x_{\al} - z_{\b}|^{1/2}} .
\]
\end{lem}

The above Lemma shows that for the linear interpolation of $f$, the maximum of the Holder norm is achieved at the grid point.

\begin{proof}
Denote by $M$ the right hand side in the Lemma. Clearly, it suffices to prove that the left hand side is bounded by $M$. We fix $x \in [x_l, x_u], z \in [z_l, z_u]$. Suppose that 
\[
x = a_l x_l + a_u x_u, \quad z = b_l z_l + b_u z_u,  a_u + a_l =1, \quad b_l + b_u = 1,
\]
for $a_l , b_l \in [0, 1]$. Denote 
\[
m_{\al \b} = a_{\al} b_{\b}, \quad \al, \b \in \{ l, u \}.
\]

Since $ f(x)$ is linear on $[x_l, x_u]$ and $[z_l, z_u]$, we get 
\[
f(x) = a_l f(x_l) + a_u f(x_u), \quad f(z) = b_l f(z_l) + b_u f(z_u). 
\]

For any function $g$ linear on $[x_l, x_u], [z_l, z_u]$, e.g., $g(x) = 1, g(x) = x, g(x) = f(x)$, we have 
\beq\label{eq:lin_interp_hol_pf1}
g(x) = \sum_{\al, \b \in \{ l, u \} } m_{\al \b} g(x_{\al}) ,\quad
g(z) - g(x) = \sum_{\al, \b \in \{ l, u \} } m_{\al \b} ( g(z_{\b} ) - g(x_{\al}) ),
\quad 
\eeq

Using the above identities and the triangle inequality and the definition of $M$, we yield
\[
|f(x) - f(z)| = \B| \sum_{\al, \b \in \{ l, u \} } m_{\al \b} (f(x_{\al}) - f(z_{\b})) \B|
\leq \sum_{\al, \b \in \{ l, u \} } m_{\al \b} M |x_{\al} - z_{\b}|^{1/2}.
\]

Using the Cauchy-Schwarz inequality, $|x_{\al} - z_{\b}| = z_{\b} - x_{\al}$ and \eqref{eq:lin_interp_hol_pf1}, we establish 
\[
\bal
|f(x) - f(z)| 
&\leq  \sum_{\al, \b \in \{ l, u \} }  m_{\al \b}
\sum_{\al, \b \in \{ l, u \} } m_{\al \b} M |x_{\al} - z_{\b}|^{1/2}
= \sum_{\al, \b \in \{ l, u \} } m_{\al \b} M |x_{\al} - z_{\b}|^{1/2}
  \\
&= M \B( \sum_{\al, \b \in \{ l, u \} } m_{\al \b} (z_{\b} - x_{\al}) \B)^{1/2}
= M (z - x)^{1/2}.
\eal
\]
The desired result follows.
\end{proof}

We generalize Lemma \ref{lem:lin_interp_hol} to 2D as follows.
\begin{lem}\label{lem:lin_interp_hol2D}
Let $I_x = [x_l, x_u], I_z = [z_l, z_u], I_y = [y_l, y_u]$ with $x_l \leq x_u \leq z_l \leq z_u$.
Suppose that $ f$ is linear on $I_x \times I_y$ and $I_z \times I_y$. Then we have
\[
\max_{ x \in I_x, z \in I_z, y \in I_y}  \f{ | f(x, y) -  f(z, y)| }{ |x-z|^{1/2}}
= \max_{\al, \b , \g \in \{l, u\} } \f{ |  f(x_{\al}, y_{\g}) -  f(z_{\b}, y_{\g})| }{ | x_{\al} - z_{\b}|^{1/2}} .
\]
\end{lem}

\begin{proof}
Note that the function $I(x,  z, y) = \f{ f(x, y) - f(z, y)}{ |x-z|^{1/2}}$ is linear in $y$. We get 
\[
|I(x, z, y | = \max( | I(x, z, y_l )|, | I(x, z, y_u)|).
\]
Applying Lemma \ref{lem:lin_interp_hol} completes the proof.
\end{proof}

Let $\hat f$ be the linear interpolation of $f$. Suppose that $x \in  I_x, z\in  I_z, y \in I_y$ with $x_u \leq z_l$.  Using the above estimates and notations,  we can bound $I(f)$ as follows 
\[
\bal
I(f)  & =  \f{ |f(z, y) - f(x, y)|}{ |x-z|^{1/2}}
\leq \f{ |\hat f(x, y) - f(x, y)| + |\hat f(z,y) - f(z,y)|}{|x-z|^{1/2}}
+\max_{\al, \b ,\g \in \{l, u\} } \f{ |  f(x_{\al}, y_{\g}) -  f(z_{\b}, y_{\g})| }{ | x_{\al} - z_{\b}|^{1/2}}  \\
& \leq  \B( \f{ h_x^2}{8} || f_{xx}||_{I_x \times I_y} 
+ \f{ h_y^2}{8}  ( || f_{yy}||_{I_x \times I_y} + || f_{yy}||_{I_z \times I_y}   )
+ \f{ h_z^2}{8} || f_{xx}||_{I_z \times I_y} \B) |x-z|^{-1/2} 
+ M.
\eal
\]

\subsection{Piecewise derivative bounds}\label{app:piece_deri}

In this section, we discuss how to obtain the sharp bound of $\f{ p(b) - p(a)}{b-a}$ using piecewise derivative bounds of $p$.

Suppose that $|p^{\pr} (y)| \leq C_i , y \in I_i = [y_i, y_{i+1}]$. For any $a \in I_k, b \in I_l, a < b$,  we have the bound 
\begin{eqnarray*}
\bal
& |p(b) - p(a)| \leq \int_{a}^{b} | p^{\pr}(y) | dy   
\leq | y_{k+1} - a| C_k  + |b - y_l | C_l 
+ \sum_{ k+ 1 \leq m \leq l-1} C_m ( y_{m+1} - y_m) \\
= & ( y_{k+1} - a )  C_k  + ( b - y_l ) C_l  + M_{kl}  ( y_l - y_{k+1} )  \one_{ l\geq k+1} ,\\
\eal
\end{eqnarray*}
where $M_{kl}$ is defined below:
\beq\label{eq:Lip_Mkl}
 M_{kl} =   | y_l - y_{k+1} |^{-1} \B( \sum_{ k+ 1 \leq m \leq l-1} C_m | y_{m+1} - y_m| \B). 
\eeq

Next, we want to bound $ \f{ |p(b) - p(a) |}{ |b-a|}$. If $ l - k \leq 1$, we get 
\[
|p(b) - p(a)| \leq (b - a) \max( C_k, C_l ). 
\]
Otherwise, if $l \geq k+2$, we have 
\[
 | p(b) - p(a) | \leq (y_{k+1} - a) (C_k - M_{kl})
+ (b- y_l) (C_l - M_{kl}) +  M_{kl} (b-a).
\]
Since $\f{y_{k+1} - a}{ b- a}$ is decreasing in $a$ and $b$, $\f{ b-y_l}{b-a}$ is increasing in $b$ and $a$, we get 
\[
 0 \leq \f{y_{k+1} - a}{ b- a} \leq \f{ y_{k+1} - y_k }{ y_l - y_k }, 
\quad  0 \leq \f{ b-y_l}{b-a} \leq \f{y_{l+1} - y_l}{ y_{l+1} - y_{k+1}}.
\]
Using the above estimates, for $a \in I_k, b \in I_l$, we obtain 
\beq\label{eq:Lip_piece}
\f{  |p(b) - p(a) | }{|b-a| } \leq 
\max( C_k - M_{kl}, 0) \f{ y_{k+1} - y_k }{ y_l - y_k } 
+ \max(  C_l - M_{kl},  0) \f{ y_{l+1} - y_l}{ y_{l+1} - y_{k+1}} + M_{kl}.
\eeq

For uniform mesh, i.e. $y_{i+1} - y_i = h$, we can simplify the above estimate as follows 
\[
\bal
\f{  |p(b) - p(a) | }{|b-a| } &\leq 
\f{  ( 
\max( C_k - M_{kl}, 0) + \max(   C_l - M_{kl},  0) )  }{ l-k }
+ M_{kl}, \quad 
M_{kl}  =  \f{1}{l-k-1} \sum_{k+1 \leq m \leq l-1 } C_m.
\eal
\]

The same argument applies to obtain piecewise bounds of $ J(a,b)= \f{p(b) - p(a)}{b-a} $. 
We use piecewise upper bounds $ p^{\prime}(y) \leq C_i, y \in I_i = [y_i, y_{i+1}]$ and obtain the same upper bounds as \eqref{eq:Lip_piece}. To get lower bounds of $J(a,b)$, we use piecewise 
lower bounds $ p^{\prime}(y) \geq C_i$ and \eqref{eq:Lip_Mkl} to get
\[
\f{  p(b) - p(a)  }{b-a } \geq
\min( C_k - M_{kl}, 0) \f{ y_{k+1} - y_k }{ y_l - y_k } 
+ \min(  C_l - M_{kl},  0) \f{ y_{l+1} - y_l}{ y_{l+1} - y_{k+1}} + M_{kl}.
\]

\section{Notations}\label{app:nota}

For the reader's convenience, we collect the main notations used in this paper. 

\vspace{0.1in}
\paragraph{\bf{Weights}}

We use the following weights defined in \eqref{wg:hol}, \eqref{wg:linf}, \eqref{wg:lin_evo} for the estimates
\[
\psi_1, \psi_2, \psi_3, \psi_{du}, \psi_u,  \quad \vp_1 , \vp_{g1}, \vp_{elli}, 
 \vp_{evo, i}, i=1,2,3, \quad  \rho_{10}, \rho_{20}, \rho_3, \rho_4 .
\]

We use $f_{\lam}(x) = f(\lam x)$ for rescaled function \eqref{eq:flam}. 

\vspace{0.1in}
\paragraph{\bf{Cutoff functions}}
We use various cutoff functions to construct the approximate solutions. 

$\chi_{ij}, i=1,2, 3, j = 1,2$ are defined in \eqref{eq:solu_cor}. 

$ \chi_{\bar \e}, \chi_{\hat \e}, f_{\chi, i}, i=1,2,3$ are defined in \eqref{eq:cutoff_near0_all}, \eqref{eq:cutoff_psi_near0}.

\vspace{0.in}
\paragraph{\bf{Operators}}

We use $\cL_{\cdot}$ to denote various linear operators. $\cL_i$ is the full linearized operator around the approximate steady state. We decompose $\cL_i$ into $\cL_i^e, \cL_i^{\bar e}, \cL_i^N$ \eqref{eq:decomp_L}.

$\cB_{op,i}$ \eqref{eq:Blin_gen}, \eqref{eq:Blin} denotes bilinear operators related to the linearized operators. 

$\cR_{\cdot}$ \eqref{eq:lin_evo_err_def}, \eqref{eq:lin_evo_main2} denotes residual error in the construction of the approximate solution to the linearized equations.

\vspace{0.1in}
\paragraph{\bf{Velocity and kernels}}

We use $K_i$ to denote the kernels of the velocity, e.g. $K_1, K_2, K_f, f = u, v, u_x, v_x, u_y$ \eqref{eq:kernel_du}. We use $K^{sym}$ for the symmetrized kernel \eqref{int:ker_sym}, $K_{ux0}, K_{00}$ \eqref{eq:u_appr_near0_coe} for the kernel of the approximation terms near $x=0$

We use $f = u, v, u_x, u_y, v_x, v_y$ to denote the original velocity and its derivatives, $\hat f$ for its finite rank approximation, and $f_A = f - \hat f$. See the beginning of Section \ref{sec:int_method}.

\vspace{0.1in}
\paragraph{\bf{Regions for integrals}}

We use $B_{lm}(r)$ \eqref{int:box_B} to denote different grids and $R_{\cdot}(\cdot)$ to denote various singular region: $R(x, k)$ \eqref{eq:rect_Rk}, $R_s(x, k), R_{s, i}(x, k)$ \eqref{eq:rect_Rsk}, $R^{\pm}(x, k)$ \eqref{eq:rect_Rk+}, $R(x, k, \al), \al = N, E, S, W$ \eqref{eq:rect_NE}

\vspace{0.1in}
\paragraph{\bf{Approximate profiles and solutions}}
We use $\om, \eta, \xi, \phi$ to denote the vorticity, $\th_x, \th_y$ ($\th$ is the density \eqref{eq:bous1}), and the stream functions, respectively. We use $\bar f$ to denote the approximate profile for $f$, e.g. $\bar \om, \bar \th$, and use $\hat f$ to denote the numeric solution, e.g. $\wh W$ \eqref{eq:solu_full} and $\hat G$ \eqref{eq:lin_evo_main1}. 

We use $\bar F_{\om}, \bar F_{\th}, \bar \cF_{i}$ \eqref{eq:bous_err} to denote the residual error of the profile.

\vspace{0.1in}
\paragraph{\bf{Mesh}}

To construct 
the approximate profile, we use the adaptive mesh $y_i$ \eqref{eq:ASS_mesh}. To estimate the integrals $\int f(x, y) d y$ 
in Section \ref{sec:vel_comp}, we use mesh $y_i$ \eqref{eq:int_mesh_y} with mesh size $h_x, h$ \eqref{int:para1}.

\vspace{0.1in}
\paragraph{\bf{Differential operators}} 
We denote \eqref{eq:diff_op} $D^2 = (D^2_1, D^2_2, D^2_3) = (\pa_{xy}, \pa_{xy}, \pa_x^2)^T$.

\vspace{0.1in}
{\bf Acknowledgments.} The research was in part supported by NSF Grants DMS-1907977 and DMS-2205590. We would like to acknowledge the generous support from Mr. K. C. Choi through the Choi Family Gift Fund and the Choi Family Postdoc Gift Fund. We are grateful to Drs. Pengfei Liu and De Huang for a number of stimulating discussions in the early stage of this project. JC is grateful to Mr. Xiaoqi Chen for several suggestions on coding and the use of High Performance Computing. Part of the computation in this paper was performed using the Caltech IMSS High Performance Computing. The support from its staff is greatly appreciated.

\bibliographystyle{plain}
\bibliography{selfsimilar}

\end{document}


\begin{abstract}

In this supplementary material, we provide all the remaining details in the estimate of integrals related to the nonlocal terms $\uu, \na \uu$. In Section \ref{sec:hol_C}, we estimate the integrals in the upper bounds of the sharp H\"older estimate established in Section {\secsharp} in  \cite{ChenHou2023a_supp}. In Section \ref{sec:linf_sing}, we provide detailed formulas in the estimate of $\na \uu, u / |x_1|$ in the singular regions and some explicit integrals. In Section \ref{sec:u_1D}, we generalize the estimate of integrals of $\uu(x), \na \uu(x)$ in Section {\secvelcomp} in Part II \cite{ChenHou2023b_supp} from $x = O(1)$ to $x$ either close to $0$ or $x$ very large. In Section \ref{app:linf_est}, we use grid points and piecewise derivative bounds estimated in Appendix in Part II  \cite{ChenHou2023b_supp} to construct three interpolating polynomials in 2D with rigorous error estimates. Then we use these interpolating polynomials to obtain sharp piecewise estimates of functions, e.g. the residual error.

\end{abstract}

 \maketitle
  \setcounter{section}{4}

 \setcounter{tocdepth}{1}
  \hypersetup{bookmarksdepth=2}
\tableofcontents

\section{Computation of the sharp constants for in the H\"older estimates}\label{sec:hol_C}

In this section, we estimate the integrals in the upper bounds of the sharp H\"older estimate established in Section {\secsharp} in \cite{ChenHou2023a_supp}.

\subsection{Integrals related to the kernels}\label{sec:int_ker}

Firstly, we derive some analytic integral formulas. The kernels associated with $\na u$ are given by
\beq\label{eq:kernel_du}
 G(y) = - \f{1}{2} \log |y|, \quad 
 K_1(y) \teq  G_{xy} = \f{y_1 y_2}{|y|^4}, \quad K_2(y) \teq  G_{xx} = -G_{yy} = \f{1}{2}\f{ y_1^2 - y_2^2}{ |y|^4} , 
\quad
\eeq
and we drop the factor $\f{1}{\pi}$ for simplicity. 

\subsubsection{Basic Lemmas}

We use the follwing Lemma from Appendix {\secholxvx} in Part I \cite{ChenHou2023a_supp} to estimate $K_2$.

\begin{lem}\label{lem:pv}
Suppose that $f \in L^{\inf}$, is H\"older continuous near $0$. For $0 < a, b < \infty$ and $Q = [0, a] \times [0, b], [0, a] \times [-b, 0], [-a, 0] \times [0, b]$, or $[-a, 0] \times [-b, 0]$, we have
\[
\bal
P.V. \int_Q K_2(y) f(y) dy 
= \lim_{\e \to 0} \int_{ Q \cap |y_1| \geq \e} K_2(y) f(y) dy 
- \f{\pi}{8} f(0) 
= \lim_{\e \to 0} \int_{ Q \cap |y_2| \geq \e} K_2(y) f(y) dy  + \f{\pi}{8} f(0).
\eal
\]
\end{lem}

We have the following estimate for the Green function from Appendix {\appkersym} in Part II \cite{ChenHou2023b_supp}.

\begin{lem}\label{lem:K_mix}
Denote $r = (x^2 + y^2)^{\f{1}{2} }$ and $G(x, y) = - \f{1}{2}\log r  $. 
For any $i, j \geq 0$ with $i+ j \geq 1$, we have 
\[
 | \pa_x^i \pa_y^j G(x, y)| \leq  \f{1}{2}(i+j - 1)! \cdot r^{- i - j}.
\]
\end{lem}

Consider the odd extension of $\om$ in $y$ from $\R^2_+$ to $\R^2$
\beq\label{eq:ext_w_odd}
W(y) = \om(y)  \ \mathrm{ for } \ y_2 \geq 0, \quad W(y) = - \om(y_1, -y_2) \ \mathrm{ for } \ y_2 < 0. 
\eeq
$W$ is odd in both $y_1$ and $y_2$ variables.

We have the following indefinite integral formulas for $K_1, K_2$ \eqref{eq:kernel_du}
\beq\label{feq:int_ker}
\bal
& \int K_1(y) d y_1 d y_2 = -\f{1}{2} \log |y| + C, \quad
 \int K_2(y) d y = \f{1}{2} \arctan \f{y_1}{y_2} + C = - \f{1}{2}\arctan \f{y_2}{y_1} + C , \\
& \int K_1(y) d y_1 = - \f{1}{2} \f{y_2}{|y|^2} + C, \quad \int K_1(y) d y_2 = -\f{1}{2} \f{y_1}{|y|^2} + C, \\
& \int K_2(y) d y_1 = - \f{1}{2} \f{y_1}{|y|^2} + C,  
\quad \int K_2(y) d y_2 = \f{1}{2} \f{y_2}{ |y|^2} + C.
\eal
\eeq

For $ K(y) = - \f{y_i}{2 |y|^2} = \pa_{y_i} G$ \eqref{eq:kernel_du}, we have 
\beq\label{feq:int_ker_u}
\bal
\int \pa_{y_2} G d y &= \int G d y_1 = \f{1}{4} (2 y_1 - 2 y_2 \arctan \f{y_1}{y_2} -  y_1 \log( y_1^2 + y_2^2) ) + C , \\
\int \pa_{y_1} G d y &= \int G d y_2 = \f{1}{4} (2 y_2 - 2 y_1 \arctan \f{y_2}{y_1} -  y_2 \log( y_1^2 + y_2^2) ) + C.
\eal
\eeq

\subsubsection{Formulas for integrals with a half power}\label{sec:int_half_pow}

We introduce 
\beq\label{eq:hol_vx_fh}
f_s(t) \teq  \int_0^t \f{ s^2}{1 + s^4} ds . 
\eeq

Using changes of variables $s = ps $ and $ s= t^2$, we yield 
\beq\label{eq:hol_vx_fh2}
\bal
\int_a^b  \f{ p}{p^2 + s^2} s^{1/2} ds 
&= p^{1/2} \int_{a/p}^{b/p}  \f{1}{1+s^2} s^{1/2} ds 
= 2 p^{1/2} \int_{ \sqrt{a/p} }^{ \sqrt{b/p}} \f{t^2}{ 1 + t^4} dt  \\
& = 2 p^{1/2}  ( f_s( \sqrt{b/p}) - f_s( \sqrt{a/p}) ). 
\eal
\eeq
Clearly $f_s$ is increasing. Using 
the Beta function $B( x, 1- x) = \f{\pi}{\sin (\pi x)} $ and $s = t^{1/4}$ we have 
\[
f_s(\infty) =\f{1}{4} \int_0^{\inf} \f{t^{1/2}}{1+t}  t^{-3/4} dt 
= \f{1}{4}\int_0^{\inf} \f{t^{-1/4}}{1+t} d t 
= \f{1}{4} B( \f{3}{4}, \f{1}{4})  = \f{1}{4} \f{\pi}{ \sin \f{\pi}{4}} = \f{\pi}{2 \sqrt{2}}.
\]

For $a,b, c, d \geq 0$, we introduce 
\beq\label{eq:hol_vx_hy}
F_{K_2,h}(a, b, c, d) = \int_{[a,b] \times [c, d]} |K_2(y)| y_2^{1/2} d y. 
\eeq
Clearly, it suffices to study $F_{K_2, h}(0, a, 0, b)$ for any $a, b> 0$. Then we have 
\beq\label{eq:ing4}
F_{K_2,h}(a, b, c, d) = | F_{K_2, h}(a, c) + F_{K_2, h}(b, d) - F_{K_2, h}(a, d) - F_{K_2, h}(b, c)  |.
\eeq
Denote $m = \min(a, b)$. Using \eqref{feq:int_ker} and \eqref{eq:hol_vx_fh2}, we yield 
\beq\label{eq:hol_vx_pt5}
\bal
& \int_0^m \int_0^m |K_2(y)| y_2^{1/2} dy 
= \int_0^m y_2^{1/2}  d y_2 ( \int_{y_2}^m K_2(y) - \int_0^{y_2} K_2(y) ) d y_1 \\
 = & \f{1}{2} \int_0^m \B(  - \f{y_1}{|y|^2} \B|_{y_2}^m +\f{y_1}{|y|^2} \B|_0^{y_2}  \B)  y_2^{1/2} dy_2
 = \f{1}{2} \int_0^m ( \f{2 y_2}{2 y_2^2} - \f{m}{m^2 + y_2^2}   ) y_2^{1/2} d y_2
  =  \sqrt{m} -\sqrt{m} f_s(1) .
\eal
\eeq

Next, we consider the integral in $R = [0, a ] \times [0, b] \bsh [0, m]^2$. If $a > b$, we get
$y_1 > y_2, K_2(y) > 0$ in $R$ and  
\[
\bal
 \B| \int_b^a \int_0^b K_2(y) y_2^{1/2}  d y \B|
 & =\B| \int_0^b  ( -\f{1}{2} \f{y_1}{|y|^2} \B|_b^a ) y_2^{1/2} d y_2 \B|
  =\B| \f{1}{2} \int_0^b ( \f{b}{b^2 + y_2^2} - \f{a}{a^2 + y_2^2} ) y_2^{1/2} d y_2 \B| \\
  & =| b^{1/2} f_s(1) - a^{1/2} f_s( \sqrt{b/a} ) |.
  \eal
\]

If $a < b$, we yield $y_1 < y_2, K_2(y) < 0$ in $R$ and 
\[
\B|\int_0^a \int_a^b -K_2(y) y_2^{1/2} d y\B|
= \B|\f{1}{2} \int_a^b \f{ y_1}{|y|^2} \B|_0^a  y_2^{1/2} d y_2\B|
 = \B| \f{1}{2} \int_a^b \f{a}{a^2 + y_2^2} y_2^{1/2} d y_2 \B|
 = a^{1/2} \B| f_s( \sqrt{b/a}) - f_s(1) \B| .
\]

For $0 < a < b$, we introduce
\beq\label{eq:hol_vx_diag}
F_{diag}(a, b) \teq \int_{ [a, b]^2 } |K_2(y)| dy = \f{1}{2} \log( \f{b}{a} ) 
- ( \f{\pi}{4} - \arctan \f{a}{b} ).
\eeq

To prove the identity, since $K_2(y)$ is symmetric in $y_1, y_2$, using \eqref{feq:int_ker}, we yield 
\[
\bal
&\int_{ [a, b]^2 } |K_2(y)| dy = 2 \int_a^b \int_a^{y_1} K_2(y) dy 
 = \int_a^b  \f{y_2}{|y|^2} \B|_a^{y_1}  d y_1 
  = \int_a^b  \f{1}{2 y_1} -  \f{a}{y_1^2 + a^2} d y_1 \\
  = & \f{1}{2} \log \f{b}{a} - \atan \f{y_1}{a} \B|_a^b 
  = \f{1}{2} \log \f{b}{a} - (\atan \f{b}{a} - \f{\pi}{4})
  = \f{1}{2} \log \f{b}{a} - ( \f{\pi}{4} -  \atan \f{a}{b} ).
  \eal
\]

Denote some constants 
\beq\label{eq:hol_vx_unit_up}
\bal
C_{ K_2, up}  & \teq \int_0^1 \int_{y_1}^1 |K_2(y)| \cdot | \f{y_1^2}{y_2} - y_2 |^{1/2} dy , \quad 
C_{K_2, low} \teq \int_0^1 \int_{0}^{y_1^2} |K_2(y)| \cdot |y_2|^{1/2} d y_2, \\
C_{K_2}  & \teq  C_{K_2, up} + C_{K_2, low}.
\eal
\eeq
Using \eqref{feq:int_ker}, \eqref{eq:hol_vx_fh}, and $1 \geq y_1 \geq y_2^{1/2} \geq y_2$, we can simplify the integral in 
$C_{K_2, low}$  as follows
\beq\label{eq:hol_vx_unit_low}
\bal
C_{K_2, low} 
& = \int_0^1  y_2^{1/2}  d y_2 \int_{y_2^{1/2}}^1 K_2(y) d y_1
 = \int_0^1 \B( - \f{1}{2}  \f{y_1}{|y|^2}  \B|_{ y_2^{1/2}}^1  \B) y_2^{1/2} d y_2  \\
  &= \f{1}{2} \int_0^1 (\f{y_2^{1/2}}{y_2^2 + y_2 } - \f{1}{y_2^2 + 1} ) y_2^{1/2} d y_2 
  = \f{1}{2} (  \log(2) -2  f_s(1) ) .
  \eal
\eeq
For $C_{K_2, up}$, we follow the strategies in Section \ref{sec:int_hol_stg}. In a small region 
$[0, \e]^2$ near the singularity $y = 0$, 
since $ 0\leq  y_2 - \f{y_1^2}{y_2}\leq y_2 $, we bound the integrand by $|K_2(y)| |y_2|^{1/2}$ and follow \eqref{eq:hol_vx_pt5} to get 
\[
\int_0^{\e} \int_{y_1}^{\e} |K_2(y)|  |y_2|^{1/2} d y
= | \int_0^{\e} y_2^{1/2} d y_2 \int_0^{y_2} K_2(y) d y_1 |
= \f{1}{2} \int_0^{\e} y_2^{1/2} \f{y_1}{|y|^2} \B|_0^{y_2} d y_2 
= \f{1}{2} \int_0^{\e} \f{ y_2^{3/2}}{2 y_2^2} d y_2 = \f{1}{2} \e^{1/2}. 
 \]

We introduce $f_h$ similar to $f_s$
\beq\label{eq:hol_ux_fs}
f_h(a) \teq \int_0^a \f{s^{3/2}}{1+s^2} ds = 2 \int_0^{a^{1/2}} \f{s^4}{1 + s^4} ds. 
\eeq
Using a change of variables, we obtain 
\beq\label{eq:hol_ux_fs2}
\int_0^a \f{s^{3/2}}{ p^2 + s^2} ds 
= p^{1/2} \int_0^{a / p}  \f{s^{3/2}}{1 + s^2} ds = p^{1/2} f_h(a/ p).
\eeq

\subsubsection{Integral formulas for $K_1, K_2$ }\label{sec:feq_dK}

We discuss how to compute the integrals 
\[
 J_{ij}(a, b, c , d) = \int_{[a, b] \times [c, d]} | K(y) y_1^i y_2^j | d y, \quad 1 \leq i + j ,
 \quad K = K_l, \quad K = \pa_k K_l.
\]
Due to symmetry of $| K(y) y_1^i y_2^j | $, we can assume that $0 \leq a \leq b, 0 \leq c \leq d$. For $K = K_1 = \f{y_1 y_2}{ |y|^4}$, the computation is simple since $K_1$ has a fixed sign in each quardrant and $| K_1(y) y_1^i y_2^| = K_1(y) y_1^i y_2^j$. For $ K = K_2, \na K$, we have 
\[
K_2 = \f{1}{2} \f{y_1^2 - y_2^2}{ |y|^4},  \quad \pa_1^3 G = \f{ y_1 (-y_1^2 + 3 y_2^2)}{ |y|^6},
\quad \pa_1^2 \pa_2 G = \f{ y_2( -3 y_1^2 + y_2^2 )}{ |y|^6},
\]
where $G$ is the Green function \eqref{eq:kernel_du} and is harmonic. The above kernels can change sign in $[a, b] \times [c, d]$, and the sign is determined by $S(y) = m^2 y_1^2 - y_2^2$ for some $m > 0$. We fix a kernel $K = K_2$ or $\pa_m K_n$ for some $m,  n =1, 2$ and $i, j$. We have 
\[
F(y) \teq K(y) y_1^i y_2^j, \quad |F| =  |K(y) y_1^i y_2^j| = |m^2 y_1^2 - y_2^2| P(y), 
\]
for some functions $P(y)$ having a fixed sign in $\R^2_{++}$. For example, $ K_2(y) y_1 y_2 = (y_1^2 - y_2^2) \f{y_1 y_2}{2 |y|^4 }, P = \f{y_1 y_2}{2 |y|^4} $. Denote the indefinite integral
\beq\label{feq:int_change_sign}
I_y( y_1, y_2) \teq \int  F(y) d y_2 , 
\quad I_{xy}  \teq \int F(y) d y_1 d y_2, \quad 
I_{diag} \teq  \int I_y( y_1, m y_1) d y_1,
\eeq
which can be derived using symbolic computation. The constant in the indefinite integral formula will be cancelled when we apply it to evaluate an integral in a specific domain. Next, we evaluate
\[
\bal
& S^l_{diag}(p, q)  \teq \int_p^q \int_{mp}^{m y_1} |F(y)|  d y, \quad 
S^u_{diag}(p, q)   \teq \int_p^q \int_{m y_1}^{m q} |F(y)|  d y, \quad S_{diag} = S_{diag}^l + S_{diag}^u.  \\
\eal
\]
The domain of the integrals $S^{\al}_{diag}, \al =l, u$ is a triangle, and $S_{diag}^l, S_{diag}^u$ denote the integral in the lower and the upper triangle $[p, q] \times [mp, mq]$, respectively. We first evaluate $S^l_{diag}$. Since $y_2 \leq m y_1$ in the domain, $F(y)$ has a fixed sign in the domain of $S^l_{diag}$ and 
\beq\label{feq:int_change_sign_d1}
S^l_{diag}(p, q) = \B| \int_p^q  I_y(y_1, \cdot) \B|_{ m p}^{m y_1} d y_1 \B|
= \B| \int_p^q I_y( y_1, m y_1) - I_y( y_1, mp) d y_1 \B|
= \B| I_{diag}(\cdot) \B|_p^q - I_{xy}(y_1, mp) \B|_{p}^q \B|.
\eeq

Similarly, we get 
\beq\label{feq:int_change_sign_d2}
S^u{diag} = \B| \int_p^q I_y(y_1, mq) - I_y( y_1, my_1 ) dy_1 \B|
=  \B|  I_{xy}(y_1, mq) \B|_p^q - I_{diag}( \cdot ) \B|_p^q \B|,
\eeq
and then yield the formula for $S_{diag}$. 

Next, we consider the integral of $|F(y)|$ in a general domain $Q = [a,  b] \times [c, d]
\subset \R^2_{++}$. We can decompose the regions into several parts according to the intersection between $Q$ and the line $y_2 = my_1$, and then apply the integral formulas $S_{diag}$ and $I_{xy}$ \eqref{feq:int_change_sign}.

 \vs{0.1in}
 \paragraph{\bf{Fixed sign: (1) $ c \geq m b$, (6) $ d \leq m a$}}
In these two cases,
we have $ y_2 \geq m y_1$ for any $y \in Q$ or $y_2 \leq m y_1$ for any $y \in Q$. Thus, $F(y)$ has a fixed sign in $Q$ and we can use the analytic formula $I_{xy}$ \eqref{feq:int_change_sign} to evaluate $S$. 

\vs{0.1in}
\paragraph{\bf{(2) $ma \leq c  \leq mb, mb \leq d$}} We decompose $Q$ as follows 
\[
[a, b] \times [c, d] = [a, c / m] \times [c, d] \cup [ c/ m, b] \times [ mb, d] \cup 
[ c/m, b ] \times [c, mb] \teq Q_1 \cup Q_2 \cup Q_3.
\]

\vs{0.1in}
\paragraph{\bf{(3) $ma \leq c  \leq mb, d \leq mb$  }} 
We decompose $Q$ as follows 
\[
[a, b] \times [c, d] = [a, c / m] \times [ c, d] \cup [d/m, b] \times [c, d] \cup [c/m, d/m] \times [ c, d] \teq Q_1 \cup Q_2 \cup Q_3.
\]

\vs{0.1in}
\paragraph{\bf{(4) $c \leq m a ,  mb \leq  d $}}
We decompose $Q$ as follows 
\[
Q = [a, b] \times [c, ma] \cup [a, b] \times [mb, d] \cup  [a, b] \times [ma, mb]  
= Q_1 \cup Q_2 \cup Q_3.
\]

\vs{0.1in}
\paragraph{\bf{(5) $c \leq m a , ma \leq d \leq mb $}}
We decompose $Q$ as follows 
\[
Q = [a, d/m] \times [c, ma] \cup [d/m, b] \times [c, d] \cup  [a, d/m ] \times [ma, d] 
\teq Q_1 \cup Q_2 \cup Q_3.
\]

In each case (2)-(5), for $Q_1, Q_2$, $(my_1)^2 - y_2^2$ and $F$ have fixed signs and we use \eqref{feq:int_change_sign} to evaluate the integral. For $Q_3$, we apply \eqref{feq:int_change_sign_d1}, \eqref{feq:int_change_sign_d2}. See Figure \ref{fig:int_abs} for an illustration of decompositions in different cases.

\begin{figure}[t]
\centering
  \includegraphics[width=0.3\linewidth]{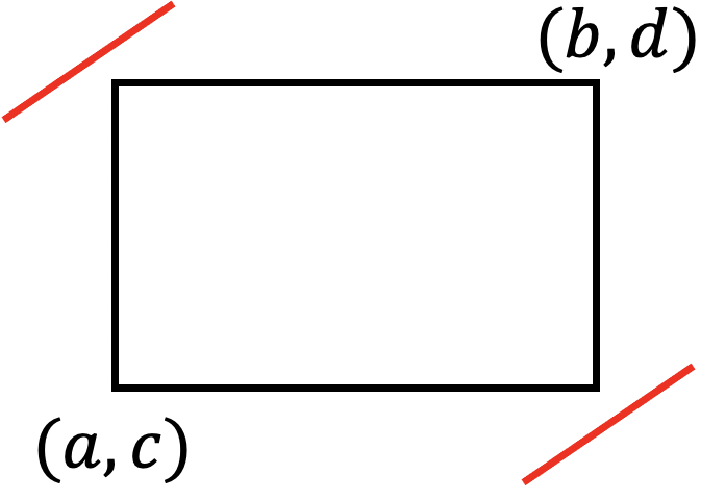}
  \includegraphics[width=0.33\linewidth]{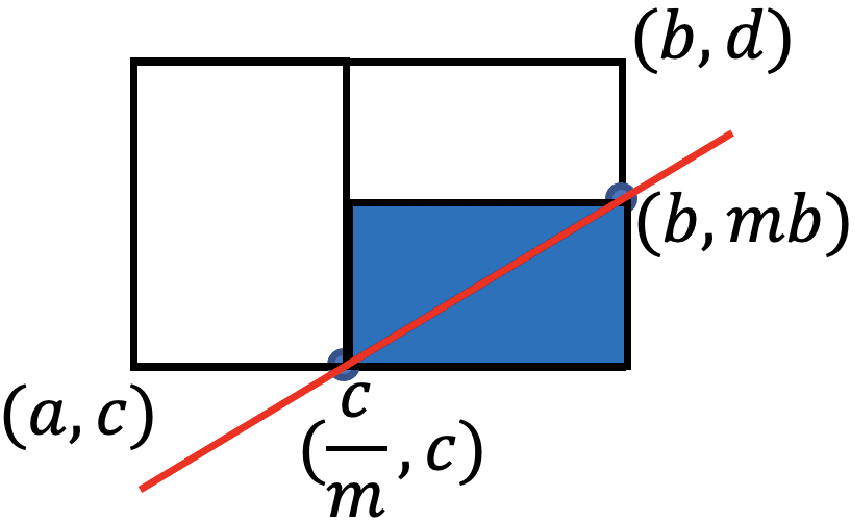}
  \includegraphics[width=0.33\linewidth]{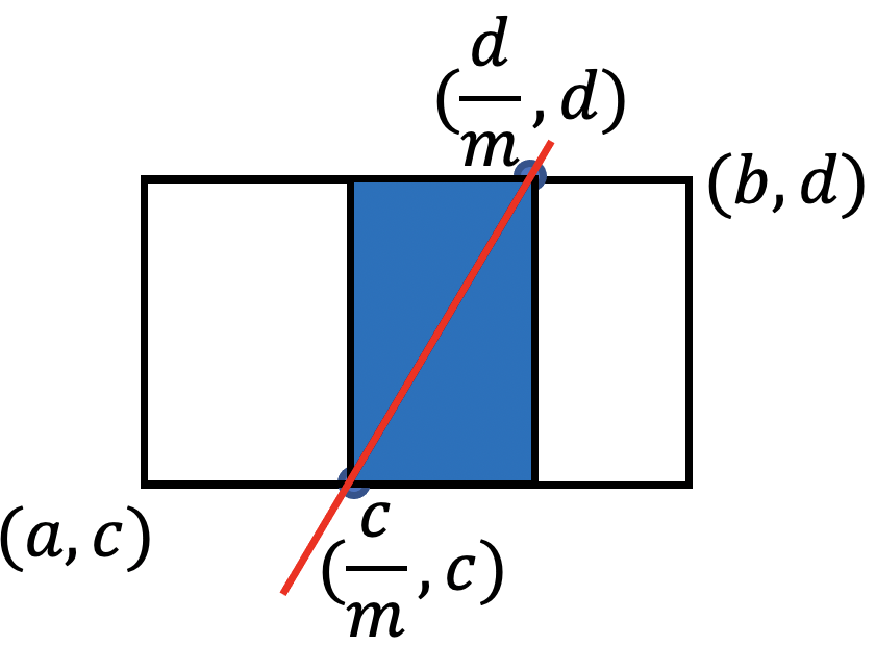}
  \includegraphics[width=0.33\linewidth]{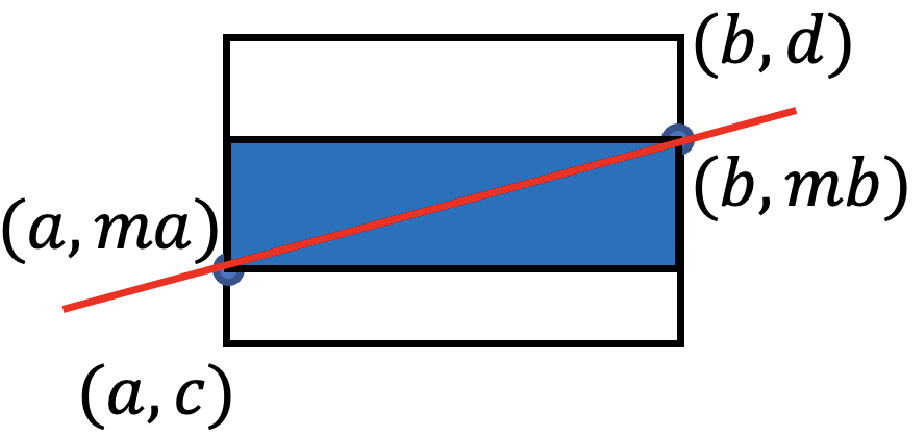}
  \includegraphics[width=0.3\linewidth]{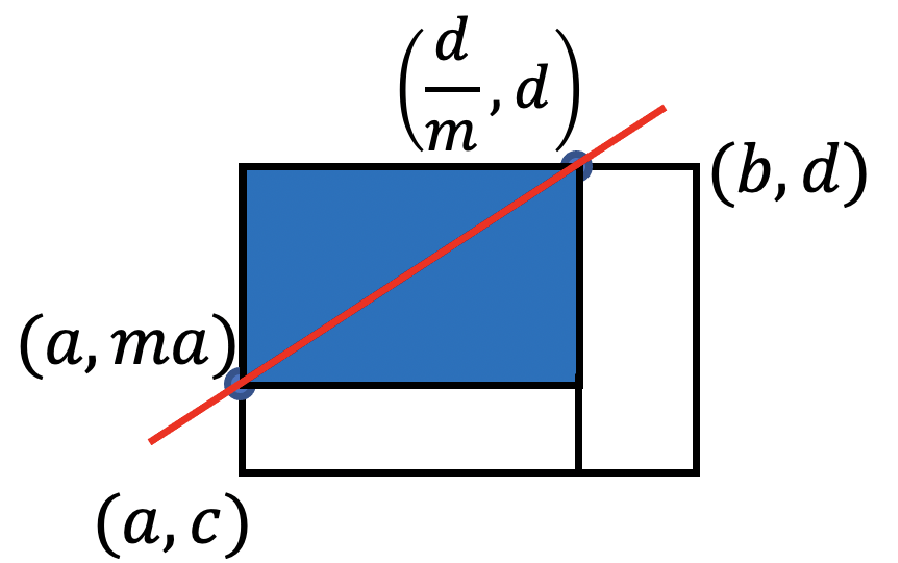}
\caption{ Decompositions in different cases. First figure for case (1), (6), the last four figures for case (2)-(5). The blue region denotes $Q_3$. }
\label{fig:int_abs}
\end{figure}

We need to estimate the integral 
\[
\int_Q |F| d y, \quad |F|= | \pa_2( K_2(y) y_2) | = \B| \f{ y_1^4 - 6 y_1^2 y_2^2 + y_2^4 }{ 2 |y|^6} \B|, 
\]
in $Q = [a, b] \times [c, d]$. Without loss of generality, we assume $Q \subset \R^2_{++}$.  
We define 
\[
p_1 = y_1^4 + y_2^4 , \quad p_2 = 6 y_1^2 y_2^2, \quad 
b(Q) = \max(p_1^u - p_2^l, p_2^u - p_1^l ), \quad  f(Q) = 1 - ( p_1^l \leq p_2^u ) ( p_2^l \leq p_1^u). 
\]

Since $p_i$ is increasing in $y_1, y_2$, it is easy to obtain the piecewise upper and lower bounds $p_i^{l}, p_i^u$. 
If $ f = 1$, we get $p_1(y) \geq p_2(y)$ or $p_2(y) \geq p_1(y)$ for any $y \in Q$, and $F(y)$ has a fixed sign in $Q$. Then we can use the analytic integral formula for $ \pa_2( K_2(y) y_2)$ to evaluate the integral. If $f = 0$, we bound the integral as follows 
\[
\int_Q |F(y)| dy \leq b(Q) \int_Q  \f{1}{2 |y|^6} d y,
\]
and further evaluate the integral using the analytic formula that can be obtained by symbolic computation.

\subsubsection{Symmetric bounds for $K(y)$}
The kernel $K_l(y)$ \eqref{eq:kernel_du} enjoys the symmetry in $y_1, y_2$
\[
K_l(y_2, y_1) = s_l K_l(y_1, y_2), \quad s_1 = 1, s_2 = -1.
\]
Therefore, we have 
\[
\pa_1^m \pa_2^n  ( K_l(y_1, y_2) )  \cdot y_1^i y_2^j
 = s_l  \pa_1^m \pa_2^n  ( K_l(y_2, y_1) ) \cdot y_1^i y_2^j
  = s_l  (\pa_1^n \pa_2^m  K_l(z) z_1^j z_2^i    ) \B|_{ (z_1, z_2) = (y_2, y_1)}.
\]
We need to derive several estimates on $\pa_1  K_l(y_1, y_2)   y_1^i y_2^j$. Using the above relations, once we obtain the estimates related to $\pa_1 K_l$, we can obtain the estimate for $\pa_2 K_l$. Moreover, since 
\[
K_1 = \pa_{12} G, \quad K_2 = \pa_1^2 G, \quad  \pa_1 K_1 = \pa_2 K_2, \quad \pa_2 K_1 = \pa_{122} G = -\pa_1^3 G = - \pa_1 K_2,
\]
once we derive the estimate related to $\pa_i K_1$, we can derive the estimate for $\pa_{3-i} K_2$ using the above relation.

\subsubsection{Asymptotic integral formulas}\label{sec:feq_dK_asym}

In the estimate of the derivatives of the regular part in $\na \uu$, we need to estimate several integrals
\[
I_{i, j, m, n}
 \teq \int_{  Q_{a, b} } | \pa_j K_i(y) y_1^m y_2^n | d y, \ m + n = 2 , \quad 
J_{i, j, l} \teq \int_{  Q_{a, b} } | \pa_j (K_i(y) y_l )  | d y, \quad 
 Q_{a, b} \teq  [-b,b]^2 \bsh [-a, a]^2
\]
for $ i, j, l = 1, 2$, where $K_i$ is given in \eqref{eq:kernel_du}. 
When $m + n = 2$, the integrand has a singularity of order $-1$ and it is locally integrable. Thus, we can apply the estimates and the integral formulas developed in Section \ref{sec:feq_dK} to estimate $I_{i, j, m, n}$. Since the integrand is symmetric in $y_1, y_2$, we only need to estimate the integral in $[0, b]^2 \bsh [0, a]^2$.

In the second integral, the integrand has a singularity of order $-2$ and is not integrable near $0$. Denote $F_{i, j,l}(y) = \pa_j K_i(y) y_l$. Using the scaling symmetry of $K_i$, we get $F(\lam y )= \lam^{-2} F(y)$ and 
\[
\bal
\pa_a J_{i, j, l}(a, b)
 & = - \int_0^a  ( | F_{i, j,l}(a, s )| + | F_{i, j, l}(s, a) | ) d s
 = -  \int_0^1  ( |F_{i, j,l}(a,  a s) |  + | F_{i, j, l}(a s, a) | ) a d s  \\
 &= - \f{1}{a} \int_0^1 ( |F_{i, j,l}(1,   s) |  + | F_{i, j, l}(s, 1) | ) d s  ,
 \eal
\]
where we have used a change of variables $s \to a s$ in the last inequality. We remark that the last integral is independent of $a, b$ and is a constant. Since $J_{i, j, l}(b,b) = 0$, we derive 
\[
J_{i, j,l }(a, b) = J_{i ,j, l}(b,b) - \int_a^b \pa_x J_{i, j,l}(s) ds
 =   C_{i, j,l}\log \f{b}{a} , \quad C_{i, j, l} \teq \int_0^1 ( |F_{i, j,l}(1,   s) |  + | F_{i, j, l}(s, 1) | ) d s  .
\]

It remains to estimate $C_{i, j, l}$. This can be done by using the above identity with $(a, b) = (1, 2)$, 
\[
C_{i, j, l} = \f{ J_{i, j,l}(1,2)}{ \log 2},
\]
and  applying the estimates in Section \ref{sec:feq_dK} to $J_{i, j, l}$ in the domain $Q_{1, 2}$. In particular, using $|K_i(y_1, y_2)| = |K_i(y_2, y_1) |$ and \eqref{feq:int_ker}
\[
\bal
& \int_0^1 ( |K_1(s, 1)|  + |K_1(1, s) | ) ds  = 2 \int_0^1 K_1(s, 1) ds
 = \f{ 1}{1 + s^2} \B|_0^1 = \f{1}{2}, \\
 & \int_0^1 |K_2(s, 1)| + | K_2(1, s)| ds 
 = 2 \int_0^1 K_2(1, s) ds  = \f{s}{s^2 + 1} \B|_0^1 = \f{1}{2},
 \eal
\]
we get 
\beq\label{feq:int_asym}
 \int_{[-b, b]^2 \bsh [-a, a]^2} |K_i(y)| dy = 2 \log(b/ a) .
\eeq

\subsection{Strategy of estimating the explicit integrals}\label{sec:int_hol_stg}

In the sharp H\"older estimates, we need to estimate several explicit integrals, e.g. \eqref{eq:int_holx_ux}
\beq\label{eq:int_du_sg1}
\B| \int_0^{\infty} \int_{1/2}^{\infty} \one_{s_2 \leq b} \one_{ s_1 \geq f(s_2)}  |T(s_1, s_2)  - s_1 |^{1/2}  \D(s)  ds \B|,  \quad \D(s) = K(s- s_*) - K(s - s_r) ,
\eeq
in the H\"older estimate of $u_x$, where $\D$ is some kernel singular at $s_*$, $T$ is the map, and $s_r$ is another fixed point, e.g. $(-1/2, 0)$, outside the domain of the integral. In all cases, the integrand can be written as 
\[
 I(s_1, s_2) = \one_{A(s_1, s_2) \geq 0} |T(s) - s_j|^{1/2} |\D(s)|
\]
for $j=1$ or $2$, where $\one_{A(s_1, s_2) \geq 0}$ is some indicator function. To obtain a sharp estimate, we decompose the domain of the integral into the bulk, and the far-field. 
In the far-field, $s$ is very large. Using the cancellation in the kernel $\D$, we have
\[
|\D(s)| \leq C |s_* - s_r| | |\pa_j K(s)|,  
\]
for some constant $C$ and $j=1$ or $2$. In the domain $A(s) \geq 0$, we have $0\leq T(s) \leq s_j$, and apply the trivial bound $|T(s) - s_j| \leq s_j$. This allows us to estimate the far-field of the integral easily. Due to the decay of the integrand, the contribution of the integral in the far-field is very small.

In the bulk, we partition the integral using dense mesh and estimate the integral in each small region $Q = [x_l, x_u] \times [y_l, y_u]$. We study the monotonicity of the function $A(s)$ and $T(s)$ and obtain their piecewise bounds in $Q$
\[
 A^l(Q) \leq A(s) \leq A^u(Q), \quad  T^l(Q) \leq T(s) \leq T^u(Q, \quad s  \in Q.
\]
The upper and lower bounds of $s_j$ are given by the endpoints of the domain. It follows
\beq\label{eq:bd_int_lu}
\bal
 & \one_{A^l \geq 0} \leq  \one_{A(s) \geq 0} \leq \one_{A^u \geq 0},  \quad | T(s) - s_j|
 \leq \max_{\al, \b = l, u} |T^{\al} - s^{\b}| ,\\
& S = \int_{Q} I(s) ds 
\leq \one_{A^u \geq 0} \max_{\al, \b = l, u} |T^{\al} - s^{\b}|^{1/2} \int_Q |\D(s)| ds. 
 \eal
\eeq
In particular, the integral is $0$ if $A^u < 0$. The set $A(s) = 0$ has zero measure in our integrals. The integral for $|\D(s)|$ can be estimated using the Trapezoidal rule in Section \ref{sec:int_ker}. The same method applies to estimate more general integrals 
\[
 \prod_{1\leq i \leq n} \one_{A_i(s) \geq 0} f(s) |\D(s)|
\]
in $Q$, for some function $f\geq 0 $ and several indicator functions $\one_{A_i(s) \geq 0} $. There are three improvements.

\subsubsection{$\D(s)$ has a fixed sign}
In the support of the integrand $I(s)$, $\D(s)$ has a fixed sign. Yet, it may not have a fixed sign in $Q$ since we estimate the indicator function $\one_{A \geq 0} \leq \one_{A^u \geq 0}$. If $A^l \geq 0$, then we have $\one_{A(s) \geq 0} \geq \one_{A^l \geq 0} = 1$, $\D(s)$ has a fixed sign in $Q$ and 
\[
 S = \int_Q | T(s) - s_j|^{1/2} |\D(s)| ds 
 \leq \max_{s\in Q} |T(s) - s_j|^{1/2}  \int_Q  |\D(s)| ds
 = \max_Q |T(s) - s_j|^{1/2} \B| \int_Q  \D(s) ds \B|.
\]

For the last integral, we can use the analytic integral formula for $\D(s)$.

\subsubsection{Estimate of $|T(s) - s|$ near the singularity }
Near the singularity of $\D(s)$ at $s_*$, we need to exploit the smallness of $|T(s) - s_j| \les |s - s_*|$ to cancel the singularity of $\D(s)$, which has order of $|s-s_*|^{-2}$. The previous piecewise bounds for $T(s) - s_j$ in \eqref{eq:bd_int_lu} do not provide such an order. We estimate 
\beq\label{eq:int_du_sg2}
\one_{A(s) \geq 0}  |T(s) - s_j| \leq c_1(Q) | s_1 - s_{*, 1}| + c_2(Q |s_2 - s_{*, 2}| .
\eeq
The piecewise constant $c_i(Q)$ can be obtained using the piecewise bounds of $\na T$. We will study the piecewise bounds of $T, \na T$ in the following section.

\subsubsection{Estimate of integral near the singularity }\label{sec:int_hol_sing}
Recall from \eqref{eq:int_du_sg1} that the kernel $\D$ can be written as $K(s-s_*) - K(s - s_r)$. Near the singularity, $K(s-s_r)$ is regular and its related integral is much smaller. Using \eqref{eq:int_du_sg2} and the triangle inequality, we get 
\[
S \leq \max_Q |T(s) - s_j|^{1/2}  \int_Q  |K(s - s_r) | ds
+  \int_Q |T(s) - s_j|^{1/2} |K(s - s_*)| ds \teq S_1 + S_2.
\] 
For $S_1$, we estimate it using the previous method. For $S_2$, we further bound $|T(s) - s_j|$ as follows 
\[
|T(s) - s_j| \leq c_3(Q) |s_j - s_{*, j}|, 
\]
for some piecewise constant $c_3(Q)$ depending on $T$ and $\na T$. Then we estimate $S_2$ as follows 
\[
S_2 \leq c_3(Q) \int_Q | K(s - s_*) |  | s_j - s_{*, j}|^{1/2} ds 
= c_3(Q) \int_{Q - a } | K(s)  | |s_j|^{1/2} ds ,
\]
for $a = (s_{*, 1}, 0)$ or $a = (0, s_{*, 2})$, and evaluate the integral using the analytic integral formula for $K(s) |s_j|^{1/2}$.

\subsubsection{Partition the domain of the integrals and the piecewise bounds}
To estimate the 2D integral, we will partition the domain $[0, \infty)^2$ using adaptive mesh $x_1 < x_2 <.., < x_{n_1 + 1} = \infty, y_1 < y_2 < .., < y_{n_2+1} = \infty$. The integrals we will estimate depend on $1$ or $2$ parameters. For example, the integral in \eqref{eq:int_holx_ux} depends on $b$, and the integrals in \eqref{eq:holx_vx}, \eqref{eq:int_holx_vx_in} depend on $[A, B]$. These parameters impose a constraint on the domain of the integral, e.g. $\one_{y_2 \leq b}$ in \eqref{eq:int_holx_ux}, $\one_{y_2 \leq A}$ in \eqref{eq:int_holx_vx_in}. Using the monotonicity of the integrals (or different parts of the integrals) on these parameters, we only need to estimate the integrals with parameters on the mesh, e.g. $b = y_i$ for \eqref{eq:int_holx_ux} or $A = y_i, B = x_j$ for \eqref{eq:int_holx_vx_in},
\[
\int_0^b |F(s_1, s_2)| d s_2 \leq \int_0^{b^u} |F(s_1, s_2) |d s_2 , \quad b \in [b^l, b^u].
\]

For $b, A, B$ on the grid point, the support of the integrand is the union of the small grid $[x^l, x^u] \times [y^l, y^u]$. Thus we can take the sum of the integrals in these grids to estimate the integrals with parameter $b, A, B$. For example, after we estimate
\[
Q_i = \int_{b_i}^{b_{i+1}} |I(s)| ds ,
\]
the cumulative sum of $\sum_{ j \leq i} Q_j$ provides a piecewise bound for $\int_0^b |I(s)| ds $. This strategy allows us to obtain piecewise bounds for parameters $b, A, B$ and the uniform bound.

\subsection{Estimate of the sharp constants for $u_x$}\label{app:sharp_ux}

We first study some properties of the transport map $T(s_1, s_2)$ for $u_x$. 
Recall from Sectin {\secholxux} \cite{ChenHou2023a_supp} that $T$ solves the cubic equation 
\beq\label{eq:cubic_T1}
0 = - 1 - 8 T s_1 + 16 T s_1 ( T^2 + T s_1 + s_1^2 ) - 8 s_2^2 ( 1 - 4  s_1 T + 2 s_2^2) .
\eeq
Solving 
\[
\D(y_1, y_2) = 0, \quad \D(y) \teq K_1(y_1 + 1/2, y_2 ) - K_1(y_1 - 1/2, y_2) = 0,
\]
for $y_1 \geq 0$, we yield the threshold $f(s_2)$ determining the sign of the kernel 
\beq\label{eq:ux_thres}
\bal
f(s_2)  &= \B(  \f{ \f{1}{2} - 2 s_2^2 + \sqrt{ 16 s_2^4 + 4 s_2^2 + 1 }  }{ 6 }  \B)^{1/2}. \\
\eal 
\eeq
For $s_1, s_2 >0$,  we establish in Section {\secholxux} \cite{ChenHou2023a_supp}
\beq\label{eq:holx_ux_sign1}
\bal
\D(s) \leq 0,  \ s_1 \in [f(s_2), \infty), \quad 
 \D(s)  \geq 0,  \ s_1 \in (0, f(s_2)].
\eal
\eeq

Firstly, for a fixed $s_2 \neq 0$ and $s_1 > 0$, we show that it has a unique solution on $[0, \infty)$. 
Dividing $s_1$ on both side of \eqref{eq:cubic_T1} yields 
\beq\label{eq:cubic_T12}
g(T) \teq T^3 + T^2 s_1 + T ( s_1^2 - \f{1}{2} + 2 s_2^2 ) - \f{ (4 s_2^2 + 1)^2}{ 16 s_1} = 0.
\eeq
Since $g(0) <0$ and $g(\inf) >0$, the above equation has at least one real root on $\R_+$. We introduce $ Z = T + \f{s_1}{3}$ and can rewrite the above equation in terms of $Z$
\beq\label{eq:cubic_T2}
\bal
0 =&(T + \f{s_1}{3})^3   - \f{s_1^3}{27}
+ T( \f{2}{3} s_1^2 - \f{1}{2} + 2 s_2^2) - \f{ (4 s_2^2 + 1)^2}{16 s_1}  \\
= &  Z^3 + Z ( \f{2}{3} s_1^2 - \f{1}{2} + 2 s_2^2) ) 
- \B( \f{ (4 s_2^2 + 1)^2}{16 s_1}  + \f{7}{27} s_1^3 + \f{s_1}{3} ( 2 s_2^2 - \f{1}{2})  \B) \\
\teq & Z^3 + p(s_1, s_2) Z + q(s_1, s_2). 
\eal
\eeq
The discriminant is given by 
\beq\label{eq:cubic_discr}
\D_Z(s_1, s_2) = -( 27 q(s_1, s_2)^2 + 4 p(s_1, s_2)^3). 
\eeq
Note that 
\[
-q \geq \f{7}{27} s_1^3 - \f{s_1}{6} + \f{1}{16 s_1} \geq ( 2 \sqrt{ \f{7}{27} \cdot \f{1}{16}} - \f{1}{6}) s_1 \geq 0 ,
\]
and $-q, p $ are increasing in $|s_2|$. We yield 
\[
-\D_Z(s_1, s_2) \geq -\D_Z(s_1, 0) =  \f{ (1-4 s_1^2)^2( 27 - 56 s_1^2 + 48 s_1^4)}{ 256 s_1^2}
 \geq 0.
\]
When $s \neq ( \f{1}{2}, 0)$, the above inequality is strict. Using the solution formula for a cubic equation, we obtain that the cubic equation for $T$ or $Z$ has a unique real root for $s \neq ( \pm \f{1}{2}, 0)$ given by 
\beq\label{eq:cubic_T3}
Z = r_1  - \f{p}{3r_1} , \quad r_1 = \B( \f{-q + \sqrt{ q^2 + \f{4}{27} p^3 } }{2}   \B)^{1/3}, \quad T = Z - \f{s_1}{3},
\eeq
where $p,q$ are defined in \eqref{eq:cubic_T2}. For $s = (  \f{1}{2}, 0)$, it is easy to get that $T = \f{1}{2} = s_1$ is the unique solution in $(0,\infty)$,  which is also given by the above formula.

We have the following basic properties for the map $T$ and the threshold $f(s_2)$ \eqref{eq:ux_thres}.

\begin{lem}\label{lem:ux_map}
The map $T(s_1, s_2)$ is increasing in $s_2$ and decreasing in $s_1$. Moreover, we have $f(s_2) \geq \f{1}{2}$, $f(s_2)$ is increasing in $s_2$ for $s_2 > 0$, and 
\beq\label{eq:ux_map_est}
\bal
 & 4 s_2^2 + 1   \geq \max(4 s_1 T(s_1, s_2) , 1), 
\quad  T^2 + T s_1 + s_1^2 + 2 s_2^2 \geq \f{3}{4},  
 \quad \mathrm{\ for \ } s_1 \geq 0 , \\
 & | T(s_1, s_2) - s_1|   \leq  \max_{x \in [f(s_2), s_1]  } ( | T_{s_1}(x, s_2) | + 1)  | s_1 - \f{1}{2}|, \quad s_1 \geq f(s_2 ).
\eal
\eeq

\end{lem}

\begin{proof}
The estimate $f(s_2) \geq \f{1}{2}$ follows from \eqref{eq:ux_thres} and $1 + 4 s_2^2 + 16 s_2^4 \geq (1 + 2s_2^2)^2$. From \eqref{eq:ux_thres}, for $s_2 >0$, since 
\[
\f{d}{d s_2} ( 6 f^2(s_2))
= -4 s_2 + \f{64 s_2^3 + 8 s_2}{2 \sqrt{ 16 s_2^4 + 4 s_2^2 + 1}}, \quad 
(8 s_2^2 + 1)^2 > 16 s_2^4 + 4 s_2^2 + 1,
\]
$f(s_2)$ is increasing in $s_2 > 0$.  Denote $P = 4 s_1 T, Q = 4 s_2^2 + 1$. Using \eqref{eq:cubic_T1}, we get 
\[
\bal
Q^2 &= (4 s_2^2 + 1)^2 = 16 T s_1( T^2 + T s_1 + s_1^2) + 32 s_2^2 s_1 T - 8 T s_1 \\
&\geq 16 T s_1 \cdot 3 T s_1 + 2( Q-1)  P - 2 P 
= 3 P^2 + 2 P Q - 4P.
\eal
\]

If $P \leq 1$, we derive $Q \geq 1 \geq P$. If $P > 1$, solving the above quadratic inequality in $Q$, we yield 
\[
Q \geq P + \sqrt{ 4 P^2 - 4 P} , \quad \mathrm{ \quad or \quad  } Q \leq P - \sqrt{ 4 P^2 - 4P}.
\]

Note that for $P > 1$, we have $P - \sqrt{ 4 P^2 - 4P} < 1$. Thus, we must have 
\[
Q \geq P + \sqrt{ 4 P^2 - 4 P} \geq P. 
\]
This proves the first inequality in \eqref{eq:ux_map_est}. Using \eqref{eq:cubic_T1} again, we derive 
\[
T^2 + T s_1 + s_1^2  + 2 s_2^2 = \f{1}{16 T s_1} ( 8 T s_1 + (4 s_2^2 + 1)^2 )
\geq  \f{1}{16 T s_1} ( 8 T s_1 + 4 T s_1 ) = \f{3}{4},
\]
where we have used the first inequality in \eqref{eq:ux_map_est} that we just proved. The last inequality in \eqref{eq:ux_map_est} follows directly from $f(s_2) \geq \f{1}{2}$ and the mean value theorem.

Since  \eqref{eq:cubic_T1} is symmetric in $T$ and $s_1$, and its has a unique positive real root for any $s_1 >0$, we get $T (T(s_1, s_2), s_2) = s_1$, or $T \circ T = Id$. Using \eqref{eq:cubic_T0} and \eqref{eq:holx_ux_sign1}, we get
\[
 \D(s_1, s_2) = T_{s_1}(s_1, s_2) \D( T(s_1, s_2), s_2).
\] 
Since $\D(s_1, s_2)$ and $ \D( T(s_1, s_2), s_2)$ have opposite sign, it follows $T_{s_1} \leq 0$. 

Taking $s_2$ derivative on both sides of \eqref{eq:cubic_T12}, we yield 
\beq\label{eq:ux_Ty}
 \f{dT}{d s_2} ( 3 T^2 + 2 T s_1 + s_1^2 - \f{1}{2} + 2 s_2^2) + s_2(  4 T - \f{ 4s_2^2 + 1}{ s_1} ) = 0 .
\eeq
Using the first and the second inequality in \eqref{eq:ux_map_est}, we prove $ \f{dT}{d s_2} \geq 0$. 
\end{proof}

\subsubsection{Piecewise bounds for $T, \na T$}\label{sec:bd_ux}

Using the monotonicity of $T$ in Lemma \ref{lem:ux_map} and its analytic formula, we can derive the piecewise upper and lower bounds for the map $T(s)$ in $[s_1^l, s_1^u] \times [s_2^l, s_2^u] \subset \R^2_+$,
\beq\label{eq:bd_ux_Tul}
T^l = T( s_1^u, s_2^l),  \quad T^u = T( s_1^l, s_2^u ).
\eeq

Such a formula can lead to a large round off error when $s_1/s_2$ and $s_1$ are large. In such a case, we derive another bound. Firstly, we write \eqref{eq:cubic_T1}, \eqref{eq:cubic_T12} as follows 
\[
T^3 + T^2 s_1 + ( T - \hat T)(s_1^2 - \f{1}{2} + 2 s_2^2) =0,  \quad \hat T = \f{ (4 s_2^2 + 1)^2}{ 16 s_1 (s_1^2 - \f{1}{2}  + 2 s_2^2) }.
\]
The function $\hat T$ can be seen as the approximation of $T$, and we can estimate the piecewise bounds for $\hat T$ easily. For $s_1 > 1, s_2>0$, since $T>0, s_1^2 - \f{1}{2} + 2 s_2^2 > 0$, we yield $T \leq \hat T$ and 
\beq\label{eq:bd_ux_Tul2}
T = \hat T  - \f{T^3 + T^2 s_1}{ s_1^2 - \f{1}{2}  + 2 s_2^2} 
\geq  \hat T - \f{ \hat T^3 + \hat T^2 s_1}{ s_1^2 - \f{1}{2}  + 2 s_2^2} .
\eeq
Using the piecewise bounds for $\hat T$ and the above estimates, we can obtain another piecewise bounds for $T$. 
Taking $s_1$ derivative of \eqref{eq:cubic_T12}, we yield 
\[
 \f{dT}{d s_1} ( 3 T^2 + 2 T s_1 + s_1^2 - \f{1}{2} + 2 s_2^2) + 
 T^2 + 2 s_1 T + \f{ (4 s_2^2 + 1)^2}{16 s_1^2} = 0.
\]
From \eqref{eq:ux_Ty} and the above formula, we obtain $\pa_{s_i} T = \f{P_i(T, s_1, s_2)}{ Q_i(T, s_1, s_2)}$ for some polynomials $P_i, Q_i$. Using the above piecewise bounds for $T$, we can further obtain piecewise bounds for $\na T$. 

Near the singularity of the integral, we need to obtain a sharp estimate of $T(s) - s_1$. Since $T( f(s_2), s_2) = f(s_2)$, and $f(s_2) \geq 1/2$, for $s_1 \geq f(s_2)$, we yield 
\beq\label{eq:bd_ux_Tmx}
\bal
|T(s) - s_1|  &= | T(s) - s_1 - ( T(f(s_2), s_2) - f(s_2) ) |
\leq |s_1 - f(s_2)| ( \max_{ \xi \in [f(s_2), s_1]} |T_{s_1}( \xi, s_2)| + 1 )\\
& \leq |s_1 - 1/2| (\max_{ \xi \in [1/2, s_1]} |T_{s_1}( \xi, s_2)|  + 1).
\eal
\eeq
We can estimate the upper bound using the piecewise bounds for $\na T$.

\subsubsection{Estimate of the explicit integrals for $[u_x]_{C_x^{1/2}}$}\label{sec:bound_ux_T}
In this section, we estimate the integral 
\beq\label{eq:int_holx_ux}
\bal
 S(b) &= \B| \int_0^b d s_2 \int_{f(s_2)}^{\inf}  |T(s_1, s_2)  - s_1 |^{1/2}  \D(s)  d s_1 \B|
 =  \B| \int_0^{\infty} \int_{1/2}^{\infty} \one_{s_2 \leq b} \one_{ s_1 \geq f(s_2)}  |T(s_1, s_2)  - s_1 |^{1/2}  \D(s)  ds \B|,  \\
\D(s)  &= K_1(s_1 + \f{1}{2}, s_2) - K_1(s_1 - \f{1}{2}, s_2), 
 \eal
\eeq
in the upper bound in Lemma {\lemholxux} \cite{ChenHou2023a_supp}, and obtain its piecewise bounds for $b \in [1, \infty)$, where $f(s_2)$ and $T(s)$ are determined by \eqref{eq:ux_thres}, \eqref{eq:cubic_T1}. 
Note that one needs to multiply $S(b)$ by a factor $\f{4}{\pi}$ to get the constant in Lemma {\lemholxux} \cite{ChenHou2023a_supp}. In the second identity, we have used $f(s_2) \geq \f{1}{2}$ from Lemma \ref{lem:ux_map}. The kernel satisfies $\D(s) \leq 0$ in the domain of the integral.

For very large $R_0, R_1$, e.g. $R_0 = R_1 = 10^8$, and small $m_1>0$,  we decompose the domain integral into four parts
\beq\label{eq:int_ux_holx_dom}
\bal
& D_1 = [ 1/2, R_1 ] \times [m_1, R_0], \
D_2 = [1/2, R_1]  \times [0, m_1],  \\
&  D_3 = [R_1, \infty) \times [0, R_0], \ D_4 = [1/2, \infty) \times [R_0, \infty).
\eal
\eeq

The first part captures the  bulk part of the integrals, $D_2$ captures the integral near the singularity, and $D_3$ and $D_4$ capture the far-field of the integral. We follow the strategy in Section \ref{sec:int_hol_stg} to estimate the integrals.

\vspace{0.1in}
\paragraph{\bf{Integral in the finite domain}}

To estimate the integrals in $D_1, D_2$, we discretize the domain  $[1/2, R_1] \times [0, R_0]$ using 
\[
1/2 = x_0 < x_1 < ...< x_{n_1} = R_1, \quad 0 = z_0 < z_1 < ...< z_{n_2} = R_0 .
\]
Since $S(b)$ is increasing in $b$, it suffices to estimate $S(b)$ for $b = z_i$ and $S(\infty)$. For a fixed domain $s   \in [x_{i-1}, x_i] \times [z_{j-1}, z_j] \teq 
X_i \times Z_j \teq Q_{ij}$, since $f(s_2)$ is increasing (Lemma \ref{lem:ux_map}), we have 
\[
S_{ij} = \int_{Q_{ij}} \one_{s_1 \geq f(s_2)} | T(s) - s_1|^{1/2}  |\D(s) | ds 
\leq \one_{ x_i \geq f( z_{j-1}) } \int_{Q_{ij}}  | T(s) - s_1|^{1/2}  |\D(s) | ds  .
\]
Using the piecewise bounds of $T(s)$ in $Q_{ij}$ \eqref{eq:bd_ux_Tul}, \eqref{eq:bd_ux_Tmx}, we get 
\[
|T(s_1, s_2) - s_1| \leq \min\B( \max( T^u - s_1^l, s_1^u -  T^l ) , |s_1-1/2|  \B( \max_{\xi \in [1/2, x_i] , s_2 \in Z_j } |T_{s_1}(\xi, s_2)| + 1 \B) \B). 
\]

For the integral $\int_Q |\D(s)|$, we use the two estimates in Section \ref{sec:int_ker} based on the Trapezoidal rule and the analytic integral formulas. If $x_{i-1} \geq f(z_j)$, 
from \eqref{eq:holx_ux_sign1}, we get 
\[
s_1 \geq x_{i-1} \geq f(z_j) \geq f(s_2), \quad |\D(s)| = - \D(s), 
 \]
 for any $s \in Q_{ij}$. Hence, we can use the analytic integral formula for $\D(s)$ in this case. 

 Near the singularity $s = (1/2, 0)$, we use the triangle inequality to obtain another estimate
 \[
S_{ij} \leq \int_{Q_{ij}} |T(s) - s_1|^{1/2} (  K_1(s_1 + 1/ 2 , s_2)  + K_1(s_1 -1/2, s_2)  ) d s \teq I_1 + I_2.
 \]
 where we have used $K_1(s_1, s_2) > 0$ for $s_1, s_2 > 0$. The regular part $I_1$ is estimated using the previous method. Using \eqref{eq:bd_ux_Tmx}, for the singular part $I_2$, we have
 \[
I_2 \leq  (\max_{\xi \in [1/2, x_i] , s_2 \in  Z_j } |T_{s_1}(\xi, s_2)|+ 1) \int_{Q_{ij}} (s_1-1/2)^{1/2} K_1(s_1 -1/2, s_2) ds,
 \]
 and obtain the integral using the analytic integral formula for $K_1(s) s_1^{1/2}$.

\vspace{0.1in}
\paragraph{\bf{Integral in the far-field}}
The integrals in the far-field $D_3, D_4$ \eqref{eq:int_ux_holx_dom} are very small. Next, we estimate the integral in $D_3$, for $s_1 \geq f(s_2)$, we derive
\beq\label{eq:bd_ux_Tmx2}
s_1 \geq f(s_2) \geq T(s), \quad |s_1 - T(s)| \leq s_1. 
\eeq
Denote $R_m = \min(R_0, R_1)$. In $D_3 \cup D_4$, we have $\max(s_1, s_2) \geq R_m$. Applying Lemma \ref{lem:K_mix} with $i+j = 3$ and the Mean Value Theorem yields
\[
|\D(s) | = | \pa_{1} K_1(\xi, s_2)| \leq \f{1}{ (\xi^2 + s_2^2)^{3/2}}
\leq C_R |s|^{-3}, \quad C_R =  (\f{R_m}{R_m-1/2})^3,
\]
for some $\xi \in [s_1 - \f{1}{2}, s_1 + \f{1}{2} ]$, where we have used $|(\xi, s_2)| \geq  |s| - 1/2 \geq |s|(1 - \f{1}{ 2 R_m} ) = |s| \f{R_m}{R_m - 1/2}$. For the integrals in $D = D_3$ or $ D = D_3 \cup D_4$ \eqref{eq:int_ux_holx_dom}, using the above estimates, we yield 
\[
\bal
S_{D}  & \teq 
\int_{D}  \one_{s_2 \leq b} \one_{s_1 \geq f(s_2)} |T - s_1|^{ \f{1}{2}} |\D(s) | d s 
\leq  C_{R}\int_{D}\f{s_1^{1/2}}{ |s|^3}  d s ,  \\
S_{D_3} & \leq C_R \int_{R_1}^{\inf} d s_1  \int_0^{R_0}
\f{s_1^{1/2}}{ |s|^3}  d s 
\leq C_R \int_{R_1}^{\inf} s_1^{-5/2} R_0 d s_1 = C_R \f{2}{3} R_1^{-3/2} R_0 = C_R \f{2}{3} R_1^{-3/2} R_0, \\
S_{D_3, D_4} & \leq C_R \int_{ |s| \geq R_m, s \in \R_2^{++}} \f{s_1^{1/2}}{|s|^3} ds  \leq C_R \int_{ |s| \geq R_m, s \in \R_2^{++}} |s|^{- \f{5}{2} } ds  = C_R \int_{R_m}^{\infty}
r^{- \f{3}{2} } dr \int_0^{ \f{\pi}{2} }  d \b  = C_R \pi R_m^{- \f{1}{2} }.
\eal
\]
Note that if $b \leq R_0$, due to the restriction $\one_{s_2 \leq b}$, we only need to estimate the integral in $D_3$.

Combining the estimate of the integrals $S_{ij}$ and $S_{D_3}$, we can estimate $S(b)$ for $b = z_i, i=1,2,.., n_2$. Adding the contribution from $S_{D_4}$, we obtain the estimate of $S(\infty)$. Since $S(b)$ is increasing, we obtain piecewise bounds and the uniform bound for $S(b)$ \eqref{eq:int_holx_ux}. We remark that to obtain the constant in Lemma {\lemholxux} \cite{ChenHou2023a_supp}, one needs to further multiply the constant $\f{4}{\pi}$.

\subsubsection{Estimate the integrals for $[u_x]_{C_y^{1/2}}$}
In this section, we estimate the integral 
\beq\label{eq:int_holy_ux}
S(a) =   \int_0^a \int_0^{\infty} y_1^{1/2} \B| \f{y_1 (1/2-y_2)}{ (y_1^2 + (1/2-y_2)^2)^2 }   +\f{y_1 (1 /2 + y_2) }{ (y_1^2 + (1 / 2+ y_2)^2)^2 } \B|  dy ,
\eeq
in Lemma {\lemholyux} \cite{ChenHou2023a_supp} for the estimate of $[u_x]_{C_y^{1/2}}$. Note that one needs to multiply $S(a)$ by a factor $\f{\sqrt{2}}{\pi}$ to get the constant in Lemma {\lemholyux}. 
Clearly, $S$ is increasing in $a$. 

Swapping the dummy variables $y_1, y_2$ and writing $(1/2- y_1) = - (y_1 - 1/2)$, we  yield 
\beq\label{eq:int_holy_ux2}
S(a) =  \int_0^{\infty}  d y_1 \int_0^a y_2^{1/2} \B| \f{y_2 (y_1-1/2)}{ (y_2^2 + (1/2-y_1)^2)^2 }   - \f{y_2 (1 /2 + y_1) }{ (y_2^2 + (1 / 2+ y_1)^2)^2 } \B|  dy_2
 = \int_0^a  y_2^{1/2}  d y_2 \int_0^{\infty}   | \D(y)| dy_1,
\eeq
where $\D(s)$ is defined in \eqref{eq:holx_ux_sign1}. 

Using the sign property \eqref{eq:holx_ux_sign1}, we can first compute the integral in $y_1$
\beq\label{eq:int_holy_ux3}
\bal
\int_0^{\infty}   | \D(y)| dy_1
&= ( \int_0^{ f(y_2)} -   \int_{f(y_2)}^{\infty} ) \D(y) d y_1
= \f{1}{2}\B(-  \f{y_2 }{y_2^2 + (y_1 + \f{1}{2})^2} 
+ \f{y_2 }{y_2^2 + (y_1 - \f{1}{2} )^2}  \B) (\B|_0^{f(y_2)} -  \B|_{f(y_2)}^{\infty} ) \\
& = g( f(y_2) - 1/2, y_2) - g( f( y_2 ) + 1/2, y_2), \quad g(y) \teq \f{ y_2}{ y_1^2 + y_2^2},
\eal
\eeq
where we have used $A (\B|_a^b  - \B|_c^d) $ to denote $(A(b) - A(a)) - (A(d) - A(c))$. The boundary term vanishes at $0$ and $\infty$. It follows
\[
S(a) = \int_0^a y_2^{1/2} \B(   g( f(y_2) - 1/2, y_2) - g( f( y_2 ) + 1/2, y_2)\B) d y_2, \quad 
 g(y) \teq \f{ y_2}{ y_1^2 + y_2^2}.
\]
We partition $[0, \infty)$ using mesh $0 = z_0 < z_1< ...< z_{n_2} = R_0$ and $\infty$.  Since $S(a)$ is increasing in $a$, it suffices to estimate $S(a)$ for $a = z_i $ and $a = \infty$.

\vspace{0.1in}
\paragraph{\bf{Integral in a finite domain}} 
Clearly, $g(y)$ is decreasing in $y_1$ for $y_1 > 0$. For the integral in $Z_j = [z_{j-1}, z_j]$, since $f(y_2)$ is increasing in $y_2$ and $f(y_2) \geq 1/2$ (see Lemma \ref{lem:ux_map}), we get 
\[
1/2 \leq f^l = f(z_{j-1}) \leq f(y_2) \leq f(z_j) = f^u, \quad y_2 \in Z_j.
\]
Using the above estimate and $ d_{y_2} \log |y| = g$, we derive
\[
\bal
S_j  & \teq \int_{z_{j-1}}^{z_j} y_2^{1/2} \B( g( f(y_2) - \f{1}{2}, y_2) - g( f(y_2) + \f{1}{2}, y_2)  \B) d y_2
\leq  z_j^{1/2} \int_{Z_j} ( g( f^l- \f{1}{2}, y_2) - g( f^u + \f{1}{2}, y_2)  ) dy_2 \\
&= z_j^{1/2} \f{1}{2} \B( \log( y_2^2 + (f^l- \f{1}{2})^2) - \log( y_2^2 + (f^u + \f{1}{2})^2) \B) \B|_{ z_{j-1}}^{z_j}. 
\eal 
\]
We remark that $f^l, f^u$ do not depend on $y_2$. Using
\[
\bal
 & g( f(y_2) - \f{1}{2}, y_2) - g( f(y_2) + \f{1}{2}, y_2) 
 = \f{y_2 }{ (f- \f{1}{2} )^2 + y_2^2}  - \f{y_2}{ (f+ \f{1}{2} )^2 + y_2^2}
 = \f{ 2 y_2 f}{ ((f- \f{1}{2} )^2 + y_2^2)( (f+ \f{1}{2} )^2 + y_2^2) } \\
 \leq & \f{f^u}{f^l} \f{2 y_2 f^l}{ ((f^l-1/2)^2 + y_2^2)( (f^l+1/2)^2 + y_2^2) } 
 = \f{f^u}{f^l} ( g( f^l - \f{1}{2}, y_2) - g( f^l + \f{1}{2}, y_2) ),
 \eal
\]
and computation similar to the above, we obtain another estimate 
\[
S_j \leq z_j^{1/2} \f{f^u}{f^l} \f{1}{2} \B( \log( y_2^2 + (f^l- \f{1}{2})^2) - \log( y_2^2 + (f^l + \f{1}{2})^2) \B) \B|_{ z_{j-1}}^{z_j}
 = z_j^{1/2} \f{f^u}{2 f^l} \log( 1 - \f{ 2 f^l}{ y_2^2 + (f^l + 1/2)^2 } ) \B|_{ z_{j-1}}^{z_j}.
\]

For the integral near the singularity, e.g. when $z_j$ is small, we need to exploit the cancellation between the kernel and $y_2^{1/2}$. Since the more regular part $-g( f(y_2) + 1/2, y_2 ) \leq 0$, we derive an improved estimate near the singularity 
\[
S_j \leq \int_{z_{j-1}}^{z_j} y_2^{1/2} g( f(y_2) -1/2, y_2  ) d y_2
\leq \int_{z_{j-1}}^{z_j} y_2^{1/2} \f{1}{y_2} dy_2
  = 2 (z_j^{1/2} - z_{j-1}^{1/2} ).
\]

\vspace{0.1in}
\paragraph{\bf{Integral in the far-field}}

For the integral in the region $y_2 \geq R_0$, we treat it as an error. 
Using $\xi^2 + y_2^2 \geq 2 |\xi y_2|$ and the Mean Value Theorem, we get
\[
|g( f(y_2) -1/2, y_2 )  -  g( f(y_2) +1/2, y_2 ) |
 =| \pa_{y_1} g(\xi, y_2)|  = \f{ 2 |\xi y_2 |}{ (\xi^2 + y_2^2)^2} \leq \f{1}{ y_2^2} ,
\]
for some $\xi \in [f - 1/2, f+1/2]$. Using \eqref{eq:int_holy_ux3} and the above estimate, we yield
\[
\bal
\int_{R_0}^{\inf} y_2^{ \f{1}{2}} \int_0^{\inf} |\D(y)| d y_1 = 
 \int_{R_0}^{\inf} y_2^{1/2} \B( g( f(y_2) -\f{1}{2}, y_2 )  -  g( f(y_2) +\f{1}{2}, y_2 ) \B) dy_2
\leq \int_{R_0}^{\inf} y_2^{ \f{1}{2}} y_2^{-2} d y_2  = 2 R_0^{- \f{1}{2}}.
\eal
\]

\subsection{Estimate the sharp constants for $u_y, v_x$}\label{sec:holx_vx}

In this section, we estimate the integrals in the sharp H\"older estimate of $u_y, v_x$. 
Recall the monotonicity lemma proved in Section {\secholyux} \cite{ChenHou2023a_supp}.
\begin{lem}\label{lem:comp_trunc}
Suppose $f , f g \in L^1$ and $g \geq 0$ is monotone increasing on $[0, \infty)$.  
For any $0 \leq k \leq b \leq c$, we have 
\[
\int_{b - k}^{b+k}  |f(x - k) | g(x) dx 
\leq \int_{b-k}^{c-k} | f(x-k) - f(x+k)| g(x) dx 
+ \int_{c-k}^{c + k}  |f( x -k) | g(x) dx. 
\]
\end{lem}

We have the following basic Lemma for the map $T$ \eqref{eq:map_vx}. See the left figure in Figure \ref{fig:vx_OT} for an illustration of the sign of the kernel and the map.

\begin{lem}\label{lem:map_vx}
For $y_1, y_2 >0$, equation \eqref{eq:map_vx} has a unique solution $T$ in $(0, \infty)$. 
For $y_1, y_2 >0$, $T(y)$ is decreasing in $y_2$, and $ |\f{d}{d y_2} T(y)| \leq 1$ if $T(y ) \leq y_2$; $T(y_1, y_2)$ is decreasing in $y_1$ for $y_1 \leq 1$, and $T$ is increasing in $y_1$ for $y_1 > 1$.
\end{lem}

\begin{proof}
We fix $y_1, y_2 > 0$. Recall the equation \eqref{eq:map_vx}
\beq\label{eq:map_vx_rec}
g(T) \teq  T^3 + T^2 y_2 +  T ( y_2^2 + 2 + 2 y_1^2) - \f{ (y_1^2 - 1)^2}{y_2} = 0.
\eeq
Since $g(0) < 0$ and $g(\infty) >0$, there exists at least one real root in $(0, \infty)$. The uniqueness is discussed below \eqref{eq:map_vx} and follows from the discriminant. Taking $y_2$ derivative, we yield 
\[
\f{d T}{d y_2} (3 T^2 + 2 T y_2 +  y_2^2 + 2 + 2 y_1^2 ) + 
T^2 + 2 T y_2  + \f{ (y_1^2 - 1)^2}{y_2^2} = 0,
\]
which implies $\f{d T}{d y_2} < 0$. For $T(y) \leq y_2$, using \eqref{eq:map_vx_rec}, we get 
\[
T^2 + 2 T y_2  + \f{ (y_1^2 - 1)^2}{y_2^2} 
= T^2 + 2 T y_2  + \f{T^3 + T^2 y_2 + T( y_2^2 + 2 + 2 y_1^2)}{ y_2}
\leq T^2 + 2 T y_2 + 2 T^2 + ( y_2^2 + 2 + 2 y_1^2),
\]
which along with the formula of $\f{d T}{d y_2}$ implies $|\f{ d T}{d y_2} | \leq 1$.

Taking $y_1$ derivative, we yield 
\[
 \f{d T}{d y_1} ( 3 T^2 + 2 T y_2 + y_2^2 + 2 + 2 y_1^2 ) + 4 T y_1 - \f{4 y_1 (y_1^2-1)}{y_2} = 0.
\]
From the above equation for $T$, we get 
\[
\bal
&(y_1^2 - 1)^2 = y_2 (  T^3 + T^2 y_2 +  T ( y_2^2 + 2 + 2 y_1^2)  ) > (T y_2)^2, \\
& y_1^2 - 1 > T y_2, \ \mathrm{ \ for  \ } y_1 > 1, \quad 
  y_1^2 - 1 < - T y_2, \ \mathrm{ \ for  \ } y_1 < 1. 
\eal
\]
Hence, we prove the monotonicity properties of $T$.
\end{proof}

The above derivations provide the formula of $ \na T$ in terms of $T, y_1, y_2$. We can use the monotonicity of $T$ to obtain its piecewise bounds. For large $y_1, y_2$, to reduce the round off error, following the argument in Section \ref{sec:bd_ux} and \eqref{eq:bd_ux_Tul2} and using \eqref{eq:map_vx_rec}, we obtain another piecewise bounds for $T$
\beq\label{eq:bd_vx_Tul2}
  \hat T \geq T  = \hat T -  \f{T^3 + T^2 y_2}{y_2^2 + 2 + 2 y_1^2} 
  \geq \hat T -  \f{\hat T^3 +\hat T^2 y_2}{y_2^2 + 2 + 2 y_1^2} , \quad \hat T \teq \f{ (y_1^2 -1)^2}{y_2 ( y_2^2 + 2 + 2 y_1^2) }. 
\eeq

Using the formulas of $\na T$ and the bounds for $T$, we can further obtain the piecewise bounds for $\na T$. Near the singularity of the kernel $(\pm 1, 0)$, we need a sharp bound for $T(y) - y_2$ by $C |y-1|$ so that it can cancel the singularity of the kernel with order $-2$. Note that the bound based on the piecewise bounds for $T(y) - y_2$ does not provide the order $|y-1|$ when $y$ is sufficiently close to $1$.  Recall $s_c(y_1)$ from \eqref{eq:vx_thres}. It satisfies $T( y_1, s_c(y_1)) = s_c(y_1)$. Following \eqref{eq:bd_ux_Tmx2}, for $y_2 \geq s_c(y_1)  \geq T(y)$ we have 
\[
|y_2 - T(y_1, y_2) | = |y_2 - T(y_1, y_2) - ( s_c(y_1) - T(y_1, s_c(y_1)) ) | 
\leq |y_2 - s_c(y_1)| ( \max_{ \xi \in [s_c(y_1), y_2] } |\pa_{y_2} T (y_1, \xi)| + 1 ).
\]

Using the formula \eqref{eq:vx_thres} and 
\[
(2 y_1^2 - 3 y_1 + 2)^2 - (y_1^4 - y_1^2 + 1) 
= 3 y_1^4  - 12 y_1^3 + (8 + 9 + 1)y_1^2 - 12 y_1 + 3 = 3( y_1 - 1)^4 \geq 0,
\]
we yield 
\[
\bal
(s_c(y_1) )^2  & = \f{ - (y_1^2 + 1) + 2(y_1^4 -y_1^2 + 1)^{1/2} }{3}
 = \f{ - (y_1^2 + 1)^2 + 4(y_1^4 -y_1^2 + 1) }{3 (  y_1^2 + 1 + 2(y_1^4 -y_1^2 + 1)^{1/2})} \\
 & \geq \f{ 3 (y_1^2 - 1)^2}{  3( y_1^2 + 1 + 2 (2 y_1^2 - 3y_1 + 2) )}
  = \f{  (y_1- 1)^2(y_1+1)^2}{  5 y_1^2 - 6 y_1 + 5 } .
 \eal
\]
Thus, we get
\[
y_2 - s_c(y_1) \leq y_2 - |y_1 - 1| f(y_1) \leq y_2, \quad f(y_1) = \f{ y_1+1}{ (5 y_1^2 - 6 y_1 + 5 )^{1/2}}.
\]
Near $y=1$, $f(y_1)\approx 1$ and the upper bound is $O(|y-1|)$. Moreover, from the above formula, since $(y_1^4 -y_1^2 + 1)^{1/2} \geq (y_1^2)^{1/2} = y_1$, we have 
\beq\label{eq:vx_thres_u}
(s_c(y_1) )^2 = \f{ 3(y_1-1)^2 (y_1+1)^2}{
3 (  y_1^2 + 1 + 2(y_1^4 -y_1^2 + 1)^{1/2})
 }
\leq  \f{ 3(y_1-1)^2 (y_1+1)^2}{
3 (  y_1^2 + 1 + 2 y_1) } = (y_1-1)^2, \quad s_c(y_1) \leq |y_1-1|.
\eeq

\subsubsection{Estimate the explicit integrals for $u_y, v_x$}

We follow the strategy in Section \ref{sec:int_hol_stg} to estimate the integrals in the $C_x^{1/2}$ estimate of $v_x, u_y$. Recall from Appendix \secholxvx  \cite{ChenHou2023a_supp}  the estimate  of $[v_x]_{C_x^{1/2}}$
\beq\label{eq:holx_vx}
\bal
& 2^{-1/2} | v_x(-1, x_2) - v_x(1, x_2)|  \leq 2^{-1/2} \f{1}{\pi}  \B( (S_{in, x} + S_{1D} ) [\om]_{C_x^{1/2}} + 4 ( S_{in, y} + S_{out} ) [\om]_{C_y^{1/2}} \B) ,\\
& S_{in, x}  =  \int_{y_1 \notin J_1, y_1 \geq 0} \sqrt{2 y_1} | \D_{1D}(y_1) | dy_1 , \\
& 
  S_{in, y} =  S_{ up} + S_{low} ,\quad  S_{up} = S_{up, ns}(\e) +  S_{in, y, \e}  , \quad S_{in,y, ns} = S_{up, ns}(\e) + S_{low},  \\
& S_{up, ns}(\e) = \int_{ y_1 \in \R_+ \bsh [1-\e, 1 + \e] }  \int_{0 }^A 
\one_{ y_2 \geq s_c(y_1)}
| T(y) - y_2|^{ \f{1}{2} }  |\D(y)| dy_2 dy_1, \\
&  S_{low} =   \int_{y_1 \in J_1^+ \bsh [1-\e, 1+\e]} \int_0^{ A } \one_{y_2 \leq T(y_1, A)}  |y_2|^{ \f{1}{2} }  |\D(y)| dy_2    dy_1   ,\\
& S_{out}   = \one_{x_2 < B} \f{1}{4}  \sqrt{2 x_2} \int_{y_1 \in \R} \int_{ x_2  }^B  |\D(y) | dy
 = \one_{x_2 < B} \f{1}{2} \sqrt{ 2 x_2 } \int_{y_1 \in \R^+} \int_{ x_2  }^B  |\D(y) | dy , \\
& S_{1D}  =  2 \int_{ \f{1}{9}}^9 | \D_{1D} |y_1-1|^{1/2} d y_1
+ \int_0^{ \f{1}{9}} |\D_{1D}| \sqrt{2 y_1} d y_1
+ (\pi + P( \f{1}{9})) \sqrt{2} ,
\eal
\eeq
and the estimate of $[u_y]_{C_x^{1/2}}$
\beq\label{eq:holx_uy}
\bal
 & 2^{-1/2} | u_y(-1, x_2) - u_y(1, x_2)|  \leq 2^{-1/2} \f{1}{\pi}  \B(  \td S_{1D}  [\om]_{C_x^{1/2}} + 4 (   S_{up} + S_{out} ) [\om]_{C_y^{1/2}} \B) , \\
&\td S_{1D} = \int_{\R_+} |\D_{1D}(y_1)| \sqrt{2 y_1}  d y_1 , 
\eal
\eeq
where $S_{out}$ and $S_{up}$ are given above, 
\beq\label{eq:holx_vx_A}
\bal
 & A  = \min(x_2, B), \quad  P(k) \teq  - \int_k^9 |\D_{1D}(y_1) | d y_1, \quad J_1 = [-9, 9] ,
\quad J_1^+ = [0, 9],  \\
& \D_{1D}(y_1) =  g_A(y_1 + 1)  \one_{ |y_1 + 1| \leq B} - g_A(y_1- 1) \one_{|y_1-1| \leq B}, 
\quad g_b(y) = \f{b}{y^2 + b^2},
\eal
\eeq
 $s_c(y_1) $ is given below, and $S_{in, y,  \e, \cdot }$ estimates the following integrals near the singularity
\beq\label{eq:holx_vx2}
\bal
& I_{\pm}(\e)   \teq \int_{y_1 \in J_{\e}} \int_0^{A} K_2(y_1 \pm 1, y_2) \td W(y) d y, \quad  
I(\e)  \teq  \int_{y_1 \in J_{\e}} \int_0^A \D(y) \td W(y) dy = I_{+}(\e) - I_-(\e),   \\
 &  |I_{\pm}(\e)| \leq S_{in, y, \e, \pm} \cdot [\om]_{C_y^{1/2}}, \quad |I(\e)| \leq S_{in, y, \e} \cdot [\om]_{C_y^{1/2}} , \quad 
   \td W(y)  = \om(y_1, x_2 - y_2) - \om( y_1, x_2)  , \\
  &  J_{\e} \teq [1-\e, 1 + \e]. 
\eal
\eeq

Here, the estimates are the same as Appendix \secholxvx  \cite{ChenHou2023a_supp}, while 
we change the definition of some variables by a constant factor, e.g. $S_{out}, S_{in, y}$, to simplify later 
presentation.

We will estimate $S_{in, y, \e, \pm}$ in Section \ref{sec:hol_near_sing}. Note that the above variables, e.g.$S_{out}, S_{in, x}$, are different from \cite{ChenHou2023a_supp} by a constant. Here, we have restricted the  
domain of the integral to $\R_2^{++}$ due to symmetry. The upper bounds for $v_x, u_y$ are the same as \cite{ChenHou2023a_supp}. The subscript \textit{ns} in \eqref{eq:holx_vx} is short for \textit{non-singular}. For $s_1, s_2 > 0$, the map $T$ and $s_c(s_1)$ are obtained from 
\[
\bal
&  \D(s) = K_{2, B}(s_1+1, s_2) - K_{2, B} (s_1 - 1, s_2), \quad \D(s_1, s_c(s_1)) = 0,  \\
&\int_{T}^{y_2} \D(y_1, s_2) ds_2 = 0,  \quad K_{2, B}(s) = \f{1}{2} \f{s_1^2 - s_2^2}{ |s|^4} \one_{ |s_1|, |s_2| \leq B} .
\eal
\]

\begin{figure}[t]
\centering
\begin{subfigure}{.35\textwidth}
  \centering
  \includegraphics[width=0.8\linewidth]{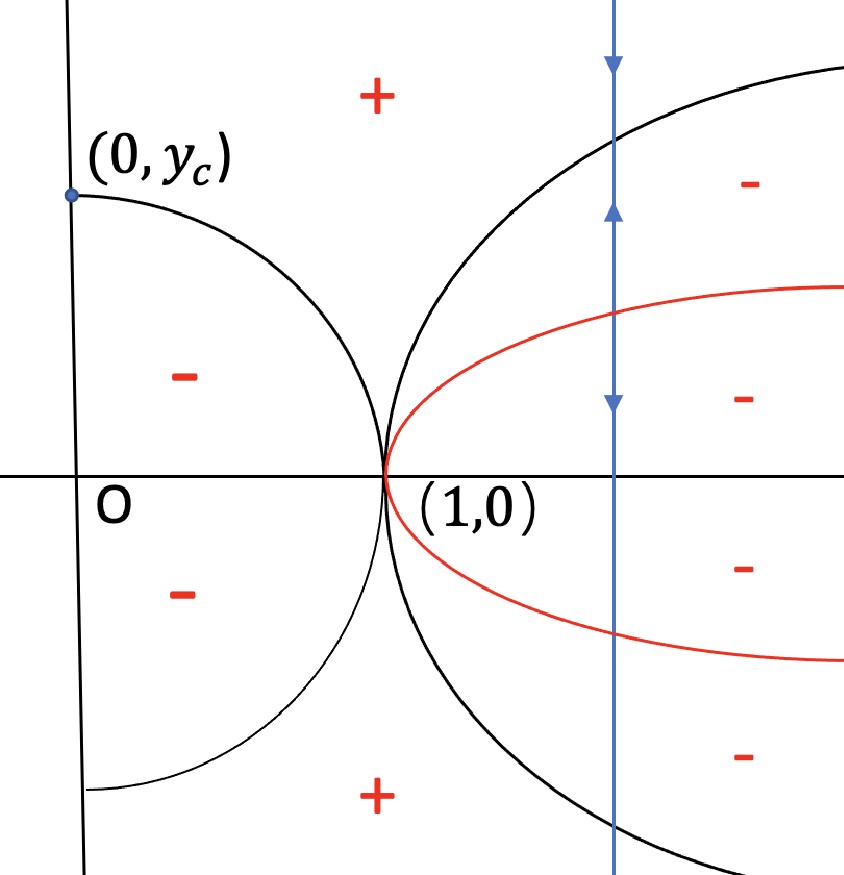}
    \end{subfigure}
        \begin{subfigure}{.48\textwidth}
  \centering
  \includegraphics[width=0.8\linewidth]{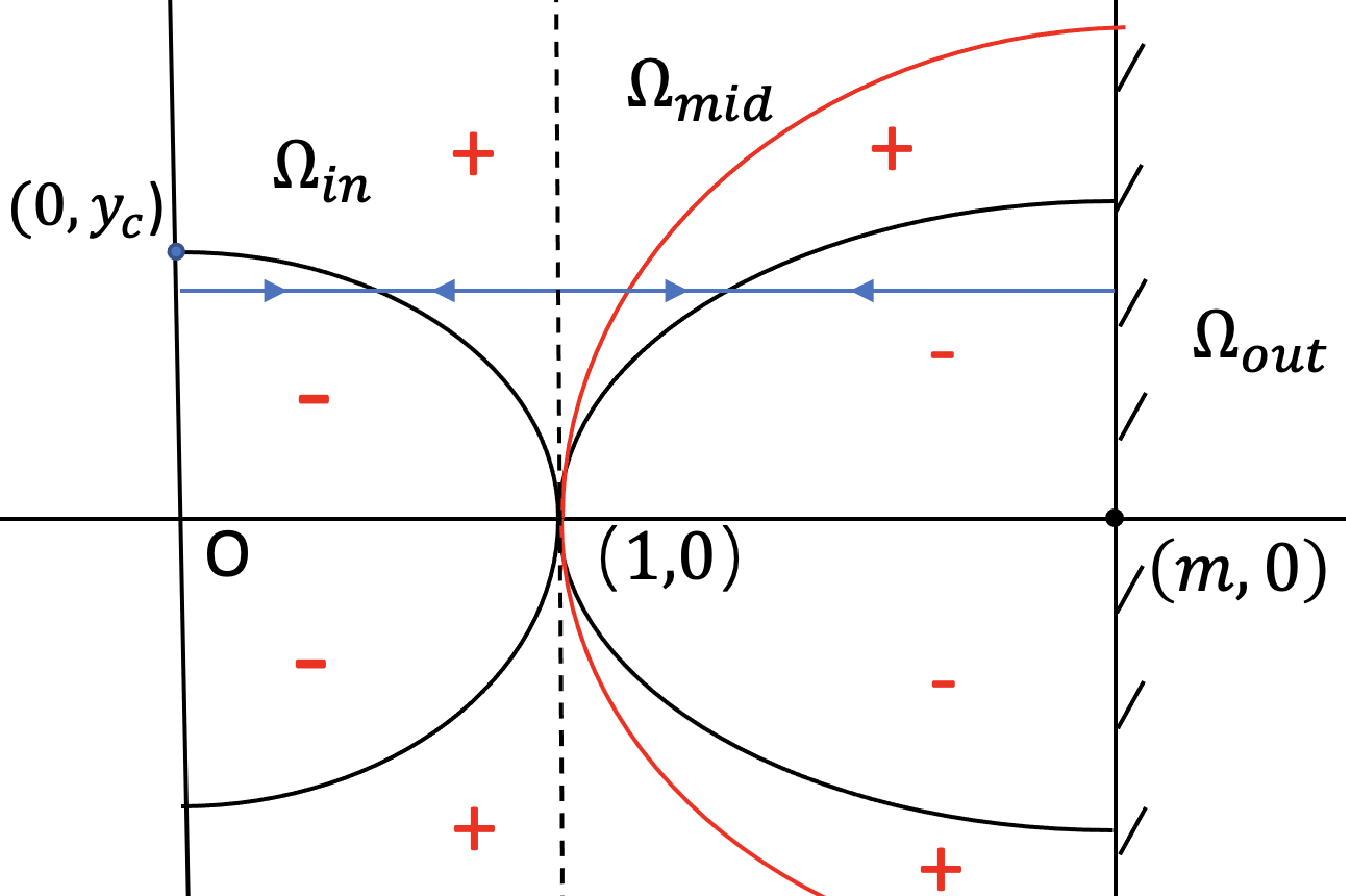}
    \end{subfigure}
\caption{
Illustration of the sign of the kernel $\D(y)$ and transportation plan. The sign of $\D(y)$ in different regions is indicated by $\pm$. 
The blue arrows indicate the direction of 1D transportation plan. 
Left for $C_x^{1/2}$ estimate: The black curve and the red curve represents $ y_2 = \pm  s_{c, in}(y_1)$, $y_2 = \pm T_{in}(y_1, A)$ for $y_1 \geq 0$, respectively. Right for $C_y^{1/2}$ estimate: The black curve is for $y_2 = \pm s_c(y_1)$, or equivalently $y_1 = h_c^-(|y_2|)$ (two left black curves) and $y_1 = h_c^+( |y_2| ) $ (two right black curves). The red curve represents $y_1 = \TTy(m, |y_2|)$. Note that these curves \textit{do not} agree with the actual functions.
 }
\label{fig:vx_OT}
\end{figure}

When $0\leq s_1 \leq B-1$, we have 
\beq\label{eq:holx_vx_del}
\D(s) = \D_{in}(s) = K_2( s_1 + 1, s_2) -  K_2(s_1 -1, s_2) , \quad s_c(s_1) = s_{c, in}(s_1) ,
\eeq
$s_{c, in}$ is given by \eqref{eq:vx_thres}, and the map $T$ is given by \eqref{eq:map_vx}, which is denoted as 
\[
T_{in} = T_{ \eqref{eq:map_vx}} .
\]
See the left figure in Figure \ref{fig:vx_OT} for an illustrations of the sign and the map in the inner regions. For $s_1 \in [B-1, B+1]$, we get 
\beq\label{eq:map_vx_out}
T_{out}(y) = \f{ (y_1-1)^2}{y_2} , \quad \D(s) = \D_{out}(s) = - K_2(s_1-1, s_2), \quad s_c(s_1) = s_1 - 1. 
\eeq
For $y_1 \geq 1$ or $y_1 \leq 1$ and $y_2 > 0$, the piecewise bounds for $T_{out}(y)$ are trivial, and $T_{out}$ is decreasing in $y_2$.

The above integral depends on two parameters $x_2, B$. Denote $A = \min(x_2, B)$. Our goal is to obtain a uniform bound for all $ 0 \leq A \leq  B \leq \infty, B \geq 2$. We partition these two parameters $0 = A_1 < ..< A_{n_1 } < A_{n_1 + 1} = \inf, 2 \leq B_1 < B_2< .. < B_{n_2+1} = \infty$, and estimate the bound for $A \in [A^l, A^u], B \in [B^l, B^u] $. We discuss how to estimate each part below.

\subsubsection{Bulk part $S_{in, y}$}
Using the above notation and the localization of the kernel, we have 
\beq\label{eq:int_holx_vx_in}
\bal
S_{in, y, ns}(\e) & = \int_{ [0, B-1] \bsh J_{\e} } \int_0^A    \B( \one_{y_2 \geq s_c(y_1) } 
| T_{in}(y) - y_2|^{1/2} + 
\one_{y_1 \leq 9} \one_{y_2 \leq  T_{in}(y_1, A)} y_2^{1/2} \B)  | \D_{in}(y) | dy  \\
& \quad  + \int_{B-1}^{B+1} \int_0^A 
 \B( \one_{y_2 \geq s_c(y_1) } 
| T_{out}(y) - y_2|^{1/2} + 
\one_{y_1 \leq 9} \one_{y_2 \leq  T_{out}(y_1, A)} y_2^{1/2} \B)  | \D_{out}(y) | dy  \\
& \teq S_{in, y1, ns} + S_{in, y2} ,
\eal
\eeq
where we have combined the integrals  
\[
\one_{y_1 \in J_1} \one_{y_2 \in [s_c(y_1), A]  } 
+ \one_{ y_1 \notin J_1} \one_{y_2 \in [s_c(y_1), A]  } 
\] 
related to $|T - y_2|^{1/2}$ in $S_{in, y}$  in \eqref{eq:holx_vx}. Since $T(y_1, y_2)$ is decreasing in $y_2$ (see Lemma \ref{lem:map_vx}, for $A \in [A^l, A^u], B\in [B^l, B^u]$, clearly, we have 
\[
\bal
S_{in, y1, ns} 
& \leq 
\int_{ [0, B^u-1] \bsh [1-\e, 1 + \e]}   \int_0^{A^u}    \B( \one_{y_2 \geq s_c(y_1) } 
| T_{in}(y) - y_2|^{1/2} + 
\one_{y_1 \leq 9}  \one_{y_2 \leq  T_{in}(y_1, A^l)} y_2^{1/2} \B)  | \D_{in}(y) | dy ,  \\
S_{in, y2} &  \leq 
\int_{B^l - 1}^ {B^u + 1} \int_0^{A^u}    \B( \one_{y_2 \geq s_c(y_1) } 
| T_{out}(y) - y_2|^{1/2} + 
\one_{y_1 \leq 9}  \one_{y_2 \leq  T_{out}(y_1, A^l)} y_2^{1/2} \B)  | \D_{out}(y) | dy . 
\eal
\]
We choose our mesh aligning with $y_1 = 9, y_1 = 1 \pm \e,  y_2 = A^u$ to partition the integrals so that for each small domain $Q = [y_1^l, y_1^u] \times [y_2^l, y_2^u]$, the restrictions 
are automatically imposed i.e.
\[
\one_{y_1 \leq 9} \one_{y_2 \leq A^u} \one_Q \one_{ [1-\e, 1+\e]}(y_1)  = \one_Q , \mathrm{ \ or \ } 0 .
\]

We follow the strategy in Section \ref{sec:int_hol_stg} to handle the indicators, the integral $\int_Q |\D_{\al}(s)|$, and $|T_{\al} - y_2|^{1/2}, \al = in, out$.  The analytic integral formula for $K_2(s)$ and $K_2(s) s_2^{1/2}$ is given in Section \ref{sec:int_ker}, and the estimate of $|T_{in} - y_2|^{1/2}$ is given in Section \ref{sec:holx_vx} after Lemma \ref{lem:map_vx}. The estimate of the second part $S_{in, y2}$ is much easier since the integrand is supported away from the singularity $(1, 0)$ since $B \geq 3$ and the map $T_{out}$ \eqref{eq:map_vx_out} and $s_c(y_1)$ \eqref{eq:vx_thres} are simple. 

For $S_{in, y1, ns}$, we apply the above strategy to estimate the integrals in $[0, 1-\e] \times [0, A], [1+ \e, \infty)$.
For the first integrand in $S_{in, y2}$, if $B^u - B^l$ is large, which is the case for large $B$ since we partition the domain of $B$ using adaptive mesh, the above estimate for $S_{in, y2}$  is not sufficient. 
Note that the second integrand in $S_{in, y2}$ vanishes for large $B$ since it is supported in $y_1 \leq 9$.

\vspace{0.1in}
\paragraph{\bf{Additional estimate for $S_{in, y2}$ }}

We consider another estimate for $S_{in, y2}$ by exploiting the boundedness of the interval $y_1 \in [B-1, B+1]$
\[
\bal
I & =  \int_{B-1}^{B+1} \int_0^A   \one_{y_2 \geq s_c(y_1) } 
| T_{out}(y) - y_2|^{1/2}    | \D_{out}(y) | dy   \\
& = \f{1}{2} \int_{B-1}^{B+1} \int_{0}^A \one_{y_2 \geq y_1-1} \B| \f{(y_1-1)^2}{y_2} - y_2 \B|^{1/2} |K_2(y_1-1, y_2)| dy =  \f{1}{2} \int_{B-2}^B \int_{ 0}^A \one_{y_2 \geq y_1} \B| \f{y_1^2}{y_2} - y_2 \B|^{1/2}  | K_2(y )| d y,
\eal
\]
in $S_{in, y2}$ in \eqref{eq:int_holx_vx_in}, where we have used 
$s_c(y_1) = y_1-1$ \eqref{eq:map_vx_out}, and a change of varialbes $y_1 \to y_1 + 1$. Since $A \leq B$ \eqref{eq:holx_vx_A} and 
\[
B-2 \leq y_1 \leq y_2 \leq A \leq B \leq y_1 + 2, \quad  y_1 + y_2 \leq  \sqrt{2} (y_1^2 + y_2^2)^{1/2}, \quad |y| \geq \sqrt{2} (B-2), 
\] 
in the support,  we get 
\[
\bal
\B| \f{y_1^2}{y_2} - y_2 \B|^{1/2} |K_2(y)| 
&= \B| \f{(y_2 - y_1)(y_2 + y_1)}{y_2} \B|^{1/2}  \f{ (y_2 - y_1)(y_2 + y_1)}{ |y|^4} 
\leq (2 \cdot 2 )^{1/2} 2 \sqrt{2} |y|^{-3} \\
&\leq 4 \sqrt{2} (\sqrt{2} )^{-3} (B-2)^{-3} = 2 (B-2)^{-3}.
\eal
\]
Since $A \leq B, B^l \leq B$, we yield
\[
I  \leq \f{1}{2} \int_{B-2}^B \int_{0}^A \one_{y_2 \geq y_1} 2 (B-2)^{-3} dy
\leq (B-2)^{-3} \int_{B-2}^B ( B- y_1)   d y_1  = 2 (B-2)^{-3} 
\leq 2 (B^l-2)^{-3}.
\]

\subsubsection{Near the singularity}\label{sec:hol_near_sing}

Near  $s_* = (1, 0)$, the integrand in $S_{in, y}$ \eqref{eq:holx_vx}, \eqref{eq:holx_vx2} is singular of order $|x|^{-3/2}$ and quite complicated. To ease our computation, we use another estimate and separate the estimate of two kernels in $\D(s) = K_2(s_1 +1, s_2) - K_2(s_1-1, s_2)$.  Below, we estimate $I_{\pm}(\e)$ from \eqref{eq:holx_vx2} and derive the bound $S_{in, y, \e, \pm}$. We fix $\e$ and then drop $\e$ for simplicity.

\vspace{0.1in}
\paragraph{\bf{(a) Regular part}}
Since the kernel $ K_2(y_1 + 1, y_2)$ is regular, using $K_2(y) = \f{1}{2} \f{y_1^2 - y_2^2}{|y|^4} $, \eqref{eq:holx_vx2},  
\beq\label{eq:holx_vx_Wt}
|\td W| \leq y_2^{1/2} [\om]_{C_y^{1/2}}, \ 
\int_{1-\e}^{1+\e}
|K_2(y_1 + 1, y_2) | d y_1 \leq \int_{1-\e}^{1+\e} \f{1}{2} |y|^{-2} d y_1 \leq  \f{\e}{(2-\e)^2 + y_2^2}, 
\eeq
 we get 
\beq\label{eq:holx_vx_mdp}
\bal
|I_+(\e)| \leq \int_{1-\e}^{1+\e } \int_0^A |K_2(y_1 + 1, y_2  )| y_2^{1/2} dy [\om]_{C_y^{1/2}}
\leq [\om]_{C_y^{1/2}} \int_0^{A^u} \f{\e y_2^{1/2} }{ (2-\e)^2 + y_2^2}  dy_2,
\eal
\eeq
where we have used the fact that the integral is increasing in $A$ and $A \in [A^l, A^u]$ in the last inequality. Following the strategy (d) in Section \ref{sec:int_hol_stg}, we partition the domain of the integral using the same mesh as that for $S_{in, y}$. In each interval $[y_2^l, y_2^u]$, we have 
\beq\label{eq:int_holx_vx_mid1}
\int_{y_2^l}^{y_2^u} \f{\e y_2^{1/2} }{ (2-\e)^2 + y_2^2}  d y_2 
\leq \sqrt{ y_2^u } \f{ \e }{ 2- \e } \arctan( \f{y_2}{2-\e} ) \B|_{y_2^l }^{y_2^u}.
\eeq

\vspace{0.1in}
\paragraph{\bf{(b) Singular part}}

Denote 
\[
a = \min(A, \e), \quad Q_a \teq [-a, a] \times [0, a], \quad Q_a^+ \teq [0, a]^2. 
\]

Recall $I_{-}(\e)$ from \eqref{eq:holx_vx2}. We have 
\[
\bal
| I_-(\e)| & \teq |\int_{1-\e}^{1+\e} \int_0^A K_2(y_1-1,y_2) \td W(y) dy| 
 = |\int_{-\e}^{\e} \int_0^A K_2( y) \td W(y_1 + 1, y_2) dy|  \\
 & \leq |\int_{Q_a} K_2(y_1,y_2)\td W(y_1 +1, y_2) dy| 
+ |\int_{ [-\e, \e] \times [0, A] \bsh Q_a } K_2(y_1 ,y_2) \td W(y_1 + 1, y_2) dy| 
\teq I_{-, 1} + I_{-,2} .
\eal
\]
For the second part, using the bound for $\td W$ \eqref{eq:holx_vx_Wt}, we get 
\[ 
| I_{-, 2} | \leq  \int_{ [-\e, \e] \times [0, A] \bsh Q_a } | K_2(y_1,y_2) | y_2^{1/2}  dy
= 2 \int_{[0, \e] \times [0, A] \bsh Q_a^+} |K_2(y)| y_2^{1/2} dy .
\]
If $A \geq A^l \geq \e$, we yield $a = \e, A \leq A^u$. Using $ K_2(y) = \f{1}{2} \f{y_1^2 - y_2^2}{|y|^4} \leq 0$ for $y_1 \leq \e \leq y_2$, we yield 
\[
I_{-, 2} \leq 2 \int_0^{\e} \int_{\e}^{A^u} |K_2(y)| y_2^{1/2} dy
 = \B|  \int_{\e}^{A^u} \f{y_1}{|y|^2} \B|_{y_2 = 0}^{\e} y_2^{1/2} d y_2 \B|
 = \int_{\e}^{A^u} \f{ \e}{\e^2 + y_2^2} y_2^{1/2} d y_2. 
\]
We use the same method as \eqref{eq:int_holx_vx_mid1} and $  \f{ \e}{\e^2 + y_2^2} = \f{d}{d y_2} \arctan(y_2 / \e)$ to estimate the integral. 

We choose $\e$ from the mesh for $A$. Then when $A^l < \e$, we get $A^u \leq \e$. (This is similar to $a < b, a, b \in \BZ$ implies $a + 1 \leq b$.) Since $ A \leq A^u \leq \e$, we get $ K_2(y) = \f{1}{2} \f{y_1^2- y_2^2}{|y|^4} \geq 0$ in $[A, \e] \times [0, A],$ 
\[
\bal
I_{-, 2} & \leq 2  \int_{A}^{\e} \int_0^{A} K_2(y) y_2^{1/2} dy  
= \int_0^A  - \f{y_1}{|y|^2} \B|_{A}^{\e} y_2^{1/2} d y_2
 = \int_0^A ( \f{A}{A^2 + y_2^2} - \f{\e}{\e^2 + y_2^2} ) y_2^{1/2}  d y_2 \\
 & = 2 A^{1/2} f_s(1 )
  - 2 \e^{1/2} f_s( (A/\e)^{1/2} )
 \leq 2 (A^u)^{1/2} f_s(1 )
  - 2 \e^{1/2} f_s( (A^l/\e)^{1/2} ),
\eal
\]
where we have used \eqref{eq:hol_vx_fh}, \eqref{eq:hol_vx_fh2}.

For $I_{-, 1}$ and a fixed $x_2$, using the scaling relation, \eqref{eq:holx_vx2} $[\td W( ay_1 , ay_2)]_{C_y^{1/2}(Q_1)} = a^{1/2} [ \td W]_{C_y^{1/2}(Q_a)} \leq a^{1/2} [\om]_{C_y^{1/2}}$, and 
\eqref{eq:Ck2_unit}, we get 
\[
|I_{-,1}| \leq  |\int_{Q_1} K_2( y_1, y_2) \td W( a y_1 +1, a y_2 ) dy |
\leq  2 C_{K_2} [ \td W( a y_1 +1, a y_2 ) ]_{C_y^{1/2}(Q_1)}
\leq 2 \min(\e, A)^{1/2} C_{K_2}  [\om]_{C_y^{1/2}} .
\]

\subsubsection{Estimate of $S_{out}$}

In this section, we estimate $S_{out}$ in \eqref{eq:holx_vx}, which is much easier than that of $S_{in, y}$. Note that if $B \leq x_2$,  we have $S_{out} = 0$. Thus, we consider $x_2  < B$, and get $A = \min(x_2, B) = x_2$.  Using \eqref{eq:holx_vx_del} and \eqref{eq:map_vx_out}, we yield 
\beq\label{eq:int_holx_vx_out}
 S_{out}(A, B) \teq  \f{ \sqrt{ 2 x_2 }}{2} \int_{\R_+} \int_{ x_2  }^B  |\D(y) | dy 
  =  \f{ \sqrt{ 2 A }}{2}  \int_{A}^B d s_2 (  \int_0^{B-1} |\D_{in}(s)|  d s_1 
  + \int_{B-1}^{B+1} |\D_{out}(s)| d s_1 ).
\eeq
For a fixed $y_2$, applying Lemma \ref{lem:comp_trunc} with $k=1, f(x) = K_2(x, y_2)$, $b=B,c = B^u$, and $B \leq B^u$ we obtain 
\beq\label{eq:int_holx_vx_out_mono1}
 \int_{B-1}^{B+1} | K_2(s_1 -1, s_2)| d s_1 
 \leq \int_{B-1}^{B^u- 1 } \B| K_2(s_1 -1, s_2) - K_2(s_1 +1, s_2) \B| d s_1
 + \int_{B^u-1}^{B^u+1} | K_2(s_1-1, s_2)| d s_1 .
\eeq

As a result, $S_{out}(A, B)$ is increasing in $B$ and we get $S_{out}(A, B) \leq S_{out}(A, B^u)$. Denote $I_{\e} = [1-\e, 1+ \e]$. Firstly, using $\D(s) = \D_{in}(s)$ for $s_1 \in I_{\e}, s_2 \in [A, B]$ , we have 
\[
S_{out} \leq  \f{ \sqrt{2 A^u} }{2}\int_{A^l}^{B^u} d s_2 \int_{s_1 \in [0, B^u+1] \bsh I_{\e} } |\D(s)| d s_1 
+  \f{ \sqrt{2 A} }{2} \int_{A}^{B} d s_2 \int_{I_{\e}}  |\D_{in}(s)| ds_1  
\teq S_{out, 1} + S_{out, 2}.
\]
We do not bound the integral region $[A, B]$ in $S_{out, 2}$ at this moment. The first part is away from the singularity. We partition the domain and apply the strategy in Section \ref{sec:int_hol_stg} to estimate it. In particular, 
the integral in $S_{out, 1}$ with $y_1 \in [B^u-1, B^u+1]$ is given by 
 \[
I = \int_{A^l}^{B^u} d s_2 \int_{B^u - 1}^{B^u+1} | K_2(s_1-1, s_2)| d s_1
 =  \int_{A^l}^{B^u} d s_2 \int_{B^u - 2}^{B^u} | K_2(s_1, s_2) |d s_1 .
 \]
We decompose the domain $(s_1, s_2)$ of the integral as follows
\[
\bal
{[} B^u - 2, B^u] \times [A^l, B^u] & =  [B^u-2, B^u] \times [ A^l,   B^u - 2] \cup [B^u-2, B^u]^2 \teq Q_1 \cup Q_2 , \quad A^l \leq B^u-2,  \\
[B^u - 2, B^u] \times [A^l, B^u] & =  [B^u-2, A^l] \times [ A^l,  B^u] \cup [A^l, B^u]^2\teq Q_1 \cup Q_2, 
\quad A^l >  B^u-2   .
\eal
\]
In  $Q_1$, since $y_1^2 - y_2^2$ has a fixed sign, $K_2(s)$ has a fixed sign. We apply the analytic formula for $K_2(s)$ to evaluate the integral in $Q_1$. In $Q_2$, we use the formula \eqref{eq:hol_vx_diag}.

For $S_{out, 2}$, we denote $A_{\e, l} = \max(\e, A^l)$. We decompose $S_{out, 2}$ as follows 
\[
S_{out, 2} \leq \f{ \sqrt{2 A}}{2} \int_{A_{\e, l}}^B d s_2 \int_{1-\e}^{1+\e}|\D_{in}(s)| d s_1 
+ \one_{\e > A} \f{ \sqrt{2 A}}{2} \int_{ A}^{\e} d s_2 \int_{ 1-\e}^{1+\e} |\D_{in}(s)| d s_1 \teq I_1 + I_2.
\]
Note that $I_2 =0$ if $A \geq \e$.  From \eqref{eq:vx_thres_u}, in the support of the integral $I_1$, we have $y_2 \geq \e \geq |y_1 -1| \geq s_c(y_1)$ \eqref{eq:vx_thres_u} and yield that $\D(s)$ has a fixed sign. We estimate $I_1$ as follows 
\[
I_1 \leq \f{ \sqrt{2 A^u} }{2} \B|  \int_{A_{\e, l}}^{B^u} d s_2 \int_{1-\e}^{1+\e} \D(s) d s_1 \B|, 
\]
and evaluate the integral using the analytic integral formula for $\D(s)$. For $I_2$ 
with $\e > A$, we use the fact that $\sqrt{A} \leq s_2^{1/2}$ in the support of the integral and triangle inequality to get 
\[
I_2 \leq  \f{\sqrt{2}}{2} \int_A^{\e} d s_2 \int_{1-\e}^{1+\e} |K_2(s_1-1, s_2)| s_2^{1/2} d s_1
+ \f{ \sqrt{2 A^u}}{2} \int_{A^l}^{\e} d s_2 \int_{1-\e}^{1+\e}  | K_2(s_1 + 1, s_2) | d s_1 = I_{21} + I_{22}. 
\]
For $I_{22}$, we choose $\e < \f{1}{100}$. In the support of the integral, $K_2 > 0 $ and we get 
\[
I_{22} =\f{ \sqrt{2 A^u}}{2}| \int_{A^l}^{\e} d s_2 \int_{1-\e}^{1+\e} K_2(s_1 + 1, s_2) d s_1 | ,
\]
and can evaluate the integral using the analytic integral formula. For $I_{21}$, we have 
\[
I_{21} \leq \sqrt{2} \int_{A^l}^{\e} d s_2 \int_{0}^{\e} | K_2(s_1, s_2)| s_2^{1/2} d s_1 
\leq  \sqrt{2} F_{K_2,h }(  0, \e, A^l, \e) ,
\]
where $F_{K_2, h}$ is defined in \eqref{eq:hol_vx_hy}.

\subsubsection{Integral $S_{in, y, \e},  S_{in, y, ns}, S_{out}$ in the far-field}

The previous method can handle the case that $A \leq B < \infty$. Next, we consider the remaining case $A \leq B, \  B^l \leq B \leq \infty$ with $B^l$ sufficiently large, e.g. $B^l = 10^8$. 
Recall $S_{in, y, ns}$ from \eqref{eq:holx_vx}, 
\eqref{eq:holx_vx2} ,\eqref{eq:int_holx_vx_in} 
\beq\label{eq:int_holx_vx_in2}
\bal
S_{in, y, ns} = 
\int_0^A  d y_2 \big( 
\int_0^{B-1}   I_{in, 1}  | \D_{in}(y) | dy_1 
+  \int_{B-1}^{B+1}  I_{in, 2}  | \D_{out}(y) | dy_1   \big)
\teq \int_0^A I_{in}(y_2) dy_2 ,
\eal
\eeq
where  $I_{in, 1}, I_{in, 2}$ denotes the integrand,
\[
\bal
I_{in, 1} &= ( \one_{y_2 \geq s_c(y_1) }  | T_{in}(y) - y_2|^{1/2} +  \one_{y_1 \leq 9} \one_{y_2 \leq  T_{in}(y_1, A)} y_2^{1/2} ) \one_{ y_1 \notin J_{\e} } \\
I_{in, 2} &= ( \one_{y_2 \geq s_c(y_1) }  | T_{out}(y) - y_2|^{1/2} + \one_{y_1 \leq 9} \one_{y_2 \leq  T_{out}(y_1, A)} y_2^{1/2} )  \one_{ y_1 \notin J_{\e} } , \quad J_{\e} = [1-\e, 1+\e].
\eal
\]
We decompose the integral in $y_1$ as follows 
\[
\int_0^{B-1}   I_{in, 1}  | \D_{in}(y) | dy_1 
= \int_0^{B^l-1}    I_{in, 1}  | \D_{in}(y) | dy_1 
+ \int_{B^l-1}^{B-1}  I_{in, 1}  | \D_{in}(y) | dy_1 
\teq I_2 + I_3.
\]
Since for $A \geq y_2 \geq s_c(y_1)$, we have $y_2 \geq s_c(y_1) \geq T(y_1, y_2) \geq T(y_1, A) $, clearly, we have 
\beq\label{eq:int_holx_vx_in4}
|I_{in, 1}| \leq  |y_2|^{1/2} \one_{ y_1 \notin J_{\e}}, \quad I_{in, 2} \leq |y_2|^{1/2} \one_{ y_1 \notin J_{\e}}. 
\eeq

Using \eqref{eq:int_holx_vx_in4} and combining $I_3$ and the integral of $I_{in, 2}$ in \eqref{eq:int_holx_vx_in2}, we further obtain
\[
I_3 + \int_{B-1}^{B+1} I_{in, 2} |\D_{out}(y) | dy_1 
\leq  y_2^{1/2} ( \int_{B^l - 1}^{B-1}  |\D_{in}(y)|  d y_1+  
\int_{B-1}^{B+1}  |\D_{out}(y) | d y_1 ). 
\]
Using the monotonicity Lemma \ref{lem:comp_trunc}, 
we get 
\[
I_3 + \int_{B-1}^{B+1} I_{in, 2} |\D_{out}(y) | d y_1 
\leq y_2^{1/2} \int_{B^l - 1}^{\infty} |\D_{in}(y)| d y_1.
\]
It follows 
\beq\label{eq:int_holx_vx_in3}
\bal
S_{in, y, ns} \leq \int_0^A  d y_2 \int_0^{B^l-1} I_{in, 1} |\D_{in}| d y_1 
+ \int_0^A y_2^{1/2} d y_2 \int_{B^l - 1}^{\infty} |\D_{in}(y)| d y_1
\eal
\eeq
for any $B \geq B^l$ and a fixed $A$. The first part is the main part.

Applying a similar argument to $S_{out}$ in \eqref{eq:int_holx_vx_out}, we yield 
\[
\bal
 S_{out}(A, B) 
  &= \f{ \sqrt{2 A} }{2} \int_{A}^B d s_2 (  \int_0^{B-1} |\D_{in}(s)|  d s_1 
  + \int_{B-1}^{B+1} |\D_{out}(s) | d s_1 )
   \leq \f{ \sqrt{2 A} }{2} \int_A^{B}  d s_2  \int_0^{\infty} |\D_{in}(s)| d s_1  , \\
   \eal
\]
where $a \vee b = \max(a, b)$. Using $\sqrt{2 A} \leq 2 s_2^{1/2}$ in the support of the integral in $S_{out}$ and $I_{in, 1} \leq y_2^{1/2}$, we can bound the integrands in $S_{in, y, ns}, S_{out}$ by $ y_2^{1/2} |\D(y)|$. 
Denote the range of $(y_1, y_2)$:
\[
 R(B^l) \teq  [0, B^l-1]\times [0, B^l]  , \quad
\quad A_2 \teq \min(A, B^l). 
\]
For $I(\e), S_{in, y, \e}$ \eqref{eq:holx_vx2}, using $|\td W | \leq y_2^{1/2} [\om]_{C_y^{1/2}}$
\eqref{eq:holx_vx_Wt}
we perform a similar decomposition 
\[
|I(\e)  | 
\leq | \int_{y_1 \in J_{\e}} \int_0^{A_2} \D(y) \td W(y) dy |
+  | \int_{y_1 \in J_{\e}} \int_{A_2}^A y_2^{1/2} |\D_{in}(y)| dy | [\om]_{C_y^{1/2}}
\teq I(\e, A_2) + S_{in, y, \e, out}[\om]_{C_y^{1/2}}.
\]

We treat the integrals in $S_{in, y, ns},   S_{out}$ beyond $R(B^l)$ 
and $S_{in, y, \e, out}$ as error $\cR$ and yield 
\[
\bal
& S_{in, y, ns} + S_{in, y, \e, out}
+ S_{out} 
\leq  (\int_0^A + \int_A^B) d y_2 \int_0^{\infty} \one_{ R(B^l)^c}(y) y_2^{1/2} |\D_{in}(y)| d y_1 \\
& + \int_0^A d y_2 \int_0^{B^l-1} \one_{ R(B^l)}(y) I_{in, 1} |\D_{in}(y)| d y_1  
  + \f{ \sqrt{2A} }{2} \int_A^B d y_2 \int_0^{\infty} \one_{ R(B^l)}(y) |\D_{in}(y)| d y_1 
\teq \cR + M_{in} + M_{out},
\eal
\]
where we have used $A^{1/2} \leq y_2^{1/2}$ for $y_2 \geq A$. Note that the range for 
$(y_1, y_2)$ satisfies
\[
\bal
& [0, A] \cup [A, B] = [0, B], \quad
( [0, B^l-1] \times [0, A]  ) \cap R(B^l) = [0, B^l-1] \times  [0, A_2]  \teq R_2(B^l), \\
& ([0, \infty) \times [A, B]  ) \cap R(B^l)
= [0, B^l-1] \times ( [A, B] \cap [0, B^l] ) 
 \subset  [0, B^l-1] \times [A_2, B^l]  ,
 \eal
\]
where we use $[A, B] \cap [0, B^l] = \emptyset$ if $A > B^l$ in the last inclusion. Since $M_{out} = 0$ if $A > B^l$, we get
\[
\bal
& S_{in, y, ns} + S_{out} + S_{in, y, \e, out} \\
& \quad \leq \int_{R(B^l)^c} y_2^{1/2} |\D_{in}(y)| d y
+ \int_0^{A_2} d y_2 \int_0^{B^l-1} I_{in, 1} |\D_{in}(y)| d y_1
+\f{ \sqrt{2 A_2} }{2}\int_{A_2}^{B^l} d y_2 \int_0^{B^l-1} |\D_{in}(y)| d y_1 .
\eal
\]
The estimate of the last two integrals and $I(\e, A_2)$  defined above (see also \eqref{eq:holx_vx2}) reduce to the case $ A_2 \leq  B^l$ studied in the previous section. The first term is very small and treated as an error term. 
For $|y| \geq B^l-1$, applying Lemma \ref{lem:K_mix} with $i+j = 3$ and mean value theorem yields
\[
|\D_{in}(y) | = |2 \pa_{y_1} K_2(\xi, y_2)| \leq \f{2}{ (\xi^2 + y_2^2)^{3/2}}
\leq C_B |y|^{-3}, \quad C_B = 2 (\f{B^l-1}{B^l-2})^3,
\]
for some $\xi \in [y_1 - 1, y_1 + 1]$, where we have used $|(\xi, y_2)| \geq  |y| - 1 \geq |y|(1 - \f{1}{B^l- 1} )$. Using $y_2^{1/2} \leq |y|^{1/2}$, we estimate the remaining part as follows
\beq\label{eq:hol_K2_far}
\bal
\int_{R(B^l)^c} y_2^{1/2} |\D_{in}(y)| d y  &\leq C_B \int_{|y| \geq B^l-1, y_1, y_2 \geq 0} |y|^{-5/2} d y
\leq C_B \int_{B^l-1} r^{-3/2} dr \int_0^{ \pi/2} 1 d \b  \\
& = C_B 2 (B^l-1)^{-\f{1}{2} } \cdot \f{\pi}{2} =
  C_B \pi  (B^l-1)^{-\f{1}{2} }.
  \eal
\eeq

\subsubsection{Estimate of $S_{in, x}$ and $S_{1D}$}
Recall $S_{in, x}, S_{1D}$ from \eqref{eq:holx_vx}. Denote $J_2 = [1/9, 9]$. Clearly, we have 
\[
\bal
F(A, B) = S_{in, x} + S_{1D}
= \int_{ y_1 \notin J_2, y_1 \geq 0 }  \sqrt{2 y_1} | \D_{1D}(y_1)| d y_1 
 + 2 \int_{ \f{1}{9}}^9 | \D_{1D} |\cdot |y_1-1|^{1/2} d y_1 
+ (\pi + P( \f{1}{9})) \sqrt{2} ,
\eal
\]
where $\D_{1D}(y_1), P(\f{1}{9})$ are given in \eqref{eq:holx_vx_A}. From the definition, we obtain $\D_{1D} < 0$ for $y_1 >0$. First, we show that $F$ is increasing in $B$. Using \eqref{eq:holx_vx_A} and $B \geq 3$, we get
 \beq\label{eq:holx_vx_1D}
 \bal
F(A, B) 
& = \pi \sqrt{2} + \int_0^{\infty} |\D_{1D}(y_1)|  h(y_1)d y_1 \\
&= \pi \sqrt{2} + \int_0^{B-1} |g_A(y_1+1) - g_A(y_1 - 1)|  h(y_1) + \int_{B-1}^{B+1}| g_A(y_1-1)|  h(y_1) d y_1 \\
h(y_1)  & =  \sqrt{2 y_1} \one_{J_2^c}(y_1) - \sqrt{2} \one_{J_2}(y_1)
+ 2 |y_1-1|^{1/2} \one_{J_2}(y_1) ,
\eal
 \eeq
 where $J_2 = [1/9, 9]$, and the negative term comes from $P(\f{1}{9})$. The function $h(y_1)$ satisfies
 \[
h(y_1) = 2|y_1-1|^{1/2} - \sqrt{2},  \ y_1 \in [1/9, 9], \quad h(y_1) = |2 y_1|^{1/2}, \ y_1 > 9 , y_1 \in [0, 1/9]. 
 \]
Since $h(9-) = 2 \sqrt{8} - \sqrt{2} =3 \sqrt{2} = h(9+) $, and $h(y_1)$ is increasing on $[2, 9]$ and $[9, \infty)$, we obtain that $h(y_1)$ is increasing. Using the monotonicity Lemma \ref{lem:comp_trunc}, we get 
\[
F(A, B) \leq F(A, \infty).
\]
Thus, it suffices to consider the case $B = \infty$, where we do not have localization of $\D_{1D}$ in \eqref{eq:holx_vx_A}. 

Now, we fix $A \in [A^l, A^u]$. We partition the integral in \eqref{eq:holx_vx_1D} with $B =\infty$ into 
\[
D_1 = [0, 1/9], D_2 = [1/9, 1-\e], D_3 =[1-\e, 1 + \e], D_4 = [1 + \e, 9], D_5 = [9, R_0], D_f = [R_0, \infty) ,
\]
for some $\e < 1/4$. In each $D_i$, we can simplify the function $h(y_1)$ and $h(y_1)$ is monotone. Thus, we can obtain its piecewise bounds easily. Next, we estimate the integral of $\D_{1D}$ uniformly for $A \in [A^l, A^u]$. Using $\D_{1D}(y_1) < 0$ for $y_1 > 0$ and 
\[
\bal
|\D_{1D}(y_1, A) | &= \f{4 A y_1}{ ( (y_1+1)^2 + A^2 ) ( (y_1-1)^2 + A^2 )}
\leq \f{A^u}{A^l} \f{4 A^l y_1}{ ( (y_1+1)^2 + (A^l)^2 ) ( (y_1-1)^2 + (A^l)^2 )} \\
 & = -\f{A^u}{A^l} \D_{1D}(y_1, A^l) 
  =  \f{A^u}{A^l}  \f{ d }{d y_1} ( \arctan \f{y_1-1}{A^l} - \arctan \f{y_1+1}{A^l} ),
 \eal
\]
we get 
\[
\int_a^b |\D_{1D}(y_1, A)| d y_1 
\leq \f{A^u}{A^l}  ( \arctan \f{y_1-1}{A^l} - \arctan \f{y_1+1}{A^l} ) \B|_{a}^{b}. 
\]

We apply the above estimate and the strategy in Section \ref{sec:int_hol_stg} to estimate the integral in the finite domains $D_1, D_2, D_4, D_5$ away from the singularity $1$. When $A$ is small and near the singularity $1$, the above formula can be $\infty$. Denote $P = (y_1-1)^2 + A^2, Q= (y_1+1)^2 + A^2$. Using the above formula for $|D_{1D}|$, we have 
\[
\f{1}{4 y_1} \pa_A  |\D_{1D}| 
= \pa_A \f{A}{PQ} = \f{1}{PQ} (1 - A \f{P^{\pr}}{P} - A \f{Q^{\pr}}{Q}  )
=  \f{1}{PQ} (1 - 2 A^2 (  \f{1}{P} + \f{1}{Q}) ) .
\]
For $|y_1-1| \geq \sqrt{3} A$, we get $Q \geq P \geq 4 A^2$ and yield 
\[
2A^2( P^{-1} + Q^{-1}) \leq 2 A^2 \cdot 2 (4 A^2)^{-1} = 1, \quad \pa_A |\D_{1D}| \geq 0.
\]
Therefore, if $\min( |a-1|, |b-1|) \geq \sqrt{3} A $, we have an improved estimate 
\[
\int_a^b |\D_{1D}(y_1, A)|  \leq \int_a^b |\D_{1D}(y_1, A^u)|
 = - \int_a^b \D_{1D}(y_1, A^u), 
\]
which can be evaluated using the analytic integral formula. 

For the term involving $P(1/9)$, using the formula of $\int \D_{1D}(s)$, we get 
\[
\bal
P(1/9,  A)  & =  \arctan \f{8}{A} - \arctan \f{10}{A} - \arctan \f{-8}{ 9 A}  + \arctan \f{10}{ 9A}, \\
  \pa_A P(1/9, A) & = -  \f{ 13120 (6400 + 6724 A^2 + 243 A^4) }{ ( A^2 + 64) (A^2 + 100 )( 64 + 81 A^2)(100 + 81 A^2 )} <0.
\eal
\] 
Thus $P(1/9, A)$ is monotone in $A$ and 
\[
\pi + P(1/9, A) \leq \max( \pi + P(1/9, A^l), \pi + P(1/9, A^u) ).
\]

In the domain $D_3$, for $A$ very small, we handle the singularity near $(A, y_1) = (0,0)$ as follows 
\[
\bal
I_3 & = \int_{1-\e}^{1+\e} ( g_A(1- y_1 ) - g_A(1 + y_1) |y_1-1|^{1/2} d y_1 
\leq \int_{1-\e}^{1+\e} ( g_A(1- y_1  ) |y_1-1|^{1/2 } d y_1 
 = 2 \int_0^{\e} \f{A}{A^2 + y_1^2} y_1^{1/2} d y_1  \\
&= 2 A^{1/2} \int_0^{\e / A} \f{1}{y_1^2 +1} y_1^{1/2 } dy_1  
\leq 2 (A^u)^{1/2}  \int_0^{\e / A^l} \f{1}{y_1^2 +1} y_1^{1/2 } dy_1
 = 4 (A^u)^{1/2}  f_s( (\e / A^l)^{1/2}), 
 \eal
\]
where $f_s$ is the function defined in \eqref{eq:hol_vx_fh}.

In the far-field $y_1 \geq R_0$, we have 
\[
|\D_{1D}(y_1)|  = |2 \pa_{y_1} g_A( \xi)| = |\f{4 A \xi}{ (A^2 + \xi^2)^2} |
\leq \f{2}{\xi^2 + A^2} \leq 2 (y_1-1)^{-2} ,
\]
for some $\xi \in [y_1-1, y_1+1]$. Since $y_1 \leq (y_1-1) \f{R_0}{R_0-1}$ for $y_1 \geq R_0$, we yield 
\[
I_f = \int_{ R_0}^{\infty} |\D_{1D}(y_1)| \sqrt{2 y_1} dy_1
\leq 2 (\f{2 R_0}{R_0-1})^{1/2} \int_{R_0}^{\inf} (y_1-1)^{-2 + 1/2} d y_1
=  4 (\f{2 R_0}{R_0-1})^{1/2} (R_0-1)^{-1/2}.
\]

Finally, we handle the case $A$ sufficiently large and can be $\infty$.  For $A \geq A^l$ with $A^l$ sufficiently large, e.g. $A^l = 10^8$, we use $h(y_1) \leq |2 y_1|^{1/2}$ \eqref{eq:holx_vx_1D} and the identity \eqref{eq:holx_vx_iden} below to get 
\[
\bal
F(A, B) \leq F(A, \infty) 
&= \sqrt{2} \pi + \int_0^{\infty} |\D_{1D}(y_1)| \sqrt{2 y_1} d y_1 
=  \sqrt{2} \pi + \pi \sqrt{2} \sqrt{ \sqrt{ A^2 + 1} - A}  \\
&= \sqrt{2} \pi (1 + ( \sqrt{A^2 + 1} + A )^{-1/2} )
\leq \sqrt{2} \pi (1 + ( \sqrt{ (A^l)^2 + 1} + A^l)^{-1/2} ),
\eal
\]
where we have used $ \sqrt{A^2+1} - A = \f{1}{ \sqrt{A^2+1} + A} $. 

Combining the above estimates, we complete the estimate of the integrals in \eqref{eq:holx_vx} for $[v_x]_{C_x^{1/2}}$.

The estimate of $[u_y]_{C_x^{1/2}}$ in \eqref{eq:holx_uy} is similar and simpler. We have estimated the terms in $S_{low}$ \eqref{eq:holx_uy} in our estimate of $S_{in, y}$ \eqref{eq:holx_vx}. The estimate of $\td S_{1D}$ is similar. Using the monotonicity Lemma \ref{lem:comp_trunc}, the identity \eqref{eq:holx_vx_iden} below, we get 
\[
\bal
\td S_{1D} & \leq \int_0^{\infty} (g_A(y_1 - 1) - g_A(y_1+1) ) ( 2y_1)^{1/2} d  y_1 
= \sqrt{2} \pi \sqrt{ \sqrt{1 + A^2} - A}  \\
& =  \sqrt{2} \pi  ( \sqrt{1 + A^2} + A )^{-1/2}   \leq \sqrt{2} \pi ( \sqrt{1 + (A^l)^2} + A^l )^{-1/2}.
\eal
\]

\vs{0.1in}
\paragraph{\bf{An integral identity}} We prove the following identity using the residue formula 
\beq\label{eq:holx_vx_iden}
M = \int_0^{\infty} (g_A(y_1 - 1) - g_A(y_1+1) ) y_1^{1/2} d  y_1 
= \pi \sqrt{ \sqrt{1 + A^2} - A} . 
\eeq
Using a change of variable $y_1 = s^2$ and writing the integral on $\R$ using the symmetry, we get 
\[
M =   \int_{-\infty}^{\infty} s^2 ( \f{A}{ (s^2 -1)^2 + A^2} - \f{ A}{ (s^2 + 1)^2 + A^2} )  ds
\teq M_1 - M_2.
\]

The integrand is analytic except a few poles. In the upper half plane, we have poles 
\[
s_1 = (1 + i A)^{1/2} = r e^{i \th}, \quad s_2 = r e^{i( \pi - \th) } = - \bar s_1, \  s_2^2 = \bar s_1^2 = 1 - i A,
\]
for some $\th \in [0, \pi]$. Denote $f(s) = (s^2 - 1)^2  + A^2$. We have $ \pa_s f = 4 s (s^2-1)$. Applying the residue formula, we get 
\[
\bal
M_1 & = 2 \pi i ( \f{ A s_1^2}{f^{\prime}(s_1)} 
+  \f{ A s_2^2}{f^{\prime}(s_2)}  )
 = \f{1}{2} \pi i A ( \f{s_1}{s_1^2-1} + \f{s_2}{s_2^2-1})
  = \f{\pi i A}{2} ( \f{s_1}{i A} + \f{- \bar s_1}{ - iA} )
  = \f{\pi}{2} (s_1 + \bar s_1). \\
  \eal
\]
Since $(s_1 + \bar s_1)^2 = s_1^2 + \bar s_1^2 + 2 |s_1|^2 = 2 + 2 (1 + A^2)^{1/2} $ and $M_1>0$, we get $M_1 = \f{\pi}{2} \sqrt{ 2 +2 \sqrt{1 + A^2} }$. For $M_2$, the computation is similar. Let $ -1 + i A = r_2^2 e^{i \th_2}, \th_2 > 0$. We have poles in $\R_2^+$
\[
s_3 = r_2 e^{i \th_2 /2},  \ s_3^2 = -1 + iA, \quad  s_4 = r_2 e^{ i (\pi -\th_2/2) }
= - \bar s_3, \quad s_4^2 = r_2^2 e^{-i \th_2} = -1 - iA,
\]
and apply the residue formula to get 
\[
M_2 
= \f{1}{2} \pi i A( \f{s_3}{s_3^2 + 1} + \f{s_4}{s_4^2 + 1} )
= \f{\pi}{2} (s_3 + \bar s_3)
= \f{\pi}{2} (-2 + 2 \sqrt{1+A^2})^{1/2}.
\]
Denote $G = (1+A^2)^{1/2}$. Using the above identities and $G^2 - 1 = A^2 $, we prove
\[
\bal
M &= M_1 - M_2 = \sqrt{2} \f{\pi}{2}( (G+1)^{1/2} - (G-1)^{1/2} )
 =  \sqrt{2} \f{\pi}{2} ( (G+1) + (G-1) - 2 (G^2 - 1)^{1/2}  )^{1/2} \\
 &= \sqrt{2} \f{\pi}{2} (2 G - 2 A  )^{1/2}
 = \pi ( (1 + A^2)^{1/2} - A )^{1/2}.
 \eal
\]

\subsection{Estimate of the constant for $[u_y]_{C_y^{1/2}}, [v_x]_{C_y^{1/2}}$}

Recall from Appendix {\secholyvx} in \cite{ChenHou2023a_supp} the estimate of $[u_y]_{C_y^{1/2}}, [v_x]_{C_y^{1/2}}$
\beq\label{eq:holy_vx_est}
\bal
 \f{| u_y(z) - u_y(x) | }{\sqrt{2}} & \leq \f{1}{\pi \sqrt{2}}
 \B( ( \td C_{in}(\e) + \min( \td C_{mid, 1}(m, \e), C_{mid, 3}(m) ) ) [\eta]_{C_x^{1/2}}
+ C_{ out}(m, \e)  [\eta]_{C_y^{1/2}}  \B) , \\
\f{ |v_x(z) - v_x(x) |}{\sqrt{2}} & \leq 
\f{1}{\pi \sqrt{2}} \B( \td C_{ in}(\e) [\eta]_{C_x^{1/2}} 
+  C_{out}(1, \e)  [\eta]_{C_y^{1/2}} + \f{\pi}{2} \sqrt{2} [\eta]_{C_x^{1/2}}   \B) , \\
\eal
\eeq
where $\eta$ is a rotation of the original variable $\om$ and the upper bounds are given by 
\beq\label{eq:holy_vx}
\bal
& \td C_{in}(\e) = C_{in, \e, ns} + C_{in, \e}, \quad \td C_{mid, 1}(m, \e) = C_{mid, 1, \e, ns}(m) + C_{mid, 1, \e}(m), \\
& C_{in, \e, ns} =  4 \int_{\e}^{y_c} d y_2 \int_0^{ h_c^-(y_2)}  |\D(y)|  |y_1 - \TTy(y)|^{1/2} d y_1  
+  2 \int_{R^{++}_{in, \e}} |\D(y) | \sqrt{2 y_1} dy , \\
&  C_{mid, 1, \e, ns}(m)  =  4 \int_{ 0 }^{ s_c(m)} \one_{y_2 \geq \e} d y_2 \int_{ h_c^+(y_2)}^m  |\D(y)|  |y_1 - \TTy(y)|^{1/2} d y_1   +  2 \int_{R^{++}_{mid}, y_2 \geq \e} |\D(y) | \sqrt{2 y_1} dy  ,\\
& C_{mid, 2}(\e)  = 4 \int_{1+\e}^{\inf} d y_1\int_{ s_c(y_1)}^{\inf} |\D(y)| |y_2 - T(y)|^{1/2} d y_2  ,\\ 
& C_{out}(m, \e)  = C_{mid, 2}(\e_m - 1) + 4 ( I_{K_2,\inf}( m+1,  \e_{ m}  + 1) + I_{K_2, \inf}(m-1, \e_{ m}-1) ),      \quad \e_{ m} = \max(m, \e + 1), \\
& C_{out}(1, \e) = C_{mid, 2}(  \e) + 4 ( I_{K_2,\inf}( 2,  2  + \e) + I_{K_2, \inf}(0, \e) ) ,
\eal
\eeq
and we have replaced the dummy parameter $\d$ used in \cite{ChenHou2023a_supp} by $\e$. 
The value of $\e$ can be different for different variables. The subscript $ns$ means non-singular. The functions $C_{mid, 3}(m), C_{in, \e}, C_{mid, 1, \e}(m)$ are upper bounds of the following integrals 
\beq\label{eq:holy_vx2}
\bal
& S_{in, \e}  \teq \lim_{\d \to 0}  \int_{ \d \leq |y_2| \leq \e ,|y_1|\leq 1} \D(y) \eta_m(y) dy , \quad 
|S_{in , \e}|\leq C_{in, \e} [\eta]_{ C_x^{1/2}},  \\
& S_{mid, 1, \e}  \teq  \lim_{\d \to 0} \int_{ \d \leq |y_2| \leq \e, 1 \leq |y_1| \leq m } \D(y) \eta_m(y) dy , \quad
 |S_{mid, 1, \e} |\leq C_{mid, 1, \e}(m)  [\eta]_{C_x^{1/2} } ,  \\
& S_{mid, int} \teq  \lim_{\d \to 0} \int_{ \d \leq  |y_2| , |y_1| \in [1, m]} \D(y) \eta_m(y) d y , \quad 
| S_{mid, int} | \leq  C_{mid, 3}(m) [\eta]_{C_x^{ \f{1}{2}} } ,  
\eal
\eeq
and will be established in Sections \ref{sec:holy_vx_in}, \ref{sec:holy_vx_mo}. In \eqref{eq:holy_vx}, $I_{K_2, \inf}$ is defined in \eqref{eq:hol_strip_spec_K2}, and $\TTy$ is the transport map in $x$ direction given in \eqref{eq:map_vx_cy}, $T$ is the previous transport map in $y$ direction \eqref{eq:map_vx}  in Section \ref{sec:holx_vx} without localization, $s_c$ is given in \eqref{eq:vx_thres}. In particular, $T$ \emph{differs from} $\TTy$.

Here, $h_c^{\pm}$ are given in \eqref{eq:vx_thres}, the kernel $\D$ and $\D_m$ are given by 
\beq\label{eq:int_holy_vx_del}
\bal
& \D(y) = K_2(y_1 + 1, y_2) - K_2(y_1 - 1, y_2), \quad \D_m(y_2) = \f{1}{2} ( g_{m+1}(y_2)
- g_{m-1}(y_2 )) , \\
 & y_m = \sqrt{m^2 - 1}, \quad  y_c = 3^{-1/2}.
 \eal
\eeq
Fix $m > 1$.  Different domains are given by 
\beq\label{eq:holy_vx_dom}
\bal
& \Om_{in} \teq \{ y: |y_1| \leq 1 \},  \quad \Om_{mid} \teq 
\{ y: |y_1| \in [1, m] \}, \quad \Om_{out} \teq \{ y: |y_1| > m \}, \\
& R_{in, \e} \teq \{  |y_1|\leq 1, |y_2| \geq y_c\} \cup \{ \TTy(0, |y_2|) \leq y_1 \leq 1 , \e \leq |y_2| \leq y_c \},  \\
& R^+_{in, \e} = R_{in, \e} \cap [0, 1] \times \R , \quad 
R^{++}_{in, \e} = R_{in, \e} \cap [0, 1] \times \R^+ , \\
& R_{mid} \teq \{ |y_1| \in [1, m], |y_2| \geq s_c(m) \} \cup \{ |y_2| < s_c(m), 1\leq |y_1| \leq \TTy(m, |y_2|) \},  \\
& R_{mid}^+ \teq R_{mid} \cap \{ y_1 \geq 0 \}, \quad R_{mid}^{++} \teq  R_{mid} \cap \R^2_{++}. \\
\eal
\eeq

To estimate these integrals, we follow the strategy in Section \ref{sec:int_hol_stg}. We need to derive the piecewise bounds for the map $\TTy$. For $y_2>0$, it is easy to see that $h_c^-(y_2)$ is decreasing and $h_c^+(y_2)$ is increasing by checking the sign of $\f{d}{d y_2} h_c^{\pm}(y_2)$. 
Thus the bounds for $h_c^{\pm}$ are trivial.

\begin{remark}
Although the bounds in \eqref{eq:holy_vx} look quite complicated, we develop several estimates for the same integral so that for the integrals in different region, especially near the singularity, we can choose the easiest one to compute. 

\end{remark}

\subsubsection{Basic properties of $\TTy, y_m, h_c^{\pm} $}

Recall that $h_c^{\pm}$ determine the sign of $\D( y)$
\beq\label{eq:map_vx_cy_sg}
 \D( y ) < 0, \quad  y_1 \in (0, h_c^-(|y_2|) ), \, \mbox{or} \, y_1 > h_c^+(|y_2|), \quad \D(y) > 0, \quad y_1 \in ( h_c^-(|y_2|), h_c^+(|y_2|) ),
\eeq
for $y_2 > 0$. See the right figure in Figure \ref{fig:vx_OT} for an illustrations of different regions in $\{ y_1 \geq 0 \}$ and the sign of the kernels. We have the following basic property.
\begin{lem}\label{lem:map_vx_cy}
Consider the equation and region $D_i$
\[
\int_{\TTy}^{y_1} \D( s, y_2) ds = 0, \quad \TTy \geq 0, \quad D_1 = \{ y_2 < 3^{-\f{1}{2}}, y_1 < h_c^-(y_2) \} \cap \R_2^{++}, \  D_2 = \{  y_1 > h_c^+(y_2)  \} \cap \R_2^{++} .
\]
For $y \in D_1$, 
there is a unique solution $\TTy(y) \in [ h_c^-(y_2), 1 )$, and $\TTy$ is decreasing in $y_1$ and $y_2$. 
For $y \in D_2$, 
there is a unique solution $\TTy(y) \in [1, h_c^+(y_2)]$, and $\TTy$ is decreasing in $y_1$ and increasing in $y_2$. In both cases, the solution is given by \eqref{eq:map_vx_cy}. 

\end{lem}

\begin{proof}
Clearly, using \eqref{eq:map_vx_cy}, we know that $\TTy$ solves the equation. Next, we consider the bound of $\TTy$. For a fixed $y_2 \leq 3^{-1/2}$. Denote 
\[
F(t) = \int_0^t \D(y_1, y_2) d y_1 = \f{1}{2} (  \f{y_1-1}{|y_1-1|^2 + y_2^2} -   \f{y_1+1}{|y_1+1|^2 + y_2^2} ) \B|_0^t.
\]
Using \eqref{eq:map_vx_cy_sg}, we get $F(t) <0, t \in [0, h_c^-]$; 
since $F^{\pr}(t) = \D(t, y_2)$, $F(t)$ is strictly increasing on $[h_c^-, 1]$, and  $F(1)  < 0$. Thus, we have a unique solution $\TTy \in [h_c^-, 1]$. 
Since $ y_1 < h_c^-(y_2) < \TTy(y) $, $\D( y )$ and $\D( \TTy(y), y_2)$ have opposite signs. Taking $y_1$ derivatives yields $\pa_{y_1} \TTy \leq 0$.  The properties of $\TTy$ with $y_1 > 1$ follow by a similar argument and studying $F(t) = \int_t^{\infty} \D(s, y_2) ds$. 

From the formula of $\TTy$ \eqref{eq:map_vx_cy}, we have $\TTy^2 = \f{P(y)}{Q(y)}$ for some polynomials $P, Q$ with $P$ increasing in $y_2$ and $Q$ decreasing in $y_2 $. Moreover, $P = Q \TTy^2$ and $Q$ have the same sign. For $y_2 > 0$,  using \eqref{eq:vx_thres}, if $y \in D_1$, we get 
 $y_1 < h_c^-(y_2)$, $y_1^2 < y_2^2  + 1$, $Q < 0$ and $P = \TTy^2 Q < 0$. Thus, $\pa_{y_2} (P/ Q) = \f{ P^{\prime} Q - P Q^{\prime}}{Q^2} <0$ and $\TTy$ is decreasing in $y_2$. As a result,  we have
 \beq\label{eq:holy_vx_Tmx1}
y_1 \leq h_c^-(y_2) \leq \TTy(y), \quad |\TTy(y) - y_1| = \TTy(y) - y_1 \leq \TTy(y_1^l, y_2^l) - y_1^l , 
 \eeq
for $y \in D_1 \cap [y_1^l, y_1^u] \times [ y_2^l, y_2^u]$. Since $y_1^l \leq y_1 < h_c^-(y_2) < h_c^-(y_2^l)$,  $(y_1^l, y_2^l)$ is still in $D_1$ and the upper bound is well-defined.

Similarly, if $y \in D_2$, we get $y_1 > 1, P > 0, Q = P /  \TTy^2 > 0$  and obtain that $\TTy$ is decreasing in $y_1$, but increasing in $y_2$. Moreover, for $y \in D_2$, we have 
 \beq\label{eq:holy_vx_Tmx2}
y_1 \geq h_c^+(y_2) \geq \TTy(y), \quad |\TTy(y) - y_1 | = y_1 - \TTy(y)
\leq y_1^u - \TTy(y_1^u, y_2^l). 
\eeq
Since $y_1^u > y_1 > h_c^+(y_2) > h_c^+(y_2^l)$, we get $(y_1^u, y_2^l) \in D_2$ and the upper bound is well-defined. We conclude the proof. 
\end{proof}

We remark that since $\int_0^1 \D(s, y_2) d s \neq 0$ and $\int_1^{\infty} \D(s, y_2) ds$, for a fixed $y_2$, the total mass of the positive part and the negative part of $\D(s, y_2) ds$ are not equal. As a result, we cannot construct the map $\TTy$ for all $y_1 > 0$.

\subsubsection{Estimate the integrals in the inner region}\label{sec:holy_vx_in}
Recall the integrals from \eqref{eq:holy_vx}. We have one parameter $m> 1$ and want to obtain a uniform estimate for all $m > 1$. We follow the strategy (d) to partition the domains of the integrals and this parameter. The singularity of the integral is at $y = (1, 0)$. 
We follow the strategy in Section \ref{sec:int_hol_stg} to modify the estimate of $S_{in, \e}$
near the singularity, which gives the bound $C_{in ,\e}$ \eqref{eq:holy_vx2}. For a fixed small $\e >0$, denote 
\beq\label{eq:int_holy_vx_dom2}
Q_1(\e) = [0, 1] \times [0, \e], \quad  Q_2(\e) = [1, m] \times [0, \e], \quad 
Q_3(\e) = [m ,\infty) \times [0, \e].
\eeq

Recall from  \eqref{eq:holy_vx} that 
\[
\td C_{in}(\e) = C_{in, \e, ns} + C_{in ,\e}.
\]

The integral in $C_{in, \e, ns}$ is restricted to $|y_2| \geq  \e$, while $C_{in, \e}$ is in $|y_2| \leq \e$. 
We estimate $C_{in, \e, ns}$ using  previous method and the strategy in Section \ref{sec:int_hol_stg}. To estimate $|\TTy(y) - y_1|$, we use the piecewise bounds \eqref{eq:holy_vx_Tmx1}. For the integral in $C_{in, \e, ns}$ in $( Q_{1}(\e) )^c \cap R_{in, 0}^{++}$, since the kernel $\D(y)$ has a fixed sign, it becomes 
\beq\label{eq:vx_cy_Cin_Qc}
\bal
C_{in, ( Q_{1}(\e) )^c } & \teq 2 \int_0^1  \B(\int_{y_c}^{\infty} \D(y ) d y_2 \B) \sqrt{2 y_1} d y_1 
+ 2 \int_{\e}^{y_c} d y_2 \int_{ \TTy(0, y_2) }^1 |\D(y) | \sqrt{2 y_1} dy_1 . \\
\eal
\eeq
For the first integral, we first apply the integral formula \eqref{feq:int_ker} for $K_2$, and then follow Section \ref{sec:int_hol_stg} by bounding the integral of $\D$ and $y_1^{1/2}$ piecewisely. See  \eqref{eq:holx_vx_mdp}, \eqref{eq:int_holx_vx_mid1} for an example. For the second integral, we estimate the piecewise integrals of $ |\D(y) | \sqrt{2 y_1}$ and then estimate the indicator function of  $y_1 \leq \TTy(0, y_2) = \max(1 - 3 y_2^2,0)^{1/2}$ (see \eqref{eq:map_vx_cy}).

Next, we estimate $S_{in, \e}$ and the bound $C_{in, \e}$ \eqref{eq:holy_vx}, \eqref{eq:holy_vx2}. 
We follow Section \ref{sec:int_hol_stg} and estimate two kernels in $\D(y) = K_2(y_1+1, y_2) - K_2(y_1-1, y_2)$ separately. Since $\D$ is odd in $y_1$, we have
\beq\label{eq:int_holy_vx_in0}
\bal
& S_{in, \e}  = \lim_{ \d \to 0} \int_{ \d \leq |y_2| \leq \e } \int_0^1  \td \eta_m(y) (K_2(y_1 + 1, y_2) - K_2(y_1-1, y_2)) dy  \\
&= \int_{  |y_2| \leq \e}  \int_{0}^1 K_2(y_1 +1, y_2) \td \eta_m(y) dy  
 - \lim_{\d \to 0} \int_{ \d  \leq |y_2|  \leq \e }  \int_{-1}^0 K_2(y_1, y_2) \td \eta_m(y_1+1, y_2) dy \\
& \teq I_1 + I_2, \qquad  \td \eta_m(y) \teq \eta_m(y) - \eta_m(-y_1, y_2),
\eal
\eeq
where we have used a change of variable $y_1 \to y_1 + 1$ for the second integral. Using 
\beq\label{eq:int_holy_vx_in1}
|\td \eta_m(y)| \leq \sqrt{2 y_1 } [\eta_m]_{C_x^{1/2}}
= \sqrt{2 y_1} [\eta]_{C_x^{1/2}}, \quad (y_1 + 1) > |y_2| ,  y \in [0, 1] \times [-\e, \e], 
\eeq
and the formula for $\int K_2 d y_2$ \eqref{feq:int_ker}, for $I_1$, we get 
\beq\label{eq:int_holy_vx_in2}
|I_1| \leq \int_0^1 ( \int_{|y_2| \leq \e}  K_2(y_1 + 1, y_2) d y_2 ) |2 y_1|^{1/2} d y_1 [\eta]_{C_x^{1/2}}
= 2 \sqrt{2} \cdot \f{1}{2} \int_0^1 \f{\e y_1^{1/2} }{ (y_1+1)^2 + \e^2}  d y_1 [\eta]_{C_x^{1/2}}. 
\eeq
We can bound the integral using the previous method, see e.g. \eqref{eq:int_holx_vx_mid1}. For $I_2$,  applying the estimate \eqref{eq:hol_strip_estx} in Section \ref{sec:hol_strip} in the domain 
$[ -1, 0] \times \pm [0 , \e]$ and 
\beq\label{eq:holy_vx_eta_sym}
|\td \eta_m(y_1+1, y_2)| \leq \sqrt{2 |y_1 + 1| } [\eta]_{C_x^{1/2}}
\leq \sqrt{2}[\eta]_{C_x^{1/2}} ,
\quad [\td \eta_m]_{C_x^{1/2}} \leq 2 [\eta_m]_{C_x^{1/2}} 
\leq 2 [\eta]_{C_x^{1/2}}, 
\eeq
for $y_1 \in [-1, 0]$ from \eqref{eq:int_holy_vx_in1}, we get 
\beq\label{eq:int_holy_vx_in3}
|I_2| \leq 2 \B( \sqrt{2} \f{1}{2} \arctan( \e )
+  2 \e^{1/2}(  f_s( \sqrt{1/\e}) - f_s(1) + C_{K_2, up}  ) \B)  [\eta]_{C_x^{1/2} },
\eeq
where $f_s$ is defined in \eqref{eq:hol_vx_fh2}. We get a factor $2$ since we have two domains $y_2 \in \pm [0, \e]$.

\subsubsection{Estimate of the integrals in the middle parts and outer parts}\label{sec:holy_vx_mo}

Other integrals in \eqref{eq:holy_vx} are in the region $\Om_{mid}(m), \Om_{out}(m)$ \eqref{eq:holy_vx_dom} and depend on the parameter $m > 1$ and we want to obtain piecewise bound for $m\in [m^l, m^u]$.
Recall the region $Q_{2, \e}= [1, m] \times [0, \e]$ \eqref{eq:int_holy_vx_dom2}. 

The integrand in $C_{mid, 1, \e, ns}(m)$ is restricted to $|y_2| \geq \e$  away from the singularity $(1, 0)$. We estimate it using the strategy in Section \ref{sec:int_hol_stg} and the above argument for $C_{in}$. 
To control $|\TTy(y) - y_1|$, since $y_1  \geq h_c^+(y_2)$ in the support of the integrand, we use the piecewise bound \eqref{eq:holy_vx_Tmx2}. Note that from \eqref{eq:vx_thres}, for $y_1 > 1,y_2>0$,  $ y_2 < s_c(y_1)$ is equivalent to $ h_c^+(y_2 ) < y_1$. See the right black curve in right Figure \ref{fig:vx_OT} for an illustration. The integral region in the first term satisfies 
\[
\{ y : h_c^+(y_2) < y_1< m, \e < y_2 < s_c(m) \}
=  \{ y: h_c^+(y_2) < y_1 < m, y_2 > \e \},  
\]
which is increasing in $m$. Thus, we  can bound the first integral in $C_{mid, 1, \e, ns}(m)$ 
in \eqref{eq:holy_vx} for $m \in [m^l, m^u]$ uniformly by the case of $m = m^u$.

Since $s_c(m)$ is increasing in $m$, the second integral of $C_{mid, 1, \e, ns}$ in $Q_{2, \e}^c \cap R_{mid}^{++}$  \eqref{eq:holy_vx_dom}
satisfies
\[
\bal
C_{mid,1, Q_{2,\e}^c}(m) & \leq
2 
 \int_{\e}^{s_c(m^u)} \int 
\one_{1 \leq y_1 \leq \TTy(m, y_2)} | \D( y ) | |2y_1|^{1/2}  dy 
+ 2 \int_1^{m^u} \int_{\max( s_c(m^l) , \e) }^{\inf} |\D(y) | d y_2  (2 y_1)^{1/2} d y_1 \\
& \teq I_1 + I_2 .
\eal
\]
For $m \in [m^l, m^u]$, we estimate the indicator function in $I_1$ using Lemma \ref{lem:map_vx_cy} for $\TTy$, and then estimate $I_1$ following Section \ref{sec:int_hol_stg} and a method similar to the above. For the second term, we denote $yc_{\al} = \max( s_c(m^\al), \e), \al = l, u$ and decompose the domain into three parts 
\[
D_1 = [1, m^l] \times [yc_l, \infty ), \quad D_2 = [m^l, m^u] \times  [yc_u, \infty ), 
\quad D_3 = [m^l, m^u] \times [ yc_l, yc_u].
\]
For  $y \in D_1, D_2$, since $s_c(y_1)$ \eqref{eq:vx_thres} is increasing in $y_1$, by definition, we get $y_2 \geq s_c( m^{\al}) > s_c(y_1), \al = l, u$. Thus, the kernel $\D$ has a fixed sign in $D_1, D_2$. We estimate the integral $I_2$ in $D_1, D_2$ using an argument similar to that in \eqref{eq:vx_cy_Cin_Qc}. The kernel $\D$ can change sign on $D_3$. We estimate $I_2$ in $D_3$ using the piecewise integral bound for $|\D|$ and $y_1^{1/2} \leq (m^u)^{1/2} $.

For $S_{mid, 1, \e}, C_{mid, 1 ,\e}$ \eqref{eq:holy_vx2}, following the estimates \eqref{eq:int_holy_vx_in1}-\eqref{eq:int_holy_vx_in3} for $S_{in, \e}$ and applying the estimate in Section \ref{sec:hol_strip} to the region $[0, m-1] \times [0, \e]$ when $ m -1 \geq \e$, we obtain 
\[
\bal
& |S_{mid, 1, \e}|  \leq \int_{|y_2| \leq \e} \int_1^{m^u} K_2(y_1 + 1, y_2) |2 y_1|^{1/2} d y_1 [\eta]_{C_x^{1/2}} \\
 & \quad +  2 \B( \sqrt{2m } \f{1}{2} \arctan( \f{\e}{m-1} ) +  2 \e^{1/2}(  f_s( \sqrt{(m-1)/\e}) - f_s(1) + C_{K_2, up}  ) \B)  [\eta]_{C_x^{1/2} } ,
\eal
\]
where $f_s$ is defined in \eqref{eq:hol_vx_fh2} and is increasing in $m$. The bound $ 2 m^{1/2}$ follows from $|\eta_m(y_1, y_2) - \eta_m(-y_1,y_2)| \leq \sqrt{2 |y_1|} [\eta]_{C_x^{1/2}} \leq \sqrt{ 2 m } [\eta]_{C_x^{1/2}} $ for $y_1 \in [1, m ] $, which is used to bound the $L^{\inf}$ norm in \eqref{eq:hol_strip_estx}. Note that $\arctan \f{\e}{m-1}$ is decreasing in $m$ and other functions in the upper bound of $II$  are increasing in $m$. Then we can obtain piecewise upper bounds for $m \in [m^l, m^u]$. We estimate the first integral following  \eqref{eq:vx_cy_Cin_Qc} using analytic formulas for piecewise integral of $K_2(y_1 +1, y_2)$ and bounding $y_1^{1/2}$ piecewisely.

\vs{0.1in}
\paragraph{\bf{Small $m-1$}}
Different from the integral $C_{in}$, for $C_{mid, 1}$, if $m$ is very close to $1$, the above method does not work since we require $m-1 \geq 2 \e$, which relates to the condition $a > b$ in the estimate of $J_{a, b}$ \eqref{eq:hol_strip_estx}. In this case, we follow the strategy in 
Section \ref{sec:int_hol_sing} to separate two kernels in $\D$ and estimate $S_{mid, int}, C_{mid, 3}$ \eqref{eq:holy_vx2}. 
Recall from \eqref{eq:holx_vx_del} that $\D$ is odd. 
To overcome the computational difficulties in this singular case, we symmetrize the integral following \eqref{eq:int_holy_vx_in0} and then decompose it as follows 
\[
\bal
S_{mid, int} &= \int_{ |y_2| \geq m- 1} \int_1^{m} \D(y) \td \eta_m(y) dy 
+ \int_{ |y_2| \leq m - 1 } \int_1^m K_2(y_1 + 1, y_2) \td \eta_m(y) dy  \\
&+  \int_{ |y_2| \leq m-1} \int_1^m K_2(y_1-1, y_2) \td \eta_m(y) dy 
= P_1 + P_2 + P_3 .
\eal
\]
In the domain of $P_1$, since $|y_2| \geq y_1 -1 \geq s_c(y_1)$ \eqref{eq:vx_thres}, \eqref{eq:vx_thres_u}, we get $\D(s) \geq 0$. Denote $\g = m-1$. Using $| \td \eta_m| \leq \sqrt{2m }[\eta]_{C_x^{1/2}}$ 
\eqref{eq:holy_vx_eta_sym} and the formula \eqref{feq:int_ker}, we get 
\[
\bal
|P_1| & \leq 2 \sqrt{2 m} [\eta]_{C_x^{1/2}} 
 \int_{  m- 1}^{\infty} \int_1^m \D(y) d y
= 2  \sqrt{2 m} [\eta]_{C_x^{1/2}}  \cdot \f{1}{2} \int_1^{1+\g} \f{\g^2}{ (y_1-1)^2 + \g^2 } - \f{\g }{ (y_1+1)^2 +\g^2} d y_1 \\
& \leq  \sqrt{2 m} [\eta]_{C_x^{1/2}} \int_0^{\g}\f{\g}{ y_1^2 + \g^2 }  dy_1
= \f{\pi}{4}\sqrt{2 m} [\eta]_{C_x^{1/2}} .
\eal
\]

For $P_2$, using $|K_2(y_1 + 1, y_2)| \leq \f{1}{2} \f{1}{(y_1+1)^2 + y_2^2} \leq \f{1}{8}$ for $y_1 \in [1, m]$ and \eqref{eq:holy_vx_eta_sym}, we get 
\[
 P_2 \leq \f{1}{8} \cdot 2 (m-1)^2 \sqrt{2 m}  [\eta]_{C_x^{1/2}}
= \f{1}{4}(m-1)^2 \sqrt{2 m}  [\eta]_{C_x^{1/2}}.
\]

For $P_3$, applying \eqref{eq:hol_strip_estx} with $a = b = m-1$ and \eqref{eq:holy_vx_eta_sym}, 
we get 
\[
\bal
P_3 & \leq 
2 \B( \sqrt{2 m } \f{1}{2} \arctan( 1 ) +  2 (m-1)^{1/2}(  f_s( 1) - f_s(1) + C_{K_2, up}  ) \B)  [\eta]_{C_x^{1/2} }  \\
&= 2  \B( \sqrt{2m} \f{\pi}{8} + 2 (m-1)^{1/2} C_{K_2, up} \B) [\eta]_{C_x^{1/2}} .
\eal
\]
Combining the above estimates, we yield the upper bounds $C_{mid, 3}$ for $S_{mid, int}$ in \eqref{eq:holy_vx2}. The above bounds are increasing in $m$ and we get $C_{mid, 3}(m) \leq C_{mid, 3}(m^u)$.

\vs{0.1in}
\paragraph{\bf{Estimate of $C_{out}$}}

Recall $C_{out}(m, \e), C_{out}(1 ,\e)$ from \eqref{eq:holy_vx}.  For a fixed $\e$ and any $m \in [m^l, m^u]$, since the  integral regions in 
$C_{mid, 2}(\e_m - 1) ,  I_{K_2,\inf}( m+1,  \e_{ m}  + 1) ,  I_{K_2, \inf}(m-1, \e_{ m}-1) $ 
(see \eqref{eq:hol_strip_spec_K2}) are 
decreasing in $m$, and since $m\geq m^l$,
we have $C_{out}(m, \e) \leq C_{out}(m^l ,\e)$. 
We fix a small $\e> 0$ and get $\max(m^l, 1 + \e)$. We have estimated the integral of $|\D(y)| |y_2 - T(y)|^{1/2}$ in $C_{mid,2}(\e_m-1)$ \eqref{eq:holy_vx} in Section \ref{sec:holx_vx}, e.g. $S_{up}$, and $I_{K_2, \inf}$ in \eqref{eq:hol_strip_spec_K2} and the paragraph therein. 
Note that the map $T$ in $C_{mid,2}, C_{out}$ \eqref{eq:holy_vx} differs from $\TTy$.

\vs{0.1in}
\paragraph{\bf{Estimate of the integrals in the far-field}}

To estimate the case of $ m \geq m_f $ with $m_f = R_1$ sufficient large, we want to reduce it to the case of $m = m_f$ with an integral in the far-field, which is small.  
Recall  $R_{mid}, R_{mid}^+$ from \eqref{eq:holy_vx_dom}
\beq\label{eq:C_mid_mf0}
R^{++}_{mid}(m) \teq \{ y_1 \in [1, m], y_2 \geq s_c(m) \} \cup \{ y_2 < s_c(m), 1\leq y_1 \leq \TTy(m, |y_2|) \}. 
\eeq

Note that from definition of $s_c, h_c^+$ \eqref{eq:vx_thres}, for $y_1 > 1, y_2>0$,  we have 
\beq\label{eq:C_mid_mf}
h_c^+(y_2) < y_1,  \   \iff  \  y_2 < s_c(y_1).
\eeq
See the right black curve in right Figure \ref{fig:vx_OT} for an illustration. For $m \geq m_f$, we decompose the domain of integrals in $C_{mid}, S_{mid, 1, \e}$ \eqref{eq:holy_vx}, \eqref{eq:holy_vx2} into $y_1 \in [1, m_f]$ and $y_1 \in [m_f, m]$ as follows 
\[
\bal
& R_{mid, low}(m) \teq \{ y: h_c^+(y_2) < y_1 < m , y_2 \geq \e  \}
 = \{ y: h_c^+(y_2) < y_1 < m_f , y_2 \geq \e \}  \\
& \qquad \qquad  \cup  \{ y: h_c^+(y_2) < y_1 , m_f < y_1 < m, y_2 \geq \e \}
\teq \Om_{1M} \cup \Om_{1R}, \\
& R_{mid}^{++}(m) = ( R_{mid}^{++}(m) \cap \{ y: y_1 \in [1, m_f] \} )
\cup ( R_{mid}^{++}(m) \cap \{ y: y_1 \in [m_f, m] \} ) 
\teq  \Om_{2 M} \cup \Om_{2 R}  , \\
&  [1, m]
\times [0, \e] 
 = [1, m_f] \times [0, \e]  \cup [m_f, m] \times [0, \e] \teq \Om_{3 M} \cup \Om_{3R}. 
\eal 
\]

In $\Om_{1M}$, due to \eqref{eq:C_mid_mf} and $s_c(y_1)$ is increasing in $y_1 > 1$ \eqref{eq:vx_thres},  we get $y_2 < s_c(y_1) < s_c(m_f)$ and $\Om_{1M} = R_{mid, low}(m_f)$.

For $\Om_{2M}$, we want to show $ \Om_{2M} \subset R_{mid}^{++}(m_f)$. We fix $y \in \Om_{2M}$. If $y_2 \geq s_c(m_f)$, since $y_1 \in [1, m_f]$ for $y \in \Om_{2M}$, from \eqref{eq:C_mid_mf0}, we get $y \in R_{mid}^{++}(m_f)$. If $y_2 < s_c(m_f)$, from \eqref{eq:C_mid_mf}, we get $m_f > h_c^+(y_2)$
and thus $(m_f, y_2)$ is in the definition of $\TTy$. See Lemma \ref{lem:map_vx_cy}. 
Moreover, since $s_c(z)= s_{c,in}(z)$ is increasing in $z \geq 1$ 
by \eqref{eq:vx_thres}, from $y_2 < s_c(m_f) < s_c(m)$, 
\eqref{eq:C_mid_mf0}, and $y\in \Om_{2M} \subset R_{mid}^{++}(m)$, 
we get $1\leq y_1 < \TTy(m ,y_2)$.  Since $\TTy(y)$ is decreasing in $y_1$ from Lemma \ref{lem:map_vx_cy}, 
and $m_f \leq m$, we get $1\leq y_1 < \TTy(m, y_2) < \TTy(m_f, y_2)$. Thus $\Om_{2M} \subset R_{mid}^{++}(m_f)$. 

For the integral $S_{mid, 1, \e}$ in $|y_1| \in [m_f, m]$, since $\D$ is odd, using $ |\eta_m(y_1, y_2) - \eta_m( -y_1, y_2)| \leq \sqrt{2y_1} [\eta]_{C_x^{1/2}} $ and following \eqref{eq:int_holy_vx_in0},\eqref{eq:int_holy_vx_in1}, we get
\[
   | \int_{ |y_1| \in [m_f, m], |y_2|\leq \e} \D(y) \eta_m(y) dy |  \leq 2 \sqrt{2} \int_{\Om_{3R}} |\D(y)| \sqrt{y_1} dy  [\eta]_{C_x^{1/2}}. 
\]

Thus, the integral in $\td C_{mid, 1}(m, \e)$ \eqref{eq:holy_vx} in $\Om_{1M}, \Om_{2M}, \Om_{3M}$ can be bounded by the case of $m = m_f$. In $\Om_{1R}$, since $ 0 <\TTy < y_1$, we get $|y_1 - \TTy(y)|^{1/2} \leq |y_1|^{1/2}$.
Since $\Om_{i R}, \Om_{i M} $ are disjoint, following \eqref{eq:hol_K2_far}, we bound the integral in $\td C_{mid, 1}$ for $|y_1| \geq m_f = R_1$ or $|y_2| \geq R_2$, which cover $\Om_{iR}$,  by 
\[
4 \int_{y_1 \geq R_1 , { \ or \ } y_2 \geq R_2} |\D(y)| y_1^{1/2} dy \leq 
4 C_{B} \pi (B-1)^{-1/2}, \quad B = \min(R_1, R_2) - 1.
\]

To estimate the integral in $C_{out}$ in the far-field, e.g. outside a domain $[0, R_1] \times [0, R_2]$ with large $R_1, R_2\geq 10^8$, we use  $|y_2 - T| \leq |y_2|$ in the support of the integral and the estimate \eqref{eq:hol_K2_far}. This map $T$ in $C_{mid, 2}, C_{out}$ differs from 
$\TTy$ in \eqref{eq:holy_vx}.

\subsubsection{Estimate in a strip}\label{sec:hol_strip}

Near the singularity, we need to obtain a sharp estimate of the integral
\beq\label{eq:hol_strip_type1}
I_{a,b} = \lim_{\e \to 0} \f{1}{2}\int_{ Q_{\e, a , b}} \f{y_1^2 - y_2^2}{ |y|^4} f(y) dy , \ 
K_2(y) = \f12 \f{y_1^2 - y_2^2}{ |y|^4}, \quad 
Q_{\e, a, b} = ( \pm [\e, a] ) \times  ( \pm [0, b] ) , \quad a < b .
\eeq
Without loss of generality, we assume $Q_{\e, a, b} = [\e, a] \times [0, b], \e < a, 0 < b$. Firstly, for a fixed $y_1$ and $y_2 \geq y_1$, we can construct a map $T_3(y) \leq y_1 $ by solving 
\[
 \int_{T_3(y)}^{y_2} K_2(y_1, s)  d s  = 0, \quad T_3(y) = \f{y_1^2}{y_2}.
\]
Since $K_2(y) > 0 $ for $0< y_2 < y_1$ and $K_2(y) < 0$ for $y_2 > y_1 > 0$, applying the transportation lemma, Lemma {\lemtrans} \cite{ChenHou2023a_supp}, to the interval $y_2 \in [y_1^2/b, b]$, we get 
\beq\label{eq:hol_strip_pf1}
I_{a, b}  \leq \lim_{\e \to 0} \int_{\e}^a \int_{y_1}^b |K_2(y)| |T_3(y) - y_2|^{1/2} d y [f]_{C_y^{1/2}}
+ \B| \int_{0}^a \int_0^{y_1^2 / b} K_2(y) f(y) dy \B| \teq I_1 [f]_{C_y^{1/2}} + I_2.
\eeq
Denote $Q = [0, a] \times [0, a]$.  Since $y_1^2 / b \leq a^2 / b \leq a$, in $I_2$, we have $|f(y)| \leq || f||_{L^{\inf}(Q)} $.

 For $I_2$, we have $K_2(y) > 0$ and  
\[
\bal
I_2 & \leq ||f||_{L^{\inf}(Q) }  \lim_{\e \to 0} \int_{\e}^a  \int_0^{y_1^2 / b } K_2(y) dy  
 = ||f||_{L^{\inf}(Q) } \int_0^a  \f{1}{2} \f{y_2}{|y|^2} \B|_0^{y_1^2 / b} d y_1 \\
 &= ||f||_{L^{\inf}(Q) } \f{1}{2 } \int_0^a  \f{ y_1^2 / b}{ y_1^2 + (y_1^2/ b)^2 } d y_1
 =||f||_{L^{\inf}(Q) } \f{1}{2} \int_0^a \f{ b}{ y_1^2 + b^2} d y_1
 = ||f||_{L^{\inf}(Q) }\f{1}{2} \arctan( \f{a}{b} ).
 \eal
\]

For $I_1$, since $ y_2 \geq y_1 \geq T_3(y)$, we have $|T_3(y)- y_2| \leq |y_2|$, and 
\beq\label{eq:hol_strip_pf2}
I_1 \leq \int_0^a \int_a^b |K_2(y)|  y_2^{1/2} d y_2 
+ \int_0^a \int_{y_1}^a K_2(y) | y_2 - \f{y_1^2}{y_2} |^{1/2} d y_2 \teq I_{11} + I_{12}.
\eeq
Using the scaling symmetry and \eqref{eq:hol_vx_unit_up}, we get 
\beq\label{eq:hol_strip_pf3}
I_{12} = a^{1/2} C_{K_2, up}.
\eeq
In $I_{11}$, $K_2(y)$ has a fixed sign in the domain of the integral. Integrating $y_1$ first and then using \eqref{eq:hol_vx_fh}, and we yield 
\beq\label{eq:hol_strip_pf4}
\bal
I_{11}  &= \f{1}{2} \int_{a}^b  \f{y_1}{y_1^2 + y_2^2} \B|_0^{a} y_2^{1/2} d y_2
=  \f{1}{2} \int_{a}^b  \f{a y_2^{1/2}  }{a^2 + y_2^2} dy_2
= \f{1}{2} \cdot 2 a^{1/2} ( f_s( \sqrt{b/a}) - f_s(1))  \\
& = a^{1/2}  ( f_s( \sqrt{b/a}) - f_s(1)) .
\eal
\eeq

In summary, we establish 
\beq\label{eq:hol_strip_est}
|I_{a, b}| \leq  |f|_{L^{\inf}(Q) }\f{1}{2} \arctan( \f{a}{b} )
+  a^{1/2}\B(  f_s( \sqrt{b/a}) - f_s(1) + C_{K_2, up}  \B) [f]_{C_y^{1/2}}, \quad
Q = [0, a]^2 .
\eeq

Using the same argument and swapping the variable $y_1, y_2$, for $0< b< a$ and
\beq\label{eq:hol_strip_type2}
J_{a,b} = \lim_{\e \to 0} \f{1}{2}\int_{ R_{\e, a , b}} \f{y_1^2 - y_2^2}{ |y|^4} f(y) dy , \quad 
R_{\e, a, b} = ( \pm [0, a] ) \times  ( \pm [\e, b] ) , \quad  b < a   , 
\eeq
we also obtain 
\beq\label{eq:hol_strip_estx}
|J_{a, b} | \leq |f|_{L^{\inf}(Q) }\f{1}{2} \arctan( \f{b}{a} )
+  b^{1/2}\B(  f_s( \sqrt{a/b}) - f_s(1) + C_{K_2, up}  \B) [f]_{C_x^{1/2}}, \ 
Q = \pm [0, b] \times \pm [0, b]. 
\eeq

\vs{0.1in}
\paragraph{\bf{A special case}}

Consider a similar integral  
\[
   M(a) \teq \int_{[0,a]^2} \f{y_1^2 - y_2^2}{2 |y|^4} \td f (y) dy , \quad \td f(y) =  f(y) - f(y_1, 0), \quad  Q_a \teq [0, a]^2.
\]
For $f \in C^{1/2}(Q_a)$, the integrand is locally integrable, and we do not need to take the principal value. Applying \eqref{eq:hol_strip_type1}, \eqref{eq:hol_strip_pf1} with $a=b$, $ |\td f(y) | \leq y_2^{1/2} [f]_{C_y^{1/2}}(Q_a), [\td f]_{C_y^{1/2}(Q_a) } \leq [ f]_{C_y^{1/2}(Q_a) }$, the scaling symmetry, and \eqref{eq:hol_vx_unit_up}, we can estimate $M(a)$ as follows 
\beq\label{eq:Ck2_unit}
|M(a)| \leq [f]_{C_y^{ \f{1}{2}}(Q_a)} ( \int_0^a \int_{y_1}^a |K_2(y)| | \f{y_1^2}{y_2} - y_2|^{ \f{1}{2}} 
+ \int_0^a \int_0^{ y_1^2 / a } |K_2(y)| y_2^{ \f{1}{2}} d y) 
\leq a^{ \f{1}{2}}  C_{K_2}  [f]_{C_y^{ \f{1}{2}}(Q_a) }. 
\eeq

\vs{0.1in}
\paragraph{\bf{The infinite length case}}
We consider estimating 
\beq\label{eq:hol_strip_spec}
I =I_{b, \infty} - I_{a, \infty} =   \int_{a }^b \int_0^{\infty} K_2(y) f(y) d y. 
\eeq
Applying the above argument and estimate \eqref{eq:hol_strip_pf1}, we get 
\beq\label{eq:hol_strip_spec_K2}
I \leq I_{K_2, \inf}(a, b) [f]_{C_y^{1/2}} , \quad  I_{K_2, \inf}(a, b) \teq \int_a^b \int_{|y_1|}^{\infty} \f{1}{2} \B|\f{y_1^2 -y_2^2}{|y|^4} \B| |\f{y_1^2}{y_2} - y_2 |^{1/2} dy.
\eeq
where $a< b$. The second term in \eqref{eq:hol_strip_pf1} vanishes since $ \f{y_1^2}{\infty}  =0$. Applying a change of variable $y_2 = s y_1$ yields 
\[
I \leq \f{1}{2} \int_a^b   y_1^{-1/2} d y_1 \int_1^{\infty} \f{s^2-1}{ (s^2 + 1)^2} \B| s - \f{1}{s} \B|^{1/2} ds [f]_{C_y^{1/2}}
=  (b^{1/2} - a^{1/2})  \int_1^{\infty} \f{s^2-1}{ (s^2 + 1)^2} \B| s - \f{1}{s} \B|^{1/2} ds [f]_{C_y^{1/2}}.
\]
The integral can be estimated using the strategy in Section \ref{sec:int_hol_stg} by partitioning $[0, \infty)$ and the integral formula for $K_2(s,1)$ \eqref{feq:int_ker}. Note that in the far-field $s \geq R > 1$, we have 
\[
\int_R^{\infty} \f{s^2 - 1}{ (s^2+1)^2} |s-\f{1}{s}|^{1/2}
\leq \int_R^{\infty} s^{-2 + 1/2} ds = 2 R^{-1/2}.
\]

\subsection{Functions and transportation maps}\label{app:OT_map}

We present the formulas of the transportation maps and the functions related to the sign of the kernels in the sharp H\"older estimate.  
Recall 
\[
K_1 = \f{y_1 y_2}{|y|^4}, \quad K_2 = \f{1}{2}\f{y_1^2 - y_2^2}{ |y|^4}.
\]

\subsubsection{Sign functions}

Solving $K_2(y_1 + 1, y_2 ) - K_2(y_1 -1, y_2) = 0$ for $y_2 \geq 0$, we yield 
\begin{subequations}\label{eq:vx_thres}
\beq
\bal
y_1 & = h_c^{\pm}(y_2) \teq \B( y_2^2 + 1 \pm  2 y_2 \sqrt{ y_2^2 + 1}  \B)^{1/2} , \\ 
  y_2 & = s_{c, in}(y_1) \teq  \B( \f{ -(y_1^2 + 1) + 2 (y_1^4 - y_1^2 + 1)^{1/2} }{3}   \B)^{1/2} .
\eal
\eeq

 Denote $g(y_1) = -(y_1^2 + 1) + 2 (y_1^4 - y_1^2 + 1)^{1/2}$ and $y_1 \geq 1$, we get 
 \[
 g^{\pr}(y_1) = - 2 y_1 + \f{4 y_1^3 - 2 y_1}{ (y_1^4 - y_1^2 + 1)^{1/2} }
 = \f{2 y_1}{ (y_1^4 - y_1^2 + 1)^{1/2} } (   2 y_1^2  -  1 -  (y_1^4 - y_1^2 + 1)^{1/2}   ).
 \]
 We have $ g^{\pr}(y_1) > 0$ for $y_1>1$, $ g^{\pr}(y_1) < 0$ for $y_1 \in (0,1)$,  and $g^{\pr}(y_1) = 0, y_1 = 0, 1$. Thus, we get 
 \beq
 s_{c,in}(y_1) \  \mbox{increases in } [1,\infty ), \quad   s_{c,in}(y_1) \ \mbox{decreases in } [0, 1 ].
 \eeq
\end{subequations}

\subsubsection{Transportation maps}

\paragraph{\bf{Map for $u_x$}}

For a fixed $s_2 \neq 0$ and $s_1 > 0$, solving 
\beq\label{eq:cubic_T0}
\int_{T(s)}^{s_1} ( K_1( s_1 +1/2, s_2) - K_1( s_1-1/2, s_2) ) d s_1 = 0,
\eeq
yields the equation of the transportation map in the $x$ direction
\beq\label{eq:map_ux}
 T^3 + T^2 s_1 + T ( s_1^2 - \f{1}{2} + 2 s_2^2 ) - \f{ (4 s_2^2 + 1)^2}{ 16 s_1} = 0.
 \eeq

\vspace{0.1in}

\paragraph{\bf{Map for $[ u_y]_{ C_x^{1/2} } $}}

For a fixed $y_1 \geq 0$, solving 
\[
\int_{T(y)}^{y_2}(  K_2(y_1 +1, y_2) - K_2(y_1-1, y_2) ) d y_2 = 0,
\]
yields the equation of the transportation map in the $y$ direction
\beq\label{eq:map_vx}
T^3 + T^2 y_2 +  T ( y_2^2 + 2 + 2 y_1^2) - \f{ (y_1^2 - 1)^2}{y_2} = 0.
\eeq
We rewrite the above equation as an equation for $W = T + \f{y_2}{3}$ 
\[
0 = W^3 + W( \f{2 y_2^2}{3} + 2 + 2 y_1^2)
 - \B( \f{ (y_1^2 - 1)^2}{ y_2} + \f{y_2^3}{27} + \f{y_2}{3} ( \f{2y_2^2}{3} +2 + 2 y_1^2 )  \B) 
 \teq W^3 + p_2(y) W + q_2(y).
\]
 Since $p_2  > 0$, using the discriminant \eqref{eq:cubic_discr}, we obtain $-\D_W(y) > 0$. Thus the cubic equation of $T$ or $W$ has a unique real root that can be obtained by the formula similar to \eqref{eq:cubic_T3}.

\vspace{0.1in}
\paragraph{\bf{Map for $[ u_y]_{ C_y^{1/2} } $}}

 Fix $y_1, y_2 \geq 0$. Solving 
\[
\int_{y_1}^{\TTy(y)} \D(s, y_2) ds = 0  ,
\]
with $\TTy(y)  \geq 0$  yields the equation of the transportation map in $x$ direction.  
In the following two formulas only, we write $T=\TTy(y)$; this $T$ is not the map in \eqref{eq:map_ux} or \eqref{eq:map_vx}. Then
\[
1 - T^2 - y_1^2 + T^2 y_1^2 - 2 y_2^2 - T^2 y_2^2 - y_1^2 y_2^2 - 3 y_2^4 = 0,
\] 
The above equation is  equivalent to
\beq\label{eq:map_vx_cy}
T^2 =  \f{ y_1^2 + 2 y_2^2 + y_1^2 y_2^2 + 3 y_2^4 - 1}{ y_1^2 - y_2^2 - 1}. 
\eeq

We apply the above map to the following two regions separately
\[
y_1 \in [0,1] , \ y_1 \leq h_{c}^-(y_2), \quad y_1 \in [1, \infty] ,\  y_1 \geq h_{c}^+(y_2).
\]

\section{Additional $L^{\inf}$ estimates of $\na \uu$ and some explicit integrals}\label{sec:linf_sing}

We have discussed the $L^{\infty}$ estimates of $\na \uu$ in Section {\secvelcomp} in  Part II \cite{ChenHou2023b_supp}. In this section, we provide a few more detailed calculations in the singular region. Recall from Section {\seclinfdecay}, {\secEE} in \cite{ChenHou2023a_supp} the energy $E_1, E_4$ 
for $W_1 = (\om_1, \eta_1, \xi_1)$, which satisfy 
\beq\label{energy1}
\bal
& \max(  ||\om_1 \vp_1||_{\inf},\sqrt{2} \tau_1^{-1} || \om_1 |x_1|^{-1/2} \psi_1||_{\inf},  \tau_1^{-1} || \om_1 \psi_1 ||_{C^{1/2}_{g_1}}
) \leq E_1(t) ,  \\
& \max( E_1(t), \mu_{g, 1} || \om_1 \vp_{g, 1}||_{\inf} ) \leq E_4(t), \quad \mu_{g, 1} = \tau_2 \mu_4 ,
\eal
\eeq
with weights and parameters given below.
In Appendix {\appparawg} in Part I \cite{ChenHou2023a_supp}, we choose the following parameters for the energy $E_1$
\beq\label{wg:EE}
\bal
&\tau_1 = 5, \ 
\quad \mu_4 = 0.065, \quad \tau_2 = 0.23 ,   \\
\eal
\eeq
the following weights in the estimate of nonlocal terms 
\beq\label{wg:hol}
\bal
& \psi_1   = |x|^{-2} + 0.5 |x|^{-1} + 0.2 |x|^{-1/6}, \quad  \psi_{du} = \psi_1, \quad \psi_{u} = |x|^{5/2} + 0.2 |x|^{-7/6}, \\
   & g_{1}(h)   = g_{10}(h) g_{10}(1, 0)^{-1}, \ 
 g_{10}(h) =  ( \sqrt{ h_1 + q_{11} h_2 } + q_{13} \sqrt{ h_2 + q_{12} h_1 }   )^{-1}, \\
 &    \vec q_{1,  }  = ( 0.12, 0.01, 0.25 ) ,
\eal
\eeq
and the following weights for $\om$ and the error 
\beq\label{wg:linf}
\bal
\vp_1 & = x^{-1/2} ( |x|^{-2.4} + 0.6 |x|^{-1/2} ) + 0.3 |x|^{-1/6}, \quad \vp_{g1} = \vp_1 +   |x|^{1/16}, \\
\vp_{elli} & = |x_1|^{-1/2} ( |x|^{-2} + 0.6 |x|^{-1/2}) + 0.3 |x|^{-1/6} .
\eal
\eeq
We do not write down the full energy since we do not use other norms in $E_i$ in this supplementary material. Below, to simplify the notation, we simplify $\om_1$ as $\om$ in the energy.

\subsection{$L^{\infty}$ estimate of $\na \uu$ in the singular region}\label{sec:singu_hol_cost}

In this section, we estimate 
\beq\label{eq:int_linf_sg}
\bal
 S & = P.V. \int_{ Q_{\d}} K(y) (W \psi )(x+ y) +  s (W \psi)(x) ,  \quad Q_{\d} =[-\d, \d]^2 , \quad    Q_{\d}^{\pm} \teq [-\d, \d] \times \pm [0, \d] ,\\
  K & = K_1 = \f{y_1 y_2}{|y|^4}, \quad K_2 = \f{1}{2} \f{y_1^2 - y_2^2}{|y|^4},   \quad s =0,  \f{\pi}{2} \mathrm{ \ or \ } -\f{\pi}{2},  \\
 \eal
\eeq
related to the piecewise $L^{\inf}$ estimate of $\psi \na \uu$ using the energy $E_1$ defined in 
\eqref{energy1} which satisfies
\beq\label{eq:w_est}
||  \om \vp||_{L^{\inf}} \leq E_1, \  [ \om \psi ]_{C_x^{1/2}} \leq \g_1 E_1,  \  [ \om \psi ]_{C_x^{1/2}} \leq \g_2 E_1 , \ \g_1 = \tau_1, \ \g_2 =  \tau_1 g_1(0, 1)^{-1} , \ \g_1 > \g_2.
\eeq
Our goal is to establish the following estimate 
\beq\label{eq:linf_cost1}
|S| \leq C_1 || \om \vp||_{L^{\inf}} + C_2 
[ \om \psi]_{C_x^{1/2}} + C_3 [ \om \psi]_{C_y^{1/2}}
\leq ( C_1  + \g_1 C_2 + C_3 \g_2 ) E_1, 
\eeq
with constant $ C_1  + \g_1 C_2 + C_3 \g_2$ as small as possible. Below, we focus on the case using the norm $|| \om \vp ||_{\inf}$ with $\vp = \vp_1$. 
The estimate using other norms $|| \om \vp||_{\inf}, \vp = \vp_{g,1}, \vp_{elli}$ is similar and is discussed in Section \ref{sec:vel_linf_otherwg}.
We adopt the notations from Section {\secsharp} in \cite{ChenHou2023a_supp}, Section {\secvelcomp} in Part II \cite{ChenHou2023b_supp}. Here, $W$ is the odd extension of $\om$ in $y$ form $\R_2^+$ to $\R_2$, $\tau_1 = 5 $ and $g_1, \psi = \psi_1 , \vp$ are the weights for $\om$ in the energy. We use 
\[
(K, s) = ( K_1 ,  0), \mathrm{\ for \ } u_x, \quad ( K_2,  -\f{\pi}{2}), \mathrm{ \ for \ } v_x, 
\quad (K_2, \f{\pi}{2} ) \mathrm{\ for \ } u_y. 
\]
Note that one needs to multiply $\f{1}{\pi}$ to get the estimate for $\psi \na \uu$. We have discussed the estimate for $u_x$ in Section {\secvellinf} in Part II \cite{ChenHou2023b_supp}.

We assume that $x_i \in [x_i^l, x_i^u] \in \R_+$, 
and derive the piecewise bound for $S(x)$.  Denote 
\beq\label{eq:int_linf_nota}
F = W \psi, \quad  x_{2, \d} = \min(x_2, \d),
\quad  \al = x_{2,\d} / \d, 
 \quad \al^l = \min( x_2^l / \d, 1), \quad \al^u = \min( x_2^u/\d, 1).
\eeq

\subsubsection{Estimate of $u_x$}
In the case of $u_x$, we have $K = K_1, s =0 $ and $K_1(s)$ is odd in $s_1, s_2$. 
 Using $| W \psi(x+y) - W \psi( x + (-y_1, y_2)) | \leq \sqrt{2 y_1} [W \psi]_{C_x^{1/2}}$, we get
\[
|S | \leq  [ W \psi]_{C_x^{1/2}} \int_{ [0, \d] \times [-\d, \d]} | K_1(s)| |2s_1|^{1/2} ds
= 2 \sqrt{2 \d }  [\om \psi]_{C_x^{1/2}} 
\int_{ [0, 1]^2 } K_1(s) |s_1|^{1/2} ds.
\]
Using \eqref{feq:int_ker} and \eqref{eq:hol_ux_fs2}, we get 
\[
\int_{ [0, 1]^2 } K_1(s) |s_1|^{1/2} ds
 = \int_0^1 - \f{1}{2} \f{s_1^{3/2}}{ |s|^2} \B|_0^1  d s_1
 =  \f{1}{2} \int_0^1 - \f{ s_1^{3/2}}{ s_1^2 + 1 } + \f{s_1^{3/2}}{  s_1^2} d s_1
 = 1-\f{1}{2} f_h(1) .
\]

The above estimate only involves $[\om]_{C_x^{1/2}}$ and is used when $x$ is close to the boundary $x_2 =0$. For $x$ away from the boundary, $W$ is also H\"older continuous in $y$ in $Q(x, r)$, since $\g_2 < \g_1$ \eqref{eq:w_est}, we can use the seminorm $[\om]_{C_y^{1/2}}$ to control $S$ and improve the estimate. We have 
\[
  S
  = \int_{-\d}^{\d} d y_1  \B( \int_{ -\d}^{-x_{2,\d}} + \int_{ -x_{2,\d}}^{x_{2, \d}}
  + \int_{x_{2,\d}}^{ \d } \B) F(x+y) K_1(y)  dy_2  \teq I_1 + I_2 + I_3.
\]
The domain in $I_1$ is below the boundary, and the domain in $I_2, I_3$ is above the boundary, where $F \in C_y^{1/2}$. For $I_1, I_3$, we use $[F]_{C_x}^{1/2}$ to control it. For $I_2$, we use $[F]_{C_y}^{1/2}$ to control it. Using the fact that $K_1(y)$ is odd in $y_1, y_2$ \eqref{eq:int_linf_nota} and \eqref{eq:w_est}, we get 
\[
\bal
S & \leq  [F]_{C_x^{1/2}} \int_0^{\d} ( \int_{-\d}^{ - x_{2,\d}} + \int_{x_{2,\d}}^{\d}) |K_1(y)| |2 y_1|^{1/2} 
+ [F]_{C_y^{1/2}} \int_{-\d}^{\d} \int_0^{ x_{2,\d}} |K_1(y)|  | 2y_2|^{1/2} d y \\
& = E_1   2 \sqrt{2 \d } \B( \g_1 \int_0^1 \int_{\al}^{1} K_1(y) y_1^{1/2} d y_1
+ \g_2 \int_0^1 \int_{0}^{\al} K_1(y) y_2^{1/2} d y_1 \B) 
\teq E_1   2 \sqrt{2 \d } \B( \g_1 I_1 
+ \g_2 I_2 \B) ,
\eal
\]
where $\al$ is defined in \eqref{eq:int_linf_nota}, and $I_1, I_2$ denote the first and the second integral and  we have rescaled the integral by changing $y \to \d y $. Using \eqref{feq:int_ker}, \eqref{eq:int_linf_nota}, and \eqref{eq:hol_ux_fs}, we yield 
\[
\bal
I_2 &= \f{1}{2} \int_0^{\al}  -\f{y_2^{3/2}}{y_1^2 +  y_2^2} \B|_{y_1=0}^1 d y_2
= \f{1}{2} \int_0^{\al} ( \f{ 1}{ y_2^{1/2}}  -  \f{y_2^{3/2}}{1 + y_2^2}   ) d y_2
\leq \f{1}{2} \int_0^{\al^u} ( \f{ 1}{ y_2^{1/2}}  -  \f{y_2^{3/2}}{1 + y_2^2}  ) d y_2 \\
&= \f{1}{2}( 2 \sqrt{\al^u} - f_h( \al^u)),
\eal
\]
where the inequality follows from $\f{ 1}{ y_2^{1/2}}  -  \f{y_2^{3/2}}{1 + y_2^2} \geq 0$. For $I_1$, using \eqref{feq:int_ker} and \eqref{eq:hol_ux_fs2}, we get 
\[
\bal
I_1 & = \f{1}{2} \int_0^1 - \f{y_1^{3/2}}{|y|^2} \B|_{y_2=\al}^1  d y_1 
= \f{1}{2} \int_0^1 \f{y_1^{3/2}}{\al^2 + y_1^2} - \f{y_1^{3/2}}{ 1 + y_1^2} d y_1
\leq \f{1}{2} \int_0^1 \f{y_1^{3/2}}{ (\al^l)^2 + y_1^2} - \f{y_1^{3/2}}{ 1 + y_1^2} d y_1 \\
 &= \f{1}{2} ( \sqrt{\al^l} f_h( \f{1}{ \al^l} ) - f_h(1)  ).
 \eal
\]
If $x_2 \geq \d$,  we get $\al^l = 1$ and the first part vanishes $I_1 = 0$.

\vs{0.1in}
\paragraph{\bf{Improvement for $x_2 \geq \d$ }}
When $x_2 \geq \d$, we have $x+ Q_{\d} \subset \R_2^+$ and $F \in C^{1/2}(x+ Q_{\d})$ . Denote 
\[
M = \g_2^2 / \g_1^2 < 1.
\]
Symmetrizing the integrals, using \eqref{eq:w_est},
\[
\bal
& |F(x + y) + F(x - y) - F(x_1 - y_1, x_1 + y_2) - F(x_1 + y_1, x_2 - y_2)|  \\
\leq & 2 \min( [F]_{C_x^{1/2}} |2 y_1|^{1/2}, [F]_{C_y^{1/2}} |2 y_2|^{1/2}  ) 
\leq 2 \sqrt{2 } \min( \g_1 |y_1|^{1/2}, \g_2 |y_2|^{1/2} )  E_1,
\eal
\]
$\g_1 > \g_2$, 
and an argument similar to the above, we derive 
\[
\bal
|S| \leq 2\sqrt{2 \d } E_1 \int_{ [0, 1]^2}  K_1(y) \min( \g_1 |y_1|^{1/2}, \g_2 |y_2|^{1/2} ).
\teq 2 \sqrt{2 \d} E_1  I .
\eal
\]
The threshold between two estimates is $y_2 = y_1 / M$ or $ y_1 = M y_2$. We have 
\[
\bal
I & = \int_{ [0, 1]^2}  K_1(y) ( \one_{ \f{y_1}{M} \leq   y_2} |y_1|^{1/2} \g_1 
+ \one_{ \f{y_1}{M} \geq  y_2} |y_2|^{1/2} \g_2 ) d y   \\
& = \g_1  \int_0^M d y_1  \int_{y_1 / M}^1 \f{y_1 y_2}{|y|^4} y_1^{1/2} d y_2 
+ \g_2 \int_0^1 d y_2  \int_{ M y_2}^1 \f{y_1 y_2}{|y|^4} y_2^{1/2} d y_1 
\teq I_1 + I_2.
\eal
\]

For $I_1$, using \eqref{feq:int_ker} and \eqref{eq:hol_ux_fs}, we get 
\[
\bal
I_1 & = \f{ \g_1}{2} \int_0^M  - \f{ y_1^{3/2} }{ y_1^2 + y_2^2} \B|_{y_1/ M}^1 d y_1
= \f{\g_1}{2} \int_0^M  - \f{ y_1^{3/2}}{ 1 + y_1^2}   + y_1^{-1/2} \f{ 1 }{ (1 / M)^2 + 1 }  dy_1  = \g_1 \f{M^{5/2}}{M^2 + 1} - \f{\g_1}{2} f_h( M).
 \eal
\]

For $I_2$, using \eqref{feq:int_ker} and \eqref{eq:hol_ux_fs}, we yield 
\[
I_2 =  \f{ \g_2}{2} \int_0^1  - \f{y_2^{3/2}}{ y_1^2 + y_2^2} \B|_{M y_2}^1 d y_2 
= \f{\g_2}{2} \int_0^1 y_2^{3/2} ( \f{1}{ (M^2 + 1 ) y_2^2 }   - \f{1}{ y_2^2 + 1}  ) d y_2
= \f{\g_2}{2} ( \f{2 }{M^2 + 1} - f_h(1) ) .
\]
Since $M = \g_2^2 / \g_1^2, \g_1 M^{1/2} = \g_2$, we establish 
\[
|S| \leq  \sqrt{2 \d} E_1 ( 2 \g_1 \f{M^{5/2}}{ 1 + M^2} - \g_1 f_h(M) 
+ 2 \g_2 \f{1}{M^2 + 1} - \g_2 f_h(1) ) 
= \sqrt{2 \d} E_1 ( 2 \g_2 - \g_1 f_h( \g_2^2 / \g_1^2) - \g_2 f_h(1) ).
\]

\subsubsection{Estimate of $v_x$}
In the case of $v_x, u_y$ \eqref{eq:int_linf_sg}, we have $K(y) = K_2(y)$ which is even in $y_1, y_2$, and $s = -\f{\pi}{2}$ in the case of $v_x$, and $s = \f{\pi}{2}$ for $u_y$. Firstly, if $x_2  \geq \d $, we get $x + Q_{\d} \subset \R_2^+$ and $\om \psi \in C^{1/2}(x + Q_{\d})$. Using Lemma \ref{lem:pv}, we rewrite $S$ in the case of $v_x$ as follows 
\[
\bal
S & = P.V. \int_{Q_{\d}} K_2(y) (F(x+y) - F(x)) dx - \f{\pi}{2} F(x)
= \lim_{\e \to 0} \int_{ Q_{\d}, |y_1| \geq \e} K_2(y) ( F(x+ y) - F(x) ) dy
-  \f{\pi}{2} F(x)  \\
& \teq S_1 - \f{\pi}{2} F(x).
\eal
\]
The first term $S_1$ has the form \eqref{eq:hol_strip_type1}. Applying \eqref{eq:hol_strip_pf1} to four regions $\pm [0, \d] \times [0, \d]$ and $[F(x+\cdot) - F(x)]_{C_y^{1/2}} = [F]_{C_y^{1/2}} $, we yield 
\beq\label{eq:hol_vx_sg1}
\bal
|S_1 | & \leq  [F]_{C_y^{1/2}} \int_{-\d}^{\d}  d y_1 \int_{ |y_1| \leq |y_2| \leq \d } | \f{y_1^2}{y_2} - y_2|^{ \f{1}{2} }  |K_2(y)| d y_2  \\ 
& \quad + \B| \int_{-\d}^{\d} d y_1 \int_{|y_2| \leq  \f{y_1^2}{\d} } K_2(y) ( F(x+y) - F(x)) dy_2 \B|  \teq S_{11} + S_{12}. 
\eal
\eeq

For $S_{11}$, using scaling symmetries, \eqref{eq:hol_vx_unit_up}, and \eqref{eq:w_est}, we get 
\[
S_{11} = 4 [F]_{C_y^{1/2}} \d^{1/2} C_{K_2,  up} 
\leq 4 \g_2 \d^{1/2} C_{K_2,  up}  E_1.
\]

For $S_{12}$, using \eqref{eq:w_est}, 
$
|F(x+y - F(x))| \leq [F]_{C_x}^{1/2} |y_1|^{1/2} + [F]_{C_y}^{1/2} |y_2|^{1/2} 
\leq E_1 ( \g_1 |y_1|^{1/2} + \g_2 |y_2|^{1/2}), 
$
the symmetries of $K_2$ in $y$, and $|y_2| \leq |y_1|$, we yield 
\[
|S_{12}| \leq 4 E_1  \int_0^{\d} \int_0^{y_1^2 / \d} ( \g_1 |y_1|^{1/2} 
+ \g_2 |y_2|^{1/2} ) K_2(y) dy
= 4 E_1( \g_1 I_1  + \g_2 I_2),
\]
where $I_i $ denotes two integrals. Using the scaling symmetry of $K_2$, \eqref{feq:int_ker}, \eqref{eq:hol_vx_fh2} and \eqref{eq:hol_vx_unit_low}, we get 
\[
\bal
I_2 & =  \d^{1/2} \int_0^1 \int_0^{y_1^2} K_2(y) |y_2|^{1/2}  dy
= \d^{1/2} C_{K_2, low},   \\
I_1 & =\d^{1/2} \int_0^1 \int_0^{y_1^2} K_2(y) y_1^{1/2} d y
= \d^{1/2} \int_0^1 \f{y_2}{ 2 |y|^2} \B|_0^{y_1^2} dy_1
= \d^{1/2}  \int_0^1 \f{1}{2( 1 + y_1^2)} y_1^{1/2}  dy_1
= \d^{1/2} f_s(1).
\eal
\]
Combining the above estimates of $S_{11}, S_{12}$ and further bounding $F(x) \leq \psi(x) / \vp(x) E_1 $ \eqref{eq:w_est}, we establish 
\[
\bal
|S(x) | 
\leq |S_{11}| + |S_{12} | +  \f{\pi}{2} |F(x) |
 &\leq 4 E_1 \d^{1/2}( \g_1 f_s(1) + \g_2 (  C_{K_2, low} + C_{K_2, up}) ) 
 + \f{\pi}{2} \f{\psi(x)}{\vp(x)} E_1.
 \eal
\]

The above estimate does not hold for $x_2 < \d$, i.e. the singular region touches the boundary, since $W$ is discontinuous across the boundary. Before we estimate $u_y$ and the case $x_2 < \d$, we discuss another $L^{\infty}$ estimate for $S_{12}$.

\subsubsection{$L^{\inf}$ estimate in a curved region}\label{sec:linf_curve}

Since $K_2(y) F(x+y) \one_{|y_2| \leq y_1^2/\d} \one_{|y_1| \leq \d}$ is locally integrable, we can bound it using $|| \om \vp ||_{\infty}$. We focus on a specific quadrant $y_1, y_2 \geq 0$. 
Then in the integral region $y_2 \leq y_1^2 / \d \leq y_1$, we get $K_2(y) \geq 0$. We partition $[0, 1]$ using mesh $0 = z_0 < ...< z_m = 1$ and use a change of variables $ y = \d s , K_2(\d s) \d^2 = K_2(s)$ to get
\[
\bal
 |S_{12}^{++}| & = |\int_{0}^{\d} \int_0^{y_2^2 / \d} K_2( y ) (W\psi) (x + y )  d y |
  = | \sum_{0\leq i \leq m-1} \int_{z_i}^{z_{i+1}} \int_0^{s_1^2} K_2(s) (W\psi)(x + \d s ) d s |  \\
  & \leq \sum_{ 0\leq i \leq m-1} || W \vp||_{L^{\inf}} 
|| \f{\psi}{\vp}||_{L^{\inf}(x+ R_i) }
   \B| \int_{z_i}^{z_{i+1} } \int_0^{ s_1^2}  K_2( s) d s \B|, \quad 
R_i = [ z_i \d, z_{i+1} \d ] \times [0, z_{i+1}^2 \d ] .
\eal
\]

The piecewise $L^{\infty}$ bound $|| \f{\psi}{\vp}||_{L^{\inf}(x+ R_i) }$ for $x \in [x_1^l, x_1^u] \times [x_2^l, x_2^u] \teq B_x$ can be obtained by covering the region $B_x + R_i$. See Section {\secintnsymone} in Part II \cite{ChenHou2023b_supp}. Since $K_2(s)$ has a fixed sign in the domain, using \eqref{feq:int_ker}, we get 
\beq\label{eq:int_K2_curve}
 \int_{z_i}^{z_{i+1}} \int_0^{s_1^2} K_2(s) d s
 = \int_{z_i}^{ z_{i+1}}  \f{1}{2} \f{s_1^2}{ s_1^2 + s_1^4 }  d s_1
 =  \f{1}{2} \arctan s_1 \B|_{z_i}^{z_{i+1}}. 
\eeq
Combining the above estimates, we obtain a sharp $L^{\infty}$ estimate for $S_{12}^{++}$. When $\d$ is small enough, we do not further partition the domain $\d \cdot [0,1] $ and estimate $S_{12}^{++}$ directly 
\[
|S_{12}^{++}| \leq || W \vp||_{\infty} || \f{\psi}{\vp} ||_{L^{\infty}( x + [0, \d]^2)} \cdot \f{\pi}{8},
\]
where $\pi/4$ is the integral of $K_2$ in the whole region, which follows from \eqref{eq:int_K2_curve} with $z_i = 0, z_{i+1} = 1$. 

Similar estimates apply to the integrals in other curved regions 
\beq\label{eq:int_linf_cur}
\bal
S_{cur, i}^{\al \b}(x, \d)  & \teq  \int_{R_i \cap \R^{\al } \times \R^{\b}}
\one_{ s \in \R_{\al} \times \R_{\b}}  K_2(s) (W \psi)(x+ s) ds , \quad \al, \b = \pm, \\
R_1  & = \{ |s_1| \leq \d, |s_2| \leq s_1^2 / \d \},  \quad R_2 = \{ |s_2| \leq \d, |s_1| \leq s_2^2 / \d \}.
\eal
\eeq

\subsubsection{Estimate of $v_x$ for $x_2 \leq \d$}\label{sec:int_linf_vx_bd}

For $x_2 \leq \d$, the singular region $x + Q_{\d}$ touches the boundary $x_2 = 0$ and $ F \notin C_y^{1/2}(x + Q_{\d})$. 
Using Lemma \ref{lem:pv}, we decompose the integral \eqref{eq:int_linf_sg} as follows 
\beq\label{eq:int_linf_vx_decom1}
S(x) = \lim_{\e \to 0} ( \int_{Q_{\d}^+, |y_1| \geq \e} 
+ \int_{Q_{\d}^-, y_2 \leq -\e } ) K_2(y) F(x+ y) dy   - \f{\pi}{2} F(x) \teq S_1 + S_2 + S_3,
\quad  
Q_{\d}^{\pm} \teq Q_{\d} \cap \R_2^{\pm}. 
\eeq

We get a factor $-\f{\pi}{4}F(x)$ in the upper part $Q^+$ and $\f{\pi}{4} F(x)$ in $Q^-$ from Lemma \ref{lem:pv}, and they are canceled. We first apply estimate \eqref{eq:hol_strip_pf1} and then the scaling symmetries of $K_2$, \eqref{eq:hol_vx_unit_up} and \eqref{eq:int_linf_cur} to get
\[
\bal
S_1 &\leq [F]_{C_y^{1/2}} \int_{-\d}^{\d} d y_1  \int_{y_1}^{\d} K_2(y) | \f{y_1^2}{y_2} - y_2|^{1/2}  dy_2  +   \int_{-\d}^{\d}  \int_{y_2 \leq y_1^2 / \d} K_2(y) F(x+ y) dy  \\
&=2  [F]_{C_y^{1/2}} \d^{1/2} C_{K_2, up} + S_{cur, 1}^{++} + S_{cur,1}^{-+},
\eal
\]
where we have a factor $2$ since the domain contains $2$ quardrants.

For $S_2$, $F \notin C_y^{1/2}(x + Q_{\d}^-)$. Thus, we apply the estimate for \eqref{eq:hol_strip_type2} using $[F]_{C_x}^{1/2}$ instead. Using an estimate similar to  \eqref{eq:hol_strip_pf1}, \eqref{eq:hol_vx_unit_up}, and \eqref{eq:int_linf_cur}, we get
\[
\bal
S_2 & = [F]_{C_x}^{1/2}  \int_{-\d}^0 d y_2 \int_{ |y_2| \leq |y_1| \leq \d } |\f{ y_2^2}{y_1} - y_1|^{1/2} | K_2(y)| dy  + \int_{-\d}^0  \int_{ |y_1| \leq y_2^2 / \d } K_2(y) F(x+ y) d y  \\
& =2  [F]_{C_x}^{1/2} \d^{1/2} C_{K_2, up} + S_{cur, 2}^{+-} + S_{cur, 2}^{--}.
\eal
\]

The remaining terms 
$S_{cur, i}^{\al, \b}$ in the above are estimated using the method in Section \ref{sec:linf_curve}. The last term $S_3 = \f{\pi}{2} F(x)$ is estimated directly using $ \f{\pi}{2 } \f{\psi(x)}{\vp(x)} E_1 $\eqref{eq:w_est}, \eqref{eq:int_linf_nota}. See the left figure in Figure \ref{fig:curve} for various regions. We estimate $S_{cur, i}^{\al, \b}$ in the shaded region.

\begin{figure}[t]
\centering
  \includegraphics[width=0.25\linewidth]{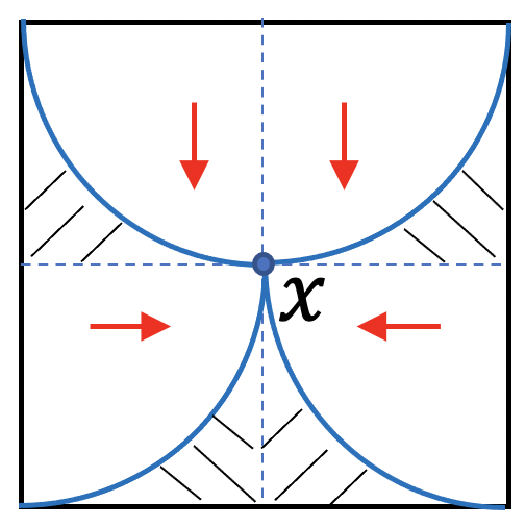}
  \includegraphics[width=0.25\linewidth]{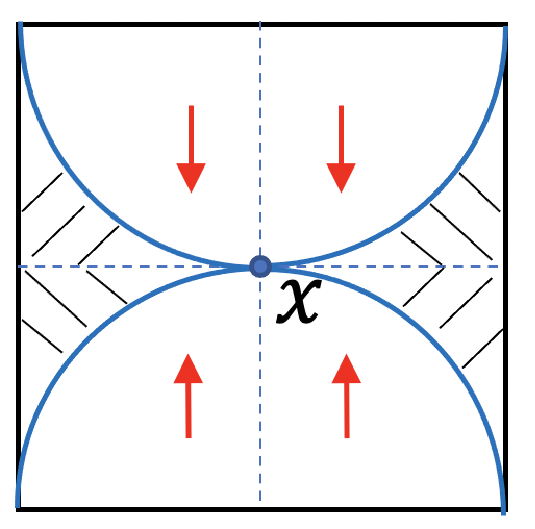}
\caption{Illustration of the estimates. The shaded region represents the curved region, where we estimate the integrals $S_{cur, i}^{\al, \b}$ using the estimate in Section \ref{sec:linf_curve} and $|| \om \vp||_{\inf}$. For other regions, we estimate the integrals using $C_{x_i}^{1/2}$ seminorm. The red arrow represents the direction of the $C_{x_i}^{1/2}$ seminorm. Left: estimate of $u_y, v_x$ with $x_2 < \d$. Right: estimate of $u_y$ with $x_2 > \d$. 
 }
\label{fig:curve}
\end{figure}

\subsubsection{Estimate for $u_y$}

The estimate of $u_y$ is completely similar. In this case, \eqref{eq:int_linf_sg} becomes 
\[
S(x) = P.V. \int_{Q_{\d}} K_2(y) F(x+y ) d y + \f{\pi}{2} F(x). 
\]
If $x_2 \geq \d$, we have $x+ Q_{\d}\subset \R_2^+$ and $F \in C^{1/2}(x+Q_{\d})$. 
Applying Lemma \ref{lem:pv} to 4 quadrants and \eqref{eq:hol_strip_pf1} yields 
\[
\bal
|S(x)| & = \B| \lim_{\e \to 0} \int_{Q_{\d}, |y_1| \geq \e} K_2(y) F(x+y) d y \B| \\
&\leq [F]_{C_y^{1/2}} \int_{-\d}^{\d} d y_1 \int_{ |y_1| \leq |y_2| \leq \d } | \f{y_1^2}{y_2} - y_2|^{1/2}
 |K_2(y)| d y_2 + \B| \int_{-\d}^{\d} d y_1 \int_{|y_2| \leq y_1^2 / \d} K_2(y) F(x+y) dy_2 \B| \\
 & = 4 \d^{1/2} [F]_{C_y^{1/2}} C_{K_2, up} 
 + S_{cur, 1}^{++} + S_{cur, 1}^{+-} + S_{cur, 1}^{-+} + S_{cur, 1}^{--}.
 \eal
\]

For $ x_2 < \d$, we do not have $ F \in C_y^{1/2}(x + Q)$. Using Lemma \ref{lem:pv} yields another decomposition 
\[
S(x) = \lim_{\e \to 0} ( \int_{Q_{\d}^+, |y_1| \geq \e} 
+ \int_{Q_{\d}^-, y_2 \leq -\e } ) K_2(y) F(x+ y) dy  dy + \f{\pi}{2} F(x) \teq S_1 + S_2 + S_3.
\]
The estimates for $S_1, S_2, S_3$ are completely the same as those in Section \ref{sec:int_linf_vx_bd}. See the left figure in Figure \ref{fig:curve} for an illustration of the estimate in the case of $x_2 \geq \d$, and right figure for $x_2 > \d$.

\subsection{Estimate of $u(x) / x_1^{1/2} $}

In the weighted $L^{\inf}$ energy estimate in \cite{ChenHou2023a_supp}, we need to estimate $u(x)/ |x_1|^{1/2}$.  We have discussed how to estimate $u(x) / |x_1|^{1/2}$ in Appendix {\secudxone} in Part II \cite{ChenHou2023b_supp}. In this section, we derive the piecewise bounds for $J_i(B)$ 
\[
J_1(B) =  \int_{ [-1, 0] \times [0, 1/B]} K_s( s) ds,
\quad 
J_2(B) = \int_{ [0,1 /B]^2} K_s(s) ds, \quad K_s(s) = \f{ 2 (s_1 + 1) s_2}{ |s|^2 ( (s_1 + 2)^2 + s_2^2) } ,
\]
which captures the singular part in the estimate of $u(x) / x_1$. Clearly, $J_i(B)$ is decreasing in $B$, for $B \in [B^l, B^u]$, we get 
$J_i(B) \leq J_i(B^l)$.
Since 
\[
K_s(s) = \f{1}{2} ( \f{ s_2}{|s|^2} - \f{s_2}{ (s_1+2)^2 + s_2^2} ),  
\]
we can derive its analytic integral formula using \eqref{feq:int_ker_u}. In particular, we have 
\[
\bal
J_1(B)  &=  \frac{B \left(\log \left(\frac{1}{B^2}+1\right)-\log \left(\frac{1}{B^2}+4\right)+\log (4)\right)+2 \arctan(B)-\arctan(2 B)}{2 B} , \\
J_2(B) & = J_{21}( B) - \log B,
\eal
\]
where the formula for  $J_2(B)$ is too lengthy and we refer it to the Mathematica code in \cite{ChenHou2023code_supp}. $J_{21}(B)$ is the regular part as $B \to 0$ and we define it by $J_{2}(B) - \log B$. Since $J_i(B)$ is decreasing in $B$, using the above formulas and $J_i(B) \leq J_i(B^l)$, we get the piecewise bounds for $J_i(B)$ in $[B^l, B^u]$. 
For $B \in [B^l, B^u] = [0, B^u]$, we derive the asymptotic behavior of the formula as $B \to 0$. For $J_1(B)$, since $2 \arctan B - \atan(2B) = O(B^3) $, we get 
\[
\lim_{B\to 0} J_1(B) = \lim_{B\to 0} \f{1}{2} \log ( \f{1 / B^2 + 1}{1/ B^2 + 4}  \cdot 4 )
 = \lim_{B \to 0} \f{1}{4}  \log( \f{ 4 + 4 B^2  }{1 +4 B^2} ) 
 = \f{1}{2} \log 4 = \log 2.
\]

Next, we show that $J_{21}(B)$ is increasing for $B$ close to $0$, which allows us to estimate $J_2(B)$ near $ B = 0$. 
Using symbolic computation, we get 
\[
\pa_B J_{21}(B) =  -\frac{-2 \log \left(\frac{1}{B^2}+\frac{2}{B}+2\right)+2 \log \left(\frac{(2 B+1)^2}{B^2}\right)-8 B+4 \tan ^{-1}(2 B)-4 \tan ^{-1}(2 B+1)+\pi }{8 B^2}
\teq - \f{S}{8 B^2}. 
\]

We can rewrite the numerator $S$ as follows 
\[
S = 2 \log( (1+2B)^2) - 2 \log( 2 B^2+ 2 B+1) -8 B+4 \tan ^{-1}(2 B)-4 \tan ^{-1}(2 B+1) + \pi.
\]
Clearly, we have $S(0) = 0$. Next, we show that $S^{\prime}(B) \leq 0$:
\[
\bal
S^{\pr}(B)  &= -8 + \f{8}{1 + (2B)^2 } - \f{8}{1 + (1+2B)^2} 
+ \f{8}{1 + 2B} - \f{2 (4B+2)}{1 + 2B + 2B^2} \\
& \leq - \f{8}{2 + 4 B + 4B^2} 
+ \f{8}{1 + 2B} - \f{2 (4B+2)}{1 + 2B + 2B^2}  
 = \f{8}{1 + 2B} - \f{8B + 8 }{1 + 2B + 2B^2} \leq 0,
\eal
\]
where we have used $1 + 2 B + 2 B^2 - (1 + B)(1 + 2B) = - B \leq 0$. Thus $S(B) \leq S(0) = 0$ and $J_{21}(B)$ is increasing for $B>0$. Recall from Appendix {\secudxone} in Part II \cite{ChenHou2023b_supp} that $B_2 = \f{\hat x_1}{h}$, where $h$ is the mesh size and $\hat x_1$ is the rescaled $x$ in the computation domain. Using the above monotonicity property, for $\hat x_1 \in [z^l, z^u],  z^u \leq \f{h}{2}$ and $\al \in (0, 1)$, we have $B_2 \leq \f{z^u}{  h} \leq \f{1}{2} $ and 
\[
J_2(B_2) \hat x_1^{1-\al} 
= J_{21}( \f{\hat x_1}{h} ) \hat x_1^{1-\al} - \log(\f{\hat x_1}{h})  \hat x_1^{1-\al}
\leq J_{21}( \f{z^u}{ h} ) (z^u)^{1-\al} +  I(\hat x_1) , 
\quad I(t) \teq \log(\f{h}{t}) t^{1-\al} .
\]

To obtain the piecewise bound for $I$, taking derivative, we yields 
\[
\pa_t I = (1-\al) t^{-\al}  \log \f{ h}{t} - t^{-\al}
= t^{-\al} ( (1-\al) \log \f{ h}{t} - 1 ),
\quad \pa_t I  >0, \mathrm{ \ for \ } t <  h e^{-1 / (1-\al)}.
\]
Therefore, for $z^u \leq \min(  h e^{-1/ (1-\al)} ,  h / 2 )$, $I(\hat x_1)$ is increasing and thus 
\[
I(\hat x_1) \leq I( z^u ), \quad  J_2(B_2) \hat x_1^{1-\al} \leq J_2( \f{z^u}{h} ) ( z^u)^{1-\al}.
\]
The quantities $ J_2(B) \hat x_1^{1-\al}$ appear in our bound for $ \f{u(x)}{ |x_1|^{\al}}$. In our application, we choose $\al = 1/2$. The above estimate allows us to control $\f{u(x)}{ |x_1|^{\al}}$ for small $x_1$.

\vs{0.1in}
\paragraph{\bf{Integral close to the singularity}}

In addition to $J_i(B)$, we need to estimate the integral
\[
\int_a^b \int_c^d  K_{du}(x, s)  ds , \quad K_{du}(x, s)= \f{ 2 (s_1 + x_1) s_2 }{ |s|^2 ( (s_1 + 2 x_1)^2 + s_2^2) } \one_{x_1 + s_1 \geq 0} , \quad Q = [a, b] \times [c, d],
\]
for $x_1 \in [x_1^l, x_1^u] \subset \R^+$, in the region close to the singularity, e.g. $ s \in [ - kh_0, kh_0]^2 \bsh [-h_0, h_0]^2$. See $II_1, II_2$ in the Appendix {\secudxone} in Part II \cite{ChenHou2023b_supp}. Denote $h = x_1^u - x_1^l$. Without loss of generality, we consider $c, d   \geq 0, s_1 \in [a, b], s_2 \in [c, d]$. For $ s_1 + x_1 \geq 0$, 
\[
K_{du} \geq 0, \quad  |s_1 + 2 x_1 |^2 \geq s_1^2 , \quad 
|s_1 + 2 x_1 |^2 - |s_1 + 2 x_1^l|^2 = (x_1 - x_1^l)( 2 s_1 + 2 x_1 + 2 x_1^l ) \geq 0, \  x_1 \leq x_1^l + h.
\]

\paragraph{\bf{Estimate I}: $a+ x_1^l \geq 0$}
In this cas, we have $s_1 + x_1 \geq a + x_1^l \geq 0$ uniformly for $s_1 \in [a, b]$ and $x \in [x_1^l, x_1^u]$. Thus, the indicator function is $1$ in $Q$, and we yield 
\[
0 \leq K_{du} \one_{x_1 + s_1 \geq 0} \leq  \f{2 (s_1 + x_1^l + h ) s_2}{ |s|^2 ( (s_1 + 2 x_1^l)^2 + s_2^2)} 
 = \f{1}{2 x_1^l} ( \f{s_2}{|s|^2} - \f{ s_2}{ (s_1 + 2 x_1^l)^2 + s_2^2 }  )
+ \f{2 h s_2}{ |s|^2 ( (s_1 + 2 x_1^l)^2 + s_2^2) } \teq I_1 + I_2.
\]

For $I_1$, if $x_1^l > 0$, we use the analytic formula \eqref{feq:int_ker_u} to evaluate the integral. For $I_2$, it is much smaller than the main term. Since  $(s_1 + 2 x_1^l)^2 + s_2^2$ is regular and $a + x_1^l \geq 0$, we bound it as follows 
\beq\label{eq:K_du_est1_I2}
\bal
& \int_Q I_2 ds \leq \max_Q  \f{2h }{ (s_1 + 2 x_1^l)^2 + s_2^2}   \int_Q \f{s_2}{|s|^2} ds , \\
&  s_1 + 2 x_1^l \geq \max( a + 2 x_1^l, x_1^l, 0) \teq dis^l, \quad s_2^2 \geq \min(c^2, d^2), 
\ c d \geq 0,
\eal
\eeq
and evaluate the integral using \eqref{feq:int_ker_u}. If $x_1^l = 0$, $I_1, I_2$ reduce to
$I_1 = \f{2 s_1 s_2}{|s|^4} = 2 K_1(s),  I_2 = \f{ 2 h s_2}{|s|^4}$. We evaluate their integrals using the analytic integral formula \eqref{feq:int_ker} and 
\[
\int \f{s_2}{|s|^4} ds = \f{1}{2} s_2^{-1} \atan \f{s_2}{s_1}  + C .
\]

Note that the integrand is singular near $0$ when $x_1^l = 0$. We only apply it to the region away from the origin $s =0$.

If $ a + x_1^l \leq 0$, since in the support of the integrand, we have
\[
0 \leq s_1 + x_1 \leq s_1 + x_1 - a - x_1^l \leq b - a + h, \quad s_1 + 2 x_1 
\geq \max( x_1^l, a + 2 x_1^l, 0) = dis^l .
\]
We bound the integrand as follows 
\[
0  \leq K_{du} \one_{x_1 + s_1 \geq 0} 
\leq \f{ 2(b-a +h) s_2}{ |s|^2 ( (s_1 + 2 x_1)^2) + s_2^2}  \one_{x_1 + s_1 \geq 0} 
\leq \f{ 2(b-a +h) s_2}{ |s|^2 ( (dis^l)^2 + s_2^2} ,
\]
and then estimate the integral following \eqref{eq:K_du_est1_I2}. 

\paragraph{\bf{Estimate II}}

If $a + x_1^u \geq 0 $, using $x_1 + s_1 \leq s_1 + x_1^u$, we have another estimate for $K_{du}$ 
\[
\bal
0 \leq x_1 K_{du} \one_{x_1 + s_1 \geq 0 } 
= \f{1}{2 } ( \f{s_2}{|s|^2} - \f{ s_2 }{ ( s_1 + 2 x_1)^2 + s_2^2 } )  \one_{x_1 + s_1 \geq 0 } 
\leq \f{1}{ 2} ( \f{s_2}{|s|^2} - \f{s_2}{ (s_1 + 2 x_1^u)^2 + s_2^2 }   ) . \\
\eal
\]
Since $a + x_1^u \geq 0 $, the right hand side is nonnegative for $s_1 \in [a, b]$.
We further use \eqref{feq:int_ker_u} to evaluate the integral of the upper bound.

If $a + x_1^u \leq 0$, in the support of the integrand, we yield 
\[
0 \leq s_1 + x_1 \leq s_1 - a + (x_1 - x_1^u) \leq b-a, \quad 
0 \leq x_1 K_{du} \one_{x_1 + s_1 \geq 0 } 
\leq \f{2 (b-a) x_1^u s_2}{ |s|^2( (s_1 + 2 x_1)^2 + s_2^2)} \one_{x_1 + s_1 \geq 0 } .
\]
Then we use $s_1 + 2 x_1 \geq \max( a + 2 x_1^l, x_1^l, 0)$ and follow \eqref{eq:K_du_est1_I2} to estimate the integral. We also have a simple bound using $  x_1 K_{du} \one_{x_1 + s_1} \leq \f{s_2}{2 |s|^2}$ and then apply the integral formula \eqref{feq:int_ker_u}.

We apply the above estimate to obtain sharp estimates of the integral of $u$ (not divided by $\f{1}{x_1}$).
Dividing both sides by $\f{1}{x_1}$ and using $ \f{1}{x_1} \leq \f{1}{x_1^l}$, we yield another estimate for the integral of $K_{du}$. This estimate is better if $x_1^u / x_1^l$ is close to $1$.

We remark that if $b + x_1^u \leq 0$, since $s_1 + x_1 \leq b+x_1^u \leq 0$, the integral is $0$.

\subsection{Estimate using other norms}\label{sec:vel_linf_otherwg}

The estimate using the weights $|| \om \vp_{g, 1}||_{\inf}$ is similar. From \eqref{energy1}, we have $|| \om \vp_{g, 1}||_{\inf} \leq  \mu_{g, 1}^{-1} E_4, E_1 \leq E_4 $. Using the norm $|| \om \vp_{g, 1}||_{\inf}, [ \om \psi_1]_{C_{x_i}^{1/2}}$, we develop another estimate for $\uu_A, (\na \uu)_A$. For the singular part $S$ \eqref{eq:int_linf_sg}, using $\mu_{g, 1} \leq 0.02$, we estimate 
\beq\label{eq:linf_cost_grow}
\bal
|S| & \leq C_1 || \om \vp_{g,1}||_{L^{\inf}} + C_2 
[ \om \psi]_{C_x^{1/2}} + C_3 [ \om \psi]_{C_y^{1/2}}
\leq ( C_1 \mu_{g, 1}^{-1}  + \g_1 C_2 + C_3 \g_2 ) E_4 \\
& \leq \mu_{g, 1}^{-1} (C_1 + \mu_{g, 1} \g_1 C_2 + \mu_{g, 1} \g_2 C_3 ) E_4
\leq \mu_{g, 1}^{-1} (C_1 + 0.02 \g_1 C_2 + 0.02 \g_2 C_3 ) E_4 ,
\eal
\eeq
similar to \eqref{eq:linf_cost1}, with constant $ C_1  + (0.02  \g_1) C_2 +  (0.02 \g_2) C_3  \g_2$ as small as possible. For this purpose, we can apply the estimates in the previous sections with $(\g_1, \g_2)$ replaced by $002 ( \g_1, \g_2 )$. 
Note that this additional estimate for $\uu_A, (\na \uu)_A $ is used to close the nonlinear estimates. The overestimate $\mu_{g,1} \leq 0.02$ only slightly increases the constant in the nonlinear estimates.

To estimate the nonlocal error $ \uu(\e), \e = \bar \om - (-\D) \bar \phi^N, \e = \hat \om - (-\D) \hat \phi^N $ for the approximate steady state or $\hat W_2$ (see Section {\seccombvelerr} in \cite{ChenHou2023a_supp}),
we develop similar estimates for $\uu_A(\e), (\na \uu)_A$ using the norm $ || \e \vp_{elli}||_{\inf}, [ \e \psi_1 ]_{C_{x_i}^{1/2}}$ \eqref{wg:hol}, \eqref{wg:linf}. For the singular part $S$, we estimate it using the norm localized to $D$ containing $x + Q_{\d}$ in \eqref{eq:int_linf_sg}
\beq\label{eq:linf_cost_err}
\bal
 |S(\e)| & \leq C_1(x) || \e \vp_{g,1}||_{L^{\inf}(D)} + C_2(x) 
[ \e \psi]_{C_x^{1/2}(D)} + C_3(x) [ \e \psi]_{C_y^{1/2}(D)} \\
& \leq  \bar B_0 ( C_1 +  \bar  B_0^{-1} \bar B_1 C_2 +  \bar B_0^{-1} \bar B_2 C_3   ),   \quad || \e \vp_{g,1}||_{L^{\inf}(D)}  \leq  \bar B_0,
[ \e \psi_1 ]_{C_{x_i}^{1/2}(D) } \leq \bar B_i, \ 
\eal
\eeq
with constant $ C_1(x)  + \bar B_0^{-1} \bar B_1 C_2 + \bar B_0^{-1} \bar B_2  C_3(x)$ as small as possible. Since $\e$ depends on the numerical solution, e.g. $\bar \om, \bar \phi^N$, locally, we can bound $\bar B_i$ directly. To get a sharp estimate, we can apply the estimate in the previous sections with $(\g_1, \g_2)$ replaced by $ ( \bar B_0^{-1} \bar B_1, \bar B_0^{-1} \bar B_2) $. See Section {\secnlocerr} in Part II \cite{ChenHou2023b_supp} for more discussions of the localized estimate. 

We can generalize the above estimates of combining different norms in the estimates of $\uu_A, (\na \uu)_A$ for $x$ very small or very large. See Section \ref{sec:1D_otherwg}.

\subsection{Estimate of some integrals}\label{sec:int_cw_K00}

We discuss the refined estimate of $u_x(0)(\e)$ and $K_{00}(\e)$ for the error $\e$ of solving the Poisson equation. The refined estimates of these terms are important for us to show that the error is small.

\subsubsection{Estimate of $K_{00}$}

In the estimate of the integrals, near the origin, we use the triangle inequality and 
we estimate the approximation term
\[
I = p_{\lam}(\hat x) \lam^{-2} C( \lam \hat x) \int_{ |\hat y|_{\inf} \leq k h} K_{00}(\hat y) W_{\lam}(\hat y) d \hat y= p(\lam x) C(x) \int_{|y| \leq \lam k h} K_{00}(y) W(y) dy, \
\]
for a fixed $k$, e.g. $k=12$, separately since the kernel given below is singular with order $|y|^{-4}$ near $y=0$
\[
 K_{00}(y) = \f{24 y_1 y_2(y_1^2 - y_2^2)}{|y|^8}  = \pa_{1}^3 \pa_2 f(y), \quad f(y) = -\log|y|.
\]
We need to estimate the above term for finitely many $\lam \leq 10$ and $\lam \leq \lam_*$ uniformly for $x$ near $0$. 
In the energy estimate for linear stability analysis, we  bound it using $|| \om \vp||_{\inf}$. In the error estimate, such an estimate is not sufficient. Since we can evaluate $\om(y)$, we exploit the cancellation in the integral. We partition the domain using fine mesh. 
Denote 
\[
0 = y_1 < y_2 <..< y_N, \quad y_{i, 1/2} =(y_i + y_{i+1}) / 2, \quad Q_{ij} = [y_i, y_{i+1}] \times [y_j, y_{j+1}] .
\]

We evaluate the integral in $Q$ using Simpson's rule 
\beq\label{eq:simpson}
\bal
& \int_{ [a_1, a_3] \times [b_1, b_3]} g(y) dy 
= \sum_{i, j \leq 3}  c_{i} c_j f(a_{i}, b_j)
+ err, 
\quad |err| \leq \f{1}{2880} ( h_1^4 ||\pa_x^4 g||_{L^{\inf}(Q)} 
+ h_2^4 || \pa_y^4 g||_{L^{\inf}(Q)} ),  \\
& h_1 = a_3 - a_1, h_2 = b_3 - b_1
\quad Q = [a_1, a_3] \times [b_1, b_3], \quad a_2 = \f{a_1 + a_3}{2}, \quad b_2 = \f{b_1 + b_3}{2}, \ c =[\f{1}{6}, \f{4}{6}, \f{1}{6}].
\eal
\eeq
We obtain the piecewise derivative bound of $K_{00}(y) W$ using the bound of $W$ established by the method in Appendix {\appsolu} in Part II \cite{ChenHou2023b_supp} and \eqref{lem:K_mix} for $K_{00}$. The above error estimate is obtained by applying the error estimate of Simpson's rule first in $x$ and then in $y$.

For the integral very close to origin , e.g. in $Q = [0, D]^2 = [0, y_m]^2$, the above method fails since the kernel is singular. 
We defer the estimate below.
With these estimates, for a fixed $\lam$, we pick $l$ such that $y_l \leq r < y_{l+1}, r = \lam k h$ and decompose the integral into three regions 
\[
 \int_{[0, r]^2} g(y) dy 
  = ( \int_{ [0, y_m]^2} 
 + \int_{ [0, y_l]^2 \bsh [0, y_m]^2} 
 + \int_{ [0, r]^2 \bsh [0, y_l]^2  } ) g(y) dy \teq S_1 + S_2 + S_3,
 |S_3| \leq  \int_{ [0, y_{l+1}]^2 \bsh [0, y_l]^2  }  |g(y)| dy.
\]
For $S_1$, we apply the estimate near $0$ discussed below. For $S_2$, we use the above Simpson's rule. Since $[0, r]^2 \bsh [0, y_l]^2$ is small, we treat $S_3$ as an error  and use the piecewise bounds for $g(y) = \om(y) K_{00}(y)$. 

The above method also provide the piecewise bound of the integral for $ \lam k h \in [y_l, y_{l+1}]$. This allows us to obtain the uniform bound for $ \lam \in [ y_m / (kh), \lam_*]$ by covering the interval. To obtain the uniform estimate for all small $\lam \leq \lam_*$, we further apply the method discussed below to estimate the case of $\lam \leq y_m / (kh)$ uniformly.

\vs{0.1in}
\paragraph{\bf{Estimate near $y=0$}} Consider $Q = [0, D]^2$. In our case, $\om$ is odd and satisfies $\om = O(|x|^3)$ near $x=0$ and we have piecewise $C^3$ bounds for $\om$. Using integration by parts, we get 
\[
\bal
&\int_Q \om(y) f_{xxxy}(y) dy 
= - \int_Q \om_x(y) f_{xxy} dy 
+ \int_{0}^D \om(D, y) f_{xxy}(D, y) dy \\
= &  \int_Q \om_{xx}(D, y) f_{xy}(D, y) dy
- \int_0^D \om_x(y) f_{xy} dy + \int_{0}^D \om(D, y) f_{xxy}(D, y) dy \teq I + II + III.
\eal
\]

The boundary term vanishes on $x=0$ since $\om f_{xxy}, \om_x f_{xy}$ is odd. Denote $M_{ij} = || \pa_x^i \pa_y^j \om||_{L^{\inf}(Q)}$. Since $\om$ is odd, for $y \in Q$, using $\om = O(|x|^3)$,  Taylor expansion 
\[
\bal
  \om = \int_0^{y_1} \om_x(z, y_2) d z  = \om_x(0, y_2) y_1  +\int_0^{y_1} \om_{xxx}(z, y_2) \f{(y_1-z)^2}{2} dz , \
 \om_x(0, y_2) =  \int_0^{y_2} \om_{xyy}(0, z) (y_2- z) dz  ,
\eal
\]
taking derivatives on the above expansions, and using $\int_0^y z^k d z = \f{y^{k+1}}{k+1}$, we get 
\[
  |\om(y)| \leq \f{y_1^3}{6} M_{30} + \f{y_1 y_2^2}{2} M_{1 2} , 
\ |\om_x| \leq \f{y_1^2}{2} M_{30}  + \f{y_2^2}{2} M_{12},  \
|\om_{xx}| \leq y_1 M_{30},  \
  \quad M_{ij} \teq || \pa_x^i \pa_y^ \om||_{L^{\inf}(Q)}.
\]

Since
\[
f_{xy} = \f{2 y_1 y_2}{|y|^4} , \quad  f_{xxy}  = \f{2 y_2 (-3 y_1^2 + y_2^2)}{|y|^6}, 
\]
$f_{xy}, f_{xxy}$ have fixed signs in $[0, D]^2, \{ D \} \times  [0, D]$, respectively.

For $I$, since $\om_{xx}$ is odd, using the scaling symmetry of the kernel, we get 
\[
|I| \leq M_{30} \int_Q | f_{xy} y_1 | dy 
 = M_{30}\int_Q \f{2 y_1^2 y_2}{|y|^4} dy 
 = M_{30} D \int_{[0,1]^2} \f{2 y_1^2 y_2}{|y|^4} dy 
= M_{30} D \cdot y_2 \atan \f{y_1}{y_2} \B|_{[0,1]^2}
  = M_{30} \f{D\pi}{4} ,
\]
where we use $\pa_{12} ( y_2 \atan \f{y_1}{y_2} )= \pa_2( y_2 \f{1/y_2}{ (y_1/y_2)^2 + 1}  )
 = \pa_2  \f{y_2^2}{y_1^2 +y_2^2}
  = -\pa_2 \f{y_1^2}{y_1^2 + y_2^2} = \f{2 y_1^2 y_2}{|y|^4}.  
$

For $II$ and $III$, we use the expansion at $(D, y)$. We get 
\[
|II| \leq \f{M_{30} D^2}{2} |\int_0^D f_{xy}(D, y) dy|
+ \f{ M_{12} }{2} |\int_0^D f_{xy}(D, y) y^2 dy|
\teq \f{M_{30} }{2}  II_1 
+ \f{ M_{12} }{2} II_2 .
\]
Using the scaling symmetry, we get 
\[
\bal
& II_1=  D |\int_0^1 f_{xy}( 1, y) dy | =  D \B|f_x(1, y) |_0^1 \B|
= D | f_x(1, 1) - f_x(1, 0)| , \\
& II_2  = D | \int_0^1 \f{2 y^3}{ (1+y^2)^2} d y|
 = D( \f{1}{ 1 + y^2 } + \log(1 + y^2) ) \B|_0^1
 = D ( \log 2 - 1/2).
 \eal
\]

For $III$, since $f_{xxy}(D, y)$ has a fixed sign in $[0, D]$, we get 
\[
|III| \leq \f{M_{30} D^3}{6} \B| \int_0^D f_{xxy}(D, y) dy |
+  \f{M_{12}}{2} D |\int_0^D f_{xxy}(D, y) y^2 d y|
= \f{M_{30}}{6} III_1 + \f{M_{12}}{2} III_2 .
\]
Since $f$ is harmonic, $f_{xxy} = - f_{yyy}$, using the scaling symmetry and integration by parts, we get 
\[
\bal
III_1 & = D  | \int_0^1 f_{xxy}(1, y) dy |
= D |f_{xx}(1, 1 ) - f_{xx}(1, 0)|,  \\
 III_2  & = D |\int_0^1  f_{xxy}(1, y) y^2 dy |
=  D \B|\int_0^1 f_{yyy}(1, y) y^2 dy \B|
= D \B|| ( f_{yy}(1, y) y^2 - 2 f_y(1, y) y + 2 f ) \B|_0^1 \B|.
\eal
\]

All the above terms are linear in $D$ since $ y_1^i y_2^j K_{00}(y) dy = D \hat y_1^i \hat y_2^j K_{00}(\hat y)d \hat y $ for $ y = D \hat y, i+ j = 3$.

\subsubsection{Estimate of $c_{\om}(\bar \om)$}

Recall from Appendix {\appsolu} in Part II \cite{ChenHou2023b_supp} that we represent $\bar \om$ using piecewise polynomials $\bar \om_2$ and the semi-analytic part $\bar \om_1 = \chi(r) r^{-\al} g(\b)$. We discuss the computation of 
\[
c_{\om}(f) = u_x(f)(0) = - \f{4}{\pi} \int_{\R_2^{++}} K(y) f(y) dy , 
\quad K(y) =\f{y_1 y_2}{|y|^4},
\]
for $f$ being a single-level B-spline and $f = \bar \om_1$. If $f$ is a multi-level B-spline (see Section {\secASS} in \cite{ChenHou2023a_supp}), i.e. $f = \sum_{i\leq n} f_i$ with $f_i$ being a single level B-spline, we estimate $c_{\om}(f_i)$ and then use linearity to estimate $c_{\om}(f)$.
For $f$ being B-spline with supporting points $y_0 < y_1< .., < y_n$, its support is contained in $D = [0, y_{n+7}]^2$. We partition the domain into 
\[
D_i = [0, y_{k_i}]^2 , \quad i = 1, 2,.., l,  \quad y_{k_1} < y_{k_2} < .. < y_{k_l} = y_{n+7}.
\]
For each $i$, we refine each mesh $Q = [y_{i_1}, y_{i_1+1}] \times [y_{j_1}, y_{j_1+1}]$ in $ D_{i+1} \bsh D_i$ $2 n_i$ times 
\[
Q = \cup_{  k  , l \leq n_i} Q_{ kl}, \   Q_{k l} = [y_{i_1} + (k-1) h_1, y_{i_1} + k h_1]
\times  [ y_{j_1} + (j-1) h_2, y_{j_1} + j h_2] , \ 
h_1 = \f{y_{ i_1+1} - y_{i_1} }{n_i}, \ h_2 = \f{ y_{j_1 +1} - y_{j_1} }{ n_i},
\]
and then apply Simpson's rule \eqref{eq:simpson} to estimate the integral $ \int_{Q_{kl}} f(y) K(y) dy$ by evaluating $ f K$ on $ y_{i_1} + \f{p h_1}{2}, y_{j_1} + \f{q h_2} {2}, 0 \leq p , q \leq 2 n_i$ and estimating the error. 
In $Q$, using \eqref{eq:simpson}, the error is given by 
\[
\f{1}{2880} \sum_{ k , l \leq  n_i} \int_{Q_{k l}}   h_1^4 ||\pa_x^4 ( f K) ||_{L^{\inf}}
+   h_2^4 ||\pa_y^4 ( f K) ||_{L^{\inf}}. 
\]
Using $ K(y) = -\f{1}{2} \pa_1 \pa_2 \log (|y|)$ and $|\pa_x^i \pa_y^j K(y) | \leq \f{ (i+j + 1)!}{2} |y|^{-i-j -2}$ from Lemma \ref{lem:K_mix}, we get 
\[
\sum_{ k, l} \int_{Q_{kl}} ||\pa_x^4 ( f K) ||_{L^{\inf}}
\leq  \sum_{ 0 \leq m \leq 4} \binom{ 4}{m} || \pa_x^{4-m} f||_{L^{\inf}(Q)} 
 h_1 h_2 \f{ (m+1)!}{2} \sum_{kl}  \max_{y \in Q_{kl}} |y|^{ - m - 2} .
\]
The bound for $\pa_y^4( fK)$ is similar.

\vs{0.1in}
\paragraph{\bf{Near $0$}}

Since the integrand is singular, near $0$, we use Taylor expansion. We choose $D_0 \subset [0, y_1]^2$. For $f(x_1, x_2 )$ odd in $x_1$, we get 
\[
 f(x) = \pa_1 f (0) x_1 + \pa_{12} f(0) x_1 x_2 + \f{ \pa_{111} f(0) x_1^3}{6} 
 + \f{ \pa_{122} f(0) x_1 x_2^2}{2} + \e,
 \  |\e| \leq \f{1}{24} \sum_{ 0 \leq i \leq 4} \binom{ 4}{i}  x_1^i x_2^{4-i}  ||\pa_x^i \pa_y^{4-i} f ||_{L^{\inf}(D_0)}.
\]
We integrate $ K(y) y_1^i y_2^j$ analytically using symbolic computation and then obtain the error estimate and the approximation of the integral by evaluating $\pa_1^i\pa_2^j f(0)$. 

\vs{0.1in}
\paragraph{\bf{Integral for $\bar \om_1$}}
For $f= \bar \om_1 = \chi(r) r^{-\al-1} g(\b)$,   $\chi(r)$ is supported in $[ A, \infty)$ and $\chi(r) = 1$ for $r \geq B, A = 5, B =2 \cdot 10^6$. Using $K(y) = \f{\sin(2\b)}{ 2 r^2}$ and a direct calculation yields 
\[
\bal
- \f{4}{\pi}\int_{\R^2_{++}} K(y) \bar \om_1 dy 
& = -\f{2}{\pi} \int_{A}^{\inf} \chi(r) r^{-\al - 1} dr 
\int_0^{\pi/2} g(\b) \sin(2\b) d \b  \\
& =  -\f{2}{\pi} ( \int_{A}^{B} \chi(r) r^{-\al - 1} dr 
+ \al^{-1} B^{-\al} )
\int_0^{\pi/2} g(\b) \sin(2\b) d \b .
\eal
\] 
We apply 1D Simpson's rule to estimate two 1D integrals.

 \section{Estimate of the velocity near $0$ and in the far-field}\label{sec:u_1D}

For $x$ very close to the origin or very large, we cannot rescale the integral by choosing finitely many rescaling factors $\lam_i$. See Section {\secholoneD} in Part II \cite{ChenHou2023b_supp} for more discussions. 
Instead, we choose $\lam = \f{ \max(x_1, x_2) }{x_c}$ and estimate the rescaled integral  with a $-d$-homogeneous kernel $K$
\beq\label{eq:int_1D_res}
 p(x) \int K(x - y) W(y) dy 
 = p_{\lam}(x) \int K(\hat x - \hat y) \lam^{2- d} W_{\lam}( \hat y) dy,
 \quad f_{\lam}(x) \teq f(\lam x),
\eeq
uniformly for all small $\lam \ll 1$ or large $\lam \gg 1$.
The rescaled singularity $\hat x = x / \lam$ satisfies $\max_i \hat x_i = x_c$ and is in the bulk of our computation domain. 
The method is essentially the same as those described in Sections {\secintmethod}, {\secvellinf} in Part II \cite{ChenHou2023b_supp}. We have the following scaling relation 
\beq\label{int:scal_w}
\bal
& || \om_{\lam} \vp_{\lam} ||_{\inf} = || \om \vp||_{\inf}, \  [ \om_{\lam} \psi_{\lam}]_{C_{x_i}^{1/2}} = \lam^{ \f{1}{2}} [ \om \psi]_{C_{x_i}^{1/2}} ,  \\
&  \pa_{ x_i } f(  x) =  \f{ d \hat x_i}{ d x_i} \pa_{ \hat x_i} f_{\lam}( \hat x )
 = \f{1}{\lam} \pa_{ \hat x_i} f_{\lam}( \hat x ), \quad 
\f{ |f(x) - f(z)| }{ |x-z|^{1/2}}= \lam^{-\f{1}{2}} \f{ |f_{\lam}(\hat x) - f_{\lam}( \hat z) | }{  |\hat x- \hat z|^{1/2} } . 
\eal
\eeq

Denote by $y_i$ the adaptive mesh for computing the approximate steady state designed in Appendix {\appsolurep} in Part II \cite{ChenHou2023b_supp}
\beq\label{eq:mesh_xy}
\bal
& y_1< y_2< .. < y_N, \quad Q_{ij} =  [y_i, y_{i+1}] \times [y_j, y_{j+1}], 
\eal
\eeq
We have $y_N > 10^{15}$. 
In Part II \cite{ChenHou2023b_supp}, we introduce the following notations: $h_x = h / 2$,
\beq\label{int:box_B}
|x|_{\inf} = \max(x_1, x_2),  \quad  x \in  B_{i_1, j_1}(h_x) \subset B_{ij}(h) , \quad 
B_{lm}(r) =   [l r, (l+1) r]  \times [ mr, (m+1) r] ,
\eeq
and singular regions near $x$
\beq\label{eq:rect_Rk}
\bal 
 R(x, k) & = [ (i-k) h, (i+1 + k) h]  \times [ (j- k) h,   (j+ k + 1] h ] ,  \\
   R_s( k ) &= [x_1 - kh, x_1 + kh] \times [x_2 - kh, x_2 + kh] . 
\eal
\eeq

We introduce the following domain and points 
\beq\label{eq:int_1D_dom1}
\bal
& M_{l i_1 j_1} =\lam_l B_{i_1 j_1 }(h_x), \quad \G(a) \teq \{ \hat x: \max_i \hat x_i = a \} =  \cup_{ i h_x \leq a} \G_{1, i}(a) \cup \G_{2, i}(a) \\
& \G_{1, i}(a) \teq  \{ a \} \times  [(i-1) h_x, i h_x] , \
\G_{2, i}(a) \teq  [(i-1) h_x, i h_x] \times \{ a \},  \\ 
& x_c = 128 h_x, \ x_{c2} = 256 h_x, \ x_{c3}= 3 x_{c2}, \ m_3 = 2 x_{c2} h_x^{-1} = 512.
\eal
\eeq
Domain $M_{l, i_1, j_1}$ is a dyadic mesh, and $x_c, x_{ci}$ can be viewed as the reference centers. 
We choose $a$ to be a multiple of $h_x$, and then $\G(a)$ consists of two intervals. In the weighted $L^{\inf}$ estimate, we rescale $x = \lam \hat x$ such that $\hat x \in \G(x_c)$. In the H\"older estimate, we rescale $(x, z)$ with $|x| \leq |z|$ such that $x = \lam \hat x$ with $\hat x \in \G(2 x_c)$.

We use the $L^{\infty}$ norm $|| \om \vp ||_{\infty}$ and the H\"older semi-norm $[\om \psi]_{C_{x_i}^{1/2}}$ to control the piecewise $L^{\inf}$ norm of $\uu, \na \uu$ and the H\"older semi-norm of the weighted quantities $ \psi_u \uu, \psi \na \uu$. Denote by $a_1, a_n$ the power of $\vp(x)$ near $0$ and $\infty$, $b_1, b_n$  the power for $\psi$, and $c_1, c_n$ for $\psi_u$. From the definitions of these weights \eqref{wg:hol}, \eqref{wg:linf} we have 
\beq\label{eq:wg_asym}
\bal
& 
\vp(x) \geq  q_1 r^{a_1} (\cos \b)^{- \f{1}{2}} + q_n r^{a_n},  \quad 
(a_1, a_n) = (-2.9, -1/6), \  (-2.9, 1/16 ), \  (-5/2, -1/6),  \\
 & \psi(x) \sim B_1 r^{b_1},  r \ll 1, \quad  \psi(x) \sim B_n r^{b_n}, \ r \gg 1, \quad (b_1, b_n) = (-2, -1/6),  \\
 & \psi_u \sim C_1 r^{c_1} , \ r \ll 1, \quad \  \psi_u \sim C_n r^{c_n} , \ r \gg 1, \qquad (c_1, c_n) = (-5/2, - 7/ 6) ,
\eal
\eeq
for some constants $q_i, B_j, C_j$, where $r = |x|, \b = \atan( \f{x_2}{x_1})$. Here $\psi = \psi_1$ \eqref{wg:hol} and we drop $1$ to simplify the notation. We use the norm $|| \om \vp||_{\inf}$ with three different weights $ \vp$ in \eqref{wg:linf} and the above power $(a_1, a_n)$. We have
\beq\label{eq:wg_asym2}
\vp_{\lam}(x) \geq \lam^{a_\al} \vp_{\infty, \al}(x), \ \al = 1, n.
\eeq

We choose $c_n = b_n - 1$ to capture the fact that $\uu$ decays one order slower than $\na \uu$.

\subsection{$L^{\infty}$ estimate}\label{sec:1D_linf}

Denote by $\phi$ the stream function. Recall from Section {\secapprvel} \cite{ChenHou2023a_supp} that
for $\lam$ very small, i.e. $x$ near $0$, the velocity with approximation terms is given by 
\beq\label{eq:u_asym_near}
(\pa_x^i \pa_y^j \phi)_A \teq \pa_x^i \pa_y^j  ( \phi -  \phi_{xy}(0) xy - \f{1}{6} \phi_{xxx}(0)(x^3  y  -  x y^3) ),
\eeq
for $i+j = 1, 2$. For example, we have $ u_x =  -\pa_{xy} \phi, v_x = \pa_{xx} \phi$.  See Section {\secapprvel} in \cite{ChenHou2023a_supp}. The vorticity $\om$ vanishes like $O(|x|^{2 + \al})$ for some $\al >0$ near $x = 0$, and $ \phi,  \phi_{xy}(0), \phi_{xxxy}(0)$ are given by 
\beq\label{eq:int_stream}
\phi(x) = -\f{1}{\pi} \int_{\R_2} G(x- y) W(y) dy , \ 
\phi_{xy}(0) = \f{4}{\pi} \int_{\R_2^{++}} \f{y_1 y_2}{|y|^4} , \ 
\phi_{xxxy}(0 ) = \f{2}{\pi} \int_{\R_2^{++}} \om(y)   \f{ 24 y_1 y_2 (y_1^2 - y_2^2) }{ |y|^8}  dy ,
   \eeq
  where $G$ is the Green function \eqref{eq:kernel_du}, and $W$ is given in \eqref{eq:ext_w_odd}.

For $\lam$ large, $x$ near $\infty$, we estimate 
\beq\label{eq:u_asym_far}
(\pa_x^i \pa_y^j \phi)_A = \pa_x^i \pa_y^j ( \phi + L_{12}(0)x y ), \quad L_{12} = - \f{4}{\pi} \int_{  y \in \R^2_{++} \bsh [0, R_n]^2 } \f{y_1 y_2}{|y|^4} \om(y) dy ,
\eeq
for some parameter $R_n$ (the largest threshold) given in Appendix {\appparavel} in \cite{ChenHou2023a_supp},  where $L_{12}(0)$ approximates $u_x(0) = -\phi_{xy}(0)$. Note that for $|x| \geq R_n$, $(\pa_x^i \pa_y^j \phi)_A $ satisfies the differential relation
\beq\label{eq:u_diff}
 \pa_x^i \pa_y^j (\pa_x^{i_2} \pa_y^{j_2} \phi)_A = (\pa_x^{i+ i_2} \pa_y^{j+ j_2} \phi)_A, \quad 
  1 \leq i + j + i_1 + j_1 \leq 2.
\eeq

Using the scaling symmetry of the kernels, we can rewrite the above quantities as integrals of $\om$ and in the form \eqref{eq:int_1D_res}. We estimate piecewise bounds for $\na \uu(\lam \hat x), \uu(\lam \hat x)$ for $\hat x \in  \G_{i, j}$ and $\lam \leq \lam_*$ or $\lam \geq \lam_*$ uniformly.

\subsubsection{Near field, far field, and the nonsingular part}\label{sec:int_1D_nsg}

For the integral in these regions, the integrand is not singular, and we follow the method in Sections {\secintmethod}, {\secvellinf} in Part II \cite{ChenHou2023b_supp} to estimate them. For example, for a region $Q \subset \R_2^{++}$ away from the singularity, using \eqref{eq:wg_asym2}, we estimate the integral \eqref{eq:int_1D_res} with symmetrized kernel 
\beq\label{eq:ker_sym}
  K^{sym}(x, y)  \teq K(x - y)  + K(x+y) - K( x_1 - y_1, x_2 + y_2)
- K(x_1 + y_1, x_2 - y_2), 
\eeq
and approximation terms $K = K^{sym} - K_{app}$ as follows 
\beq\label{eq:int_1D_nsg}
\bal
 &  |\int_D K(\hat x, \hat y) W_{\lam}(\hat y) d \hat y |
  \leq || W_{\lam} \vp_{\lam} ||_{\infty} 
 \int_D | K( \hat x, \hat y)| \vp_{\lam}^{-1}(\hat y) d \hat y   \leq \lam^{- a_\al} || \om \vp||_{\infty} \int_D |K( \hat x, \hat y)| \vp_{\infty, \al}^{-1} (\hat y) d \hat y ,
 \eal
\eeq
for $\al = 1, n$. Using the decay estimates of the symmetrized integrals of \eqref{eq:int_stream} 
with approximations in Appendix {\appkersym} in Part II \cite{ChenHou2023b_supp} and \eqref{eq:wg_asym}, one can show that $ |K(x, y)| \vp_{\infty, \al}^{-1} (y)$ is integrable away from the singularity. The last integral does not depend on $\lam$ and can be estimated using the method in \cite{ChenHou2023a_supp}. We track the integral and the power $\lam^{-a_{\al}}$.
Thus, the estimates of the integrals in these regions are essentially the same as those for finite rescaling factor $\lam_i$ discussed in Sections {\secintmethod}, {\secvellinf} in Part II \cite{ChenHou2023b_supp}, except that we use the asymptotic behavior of the weights $\vp_{\infty, 1}, \vp_{\infty, n}$ instead of $\vp$.

For large $\lam$, we need to estimate \eqref{eq:u_asym_far}. Since  $ [0, R_n ]^2$ does not enjoy the scaling symmetry, $y = \lam \hat y \in [0, R_n]^2$ if and only if $\hat y \in [0, R_n / \lam]^2$, we decompose \eqref{eq:u_asym_far} as follows 
\[
\bal
L_{12} &= - \f{4}{\pi} \int_{\R_2^{++} \bsh [0, R_n / \lam]^2} \f{\hat y_1 \hat y_2}{| \hat y|^4} \om_{\lam}(\hat y) d \hat y 
= -  \f{4}{\pi} \int_{y \in \R_2^{++}}  \one_{D_1^c}  + \B(\one_{ ( [ 0, R_n /\lam]^2)^c}(\hat y)  - \one_{D_1^c}(\hat y) \B) K_0(\hat y) \om_{\lam}(\hat y) d \hat y \\
& \teq \hat L_{12}(\lam) + I , \quad 
K_0(y) = \f{4 y_1 y_2}{|y|^4}, \quad D_1 = [0, k_1 h]^2. 
\eal
\]
The domain of the rescaled integral in $L_{12}(\lam)$ does not depend on $\lam$, and we estimate the integrals for $\pa_x^i \pa_y^j \phi -\pa_x^i \pa_y^j( xy)  \hat L_{12})(\lam) $ using the above method. The second part $I$ is treated as an error term. Using \eqref{eq:wg_asym}, \eqref{eq:wg_asym2} and $| \om_{\lam}(\hat y)| \leq \lam^{-a_n} q_n^{-1} |\hat y|^{-a_n} || \om \vp||_{\infty}$, we have 
\[
|I| =\f{1}{\pi} |\int (\one_{D_1} - \one_{[0,  \f{R_n}{ \lam }  ]^2} ) | K_0(\hat y) \om_{\lam}(\hat y) | d \hat y
\leq  
\f{|| \om \vp ||_{\inf}}{\pi q_n} \lam^{-a_n}  \int
 |\one_{D_1} - \one_{[0,  \f{R_n}{ \lam } ]^2} | \f{ 4 \hat y_1 \hat y_2}{ |\hat y|^4} |\hat y|^{-a_n}d \hat y 
\teq \f{|| \om \vp ||_{\inf}}{\pi q_n} \cdot II
 .  \\
\]

Next, we  estimate $|II|$ uniformly for $\lam \geq \lam_*$. For $0 < c < d$, since $\f{y_1 y_2}{|y|^{\b}}$ is symmetric in $y_1, y_2$,  we get
\[
\bal
J(c, d) &= \int_{ [0, d]^2 \bsh [0, c]^2} \f{4 y_1 y_2}{|y|^{4 + a_n}} d y 
= 8 \int_{c}^{d} \int_0^{y_1} \f{ y_1 y_2}{|y|^{4 + a_n}} dy 
= \f{8}{2 + a_n} \int_{c}^{d} \f{ - y_1}{ |y|^{ 2 + a_n}} \B|_0^{y_1} d y_1 \\
 &= \f{8(1 - 2^{- (2 + a_n)/2})}{2 + a_n} \int_c^d y_1^{-1 - a_n}  d y_1
 = C_{a_n} |d^{-a_n } - c^{-a_n}|, \quad C_{\al} = \f{8}{(2 + \al) |\al|}(1 - 2^{-(2+\al)/2}).
 \eal
\]
Applying the above estimate to $II$, we obtain 
\[
|II| \leq  C_{a_n}  | ( \f{R_n}{\lam})^{-a_n} 
- (k_1h)^{-a_n} |  \lam^{-a_n}
=  C_{a_n}  |  R_n^{-a_n} 
- (\lam k_1h)^{-a_n} | .
\]

If $R_n \leq \lam_* k_1h$, for $\lam \geq \lam_*$, we have 
\[
|II|\leq  C_{a_n}  (  R_n^{-a_n} 
- (\lam k_1h)^{-a_n} )  
 \leq C_{a_n} R_n^{-a_n} , a_n >0 ,
 \quad |II| \leq C_{a_n} ((\lam k_1h)^{-a_n} -  R_n^{-a_n}  ), a_n < 0. 
\]

If $R_n > \lam_* k_1 h$ and $a_n < 0$, for $\lam \geq \lam_*$, we have two cases, $ R_n \geq \lam k_1 h$ and $R_n < \lam k_1 h$. In the first case, since $\lam \geq \lam_*, a_n <0$, we get 
$\lam^{a_n} \leq \lam_*^{a_n}$ and
\[
|II| \leq C_{a_n} ( (\f{R}{\lam})^{-a_n}  - (k_1 h)^{-a_n} ) \lam^{-a_n}
\leq C_{a_n} ( (\f{R}{\lam_*})^{-a_n}  - (k_1 h)^{-a_n} ) \lam^{-a_n} .
\]
In the second case, we get 
\[
|II| \leq C_{a_n} ( (k_1h)^{-a_n} - ( R / \lam)^{-a_n} ) \lam^{-a_n}
\leq C_{a_n}  (k_1h)^{-a_n} \lam^{-a_n}.
\]
Combining two cases, we yield 
\[
|II| \leq  C_{a_n}\lam^{-a_n} \max( (k_1h)^{-a_n} ,   ( R/ \lam_*)^{-a_n}  - (k_1 h)^{-a_n} ).
\]

In our estimate, we do not have the case $R_n > \lam_* k_1 h$ and $a_n > 0$. We can pick $\lam_*$ large enough and the power $a_n$ in the weight to avoid such a case. In all cases, we can estimate 
\[
|II| \leq C_1 \lam^{-a_n} + C_2 
\]
for some $C_1, C_2$  (can be negatively) independent of $\lam$ uniformly for $\lam \geq \lam_*$.

\subsubsection{Singular part}

We focus on the estimate of $\na \uu$. which is much more difficult. For the singular part, we follow Section {\secvellinf} in Parr II \cite{ChenHou2023b_supp} to decompose it as follows 
\[
\bal
(\na \uu)_S  & = \int_{R(x, k)} K( \hat x - \hat y) W_{\lam}( \hat y) d \hat y = 
 \int_{ R(x,k) \bsh R_s( b) } + 
 \int_{R_s(b) \bsh R_s(a) }
+ \int_{ R_s( a)} )   K(\hat x-\hat y) W_{\lam}(\hat y) d \hat y \\
&  \teq I + II + III,
\eal
\]
for fixed $k, b \geq 1$ and $a$ to be chosen, where $R(x, k), R_s( k)$ are singular regions defined in \eqref{eq:rect_Rk}. The integrand in $I$ is nonsingular and we estimate it by $|| \om \vp||_{\infty}$ using the same method as that in Section \ref{sec:int_1D_nsg} and Sections 
{\secintmethod}, {\secvellinf}  in Part II \cite{ChenHou2023b_supp}, For $II$, using $L^{\inf}$ estimate, \eqref{eq:wg_asym2}, \eqref{feq:int_asym}, and $R_s(b) \subset R(b)$, we get 
\[
|II|
\leq || W \vp||_{L^{\inf}}
 \lam^{-a_{\al}} || \vp^{-1}_{\inf, \al}||_{L^{\inf}(R(b))}  \int_{R_s(b) \bsh R_s(a)} | K(\hat x-\hat y) | d \hat y  
 = || \om \vp||_{L^{\inf}}
 \lam^{-a_{\al}} || \vp^{-1}_{\inf, \al}||_{L^{\inf}(R(b))} 2 \log( \f{b}{a}).
\]

Following the estimate of the singular part in Section {\secvellinf} in Part II \cite{ChenHou2023b_supp}, we get 
\[
III =  \int_{R_s(a)} K(\hat x - \hat y)  (\psi_{\lam}^{-1}( \hat x) - \psi^{-1}_{\lam}(\hat  y)) \om_{\lam}  \psi_{\lam}(\hat y) dy + \psi_{\lam}^{-1}(\hat x)  \int_{R_s(a)} K(\hat x - \hat y) \om_{\lam}(\hat y) \psi_{\lam}(\hat y) d \hat y 
= III_1 + III_2.
\]

For $III_1$, we follow Section {\secvellinf} in Part II \cite{ChenHou2023b_supp} using Taylor expansion and the $L^{\inf}$ estimate
\[
\bal
& \psi_{\lam}^{-1}(\hat y) - \psi_{\lam}^{-1}(\hat x) 
= (\na \psi_{\lam}^{-1})(\hat x) \cdot (\hat y - \hat x) + P_e, \quad 
 |P_e| \leq  C |\na^2 \psi_{\lam}^{-1}| | \hat y - \hat x|^2, 
 \quad  |W_{\lam} \psi_{\lam}(x)|
 \leq  \f{\psi_{\lam}}{\vp_{\lam}}(\hat x) || \om \vp||_{L^{\inf}}, \\
&  \int_{R_s(a)} |K(\hat x - \hat y) (\hat x_1 - \hat y_1)^i (\hat x_2 - \hat y_2)^j | d \hat y 
=4 \int_{[0, a]^2} |K(s) s_1^i s_2^j |ds
= 4 a^{i+j}  \int_{[0, 1]^2} |K(s) s_1^i s_2^j| ds, \  i+ j \geq 1,
\eal
\]
where the last integral can be evaluated using the methods in Section \ref{sec:feq_dK}. In particular, we gain a small factor $|a|$ in the estimate of $III_1$. We refer detailed estimate of $III_1$ to Section {\secvellinf} in Part II \cite{ChenHou2023b_supp}. To factorize out the dependence of $\lam$ in the above estimates, we need the following estimates for the weights 
\beq\label{eq:wg_asym3}
| \f{\pa^i}{\pa \hat x_1} \f{\pa^j}{\pa \hat x_2} \psi_{\lam}^{-1}(\hat x)| \leq \lam^{ a_{\al}}  (\pa_1^i \pa_2^j \psi)_{\infty, \al}(\hat x), \quad 
 \f{\psi_{\lam}(\hat x)}{ \vp_{\lam}(\hat x)} \leq \lam^{ b_{\al} - a_{\al}} ( \f{\psi}{\vp} )_{\inf, \al}(\hat x). 
\eeq
uniformly for $\lam \leq \lam_1$ for $\al = 1$, i.e. $x$ close to $0$, and $\lam \geq \lam_n$ for $\al = n$, i.e. $x$ in the far-field, which have been established in Appendix {\appwgradial}, {\appwgmix} in Part II \cite{ChenHou2023b_supp}.
Note that we do not pick up extra $\lam$ power in the asymptotic analysis when we take derivatives $ \pa_{\hat x_i}^j$ on $ \psi_{\lam}^{-1}$. For example, if $\psi(x) = |x|^{-2}$, $ \pa_{\hat x_1} \psi_{\lam}^{-1}( \hat x) =  \pa_{\hat x_1}  \lam^2 |\hat x|^2 = \lam^2 2 \hat x_1$. Using these estimates, we derive 
\[
III_1 \leq \lam^{-a_{\al}} |a| C(\hat x),
\]
for some constant $C(\hat x)$. 

For $III_2$, using the estimates in Section \ref{sec:singu_hol_cost}, we yield 
\[
 |\psi_{\lam}(\hat x) III_2| 
 \leq C_1(\hat x)  || \f{\psi_{\lam}}{ \vp_{\lam}} ||_{L^{\inf}(R(b))} || \om_{\lam} \vp_{\lam}||_{L^{\inf}}
 + C_2(\hat x) |a|^{1/2} [ \om_{\lam} \psi_{\lam}]_{C_x^{1/2}}
 + C_3(\hat x)  |a|^{1/2}[ \om_{\lam} \psi_{\lam}]_{C_y^{1/2}},
\]
for some $C_i(\hat x)$ independent of the weights $\vp$ and $\lam$, which can be derived using the same method as that in Section \ref{sec:singu_hol_cost}. Using the scaling relation \eqref{int:scal_w}, we yield 
\[
 |\psi_{\lam}(\hat x) III_2| \leq 
 C_1(\hat x)
 \lam^{ a_{\al} - b_{\al}}  || \om \vp||_{L^{\inf}} 
 +
 C_2(\hat x)  |\lam a|^{1/2} [ \om \psi ]_{C_x^{1/2}}
 + C_3(\hat x) |\lam a|^{1/2} [ \om \psi]_{C_y^{1/2}},
\]
where we have changed $C_1(\hat x)$ to track the constant $|| \f{\psi_{\lam}}{ \vp_{\lam}} ||_{L^{\inf}(R(b))}$.

Using the above estimates and \eqref{eq:w_est}, we can bound  the singular part as follows 
 \beq\label{eq:int_1D_sg}
 \bal
 (\na \uu)_S \leq & \lam^{-a_{\al}} ( C_1(\hat x) +  C_2(\hat x) \log \f{b}{a} + C_3(\hat x) |a| )  || \om \vp ||_{\infty}  + \lam^{-b_{\al} + \f{1}{2} } (C_{41}(\hat x)  [ \om \psi]_{C_x^{1/2}}
+ C_{42}(\hat x)  [ \om \psi]_{C_y^{1/2}}) \\
\leq & 
E_1 \B(  \lam^{-a_{\al} }( C_1(\hat x) +  C_2(\hat x) \log \f{b}{a} + C_3(\hat x) |a| )
+  \B( \g_1 C_{41}(\hat x) + \g_2 C_{42} (\hat x)  \B) \lam^{-b_{\al} + 1/2}  \B),
\eal
 \eeq
for $x = \lam \hat x$ close to $0$ , $\al = 1$, and for $x = \lam \hat x$ in the far-field, $\al = n$, with $\hat x$ in each $\G_{1, i}, \G_{2, i}$ \eqref{eq:int_1D_dom1}.

 We perform the above estimate for a list of $a, a_1 <a_2< .., < a_N \leq b$ and uniform estimate for $a h \leq h \leq bh$. We will optimize $a$ to obtain a sharp estimate in Section \ref{sec:1D_linf_far}.

\subsubsection{Summary of the estimates and scaling}
Combining the estimates in Section \ref{sec:int_1D_nsg}, e.g. \eqref{eq:int_1D_nsg}, and \eqref{eq:int_1D_sg}, we obtain the estimate for $\na \uu$. From the above estimates, we can apply the same estimates as those in \cite{ChenHou2023a_supp} and Section \ref{sec:singu_hol_cost} except that we need to perform the estimates using the asymptotic properties of the weights \eqref{eq:wg_asym}, \eqref{eq:wg_asym2}, \eqref{eq:wg_asym3} uniformly for small $\lam$ or large $\lam$, instead of the weights $\vp^{-1}, \f{\psi}{\vp}, \pa_1^i \pa_2^j ( \psi^{-1} )$ in the estimates for the finite $\lam$ case. 
Moreover, we need to track the power in $\lam$. In particular, we can obtain piecewise estimates 
for $\uu, \na \uu$ \eqref{eq:int_1D_res} for $x \in \G_{i, j}(x_c)$ \eqref{eq:int_1D_dom1} and $\lam \leq \lam_*$ or $\lam \geq \lam_*$ uniformly. 
The estimates weighted by $\psi_{\lam}(\hat x)$ can be obtained similarly.

For $\uu(\lam \hat x)$, the estimate is much easier since the kernel is locally integrable. 
Using the argument in Section \ref{sec:int_1D_nsg}, the scaling relation \eqref{eq:int_1D_res} with $d = 1$,
and the same method in \cite{ChenHou2023a_supp} except that we use the weight $\vp_{\inf, \al}$ \eqref{eq:wg_asym2} instead of $\vp$, we can bound the integral and $\uu(\lam \hat x)$ by
\[
C(\hat x)  \lam^{ -a_{\al} + 1} || \om \vp ||_{\infty}.
\]
We get the power $1$ since the kernel in $\uu$ is $-1$ homogeneous and we use the scaling relation \eqref{eq:int_1D_res} with $d=1$.

\begin{remark}\label{rem:scal_law}
To track the $\lam$ power in the estimates \eqref{eq:int_1D_nsg}, \eqref{eq:int_1D_sg},  we have the power $\lam^{-a_{\al}}$ in the coefficient of $|| \om \vp||_{\infty}$ since $\vp^{-1}(\lam x)$ has asymptotic scaling property $\lam^{-a_{\al}}$\eqref{eq:wg_asym2}, and $\lam^{-b_{\al} + 1/2}$ in the coefficient of $[ \om \psi]_{C_{x_i}^{1/2}}$ since $\psi^{-1}(\lam x)$ has asymptotic scaling property $\lam^{-b_{\al}}$ \eqref{eq:wg_asym}, \eqref{eq:wg_asym2} and we get an extra $\lam^{1/2}$ from the scaling law \eqref{int:scal_w}. In the weighted estimate, e.g. $\psi_{\lam} \na \uu$, we will obtain an additional power $\lam^{b_{\al}}$ in the upper bound since the weight $\psi_{\lam}$ has asymptotics $\lam^{-b_{\al}}$. In the estimate of $\uu$, we gain the factor $\lam$ in $\lam^{-a_{\al} + 1}$ since $\uu$ is 1 order more regular than $\na \uu$.
\end{remark}

\subsection{H\"older estimates of $\na \uu , \uu$ }\label{sec:1D_hol}

Denote 
\beq\label{eq:int_1D_hol_dom1}
\bal
 & x_{c2} = 2 x_c, \ \mu = 1/8, \quad 
\Om_1 \teq [x_{c2}, (1 + \mu) x_{c2} ] \times [0, x_{c2}] \cup [0, (1 + \mu) x_{c2}] \times \{  x_{c2}\} \\
& \Om_2 \teq \{ x_{c2} \} \times [0, (1 + \mu) x_{c2} ] \cup [ 0, x_{c2} ] \times [ x_{c2} , (1 + \mu) x_{c2}].
\eal
\eeq

Firstly,  we choose $\lam$ and rescale $x = \lam \hat x$ with $ |\hat x|_{\inf} = x_{c2}$.  In the $C_x^{1/2}$ estimate, for $z = \lam \hat z$ with $x_1 \leq z_1, z_2 = x_2$, if $|\hat x - \hat z | \leq \mu x_{c2}$, we have $\hat z  \in \Om_1$. In the $C_y^{1/2}$ estimate, for $z = \lam \hat z$ with $x_2 \leq z_2, x_1 = z_1$, we have $z \in \Om_2$. For $|\hat x - \hat z | > \mu x_{c2}$, we apply the triangle inequality to estimate $\na \uu(x) - \na \uu(z)$ directly. In Section {\secholasym} in Part II \cite{ChenHou2023b_supp}, we discuss the estimate of the derivatives of the regular part of the integrand in \eqref{eq:int_1D_res}
\[
\pa_{\hat x_i} J( \hat x, \hat y) , \quad J = K(\hat x, \hat y) p_{\lam}( \hat x),  \mathrm{ \ or \ }
J = K^C( \hat x,\hat y) ( p_{\lam}(\hat x) - p_{\lam}(\hat y)) + K^{NC}(\hat x,\hat y) p_{\lam}(\hat x),
\]
with weight $p( x) = \psi(x)$, where $K^C, K^{NC}$ are parts of the symmetrized kernel $K^{sym}$ 
\eqref{eq:ker_sym} determined by the distance between the $y$ and the singularity $x$. See Section {\secintsym} in Part II \cite{ChenHou2023b_supp}. Using such estimates for $ \pa_{\hat x_i} J $, we can estimate the integral of $\pa_{\hat x_i} J$ in terms of $|| \om \vp||_{\infty}$ using the method in Section \ref{sec:int_1D_nsg} and the method in Section {\secvelholcomp} in Part II \cite{ChenHou2023b_supp} with weight $\vp_{\inf, \al}$ \eqref{eq:wg_asym2}. It particular, these parts are bounded by 
\[
C(\hat x) \lam^{ b_{\al}- a_{\al}} || \om \vp||_{\infty}.
\]
We have the power $\lam^{b_{\al}}$ from the weight $p_{\lam} = \psi_{\lam}$ \eqref{eq:wg_asym}. Here, $\al = 1$ means that we estimate $x = \lam \hat x$ for very small $x$ and $\lam$, and $\al = n$ means that we estimate $x = \lam \hat x$ for very large $x$ and $\lam$. We use the same notatons below.

\begin{remark}
We estimate $\pa_{\hat x_i} J$ rather than  $\pa_{ x_i} J$. We have the scaling relation \eqref{int:scal_w} between these two quantities. 
\end{remark}

The remaining part is the integral in the singular region (see $I_5$ in Section 4.3.4 in Part II \cite{ChenHou2023b_supp})
\[
\bal
I_5(\hat x, k_2) 
&= \int_{R(k_2) } K(\hat x-\hat y) ( \psi_{\lam}( \hat x)- \psi_{\lam}(\hat y)) W_{\lam}(\hat y) d \hat y .
\eal
\]
In Part II \cite{ChenHou2023b_supp}, we discuss the case of finite $\lam$ and choose $\lam = 1$ for simplicity. The case of general $\lam$ is given above using the rescaling relation \eqref{eq:int_1D_res}. For the H\"older estimate, in Sections {\sechol} in Part II \cite{ChenHou2023b_supp}, we decompose $I_5$ into several parts and estimate the piecewise $L^{\inf}$ norm or the piecewise derivatives bounds for each part. 
For each part, we can bound it using the weights 
\[
\psi_{\lam}(\hat x) \pa_{\hat x_1}^i \pa_{\hat x_2}^j \psi^{-1}_{\lam}(\hat x), \quad 
\psi_{\lam}(\hat x) / \vp_{\lam}(\hat x),
\]
the norm $|| W_{\lam} \vp_{\lam}||_{L^{\inf}}$ and the semi-norm $[ \om_{\lam} \psi_{\lam}]_{C_{x_i}^{1/2}}$. The above weights are the same as those in \eqref{eq:wg_asym3}, and we establish the uniform estimates in Appendix {\appwgradial}, {\appwgmix} in Part II \cite{ChenHou2023b_supp}. Using the scaling relation \eqref{int:scal_w}, we can bound each part by 
\[
C_1( \hat x) \lam^{b_{\al} - a_{\al}} || \om \vp||_{L^{\inf}} 
+ C_2(\hat x) \lam^{1/2} [ \om \psi]_{C_{x_l}^{1/2} },
\]
for $l = 1$ or $2$.
In the $C_{x_l}^{1/2}$ estimate, we only need the semi-norm $ [ \om \psi]_{C_{x_l}^{1/2} }$. The above estimates allow us to estimate different parts in the H\"older estimate in Section {\secvelholcomp} in Part II \cite{ChenHou2023b_supp}. Similar to the $L^{\inf}$ estimate in Section \ref{sec:1D_linf}, we can apply essentially the same estimate in \cite{ChenHou2023a_supp} for the case of finite $\lam$ to the case of $\lam \to 0$ or $\lam \to \infty$ except that we need to use the asymptotic properties of the weights \eqref{eq:wg_asym}, \eqref{eq:wg_asym2}, \eqref{eq:wg_asym3} instead of using $\vp_{\lam}^{-1}, \pa_x^i \pa_y^j \psi_{\lam}^{-1} $ etc for a fixed $\lam$, and track the power. 

Similarly, in the H\"older estimate of $\psi_{u, \lam}(x) \uu(\lam x)$ (see Section 4.3.8 in Part II \cite{ChenHou2023b_supp}), we can bound the piecewise $L^{\inf}$ norm or derivatives of each part by 
\beq\label{eq:hol_1D_u}
 C(\hat x) \lam^{1  + c_{\al} - a_{\al}} || \om \vp||_{\infty},
\eeq
where $c_{\al}$ is the leading order power of $ \psi_{u, \lam}$ \eqref{eq:wg_asym}. We refer to Remark \ref{rem:scal_law} for tracking the power in the upper bounds.

Once we obtain the above bounds for each parts, we assemble the H\"older estimate following Section {\secholcomb} in Part II \cite{ChenHou2023b_supp} and the scaling relation \eqref{int:scal_w}. See Section \ref{sec:hol_R2}.

\subsection{Assemble $L^{\inf}$ estimate of $\uu, \na \uu$ in $\R_2^{++}$ }\label{sec:linf_R2}

We use the method in Section {\secvellinf} in Part II \cite{ChenHou2023b_supp} to estimate the piecewise bounds of $\uu(x), \na \uu(x)$ for $\hat x$ in each $h_x \times h_x$ grid $B_{i_1 j_1 }(h_x) \subset  [0, 2 x_c]^2 \bsh [0, x_c]^2 $
\[
x = \lam \hat x, \quad \hat x \in [0, 2 x_c]^2 \bsh [0, x_c]^2 = \cup_{i_1, j_1} B_{i_1 j_1 }(h_x),
\]
with finitely many $\lam = \lam_i, i =1, 2, .., n_{1}$. This allows us to bound $\uu, \na \uu$ for $x \in \lam_{n_1} [0, 2 x_c]^2 \bsh \lam_1 [0, x_c]^2$.  To assemble the estimate in $\R_2^{++}$, we first pass the estimates from the dyadic mesh to the mesh \eqref{eq:mesh_xy} for computing the approximate steady state. For each mesh $Q_{ij}$ \eqref{eq:mesh_xy} with $ \lam_1 x_c \leq \max(y_i, y_j) 
\leq \max(y_{i+1}, y_{j+1}) \leq 2 \lam_{n_1} x_c $, we can cover  $Q_{ij}$ by finitely many dyadic mesh $M_{li_1j_1}$ \eqref{eq:int_1D_dom1}. Then we use 
\[
|| f||_{L^{\inf}(Q_{ij})}  \leq \max_{Q_{ij} \cap M_{l i_1 j_1} \neq \emptyset } || f||_{L^{\inf}( M_{l, i_1, j_1})}
\]
to obtain the piecewise bound on $Q_{ij}$.
In our implementation, we loop over $l, i_1, j_1$ and update the bound in $Q_{ij}$ if $Q_{ij} \cap M_{l i_1 j_1} \neq \emptyset$.  The same method applies to piecewise bounds of weighted velocity $\rho_{10} \uu, \psi_1  \na \uu$ for some weights since we have piecewise bounds for the weights. See Appendix {\appwgradial}, {\appwgmix} in Part II \cite{ChenHou2023b_supp}.

\subsubsection{Near-field and far-field}\label{sec:1D_linf_far}

We need to further bound $\uu, \na \uu$ for $x \in \lam_1 [0, x_c]^2$ and $|x|_{\inf} \geq 2 \lam_n x_c$. We focus on the weighted estimate $ (p \na \uu)(x)$. The estimate for $\uu$ is similar. To estimate $\uu, \na \uu$ in the far field, we use the estimate in Section \ref{sec:1D_linf} with $\lam \geq \lam_*, \lam_* x_c \leq 2 \lam_n x_c$. 

We focus on the estimate of the singular part \eqref{eq:int_1D_sg} in the far-field. The estimates of other parts are easier. In the first estimate, we assume that $x$ is in the computational domain $x =\lam \hat x \in  [y_i, y_{i+1}]\times [y_j, y_{j+1}]= Q_{ij}$ with $\lam = \f{|x|_{\inf} }{x_c}$ \eqref{int:box_B} and $|\hat x|_{\inf} = x_c$. We can cover the range of $\hat x $ by finitely many intervals $\G_{1, i_1}, \G_{2, j_1}$ in \eqref{eq:int_1D_dom1}. For $\hat x$ in each interval, from \eqref{eq:int_1D_sg}, we have 
\beq\label{eq:1D_linf_est1}
|(\na \uu)_S( x) | \leq 
E_1 \B(  \lam^{-a_{n} }( C_0  +  C_1  \log \f{b}{a} + C_2  |a| ) 
+ C_3  |a|^{1/2} \lam^{- b_{n} + 1/2 })  \teq  g(\lam, a) \cdot E_1  ,
\eeq
for some constants $C_i$ depending on the interval, where $a_n, b_n$ are the slowest decay power in $\vp, \psi$ \eqref{eq:wg_asym}. Suppose that  
\beq\label{eq:1D_linf_est2}
 p(x) \leq c |x|^{\al}.
\eeq

We can obtain the upper and lower bound of $\lam / |x| = O(1), \lam, |x|$ for $x \in Q_{ij}$. 
We first bound $\lam^{\al}$ by $\max( \lam_l^{\al}, \lam_u^{\al} )$ and then  optimize the estimate \eqref{eq:1D_linf_est1} over finite many $a = a_i$, established in \eqref{eq:int_1D_sg}. In the second estimate, we minimize the upper bound in \eqref{eq:1D_linf_est1} by
choosing 
\beq\label{eq:1D_linf_est3}
a = \min(h, \min( f(\lam_l), f(\lam_u) ) ), \quad f(\lam) = ( \f{2 C_1}{C_3} )^2 \lam^{2 (b_n - a_n) -1} 
\eeq
With the above choice of $a$, the upper bound in \eqref{eq:1D_linf_est1} is of order $\lam^{- a_n} \log \lam$. We further use the piecewise bounds for $\lam / |x|$, $|x|$, and \eqref{eq:1D_linf_est2} to obtain piecewise bounds of $ p \na \uu(x)$ for $x \in Q_{ij}$ and $\hat x$ in one of the intervals $\G_{1, i_1}, \G_{2, j_1}$. Taking the maximum of different cases yields the piecewise estimates in $Q_{ij}$.

For $x$ sufficiently large $\max x_i \geq R$, $x$ is outside the domain spanned by $Q_{ij}$. We take $\lam_* = \f{R}{x_c}$. In this case, we estimate $p(x) \na \uu(x)$ uniformly for $x$ with $|x|_{\inf} \geq R$ \eqref{int:box_B}. We assume \eqref{eq:1D_linf_est2} for all $x$ with $\max x_i \geq R$ and require $\al < \min(a_n, 0)$. We optimize \eqref{eq:1D_linf_est1} by choosing $a = f(\lam)$ \eqref{eq:1D_linf_est3}.

Recall the power $(a_{\al}, b_{\al})$ of the weights  \eqref{eq:wg_asym}. Since $2 (b_n - a_n) - 1 \leq -\f{1}{2}$ and $\lam$ is sufficiently large, we have $a = f(\lam) \leq h$ in our application. With the above choice of $a$, we get 
\[
 C_3 a^{1/2} \lam^{ -b_n + 1/2} = C_3 \f{2 C_1}{C_3} \lam^{-b_n + 1/2 + b_n - a_n - 1/2} 
 = 2 C_1 \lam^{-a_n} ,
\]
and $g(\lam, a)$ \eqref{eq:1D_linf_est1} reduces to 
\[
g(\lam, a) = \lam^{-a_n} ( C_0 + C_1 \log \f{b}{a} + C_2 (\f{2 C_1}{C_3})^2 \lam^{2 (b_n - a_n)-1} + 2 C_1 ) . 
\]

\vs{0.1in}
\paragraph{\bf{Estimate of the logarithm term}} We need to further control the term $\lam^{-a_n}\log \f{b}{a}$. Denote 
\[
F( \lam, b, C, d, \b) =  \log( b / A(\lam)) \lam^{C},  \quad A(\lam) = d \lam^{-\b},
\]
for some $\b > 0 > C$. We maximize it over $\lam \geq \lam_*$. Clearly, we have 
\[
 \f{d}{d\lam} F(\lam) = \lam^{C - 1}  C \log \f{b}{A(\lam)} 
 - \lam^C \f{A^{\prime}}{A} 
 = \lam^{C - 1}  C \log \f{b}{A(\lam)} +\b \lam^{C-1}
 = \lam^{C-1} ( C \log \f{b}{A(\lam)} +\b ).
\]
The critical point $\lam_c$ is given by 
\[
A(\lam_c) = b e^{ \b / C}, \quad \lam_c = (\f{d}{A(\lam_c)})^{1/\b} 
= (\f{d}{b})^{1/\b} e^{-1/C}.
\]

For $\lam \leq \lam_c$, $F(\lam)$ is increasing. For $\lam \geq \lam_c$, $F(\lam)$ is decreasing. Thus for $\lam \geq \lam_*$, we get 
\[
F(\lam, b, C, d, e) \leq  F( \max( \lam_c, \lam_*), b, C, d, e) .
\]
Finally, using $|x| \geq \max_i(x_i) = \lam x_c$, $\al \leq \min(a_n, 0), \lam \geq \lam_*$, and $ p \leq c |x|^{\al}$ \eqref{eq:1D_linf_est2}, we obtain 
\[
\bal
|p (\na \uu)|
&\leq c x_c^{\al} g(\lam, b) \lam^{\al} E_1
\leq  c x_c^{\al} E_1 \B( (C_0 + 2 C_1) \lam_*^{\al - a_n} \\
& + C_2 (\f{2 C_1}{C_3})^2 \lam_*^{2(b_n - a_n) - 1 + \al - a_n} +
C_1 F( \lam, b, \al - a_n,  (\f{ 2C_1}{C_3}  )^2,  2 (-b_n + a_n) + 1)
  \B) ,
  \eal
\]
where $F(\lam)$ is further bounded using the above estimate.

\vs{0.1in}
\paragraph{\bf{Estimate of the regular part}}

From Section \ref{sec:int_1D_nsg}, the regular part of $\na \uu$ with approximation terms can be bounded by 
\[
|(\na \uu)_R |\leq  C_1(\hat x) \lam^{-a_n}  + C_2(\hat x),
\]
where $C_i(\hat x)$ are some piecewise constants and $C_2(\hat x)$ can be negative. For a weight satisfying \eqref{eq:1D_linf_est2} and $x \in Q_{ij} = [y_i, y_{i+1}] \times [y_j, j_{j+1}]$, it is easy to obtain the upper and lower bounds for $r = |x|$ and $\lam = \f{\max x_i}{x_c}$. Then using 
\beq\label{eq:1D_linf_est4}
 \min( C t_l^{\al}, C t_u^{\al}) \leq  C t^{\al} \leq \max( C t_l^{\al}, C t_u^{\al}), \quad (f g)_u = \max( f_l g_l, f_l g_u, f_u g_l, f_u g_u),
\eeq
for any $C, \al$, we obtain the upper bound in $Q_{ij}$. For $x$ sufficiently large $\max x_i \geq R$, we assume that the weight $p$ satisfies \eqref{eq:1D_linf_est2} with $\al < a_n$ so that $|p (\na \uu)_R|$ decays for large $x$. In this case, we have $r_l = R, r_u = \infty$, $\lam_l = \f{R}{x_c}, \lam_u = \infty$, and the estimates follow from \eqref{eq:1D_linf_est4}.

The weighted estimate of $\na \uu$ in the near field and the estimate for $\uu$ are similar.

\vs{0.1in}
\paragraph{\bf{Estimate of $ u_A / |x|^{1/2}$}}

In the energy estimate, we need to control $ u_A / |x|^{1/2}$ in the far-field, where $ u_A = -\pa_y (\phi + L_{12}(0) xy )$  \eqref{eq:u_asym_far}. 
We want to control $ \f{u_A }{x^{1/2}} r^{\g}$ for some $\g < -1$ and $|x|_{\inf} \geq R$ \eqref{int:box_B}. Firstly, using the previous methods and $u_{x, A} = u_{A, x}$ \eqref{eq:u_diff}, we obtain 
\[
| u_{A,x}(x) | r^{\g+1} \leq C_2 , \quad |u_A(x)|  r^{\g}  \leq C_1, \quad |x|_{\inf} \geq R.
\]

Denote $r = |x|, \b = \atan \f{x_2}{x_1}$. For $x_1 \geq x_2$, we get $\b \leq \pi / 4$,  $x_1
= |x|_{\inf} \geq R, x_1 \geq r / \sqrt{2}$. Thus, we get
\beq\label{eq:int_1D_u_dx1}
 \f{| u_A|}{|x_1|^{1/2}} r^{\g} \leq \f{C_1}{ ( r \cos \b)^{1/2} } ,  \ \b \in [0, \f{\pi}{2} ],
 \quad 
 \f{| u_A|}{|x_1|^{1/2}} r^{\g} \leq  \min(  \f{C_1}{  R^{1/2} }, \f{C_1 2^{1/4}}{r^{1/2}} ), \quad \b \leq \f{\pi}{4} .
\eeq

Using the estimate for $u_{A,x}$ and $\g + 1 < 0$, we yield 
\[
| u_A| =\int_0^{x_1} | u_{A,x}(z, x_2)| d x_1
 \leq \int_0^{x_1} C_2 |(z, x_2)|^{ -\g-1} d x_1
 \leq  C_2 x_1 r^{-\g-1},
 \quad  \f{|  u_A| r^{\g} }{ |x_1|^{1/2}}\leq C_2  \f{  |x_1|^{1/2}  }{r}
 \leq \f{C_2 (\cos \b)^{1/2}}{ r^{1/2} }.
\]

Combining the above two estimates, for $x_2 > x_1$, we establish 
\[
 \f{| u_A| r^{\g}}{|x_1|^{1/2}}
 \leq r^{-1/2} \min( \f{C_1}{ (\cos \b)^{1/2} }, C_2 (\cos \b)^{1/2})
\leq R^{-1/2} \min( \sqrt{C_1 C_2}, \f{C_1}{ (\cos \b^u)^{1/2} }, C_2(\cos \b^l)^{1/2} ).
\]

Combining \eqref{eq:int_1D_u_dx1} and the above estimate, we establish the estimate for $u_A / |x_1|^{1/2}$.

\subsection{Assemble the H\"older estimate in $\R_2^+$}\label{sec:hol_R2}

In the remaining part of this section, $\uu, \na \uu$ denote the velocity with approximations, i.e. $\uu_A, (\na \uu)_A$. We do not write down the approximation terms to simplify the notations.

In Section {\secholcomb} in Part II \cite{ChenHou2023b_supp}, we discuss how to assemble the estimates of different parts of \eqref{eq:int_1D_res} to obtain the $C_{x_i}^{1/2}$ estimate $\d_i(f, x, z)$ \eqref{eq:hol_dif}
\beq\label{eq:hol_dif}
\d_i(f, x, z) \teq \f{ |f( x ) - f(z)|}{|x-z|^{1/2}} , \quad  z_i > x_i, \ z_{3-i} = x_{3-i},
\eeq
for $f$ being $\uu, \na \uu$ with the approximation terms, $x = \lam \hat x, z = \lam \hat z$ with $\hat x  \in [0, 2 x_c]^2 \bsh [0, x_c]^2, |\hat z - \hat x| \leq 2 \nu x_c$, and $\lam = \lam_1, \lam_2, .. \lam_{n_1}$.
For $x, z$ with $2 \nu x_c \leq |\hat x-\hat z| \leq  2 \nu_2 x_c$, we get $|x-z| \geq 2 \nu \lam x_c \geq \nu |x|_{\inf}$ and apply triangle inequality to estimate $\d_i(f, x, z)$ for $\hat x, \hat z$ in a larger domain. 
Then, in the $C_x^{1/2}$ estimate, we obtain the piecewise estimate of $\d_i(f, x, z)$ on the dyadic mesh \eqref{eq:int_1D_dom1} 
\beq\label{eq:hol_dyad1}
x \in M_{l p q}, \  z \in M_{l r q},  \ p \leq r \leq p + m_1, \ \hat x \in B_{pq}(h_x) \subset [0, 2 x_c]^2 \bsh [0, x_c]^2.
\eeq
In the above meshes, $l$ denotes the scaling factor $\lam_l$. We have the same subindex $q$ for $x, z$ since $x_2 = z_2$. Next, we pass the estimate from the dyadic mesh to the mesh $Q_{ij}$ \eqref{eq:mesh_xy}.

\subsubsection{The H\"older estimate in the bulk}\label{sec:hol_R2_bulk}
Recall $M_{l, i_1,j_1} $ from \eqref{eq:int_1D_dom1}. We focus on the $C_x^{1/2}$ estimate. 
For $x \in Q_{i_1 j_1} \subset \lam_{n_2} [0, 2 x_c]^2 \bsh \lam_1 [0, x_c]^2$, and $z \in Q_{i_2, j_1}$ close to $x$  $i_1 \leq i_2 \leq i_1 + m$, 
by covering  $Q_{i_1 j_1}, Q_{i_2 j_1}$, we obtain 
\beq\label{eq:hol_cov1}
\max_{ x \in Q_{i_1 j_1} , z \in Q_{ i_2, j_1} } |\d_1(f, x, z)|
\leq \max_l \B( \max_{  x \in  M_{l pq} \cap Q_{i_1 j_1} \neq \emptyset, z \in  M_{l s q} \cap Q_{i_2 j_1} \neq \emptyset  
 } \d_1( f, x,  z)  \B).
\eeq

Suppose that $Q_{i_1 j_1} \subset \cup M_{l p q}$. For each $x \in M_{l p q}$,  we have $\hat x \in B_{p q}, \hat z_2 = \hat x_2 \in I_q$.
Denote 
\[
I_q = [q h_x, (q+1  ) h_x],  
\quad S_z \teq \{ a \in \R^2 : a_2 \in \lam_l I_q \}  \cap Q_{i_2 j_1}, 
\quad C_z \teq \cup_{p \leq p_2 \leq p + m_1} M_{l p_2 q}.
\]
The set $S_z$ is the location of $z \in Q_{i_2j_1}$  since $z_2 = x_2$, and $C_z$ is the location of $z$ we have computed $\d_1(f, x, z)$ in the dyadic mesh \eqref{eq:hol_dyad1}. 
For each $Q_{i_2 j_1}$  $i_2 \leq i_1 + m$, we have two cases
\[
(a) \  y_{i_2 + 1} \leq \lam_l (p + m_1 + 1) h_x, \quad (b)\  y_{i_2 + 1} > \lam_l (p + m_1 + 1) h_x ,
\]
where $y$ is the mesh \eqref{eq:mesh_xy}.  In case (a), $z$ is close to $x$ or $i_2$ close to $i_1$, e.g. $i_2 = i_1$. We estimate $\d_1(f, x , z)$ using the estimate of $\d_1(f, x, z)$ on the dyadic mesh \eqref{eq:hol_dyad1} $C_z$ we have computed.

In case (b), 
the mesh $C_z$ does not cover $S_z$, the range of $z$ with $z_1 \in [y_{i_2}, y_{i_2 + 1} ], z_2 \in I_q$.
For $z_1 \in [ y_{i_2}, \lam_l (p + m_1+1)h_x]$, $C_z$ covers parts of $S_z$ (or none), and we use the estimate on the dyadic mesh \eqref{eq:hol_dyad1} to estimate $\d_1(f, x,z)$. For $z_1 \in [ \max( y_{i_2} , \lam_l (p + m_1 +1)h_x), y_{i_2+1} ]$, 
since $x \in Q_{i_1 j_1} \cap M_{l p q}$ and we choose  $\lam_l = 2^{k_l}$ for $k_l\in Z$, we have 
\beq\label{eq:hol_R2_est1}
x_1 \leq \lam_l (p+1) h_x, \ \max(y_{i1}, y_{j1}) \leq |x|_{\inf} = \lam_l |\hat x|_{\infty} \leq \lam_l 2 x_c, \ 
\lam_l \geq 2^{ \lceil \log_2(  \f{\max( y_{i_1}, y_{j_1}) } {  2 x_c}  \rceil   } \teq \lam_{i_1, j_1}, 
\eeq
where $\lceil a \rceil$ is the smallest integer no less than $a$. It follows 
\[
z_1 - x_1 \geq \lam_l m_1 h_x \geq \f{m_1 h_x}{ 2x_c}  |x|_{\inf}
\geq \f{m_1 h_x}{ 2x_c} \max( y_{i_1}, y_{j_1}), \quad 
z_1 - x_1 \geq   \max( y_{i_2} - y_{i_1+1} , 0).
\]

We apply the triangle inequality, the piecewise bound of $f =\na \uu$ in $Q_{i_2j_1}, Q_{i_1 j_1}$
established in Section \ref{sec:linf_R2}, and the above lower bound on $|x-z|$ to estimate $\d_1(f, x, z)$. 
Since $x_1 \in \lam_l h_x [ p, p+1]$ ($x \in M_{l p q}$), conditions $y_{i_2 + 1} > \lam_l ( p + m_1 + 1) h_x , x \in Q_{i_1 j_1} $, imply
\[
y_{i_2+1}  - y_{i_1} \geq y_{i_2 + 1} - x_1 \geq  y_{i_2+1} - \lam_l (p+1) h_x 
\geq \lam_l m_1 h_x 
\geq m_1 h_x 2^{ \lceil \log_2(  \f{\max( y_{i_1}, y_{j_1}) } {  2 x_c}  \rceil   }.
\]

Recall  $ \lam_{i_1, j_1}$ from \eqref{eq:hol_R2_est1}. For $i_1 \leq i_2$, we introduce
\[
D(i_1, i_2, j_1) = \max(\lam_{i_1, j_1} h_x ,  \max( y_{i_2} - y_{i_1+1} , 0) ),
\quad  
S \teq \{ (i_1, i_2, j_1) : y_{i_2+1}  - y_{i_1} \geq \lam_{i_1,  j_1}  h_x m_1 \} .
\]
 From the above discussion, we can modify the estimate \eqref{eq:hol_cov1} as follows 
\[
\bal
\max_{ x \in Q_{i_1 j_1} , z \in Q_{ i_1, j_2} } |\d_1(f, x, z)| 
&\leq \max \B\{ \max_l \B( \max_{  x \in  M_{l pq} \cap Q_{i_1 j_1} \neq \emptyset, z \in  M_{l p_2 q} \cap Q_{i_2 j_1} \neq \emptyset ,  
 p \leq p_2  \leq p + m_1 } \d_1( f, x,  z)  \B) , \\
 & \qquad \one_{ S}(i_1, i_2, j_1) \f{ || f||_{\inf(Q_{i_1 j_1})} + || f ||_{\inf}(Q_{i_2 j_1}) }{ 
D_{i_1, i_2, j_1}^{1/2} }
  \B\} .
 \eal
 \]

Note that $ D(i_1, i_2, j_1), \lam_{i_1, j_1} $, $S$, and the second bound in the maximum are independent of the dyadic estimates, and only depend on the piecewise estimates on mesh $Q_{ij}$. 
The above estimate allows us to obtain a piecewise $C_x^{1/2}$ estimates for  $x \in Q_{i_1,j_1}, z \in Q_{i_2, j_1}$ with $|i_1 - i_2| \leq m$ and $x$ in the bulk of the domain. Similarly, we obtain the piecewise bounds for the $C_y^{1/2}$ estimate.

\subsubsection{The H\"older estimate in the far-field and near-field for small distance}\label{sec:hol_R2_1D}

We focus on the estimate for $\na \uu$. The estimate of $ \uu$ is similar and easier. Firstly, 
in the $C_{x_i}^{1/2}$ estimate 
using \eqref{eq:w_est} and the estimates in Section \ref{sec:1D_hol}, we can decompose \eqref{eq:int_1D_res} into several parts and bound the piecewise $L^{\inf}$ norm or derivatives in $\hat x$ for each part in a small domain $B_{i_1 j_1}(h_x) \subset \Om_i$ \eqref{int:box_B}, \eqref{eq:int_1D_hol_dom1} by 
\[
C_1( \hat x) \lam^{b_{\al} - a_{\al}} || \om \vp||_{L^{\inf}} 
+ C_2(\hat x) \lam^{1/2} [ \om \psi]_{C_{x_i}^{1/2} }
\leq \lam^{1/2} ( C_1( \hat x) \lam^{b_{\al} - a_{\al} -1/2} + \g_i C_2(\hat x)  ) E_1 .
\]

Recall the power $a_{\al}, b_{\al}$ of the weights \eqref{eq:wg_asym}. We have 
\[
b_1 - a_1 - \f{1}{2} \geq  -2 + \f{5}{2} -\f{1}{2}  \geq  0, \quad b_n - a_n - \f{1}{2} 
\leq - \f{1}{6} + \f{1}{6} - \f{1}{2} \leq -\f{1}{2} \leq 0.
\]

In the near field $\lam \leq \lam_* \ll 1$ or the far field $\lam \geq \lam_*$, using $\lam^{b_{\al} - a_{\al} - 1/2} \leq \lam_*^{ b_{\al} - a_{\al} - 1/2 }$, we can bound the estimate of each part uniformly in $\lam$ by 
\[
\lam^{1/2} ( C_1( \hat x) \lam_*^{b_{\al} - a_{\al} -1/2} + \g_i C_2(\hat x)  ) E_1 .
\] 

The power $\lam^{1/2}$ is factorized out in the estimates. Thus we can treat the H\"older estimate  $\d_i( f , \lam \hat x, \lam \hat z) $ \eqref{eq:hol_dif} for $\lam \geq \lam_*$ or $\lam \leq \lam_*$ in the same way as that for finite $\lam$ case using the $\lam$-independent bound $( C_1( \hat x) \lam_*^{b_{\al} - a_{\al} -1/2} + \g_i C_2(\hat x)  ) E_1 $ and tracking the power $\lam^{1/2}$. Using the scaling relation \eqref{int:scal_w} and the estimates in Section {\secholcomb} in Part II \cite{ChenHou2023b_supp}, we obtain the weighted H\"older estimate $\d_i(f, \lam \hat x, \lam \hat z)$ \eqref{eq:hol_dif} for $f = \psi \na \uu$  and $|\hat x|_{\inf} = x_{c2}, |\hat z - \hat x | \leq \mu x_{c2}$ uniformly for $\lam \leq \lam_*$ or $\lam \geq \lam_*$.  Note that the power $\lam^{1/2}$ in the above estimate and that in \eqref{int:scal_w} are exactly canceled.

\vs{0.1in}
\paragraph{\bf{Large distance}}

The above argument applies to obtain H\"older estimate for $x,z$ close relative to $|x|$. 
To estimate $ \f{|f(\lam \hat x) - f(\lam \hat z)|}{ |\lam (\hat x - \hat z)|^{1/2}}$ with large $|\hat x - \hat z|, |\hat x| = x_{c2}$, we need to bound $f(\lam \hat z)$ for large $\hat z$.
We focus on $C_x^{1/2}$ estimate of $\na \uu$. Other estimates are similar.

\vs{0.1in}
\paragraph{\bf{$L^{\inf}$ estimate}}
From the above discussions, we seek a bound 
\beq\label{eq:1D_hol_est0}
|f(\lam \hat z)| \leq \lam^{\b} C(\hat z),
\eeq
for $f = \psi \na \uu$ with approximation terms uniformly in $\lam \geq \lam_*, \b \leq \f{1}{2}$ or $\lam \leq \lam_*, \b \geq \f{1}{2}$, and $\hat z$ in a large domain, e.g. $[0, x_{c3}]^2 \bsh [0, x_{c2}]^2$ \eqref{eq:int_1D_dom1}. 

Recall the estimate of the regular parts, e.g. \eqref{eq:int_1D_nsg}, and singular parts \eqref{eq:int_1D_sg} in Section \ref{sec:1D_linf}. We first optimize the bound in \eqref{eq:int_1D_sg} with a fixed $\lam = \lam_*$ over $a = a_1, a_2, .., a_N$. Denote by $a_*$ be the minimizer among $a_i$. Choosing $a = a_*$ in 
\eqref{eq:int_1D_sg}, we can bound $ (\na \uu)_S$ by 
\[
E_1 (  \lam^{-a_{\al}} C_1(\hat x) +  \lam^{ - b_{\al} + 1/2} C_2(\hat x) ),
\]
for another piecewise constants $C_i(\hat x)$. For the nonsingular parts in Section \ref{sec:int_1D_nsg}, for $\lam \in [\lam_l, \lam_u]$, we can bound it by 
\[
E_1 (  C_3(\hat x) \lam^{ - a_{\al}} - C_4(\hat x) )
 \leq  E_1 \lam^{-a_{\al}}( C_3(\hat x ) - C_4(\hat x ) \min( \lam_l^{ a_{\al} }, \lam_u^{  a_{\al}}  ))
  = C_5(\hat x) \lam_l^{ a_{\al}} E_1.
\]
If $ \lam_u = \infty$, we simply drop the term $ -C_4(\hat x)$. Using this method, we can bound $\na \uu$ by 
\beq\label{eq:1D_hol_est1}
 E_1  \cdot \sum_{i\leq N} C_i(\hat x) \lam^{ \b_i}, \quad N = 2, \ \b = (-a_{\al}, - b_{\al} + 1/2  ),
 \quad |\hat x| = x_c , \quad  x = \lam \hat x ,
\eeq
for some piecewise constant $C_1(\hat x) , C_2(\hat x)$. We do not further optimize the estimate \eqref{eq:int_1D_sg} over all small $a$ since the above estimate is good enough for our purpose and is simpler. For $\uu$, we can obtain similar estimates and we only have the term $C_1(\hat x) \lam^{- a_{\al} + 1}$. Next, for the weight 
\beq\label{eq:1D_hol_est2}
\psi_{\lam}(x) \leq \psi_{\inf, u}( x) \lam^{ b_{\al}},
\eeq
we estimate the piecewise bound 
\[
( \psi \na \uu)_{\tau}(\hat z) \leq C_{1, ij} \tau^{b_{\al} - a_{\al}} + C_{2, ij} \tau^{1/2} 
\]
for $\hat z$ in each grid $B_{i j}(h_x) \subset [0, x_{c3}]^2 \bsh [0, x_{c2}]^2$, $\tau \leq \lam_*, \al = 1$ or $\tau \geq \lam_*, \al = n$, which gives the desired estimate \eqref{eq:1D_hol_est0}.  
To apply \eqref{eq:1D_hol_est1}, we need to rescale $\tau \hat z  = \lam \td z $ with $ |\td z|_{\inf} = x_c $. Without loss of generality, we assume $ \hat z_1 \geq \hat z_2$. For $\hat z \in B_{ij}(h_x)$, we have 
\[
 \lam = \f{ |\tau \hat z|_{\inf} }{ x_c } = \f{\tau \hat z_1}{ x_c} 
 \in \tau [ i h_x / x_c, (i+1) h_x / x_c ].
\] 
Applying \eqref{eq:1D_hol_est1} and $\psi_{\tau}(\hat z) \leq \tau^{b_{\al}} \psi_{\inf, u}(\hat z)$ \eqref{eq:1D_hol_est2}, we yield 
\beq\label{eq:1D_hol_est3}
 |(\psi \na \uu)(\tau \hat z)|
 \leq E_1 \sum_{l \leq N}  \tau^{\b_l + b_{\al}}  C_l( \td z ) 
 \max( (i h_x / x_c)^{\b_l}, ( (i+1) h_x/ x_c)^{\b_l}) \psi_{\inf, u}(\hat z),
 \ \td z =  \f{  \hat z \cdot x_c} { |  \hat z|_{\inf}} .
\eeq

We cover the range of $\td z$ by the intervals $\G_{ p, q}(x_c)$ \eqref{eq:int_1D_dom1} and apply the piecewise bound of $C_i( \hat x)$ in $\G_{ p, q}(x_c)$
 to bound $C_i(\td z)$. Similarly, we derive the piecewise bound of $ (\psi \na \uu)_{\tau}( \hat z)$ for $\hat z \in \G(x_{c2})$ \eqref{eq:int_1D_dom1}. Since $\hat z / 2 \in \G( x_{c2}/ 2) = \G( x_{c})$
and $\tau \hat z = (2 \tau) \cdot \hat z / 2$, applying the above estimates, we yield 
\[
 |(\psi \na \uu)(\tau \hat z) |
 \leq E_1 \sum_{l \leq N}  \tau^{\b_l + b_{\al}}  C_l( \td z ) 2^{\b_l} \psi_{\inf, u}(\hat z).
\]

Using the above estimate, we obtain the uniform bound for $z = \lam \hat z$ with $|\hat z| = x_{c2}$ and $\lam \geq \lam_*$
\beq\label{eq:1D_hol_est4}
\max_{ \hat z \in \G_{k, i}(x_{c2}), i}
 \f{  |(\psi \na \uu)(\lam \hat z) |}{ |\lam \hat z|_{\inf}^{1/2}  }
 \leq E_1 \max_{ \hat z \in \G_{k, i}(x_{c2}), i} \sum_{ l \leq N} \lam^{ \b_l + b_{\al} - 1/2} \f{C_l(\hat z)}{ |x_{c2}|^{1/2}} \psi_{\inf, u}(\hat z) 2^{\b_l} 
 \teq E_1 \cdot F_{ha, k}(\lam).
\eeq

From \eqref{eq:1D_hol_est1}, it is easy to check that the power $\lam^{\b_l  + b_{\al} -1/2}$ and $F_i(\lam)$ are increasing in $\lam$ for $\lam \leq \lam_*, \al = 1$ and decreasing in $\lam$ for $\lam \geq \lam_*, \al = n$. The functions $F_{ha, k}(\lam)$ allow us to control $\psi \na \uu(z)$ for $z_2 \geq z_1$ and $z_1 \leq z_2$ uniformly in $\lam$.

\subsubsection{The H\"older estimate with large distance}\label{sec:hol_R2_1D_lg}
Now, we are in a position to perform $C_x^{1/2}$ estimate $\d_1(f, x, z)$ with large $|x-z|$ relative to $|x|$. Denote $f = \psi \na \uu$. Firstly, using \eqref{eq:1D_hol_est3}, we can obtain piecewise bound of $f( \lam \hat z)$ for $\hat z \in [0, x_{c3}]^2 \bsh [0, x_{c2}]^2$. Using the triangle inequality, we can bound 
\[
\d_1(f, \lam \hat x, \lam \hat z) , \quad  |\hat x|_{\inf} = x_{c2}, \ \hat z_2 = \hat x_2, \ 
\hat z_1 - \hat x_1 \leq m_3 h_x,
\]
uniformly in $\lam$, where $m_3$ a parameter given in \eqref{eq:int_1D_dom1}.  Note that we have estimated $\d_1(f, \lam \hat x, \lam \hat z) $ based on H\"older regularity of $\om$ for $|\hat z - \hat x| \leq \mu |x_{c2}|$ at the beginning of Section \ref{sec:hol_R2_1D}. In such a case, we have two estimates and we will optimize them. 
These allow us to estimate $\d_1(f, \hat x, \hat z)$ with $|\hat x - \hat z| \leq m_3 h_x =  x_{c3} - x_{c2}$. 

Finally, we consider the case 
\beq\label{eq:1D_hol_est5}
|\hat z - \hat x| \geq m_3 h_x \geq x_{c2} = |\hat x|_{\inf}.
\eeq

We consider the far-field estimate $\lam \geq \lam_*$. The estimate in the near-field is similar. We have two estimates. In the first estimate, (1) if $ \hat x_1 \geq \hat x_2 $, we get $ \hat z_1 \geq \hat x_1 + m_3 h_x > \hat x_2 = \hat z_2$; (2) if $\hat x_1 \leq \hat x_2$, we get $\hat z_1 \geq \hat x_1 + m_3 h_x$. In case (l), we use \eqref{eq:1D_hol_est4} with $F_{ha, l}$ to bound $f_{\lam}(\hat x)$. In both cases, we use $\max_{s=1,2}( F_{ha, s}(\lam \hat z)$ to bound $f(\lam \hat z)$. Using the triangle inequality and $|\hat z - \hat x| \geq |m_3 h_x|^{1/2}$, we can bound $\d_1(f, x, z)$.

We have an improved estimate for $x \in Q_{i_1 j_1}, z \in Q_{i_2 j_1}$ \eqref{eq:mesh_xy}. Suppose that $|x|_{\inf} \geq |x|_{\inf,l}, |z| \geq |z|_{\inf,l}$. We consider two scenarios (1) $x_1 \geq x_2$, (2) $x_2 \leq x_1$. In case (1), we get $z_1 \geq x_1 \geq x_2 = z_2$. Suppose that 
\beq\label{eq:1D_hol_est52}
 |\hat z- \hat x| =  \hat z_1 - \hat x_1 \geq \tau \hat x_1. 
\eeq
From \eqref{eq:1D_hol_est5}, we have  $\tau \gtr 1$. 
For example, if $x_1 \in [y_i, y_{i+1}], z_1 \in [y_j, y_{j+1}], j \geq i+1$ for the mesh \eqref{eq:mesh_xy}, we get $\f{\hat z_1 - \hat x_1}{\hat x_1} \geq \f{y_j}{ y_{i+1}}  -1$. Using \eqref{eq:1D_hol_est4} and the fact that $F_{ha, \cdot}(\lam)$ is decreasing, we yield 
\[
\bal
\d_1(f, \lam \hat x, \lam \hat z) 
& \leq  \f{ F_{ha, l}( \f{ |x|_{\inf}, l}{ x_{c2} } ) |x|_{\inf}^{1/2} 
+ F_{ha, l}( \f{ |z|_{\inf}, l}{ x_{c2} } ) |x|_{\inf}^{1/2}  }{ |x-z|^{1/2}} \cdot E_1 \\
& \leq \B( F_{ha,1} ( \f{ |x|_{\inf, l} }{x_{c2}} )  ( \f{\hat x_1}{ |\hat z - \hat x| })^{1/2}
+ F_{ha, 1}( \f{ |z|_{\inf, l}}{x_{c2}} )  ( \f{\hat z_1}{ |\hat z - \hat x| })^{1/2} \B) E_1. 
\eal
\]

From \eqref{eq:1D_hol_est52}, we derive
\[
\f{\hat x_1}{ \hat z_1 - \hat x_1} \leq \tau^{-1}, 
\quad \hat z_1 \geq (1 + \tau) \hat x_1,  \quad 
 \hat z_1 - \hat x_1 \geq (1 - \f{1}{\tau+1} ) \hat z_1 = \f{\tau}{\tau + 1} \hat z_1 ,
 \quad \f{ \hat z_1}{ \hat z_1 - \hat x_1} \leq \f{\tau + 1}{\tau},
\]
and the upper bounds of $ \f{\hat x_1}{ \hat z_1 - \hat x_1}, \f{\hat z_1}{ \hat z_1 - \hat x_1}$ are decreasing in $\tau$. Combining the above two estimates, we yield the estimate of $\d_1(f, x, z)$.

In case (2) $x_1 \leq x_2$, since $z_1 \geq x_1 + x_{c3} - x_{c2} \geq 2 x_{c2} > x_2 = z_2$, we get $|z|_{\inf} = z_1$ and 
\beq\label{eq:1D_hol_est6}
\bal
\d_1(f, \lam \hat x, \lam \hat z) 
& \leq E_1  \max_{s=1,2} F_{ha, s} ( \f{ |x|_{\inf, l} }{x_{c2}} )  ( \f{\hat x_2}{ |\hat z_1 - \hat x_1| })^{ \f{1}{2} }
+ E_1 \max_{s=1,2} F_{ha, s}( \f{ |z|_{\inf, l}}{x_{c2}} )  ( \f{\hat z_1}{ |\hat z_1 - \hat x_1| })^{ \f{1}{2} } .\\
\eal
\eeq

We further bound the ratio. Suppose that $ \hat x_2 / \hat x_1 \in  [ y_l / x_u, y_u/ x_l]$. Since $\hat x_2 =  |\hat x| = x_{c2}$ and $m_3 h_x \geq \hat x_2$ \eqref{eq:1D_hol_est5},  we get 
\[
\bal
 \f{\hat z_1}{\hat x_2} \geq \f{ \hat x_1 + m_3 h_x}{  \hat x_2 }
\geq \f{ x_l }{y_u} + \f{m_3 h_x}{ x_{c2}} , \  
\f{\hat z_1 }{\hat x_1} \geq 1 + \f{ m_3 h_x}{ \hat x_1} 
\geq 1 + \max(\f{m_3 h_x}{ x_{c2}}, \f{\hat x_2}{\hat x_1} )
\geq 1 + \max(\f{m_3 h_x}{ x_{c2}}, \f{y_l}{x_u} ).
\eal
\]
Combining the above estimates and using 
\[
 \f{\hat x_2}{| \hat z_1 - \hat x_1|} 
 = \f{ \hat x_2 / \hat z_1}{1 - \hat x_1 / \hat z_1},
 \quad 
  \f{\hat z_1}{ |\hat z_1 - \hat x_1| } = \f{1}{1 - \hat x_1 / \hat z_1}, 
\]
we obtain the bound for $\d_1(f, \lam \hat x, \lam \hat z)$. From the estimates of the above two cases, if $|\hat z_1 - \hat x_1| \gg |\hat x|_{\inf}$, we have 
$|\d_1(f, x, z)| \leq C \max_s F_{ha, s }( \f{ |z|_{\inf,l}} {x_{c2}} ) $ with $C\approx 1$.

\vs{0.1in}
\paragraph{\bf{Covering the large domain}}
Using the above H\"older estimate $\d_1(f, \lam \hat  x, \lam \hat z)$ for $\hat x, \hat z$ uniformly in $\lam$ and the covering argument in Section \ref{sec:hol_R2_bulk}, we can obtain  
piecewise estimate 
\[
 \d_1(f, x, z) , \quad x \in Q_{i_1 j_1}, \ z \in Q_{i_2 j_1}, \  i_1 \leq i_2 \leq i_1 + m,
\]
on the mesh \eqref{eq:mesh_xy} for large $\max(i_1, j_1)$. For $x \in Q_{i_1 j_1}, \ z \in Q_{i_2 j_1}$, we can derive the upper and lower bounds for $x_i, z_i$ and can estimate the ratio 
$\hat z_1 / \hat x_1$ and $ \hat x_2 / \hat x_1$ used in the above estimates \eqref{eq:1D_hol_est52},\eqref{eq:1D_hol_est6} for large $|\hat z_1 - \hat x_1|$.

\vs{0.1in}
\paragraph{\bf{Infinite region}}
To estimate $\d_1(f, x, z)$ with $|x|_{\inf} \geq R$ with sufficiently large $R$, we simply apply the estimate in Section \ref{sec:hol_R2_1D} for small distance $|\hat x - \hat z|$, the estimate above \eqref{eq:1D_hol_est5} for $|\hat x - \hat z| \leq m_3 h_x$, and the first estimate below \eqref{eq:1D_hol_est5} for $|\hat x - \hat z| \geq m_3 h_x$. Since we do not have upper bound for $x, z$, we do not apply the improved estimate in Section \ref{sec:hol_R2_1D_lg}.

The $C_y^{1/2}$ estimates for $\na \uu$ are completely similar.

\vs{0.1in}
\paragraph{\bf{H\"older estimate of $\uu$}}

The H\"older estimates of $\uu$ are completely similar. From \eqref{eq:hol_1D_u},\eqref{int:scal_w}, and Section \eqref{sec:1D_hol}, we can obtain H\"older estimate of $\psi_{u, \lam} \td \uu_{\lam}$ as follows 
\[
\f{ | \psi_{u,\lam} \uu_{\lam}(\hat x) -\psi_{u, \lam} \uu_{\lam}(\hat z) | }{ |\lam \hat x - \lam\hat z|^{1/2}}
\leq \lam^{c_{\al} - a_{\al} + 1/2} C(\hat x, \hat z),
\]
uniformly for $\lam \leq \lam_*, \al = 1$ or $\lam \geq \lam_*, \al = n$ for $
|\hat x|_{\inf} = x_{c2}, |\hat x- \hat z| \leq \mu x_{c2}$, see \eqref{eq:int_1D_hol_dom1}. We lose a power $\lam^{1/2}$ due to  \eqref{int:scal_w}. The power $c_1, c_n$ are given in \eqref{eq:wg_asym}. In particular, we have
\beq\label{eq:1D_hol_est6_pow}
c_1 - a_1 + \f{1}{2} \geq - \f{5}{2} + \f{5}{2} + \f{1}{2} > 0, \quad c_n - a_n + \f{1}{2} 
\leq -\f{7}{6} + \f{1}{2} + \f{1}{2} < 0. 
\eeq
In the near-field $\lam \leq \lam_*$ or in the far-field $\lam \geq \lam_*$, we have a uniform estimate $\lam^{c_{\al} - a_{\al} + 1/2} \leq \lam_*^{c_{\al} - a_{\al} + 1/2} $ . For a large distance, we apply $L^{\inf}$ estimate similar to that in Section \ref{sec:hol_R2_1D_lg}. We can obtain  a piecewise $L^{\inf}$ estimate similar to the estimate above \eqref{eq:1D_hol_est4} as follows 
\[
  |(\psi_{u}  \uu)(\lam  \hat z) |
 \leq E_1  \lam^{  c_{\al - a_{\al} + 1}}  C( \hat z ) , \quad \hat z \in [0, x_{c3}]^2  \bsh [0, x_{c2}]^2 ,
\]
and another estimate similar to \eqref{eq:1D_hol_est4}. For $\uu$, since the kernel is locally integrable, we only have one term related to $|| \om \vp||_{\inf}$ in the above bound. From the power law \eqref{eq:1D_hol_est6_pow} and the scaling relation \eqref{int:scal_w}, using the argument for the H\"older estimate of $\na \uu$, we obtain H\"older estimate $\d_i(f, \lam \hat x, \lam \hat z)$ for large $|\hat x - \hat z|$ uniformly in $\lam$.

\subsection{Estimate using other norms}\label{sec:1D_otherwg}

Similar to Section \ref{sec:vel_linf_otherwg}, we estimate $\uu_A, (\na \uu)_A$ using another 
combiniation of norms, e.g. $|| \om \vp_{1, g}]||_{\inf}, [ \om \psi_1]_{C_{x_i}^{1/2}}$. We 
use the estimates in previous sections of Section \ref{sec:u_1D} to obtain the estimates in the near-field and the far-field and bound different norms using the energy \eqref{eq:linf_cost_grow} or the direct estimate for the error \eqref{eq:linf_cost_err}. In addition to the $L^{\inf}$ estimate mentioned in Section \ref{sec:vel_linf_otherwg}, we use $||\e \vp_{elli} ||_{\inf}, [\e \psi_1]_{C_{x_i}^{1/2}}$ for the H\"older estimate of $\uu_A(\e), (\na \uu)_A(\e)$, where $\e = \bar \om - \bar \phi^{N}, \e = \hat \om - \hat \phi^N $ is the error of solving the Poisson equations. We use these estimates to control the nonlocal error.

\section{Estimate the piecewise bounds of functions}\label{app:linf_est}

In this appendix, we estimate $|| f||_{L^{\inf}(D) }$ in the domain $D = [a, b] \times [c, d]$ given the grid point values of $f$ 
in $D$ and the derivatives bound $|| \pa_x^i \pa_y^j f ||_{L^{\inf}}$. We want to obtain an error term as small as possible without 
evaluating $f$ on too many grid points and its high order derivatives, which are expensive 
for some complicated function $f$, e.g. the residual error $f = (\pa_t -\cL) \wh W$ in Section {\seclinevo} in Part II \cite{ChenHou2023b_supp}. Based on these $L^{\inf}$ estimates, we further develop the H\"older estimate in Appendix {\applinfestPII} in Part II \cite{ChenHou2023b_supp}.
We use these estimates to verify the smallness of the residual error, e.g. $f = (\pa_t -\cL) \wh W$, in suitable energy norm.

We will develop three estimates based on the Newton polynomial, the Lagrangian interpolating polynomial, and the Hermite interpolation. Each method has its own advantages.
For the Newton and the Lagrangian method, to obtain 4-th order error estimates, we only need to evaluate $ 4\times 4$ grid point values of $f$. 

(a) For the Newton method, we have a sharp error bound with a much smaller constant than that of the Lagrangian method. 

(b) For the Lagrangian method, it is easier and more efficient to estimate the Lagrangian interpolating polynomials for grid points $(x_i, y_j)$ in a general position.

In some situation, we need to estimate both $f$ and $\na f$. 
We use grid point values $ f(x_i, y_j)$ and $\na f(x_i, y_j)$ to build 4-th and 5-th order interpolating polynomials based on the Hermite interpolation. The 4-th order error estimate is as sharp as the Newton method. Moreover, in the 4-th order Newton or Lagrangian interpolation, we need $f \in C^4[x_0, x_3]$. 
When $f$ is only piecewise smooth in $[x_i, x_{i+1}], i\leq 3$, we cannot use these two methods. Instead, we evaluate $f(x_i), f^{\pr}(x_i)$ and construct the Hermite interpolation in each interval $[x_i, x_{i+1}]$. One disadvantage is that the estimate of the interpolating polynomials is more complicated and takes longer time in practice.

We do not pursue higher order error estimates since most of these estimates are applied to estimate the residual errors, e.g. $ (\pa_t - \cL) \wh W$ in Section {\seclinevo} \cite{ChenHou2023b_supp}, which is only piecewise smooth, and we do not use very small $h$ in the whole computational domain.

\subsection{Estimates based on the Newton polynomials}\label{sec:newton}

Given $x_0 ,  x_1 , x_2 , .., x_k$, we first define the divided differences recursively 
\[
f[x] = f(x), \quad 
f[x, y] = \f{ f(y) - f(x)}{y-x},  \quad  f[ x_i, x_1,.., x_{j+1}] =  \f{  f[x_{i+1}, x_{i+2}, .., x_{j+1}]   - f( x_i, x_{i+1}, .., x_{j})  }{ x_{j+1} - x_i } .
\] 

\subsubsection{The Newton polynomials in 1D}\label{sec:newton_1D}
We first discuss how to bound $f(x)$ in 1D. We consider the domain $[a, b]$ and denote
\[
z_0 =a,  \quad h = (b-a) / 3,  \quad z_i = z_0 + ih .
\]
Denote 
\beq\label{eq:newton_dif}
\bal
&D_{1, i} f = f(z_{i+1}) - f(z_i),  \ i= 0,1,2 ,
\quad D_{2, i} f = D_{1,i+1} f- D_{1, i} f, \ i=0,1 , \\
&D_3 f = D_{2,1} f - D_{2,0} f =  f(z_3) - 3 f(z_2) + 3 f(z_1) - f(z_0).
\eal
\eeq

Let  $ \{ x_i \}_{i=0}^3 $ be a permutation of $\{ z_i\}_{i=0}^3$. We construct the Newton polynomial 
\beq\label{eq:newton_P1}
\bal
P(x) &= (  f[x_0] + f[x_0, x_1] (x-x_0) ) + 
f[x_0, x_1, x_2] (x - x_0) (x- x_1)  \\
& \quad + f[x_0,x_1, x_2, x_3] (x - x_0) (x- x_1) (x-x_2) 
\teq l(x) + q(x) +c(x),
\eal
\eeq
where $l(x), q(x), c(x)$ denote the linear, quadratic, and the cubic parts, respectively. We remark that the above Newton polynomial agrees with the Lagrangian polynomial interpolating $(z_i, f(z_i))$.

By standard error analysis of the Newton interpolation, the error part $R(x) $ can be bounded as follows 
\beq\label{eq:newton_est1}
| f(x) - P(x)| \leq \f{1}{24}  || \pa_x^4 f||_{L^{\inf}[a, b]} \max_{x \in [a, b]} | \Pi_{ 0\leq i \leq 3} (x- x_i)  |
= \f{1}{24} || \pa_x^4 f ||_{L^{\inf}[a, b] } \f{(b-a)^4}{81}.
\eeq
To obtain the last equality, using the definition of $x_i, z_i$, we write $z = a + t h, t \in [0,3]$ and get
\beq\label{eq:Lag4_err}
\bal
\max_{x \in [a, b]} | \Pi_{ 0\leq i \leq 3} (x- x_i)  |
& = \max_{z \in [a, b]} | \Pi_{ 0\leq i \leq 3} (z- z_i)  |
 = \max_{ t \in [0, 3]} h^4 | \Pi_{ 0\leq i \leq 3} ( t -i)  | \leq h^4,
\eal
\eeq
where we have used \eqref{eq:pol_4th_1} in Lemma \ref{lem:pol_est} in the last inequality.

To bound $f(x)$, given the derivative bound of $f$ and the above estimate, we only need to control $P(x)$. We choose different permutation $\{ x_i\}_{i=0}^3$ of $\{ z_i\}_{i=0}^3$ for $z$ in different part of $[a, b]$: 
\[
\bal
& (a) \, (x_0,x_1,x_2, x_3) = (z_0, z_1, z_2, z_3) , \  \mbox{for} \,  z \in [z_0,z_1],  \quad (b) \, (x_0, x_1, x_2, x_3) = (z_2, z_1, z_0, z_3),  \ \mbox{for} \, z \in [z_1, z_2], \\
  & (c) \, (x_0, x_1, x_2, x_3) = (z_3, z_2, z_1, z_0), \ \mbox{for} \, z \in [z_2, z_3]. 
  \eal
\]
In all three cases, we have $|x_1 - x_0| = h$. Let $I_z$ be the interval with endpoints $x_0, x_1$. We have $z \in I_z$. Since $l(x)$ in \eqref{eq:newton_P1} is linear with $l(x_i) = f(x_i)$ and $|(x-x_0)(x-x_1)| 
\leq \f{(x_1 - x_0)^2}{4}$, we get 
\[
 |l(z)| \leq \max( |f(x_0)| , |f(x_1)| ), \quad \max_{ z \in I_z} |q(z) | 
 \leq  |  f[x_0, x_1, x_2] | \f{ (x_1 - x_0)^2}{4} = |  f[x_0, x_1, x_2] | \f{h^2}{4}.
\]

Since $x_0, x_1, x_2$ are three consecutive points with distance $h$ and $I \teq [\min(x_0, x_1, x_2), \max( x_0, x_1, x_2)]$ covers $z$, we have 
\beq\label{eq:newton_cubic_ineq}
\max_{z \in I} | (z-x_0)(z-x_1) (z -x_2) | = \max_{ z \in [0, 2h]} | z (z-h)(z-2h)|
\leq \f{2}{3 \sqrt{3}} h^3,
\eeq
where we have used \eqref{eq:pol_3rd_1} in Lemma \ref{lem:pol_est} in the last inequality.

Next, we use \eqref{eq:newton_dif} to simplify  $f[x_i, x_{i+1}, .., x_j]$. For each case, a direct calculation yields 
\[
\bal
z \in [z_0, z_1]:  |f[z_0, z_1, z_2]  | &= \B| \f{ f[z_2, z_1] - f[z_1, z_0]  }{z_2 - z_0} \B|
= \B| \f{1}{2h^2} ( D_{1, 1} f - D_{1,0} f ) \B| = \f{1}{2h^2} | D_{2,0} f |, \\
z \in [z_1, z_2] : | f(z_2, z_1, z_0)| & =\B| \f{ f[z_0, z_1] - f[z_1, z_2]  }{z_0 - z_2} \B|
= \B| \f{1}{2h^2} ( D_{1, 1} f - D_{1,0} f ) \B| = \f{1}{2h^2} | D_{2,0} f |. \\
\eal
\]

Similarly, for $ z \in [z_2, z_3]$, we get 
\[
| f(z_3, z_2, z_1)| = \f{1}{2h^2} |D_{2, 1} f |.
\]
A direct calculation yields $|f[x_0, x_1, x_2, x_3]| =\f{1}{6 h^3} |D_3 f|$ for any permutation $\{ x_i \}$ of $\{ z_i \}$. Thus, for $c(x)$ \eqref{eq:newton_P1}, using \eqref{eq:newton_cubic_ineq} and the above estimate, we get 
\[
\max_{z \in [a, b]} | c(x)| \leq \f{2 }{3 \sqrt{3}}  h^3|f[x_0, x_1, x_2, x_3]| 
= \f{2}{3 \sqrt{3} } h^3 \cdot \f{1}{6 h^3} | D_3 f| 
= \f{1}{9 \sqrt{3}} |D_3 f|.
\]

Combining the above estimates, we obtain 
\beq\label{eq:newton_est2}
|P(x)| \leq \max\B(  \max_{i=0,1,2}  |f(z_i)|  + c_1 | D_{2,0}f | , 
 \max_{i=1,2,3} |f(z_i)|  + c_1 | D_{2,1}f |  \B) + c_2 |D_3 f|,  \ c_1 = \f{1}{8}, \  c_2 = \f{1}{9 \sqrt{3}}. 
\eeq

\subsubsection{A quadratic interpolation}\label{sec:newton_cubic}

We also need a cubic interpolation in the Hermite interpolation in Section \ref{sec:Herm_2D}. Given $x_0 < x_1 < x_2$ with $x_2 - x_1 = x_1 -x_0 = h$, we define
\beq\label{eq:newton_cubic_dec}
N_2(f, x_0 ,x_1, x_2)(x) \teq f(x_0) + f[x_0, x_1] (x-x_0) + f[x_0, x_1, x_2] ( x - x_0)( x - x_1),
\eeq
and construct $P(x) = N_2(f, x_0, x_1, x_2)(x)$. 
Here, the $x_i$ are \emph{different from those} in Section \ref{sec:newton_1D}.

We have an error estimate similar to \eqref{eq:newton_est1} for $x \in [x_0, x_2], h = x_1 - x_0$
\beq\label{eq:newton_cubic1}
 | P(x) - f(x)| \leq \f{ || \pa_x^3 f||_{ L^{\inf} [x_0, x_2]}}{6} 
\max_{ x \in[x_0, x_2]}| \Pi_{i =0,1,2}  (x-x_i)|
\leq   \f{ || \pa_x^3 f||_{ L^{\inf} [x_0, x_2]}}{6}  \f{2 h^3}{3 \sqrt{3}}
= \f{ || \pa_x^3 f||_{ L^{\inf} [x_0, x_2]}   h^3}{9 \sqrt{3}},
\eeq
where we have used \eqref{eq:newton_cubic_ineq} in the last inequality. 

Using the same estimates in Section \ref{sec:newton_1D} for the linear part and quadratic part, we obtain 
\beq\label{eq:newton_cubic2}
\bal
x  \in[x_0, x_1] : \ |P(x) | \leq \max( |f(x_0)|, |f(x_1)| ) + \f{1}{8} | f(x_0 ) - 2 f(x_1) + f(x_2)| ,\\
x  \in[x_1, x_2] : \ |P(x) | \leq \max( |f(x_1)|, |f(x_2)| ) + \f{1}{8} | f(x_0 ) - 2 f(x_1) + f(x_2)|.
\eal
\eeq
Note that using the notation \eqref{eq:newton_dif}, we have $|D_{2, 0} f| =  | f(x_0 ) - 2 f(x_1) + f(x_2)| $.

\subsubsection{Generalization to 2D}

Denote 
\[
D = [a, b] \times [c, d], \quad x_i = a + ih_1, \quad y_j = c + j h_2, \quad h_1 = (b-a) / 3, \quad h_2 = (d-c) / 3 .
\]

Suppose that $f(x_i, y_j) $ and $|| \pa_x^k \pa^l f||_{L^{\inf}} , k+ l \leq 4$ are given. Firstly, we treat $y$ as a parameter and interpolate $f(x, y)$ in $x$. Denote 
\[
\bal
&D_{i, 1} f( y) = f(x_{i+1}, y) - f(x_i, y), \ 0 \leq i \leq 2, \quad 
D_{i, 2} f(y) = D_{i+1, 1} f(y) - D_{i, 1} f(y), \ 0 \leq i \leq 1,  \\
& D_3 f(y) = D_{2,1} f(y) - D_{2, 0} f(y).
\eal
\]

Applying \eqref{eq:newton_est1}, \eqref{eq:newton_est2}, we get 
\[
\bal
\max_{x \in [a, b]} | f(x, y)|
&\leq 
\max\B( \max_{i=0,1,2} | f(x_i, y)| + c_1 | D_{2,0} f(y) |, 
\max_{i=1,2, 3} | f(x_i, y)| + c_1 | D_{2,1} f(y) | \B)  \\
& \quad  + c_2 |D_3 f(y)| + \f{1}{24} h_1^4 || \pa_x^4 f||_{L^{\inf}(D)}.
\eal
\]

Note that $f(x_i, y), D_{i, j} f(y)$ are 1D functions in $y$ and their grid point values on $y_j$ can be obtained from $f(x_i, y_j)$. We further apply  \eqref{eq:newton_est1},\eqref{eq:newton_est2} to estimating $|| g||_{L^{\inf}[c, d]}$ with $ g = f(x_i, y), D_{i, j} f(y)$. Maximizing the above estimate over $y$ yields the bound for $f$.

\subsection{Estimates based on the Lagrangian interpolation}\label{sec:lag}

If the grid points $x_i$ are not equi-spaced, the estimates of the Newton polynomials can be more complicated. We develop another estimate based on the Lagrangian interpolating polynomials. Although these two interpolating polynomials on $(x_i, f(x_i))$
are equivalent, the Lagrangian formulation is easier to estimate.

Firstly, let $p_i(x), q_j(y)$ be the Lagrange interpolating polynomials 
associated to the points $x_1 < .. < x_k \in [a, b],  y_1 < y_2 <... < y_k \in [c, d]$
\beq\label{eq:pol_lag}
p_i(x) = \Pi_{j \neq i} \f{ x- x_j}{x_i - x_j} , \quad q_i(y) = \Pi_{j \neq i} \f{ y- y_j}{y_i - y_j} .
\eeq
For any $(x, y) \in D$, we consider the following decomposition by first interpolating $f$ in $x$ and then in $y$
\beq\label{eq:Lag_pol_dec}
\bal
f(x, y)  &= \sum_{i=1}^k f(x_i, y) p_i(x) + R_1(x, y)
= \sum_{ i=1}^k ( \sum_{j=1}^k  f(x_i, y_j) q_j(y) + R_2(x_i, y ) )  p_i(x) + R_1(x, y) \\
& = \sum_{i, j=1}^k p_i(x) q_j(y) f(x_i, y_j) + \B( \sum_i R_2(x_i, y) p_i(x) + R_1(x, y) \B) \teq I + II.
\eal 
\eeq

By standard error analysis of the Lagrange interpolation, the error part $R_1(x, y), R_2(x, y) $ can be bounded as follows 
\beq\label{eq:Lag_pol_err}
\bal
|  R_1(x, y) | &\leq \f{1}{k !} || \pa_x^k f(x, y)||_{L^{\inf} (D)} \max_{ x \in [a, b] } \B|\Pi_{i=1}^k ( x  - x_i) \B| , \\
|  R_2(x_i, y) | & \leq \f{1}{k !} || \pa_y^k f(x, y)||_{L^{\inf} (D)} \max_{ y \in [c, d]} \B|\Pi_{i=1}^k (y-y_i) \B| 
. \\
\eal
\eeq

Denote 
\beq\label{eq:Lag_Clag1}
C_1 = \max_{x \in [a, b]} \sum_{i=1}^k |p_i(x)| , \quad a_{ij} = f(x_i, y_j).
\eeq
Note that the value $C_1$ only depends on the ratio $ \f{x_{i+1} - x_i}{b-a},i=1, .., k $, since from \eqref{eq:pol_lag}, we have
\[
p_i(x) = \Pi_{j \neq i} \f{ x- x_j}{x_i - x_j} = \Pi_{ j \neq i} \f{ t - t_j}{t_i - t_j}, \quad t = \f{x - a}{b-a}, \quad t_j = \f{ x_j - a}{b-a}.
\] 
We will choose the grid points with $\f{x_{i+1} - x_i}{b-a} = \f{y_{i+1} - y_i}{d-c}$ so that we also have $C_1 = \max_{ y \in [c, d]} \sum_{i=1}^k |q_j(y)|$. See \eqref{eq:pol_lag3}. We have the following trivial estimate for any $c_j$
\beq\label{eq:Lag_pol_pf1}
|\sum_i p_i(x) c_i| \leq \sum_i |p_i(x)| \max |c_i| \leq C_1 \max |c_i|,
\quad |\sum_j q_j(y) c_j| \leq C_1 \max |c_i|.
\eeq

Next, we estimate $I$. Since $\sum_i p_i(x) = \sum_j q_j(y ) = 1$, we expect that $I \approx f(x_i, y_j) + O(\max(h_1, h_2))$, where $h_1 = b- a, h_2 = d-c$. Thus, for some $m$ to be chosen, we further decompose it into the mean and the variation and apply \eqref{eq:Lag_pol_pf1} to obtain
\[
\bal
|I| &=| m +  \sum_{i, j=1}^k p_i(x) q_j(y) ( a_{ij}  - m) |
\leq |m| + \max_{i,j} | a_{ij} - m|  \sum_{i} |p_i(x)| \sum_j |q_j(y)| \\
& = |m| + C_1^2 \max_{i,j} | a_{ij} - m|.
\eal
\]

We use the following trivial equality for $b_1 , b_2, .., b_n$
\beq\label{eq:Lag_pol_pf2}
\max_{ i} |b_i - b| =  \f{1}{2} ( \max b_i - \min b_i ), \quad b = \f{1}{2} ( \max_i b_i + \min_i b_i),
\eeq
which can be proved by ordering $b_i$. Thus, we optimize the estimate of $I$ by choosing  $ m = \f{1}{2} ( \max_{i,j} a_{i j} + \min_{i, j} a_{ij} )$.

We can obtain a sharper estimate as follows 
\beq\label{eq:Lag_pol_pf3}
\bal
|I| &=\B|  \sum_j q_j(y) \B( \bar a_j + \sum_i p_i(x) ( a_{ij} - \bar a_j) \B) \B|
= \B| \sum_j q_j(y) ( \bar a_j + S_j(x ) ) \B| \\
&\leq | \bar a | + |\sum_j  q_j(y) (\bar a_j - \bar a)|
+ |\sum_{j} q_j(y) S_j(x) | , \qquad  S_j(x) = \sum_i p_i(x) ( a_{ij} - \bar a_j).
\eal
\eeq
For a fixed $j$, we choose 
\[
\bar a_j = \f{\max_i a_{ij} + \min_i a_{ij}}{2}, \quad \bar a = \f{\max_j \bar a_j + \min_j \bar a_j}{2}.
\]
Applying \eqref{eq:Lag_pol_pf1}, \eqref{eq:Lag_pol_pf2}, and the definition of $S_j(x)$ in \eqref{eq:Lag_pol_pf3}, we yield
\[
\bal
&|\sum_j q_j(y)S_j(x)| \leq  C_1 \cdot \max_j |S_j(x)| , 
\quad |S_j(x)| 
\leq C_1 \max_i | a_{ij} - \bar a_j|
= \f{C_1 }{2} |\max_i a_{ij} - \min_i a_{ij}|. 
\eal
\]
Similarly, we have
\[
 \B| \sum_j q_j(x) (\bar a_j - \bar a) \B| \leq 
C_1 \max_j |\bar a_j - \bar a| = \f{C_1}{2}  (\max_j \bar a_j - \min_j \bar a_j)  .
\]

Combining two parts, we yield an improved estimate for $|I|$
\[
|I| \leq \f{1}{2} |\max_j \bar a_j +\min_j \bar a_j| 
+ \f{C_1}{2}  (\max_j \bar a_j - \min_j \bar a_j )
+ \f{C_1^2}{2} \max_j  | \max_i a_{ij} - \min_i a_{ij} | .
\]

The above estimate is better if $a_{ij} = f(x_i, y_j)$ is smooth in $x$. Similarly, we can first sum over $p_i(x)$ and then sum over $q_j(y)$ in \eqref{eq:Lag_pol_pf3} to obtain another improved estimate.

We apply the above method to the fourth and third order estimate of $f$ on $[a, b] \times [c, d]$. We choose 
\[
 (x_1, x_2, x_3, x_4) = (a, a+\f{1}{3} h_1, a+\f{2}{3} h_1, b) , \quad 
(y_1, y_2, y_3, y_4) = (c, c+\f{1}{3} h_2, c+\f{2}{3} h_2, d) , \\
\]
for the fourth order estimate, and 
\beq\label{eq:pol_lag3}
\bal
 (x_1, x_2, x_3) = (a+ \f{h_1}{32} , a+\f{h_1}{2} , a+\f{31}{32} h_1) , \quad
(y_1, y_2, y_3) = (c+\f{ h_2}{32} , c+\f{ h_2}{2} , c+ \f{31}{32} h_2) ,
\eal
\eeq
for the third order estimate, where $ h_1 = b-a$, $h_2  = d-c$. To estimate the constant $C_1$ 
\eqref{eq:Lag_Clag1} in each case, since the interpolation polynomial $p_i(x)$  \eqref{eq:pol_lag} has degree at most $3$ and $C_1$ only depends on the ratio $ \f{x_{i+1} - x_i}{b-a}$, we can first assume $I = [a, b] = [0, 3]$ or $[0, 1]$ and partition $I$ into $N$ small sub-intervals 
$I_j$ with $|I_j| = |I| / N$. Then we estimate the piecewise bound of $p_i(x)$ in $I_j$ using \eqref{eq:newton_est2} and then $\sum |p_j|$ using the triangle inequality. We get
\[
C_1^{(4)} \leq  1.635, \quad C_2^{(3)}  \leq 1.276
\]
for the fourth order and third order case, respectively. To estimate the 3rd order error term in \eqref{eq:Lag_pol_err}, we use \eqref{eq:newton_est2} again to estimate the third order polynomial $p(t)$ below and yield 
\[
\bal
& \max_{x \in [a, b]}  \Pi_{i=1}^3 |x - x_i| 
= h_1^3 \max_{t \in [0, 1]} |p(t)| =C_3 h_1^3 ,
\  \max_{x \in [a, b]} \Pi_{i=1}^3 |y - y_i| = C_3 h_2^3, \ 
  p(t) = (t - \f{1}{32})(t-\f{1}{2})(t - \f{31}{32}),  \\
& C_3^{(3)} \leq 0.0397.
 \eal
\]
For the fourth order error term, since $\{x_i\}$ are equi-spacing, following \eqref{eq:Lag4_err}, we get
\[
\max_{x \in [a, b]}  \prod_{i=1}^4 |x - x_i| = (\f{h_1}{3})^4 \max_{t \in [0, 3]} | t (t-1)(t-2)(t-3)| \leq (\f{h_1}{3})^4.
\]
Using \eqref{eq:Lag_pol_err}, \eqref{eq:Lag_pol_pf1}, and the above estimate, for the 3rd order estimate, we get
\[
|II| \leq C_2^{(3)} \max_i || R_2(x_i, y)||_{L^{\inf}} + || R_1||_{L^{\inf}}
\leq \f{1}{6} \cdot C_3^{(3)} (  C_2^{(3)} || \pa_y^3 f ||_{L^{\inf}(D)} h_2^3 + 
|| \pa_x^3 f||_{L^{\inf}(D)}  h_1^3) .
\]
We can also first interpolate $f$ in $y$ and then in $x$ \eqref{eq:Lag_pol_dec} to obtain another form of $II$. Similar estimates yield 
\[
|II| \leq \f{1}{6} \cdot C_3^{(3)} (  C_2^{(3)} || \pa_x^3 f ||_{L^{\inf}(D)} h_1^3 + 
|| \pa_y^3 f||_{L^{\inf}(D)} h_2^3) .
\]
We minimize the above two estimates to bound $II$. Similar arguments apply to the fourth order estimate.

\subsection{Hermite interpolation in 1D}\label{sec:Herm_1D}

We first discuss the Hermite interpolation in 1D, and then generalize it to 2D. Consider $x_0 <  x_1 < x_2$ with $x_2 - x_1 = x_1 - x_0$. Denote by $p_i,  q_i$ the cubic polynomials
such that 
\[
\bal
& p_i(x_i ) = 1, \quad  p_i( x_{1-i}) = 0, \ i=0, 1, \quad  p^{\prime}_i(x_j) = 0, \ j = 0 , 1 , \\
& q_i^{\pr}(x_i) = 1, \quad q^{\pr}_i(x_{1-i}) = 0, \ i=0, 1, \quad  q_i(x_j ) = 0, \ j= 0, 1,  \\
\eal
\]
and 
\beq\label{eq:Herm_pol_pf1}
l_i = f(x_i),  \quad  m_i = h f^{\prime}(x_i), \quad h = x_1 - x_0.
\eeq

We consider the $4-th$ and $5-th$ order Hermite interpolations for $f$ 
\beq\label{eq:Herm_int1}
\bal
 H_4(f, x_0, x_1)(x) &= \sum_{i=0, 1} ( f( x_i) p_i(x) +  f^{\pr}(x_i) q_i(x) ),  \\
  H_5(f, x_0, x_1, x_2)( x) &= H_4(x) + ( f(x_2) - H_4(x_2) ) \f{ (x- x_0)^2(x-x_1)^2 }{ (x_2 - x_0)^2( x_2 - x_1)^2 }.
  \eal
\eeq
For simplicity, we drop the dependence of $f, x_i$. Note that the coefficients of the polynomials $p_i, q_i$ depend on $x_0, x_1, x_2$ only. It is easy to see that 
\[
H_4(x_i) = H_5(x_i) = f(x_i), \quad  H_4^{\pr}(x_i) = H_5^{\pr}(x_i) = f^{\pr}(x_i), \ i=0,1,  \quad H_5(x_2) = f(x_2).
\]

\subsubsection{Estimates of the interpolation error}\label{sec:Herm_1D_err}
For $[x_0, x_2]$, we have the standard error estimate 
\beq\label{eq:Herm_est0}
|f(x) - H_4(x)| \leq \f{1}{24} || \pa_x^4 f ||_{L^{\inf}([x_0, x_2])}  (x-x_0)^2(x-x_1)^2, \quad x \in [x_0, x_2] .
\eeq

For $x \in [x_0, x_1]$, we have the following error estimates with further simplifications 
\beq\label{eq:Herm_est1}
\bal
| f(x) - H_4(x) | & \leq \f{1}{24} || \pa_x^4 f ||_{L^{\inf}([x_0, x_1])}  (x-x_0)^2(x-x_1)^2
\leq \f{h^4}{384}  || \pa_x^4 f ||_{L^{\inf}([x_0, x_1]) },  \\ 
| f(x) - H_5(x)| & \leq 
\f{ || \pa_x^5 f ||_{ L^{\inf}( [x_0, x_2] ) } }{120}   | (x-x_0)^2(x-x_1)^2  (x - x_2) |
\leq \f{h^5}{1200} || \pa_x^5 f ||_{L^{\inf}([x_0, x_2])} ,
\eal
\eeq
where we have used \eqref{eq:pol_2nd}, \eqref{eq:pol_5th} in Lemma \ref{lem:pol_est} in the last inequality. 
The proof of the first inequality is standard and easier. We consider the second estimate. For any $t \in [x_0, x_2]$, denote 
\[
R_t(x) = f(x) - H_5(x) - ( f(t) - H_5(t)) \f{ (x-x_0)^2(x-x_1)^2( x-x_2)}{ (t-x_0)^2(t-x_1)^2( t-x_2 )  }.
\]

Clearly, we have $R_t(x_i) = 0, i=0,1,2, R_t^{\pr}(x_i) = 0, i =0 ,1, R_t(t) = 0$. Thus, $R_t$ has $6$ zeros. 
For $f \in C^{4,1}$, applying the Rolle's theorem repeatedly up to $\pa_x^4 f$, we yield $\xi_1 \neq \xi_2 \in (x_0, x_2)$ with
\[
 0 = \pa_x^4 f(\xi_i) - C_1 - \f{(f(t) - H_5(t)) (120 \xi_i - C_2)}{ (t-x_0)^2(t-x_1)^2 (t -x_2) },
\]
where we have used $\pa_x^4 H_5(x) = C_1, \pa_x^4 (x-x_0)^2(x - x_1)^2(x-x_2) = 120 x - C_2$. 
Rewritting the above identities and computing the difference, we obtain 
\[
|f(t) - H_5(t) |= \B|\f{  \pa_x^4 f( \xi_2) - \pa_x^4 f(\xi_1) }{ 120 ( \xi_2 - \xi_1 ) } (t-x_0)^2(t-x_1)^2 (t -x_2) \B|
\leq \f{|| \pa_x^5 f||_{L^{\inf} [x_0, x_2]}  }{120} 
\B|  (t-x_0)^2(t-x_1)^2 (t -x_2) \B|.
\]
Since $t$ is arbitrary, this proves the second estimate in \eqref{eq:Herm_est1}. The first estimate can be proved similarly. Next, we estimate the interpolating polynomials $H_4, H_5$.

We should compare the first estimate \eqref{eq:Herm_est1} with \eqref{eq:newton_est1}. In \eqref{eq:newton_est1}, $x_1 - x_0 = \f{b-a}{3} = h$ and the upper bound is $\f{1}{24} || \pa_x^4 f||_{\inf} h^4$. In \eqref{eq:Herm_est1}, for $x \in [x_0, x_1]$, using $|(x - x_0)(x  - x_1)| \leq \f{ (x_1 - x_0)^2}{4}$, we obtain an extra small factor $\f{1}{16}$. This is one of the main advantages of the Hermite interpolation.

\subsubsection{Estimate $H_4, H_5$ }

We consider $x \in [x_0, x_1]$. 
Recall $l_i, m_i$ from \eqref{eq:Herm_pol_pf1}. We have
\beq\label{eq:Herm_int2}
\bal
G_4(t) & \teq H_4( x_0 + th) = H_4(x), \quad t = \f{x - x_0}{h} , \\
 G_4(t) &= \B( l_0 (1- t) + l_1 t \B) +\B( t^2( t-1)( m_1 - (l_1 - l_0 ) )
+ (t-1)^2 t ( m_0 - (l_1 - l_0)) \B)  \\
& \teq I(t) +II(t).
\eal
\eeq

To show that $G_4$ defined via the first identity has the second expression, we only need to verify that the expression satisfies $G_4(i) = l_i= f(x_i), \pa_t G_4(i) = m_i  = h f^{\pr}(x_i), i=0,1$, which is obvious. Then both expressions are Hermite polynomials interpolating $(x_i, f(x_i)), (x_i, f^{\pr}(x_i))$, and thus they must be the same. To estimate $H_4(x)$ on $[x_0, x_1]$, we only need to estimate $ G_4(t)$ on $[0, 1]$. The estimate of the linear part is trivial 
\[
|I(t)| = |l_0 (1- t) + l_1 t | \leq  \max( |l_0|, |l_1|). 
\]
The second part $II$ is treated as error and we want to obtain a sharp constant. Denote 
\beq\label{eq:Herm_pol_pf2}
 M_1 = \max( | m_1 - m_0|, |m_1 - (l_1 - l_0)|, |m_0 - (l_1 - l_0)| ).
 \eeq
 For $ 0 \leq t \leq \f{1}{2}$, using  $t(t-1)^2 \leq \f{4}{27}$ \eqref{eq:pol_3rd_2}, we have
\[
 \bal
II(t)  &= t(t-1)( t(m_1 - (l_1-l_0)) + (t-1) (m_0 - (l_1 - l_0) ) \\
 &= t(t-1)(  t( m_1 - m_0) + (2t-1)( m_0 - (l_1 -l_0))    ) , \\
 |II(t) & \leq M_1 | t(t-1)| ( t + |1-2t|) 
 = M_1 t(1-t)^2 \leq \f{4}{27} M_1.
 \eal
\]
 Similarly, for $ t \in [1/2, 1]$, writing $m_0 - (l_1 -l_0) = m_0 - m_1 + ( m_1 - (l_1-l_0) )$,
we get 
 \[
 \bal
II(t) &= t(t-1)(  (t-1)(m_0 - m_1) + (2t-1)( m_1 - (l_1 -l_0))  ), \\
|II(t)| & \leq |t(t-1)| ( |t-1| + |2t-1| ) M_1
 = t(1-t) (1-t + 2t-1) M_1 = t^2 (1- t) M_1 \leq \f{4}{27} M_1.
\eal
 \]
To obtain the last inequality, we apply \eqref{eq:pol_3rd_2} with $s = 1- t$. Therefore, we prove 
\beq\label{eq:Herm_est2}
    |H_4(x)  | = |G_4(t)|  \leq \max( |l_0|, |l_1|)
  + \f{4}{27} \max( | m_1 - m_0|, |m_1 - (l_1 - l_0)|, |m_0 - (l_1 - l_0)| ).
\eeq

For $H_5$, we estimate the extra term in \eqref{eq:Herm_int1}. Since $x_2 - x_1 = x_1 - x_0 = h$, we have 
\[
H_4(x_2) = G_4(2) = -l_0 + 2l_1+ 4(m_1 - (l_1-l_0)) + 2 (m_0 - (l_1 - l_0))
=
l_0 + 2 m_0 + 4 ( m_1 - (l_1 - l_0)).
\]

Since $ x = x_0 + t h \in [x_0, x_1]$, we have $t \in [0, 1]$, $|t(1-t) |\leq \f{1}{4}$,  
\beq\label{eq:Herm_5th_H5err}
 \f{ (x - x_0)^2(x-x_1)^2}{ (x_2 -x_0)^2( x_1 - x_0)^2} 
 = \f{ t^2(t-1)^2}{ 4}  \leq \f{1}{64} .
 \eeq

We obtain 
\beq\label{eq:Herm_est3}
\max_{x\in[ x_0, x_1]} | H_5| \leq \max_{x \in [x_0, x_1]} | H_4(x)| + \f{1}{64} | f(x_2) - H_4(x_2)|.
\eeq

\subsubsection{Estimate derivatives in 1D}\label{sec:Herm_C1_1D}

In this subsection, we discuss how to estimate $\pa f(x)$ with fourth order error term using the Hermite interpolation $\pa H_5(x)$. We consider $\pa_x$ without loss of generality. Firstly, since $f(x) - H_5(x)$ has five zeros: two zeros at $x_0$, two zeros at $x_1$, and one zero at $x_2$, we know that $\pa_x ( f(x) - H_5(x))$ has four zeros: $x_0 < \xi< x_1 < \eta $. Using the Rolle's theorem and an argument similar to that in Section \ref{sec:Herm_1D_err}, we get 
\[
|\pa_x( f(x) - H_5(x))| 
\leq \f{1}{24} || \pa_x^5 f||_{L^{\inf}[x_0, x_2]} | (x-x_0)(x-x_1) (x- \xi)(x- \eta ) |.
\]
Next, for $x \in [x_0, x_1]$, we simplify the upper bound. Clearly, we have $|x-\xi| \leq \max(|x-x_0|, |x_1-x|), |x-\eta| = \eta - x \leq x_2 - x$. We yield 
\[
\bal
p(x) &\teq | (x-x_0)(x-x_1) (x- \xi)(x- \eta )
\leq |(x-x_0) (x-x_1) (x-x_2)  | \max( |x-x_0|, |x-x_1|)  \\
& =  h^4 |t (1-t) (2-t)| \max( |1-t|, t) \teq h^4 q(t), \quad t = (x-x_0) h^{-1}.
 \eal
\]

If $ t \leq \f{1}{2}$, we denote $s = (1-t)^2 \in [0, 1]$ and get 
\[
q(t) = t(1-t)^2(2-t) = (2t - t^2) s  = (1-s) s\leq \f{1}{4}.
\]
If $t  \in [1/2, 1]$, we get 
\[
q(t) = t(1-t)  \cdot t(2-t) \leq \f{1}{4} \cdot 1 = \f{1}{4}.
\]

Thus, for $x \in [x_0, x_1]$, we prove the error estimate 
\beq\label{eq:Herm_C1_est1}
|\pa_x ( f(x) - H_5(x))| 
\leq \f{h^4}{96} || \pa_x^5 f||_{L^{\inf}[x_0, x_2]} .
\eeq

Next, we estimate $\pa_x H_5$. Recall $t = \f{x-x_0}{h}$. Using \eqref{eq:Herm_int1} and the chain rule, we get 
\beq\label{eq:Herm_int3}
\bal
& H_5(x)   = H_4(x) + (f(x_2 ) - H_4(x_2)) \f{t^2(t-1)^2}{4} = G_4(t) 
+ (f(x_2 ) - H_4(x_2)) \f{t^2(t-1)^2}{4} , \\
& \pa_x H_5(x) = \f{1}{h} ( \pa_t G_4(t) +   (f(x_2 ) - H_4(x_2)) \pa_t \f{t^2(t-1)^2}{4}   ) 
\teq \f{1}{h} ( I(t) + II(t)) .
\eal
\eeq
We estimate two parts separately. Using \eqref{eq:Herm_int2}, we get 
\beq\label{eq:Herm_int3_I}
I(t)\teq \pa_t G_4 = (l_1 - l_0)
+ (3 t^2 - 2 t)(m_1 - (l_1 - l_0)) 
+ (3 t^2 - 4t + 1)( m_0 - (l_1 - l_0)) .  
\eeq
Note that $l_1 - l_0, m_1, m_0$ are approximations of $h f^{\prime}(x)$ and have cancellations. We discuss different $t \in [0, 1]$ to exploit the cancellations. 

Denote $a \vee b = \max(a, b)$. If $ t \leq \f{1}{3}$, we get 
\[
I(t)= (4 t - 3t^2) (l_1 - l_0) + 
(3 t^2 - 4 t + 1) m_0 + (2t- 3t^2)( l_1 -l_0 - m_1).
\]
The first two terms are the main terms, and the last term is the error term. Since $ t\leq \f{1}{3}$, we get 
\[
4 t - 3 t^2 > 0, \quad 3 t^2 - 4 t + 1 = (1-3t)(1-t) \geq 0, \quad  0\leq 2t - 3t^2 = 3t( \f{2}{3} - t) \leq \f{1}{3},
\]
where the last inequality is equivalent to $3 (t-\f{1}{3})^2 \geq 0$. It follows 
\beq\label{eq:Herm_C1_1D_1}
|I(t)| \leq ( |l_1 - l_0| \vee m_0 ) \cdot
(4 t - 3 t^2 + 3 t^2 - 4 t  + 1) + \f{1}{3} | l_1 - l_0 - m_1| 
=  |l_1 - l_0| \vee m_0  + \f{1}{3} | l_1 - l_0 - m_1| .
\eeq

The estimate of $t \in [2/3 , 1]$ is similar by swapping  $t $ and $1-t$, $m_0$ and $m_1$, $p_0$ and $p_1$. We get 
\beq\label{eq:Herm_C1_1D_3}
|I(t) | \leq |l_1 - l_0| \vee m_1  + \f{1}{3} | l_1 - l_0 - m_0| .
\eeq

For $ t \in [1/3, 2/3]$, we rewrite $I(t)$ \eqref{eq:Herm_int3_I} as follows 
\[
I(t) =l_1 - l_0 + (4 t - 3 t^2 - 1)( l_1 - l_0 - m_0)
+ (2 t -3t^2) (l_1 -l_0 - m_1). 
\]

Since $t \in [1/3, 2/3 ]$, we get 
\[
\bal
 & 4 t -3t^2 - 1 = (3t-1)(1-t) \geq 0, \quad 2t - 3t^2 = t(2-3t) \geq 0 , \\
& 0 \leq  4 t -3t^2 - 1 + 2t - 3t^2 = 6t - 6t^2 -1= 
6t(1-t) -1\leq 6 \cdot \f{1}{4} -1 = \f{3}{2} - 1\leq \f{1}{2} .
\eal
\]
Thus, we obtain 
\beq\label{eq:Herm_C1_1D_2}
\bal
I(t) & \leq |l_1 - l_0| + \f{1}{2} ( | l_1 -l_0 - m_0| \vee |l_1 - l_0 - m_1| ).
\eal
\eeq

Combining three cases \eqref{eq:Herm_C1_1D_1}-\eqref{eq:Herm_C1_1D_2}, we obtain the bound for $\pa_t G_4$
\beq\label{eq:Herm_C1_1D}
 \f{1}{h} |\pa_t G_4| \leq \max\B( \max_{i=0,1}( |l_1 - l_0| \vee m_i  + \f{1}{3} | l_1 - l_0 - m_{1-i}| , |l_1 - l_0| + \f{1}{2} ( | l_1 -l_0 - m_0| \vee |l_1 - l_0 - m_1| )  \B).
\eeq

For the second term in \eqref{eq:Herm_int3}, using \eqref{eq:pol_3rd_1} with $2 t \in [0, 2]$ (we consider $x\in[x_0, x_1]$), we get
\beq\label{eq:Herm_int3_II}
\pa_t \f{ t^2(t-1)^2}{4} = t(t-1)(t-\f{1}{2}), \ 
| t(t-1)(t-\f{1}{2}) |
= \f{1}{8} |2t (1-2t) (2-2t)| \leq \f{1}{8} \cdot \f{2}{3\sqrt{3}} \leq \f{1}{12 \sqrt{3}},
\eeq
which implies  
\beq\label{eq:Herm_C1_1D_4}
\B|\f{1}{h} II(t) \B| \leq \f{1}{ 12 \sqrt{3} h  } | f(x_2) - H_4(x_2)| .
\eeq

Combining the estimates \eqref{eq:Herm_C1_1D_1}-\eqref{eq:Herm_C1_1D_4}, we obtain the estimate for $\pa_x H_5$.

\subsubsection{A 4-th order interpolation in 1D}\label{eq:Herm4_1D}

We will also use a 4-th order interpolation $H_4(x)$ to approximate $f$ and $\pa_x f$. Below, we restrict the derivation to $x \in [x_0, x_1]$. We have estimated $ f - H_4$ in \eqref{eq:Herm_est1} and $H_4$ in \eqref{eq:Herm_est2}. For the derivatives, we have 
\[
\pa_x f = \pa_x( f - H_4) + \pa_x H_4 = \pa_x( f - H_4) + \f{1}{h} \pa_t G_4 \teq II_1 + II_2.
\]
We have estimated $ II_2 = \pa_t G_4$ in \eqref{eq:Herm_int3_I} and \eqref{eq:Herm_C1_1D}. For $II_1$, following the argument at the beginning of Section \ref{sec:Herm_C1_1D} and \eqref{eq:Herm_C1_1D_4} and using Rolle's theorem, we get 
\[
 |\pa_x ( f(x) - H_4(x))| \leq \f{1}{6} ||\pa_x^4 f ||_{L^{\inf}[x_0, x_1] } | (x -x_0)(x-x_1) (x - \xi(x))|
 \teq \f{1}{6} ||\pa_x^4 f ||_{L^{\inf}[x_0, x_1] } Q(x),
\]
for some $\xi(x) \in [x_0,x_1]$. Without loss of generality, we assume that $x \leq \f{x_0 + x_1}{2}$. Using $ h =x_1 - x_0$, $|x- \xi(x)| \leq \max( x_1 - x, x - x_0 ) = x_1 - x$, $ x = x_0 + t \cdot h$, and \eqref{eq:pol_3rd_2}, we get 
\beq\label{eq:Herm4_C1}
\bal
& Q(x) \leq (x -x_0)(x_1 - x)^2 = h^3 t (1-t)^2 \leq  \f{4}{27} h^3, \\
&  |\pa_x ( f(x) - H_4(x))| \leq  \f{4}{27} \cdot \f{1}{6} h^3 || \pa_x^4 f ||_{L^{\inf}[x_0, x_1]}= \f{2 h^3}{81} || \pa_x^4 f ||_{L^{\inf}[x_0, x_1]} .
\eal
\eeq

\subsection{Hermite interpolation in 2D}\label{sec:Herm_2D}

The estimate in 2D is more involved.  Consider 
\[
x_0 < x_1 < x_2, y_0 < y_1 < y_2,  \ x_2 - x_1 = x_1 - x_0, \ y_2 - y_1= y_1 - y_0 .
\]
The domain $D$ can be decomposed into 4 blocks with size $h_1 \times h_2$. For the 5-th order interpolation, we interpolate $f$ in $(x, y) \in [x_0, x_1] \times [y_0, y_1]$ without loss of generality. For the 4-th order interpolation, we only need the value of $f$ in 1 block $[x_0, x_1] \times [y_0, y_1]$. Denote 
\beq\label{eq:Herm_2D_nota}
\bal
& D = [x_0, x_2] \times [y_0, y_2], \ 
h_1 = x_1 - x_0, \  h_2 = y_1 -y_0,  \\
&  t = (x- x_0) h_1^{-1}, \  s = (y - y_0) h_2^{-1} .
\eal
\eeq

\subsubsection{Estimate the $L^{\inf}$ norm}\label{sec:Herm_linf_2D}

We first consider the estimate of $|| f||_{L^{\inf}}$. We treat $y$ as a parameter and use \eqref{eq:Herm_int1} to construct the 4-th and 5-th order interpolations in $x$: $H_4(x, y), H_5(x, y)$. Using the formula in \eqref{eq:Herm_int2}, we have
\beq\label{eq:Herm_5th_dec}
\bal
G_5(t, s) &\teq H_5(x, y) = H_4(x, y) + ( f( x_2, y) - H_4(x_2, y)  ) \f{t^2(t-1)^2}{4} , \\
H_4(x, y) & = f(x_0, y) ( 1-t) + f(x_1, y) t 
+ t^2(t-1) ( h f^{\pr}( x_1, y) - ( f(x_1, y) - f(x_0, y) ) ) \\
 & \quad + (t-1)^2 t (  h f^{\pr}(x_0, y ) -   ( f(x_1, y) - f(x_0, y) ) ) .
\eal
 \eeq

Using \eqref{eq:Herm_est1}, we have the error bound in $x$
\beq\label{eq:Herm_5th_err1}
| H_5(x, y) - f(x, y)| \leq \f{1}{1200} || \pa_x^5 f||_{L^{\inf}(D) } h_1^5.
\eeq

We need to further interpolate $G_5(t, s), H_5(x, y)$ in the $y$ direction. To achieve the overall 5-th order error, we do not need to apply a high order interpolation to each coefficient. For the linear term, we apply the 5-th order Hermite interpolation to $f(x_i, y)$ in $y$ using $f(x_i, y_j), j=0,1,2$ and $\pa_y f(x_i, y_j), j=0, 1$ for $i=0, 1$ and denote it by $A^{(5)}_i(y)$, i.e.
\beq\label{eq:Herm_5th_inty5_1}
A_i^{(5)}(y) = H_5( f(x_i, \cdot), y_0, y_1, y_2)(y)
\eeq
using the notation in \eqref{eq:Herm_int1}. Applying \eqref{eq:Herm_est1}, we have the error bound in $y$
\beq\label{eq:Herm_5th_err2}
 |f(x_i, y) - A_i^{(5)}(y)| \leq \f{1}{1200}|| \pa_y^5 f||_{L^{\inf}(D)}  h_2^5 . 
\eeq

Denote 
\beq\label{eq:Herm_5th_pol2}
M_i( y) \teq h f^{\pr}( x_i, y) - ( f(x_1, y) - f(x_0, y) ). 
\eeq

For $ M_i(y), i=1,2$, it is of order $h_1^2$. Thus, we apply the cubic interpolation in Section \ref{sec:newton_cubic} to these functions in the $y$ direction on grids $y_0, y_1, y_2$ and denote it by $Q_i(y)$. Using the notation \eqref{eq:newton_cubic_dec}, we have 
\beq\label{eq:Herm_5th_inty3_1}
Q_i(y) = N_2( M_i, y_0, y_1, y_2)(y).
\eeq
Applying \eqref{eq:newton_cubic1}, we have the error bound 
\beq\label{eq:Herm_5th_err3}
\bal
&| M_i(y) - Q_i(y) | 
\leq  \f{h_2^3}{9\sqrt{3}} || \pa_y^3 M_i(y)   ) ||_{L^{\inf}([y_0, y_2])} 
 \leq \f{h_2^3 h_1^2}{ 18 \sqrt{3}}  || \pa_x^2 \pa_y^3 f||_{L^{\inf}(D)},
\eal
\eeq
where we have used the following estimate with $c = a, b, g = \pa_y^3 f(\cdot, y)$ in the last inequality 
\beq\label{eq:Herm_2nd_Mi}
\B| (b-a) g^{\pr}(c) - (g(b) - g(a))\B| = \B|\int_a^{b} g^{\pr}(c) - g^{\pr}(s) ds\B|
\leq ||  g^{\pr \pr} ||_{L^{\infty}[a, b]} \int_a^b |c-s| ds 
= \f{(b-a)^2}{2} ||  g^{\pr \pr} ||_{L^{\infty}[a, b]}.
\eeq

The last term $ f( x_2, y) - H_4(x_2, y)$ is already very small and of order $h_1^4$. From \eqref{eq:Herm_5th_dec}, since $\pa_y$ commutes with $t = (x - x_0) h^{-1}$ and $\pa_x$, for a fixed $y$, $\pa_y H_4(x, y)$ is the Hermite polynomials for $\pa_y f(x, y)$ in $x$. We use the first order estimate in $y$ and then \eqref{eq:Herm_est0} with $\pa_y f $ and $x= x_2$ to obtain
\beq\label{eq:Herm_5th_err4}
\bal
|f(x_2, y) - H_4(x_2, y)| &  \leq \max_{j=0,1} | f(x_2, y_j) - H_4(x_2, y_j)|
+ \f{h_2}{2} |\pa_y ( f(x_2, y) -H_4(x_2, y) ) | \\
& \leq \max_{j=0,1} | f(x_2, y_j) - H_4(x_2, y_j)|
+ \f{h_2 h_1^4}{2 \cdot 6} ||\pa_y \pa_x^4 f ||_{L^{\inf}(D)} ,
\eal
\eeq
where we have used $x_2 - x_0 = 2 h_1 , x_1 - x_0 = h_1, \f{(x_2 -x_0)^2 (x_1-x_0)^2}{24} = \f{ 4h_1^2 \cdot h_1^2}{24} = \f{h_1^4}{6}$ in \eqref{eq:Herm_est0}.

\vspace{0.1in}
\paragraph{\bf{Estimate the 2D interpolating polynomials for $f$}}

We obtain the 2D interpolating polynomials in $(x, y) \in [x_0, x_1] \times [y_0, y_1]$ for $H_4(x, y)$ as follows
\beq\label{eq:Herm_5th_dec2}
P_5(x, y) = \B( A_0^{(5)}(y) (1- t) + A_1^{(5)}(y) t \B)
+ \B( t^2 (t-1) Q_1(y) + t(t-1)^2 Q_0(y) \B) = I(t, y)+ II(t, y).
\eeq

We restrict $t \in [0, 1], x \in [x_0, x_1], y\in[y_0, y_1]$. Applying \eqref{eq:Herm_5th_err1} for $H_5 - f$, \eqref{eq:Herm_5th_err4} for $H_4 - H_5$ in \eqref{eq:Herm_5th_dec}, \eqref{eq:Herm_5th_err2}-\eqref{eq:Herm_5th_err3} for $H_4 - P_5$,  $t + (1-t)=1$, we have the following error bound 
\beq\label{eq:Herm_5th_est1}
\bal
& |P_5(x, y) - f(x, y)|  \leq  |H_4(x, y) - H_5(x, y)| 
+ |H_5(x, y) - f(x, y)| + |H_4(x,y) - P_5(x, y)|  ,\\
&|H_4 - H_5| \leq \f{1}{64} |f(x_2, y) - H_4(x_2, y)|  \leq \f{1}{64}\max_{j=0,1} | f(x_2, y_j) - H_4(x_2, y_j)| 
 + \f{h_2 h_1^4}{768 } ||\pa_y \pa_x^4 f ||_{L^{\inf}(D)}, \\
& |H_5 - f| +  |H_4  - P_5|  \leq  \f{1}{72 \sqrt{3}} h_1^2 h_2^3 || \pa_x^2 \pa_y^3 f||_{L^{\inf}(D)}  + \f{1}{1200} ( h_1^5 ||\pa_x^5 f||_{L^{\inf}(D)} + h_2^5 || \pa_y^5 f ||_{L^{\inf}(D)} ) ,
\eal
\eeq
where the constant $\f{1}{64}$ in $H_4-H_5$ is from \eqref{eq:Herm_5th_H5err}, $\f{1}{768}$ comes from \eqref{eq:Herm_5th_err4} and $\f{1}{768} = \f{1}{64} \cdot \f{1}{16}$, $\f{1}{72 \sqrt{3}}$ from $ \f{1}{18 \sqrt{3}} (t^2|t-1| + (t-1)^2 t) 
=  \f{1}{18 \sqrt{3}} (1-t) t \leq \f{1}{72 \sqrt{3}}$ \eqref{eq:Herm_5th_err3} .
To estimate $P_5(x, y)$ \eqref{eq:Herm_5th_dec2}, we use estimates similar to \eqref{eq:Herm_est2}. The estimate of the linear part is trivial 
\beq\label{eq:Herm5_P5_I}
|I(t, y)| \leq \max_{i=0,1} || A_i^{(5)}(y)||_{L^{\inf}[y_0, y_1]}, \quad y \in [y_0, y_1].
\eeq
Since $A_i(y)$ \eqref{eq:Herm_5th_inty5_1} is the Hermite polynomial in $y$ , we can use the method in Section \ref{sec:Herm_1D} to estimate it. For $II$, following the derivations between \eqref{eq:Herm_pol_pf2} to \eqref{eq:Herm_est2} and by considering two cases: $t \leq \f{1}{2}$ and $t>\f{1}{2}$, we obtain 
\beq\label{eq:Herm5_P5_II}
|II(t, y)| \leq \f{4}{27} \max( |Q_1(y) - Q_0(y) |, |Q_1(y)|, |Q_0(y)|). 
\eeq
Since $Q_i, Q_1 - Q_0$ are quadratic interpolating polynomials of $M_i, M_1 - M_0$ \eqref{eq:Herm_5th_pol2} 
on $y_0, y_1, y_2$, we can use \eqref{eq:newton_cubic2} to estimate the $L^{\inf}[y_0, y_1]$ norm.

\subsubsection{Estimates of the $\pa f(x, y)$ in $2D$}\label{sec:Herm_C1_2D}

Now, we consider how to estimate $\pa f(x, y)$ using the Hermite interpolation. We consider $\pa_x$ without loss of generality. Recall the notation \eqref{eq:Herm_2D_nota}. We first fix $y$ as a parameter and interpolate $f(x, y)$ in $x$ using the same method as \eqref{eq:Herm_5th_dec}. Using the error estimate \eqref{eq:Herm_C1_est1}, we yield 
\beq\label{eq:Herm_5th_C1_err0}
 |\pa_x ( f(x, y) - H_5(x, y))| \leq \f{h_1^4}{96} || \pa_x^5 f||_{L^{\inf}(D)} .
\eeq

Using the computation \eqref{eq:Herm_int3}, \eqref{eq:Herm_int3_I}, \eqref{eq:Herm_int3_II}  in Section \ref{sec:Herm_C1_1D} and the notation \eqref{eq:Herm_pol_pf1} for $m_i, l_i$, we have 
\beq\label{eq:Herm5_C1}
\bal
h_1 \pa_x H_5(x, y)  & = h_1 \pa_x H_4(x, y) + ( f(x_2, y) - H_4(x_2, y)) t(t-1)(t-\f{1}{2}) , \\
h_1 \pa_x H_4(x, y)  &= f( x_1, y ) - f(x_0, y) + (3 t^2 - 2t) ( h f^{\pr}( x_1, y) - ( f(x_1, y) - f(x_0, y) ) ) \\
& \quad + (3 t^2 - 4 t + 1) ( h f^{\pr}( x_0, y) - ( f(x_1, y) - f(x_0, y) ) )  \\
& =  f( x_0, y ) - f(x_1, y) + (3 t^2 - 2t) M_1(y)  + (3 t^2 - 4 t + 1) (M_0(y) , \\
\eal
\eeq
where $ t = \f{x-x_0}{h_1}$ \eqref{eq:Herm_2D_nota}, and we have used $M_i$ \eqref{eq:Herm_5th_pol2} to simplify the presentation.

Next, we interpolate the above functions in $y$. We want to achieve an overall 4-th order approximation for $\pa_x H_4(x, y)$. For $f(x_1, y) - f(x_0, y)$, we use the 4-th order Hermite interpolation in $y$ based on the grid point values $f(x_1, y_j) - f(x_0, y_j), \pa_y ( f(x_1, y_j) - f(x_0, y_j) ), j=0, 1$ and denote it as $B(y)$, i.e.
\beq\label{eq:Herm_5th_inty5_2}
B(y) \teq H_4( f(x_1, \cdot) - f(x_0, \cdot), y_0, y_1),
\eeq
using the notation \eqref{eq:Herm_int1}.
By \eqref{eq:Herm_est1}, we have the error estimate 
\beq\label{eq:Herm_5th_C1_err1}
 \B|\f{ B(y)}{ h_1} -  \f{ f(x_1, y) - f(x_0, y)}{ h_1} \B|
 \leq \f{h_2^4}{384} \B|\B| \pa_y^4  \f{ f(x_1, y) - f(x_0, y)}{ h_1} \B| \B|_{L^{\inf}([y_0, y_1])}
 \leq \f{h_2^4}{384} \B|\B| \pa_y^4  \pa_x  f \B| \B|_{L^{\inf}(D)}. 
\eeq

For $M_i$, we apply the same quadratic interpolation $Q_i$ in $y$ \eqref{eq:Herm_5th_inty3_1}. Using the error bound \eqref{eq:Herm_5th_err3} and \eqref{eq:pol_2nd_2} in Lemma \ref{lem:pol_est}, we obtain 
\[
\bal
& h_1^{-1} \B| (3 t^2 - 2 t) M_1(y) + (3 t^2 - 4 t +1) M_0(y)
- (3 t^2 - 2 t) Q_1(y) - (3 t^2 - 4 t +1) Q_0(y)\B|  \\
\leq &  h_1^{-1} ( |3 t^2 -2 t| + |3t^2 - 4 t + 1|) \max_{i=1,2} |Q_i - M_i| 
\leq \f{h_2^3 h_1}{ 18 \sqrt{3}}  || \pa_x^2 \pa_y^3 f||_{L^{\inf}(D)} .
\eal
\]

For $ f( x_2, y) - H_4(x_2, y)$, we use the same estimate \eqref{eq:Herm_5th_err4}, which along with \eqref{eq:Herm_int3_II} implies 
\beq\label{eq:Herm_5th_C1_err2}
\bal
\f{1}{h_1} \B|( f(x_2, y) - H_4(x_2, y)) t(t-1)(t-\f{1}{2})\B|
&\leq \f{1}{12 \sqrt{3} h_1} \max_{j=0,1} | f(x_2, y_j) - H_4(x_2, y_j)| \\
& \quad + \f{h_2 h_1^3}{12 \cdot 12 \sqrt{3}} ||\pa_y \pa_x^4 f ||_{L^{\inf}(D)} .
\eal
\eeq

\vspace{0.1in}
\paragraph{\bf{Estimate the 2D interpolating polynomials for $\pa_x f$}}
Now, we use \eqref{eq:Herm_5th_inty5_2}, \eqref{eq:Herm_5th_inty3_1} to construct the interpolating polynomials for $  h_1 \pa_x H_4$ 
\beq\label{eq:Herm5_C1_int}
S_4(x, y)
=  B(y) + (3 t^2 - 2 t) Q_1(y) + (3 t^2 - 4 t + 1) Q_0(y) .
\eeq
Combining the estimate \eqref{eq:Herm_5th_C1_err1} and using the triangle inequality, we can estimate the error $\f{1}{h_1} S_4(x, y) - f(x, y)$.

It remains to estimate $S_4(x, y)$. We further decompose the above approximation as the linear part and the nonlinear part in $y$. The linear part in $y$ is the main term, and we want to obtain a sharper estimate. 
Since $B$ is the $4-th$ order Hermite interpolation in $y$ \eqref{eq:Herm_5th_inty5_2}, 
we can apply the decomposition \eqref{eq:Herm_int2} into the linear and the nonlinear parts to $B$. Since $Q_i$ is the quadratic interpolation of $M_i$ \eqref{eq:Herm_5th_pol2}, \eqref{eq:Herm_5th_inty3_1},  we can apply the decomposition \eqref{eq:newton_cubic_dec} into the linear and the quadratic terms to $Q_i$
\beq\label{eq:Herm_5th_dec4}
S_4 = S_{lin} + S_{nlin}, \quad 
S_{ \s }(t, y) \teq B_{ \s}(y) + (3 t^2 - 2 t)  Q_{1, \s}(y) + (3 t^2 - 4t+1 ) Q_{0, \s}(y), 
\eeq
where $\s \in \{ lin, nlin \}$, $f_{lin}, f_{nlin}$ denote the linear part and nonlinear part
in $y$, respectively.

Since $S_{lin}$ is linear in $y$ and $B(y_j) = f(x_1, y_j) - f(x_0, y_j)  , Q_i(y_j) = M_i(y_j)$ for $j =0, 1$ (the interpolating polynomials agree with the functions on the grid points), we get 
\beq\label{eq:Herm5_C1_Slin}
\bal
|S_{lin}(t, y)| & \leq \max_{ i=0,1} \B| B_{ lin}(y_i) + (3 t^2 - 2 t)  Q_{1, lin }(y_i) + (3 t^2 - 4t+1  ) Q_{0, lin}(y_i)     \B|,  \\
& =   \max_{ i=0,1} \B| f(x_1, y_i) -f(x_0, y_i) + (3 t^2 - 2 t)  M_{1}(y_i) + (3 t^2 - 4t+1 ) M_{0}(y_i)     \B|.
\eal
\eeq
Recall the definition of $M_i$ in \eqref{eq:Herm_5th_pol2}. The polynomials inside $| \cdot | $ is the 1D polynomial in $t$ with the same form as $I(t)$ \eqref{eq:Herm_int3_I}. We estimate it 
using \eqref{eq:Herm_C1_1D}, \eqref{eq:Herm_C1_1D_4} and following Section \ref{sec:Herm_C1_1D}.

For the nonlinear part, we have
\beq\label{}
S_{nlin} = B_{nlin}( y)
+  (3 t^2 - 2 t)  Q_{1, nlin}(y) + (3 t^2 - 4t + 1 ) Q_{0, nlin}(y), 
\eeq
we recall the definition of $B$ \eqref{eq:Herm_5th_inty5_2} and $Q$ \eqref{eq:Herm_5th_inty3_1}. 
The estimate of $B_{nlin}$ follows the method of estimating $II(t)$ in \eqref{eq:Herm_int2}, \eqref{eq:Herm_est2}  with 
$l_i, m_i$ replaced by 
\[
\td l_i = f(x_1, y_i) - f(x_0, y_i), \quad  \td m_i = h_2  \pa_y (f(x_1, y_i) - f(x_0, y_i) ),
\]
which gives 
\beq\label{eq:Herm5_C1_Bnlin}
|B_{nlin}(y) | 
\leq \f{4}{27} \max( |\td m_0, \td m_1|,  
| \td m_0 - ( \td l_1 - \td l_0)|, | \td m_1 - (\td l_1 - \td l_0)| ) .
\eeq

The estimate of the $Q_{nlin}$ follows that in \eqref{eq:newton_cubic2} and Section \ref{sec:newton_cubic}, which gives 
\[
|Q_{i, nlin} | \leq \f{1}{8} | M_i(y_0) - 2 M_i(y_1) + M_i(y_2)|. 
\]

Using \eqref{eq:pol_2nd_2} in Lemma \ref{lem:pol_est}, for $t\in [0, 1]$, we prove 
\[
\bal
 &| (3 t^2 - 2 t)  Q_{1, nlin}(y) + (3 t^2 - 4t + 1 ) Q_{0, nlin}(y)| \\
 \leq &  (|3t^2 -2t| + |3t^2 - 4t + 1|  ) \max_{i=0, 1} | Q_{i, nlin}(y)| 
 \leq \f{1}{8} \max_{i=0, 1}  | M_i(y_0) - 2 M_i(y_1) + M_i(y_2)|. 
 \eal
\]
Thus, we yield the estimate of $S_{nlin}$. Combining the above estimates of $S_{lin}$ and $S_{nlin}$ and using \eqref{eq:Herm_5th_dec4}, we obtain the estimate of $S_4$. We remark that one needs to further divide the above bounds by $\f{1}{h_1}$ to get the bound for $\f{S_4}{h_1}$.

\subsubsection{The fourth order estimate}

Denote $D = [x_0, x_1] \times [y_0, y_1]$. We only need the value of $f$ in $D$.
The above method interpolates $f(x, y)$ with a fifth order error estimate. We also use a simpler fourth order method to estimate $f(x, y)$. Recall the interpolation polynomials $H_4$ and $M_i$ in $x$ defined in \eqref{eq:Herm_5th_dec}, \eqref{eq:Herm_5th_pol2}. We consider the following interpolation for $f(x, y)$
\beq\label{eq:Herm_4th_dec2}
P_4(x, y) = \B( A^{(4)}_0(y) (1- t) + A_1^{(4)}(y) t \B)
+ \B( t^2 (t-1) l_1(y) + t(t-1)^2 l_0(y) \B) = I(t, y)+ II(t, y),
\eeq
 similar to \eqref{eq:Herm_5th_dec2}, where $l_i$ is a linear interpolation of $M_i$ in $y$, $A^{(4)}$ is the 4-th order Hermite interpolation \eqref{eq:Herm_int1} of $f(x_i, y)$ in $y$ using $f(x_i, y_j), j=0,1$ and $\pa_y f(x_i, y_j), j=0, 1$ for $i=0, 1$ 
\beq\label{eq:Herm_4th_inty4_1}
A^{(4)}_i(y) = H_4( f(x_i, \cdot), y_0, y_1)(y).
\eeq
For a fixed $t$, applying  the linear interpolation error bound for $M_i - l_i$ and \eqref{eq:Herm_2nd_Mi}, we yield 
\beq\label{eq:Herm4_2nd_M1_lin}
|l_i - M_i| \leq \f{h_2^2}{8} ||\pa_{yy} M_i||_{L^{\inf}(D)}
\leq \f{h_2^2 h_1^2}{16} || \pa_x^2 \pa_y^2 f ||_{L^{\inf}(D)}, \quad i = 1, 2, 
\eeq
which along with \eqref{eq:Herm_est1} for interpolation error $A_i^{(4)} - f(x_i, y)$ in $y$, $t(1-t) \leq \f{1}{4}$ implies 
\[
|H_4 - P_4|  \leq \f{h_2^4}{384} || \pa_y^4 f ||_{L^{\inf}(D) }
+ ( |t^2(t-1)| + | t(1-t)^2 | ) \max_i( |l_i - M_i|)
\leq \f{h_2^4}{384} || \pa_y^4 f ||_{L^{\inf}(D) } + \f{h_1^2 h_2^2}{64} || \pa_x^2 \pa_y^2 f ||_{L^{\inf}(D)}. 
\]

Applying \eqref{eq:Herm_est1} for the interpolation error $H_4 - f$ in $x$, we obtain 
\[
| P_4 -f | \leq |H_4 - f| + |H_4(x, y) - P_4(x, y)|
\leq \f{ 1}{384} (h_1^4 || \pa_x^4 f ||_{L^{\inf}(D)} + h_2^4 || \pa_y^4 f ||_{L^{\inf} (D) } ) + \f{h_1^2 h_2^2}{64} || \pa_x^2 \pa_y^2 f ||_{L^{\inf}(D)}.
\]

To estimate $P_4$, we follow the estimates \eqref{eq:Herm5_P5_I} and \eqref{eq:Herm5_P5_II}
\[
|I(y)| \leq \max_{i=0,1} |A_i^{(4)}(y)|_{L^{\inf}[y_0, y_1]}, \quad 
|II(y)| \leq \f{4}{27} \max( |l_1(y) - l_0(y)|, |l_0(y)|, |l_1(y)| ). 
\]
The estimate of the 4-th order interpolation polynomial $A_i^{(4)}(y)$ in $y$ follows \eqref{eq:Herm_est2}. Since $l_i(y), l_1(y) - l_0(y)$ is linear in $y$, the maximum of each term is achieved at the endpoint $y = y_0, y= y_1$, and we can bound $l_i, l_1(y) - l_0(y)$ and $II(t)$. 

\vs{0.1in}

\paragraph{\bf{Estimate $\pa f$}}

We use $\pa H_4(x, y)$ to approximate $\pa f$, which has the formula \eqref{eq:Herm5_C1}. 
We further approximate $h_1 \pa_x H_4 $ by constructing interpolation similar to \eqref{eq:Herm5_C1_int}
\[
 \td S_4 = B(y) + (3 t^2 - 2 t) l_1(y) + (3 t^2 - 4 t + 1) l_0(y) ,
\]
with $Q_i$ in \eqref{eq:Herm5_C1_int} replaced by the linear interpolation of $M_i$ \eqref{eq:Herm5_C1}, $l_i$, $B(y)$ is the same as \eqref{eq:Herm_5th_inty5_2}. Using \eqref{eq:Herm_5th_C1_err1}, \eqref{eq:Herm4_2nd_M1_lin}, and \eqref{eq:pol_2nd_2}, we yield 
\[
\bal
& | \pa_x f - \f{\td S}{h_1}|
\leq  | \f{\td S_4}{h_1} - \pa_x H_4 | 
+ | \pa_x f - \pa_x H_4 | = I + II, \\
& |I| \leq 
 \B|\f{ B(y)}{ h_1} -  \f{ f(x_1, y) - f(x_0, y)}{ h_1} \B|
+  \f{1}{h_1} ( | 3 t^2 - 2 t | + |3 t^2 - 4 t+ 1| ) \max_i | M_i - l_i| \\
& \quad  \leq \f{h_2^4}{384} || \pa_y^4 \pa_x f ||_{L^{\inf}(D)}
+ \f{h_1 h_2^2}{16} || \pa_x^2 \pa_y^2 f ||_{L^{\inf}(D)} ,
\eal 
\]
where $D = [x_0, x_1] \times [y_0, y_1]$. For $II$, we apply \eqref{eq:Herm4_C1}.  It remains to estimate $\td S$. Using the notation in \eqref{eq:Herm_5th_dec4}, $Q_{i, lin} = l_i, B = B_{lin} + B_{nlin}$, we obtain
\[
\td S_4 = B + (3t^2 - 2t) Q_{1, lin} + (3t^2 - 4 t + 1) Q_{0, lin}
= S_{lin} + B_{nlin},
\]
and then estimate $\td S_4$ using \eqref{eq:Herm5_C1_Slin} and \eqref{eq:Herm5_C1_Bnlin}. The estimate for $\pa_y f$ is similar.

\subsection{Weighted estimates of a function using derivatives}\label{sec:sing_near0}

In our estimate of the residual error or some norms, we need to estimate $F(x) \rho(x)$ near $x=0$ with a singular weight $\rho$. In this section, we discuss how to use the estimate of the derivatives of $F$ to estimate the weighted norm of $F$. Note that $\pa_x^i \pa_y^j F$ can be 
estimated by the methods in Sections \ref{sec:newton}-\ref{sec:Herm_2D}.

Typical behavior of $\rho$ near $x=0$ is $|x|^{-k}, k =2, \f{5}{2}, 3$ or $|x|^{-k} x_1^{-1/2}$. See \eqref{wg:linf}, \eqref{wg:hol}. By decomposing 
$\rho = |x|^{-i} x_1^{-j/2} \rho_{m}$, where $\rho_m$ is regular near $x=0$,
we only need to estimate $F |x|^{-i} x^{-j/2}$. Denote 
\[
\bal
E_x(F, i)(x, y)  \teq  \f{1}{x^{i+1}} \int_0^x   F(z, y)( x-z)^i dz , 
\quad 
E_y( F, i)(x, y) \teq \f{1}{y^{i+1}}  \int_0^y  F(x, z) (y-z)^i dz .
\eal
\]
Using Fubini's Theorem, we get 
\[
E_x( E_y(F, j), i) = E_y( E_x(F, i), j) = \f{1}{x^{i+1} y^{j+1}} \int_0^x \int_0^y F (x-t)^i (y-s)^j dt ds \teq 
E_{ij}(F)(x, y).
\]
Using the piecewise bound of $ F$, we can bound these functions easily. See \eqref{eq:piece_power}.

For $F(x, y)$ odd in $x$ and $\na^k F(0) = 0$ for $k \leq 2$, we have the following identities 
\[
\bal
F(x, y)  & = \int_0^y  F_y(x, z) dz + F(x, 0) \\ 
 &= \int_0^x \int_0^y F_{xyy}( t, s) (y-s) dt ds
+ y \int_0^x F_{xxy}(t, 0) (x - t) dt + \f{1}{2} \int_0^x F_{xxx}(t, 0) (x-t)^2d t.
\eal
\]

Using the average operators, we yield 
\[
|F(x, y)| 
\leq  E_{xy}(|F_{xyy}|, 0, 1)  x y^2+ x^2 y  E_x( | F_{xxy}(\cdot ,0) | , 1)
+ \f{x^3}{2}  E_x( | F_{xxx}(\cdot ,0) |, 2)(x, 0). 
\]
Denote $\b = \arctan( \f{y}{x}), r = (x^2 + y^2)^{1/2}$. Using these estimates, 
for $a+b = 3, b \leq 1$, we get 
\beq\label{eq:max_near1}
\bal
\f{|F(x,y)|}{ r^{a} x^{b}} &\leq 
 \f{1}{2} E_x( |F_{xxx}|, 2)(x, 0) \cos^{3-b}(\b)   + E_{xy}( |F_{xyy} |, 0, 1) \cos \b^{1-b} \sin^2 \b  \\
 & \quad  + E_x(|F_{xxy}(\cdot ,0)  | , 1 ) \cos^{2-b} \sin \b .
\eal
\eeq
Since we have the bounds for these coefficients, e.g., $E_{xy}( |F_{xxy}|, 0, 1)$, by maximizing $\b \in [0, \pi/2]$, we obtain the bounds for $\f{F}{r^3}$ and $ \f{F}{ r^{5/2} x^{1/2}} $. 

Similarly, we can bound $\f{ \pa_x^i \pa_y^j F}{ r^k}, i + j + k \leq 3$. For odd $F$, we have 
\beq\label{eq:max_near2}
\bal
|F| &= \B| \int_0^x F_x(z, y) dz \B| \leq E_x( |F_x|, 0) x   ,  \\ |F| &= \B| \int_0^y F_y(x, z)  d z + F(x, 0) \B|
= \B| \int_0^x \int_0^y F_{xy}(t, s) dx dy + \int_0^x F_{xx}(t, 0) (x- t) dt \B| \\
& \leq E_{xy}(| F_{xy} | , 0, 0 )  xy  + E_x( |F_{xx}(\cdot ,0)|, 1)( x, 0) x^2 , \quad  \na^k F(0 )  = 0, k=0,1 , \\
|F_x| &= 
|\int_0^y F_{xyy}(x, z) (y-z) d z 
+  y \int_0^x F_{xxy}(z, 0) dz 
+ \int_0^x F_{xxx}(z, 0)(x-z) dz  | \\
& \leq  E_y( |F_{xyy}|, 1 )  y^2 
+ x y E_x( |F_{xxy}| (\cdot, 0)|, 0 )
+ x^2 E_x( |F_{xxx}(\cdot, 0)|, 1) , \quad \na^k F(0)=0, k \leq 2.
\eal
\eeq
Using estimate similar to \eqref{eq:max_near1}, we can bound $\f{F}{|x|^2}, \f{F}{ |x|^{3/2} |x_1|^{1/2} } ,  \f{F}{|x|}$.

\paragraph{\bf{Weighted derivatives}}
Similarly, we estimate $\f{ \pa_{x_i} F}{|x|^2}, \f{ \pa_{x_i} F x_j^{1/2}}{|x|^{5/2}}$. For odd $F$ with $\na^k F = 0, k \leq 2$, using \eqref{eq:max_near2} with $F$ replaced by $F_y$, we get 
\[
|F_y| \leq E_{xy}( |F_{xyy}| ,0, 0) xy + E_x ( |F_{xxy}(\cdot ,0), 1)(x,0) x^2. 
\]

Then, we can use the method in \eqref{eq:max_near1} to estimate 
\[
 \f{|x_j|^{\al}\pa_{x_i} F}{|x|^{2+\al}} = (g(\b))^{\al}  \f{ \pa_{x_i} F}{|x|^2}, \quad g(\b) = \cos \b, j = 1, \mathrm{ \ or \ } \sin \b, j = 2.
\]

We also need to estimate $\f{ F_x}{|x|}$ and $\f{F_y}{|x|}$. 
Using \eqref{eq:max_near2} with $F$ replaced by $F_y$, we get 
\[
|F_y |  \leq E_x( |F_{xy}|, 0) x. \]

For $ \f{F_x}{|x|}$, we have two cases. If $F(x, 0) \equiv 0$, 
we yield 
\[
|F_x | \leq E_y( |F_{xy}, 0|) y . \]

Without the vanishing conditions, we require $\na F (0,0) =0 $ and yield 
\[
\bal
|F_x(x, y)| &= \B| \int_0^y F_{x y}(x, z) dz + F_x(x, 0) \B|
 = \B| \int_0^y F_{x y}(x, z) dz + \int_0^x F_{xx}( z, 0) dz \B| \\
 & \leq E_y( |F_{xy}|, 0) y + E_x( |F_{xx}|(\cdot ,0), 0) x .
 \eal
\]
Then we apply the method in \eqref{eq:max_near1} to estimate $\f{ \pa_{x_i} F}{|x|}$.

\vs{0.1in}
\paragraph{\bf{Piecewise bound for powers}}

Let $x_i = ih, I_i = [x_{i-1}, x_i]$. Suppose that $0 \leq f$ with $f(x) \leq f_i , x \in I_i, i \geq 1 $. We can obtain a piecewise bound for
\beq\label{eq:piece_power}
  I_k(f) = \f{1}{x^{k+1}} \int_0^x f(z)(x-z)^k dz  
  = \f{1}{x^{k+1}} \int_0^x f(x-z) z^k dz  
  =  \int_0^1 f( x(1 - t)) t^k d t,
\eeq
from above for $x \in I_m = [(m-1) h, m h]$,  where we have used a change of variable $z = t x$. We fix $x \in I_m = [(m-1) h, m h]$. 
We partition $[0, 1]$ into $J_i = [(i-1)/ m, i / m], 1 \leq i \leq m$.
Since for $t \in [(i-1)/m, i /m]$, we have $ (m-1)(m-i) = m^2 - m(i+1) + i \geq m( m-i-1)$ and 
\[
   x(1-t) \geq (m-1)h (1 - \f{i}{m} ) = \f{ (m-1) (m-i) h }{m} \geq (m-i-1)h,
   \quad x(1-t) \leq m h (1 - \f{i-1}{m} )
   = (m-i + 1) h.
\]

Thus for $t \in J_i$, we get $x(1-t) \in I_{m-i} \cup I_{m-i+1}, I_0 = \emptyset, f(x(1-t))
\leq \max( f_{m-i}, f_{m-i+1} )$. It follows 
\[
I_k(f) = \sum_{1\leq j \leq m } \int_{J_i} f(x(1-t)) t^k d t
\leq \f{1}{k+1} \sum_{ 1 \leq i \leq m}
 \f{1}{m^{k+1}}  (   i^{k+1} - (i-1)^{k+1}) \max( f_{m-i}, f_{m-i+1}  ), \ f_0 = 0.
\]

\subsection{Estimate of some explicit polynomials}

We use the following bounds for some polynomials in the error estimate of the interpolation.

\begin{lem}\label{lem:pol_est}
We have the following estimates
\begin{align}
& | (t-a)(t- b)| \leq \f{ (b-a)^2}{4}, \quad  t \in [a, b] , \label{eq:pol_2nd} \\
& |3 t^2 - 2t | + |3 t^2 - 4 t + 1| \leq 1, \quad t \in [0, 1], \label{eq:pol_2nd_2} \\
& |t(t-1)(t-2)| \leq \f{2}{ 3 \sqrt{3}}, \quad t \in [0, 2],  \label{eq:pol_3rd_1} \\
&  t(t-1)^2  \leq \f{4}{27} , \quad t\in [0, 1],  \label{eq:pol_3rd_2} \\
& | t (t-1) (t-2)(t-3)| \leq 1,   \quad  t \in [0, 3],  \label{eq:pol_4th_1} \\  
& |t^2 (t-1)^2 (t-2)| \leq \f{1}{10}, \quad t \in [0, 1]. \label{eq:pol_5th}
\end{align}

\end{lem}

\begin{proof}

The proof of \eqref{eq:pol_2nd} follows from the inequality of arithmetic and geometric means (AM-GM) or a direct calculation. 

For \eqref{eq:pol_2nd_2}, firstly, we note that $|a| + |b| = |a+b|$ or $|a-b|$. It suffices to prove $|a + b| \leq 1$ and $|a-b| \leq 1$ for $a =3t^2 - 2t, b = 3 t^2 - 4 t+1$. Since $t^2 - t \in [-1/4, 0], 2t-1 \in [-1, 1]$, we have 
\[
a+b = 6 t^2 - 6 t + 1 \in [-1/2, 1], \quad a- b  = 2t -1 \in [-1, 1],
\]
which implies $|a+b|, |a- b|\leq 1$. We prove the desired result.

Denote $s = t-1 \in [-1, 1]$. Then using the AM-GM inequality, we have 
\[
t^2(t-1)^2(t-2)^2
 = (t-1)^2( t^2-2 t   )^2 
= \f{1}{2} 2 s^2(1-s^2 )^2 \leq \f{1}{2} (  \f{2 s^2 + 2(1-s^2)  }{3} )^3 = \f{4}{27}.
\]
Taking the squart root on both sides proves \eqref{eq:pol_3rd_1}.

To prove \eqref{eq:pol_3rd_2}, applying the AM-GM inequality, we get 
\[
t(1-t)^2 = \f{1}{2} 2 t(1-t)^2 \leq \f{1}{2}  (\f{ 2t + 2(1-t)}{3})^3 = \f{4}{27}.
\]

Denote $s = t (3 - t) \in [0, \f{9}{4}]$. Then we have $|s-1|^2 \in [0, 2]$ and 
\[
|(t-1)(t-2)t(t-3)| = | s( t^2 -3t+2) | = | (2 - s) s| =| 1 - (s-1)^2 |  \leq 1,
\]
which implies \eqref{eq:pol_4th_1}.

To prove \eqref{eq:pol_5th}, we use \eqref{eq:pol_2nd} with $a=0,b=1$ and \eqref{eq:pol_3rd_1} to obtain 
\[
|t^2(t-1)^2(t-2)| \leq \f{1}{4} \f{2}{3\sqrt{3}} = \f{1}{ 6 \sqrt{3}} < \f{1}{10},
\]
where the last inequality follows from $ (6\sqrt{3})^2 = 108 > 100 = 10^2$.
\end{proof}

\bibliographystyle{plain}
